\documentclass[11pt,leqno]{report}
\usepackage{amsmath,amsthm,amsfonts,amssymb}
\usepackage[colorlinks=true, pdfstartview=FitV, linkcolor=blue, citecolor=red, urlcolor=blue]{hyperref}
\hypersetup{breaklinks=true}
\usepackage{graphicx}
\usepackage{caption,subfig}
\usepackage{color}
\usepackage{calrsfs}
\usepackage[dvipsnames]{xcolor}
\usepackage{tikz}
\usetikzlibrary{calc}

\usepackage{geometry}
\newcommand{\N}{{\mathbb N}}
\newcommand{\Z}{{\mathbb Z}}
\newcommand{\R}{{\mathbb R}}
\newcommand{\C}{{\mathbb C}}

\newcommand{\eps}{{\varepsilon}}
\newcommand{\bmu}{\boldsymbol{\mu}}
\newcommand{\bnu}{\boldsymbol{\nu}}
\newcommand{\bpsi}{\boldsymbol{\psi}}
\newcommand{\btheta}{\boldsymbol{\theta}}
\newcommand{\avbmu}{\langle \boldsymbol{\mu} \rangle}

\newcommand{\bA}{\mathbb{A}}

\newcommand{\bC}{\mathbb{C}}

\newcommand{\bF}{\mathbb{F}}
\newcommand{\bG}{\mathbb{G}}
\newcommand{\bH}{\mathbb{H}}
\newcommand{\bI}{\mathbb{I}}

\newcommand{\bN}{\mathbb{N}}

\newcommand{\bR}{\mathbb{R}}

\newcommand{\bT}{\mathbb{T}}
\newcommand{\bU}{\mathbb{U}}

\newcommand{\bX}{\mathbb{X}}

\newcommand{\bZ}{\mathbb{Z}}

\newcommand{\cA}{\mathcal{A}}
\newcommand{\cB}{\mathcal{B}}
\newcommand{\cC}{\mathcal{C}}
\newcommand{\cD}{\mathcal{D}}
\newcommand{\cE}{\mathcal{E}}
\newcommand{\cF}{\mathcal{F}}

\newcommand{\cH}{\mathcal{H}}
\newcommand{\cI}{\mathcal{I}}
\newcommand{\cJ}{\mathcal{J}}

\newcommand{\cL}{\mathcal{L}}
\newcommand{\cM}{\mathcal{M}}

\newcommand{\cP}{\mathcal{P}}
\newcommand{\cQ}{\mathcal{Q}}
\newcommand{\cR}{\mathcal{R}}

\newcommand{\cT}{\mathcal{T}}
\newcommand{\cU}{\mathcal{U}}

\newcommand{\cX}{\mathcal{X}}

\newcommand{\hatU}{\widehat{U}}
\newcommand{\hatu}{\widehat{u}}
\newcommand{\hatH}{\widehat{H}}
\newcommand{\hatq}{\widehat{q}}
\newcommand{\hatF}{\widehat{F}}
\newcommand{\hatG}{\widehat{G}}
\newcommand{\hatpsi}{\widehat{\psi}}

\newcommand{\uu}{\underline{u}}

\newcommand{\uH}{\underline{H}}

\newcommand{\uq}{\underline{q}}
\newcommand{\uU}{\underline{U}}

\newcommand{\uS}{\underline{S}}

\newcommand{\uF}{\underline{F}}
\newcommand{\uG}{\underline{G}}

\newcommand{\ub}{\underline{b}}

\newcommand{\uJ}{\underline{J}}

\newcommand{\ds}{\displaystyle}
\newcommand{\sg}{\cdot \nabla}
\newcommand{\dv}{\nabla \cdot}

\newcommand{\p}{\partial}

\newcommand{\sbt}{\ \begin{picture}(-1,1)(-1,-3)\circle*{2}\end{picture}\ \, }

\newtheorem{theorem}{Theorem}[chapter]
\newtheorem{proposition}[theorem]{Proposition}
\newtheorem{corollary}[theorem]{Corollary}
\newtheorem{lemma}[theorem]{Lemma}
\newtheorem{definition}[theorem]{Definition}

\begin{document}

\title{Weakly nonlinear surface waves\\ in magnetohydrodynamics}

\author{Olivier {\sc Pierre}\thanks{Laboratoire de math\'ematiques Jean Leray - UMR CNRS 6629, 
Universit\'e de Nantes, 2 rue de la Houssini\`ere, BP 92208, 44322 Nantes Cedex 3, France. 
Email: {\tt olivier.pierre@univ-nantes.fr}}
\& Jean-Fran\c{c}ois {\sc Coulombel}\thanks{Institut de Math\'ematiques de Toulouse - UMR 5219, 
Universit\'e de Toulouse ; CNRS, Universit\'e Paul Sabatier, 118 route de Narbonne, 31062 Toulouse Cedex 9 , France. 
Research of J.-F. C. was supported by ANR project BoND, ANR-13-BS01-0009, and ANR project Nabuco, ANR-17-CE40-0025. 
Email: {\tt jean-francois.coulombel@math.univ-toulouse.fr}}}

\maketitle

\begin{abstract}
This work is devoted to the construction of weakly nonlinear, highly oscillating, current vortex sheet solutions 
to the incompressible magnetohydrodynamics equations. Current vortex sheets are piecewise smooth solutions 
to the incompressible magnetohydrodynamics equations that satisfy suitable jump conditions for the velocity and 
magnetic field on the (free) discontinuity surface. In this work, we complete an earlier work by Al\`i and Hunter 
[Quart. Appl. Math. 61(3), 451-474, 2003] and construct approximate solutions at any arbitrarily large order 
of accuracy to the free boundary problem in three space dimensions when the initial discontinuity displays high 
frequency oscillations. As evidenced in earlier works, high frequency oscillations of the current vortex sheet give 
rise to `surface waves' on either side of the sheet. Such waves decay exponentially in the normal direction to the 
current vortex sheet and, in the weakly nonlinear regime that we consider here, their leading amplitude is governed 
by a nonlocal Hamilton-Jacobi type equation known as the `HIZ equation' (standing for Hamilton-Il'insky-Zabolotskaya 
[J. Acoust. Soc. Am. 97(2), 891-897, 1995]) in the context of Rayleigh waves in elastodynamics.
\bigskip

The main achievement of our work is to develop a systematic approach for constructing arbitrarily many correctors to 
the leading amplitude. Based on a suitable duality formula, we exhibit necessary and sufficient solvability conditions 
for the corrector equations that need to be solved iteratively. The verification of these solvability conditions is based 
on a combination of mere algebra and arguments of combinatorial analysis. The construction of arbitrarily many 
correctors enables us to produce infinitely accurate approximate solutions to the free boundary problem. Eventually, 
we show that the rectification phenomenon exhibited by Marcou in the context of Rayleigh waves [C. R. Math. Acad. 
Sci. Paris 349(23-24), 1239-1244, 2011] does not arise in the same way for the current vortex sheet problem.
\end{abstract}

\tableofcontents
\newpage

\chapter*{Notations}

\noindent {\bf The variables}

$t \quad$ the time variable

$x_1,x_2,x_3 \quad$ the original space variables

$y_1,y_2,y_3 \quad$ the straightened space variables

$x' \, = \, (x_1,x_2) \quad$ the original tangential space variable

$y' \, = \, (y_1,y_2) \quad$ the straightened tangential space variable

$\theta \quad$ the fast tangential variable

$Y_3 \quad$ the fast normal variable
\bigskip

\noindent {\bf The frequencies}

$\eps \quad$ the small wavelength for the oscillating problem

$\tau \quad$ the time frequency

$\xi_1,\xi_2,\xi_3 \quad$ the space frequencies

$k \quad$ the fast tangential frequency (associated with $\theta$)
\bigskip

\noindent {\bf The indices}

$\alpha \quad$ an index in $\{ 1,2,3 \}$

$j \quad$ a tangential index in $\{ 1,2 \}$

$j' \quad$ a tangential index in $\{ 1,2 \}$

$m,\mu,\ell \quad$ nonnegative integers for the WKB cascade
\bigskip

\noindent {\bf The domains}

$[0,T] \quad$ the time interval

$\Omega_\eps^\pm(t) \quad$ the original spatial domains

$\Gamma_\eps(t) \quad$ the oscillating free discontinuity

$\Omega_0^\pm \quad$ the straightened spatial domains

$\Gamma_0 \quad$ the straightened discontinuity

$\Gamma^\pm \quad$ the top and bottom boundaries

$I^+$, resp. $I^- \quad$ the interval $(0,1)$, resp. $(-1,0)$

$\bT \quad$ the torus $\bR /(2\, \pi \, \bZ)$

$\bT^2 \quad$ the two-dimensional torus $(\bR /(2\, \pi \, \bZ))^2$
\bigskip

\noindent {\bf The unknowns}

$u \quad$ the velocity field (a three dimensional vector)

$H \quad$ the magnetic field (a three dimensional vector)

$p \quad$ the pressure (a scalar quantity)

$q \quad$ the total pressure (a scalar quantity)

$U \quad$ the vector of unknowns $(u,H,q)^T \in \bR^7$

$U^\pm_\eps \quad$ the exact solution to the oscillatory problem (on either side of the current vortex sheet)

$\psi \quad$ the front (a scalar quantity)

$\psi_\eps \quad$ the exact oscillating front
\bigskip

\noindent {\bf The profiles}

$U^{\, m,\pm} \quad$ the $m$-th profile in the WKB expansion of the exact solution $U^\pm_\eps$

$\psi^m \quad$ the $m$-th profile in the WKB expansion of the exact front $\psi_\eps$

$\chi^{[\ell]},\dot{\chi}^{[\ell]} \quad$ profiles arising when straightening the original spatial domains
\bigskip

\noindent {\bf The operators}

$\times \quad $ the cross product in $\bR^3$

$\nabla \quad$ the gradient (with respect to $x \, = \, (x_1,x_2,x_3)$ or $y \, = \, (y_1,y_2,y_3)$, unless otherwise specified)

$\nabla \cdot \quad$ the divergence (with respect to $x \, = \, (x_1,x_2,x_3)$ or $y \, = \, (y_1,y_2,y_3)$)

$\nabla \times \quad$ the curl (with respect to $x \, = \, (x_1,x_2,x_3)$ or $y \, = \, (y_1,y_2,y_3)$)

$\cL_f^\pm(\p) \quad$ the fast operators

$L_s^\pm(\p) \quad$ the slow operators
\bigskip

\noindent {\bf Miscellanea}

$X^T \quad$ the transpose of a matrix (or vector) $X$
 
$\bmu,\bnu \quad$ sequences of nonnegative integers (with finitely many possible nonzero entries)

$|\bmu| \quad$ the length of the sequence $\bmu$

$\avbmu \quad$ the weight of the sequence $\bmu$

$\cdot \quad$ the Euclidean product between two real vectors

$\sbt \quad$ the Hermitian product between two complex vectors

$\cM_{n_1,n_2}({\mathbb K}) \quad$ the set of $n_1 \times n_2$ matrices with entries in the field ${\mathbb K}$

$f_\alpha \quad$ fluxes in the conservative form of the MHD equations

$A_\alpha^\pm \quad$ Jacobian matrices of the fluxes

$\bA_\alpha \quad$ Hessian matrices of the fluxes

$\widehat{U}(k) \quad$ the $k$-th Fourier mode of the profile $U$ with respect to the fast variable $\theta$

${\bf c}_0 \quad$ the zero Fourier mode with respect to the fast variable $\theta$
\newpage

\chapter{Introduction and main result}
\label{chapter1}

\section{Motivation}

This work is devoted to the asymptotic analysis of a free boundary problem arising in magnetohydrodynamics 
(MHD), namely the current vortex sheet problem. We consider a homogeneous, perfectly conducting, inviscid 
and incompressible plasma. The model consists of the so-called ideal incompressible MHD system, which reads 
in nondimensional form:
\begin{equation}
\label{int-equations_MHD}
\left\{
\begin{array}{l}
\p_t u + \dv (u\otimes u - H\otimes H) +\nabla q = 0 \, ,\\
\p_t H - \nabla \times (u \times H) = 0 \, ,\\
\dv u = \dv H = 0\, .
\end{array}
\right.
\end{equation}
In \eqref{int-equations_MHD}, $u \in \R^3$ and $H \in \R^3$ stand for the velocity and the magnetic field of 
the plasma respectively, $\times$ denotes the cross product in $\R^3$ and $\dv$, resp. $\nabla$, denotes the 
divergence, resp. gradient, operator with respect to the three-dimensional space variable $x=(x_1,x_2,x_3)$. 
The scalar unknown $q:=p +|H|^2/2$ is the `total' pressure, $p$ being the `physical' pressure.

We are interested here in a special class of (weak) solutions to \eqref{int-equations_MHD}: we want $(u,H,q)$ 
to be smooth, for each time $t$, on either side of a hypersurface $\Gamma(t) \subset \R^3$, and to give rise to a 
\emph{tangential} discontinuity across $\Gamma(t)$. The appropriate jump conditions on $\Gamma (t)$ are described 
below. For simplicity, we shall assume that the hypersurface $\Gamma(t)$ is a graph that can be parametrized by 
$\Gamma(t) = \{ x \in \R^3 \, / \, x_3 = \psi(t,x') \}$ for some smooth function $\psi$ of $(t,x')$ to be determined, 
with $x':=(x_1,x_2)$ the tangential space variable which we shall consider to be lying in the two-dimensional 
torus $\bT^2 :=(\R/(2\, \pi \, \Z))^2$. The unknown $\psi$ that parametrizes $\Gamma$ will be called the `front' 
of the discontinuity later on. We shall thus consider the incompressible MHD system \eqref{int-equations_MHD} 
in the time-dependent domain:
\begin{equation*}
\Omega(t) \, := \, \Omega^+(t) \sqcup \Omega^-(t), \quad \text{ where } \Omega^\pm(t) \, := \, \{ x_3 \gtrless \psi(t,x') \} \, ,
\end{equation*}
with the following jump conditions on $\Gamma(t)$:
\begin{equation}
\label{int-cond_bord}
\p_t \psi \, = \, u^+\cdot N \, = \, u^- \cdot N \, ,\quad H^+ \cdot N \, = \, H^- \cdot N \, = \, 0 \, ,\quad [ \, q \, ] = 0 \, .
\end{equation}
The notation $[ \, q \, ]$ in \eqref{int-cond_bord} stands for the jump of the total pressure $q$ across $\Gamma(t)$:
\begin{equation*}
[ \, q \, ] \, := \, \left.q^+\right|_{\Gamma(t)} - \left.q^-\right|_{\Gamma(t)} \, ,
\end{equation*}
and the notation $N$ in \eqref{int-cond_bord} stands for the normal vector to $\Gamma(t)$ chosen as follows:
\begin{align*}
N \, := \, (-\p_{x_1} \psi, -\p_{x_2} \psi, 1)^T \, .
\end{align*}

The boundary conditions \eqref{int-cond_bord} correspond to a \emph{tangential} discontinuity. The velocity $\p_t \psi$ 
of the front is given by the normal component of the fluid velocity on either side of the free discontinuity, meaning 
that the fluid does not flow through the interface $\Gamma(t)$. The normal magnetic field $H\cdot N$ is zero (hence 
continuous) on either side of the discontinuity, and the total pressure $q$ should also be continuous across $\Gamma 
(t)$. Such boundary conditions account for the evolution of a plasma which gives rise to a current vortex sheet (see, 
\emph{e.g.}, Figure \ref{s3-fig_nappe} below). Both $\nabla \times u$ and $\nabla \times H$ have a singular component 
on $\Gamma (t)$. We refer to \cite{Chandra,BT} for other types of discontinuities in compressible or incompressible MHD.

\begin{figure}[!h]
\begin{center}
\includegraphics[scale=0.4]{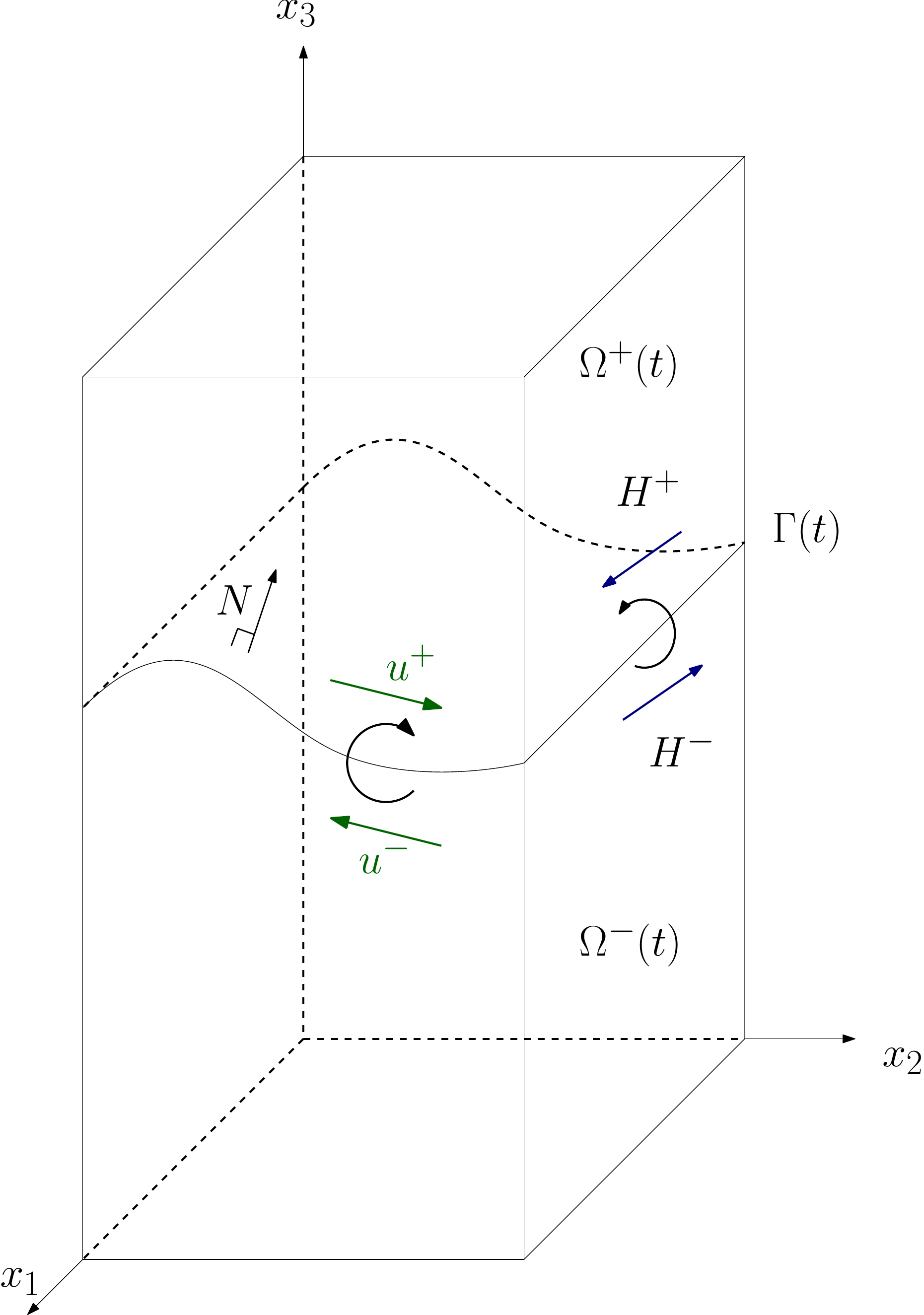}
\caption{Schematic picture of a current vortex sheet}
\label{s3-fig_nappe}
\end{center}
\end{figure}

To be consistent with several earlier works on current vortex sheets \cite{CMST,SWZ,Pierre}, we shall assume that 
the plasma is confined in the strip $x_3 \in (-1,1)$. In particular, the front $\psi$ should satisfy $-1<\psi(t,x')<1$ for all 
$(t,x')$ so that the current vortex sheet itself is located within the strip. We then impose the standard boundary conditions 
on the fixed `top' and `bottom' boundaries:
\begin{equation*}
\Gamma^\pm \, := \, \big\{ (x', \pm 1) \, , \, x' \in \bT^2 \big\} \, .
\end{equation*}
On $\Gamma^\pm$, the plasma should have zero normal velocity and zero normal magnetic field. In its quasilinear form, the 
system of current vortex sheets eventually reads as follows:
\begin{equation}
\label{s3-equations_nappes_MHD}
\left\{
\begin{array}{r l}
\p_t u^\pm \, +(u^\pm\sg) \, u^\pm \, -(H^\pm\sg) \, H^\pm \, +\nabla q^\pm \, = \, 0 \, , & 
\text{ in } \Omega^\pm(t) \, ,\quad t\in [0,T] \, ,\\[0.5ex]
\p_t H^\pm \, +(u^\pm\sg) \, H^\pm \, -(H^\pm\sg) \, u^\pm \, = \, 0 \, , & \text{ in } \Omega^\pm(t) \, ,\quad t\in [0,T] \, ,\\[0.5ex]
\dv u^\pm(t) \, =\dv H^\pm(t) \, = \, 0 \, , & \text{ in } \Omega^\pm(t) \, , \quad t\in [0,T] \, ,\\[0.5ex]
\p_t \psi \, = \, u^\pm \cdot N \, ,\quad H^\pm\cdot N \, = \, 0 \, ,\quad [ \, q \, ] \, = \, 0 \, , & 
\text{ on } \Gamma(t) \, ,\quad t\in [0,T] \, ,\\[0.5ex]
u_3^\pm \, = \, H_3^\pm \, = \, 0 \, , & \text{ on } [0,T] \times \Gamma^\pm \, .
\end{array}
\right.
\end{equation}
The superscript $\pm$ in \eqref{s3-equations_nappes_MHD} refers to the unknowns $u$, $H$ and $q$ restricted to 
the subdomains $\Omega^\pm(t)$. Of course, \eqref{s3-equations_nappes_MHD} should be supplemented with initial 
conditions for $\psi,u^\pm,H^\pm$ that satisfy suitable compatibility requirements (\emph{e.g.}, the divergence free 
constraints in \eqref{s3-equations_nappes_MHD} and the boundary conditions on $\Gamma^\pm$).

The local in time solvability of \eqref{s3-equations_nappes_MHD} in Sobolev spaces has been recently proved by 
Sun, Wang and Zhang \cite{SWZ} by a clever reduction to the free boundary in the spirit of water wave theory. This 
reduction yields a second order scalar hyperbolic equation, a simplified version of which will play a crucial role in the 
analysis below. The well-posedness result of \cite{SWZ} relies on a \emph{stability} condition that dates back, at least, 
to \cite{Syro,Axford1} and that also plays a crucial role in the present analysis. An alternative approach to \cite{SWZ}, 
which does not rely on any stability condition but that is restricted to analytic data, has been recently proposed by the 
first author \cite{Pierre} with the aim of using it also in the compressible case. Within this article, we are interested in 
the \emph{qualitative behavior} of exact solutions to \eqref{s3-equations_nappes_MHD} for \emph{highly oscillating} initial 
data. This problem has been first addressed by Al\`i and Hunter \cite{AliHunter} who have considered the two-dimensional 
problem and who have shown that for some specific oscillation phase, and in the \emph{weakly nonlinear} regime, the 
leading amplitude of the solution on either side of the current vortex sheet displays a \emph{surface wave} structure: it 
oscillates with the same phase as the front and it is localized near the free surface with exponential decay in the normal 
direction to the free surface. Such a phenomenon is entirely analogous to the description of Rayleigh waves in elastodynamics, 
see, \emph{e.g.}, \cite{Lardner1,Lardner2,Parker,ParkerTalbot,HIZ} and further references therein.

That surface waves occur in the current vortex sheet problem can be explained by performing a so-called normal mode 
analysis. Given a reference piecewise constant solution to the current vortex sheet system 
\eqref{s3-equations_nappes_MHD}, we seek for plane waves of the form $\exp\big( i \, (\tau \, t +\xi_1 \, x_1 +\xi_2 \, x_2 
+\xi_3 \, x_3) \big)$, with $\tau\in\bC$, $(\xi_1,\xi_2)\in\bR^2$ and $\xi_3\in\bC$, which can be solutions to the linearization 
of \eqref{s3-equations_nappes_MHD} at the given piecewise constant solution. Here, the normal coordinate to the (flat) sheet 
is denoted by $x_3$. Due to the divergence-free constraints on the velocity and the magnetic field, the resulting system is 
not a `standard' hyperbolic system. However, the method we use is analogous to the case of free boundary hyperbolic 
problems \cite{BS}: the goal is to verify whether the \emph{weak} and/or the \emph{uniform} Kreiss-Lopatinskii condition \cite{K} 
(ULC for short hereafter) is satisfied in order to obtain a linear stability criterion for planar current vortex sheets. The analysis 
of the linearized problem performed in \cite{Syro,Chandra,Axford1}, and more recently in \cite{MTT}, leads to a necessary 
stability criterion by eliminating the case $\text{Im} \, \tau < 0$ for which the normal modes blow up (the problem would then 
typically be \emph{strongly ill-posed} unless the data are analytic). The limit `neutral' case we are interested in corresponds 
to $\text{Im} \, \tau = 0$; we shall say that the problem is \emph{weakly well-posed}. Weak well-posedness is associated 
with frequencies $(\tau,\xi_1,\xi_2) \in\bR^3$ for which the so-called Lopatinskii determinant vanishes. For the current 
vortex sheet problem \eqref{s3-equations_nappes_MHD}, only the \emph{weak} Lopatinskii condition is fulfilled at best, 
meaning that there does not exist any planar current vortex sheet for which the ULC is satisfied. We refer for instance 
to \cite{MTT} and to Chapter \ref{chapter3} below for more details. The main (striking) result of \cite{SWZ} shows that the 
linear stability criterion that precludes violent instabilities is actually a \emph{sufficient} condition for \emph{nonlinear} 
stability in the Sobolev regularity scale.

In our problem, the roots $(\tau, \xi_1, \xi_2) \in \R^3$ of the Lopatinskii determinant can be parametrized by $(\xi_1,\xi_2)$. 
Namely, under the linear stability condition which we shall recall below, for any fixed tangential frequency $(\xi_1,\xi_2) \in \R^2 
\setminus \{ 0 \}$, there exist two \emph{simple} roots $\tau_\pm(\xi_1,\xi_2)$ of the Lopatinskii determinant, and these roots 
belong to the set of so-called \emph{elliptic frequencies} because the corresponding normal frequency $\xi_3$ is not real (it 
is even a purely imaginary number). Those frequencies are responsible for the creation of surface waves, which corresponds 
to the case $\text{Im} \, \xi_3 \gtrless 0$ (depending on the sign of $x_3$). The solution associated with such frequencies decays 
exponentially with respect to $x_3$. In other problems related to hydrodynamics, such as detonation waves or compressible 
vortex sheets \cite{MR,AM}, the normal frequency $\xi_3$ associated with the roots of the Lopatinskii determinant is real, which 
gives rise to \emph{bulk waves} that radiate into the whole domain, see \cite{BRSZ} for a general description of this class of 
problems. The latter case does not arise when studying current vortex sheets in incompressible MHD. At the opposite, the MHD 
problem we consider here is closer to the one studied by Sabl\'e-Tougeron \cite{S-T} whose prototype example is the system of 
elastodynamics with zero normal stress on the boundary (which gives rise to the so-called Rayleigh waves). Another occurrence 
of surface waves in MHD is the so-called plasma-vacuum interface problem \cite{SecchiTrak,Secchi}.

The main question we address here follows a long line of research, whose rigorous mathematical formulation dates back to Hunter 
\cite{Hunter1989}, see also \cite{AliHunterParker,AustriaHunter,BenzoniRosini,Marcou,CW,WW}, and is concerned with the evolution 
of \emph{weakly nonlinear} surface waves. Up to a time rescaling, we shall thus be concerned with the `slow' modulation of high 
frequency, small amplitude surface wave solutions to \eqref{s3-equations_nappes_MHD}. We follow the seminal work of Al\`i and 
Hunter \cite{AliHunter} with two main extensions; not only do we consider the three dimensional case to the price of some more 
algebra (the analysis in \cite{AliHunter} is performed in two space dimensions), but what is more significant is that we give a complete 
construction of infinitely accurate solutions to \eqref{s3-equations_nappes_MHD} (the analysis in \cite{AliHunter} is restricted more or 
less to the construction of the leading order amplitude). This is done by enlightening several algebraic properties in the analysis of the 
WKB cascade, some of which might be useful in other contexts. The construction of arbitrarily many correctors is not a mere technical 
issue. It is a crucial step towards the rigorous justification that \emph{exact} solutions to \eqref{s3-equations_nappes_MHD} with highly 
oscillating data are actually close to the WKB expansion we shall construct here, see, \emph{e.g.}, \cite{Gues,JMR1,Marcou,Rauch,WW}. 
However, we do not address this \emph{stability} issue here, because of intricate nonlocality issues \cite{SWZ}, and rather focus on the 
construction of a solution to the WKB cascade.

In the following Sections of this introduction, we state our main result by first stating the assumptions on the reference 
planar current vortex sheet and on the frequencies we shall work with. We then introduce the functional framework in 
which we shall solve the WKB cascade that will be made explicit in Chapter \ref{chapter2}. Eventually we state our main 
result and give the plan for its (slightly long) proof.

\section{Choice of parameters and initial data for the front}

Our goal is to construct highly oscillating solutions to \eqref{s3-equations_nappes_MHD} that are small perturbations 
of a reference \emph{piecewise constant} solution to \eqref{s3-equations_nappes_MHD}. The starting point is to fix 
the reference current vortex sheet. By imposing a suitable stability condition, inequality \eqref{s3-hyp_stab_nappe_plane} 
below, this will enable us to fix the planar phase of the oscillations for the front. The goal will then be to describe the 
behavior of the solution to \eqref{s3-equations_nappes_MHD} on either side of the oscillating front by choosing (and 
hopefully one day justifying) a suitable WKB ansatz. As a long term goal, this will justify the asymptotic behavior of the 
exact solution $(U_\eps^\pm,\psi_\eps)$ to \eqref{s3-equations_nappes_MHD} when we impose highly oscillating initial 
data, that is, displaying oscillations at frequencies $\sim \eps^{-1}$, $\eps \ll 1$. (Anticipating a little the notation described 
below, we have collected here all unknowns $(u^\pm,H^\pm,q^\pm)^T$ for \eqref{s3-equations_nappes_MHD} into a 
single vector $U^\pm$ with seven components.) In particular, part of our work aims at justifying that for suitably chosen 
oscillating initial data, the exact solution $(U_\eps^\pm,\psi_\eps)$ to \eqref{s3-equations_nappes_MHD} exists on a 
time interval $[0,T]$ that does not depend on the small wavelength $\eps>0$.

\subsubsection{The reference current vortex sheet}

To be consistent with the notation below for the WKB ansatz, we consider a (steady) piecewise constant solution to 
\eqref{s3-equations_nappes_MHD} of the form
\begin{equation}
\label{reference-current vortex}
U^\pm(t,x) \, = \, \begin{cases}
U^{0,+}\, ,& \text{\rm if } x_3 \in (0,1) \, ,\\
U^{0,-}\, ,& \text{\rm if } x_3 \in (-1,0) \, ,
\end{cases}
\end{equation}
where the two constant states $U^{0,\pm}$, and the corresponding fixed reference front\footnote{We use two functions $\psi^0$ and 
$\psi^1$ since the leading front should be rather thought of as $\psi^0+\eps \, \psi^1$. One possible extension of our work would be 
to study high frequency oscillations on a \emph{curved} current vortex sheet, and in that case, both $\psi^0$ and $\psi^1$ would be 
nontrivial.} $\psi^0,\psi^1$, are given by:
\begin{equation}
\label{s3-def_U^0,pm}
U^{0,\pm} \, := \, (u_1^{0,\pm},u_2^{0,\pm},0,H_1^{0,\pm},H_2^{0,\pm},0,0)^T \, ,\quad \quad \psi^0,\psi^1 \, := \, 0 \, .
\end{equation}
Let us recall that here and from on, the notation $U$ stands for a (column) vector in $\R^7$ whose 
coordinates are labeled $u_1,u_2,u_3,H_1,H_2,H_3,q$. The normalization of the total pressure $q^{0,\pm}=0$ in 
\eqref{s3-def_U^0,pm} is consistent with the choice that is made below for the solution to \eqref{s3-equations_nappes_MHD}, 
namely\footnote{Recall the jump condition $[ \, q \, ]=0$ in \eqref{s3-equations_nappes_MHD} across the interface $\Gamma(t)$, 
so the total pressure in $\Omega^+(t) \cup \Omega^-(t)$ is defined up to a function of time only.}:
$$
\int_{\Omega^+(t)} q^+(t,x) \, {\rm d}x  \, + \, \int_{\Omega^-(t)} q^-(t,x) \, {\rm d}x \, = \, 0 \, .
$$
For later use, we assume that the reference current vortex sheet \eqref{s3-def_U^0,pm} fulfills the following stability 
criterion:
\begin{equation}\tag{H1}
\label{s3-hyp_stab_nappe_plane}
\left| H^{0,+} \times [u^0] \right|^2 \, + \, \left| H^{0,-} \times [u^0] \right|^2 \, < \, 2 \, \left| H^{0,+} \times H^{0,-} \right|^2 \, ,
\end{equation}
where $[u^0] := u^{0,+} - u^{0,-}$ stands for the jump of the velocity $u^0$ across the (flat) sheet $\{x_3 \, = \, \psi^0 (t,x') 
+\eps \, \psi^1(t,x') \, = \, 0 \}$, and $\times$ denotes the cross product in $\R^3$. (Note however that the third coordinate 
of all three vectors $H^{0,\pm}$ and $[u^0]$ is zero hence \eqref{s3-hyp_stab_nappe_plane} involves two dimensional 
vectors only.) The stability condition \eqref{s3-hyp_stab_nappe_plane} has been highlighted in \cite{Syro,Chandra,Axford1} 
and more recently in \cite{MTT,SWZ} and is further discussed in Chapter \ref{chapter3} below. (A restricted version of 
\eqref{s3-hyp_stab_nappe_plane} is used in \cite{Trakhinin,CMST}.)

\subsubsection{The frequencies}

Let us begin with a few notations. The $2 \, \pi-$periodic torus $\R/(2\, \pi \, \Z)$ is denoted by $\bT$ and the tangential 
spatial variable is $x'=(x_1,x_2) \in \bT^2$. We consider a given tangential frequency vector $\xi' =(\xi_1,\xi_2) \in \R^2$ 
which we normalize by assuming $|\xi'| :=\sqrt{\xi_1^2 + \xi_2^2} =1$. We also choose a (real) time frequency $\tau$ 
which will be assumed to meet several requirements below, but let us right away define the planar phase $\tau \, t +\xi' 
\cdot x'$, the notation $\, \cdot \,$ here referring to the inner product of $\bR^2$. With the reference planar current 
vortex sheet defined by \eqref{s3-def_U^0,pm}, we define the following parameters:
\begin{equation}
\label{s3-def_a,b,c}
a^\pm \, := \, \xi_1 \, u_1^{0,\pm} +\xi_2 \, u_2^{0,\pm} \, ,\quad b^\pm \, := \, \xi_1 \, H_1^{0,\pm} +\xi_2 \, H_2^{0,\pm} 
\, ,\quad c^\pm \, := \, \tau + a^\pm \, .
\end{equation}

Given a $2 \, \pi-$periodic function $v=v(x',\theta)$ with respect to each of its arguments $(x',\theta) \in \bT^3$, we shall 
require below that functions of the form:
$$
v_\eps(x') \, := \, v \left( x',\dfrac{\xi'\cdot x'}{\eps} \right) \, ,
$$
be $2 \, \pi-$periodic with respect to $x'$. To do so, we need to impose some additional conditions on the frequency 
vector $\xi'$. We choose the frequency $\xi'$ of the form:
\begin{equation}\tag{H2}
\label{s3-hyp_xi'_rationnelle}
\xi' \, = \, \dfrac{1}{\sqrt{p^2 + q^2}} \, (p,q)^T \quad \text{ with } (p,q)^T \in \Z^2 \backslash \{ 0 \} \, .
\end{equation}
Then, considering the sequence $(\eps_\ell)_{\ell \ge 1}$ defined by:
\begin{equation}
\label{s3-def_eps_j}
\eps_\ell \, := \, \dfrac{1}{\ell \, \sqrt{p^2+q^2}} \, , \quad \ell \ge 1 \, ,
\end{equation}
which tends to $0$ as $\ell$ goes to $+\infty$, we will indeed have $\xi'/\eps_\ell \in \Z^2$ and therefore the above 
function $v_{\eps_\ell}$ will be $2 \, \pi-$periodic with respect to $x'$ for any integer $\ell \ge 1$. In the following, 
the frequency vector $\xi'$ is chosen of the form \eqref{s3-hyp_xi'_rationnelle} and the (small) parameter $\eps$ 
stands for one element of the sequence $(\eps_\ell)_{\ell \ge 1}$ in \eqref{s3-def_eps_j}. When we write $\eps \to 0$, 
we mean that we consider $\eps_\ell$ with $\ell \to +\infty$.

We add another assumption on the frequency $\xi'$ in order to fulfill the technical condition $(c^\pm)^2 \neq (b^\pm)^2$ 
used in Appendix \ref{appendixA}, see in particular the proof of Theorem \ref{theorem_fast_problem} hereafter in Chapter 
\ref{chapter3}:
\begin{equation}\tag{H3}
\label{s3-hyp_delta_neq_0}
\left\{
\begin{array}{l}
\big| a^+ - a^- \big| \, \neq \, \big| b^+ - b^- \big| \, , \\[0.5ex]
\big| a^+ - a^- \big| \, \neq \, \big| b^+ + b^- \big| \, . \\
\end{array}
\right.
\end{equation}
In other words, if we define the following three vectors in $\R^2$:
\begin{equation*}
\mathfrak{u} \, := \, \big( u_1^{0,+} - u_1^{0,-} , u_2^{0,+} - u_2^{0,-} \big) \, ,\quad 
\mathfrak{h}_\pm \, := \, \big( H_1^{0,+} \pm H_1^{0,-} , H_2^{0,+} \pm H_2^{0,-} \big) \, ,
\end{equation*}
then we ask the frequency $\xi'$ not to be orthogonal to the four vectors $\mathfrak{u}-\mathfrak{h}_+$, $\mathfrak{u}-\mathfrak{h}_-$, 
$\mathfrak{u} +\mathfrak{h}_+$, $\mathfrak{u} +\mathfrak{h}_-$. None of these four vectors is zero because of Assumption 
\eqref{s3-hyp_stab_nappe_plane}. Indeed, if we have for instance $\mathfrak{u}=\mathfrak{h}_-$, then it would lead to the 
identity $[u^0] =[H^0]$, and plugging this equality into \eqref{s3-hyp_stab_nappe_plane}, we would obtain:
$$
2 \, \big| H^{0,+} \times H^{0,-} \big|^2 \, < \, 2 \, \big| H^{0,+} \times H^{0,-} \big|^2 \, ,
$$
which is a contradiction. The same argument applies for the three remaining cases. Choosing $\xi'$ of the form 
\eqref{s3-hyp_xi'_rationnelle} and satisfying \eqref{s3-hyp_delta_neq_0} is possible because satisfying 
\eqref{s3-hyp_delta_neq_0} amounts to excluding at most four directions on the unit circle $\mathbb{S}^1$ and 
unit vectors of the form \eqref{s3-hyp_xi'_rationnelle} are dense in $\mathbb{S}^1$.

Given $\xi'\in\bR^2$ satisfying \eqref{s3-hyp_xi'_rationnelle} and \eqref{s3-hyp_delta_neq_0}, it remains to make 
the restrictions on the time frequency $\tau$ explicit. In all what follows, we choose the time frequency $\tau \in \bR$ 
as one (among the two) root(s) of the so-called Lopatinskii determinant defined by equation \eqref{s3-def_det_lop} 
herafter. Anticipating on the computation of the latter quantity in Chapter \ref{chapter3}, we choose $\tau$ as a root 
to the following polynomial equation of degree $2$ (recall the definition \eqref{s3-def_a,b,c}):
\begin{equation}\tag{H4}
\label{s3-hyp_tau_racine_det_lop}
(c^+)^2 +(c^-)^2 \, = \, (b^+)^2 + (b^-)^2 \, .
\end{equation}
That Assumption \eqref{s3-hyp_stab_nappe_plane} on the reference planar current vortex sheet implies that 
\eqref{s3-hyp_tau_racine_det_lop} has two \emph{real} roots follows from elementary algebraic considerations 
which we shall recall in Chapter \ref{chapter3} for the sake of completeness.

The last requirement on the time frequency is the assumption $\tau \neq 0$. This property is also used in Appendix 
\ref{appendixA} to parametrize some eigenspaces. The condition $\tau \neq 0$ automatically follows from 
\eqref{s3-hyp_tau_racine_det_lop} if $u_0^++u_0^-=0$ (in that case $a^+=-a^-$), which can always be achieved by 
using the Galilean invariance of system \eqref{s3-equations_nappes_MHD}.

Let us focus on the fact that assumptions \eqref{s3-hyp_xi'_rationnelle}, \eqref{s3-hyp_delta_neq_0}, 
\eqref{s3-hyp_tau_racine_det_lop} and $\tau \neq 0$ allow to ensure the condition $(c^\pm)^2 \neq (b^\pm)^2$, 
see Appendix \ref{appendixA}, which will turn out to be crucial in the analysis of the WKB cascade.
\bigskip

From now on, the reference planar current vortex sheet and the frequencies $(\tau,\xi_1,\xi_2) \in \R^3$ satisfying 
Assumptions \eqref{s3-hyp_stab_nappe_plane}, \eqref{s3-hyp_xi'_rationnelle}, \eqref{s3-hyp_delta_neq_0}, 
\eqref{s3-hyp_tau_racine_det_lop}, together with $\tau\neq 0$, are {\bf fixed}. We now describe the oscillating data 
that we consider for \eqref{s3-equations_nappes_MHD}.

\subsubsection{Initial data for the front and WKB ansatz}

We consider small, highly oscillating perturbations of the reference constant state $(U^{0,\pm},\psi^0 +\eps \, \psi^1)$. 
To be specific, we shall consider initial data for the front of the form:
\begin{equation}
\label{s3-def_cond_init_oscil_psi}
\psi_\eps(t=0,x') \, := \, \eps^2 \, \psi_0^2 \left( x',\dfrac{\xi'\cdot x'}{\eps} \right) \, , 
\end{equation}
where the initial \emph{profile} $\psi_0^2 \in \cC^\infty (\bT^3)$ is assumed to have zero mean with respect to its last 
argument $\theta\in\bT$. Let us recall that the small parameter $\eps \in (0,1]$ actually stands for any $\eps_\ell$ defined 
by \eqref{s3-def_eps_j} so that $\psi_\eps$ in \eqref{s3-def_cond_init_oscil_psi} is indeed $2\, \pi$-periodic with respect 
to $x'$. We could consider a sequence of profiles $\psi_0^2, \psi_0^3, \psi_0^4, \dots$ and the corresponding initial 
datum:
\begin{equation*}
\psi_\eps(t=0,x') \, = \, \eps^2 \, \psi_0^2 \left( x',\dfrac{\xi'\cdot x'}{\eps} \right) 
+\eps^3 \, \psi_0^3 \left( x',\dfrac{\xi'\cdot x'}{\eps} \right) +\cdots \, ,
\end{equation*}
the series in $\eps$ being either convergent or understood as an asymptotic expansion in $\eps$, but this would not add any 
new phenomenon nor any analytical difficulty; we therefore restrict to the initial datum \eqref{s3-def_cond_init_oscil_psi} for 
notational convenience. Choosing the initial \emph{profile} $\psi_0^2$ to have zero mean with respect to $\theta$ is also 
done for the sake of convenience (see Chapter \ref{chapter5}). In any case, the mean of $\psi_0^2$ with respect to $\theta$ 
does not affect the leading amplitude of the solution on either side of the current vortex sheet.

The initial front $\psi_\eps|_{t=0}$ will take its values in $(-1,1)$ up to restricting $\eps$ if necessary. 
Since $\eps$ is meant to be small, we shall not go back to this issue any longer.

Continuing the analysis of \cite{AliHunter}, we seek an asymptotic expansion of the exact solution $(U_\eps^\pm,\psi_\eps)$ 
to \eqref{s3-equations_nappes_MHD} as a small, highly oscillating perturbation of the reference planar current vortex sheet 
\eqref{s3-def_U^0,pm}. Some attention needs to be paid when formulating the WKB ansatz for $(U_\eps^\pm,\psi_\eps)$. 
The front $\psi_\eps$ is meant to oscillate with the planar phase $\tau \, t +\xi' \cdot x'$ with a slow modulation in the variables 
$(t,x')$. The interior solution $U_\eps^\pm$ will display oscillations with the same planar phase $\tau \, t +\xi'\cdot x'$, and 
exponential decay with respect to the fast normal variable $(x_3-\psi_\varepsilon(t,x'))/\eps$. However, describing the slow 
modulation of $U_\eps^\pm$ requires taking into account the slow normal variable $x_3-\psi_\varepsilon(t,x')$ and the fixed 
top and bottom boundaries $\Gamma^\pm$ too. We thus introduce once and for all a fixed cut-off function $\chi \in 
\cC^\infty(\R)$ such that $\chi \equiv 1$ on $[-1/3,1/3]$ and $\chi$ vanishes outside of  $[-2/3,2/3]$. We aim at constructing, 
and possibly justifying, an asymptotic expansion for $(U_\eps^\pm,\psi_\eps)$ of the following form:
\begin{subequations}
\label{s3-def_dvlpt}
\begin{align}
U_\eps^\pm (t,x) \, & \, \sim \, U^{0,\pm} +\sum_{m \ge 1} \eps^m \, U^{\, m,\pm}\left( t,x',x_3-\chi(x_3) \, \psi_\eps(t,x'), 
\dfrac{x_3 -\psi_\eps(t,x')}{\eps},\dfrac{\tau \, t +\xi'\cdot x'}{\eps} \right) \, ,\label{s3-def_dvlpt_BKW} \\
\psi_\eps (t,x') \, & \, \sim \, \psi^0 +\eps \, \psi^1 +\sum_{m \ge 2} \eps^m \, \psi^m \left( t,x',\dfrac{\tau \, t +\xi'\cdot x'}{\eps} \right) \, .
\label{s3-def_dvlpt_BKW_psi}
\end{align}
\end{subequations}
Of course, the two first terms $\psi^0,\psi^1$ on the right hand side of \eqref{s3-def_dvlpt_BKW_psi} are harmless but are 
placed here to highlight the consistency of our notation. By $\sim$, we mean in \eqref{s3-def_dvlpt} that the series should 
be understood in the sense of asymptotic expansions in $\eps$, see, \emph{e.g.}, \cite{Rauch}. We require the front $\psi_\eps$ 
to match with the function \eqref{s3-def_cond_init_oscil_psi} at $t=0$:
$$
\forall \, (x',\theta) \in \bT^3 \, ,\quad \psi^2(0,x',\theta) \, = \, \psi_0^2 (x',\theta) \, ,\quad \text{ and } \quad 
\forall \, m \ge 3 \, , \quad \psi^m(0,x',\theta) \, = \, 0 \, .
$$

In \eqref{s3-def_dvlpt_BKW}, the profiles $U^{\, m,\pm}$ are functions of 6 variables which we denote $(t,y',y_3,Y_3,\theta)$ 
from now on ($y'$ is two-dimensional). The \emph{slow} variables are $(t,y',y_3)$; $t \in [0,T]$ is the time variable, $y'=x' \in 
\bT^2$ is the tangential spatial variable, and $y_3 =x_3 -\chi (x_3) \, \psi_\varepsilon(t,x') \in (-1,1)$ is the normal variable 
which allows both to lift the free surface in \eqref{s3-equations_nappes_MHD} and to match with the top and bottom 
boundaries. The oscillating current vortex sheet $\Gamma_\eps(t) \, := \, \{ x_3 =\psi_\eps(t,x') \}$ in the original space 
variables corresponds to the fixed interface $\{ y_3=0 \}$ in the straightened variables, while the top and bottom boundaries 
$\Gamma^\pm$ correspond to $\{ y_3 =\pm 1\}$ (at least for any sufficiently small $\eps$). The \emph{fast} variables 
are $(Y_3,\theta)$: $Y_3 =(x_3-\psi_\eps(t,x'))/\eps \in \bR^\pm$ is the fast normal variable which will describe the 
exponential decay of the surface wave, and $\theta =(\tau \, t +\xi'\cdot x')/\eps \in \bT$ is the fast tangential variable 
which describes the oscillations. Let us observe that we do not incorporate the cut-off function $\chi$ in the fast normal 
variable $Y_3$ since exponential decay will yield $O(\eps^\infty)$ -hence negligible- terms outside of $\{ |x_3| \le \delta \}$ 
for any fixed constant $\delta>0$.

One of the main issues here will consist in constructing the profiles $(U^{\, m,\pm},\psi^{\, m+1})_{m \ge 1}$ in the WKB expansions 
\eqref{s3-def_dvlpt_BKW}, \eqref{s3-def_dvlpt_BKW_psi}. To do so, we shall need both the divergence-free constraints 
on the velocity $u_\eps^\pm$ and the magnetic field $H_\eps^\pm$. Although the condition $\dv H_\eps^\pm = 0$ is known to 
be propagated by the solutions of system \eqref{s3-equations_nappes_MHD}, see, \emph{e.g.}, \cite{Trakhinin,Trakhinin-comp,SWZ}, 
it is not clear that the associated constraints on the \emph{profiles} $(H^{\, m,\pm})_{m \ge 1}$ are propagated in time \emph{one 
by one} as well. This is one major algebraic obstacle that we have to tackle here, and it explains why we choose to keep the 
divergence-free constraint \eqref{s3-contrainte_div_H} on $H_\eps^\pm$ separate from the other equations in system 
\eqref{s3-MHD_conservative} below.

It is important to notice that the initial datum associated with $U_\eps^\pm$ \emph{is not free}, as is well known in geometric optics 
because of \emph{polarization}, see \cite{Rauch}. Actually, it turns out that part of the profiles $U^{\, m,\pm}$ will be determined for any 
time $t\in [0,T]$ by solving algebraic equations. In particular, (part of) the initial data for $U_\eps^\pm$ will be computed alongside the 
whole approximate solution. This restricts the choice of initial data for $U_\eps^\pm$; nevertheless the choice of the initial profile 
$\psi_0^2$ for the front is \emph{free}. There are even more degrees of freedom that are clarified later on. However, for simplicity, 
we focus here mainly on the choice of the initial condition $\psi_\eps|_{t=0}$.

The scaling \eqref{s3-def_dvlpt_BKW}, \eqref{s3-def_dvlpt_BKW_psi} we choose here is analogous to the scaling of weakly 
nonlinear geometric optics for the Cauchy problem that can be found in \cite{Gues,JMR1,JMR2}, \cite{Hunter1989,Marcou} 
(for surface waves in a fixed half-space) or \cite{Lescarret}. Discarding the two first zero terms $\psi^0,\psi^1$, the expansion 
\eqref{s3-def_dvlpt_BKW_psi} of $\psi_\eps$ starts with an $O(\eps^2)$ amplitude, since the \emph{gradient} of $\psi_\eps$ 
in \eqref{s3-equations_nappes_MHD} has the same regularity as the trace of $U_\eps^\pm$ on $\Gamma(t)$, see 
\cite{CMST,SWZ}. We can also notice the similarity with uniformly stable shocks studied by Williams \cite{Williams99}, where 
the profile $\psi^1$ (which is zero in our case) does not depend on the fast variables. Because of the difference of one power 
of $\eps$ in \eqref{s3-def_dvlpt_BKW_psi} with respect to \eqref{s3-def_dvlpt_BKW}, the functions $U_\eps^\pm-U^{0,\pm}$ 
and $\nabla_{t,x'} \psi_\eps$ have amplitude $\eps$ in $L^\infty$ and oscillate with frequency $\sim \eps^{-1}$. In what follows, 
we usually study the profiles $U^{\, m,\pm}$ jointly with $\psi^{\, m+1}$. In particular, we shall refer to $(U^{\, 1,\pm},\psi^2)$ as the 
leading amplitude in the WKB expansion \eqref{s3-def_dvlpt}.

\section{The functional framework}

The functional framework we are going to define is inspired from Marcou \cite{Marcou} and Lescarret \cite{Lescarret} 
but we incorporate here some new ingredients. The final time $T >0$ below will be fixed once and for all by Theorem 
\ref{s3-thm_hunter} hereafter and it will only depend on a \emph{fixed} Sobolev norm of the initial profile $\psi_0^2$ in 
\eqref{s3-def_cond_init_oscil_psi} (to be precise, the $H^4$ norm does the job). The spaces of profiles for the WKB ansatz 
\eqref{s3-def_dvlpt_BKW} are defined as follows.

\begin{definition}[Spaces of profiles]
\label{s3-def_espaces_fonctionnels}
The space $\uS^\pm$ denotes the set of functions in $H^\infty \big( [0,T] \times \bT^2 \times I^\pm \times \bT \big)$, 
where $I^+$ (resp. $I^-$) stands for the interval $(0,1)$ (resp. $(-1,0)$). Functions in $\uS^\pm$ depend on the slow 
variables $(t,y)$ and on the fast tangential variable $\theta$.
\bigskip

The space $S_\star^\pm$ denotes the set of functions in $H^\infty \big( [0,T] \times \bT^2 \times I^\pm \times \bR^\pm 
\times \bT \big)$ that decay exponentially as $Y_3 \to \pm \infty$ as well as all their derivatives uniformly with respect 
to all other arguments:
\begin{equation*}
\exists \, \delta >0 \, ,\quad \forall \, \alpha \in \bN^6 \, ,\quad \exists \, C_\alpha > 0 \, ,\quad 
\forall \, Y_3 \gtrless 0, \quad \left\| \, \p^\alpha u_\star^\pm (\cdot,\cdot,Y_3,\cdot) \, \right\|_{L_{t,y,\theta}^\infty} \, 
\le \, C_\alpha \, \mathrm{e}^{\mp \, \delta \, Y_3} \, .
\end{equation*}
Functions in $S_\star^\pm$ depend on both the slow variables $(t,y)$ and the fast variables $(Y_3,\theta)$.
\bigskip

The space of profiles is $S^\pm :=\uS^\pm \oplus S_\star^\pm$ (the sum is direct because functions in $S_\star^\pm$ 
decay exponentially with respect to $Y_3$ and functions in $\uS^\pm$ do not depend on $Y_3$).
\end{definition}

The profiles $U^{\, m,\pm}$, for $m \ge 1$, will be sought in the functional space $S^\pm$. The profiles $\psi^m$, $m \ge 2$, 
in \eqref{s3-def_dvlpt_BKW_psi} will be sought in the functional space $H^\infty ([0,T] \times \bT^3)$. The component on 
$\uS^\pm$ of some $U^\pm \in S^\pm$ is called the \emph{residual} component while the component on $S_\star^\pm$ 
is called the \emph{surface wave} component. Though we are mainly interested in the component on $S_\star^\pm$ of the 
leading amplitude in \eqref{s3-def_dvlpt}, determining the residual components of the correctors is one major obstacle in the 
analysis below. It seems likely that the WKB cascade below can not be solved with profiles $U^{\, m,\pm} \in S_\star^\pm$ 
for all $m \ge 1$. Namely, though the leading profile $U^{\, 1,\pm}$ will belong to $S_\star^\pm$, it is likely that one corrector 
$U^{\, m,\pm}$ will have a nontrivial residual component, which corresponds to a \emph{rectification} phenomenon. Such 
a phenomenon has been rigorously justified by Marcou \cite{MarcouCRAS} for a two-dimensional model of elasticity. In 
\cite{MarcouCRAS}, it is shown that the first corrector has a nontrivial residual component. This will not be the case here 
because the leading profile exhibits interesting orthogonality properties which will imply that the first corrector $U^{\, 2,\pm}$ 
will also belong to $S_\star^\pm$. We have not been able to push further the calculations, but it is likely though that the 
second corrector $U^{\, 3,\pm}$ has a nontrivial residual component. This will be explained in Chapter \ref{chapter6}.
\bigskip

Let us observe that in \cite{Marcou}, functions in $\uS^\pm$ are chosen not to depend on the fast tangential variable $\theta$ 
(the same in \cite{WW}). It does not seem possible to use this framework here due to the form of the source terms in the WKB 
cascade below. Our source terms differ from those in \cite{Marcou,WW} because we deal here with a free boundary problem 
and we consider additional fixed top and bottom boundaries. We believe that our extension of the functional framework might 
be useful in other contexts which also give rise to surface waves on free discontinuities.

Both $\uS^\pm$ and $S^\pm$ are algebras and are stable under differentiation with respect to any of the arguments.

\subsubsection{Notation for profiles}

We shall expand profiles $U^\pm \in S^\pm$ into Fourier series in the fast tangential variable $\theta$. Given $k \in \Z$, the $k$-th 
Fourier coefficient with respect to $\theta$ is denoted $\widehat{U}^\pm (k)$, that is we use the decomposition:
\begin{equation*}
U^\pm (t,y,Y_3,\theta) \, = \, \sum_{k \in \bZ} \, \widehat{U}^\pm (t,y,Y_3,k) \, \mathrm{e}^{i\, k\, \theta} \, .
\end{equation*}
Taking the definition of $S^\pm$ into account, we can also split $U^\pm$ as follows:
\begin{equation}
\label{s3-def_decomp_S^pm}
U^\pm (t,y,Y_3,\theta) \, = \, \underbrace{\uU^\pm(t,y,\theta)}_{\textstyle \in \uS^\pm} 
\, + \, \underbrace{U_\star^\pm (t,y,Y_3,\theta)}_{\textstyle \in S_\star^\pm} \, .
\end{equation}
The zero Fourier mode plays a special role in the analysis of the WKB cascade, as opposed to the nonzero Fourier modes. 
Consistently with \eqref{s3-def_decomp_S^pm}, we split:
\begin{equation*}
\widehat{U}^\pm (0) \, = \, \widehat{\uU}^\pm (0) \, + \, \widehat{U}_\star^\pm (0) \, ,
\end{equation*}
the first term $\widehat{\uU}^\pm (0)$ being referred to as the \emph{slow mean}, and the second term $\widehat{U}_\star^\pm 
(0)$ being referred to as the \emph{fast mean}. Later on, we shall need to further split the fast mean $\widehat{U}_\star^\pm (0)$ 
as follows:
\begin{equation*}
\widehat{U}_\star^\pm (0) \, =\Pi \,\widehat{U}_\star^\pm (0) \, +(I-\Pi) \, \widehat{U}_\star^\pm (0) \, ,
\end{equation*}
where $\Pi := \text{diag}(1,1,0,1,1,0,0) \in \cM_7(\bR)$ is a projector onto the kernel of the Jacobian matrix $A_3^\pm$ defined in 
\eqref{s3-def_A_j_cA} below (the projector does not depend on the state $\pm$ so we omit the superscript here). In other words, 
for $U^\pm =(u^\pm,H^\pm,q^\pm)^T$, the vector:
\begin{equation*}
\Pi \, U^\pm \, =(u_1^\pm,u_2^\pm,0,H_1^\pm,H_2^\pm,0,0)^T \, ,
\end{equation*}
consists in the \emph{tangential} components associated with the velocity $u^\pm$ and the magnetic field $H^\pm$. The vector 
$(I-\Pi) \, U^\pm$ gathers the \emph{noncharacteristic} components of $U^\pm$, which are the normal velocity, the normal 
magnetic field and the total pressure.

We shall see in Chapters \ref{chapter4} and \ref{chapter5} that we own several degrees of freedom for the initial data of the 
mean of the profiles $U^{\, m,\pm}$; for the sake of simplicity, we shall choose to impose zero initial conditions for the fast means 
$\Pi \, \widehat{U}_\star^{\, m,\pm}(0)$. We shall also impose zero initial conditions on the mean of the residual component of 
the leading amplitude $\widehat{\uU}^{\, 1,\pm}(0)$. This choice will allow us to simplify part of the construction of the profiles 
$(U^{\, m,\pm},\psi^{\, m+1})_{m \ge 1}$ and to focus on the `surface wave' component of the leading amplitude, \emph{i.e.} 
the component $U_\star^{\, 1,\pm} \in S_\star^\pm$.

\section{The main result}

The aim of this work is to show the existence of a sequence of profiles $(U^{\, m,\pm},\psi^{\, m+1})_{m \ge 1}$ such that in the sense 
of formal series, \eqref{s3-def_dvlpt_BKW} and \eqref{s3-def_dvlpt_BKW_psi} satisfy \eqref{s3-equations_nappes_MHD} with 
accuracy $O(\eps^\infty)$. A precise statement is the following Theorem.

\begin{theorem}
\label{thm_principal}
Let the reference current vortex sheet defined by \eqref{reference-current vortex}, \eqref{s3-def_U^0,pm} satisfy Assumption 
\eqref{s3-hyp_stab_nappe_plane} and let the frequencies $\tau,\xi'$ satisfy Assumptions \eqref{s3-hyp_xi'_rationnelle}, 
\eqref{s3-hyp_delta_neq_0}, \eqref{s3-hyp_tau_racine_det_lop} together with $\tau \neq 0$. Let also $\psi^2_0 \in H^\infty 
(\bT^2 \times \bT)$ have zero mean with respect to its last argument $\theta$. Then there exists a time $T>0$, that only depends 
on the norm $\| \psi^2_0 \|_{H^4 (\bT^2 \times \bT)}$ such that, with the spaces $S^\pm =\uS^\pm \oplus S_\star^\pm$ of Definition 
\ref{s3-def_espaces_fonctionnels} associated with this given time $T$, there exists a sequence of profiles $(U^{\, m,\pm},\psi^{\, m+1} 
)_{m \ge 1}$ in $S^\pm \times H^\infty ([0,T] \times \bT^2 \times \bT)$ verifying the following properties:
\begin{itemize}
 \item for any $m \ge 2$, $\psi^m|_{t=0} =\delta_{m \, 2} \, \psi^2_0$ (with $\delta$ the Kronecker symbol) and 
 $\p_t \widehat{\psi}^m(0)|_{t=0}=0$,
 \item $\uU^{\, 1,\pm}=0$, $\widehat{U}_\star^{\, 1,\pm}(0)=0$, $\uU^{\, 2,\pm}=0$,
 \item for any $m \ge 2$, $\Pi \, \widehat{U}_\star^{\, m,\pm}(0)|_{t=0} =0$,
 \item for all integer $M \ge 1$, the functions:
\begin{align*}
\psi_\eps^{{\rm app},M} (t,x') \, &:= \, \psi^0 +\eps \, \psi^1 +\sum_{m=2}^{\, m+1} \eps^m \, \psi^m \left( t,x',\dfrac{\tau \, t +\xi'\cdot x'}{\eps} \right) \, ,\\
U_\eps^{{\rm app},M,\pm}(t,x) \, &:= \, U^{0,\pm} +\sum_{m=1}^M \eps^m \, U^{\, m,\pm} \left( t,x',x_3-\chi(x_3) \, \psi_\eps^{{\rm app},M}, 
\dfrac{x_3 -\psi_\eps^{{\rm app},M}}{\eps},\dfrac{\tau \, t +\xi'\cdot x'}{\eps} \right) \, ,
\end{align*}
satisfy (with $N_\eps^{{\rm app},M} :=(-\p_{x_1} \psi_\eps^{{\rm app},M},-\p_{x_2} \psi_\eps^{{\rm app},M},1)^T$):
\begin{align*}
& \, \p_t u_\eps^{{\rm app},M,\pm} +(u_\eps^{{\rm app},M,\pm} \sg) u_\eps^{{\rm app},M,\pm} 
-(H_\eps^{{\rm app},M,\pm} \sg) H_\eps^{{\rm app},M,\pm} +\nabla q_\eps^{{\rm app},M,\pm} \, = \, R_\eps^{\, 1,\pm} \, ,\\
& \, \p_t H_\eps^{{\rm app},M,\pm} +(u_\eps^{{\rm app},M,\pm} \sg) H_\eps^{{\rm app},M,\pm} 
-(H_\eps^{{\rm app},M,\pm} \sg) u_\eps^{{\rm app},M,\pm} \, = \, R_\eps^{\, 2,\pm} \, ,\\
& \, \dv u_\eps^{{\rm app},M,\pm} \, = \, R_\eps^{\, 3,\pm} \, ,\quad \dv H_\eps^{{\rm app},M,\pm} \, = \, R_\eps^{4,\pm} \, ,\\
& \, \p_t \psi_\eps^{{\rm app},M} -u_\eps^{{\rm app},M,\pm}|_{\Gamma_\eps^{{\rm app},M}(t)} \cdot N_\eps^{{\rm app},M} \, = \, R_{b,\eps}^{\, 1,\pm} \, ,\\
& \, H_\eps^{{\rm app},M,\pm}|_{\Gamma_\eps^{{\rm app},M}(t)} \cdot N_\eps^{{\rm app},M} \, = \, R_{b,\eps}^{\, 2,\pm} \, ,\quad 
[ \, q_\eps^{{\rm app},M} \, ] \, =0 \, ,\\
& \, u_{\eps,3}^{{\rm app},M,\pm}|_{\Gamma^\pm} \, = \, R_{b,\eps}^{\, 3,\pm} \, ,\quad 
H_{\eps,3}^{{\rm app},M,\pm}|_{\Gamma^\pm} \, = \, R_{b,\eps}^{4,\pm} \, ,
\end{align*}
where the error terms satisfy the following bounds:
\begin{align*}
& \, \sup_{t \in [0,T] \, , \, x \in \Omega_\eps^{{\rm app},M,\pm}(t)} \, 
\big( |R_\eps^{\, 1,\pm}| +|R_\eps^{\, 2,\pm}| +|R_\eps^{\, 3,\pm}| +|R_\eps^{4,\pm}| \big) \, = \, O(\eps^M) \, , \\
& \, \sup_{t \in [0,T] \, , \, x \in \Gamma_\eps^{{\rm app},M}(t)} \, \big( |R_{b,\eps}^{\, 1,\pm}| +|R_{b,\eps}^{\, 2,\pm}| \big) \, = \, O(\eps^{\, m+1}) \, , \\
& \, \sup_{t \in [0,T] \, , \, x \in \Gamma^\pm} \, \big( |R_{b,\eps}^{\, 3,\pm}| +|R_{b,\eps}^{4,\pm}| \big) \, = \, O(\eps^\infty) \, ,
\end{align*}
where we have used the notation $\Omega_\eps^{{\rm app},M,\pm}(t) := \{ x \in \bT^2 \times (-1,1) \, / \, \pm (x_3 -\psi_\eps^{{\rm app},M}(t,x'))>0 \}$, 
and $\Gamma_\eps^{{\rm app},M}(t) := \{ x \, / \, x_3 = \psi_\eps^{{\rm app},M}(t,x') \}$.
\end{itemize}
\end{theorem}

In other words, we can produce approximate solutions to the original free boundary value problem \eqref{s3-equations_nappes_MHD} 
at any desired order of accuracy. By using a Borel summation procedure, one can also achieve infinitely accurate approximate solutions 
(meaning with all error terms being $O(\eps^\infty)$ in the appropriate $L^\infty$ norms).

It is likely that the methods we develop here may prove useful in other related problems of magnetohydrodynamics or other (free) 
boundary value problems for systems of partial differential equations that exhibit surface waves at the linearized level. One such 
example is the plasma vacuum interface problem studied in \cite{Secchi} that presents many similarities with the current vortex 
sheet problem we study here (in particular the evolution equation that governs the leading amplitude evolution has been proved 
in \cite{Secchi} to be the same as the one exhibited in \cite{AliHunter} for current vortex sheets).

The plan of this work is the following. In Chapter \ref{chapter2}, we exhibit the so-called WKB cascade that must be satisfied 
by the profiles $(U^{\, m,\pm},\psi^{\, m+1})$ in \eqref{s3-def_dvlpt} in order to get high order approximate solutions to the original 
equations \eqref{s3-equations_nappes_MHD}. As will be made clear in Chapter \ref{chapter2}, one central problem in the 
iterative construction of the profiles $(U^{\, m,\pm},\psi^{\, m+1})$ is the resolution of the so-called \emph{fast problem} (system 
\eqref{fast_problem} below). Therefore Chapter \ref{chapter3} is devoted to solving \eqref{fast_problem} and in particular 
to making clear the solvability conditions for the source terms in \eqref{fast_problem}. The construction of the profiles 
$(U^{\, m,\pm},\psi^{\, m+1})$ is done in Chapters \ref{chapter4} (for $m=1$, that is for the leading profile) and \ref{chapter5} 
(for $m \ge 2$, that is for the correctors). Chapter \ref{chapter4} is a kind of warm-up for the systematic construction of 
correctors in Chapter \ref{chapter5}. In the weakly nonlinear regime that we consider here, only the leading amplitude will 
satisfy nonlinear evolution equations, hence a separate treatment. The correctors will satisfy linearized versions of the nonlinear 
equations satisfied by the leading amplitude but with nonzero source terms. This is a standard feature of weakly nonlinear geometric 
expansions \cite{Rauch}. Chapter \ref{chapter6} is devoted to the analysis of the so-called rectification phenomenon. Opposite 
to the case of elastodynamics, we show here that the first corrector $U^{\, 2,\pm}$ has no residual component, provided of course 
that the initial data (that can be imposed) are suitably tuned to zero. As explained in Chapter \ref{chapter6}, it seems likely that 
the second corrector $U^{\, 3,\pm}$ always has (or at least generically has) a nontrivial residual component, though a complete 
verification of this fact has been left aside because the algebra involved was too heavy. At last, Appendix \ref{appendixA} gathers 
the expressions of several matrices, eigenvectors and bilinear operators that are involved in the calculations of Chapters 
\ref{chapter3}, \ref{chapter4} and \ref{chapter5}. Appendix \ref{appendixB} gathers what is probably the most original part 
of this work, which is the proof of several algebraic relations between the profiles involved in the WKB cascade. The results 
of Appendix \ref{appendixB} have been proved in a slightly more general framework than the one we consider here in order 
to be applicable to geometric optics problem on a curved background (as opposed to the constant reference solution 
\eqref{reference-current vortex} we consider here).

\chapter{The WKB cascade}
\label{chapter2}

Starting from here, we use the notation $\alpha \in \{ 1,2,3 \}$ to refer to a spatial coordinate and we use the notation $j \in \{ 1,2 \}$ to 
refer to a tangential spatial coordinate $y_j$. When several tangential coordinates are involved, we use both $j$ and $j'$. We also use 
Einstein summation convention on repeated indices.

For practical reasons, we rewrite both evolution equations of \eqref{s3-equations_nappes_MHD} for the velocity and magnetic field together 
with the divergence-free constraint on the velocity\footnote{We remind that this constraint allows to define the total pressure $q^\pm$ through 
the resolution of a suitable Laplace problem, as for the incompressible Euler equations (see, \emph{e.g.}, \cite{Chemin}).} into the following 
\emph{conservative} form:
\begin{equation}
\label{s3-MHD_conservative}
A_0 \, \p_t U^\pm +\p_{x_\alpha} f_\alpha (U^\pm) \, = \, 0 \, , \quad x \in \Omega^\pm(t) \, , \quad t \in [0,T] \, ,
\end{equation}
where we recall that $U$ stands for the vector $(u_1,u_2,u_3,H_1,H_2,H_3,q)^T \in \R^7$, and the matrix $A_0$ is defined by:
\begin{equation}
\label{s3-def_A0}
A_0 \, := \, \begin{bmatrix}
\textbf{I}_6 & 0 \\
0 & 0 \end{bmatrix} \in \cM_7(\bR) \, .
\end{equation}
Let us observe that $A_0$ is not invertible, the last equation in \eqref{s3-MHD_conservative} corresponding to the divergence-free constraint 
on the velocity, which does not include any time derivative. The fluxes $(f_\alpha)_{\alpha=1,2,3}$ in \eqref{s3-MHD_conservative} are explicit 
polynomial expressions of degree at most 2:
\begin{equation}
\label{s3-def_f_j}
f_1(U) :=\begin{bmatrix}
u_1^2 - H_1^2 + q \\
u_1 u_2 - H_1 H_2 \\
u_1 u_3 - H_1 H_3 \\
0 \\
u_1 H_2 - H_1 u_2 \\
u_1 H_3 - H_1 u_3 \\
u_1\end{bmatrix} \, ,\quad f_2(U) :=\begin{bmatrix}
u_2 u_1 - H_2 H_1 \\
u_2^2 - H_2^2 + q \\
u_2 u_3 - H_2 H_3 \\
u_2 H_1 - H_2 u_1 \\
0 \\
u_2 H_3 - H_2 u_3 \\
u_2 \end{bmatrix} \, ,\quad f_3(U) :=\begin{bmatrix}
u_3 u_1 - H_3 H_1 \\
u_3 u_2 - H_3 H_2 \\
u_3^2 - H_3^2 + q \\
u_3 H_1 - H_3 u_1 \\
u_3 H_2 - H_3 u_2 \\
0 \\
u_3 \end{bmatrix} \, .
\end{equation}
For later use, we introduce the Jacobian matrices:
\begin{equation}
\label{s3-def_A_j_cA}
\forall \, \alpha \, = \, 1,2,3 \, ,\quad A_\alpha^\pm \, := \, {\rm d} f_\alpha(U^{0,\pm}) \, ,\quad 
\cA^\pm \, := \, \tau \, A_0 +\xi_1 \, A_1^\pm +\xi_2 \, A_2^\pm \, ,
\end{equation}
and the symmetric bilinear mappings
\begin{equation}
\label{s3-def_bA_j^pm}
\forall \, \alpha \, = \, 1,2,3 \, ,\quad \bA_\alpha(\cdot,\cdot) \, := \, {\rm d}^2 f_\alpha(U^{0,\pm})(\cdot,\cdot) \, .
\end{equation}
Observe that since $f_\alpha$ is polynomial of degree at most $2$, $\bA_\alpha$ does not depend on the state $U^{0,\pm}$ which is 
the reason why we have omitted the $\pm$ superscript. The Jacobian and Hessian matrices of $f_\alpha$ in \eqref{s3-def_A_j_cA} 
and \eqref{s3-def_bA_j^pm}, which appear in the WKB cascade below, are given explicitly in Appendix \ref{appendixA}.

We have decided not to include in \eqref{s3-MHD_conservative} the divergence-free constraint on the magnetic field, but we shall rather 
keep this constraint separate from the remaining partial differential equations:
\begin{equation}
\label{s3-contrainte_div_H}
\dv H^\pm(t,x) \, =0 \, ,\quad x \in \Omega^\pm(t) \, ,\quad t \in [0,T] \, .
\end{equation}
For \emph{exact} solutions to \eqref{s3-equations_nappes_MHD} (supplemented with suitable initial data), it is known that the divergence-free 
constraint \eqref{s3-contrainte_div_H} is only a restriction on the initial data, see \cite{Trakhinin,SWZ}, so it could be `omitted' from system 
\eqref{s3-equations_nappes_MHD}. Nevertheless, since we shall not prescribe arbitrary initial data for $(u^\pm,H^\pm,q^\pm)$, we will need 
to keep the constraint \eqref{s3-contrainte_div_H} to make sure that it is satisfied (at least asymptotically in $\eps$). We shall come back to 
this point later on. Our goal now is to derive the profile equations that are sufficient for \eqref{s3-def_dvlpt_BKW}, \eqref{s3-def_dvlpt_BKW_psi} 
to be an approximate solution to \eqref{s3-equations_nappes_MHD} asymptotically in $\eps$.

\section{The evolution equations and divergence constraints}

We plug the asymptotic expansions \eqref{s3-def_dvlpt_BKW}, \eqref{s3-def_dvlpt_BKW_psi} in \eqref{s3-equations_nappes_MHD} and collect 
the various terms in powers of the small parameter $\eps$. The point is to collect all terms under the form
$$
\sum_{m \ge 0} \eps^m \, \cF^{\, m,\pm} (t,y,Y_3,\theta) 
\Big|_{y'=x',y_3=x_3-\chi(x_3) \, \psi_\eps (t,x'),Y_3=(x_3-\psi_\eps(t,x'))/\eps,\theta=(\tau\, t+\xi' \cdot x')/\eps} \, . 
$$
The only real new additional difficulty compared with previous works, see for instance \cite{AliHunter,Marcou,CW,WW}, is that when we differentiate 
$U_\eps^\pm$ with respect to the variables $t,x_1,x_2$, there are terms of the form:
$$
\p_{y_3} U^{\, m,\pm} (t,y,Y_3,\theta) \, \chi(x_3) \, \p_{t,x_1,x_2} \psi_\eps (t,y',\theta) \, ,
$$
and when we differentiate $U_\eps^\pm$ with respect to $x_3$, there are terms of the form:
$$
\p_{y_3} U^{\, m,\pm} (t,y,Y_3,\theta) \, \chi'(x_3) \, \psi_\eps (t,y',\theta) \, .
$$
In either situation, $\chi(x_3)$ or $\chi'(x_3)$ does not directly read as a function of $(t,y,Y_3,\theta)$ since we easily have $y_3$ in terms of $x_3$ 
but not the other way round. We tus need to invert the relation $y_3=x_3 -\chi(x_3) \, \psi_\eps(t,x')$ in order first to get $x_3$, and then to compose 
with either $\chi$ of $\chi'$ in order to get $\chi(x_3)$ and $\chi'(x_3)$ as asymptotic expansions in $\eps$. Let us observe that the fast tangential 
variable $\theta$, as well as the slow variables $t,y'$, play the role of parameters here. We thus write, in the sense of formal series in $\eps$:
\begin{equation}
\label{defchi}
y_3 \, = \, x_3 -\chi(x_3) \, \sum_{m \ge 2} \eps^m \, \psi^m(t,y',\theta) \, ,
\end{equation}
and invert the latter relation to get $x_3$ in terms of $(t,y,\theta)$. This cumbersome process is done, for instance, by plugging inductively the latter 
expression of $x_3$:
$$
x_3  \, = \, y_3 +\chi \left( y_3 +\chi(x_3) \, \sum_{m \ge 2} \eps^m \, \psi^m(t,y',\theta) \right) \, \sum_{m \ge 2} \eps^m \, \psi^m(t,y',\theta) \, ,
$$
and by performing Taylor expansions in $\eps$. The first terms of this expansion read:
\begin{align*}
x_3  \, = \, y_3 \, & \, +\eps^2 \, \chi (y_3) \, \psi^2 (t,y',\theta) \\
& \, +\eps^3 \, \chi (y_3) \, \psi^3 (t,y',\theta) \\
& \, +\eps^4 \, \left( \chi (y_3) \, \psi^4 (t,y',\theta) +\chi (y_3) \, \chi' (y_3) \, (\psi^2 (t,y',\theta))^2 \right) +\cdots \, .
\end{align*}
For later use, we write the asymptotic expansion of $\chi(x_3)$ and $\chi'(x_3)$ under the following `abstract' form (still in the sense of formal series 
in $\eps$):
\begin{equation}
\label{defchichipointl}
\chi(x_3) \, \sim \, \sum_{m \ge 0} \eps^m \, \chi^{[m]}(t,y,\theta) \, ,\quad 
\chi'(x_3) \, \sim \, \sum_{m \ge 0} \eps^m \, \dot{\chi}^{[m]}(t,y,\theta) \, .
\end{equation}
For instance, some straightforward calculations yield the first terms:
\begin{align*}
\chi^{[0]}(t,y,\theta)&  \, = \, \chi(y_3) \, ,\quad \chi^{[1]}(t,y,\theta)  \, = \, 0 \, ,\quad \chi^{[2]}(t,y,\theta)  \, = \, \chi(y_3) \, \chi'(y_3) \, \psi^2 (t,y',\theta)\, ,\\
\dot{\chi}^{[0]}(t,y,\theta)&  \, = \, \chi'(y_3) \, ,\quad \dot{\chi}^{[1]}(t,y,\theta)  \, = \, 0 \, ,\quad 
\dot{\chi}^{[2]}(t,y,\theta)  \, = \, \chi(y_3) \, \chi''(y_3) \, \psi^2 (t,y',\theta) \, .
\end{align*}
More algebraic properties and relations between the functions $ \chi^{[m]}$, $\dot{\chi}^{[m]}$ and the profiles $\psi^\mu$ will play a crucial role in the 
analysis of Chapter \ref{chapter5}. The details can be found in Appendix \ref{appendixB}. (Actually, we shall obtain in Appendix \ref{appendixB} an 
explicit formula for each function $\chi^{[m]}, \dot{\chi}^{[m]}$ in terms of the $\psi^2,\dots,\psi^m$.)

Once we have written the quantities $\chi(x_3)$ and $\chi'(x_3)$ under the form \eqref{defchichipointl}, it is a long but mere calculus exercise to write down 
the asymptotic expansion of the quantity
$$
A_0 \, \p_t U_\eps^\pm +\p_{x_\alpha} f_\alpha (U_\eps^\pm) \, ,
$$
in terms of $\eps$, with $U_\eps^\pm,\psi_\eps$ as in \eqref{s3-def_dvlpt}. Here we use the crucial fact that the fluxes $f_\alpha$ are quadratic polynomials 
so all derivatives of $f_\alpha$ higher than $3$ vanish. After some calculations, we eventually get:
\begin{multline}
\label{dev_BKW_int0}
A_0 \, \p_t U_\eps^\pm +\p_{x_\alpha} f_\alpha (U_\eps^\pm) \\
\, \sim \, \sum_{m \ge 0} \eps^m \, \Big( \cL_f^\pm(\partial) \, U^{\, m+1,\pm} -F^{\, m,\pm} \Big) 
|_{y'=x',y_3=x_3-\chi(x_3) \, \psi_\eps (t,x'),Y_3=(x_3-\psi_\eps(t,x'))/\eps,\theta=(\tau\, t+\xi' \cdot x')/\eps} \, ,
\end{multline}
where the \emph{fast} operators $\cL_f^\pm(\partial)$ in \eqref{dev_BKW_int0} are defined by:
\begin{equation}
\label{s3-def_L(d)^pm}
\cL_f^\pm(\p) := A_3^\pm \, \p_{Y_3} +\cA^\pm \, \p_\theta \, ,
\end{equation}
and the source term $F^{\, m,\pm}$ in \eqref{dev_BKW_int0} is given for any integer $m \in \N$ by:
\begin{align}
F^{\, m,\pm} \, := \, & \, -L_s^\pm(\p) U^{\, m,\pm} \, 
+\sum_{\ell_1+\ell_2=m+2} \p_\theta \psi^{\ell_1} \, \cA^\pm \, \p_{Y_3} U^{\ell_2,\pm} 
\, +\sum_{\ell_1+\ell_2=m+1} \big( \p_t \psi^{\ell_1} \, A_0 +\p_{y_j} \psi^{\ell_1} \, A_j^\pm \big) \, \p_{Y_3} U^{\ell_2,\pm} \notag \\
& \, +\sum_{\ell_1+\ell_2+\ell_3=m+1} \chi^{[\ell_1]} \, \p_\theta \psi^{\ell_2} \, \cA^\pm \, \p_{y_3} U^{\ell_3,\pm} 
\, +\sum_{\ell_1+\ell_2+\ell_3=m} \chi^{[\ell_1]} \, \big( \p_t \psi^{\ell_1} \, A_0 +\p_{y_j} \psi^{\ell_1} \, A_j^\pm \big) \, \p_{y_3} U^{\ell_3,\pm} \notag \\
& \, +\sum_{\ell_1+\ell_2+\ell_3=m} \dot{\chi}^{[\ell_1]} \, \psi^{\ell_2} \, A_3^\pm \, \p_{y_3} U^{\ell_3,\pm} 
-\sum_{\substack{\ell_1+\ell_2=m \\ \ell_1,\ell_2 \ge 1}} \bA_\alpha (U^{\ell_1,\pm},\p_{y_\alpha} U^{\ell_2,\pm}) \notag \\
& \, -\sum_{\substack{\ell_1+\ell_2=m+1 \\ \ell_1,\ell_2 \ge 1}} \xi_j \, \bA_j (U^{\ell_1,\pm},\p_\theta U^{\ell_2,\pm}) 
-\sum_{\substack{\ell_1+\ell_2=m+1 \\ \ell_1,\ell_2 \ge 1}} \bA_3(U^{\ell_1,\pm} ,\p_{Y_3} U^{\ell_2,\pm}) \label{s3-def_terme_source_F^m,pm} \\
& \, +\sum_{\substack{\ell_1+\ell_2+\ell_3=m+2 \\ \ell_2,\ell_3 \ge 1}} \p_\theta \psi^{\ell_1} \, 
\xi_j \, \bA_j (U^{\ell_2,\pm},\p_{Y_3} U^{\ell_3,\pm}) 
+\sum_{\substack{\ell_1+\ell_2+\ell_3=m+1 \\ \ell_2,\ell_3 \ge 1}} \p_{y_j} \psi^{\ell_1} \, 
\bA_j (U^{\ell_2,\pm},\p_{Y_3} U^{\ell_3,\pm}) \notag \\
& \, +\sum_{\substack{\ell_1+\cdots+\ell_4=m+1 \\ \ell_3,\ell_4 \ge 1}} \chi^{[\ell_1]} \, \p_\theta \psi^{\ell_2} \, 
\xi_j \, \bA_j (U^{\ell_3,\pm},\p_{y_3} U^{\ell_4,\pm}) 
+\sum_{\substack{\ell_1+\cdots+\ell_4=m \\ \ell_3,\ell_4 \ge 1}} \chi^{[\ell_1]} \, \p_{y_j} \psi^{\ell_2} \, 
\bA_j (U^{\ell_3,\pm},\p_{y_3} U^{\ell_4,\pm}) \notag \\
& \, +\sum_{\substack{\ell_1+\cdots+\ell_4=m \\ \ell_3,\ell_4 \ge 1}} 
\dot{\chi}^{[\ell_1]} \, \psi^{\ell_2} \, \bA_3 (U^{\ell_3,\pm} ,\p_{y_3} U^{\ell_4,\pm}) \, .\notag 
\end{align}
Just a few words on the expression \eqref{s3-def_terme_source_F^m,pm}. First, the \emph{slow} operator $L_s^\pm(\p)$ appearing on the first 
line of \eqref{s3-def_terme_source_F^m,pm} is defined by:
\begin{equation}
\label{s3-def_op_diff_cT^pm}
L_s^\pm(\p) \, := \, A_0 \, \p_t \, + \, A_\alpha^\pm \, \p_{y_\alpha} \, ,
\end{equation}
which corresponds to the linearization of the MHD equations around the constant states given in \eqref{s3-def_U^0,pm}. In all expressions such 
as \eqref{s3-def_terme_source_F^m,pm} (and many to come later on), we keep the convention $\psi^0,\psi^1 \equiv 0$, see \eqref{s3-def_U^0,pm}, 
so when a sum involves the functions $\psi^{\ell}$, the terms corresponding to $\ell=0$ and $\ell=1$ can be discarded (though we most often omit to 
mention it). Similarly, all terms involving a partial derivative of the profile $U^{0,\pm}$ can be discarded since these profiles are constant, see 
\eqref{s3-def_U^0,pm}.

The precise -though lengthy- expression \eqref{s3-def_terme_source_F^m,pm} of the source term $F^{\, m,\pm}$ will be absolutely crucial to verify 
several compatibility conditions in Chapter \ref{chapter5} when we construct the profiles $(U^{\, m,\pm},\psi^{\, m+1})_{m \ge 1}$. As is customary in 
geometric optics \cite{Rauch}, the source term $F^{\, m,\pm}$ is entirely defined by the profiles $U^{\, 1,\pm},\dots,U^{\, m,\pm},\psi^2,\dots,\psi^{\, m+1}$. 
More precisely, the very last front profile $\psi^{\, m+1}$ only enters $F^{\, m,\pm}$ through the fast partial derivative $\p_\theta \psi^{\, m+1}$. In other 
words, we do not need to know the mean $\widehat{\psi}^{\, m+1}(0)$ of $\psi^{\, m+1}$ with respect to $\theta$ to compute $F^{\, m,\pm}$. Computing 
the limit of \eqref{s3-def_terme_source_F^m,pm} as $Y_3$ tends to infinity, we get the expression:
\begin{align}
\uF^{\, m,\pm} := & \, -L_s^\pm(\p) \uU^{\, m,\pm} \, 
+\sum_{\ell_1+\ell_2+\ell_3=m+1} \chi^{[\ell_1]} \, \p_\theta \psi^{\ell_2} \, \cA^\pm \, \p_{y_3} \uU^{\ell_3,\pm} \notag \\
& \, +\sum_{\ell_1+\ell_2+\ell_3=m} \chi^{[\ell_1]} \, \big( \p_t \psi^{\ell_1} \, A_0 +\p_{y_j} \psi^{\ell_1} \, A_j^\pm \big) \, \p_{y_3} \uU^{\ell_3,\pm} 
+\sum_{\ell_1+\ell_2+\ell_3=m} \dot{\chi}^{[\ell_1]} \, \psi^{\ell_2} \, A_3^\pm \, \p_{y_3} \uU^{\ell_3,\pm} \notag \\
& \, -\sum_{\substack{\ell_1+\ell_2=m+1 \\ \ell_1,\ell_2 \ge 1}} 
\xi_j \, \bA_j (\uU^{\ell_1,\pm},\p_\theta \uU^{\ell_2,\pm}) 
-\sum_{\substack{\ell_1+\ell_2=m \\ \ell_1,\ell_2 \ge 1}} \bA_\alpha (\uU^{\ell_1,\pm},\p_{y_\alpha} \uU^{\ell_2,\pm}) 
\label{terme_source_F^m,pmbarre} \\
& \, +\sum_{\substack{\ell_1+\cdots+\ell_4=m+1 \\ \ell_3,\ell_4 \ge 1}} 
\chi^{[\ell_1]} \, \p_\theta \psi^{\ell_2} \, \xi_j \, \bA_j (\uU^{\ell_3,\pm},\p_{y_3} \uU^{\ell_4,\pm}) 
+\sum_{\substack{\ell_1+\cdots+\ell_4=m \\ \ell_3,\ell_4 \ge 1}} 
\chi^{[\ell_1]} \, \p_{y_j} \psi^{\ell_2} \, \bA_j (\uU^{\ell_3,\pm},\p_{y_3} \uU^{\ell_4,\pm}) \notag \\
& \, +\sum_{\substack{\ell_1+\cdots+\ell_4=m \\ \ell_3,\ell_4 \ge 1}} 
\dot{\chi}^{[\ell_1]} \, \psi^{\ell_2} \, \bA_3 (\uU^{\ell_3,\pm},\p_{y_3} \uU^{\ell_4,\pm}) \, .\notag
\end{align}

By using the symmetry of the $\bA_\alpha$'s, we get from \eqref{s3-def_terme_source_F^m,pm} the expressions:
\begin{subequations}
\label{s3-def_F^1,pm}
\begin{align}
F^{0,\pm} & \, =0 \, ,\label{s3-def_terme_source_F^0,pm} \\
F^{\, 1,\pm} & \, =-L_s^\pm(\p)U^{\, 1,\pm} \, + \, \p_\theta\psi^2 \, \cA^\pm \, \p_{Y_3}U^{\, 1,\pm} \, 
- \, \dfrac{1}{2} \, \big( \xi_j \, \p_\theta \bA_j(U^{\, 1,\pm},U^{\, 1,\pm}) + \p_{Y_3}\bA_3(U^{\, 1,\pm},U^{\, 1,\pm}) 
\big) \, .\label{s3-def_terme_source_F^1,pm}
\end{align}
\end{subequations}
The symmetry of the $\bA_\alpha$'s will be useful in Chapter \ref{chapter4} when constructing the leading amplitude of the WKB ansatz, and 
also later in Chapter \ref{chapter5} when constructing the correctors. This is the reason why we have chosen to write the MHD system in its 
conservative form \eqref{s3-MHD_conservative} rather than in its nonconservative form.

Since we wish to solve \eqref{s3-MHD_conservative} asymptotically at any order in $\eps$, a sufficient condition for doing so is to require that 
each term in the asymptotic expansion \eqref{dev_BKW_int0} vanishes. We are then led to the so-called WKB cascade which the profiles 
$(U^{\, m,\pm},\psi^{\, m+1})_{m \ge 1}$ should satisfy:
\begin{equation}
\label{dev_BKW_int}
\forall \, m \ge 0 \, ,\quad \cL_f^\pm(\partial) \, U^{\, m+1,\pm} \, = \, F^{\, m,\pm} \, ,\quad 
(t,y',y_3,Y_3,\theta) \in [0,T] \times \bT^2 \times I^\pm \times \R^\pm \times \bT \, ,
\end{equation}
with the source term $F^{\, m,\pm}$ defined by \eqref{s3-def_terme_source_F^m,pm}. The slow variables $(t,y)$ play the role of parameters in 
\eqref{dev_BKW_int} since the operators $\cL_f^\pm(\partial)$ only act on the fast variables $(Y_3,\theta)$.

Let us be a little more specific on the seventh equation in \eqref{dev_BKW_int}, which corresponds to the divergence-free constraint on the 
velocity field. Using that the seventh line of the $\bA_\alpha$'s is zero (the divergence constraint is linear !), see Appendix \ref{appendixA}, 
we find that the seventh line of \eqref{dev_BKW_int} reads:
\begin{subequations}\label{s3-cascade_div_u^m+1,pm_div_H^m+1,pm}
\begin{equation}
\label{s3-cascade_div_u^m+1,pm}
\begin{aligned}
\p_{Y_3} u_3^{\, m+1,\pm} \, +\xi_j \, \p_\theta u_j^{\, m+1,\pm} \, =& \, -\nabla \cdot u^{\, m,\pm} \, 
+\sum_{\ell_1+\ell_2=m+2} \p_\theta \psi^{\ell_1} \, \xi_j \, \p_{Y_3} u_j^{\ell_2,\pm} 
+\sum_{\ell_1+\ell_2=m+1} \p_{y_j} \psi^{\ell_1} \, \p_{Y_3} u_j^{\ell_2,\pm} \\[0.5ex]
& \, +\sum_{\ell_1+\ell_2+\ell_3=m+1} \chi^{[\ell_1]} \, \p_\theta \psi^{\ell_2} \xi_j \, \p_{y_3} u_j^{\ell_3,\pm} 
+\sum_{\ell_1+\ell_2+\ell_3=m} \chi^{[\ell_1]} \, \p_{y_j} \psi^{\ell_2} \, \p_{y_3} u_j^{\ell_3,\pm} \\[0.5ex]
& \, +\sum_{\ell_1+\ell_2+\ell_3=m} \dot{\chi}^{[\ell_1]} \, \psi^{\ell_2} \, \p_{y_3} u_3^{\ell_3,\pm} \, .
\end{aligned}
\end{equation}
Similarly, plugging the WKB ansatz for the magnetic field in the constraint \eqref{s3-contrainte_div_H}, we find that the profiles 
$(H^{\, m,\pm})_{m \ge 1}$ must satisfy the following constraints:
\begin{equation}
\label{s3-cascade_div_H^m+1,pm}
\begin{aligned}
\p_{Y_3} H_3^{\, m+1,\pm} \, +\xi_j \, \p_\theta H_j^{\, m+1,\pm} \, =& \, -\nabla \cdot H^{\, m,\pm} \, 
+\sum_{\ell_1+\ell_2=m+2} \p_\theta \psi^{\ell_1} \, \xi_j \, \p_{Y_3} H_j^{\ell_2,\pm} 
+\sum_{\ell_1+\ell_2=m+1} \p_{y_j} \psi^{\ell_1} \, \p_{Y_3} H_j^{\ell_2,\pm} \\[0.5ex]
& \, +\sum_{\ell_1+\ell_2+\ell_3=m+1} \chi^{[\ell_1]} \, \p_\theta \psi^{\ell_2} \xi_j \, \p_{y_3} H_j^{\ell_3,\pm} 
+\sum_{\ell_1+\ell_2+\ell_3=m} \chi^{[\ell_1]} \, \p_{y_j} \psi^{\ell_2} \, \p_{y_3} H_j^{\ell_3,\pm} \\[0.5ex]
& \, +\sum_{\ell_1+\ell_2+\ell_3=m} \dot{\chi}^{[\ell_1]} \, \psi^{\ell_2} \, \p_{y_3} H_3^{\ell_3,\pm} \, .
\end{aligned}
\end{equation}
\end{subequations}
For future use, we rewrite \eqref{s3-cascade_div_H^m+1,pm} under the form:
\begin{equation}
\label{s3-cascade_div_H^m+1,pm'}
\p_{Y_3} H_3^{\, m+1,\pm} \, +\xi_j \, \p_\theta H_j^{\, m+1,\pm} \, = \, F_8^{\, m,\pm} \, ,
\end{equation}
where $F_8^{\, m,\pm}$ is a short notation for the right hand side of \eqref{s3-cascade_div_H^m+1,pm}. Observe again that $\psi^{\, m+1}$ 
enters the right hand side of \eqref{s3-cascade_div_H^m+1,pm} only through its fast derivative $\p_\theta \psi^{\, m+1}$. Equation 
\eqref{s3-cascade_div_H^m+1,pm'} is the analogue of \eqref{s3-cascade_div_u^m+1,pm}, which reads:
$$
\p_{Y_3} u_3^{\, m+1,\pm} \, +\xi_j \, \p_\theta u_j^{\, m+1,\pm} \, = \, F_7^{\, m,\pm} \, .
$$
Equations \eqref{s3-cascade_div_u^m+1,pm_div_H^m+1,pm} will be later referred to as the \emph{fast divergence} constraints.

In Chapters \ref{chapter4} and \ref{chapter5}, we shall construct a sequence of profiles that satisfy \eqref{dev_BKW_int} and 
\eqref{s3-cascade_div_H^m+1,pm'} together with several jump and boundary conditions which we are going to derive now.

\section{The jump conditions}

The jump conditions for the WKB cascade are obtained by plugging the expressions \eqref{s3-def_dvlpt_BKW}, \eqref{s3-def_dvlpt_BKW_psi} 
in the jump conditions \eqref{int-cond_bord} which appear in the current vortex sheet system \eqref{s3-equations_nappes_MHD}. Let us recall 
that on the free surface $\{ x_3=\psi_\eps(t,x') \}$, there holds $y_3=Y_3=0$ (recall $\chi \equiv 1$ near $0$ so $\chi(\psi_\eps) \sim 1$ at any 
order in $\eps$), which is the reason why a double trace appears in the jump conditions for the WKB cascade below. Collecting the powers of 
$\eps$, we derive the following set of equations, where $m \ge 0$:
\begin{subequations}\label{s3-cascade_BKW_bord_simplifiee_m_geq_0}
\begin{equation}
\label{s3-cascade_BKW_bord_simplifiee_m_geq_2}
\left\{
\begin{array}{r c l}
\ds u_3^{\, m+1,\pm} \, -c^\pm \, \p_\theta \psi^{\, m+2} & = & \p_t\psi^{\, m+1} +u_j^{0,\pm} \, \p_{y_j} \psi^{\, m+1} \\[1ex]
& & \ds +\sum_{\substack{\ell_1+\ell_2=m+2 \\ \ell_2 \ge 1}} \p_\theta \psi^{\ell_1} \, \xi_j \, u_j^{\ell_2,\pm} 
+\sum_{\substack{\ell_1+\ell_2=m+1 \\ \ell_2 \ge 1}} \p_{y_j} \psi^{\ell_1} \, u_j^{\ell_2,\pm} \, , \\[1ex]
\ds H_3^{\, m+1,\pm} \, -b^\pm \, \p_\theta \psi^{\, m+2} & = & \ds H_j^{0,\pm} \, \p_{y_j} \psi^{\, m+1} 
+\sum_{\substack{\ell_1+\ell_2=m+2 \\ \ell_2 \ge 1}} \p_\theta \psi^{\ell_1} \, \xi_j \, H_j^{\ell_2,\pm} 
+\sum_{\substack{\ell_1+\ell_2=m+1 \\ \ell_2 \ge 1}} \p_{y_j} \psi^{\ell_1} \, H_j^{\ell_2,\pm} \, , \\[1ex]
\ds q^{\, m+1,+} \, -q^{\, m+1,-} & = & 0 \, ,
\end{array}
\right.
\end{equation}
where all functions $u_\alpha^{\ell,\pm}$, $H_\alpha^{\ell,\pm}$, $q^{\ell,\pm}$ are evaluated at $y_3=Y_3=0$, and the set of equations 
\eqref{s3-cascade_BKW_bord_simplifiee_m_geq_2} should be satisfied for all $(t,y',\theta) \in [0,T] \times \bT^2 \times \bT$.

Let us make the two first cases $m=0$ and $m=1$ in \eqref{s3-cascade_BKW_bord_simplifiee_m_geq_2} more explicit. For $m=0$, we get 
the \emph{homogeneous} system: 
\begin{equation}
\label{s3-cond_bord_U^1,pm}
\left\{
\begin{array}{r c l}
u_3^{\, 1,\pm} |_{y_3=Y_3=0} \, -c^\pm \, \p_\theta \psi^2 & = & 0 \, , \\[0.5ex]
H_3^{\, 1,\pm} |_{y_3=Y_3=0} \, -b^\pm \, \p_\theta \psi^2 & = & 0 \, , \\[0.5ex]
q^{\, 1,+} |_{y_3=Y_3=0} \, -q^{\, 1,-} |_{y_3=Y_3=0} & = & 0 \, .
\end{array}
\right.
\end{equation}
For $m=1$, we get the \emph{inhomogeneous} system:
\begin{equation}
\label{s3-cascade_BKW_bord_simplifiee_m=1}
\left\{
\begin{array}{r c l}
\ds u_3^{\, 2,\pm} |_{y_3=Y_3=0} \, -c^\pm \, \p_\theta \psi^3 & = & \p_t \psi^2 +u_j^{0,\pm} \, \p_{y_j} \psi^2 
+\p_\theta \psi^2 \, \xi_j \, u_j^{\, 1,\pm} |_{y_3=Y_3=0} \, , \\[1ex]
\ds H_3^{\, 2,\pm} |_{y_3=Y_3=0} \, -b^\pm \, \p_\theta \psi^3 & = & H_j^{0,\pm} \, \p_{y_j} \psi^2 
+\p_\theta \psi^2 \, \xi_j \, H_j^{\, 1,\pm} |_{y_3=Y_3=0} \, , \\[1ex]
q^{\, 2,+} |_{y_3=Y_3=0} \, -q^{\, 2,-} |_{y_3=Y_3=0} & = & 0 \, .
\end{array}
\right.
\end{equation}
\end{subequations}

It is convenient for later use to rewrite the set \eqref{s3-cascade_BKW_bord_simplifiee_m_geq_2} of jump conditions in a more compact form. 
We first define the following matrices $B^\pm \in \cM_{5,7}(\R)$:
\begin{equation}
\label{s3-def_B^pm}
B^+ := \begin{bmatrix}
0 & 0 & 1 & 0 & 0 & 0 & 0 \\
0 & 0 & 0 & 0 & 0 & 1 & 0 \\
0 & 0 & 0 & 0 & 0 & 0 & 0 \\
0 & 0 & 0 & 0 & 0 & 0 & 0 \\
0 & 0 & 0 & 0 & 0 & 0 & 1 \end{bmatrix} \, ,\quad B^- := \begin{bmatrix}
0 & 0 & 0 & 0 & 0 & 0 & 0 \\
0 & 0 & 0 & 0 & 0 & 0 & 0 \\
0 & 0 & 1 & 0 & 0 & 0 & 0 \\
0 & 0 & 0 & 0 & 0 & 1 & 0 \\
0 & 0 & 0 & 0 & 0 & 0 & -1 \end{bmatrix} \, ,
\end{equation}
as well as the vector:
\begin{equation}
\label{s3-def_cal_b}
\ub \, := -(c^+,b^+,c^-,b^-,0)^T \in \bR^5 \, .
\end{equation}
We can then rewrite the jump conditions \eqref{s3-cascade_BKW_bord_simplifiee_m_geq_0} in the more compact form:
\begin{equation}
\label{s3-cascade_BKW_bord_matriciel}
\forall \, m \ge 0 \, ,\quad 
B^+ \, U^{\, m+1,+}|_{y_3=Y_3=0} \, +B^- \, U^{\, m+1,-}|_{y_3=Y_3=0} \, +\p_\theta \psi^{\, m+2} \, \ub \, = \, G^m \, ,
\end{equation}
where the source term $G^m$ has the form:
\begin{subequations}
\label{terme_source_bord_BKW}
\begin{equation}
\label{s3-def_G^m}
G^m \, := \, \big( \, G_1^{\, m,+},G_2^{\, m,+},G_1^{\, m,-},G_2^{\, m,-},0 \big)^T \, ,
\end{equation}
with:
\begin{equation}
\label{s3-def_G_1^m,pm}
G_1^{\, m,\pm} \, := \, \p_t \psi^{\, m+1} +u_j^{0,\pm} \, \p_{y_j} \psi^{\, m+1} 
+\sum_{\substack{\ell_1+\ell_2=m+2 \\ \ell_2 \ge 1}} \p_\theta \psi^{\ell_1} \, \xi_j \, u_j^{\ell_2,\pm}|_{y_3=Y_3=0} 
+\sum_{\substack{\ell_1+\ell_2=m+1 \\ \ell_2 \ge 1}} \p_{y_j} \psi^{\ell_1} \, u_j^{\ell_2,\pm}|_{y_3=Y_3=0} \, ,
\end{equation}
\begin{equation}
\label{s3-def_G_2^m,pm}
G_2^{\, m,\pm} \, := \, H_j^{0,\pm} \, \p_{y_j} \psi^{\, m+1} 
+\sum_{\substack{\ell_1+\ell_2=m+2 \\ \ell_2 \ge 1}} \p_\theta \psi^{\ell_1} \, \xi_j \, H_j^{\ell_2,\pm}|_{y_3=Y_3=0} 
+\sum_{\substack{\ell_1+\ell_2=m+1 \\ \ell_2 \ge 1}} \p_{y_j} \psi^{\ell_1} \, H_j^{\ell_2,\pm}|_{y_3=Y_3=0} \, .
\end{equation}
\end{subequations}
It is important to observe that the mean $\widehat{\psi}^{\, m+1}(0)$ does enter the definition of the source term $G^m$. More precisely, the 
nonzero Fourier coefficients of $G^m$ depend on $U^{\, 1,\pm},\dots,U^{\, m,\pm},\psi^2,\dots,\psi^m$ and $\p_\theta \psi^{\, m+1}$, but the mean 
$\widehat{G}^m(0)$ does depend on $\widehat{\psi}^{\, m+1}(0)$. This fact will be important in the induction process of Chapter \ref{chapter5}.

\section{The fixed boundaries}

We now focus on the boundary conditions on the top and bottom boundaries $\Gamma^\pm$. Here we should recall that profiles in the space 
$S^\pm$ are decomposed following \eqref{s3-def_decomp_S^pm} as the sum of a first profile that is independent of the fast normal variable 
$Y_3$, and of a second profile that decays exponentially with respect to $Y_3$. For $x_3 =\pm 1$, there holds $|x_3 -\psi_\eps(t,x')|/\eps \ge 
1/(2\, \eps)$ for any sufficiently small $\eps$, so the surface wave component of the profile is $O(\eps^\infty)$ when evaluated at the top and 
bottom boundaries.

The boundary conditions $(u_\eps)_3^\pm|_{x_3=\pm 1}=(H_\eps)_3^\pm|_{x_3=\pm 1}=0$ will therefore be satisfied asymptotically in $\eps$ 
if there holds:
\begin{equation}
\label{s3-cond_bords_fixes_uH_3^m,pm}
\forall \, m \ge 0 \, ,\quad 
\underline{u}_3^{\, m+1,\pm}|_{y_3=\pm 1} \, = \, \underline{H}_3^{\, m+1,\pm}|_{y_3=\pm 1} \, = \, 0 \, .
\end{equation}
Observe that in \eqref{s3-cond_bords_fixes_uH_3^m,pm}, the trace of $\underline{u}_3^{\, m+1,\pm}$ at $y_3=\pm 1$ is a function of $(t,y',\theta)$ 
so \eqref{s3-cond_bords_fixes_uH_3^m,pm} is a condition for all Fourier modes with respect to $\theta$ which we can rephrase as:
$$
\forall \, m \ge 0 \, ,\quad \forall \, k \in \Z \, ,\quad 
\widehat{\underline{u}}_3^{\, m+1,\pm} (t,y',\pm 1,k) \, =\widehat{\underline{H}}_3^{\, m+1,\pm} (t,y',\pm 1,k) \, = \, 0 \, .
$$

Here we see why introducing the cut-off function $\chi$ in the slow normal variable $y_3$ was convenient. If we had chosen the probably more 
natural candidate $y_3=x_3 -\psi_\eps(t,x')$ as the slow normal variable, the boundary conditions at $x_3=\pm 1$ would have mixed the profiles 
$(U^{\, m,\pm})_{m \ge 1}$ for the physical quantities (velocity, magnetic field) with the profiles $(\psi^{\, m+1})_{m \ge 1}$ for the front. Beyond the 
notational inconvenience, this would have made the verification of some compatibility conditions even more cumbersome than they will be in 
our framework. The price to pay with our choice for the slow normal variable is the introduction of the many additional terms in 
\eqref{s3-def_terme_source_F^m,pm} that involve the functions $\chi^{[\ell]}, \dot{\chi}^{[\ell]}$. The algebra involved with these additional terms 
is one of the achievements of our work and is detailed in Appendix \ref{appendixB}.

\section{Normalizing the total pressure}

It will be convenient in the induction argument described in Chapters \ref{chapter4} and \ref{chapter5} to determine the total pressure 
in a fixed given way. Observe indeed that in \eqref{s3-equations_nappes_MHD}, the total pressure is defined up to a function of time. 
In other words, we can always shift $q^\pm$ by a given function of time:
$$
q^+(t,x) +Q(t) \, ,\quad q^-(t,x) +Q(t) \, ,
$$
and we still get a solution to \eqref{s3-equations_nappes_MHD}. To avoid this indeterminacy, we make the same choice as in 
\cite{CMST,SWZ} and fix the total pressure $(q^+,q^-)$ by imposing the zero mean condition:
$$
\int_{\Omega^+(t)} q^+(t,x) \, {\rm d}x \, + \, \int_{\Omega^-(t)} q^-(t,x) \, {\rm d}x =0 \, .
$$
For the oscillating problem, since the subdomains also depend on the wavelength $\eps$, the latter normalization conditions reads:
\begin{equation}
\label{s3-cond_moy_nulle_Omega_q_eps^pm_2}
\forall \, t \in [0,T] \, ,\quad 
\int_{\Omega_\varepsilon^+(t)} q_\varepsilon^+(t,x) \, {\rm d}x \, + \, 
\int_{\Omega_\varepsilon^-(t)} q_\varepsilon^-(t,x) \, {\rm d}x \, = \, 0 \, ,
\end{equation}
where the domains $\Omega_\varepsilon^\pm(t)$ are given by:
$$
\Omega_\varepsilon^\pm(t) \, := \, \big\{ x \in \bT^2 \times (-1,1) \, \big| \, x_3 \gtrless \psi_\varepsilon(t,x') \big\} \, .
$$

Our goal here is to compute the asymptotic expansion with respect to the small parameter $\eps$ of the integrals in 
\eqref{s3-cond_moy_nulle_Omega_q_eps^pm_2}, when the total pressure and the front follow the asymptotic expansions 
\eqref{s3-def_dvlpt_BKW}, \eqref{s3-def_dvlpt_BKW_psi}. The expansion will give rise to the normalization conditions for 
each profile of the sequence $(q^{\, m,\pm})_{m \ge 1}$.

There are several steps in the calculation. First, we observe that the domains $\Omega_\varepsilon^\pm (t)$ have measure $1+o(1)$ as 
$\eps$ tends to zero, so any $O(\eps^k)$ contribution in $q_\varepsilon^\pm$ will give at best an $O(\eps^k)$ contribution in the integrals 
of \eqref{s3-cond_moy_nulle_Omega_q_eps^pm_2}. This means that if we write
\begin{equation}
\label{BKW_norm_pression}
\int_{\Omega_\varepsilon^+(t)} q_\varepsilon^+(t,x) \, {\rm d}x \, + \, 
\int_{\Omega_\varepsilon^-(t)} q_\varepsilon^-(t,x) \, {\rm d}x \sim \sum_{m \ge 1} \eps^m \, \cI^m (t) \, ,
\end{equation}
then $\cI^m (t)$ only depends on the finitely many profiles $q^{\, 1,\pm},\dots,q^{\, m,\pm}$. Let us emphasize that the asymptotic expansion 
\eqref{BKW_norm_pression} begins indeed with the $\eps^1$ scale because of the normalization condition \eqref{s3-def_U^0,pm} for the 
reference current vortex sheet (which makes the term $\cI^0$ automatically vanish). We are now going to make the functions $\cI^m$, 
$m \ge 1$, in \eqref{BKW_norm_pression} explicit. For each profile
$$
q^{\, m,\pm}\left( t,x',x_3-\chi(x_3) \, \psi_\eps(t,x'),\dfrac{x_3 -\psi_\eps(t,x')}{\eps},\dfrac{\tau \, t +\xi'\cdot x'}{\eps} 
\right) \, ,\quad m \ge 1\, ,
$$
we need to compute its integral on $\Omega_\eps^\pm(t)$, then multiply by $\eps^m$, see \eqref{s3-def_dvlpt_BKW}, sum with respect to 
$m \ge 1$ and rearrange the formal series in terms of powers of $\eps$ thus obtaining \eqref{BKW_norm_pression}. We explain below how 
one can derive the asymptotic expansion of the integral:
\begin{equation}
\label{def_int_Jm+(t)}
J_\eps^{\, m,+}(t) \, := \, \int_{\Omega_\eps^+(t)} q^{\, m,+} \left( t,x',x_3-\chi(x_3) \, \psi_\eps(t,x'), 
\dfrac{x_3 -\psi_\eps(t,x')}{\eps},\dfrac{\tau \, t +\xi'\cdot x'}{\eps} \right) \, {\rm d}x \, .
\end{equation}
The integral $J_\eps^{\, m,-}(t)$ on the other side of the current vortex sheet is dealt with in exactly the same way so we omit the details.
\bigskip

$\bullet$ \emph{Straightening the domain}. We first make the change of variables $(x',x_3) \rightarrow (y',y_3)$ with $y':=x'$ and 
$y_3 :=x_3 -\chi(x_3) \, \psi_\eps (t,x')$. This is indeed, at least formally, a change of variable for $\eps$ sufficiently small since $\psi_\eps$ 
is expected to be $O(\eps^2)$ in $L^\infty$ so the Jacobian of this transformation does not vanish. The domain $\Omega_\eps^+(t)$ is 
mapped onto the fixed domain $\Omega_0^+ := \bT^2 \times I^+$, which corresponds to the flat front $\psi_\eps \equiv 0$. The integral 
$J_\eps^{\, m,+}(t)$ in \eqref{def_int_Jm+(t)} is rewritten accordingly:
$$
J_\eps^{\, m,+}(t) \, = \, \int_{\Omega_0^+} q^{\, m,+} \left( t,y',y_3,\dfrac{y_3 +(\chi(x_3)-1) \, \psi_\eps(t,y')}{\eps}, 
\dfrac{\tau \, t +\xi'\cdot y'}{\eps} \right) \, \dfrac{{\rm d}y}{1-\chi'(x_3) \, \psi_\eps(t,y')} \, ,
$$
where, for any given time $t$ and $y \in \Omega_0^+$, $x_3 \in (\psi_\eps(t,y'),1)$ denotes the unique solution to the equation:
$$
y_3 =x_3-\chi(x_3) \, \psi_\eps(t,y') \, .
$$
The decomposition \eqref{s3-def_decomp_S^pm} of the profile $q^{\, m,+}$ on $\uS^+ \oplus S_\star^+$ yields a decomposition of the above 
integral $J_\eps^{\, m,+}(t)$. Each of the two terms is examined separately.
\bigskip

$\bullet$ \emph{The surface wave component}. Let us first look at:
$$
J_{\eps,\star}^{\, m,+}(t) \, := \, \int_{\Omega_0^+} 
q_\star^{\, m,+} \left( t,y',y_3,\dfrac{y_3 +(\chi(x_3)-1) \, \psi_\eps(t,y')}{\eps},\dfrac{\tau \, t +\xi'\cdot y'}{\eps} \right) 
\, \dfrac{{\rm d}y}{1-\chi'(x_3) \, \psi_\eps(t,y')} \, .
$$
We recall that the cut-off function $\chi$ equals $1$ on the interval $[-1/3,1/3]$ so $\chi(x_3)$ equals $1$ at any order in $\eps$ for, say, 
$y_3 \in [0,1/6]$. In the same way, there holds $\chi'(x_3)=0$ for $y_3 \in [0,1/6]$ and any $\eps$ sufficiently small. For $y_3 \ge 1/6$, we 
recall that $q_\star^{\, m,+}$ has exponential decay with respect to the fast normal variable $Y_3$, so the `remainder' term:
$$
\int_{\bT^2 \times (1/6,1)} q_\star^{\, m,+} \left( t,y',y_3,\dfrac{y_3 +(\chi(x_3)-1) \, \psi_\eps(t,y')}{\eps},\dfrac{\tau \, t +\xi'\cdot y'}{\eps} \right) 
\, \dfrac{{\rm d}y' \, {\rm d}y_3}{1-\chi'(x_3) \, \psi_\eps(t,y')}
$$ 
is an $O(\eps^\infty)$. (It is actually an exponentially small term.) We may thus write:
\begin{align*}
J_{\eps,\star}^{\, m,+}(t) \, =& \, \int_{\bT^2 \times (0,1/6)} 
q_\star^{\, m,+} \left( t,y',y_3,\dfrac{y_3}{\eps},\dfrac{\tau \, t +\xi'\cdot y'}{\eps} \right) \, {\rm d}y' \, {\rm d}y_3 +O(\eps^\infty) \\
=& \, \int_{\Omega_0^+} q_\star^{\, m,+} \left( t,y',y_3,\dfrac{y_3}{\eps},\dfrac{\tau \, t +\xi'\cdot y'}{\eps} \right) 
\, {\rm d}y' \, {\rm d}y_3 +O(\eps^\infty) \, ,
\end{align*}
where the final equality comes again from the exponential decay of $q_\star^{\, m,+}$ with respect to $Y_3$. We now decompose the profile 
$q_\star^{\, m,+}$ in Fourier series with respect to $\theta$. For each $k \neq 0$, the integral
$$
\int_{\Omega_0^+} \widehat{q}_\star^{\, m,+} \left( t,y',y_3,\dfrac{y_3}{\eps},k \right) \, 
\exp \left( i\, k \, \dfrac{\tau \, t +\xi'\cdot y'}{\eps} \right) \, {\rm d}y' \, {\rm d}y_3 \, ,
$$
is shown to be an $O(\eps^\infty)$ by means of integration by parts in $y'$ - the so-called stationary or rather nonstationary phase method 
(here we use $\xi' \neq 0$ and the fact that the fast normal variable $y_3/\eps$ is independent of $y'$). We are thus left with:
\begin{align*}
J_{\eps,\star}^{\, m,+}(t) \, = \, & \, \int_{\Omega_0^+} \widehat{q}_\star^{\, m,+} \left( t,y',y_3,\dfrac{y_3}{\eps},0 \right) \, {\rm d}y' \, {\rm d}y_3 
+O(\eps^\infty) \\
=\, & \, \eps \, \int_{\bT^2 \times [0,1/\eps]} \widehat{q}_\star^{\, m,+} (t,y',\eps \, Y_3,Y_3,0) \, {\rm d}y' \, {\rm d}Y_3 +O(\eps^\infty) \\
\sim \, & \, \sum_{\ell \ge 0} \, \dfrac{\eps^{\ell+1}}{\ell \, !} \, \int_{\bT^2 \times \R^+} Y_3^\ell \, 
\p_{y_3}^\ell \widehat{q}_\star^{\, m,+} (t,y',0,Y_3,0) \, {\rm d}y' \, {\rm d}Y_3 \, .
\end{align*}
Collecting the contributions of each integral $J_{\eps,\star}^{\, m,+}(t)$, which we recall is part of the integral $J_\eps^{\, m,+}(t)$ in 
\eqref{def_int_Jm+(t)}, and adding up with the analogous contributions from the minus side $\Omega_\eps^-(t)$, we end up with:
\begin{multline}
\label{expression_Jm(t)-1}
\sum_{m \ge 1} \eps^m \, \Big( J_{\eps,\star}^{\, m,+}(t) +J_{\eps,\star}^{\, m,-}(t) \Big) \\
\sim \, \sum_{m \ge 2} \eps^m \, \sum_{\substack{\ell_1+\ell_2=m-1 \\ \ell_1 \ge 1}} \dfrac{1}{\ell_2 \, !} \, \left\{ 
\int_{\bT^2 \times \R^+} Y_3^{\ell_2} \, \p_{y_3}^{\ell_2} \widehat{q}_\star^{\, \ell_1,+} (t,y',0,Y_3,0) \, 
{\rm d}y' \, {\rm d}Y_3 \right. \\
\left. +\int_{\bT^2 \times \R^-} 
Y_3^{\ell_2} \, \p_{y_3}^{\ell_2} \widehat{q}_\star^{\, \ell_1,-} (t,y',0,Y_3,0) \, {\rm d}y' \, {\rm d}Y_3 \right\} \, .
\end{multline}
\bigskip

$\bullet$ \emph{The residual component}. We now examine the integral:
$$
\uJ_\eps^{\, m,+}(t) \, := \, \int_{\Omega_0^+} 
\uq^{\, m,+} \left( t,y',y_3,\dfrac{\tau \, t +\xi'\cdot y'}{\eps} \right) \, \dfrac{{\rm d}y}{1-\chi'(x_3) \, \psi_\eps(t,y')} \, .
$$
We know from \eqref{defchichipointl} that $\chi'(x_3)$ can be expanded in terms of $\eps$ as:
$$
\chi'(x_3) \, \sim \, \sum_{m \ge 0} \, \eps^m \, \dot{\chi}^{[m]} \left( t,y',y_3,\dfrac{\tau \, t +\xi'\cdot y'}{\eps} \right) \, ,
$$
with suitable profiles $\dot{\chi}^{[m]}$, $m \ge 0$, and the asymptotic expansion of $\psi_\eps$ is given by 
\eqref{s3-def_dvlpt_BKW_psi}. We may thus write:
\begin{equation}
\label{dev-jacobien}
\dfrac{1}{1-\chi'(x_3) \, \psi_\eps(t,y')} \, \sim \, 1 +\sum_{m \ge 2} \, \eps^m \, \cJ^m 
\left( t,y',y_3,\dfrac{\tau \, t +\xi'\cdot y'}{\eps} \right) \, .
\end{equation}
Each function $\cJ^m$, $m \ge 2$, in \eqref{dev-jacobien} can be computed by starting from:
$$
1-\chi'(x_3) \, \psi_\eps(t,y') \, \sim \, 1-\sum_{m \ge 2} \eps^m \, 
\sum_{\ell_1+\ell_2=m} \dot{\chi}^{[\ell_1]} (t,y,\theta) \, \psi^{\ell_2} (t,y',\theta) \, ,
$$
and then computing the asymptotic expansion in $\eps$ of the inverse (see \cite{Comtet} for some explicit formula). In particular, there holds:
$$
\cJ^2 (t,y,\theta) \, = \, \chi'(y_3) \, \psi^2 (t,y',\theta) \, .
$$

At this stage, we plug the expansion \eqref{dev-jacobien} into the expression of the integral $\uJ_\eps^{\, m,+}(t)$. Using again the stationary 
phase method, we end up with:
$$
\uJ_\eps^{\, m,+}(t) \, \sim \, \int_{\Omega_0^+} \widehat{\uq}^{\, m,+} (t,y,0) \, {\rm d}y 
+\sum_{\ell \ge 2} \, \eps^\ell \, \int_{\Omega_0^+} \! \! \int_\bT \uq^{\, m,+} (t,y,\theta) \, \cJ^\ell \, (t,y,\theta) \, 
{\rm d}y \, \dfrac{{\rm d}\theta}{2\, \pi} \, .
$$
Multiplying by $\eps^m$ and summing over $m \ge 1$, we end up with:
\begin{multline}
\label{expression_Jm(t)-2}
\sum_{m \ge 1} \, \eps^m \, \Big( \uJ_\eps^{\, m,+}(t) +\uJ_\eps^{\, m,-}(t) \Big) 
\, \sim \, \sum_{m \ge 1} \, \eps^m \, \left( \int_{\Omega_0^+} \widehat{\uq}^{\, m,+} (t,y,0) \, {\rm d}y 
+\int_{\Omega_0^-} \widehat{\uq}^{\, m,-} (t,y,0) \, {\rm d}y \right) \\
+\sum_{m \ge 3} \, \eps^m \, \sum_{\substack{\ell_1+\ell_2=m \\ \ell_1 \ge 1,\ell_2 \ge 2}} 
\int_{\Omega_0^+} \! \! \int_\bT \uq^{\ell_1,+} (t,y,\theta) \, \cJ^{\ell_2} \, (t,y,\theta) \, {\rm d}y \, \dfrac{{\rm d}\theta}{2\, \pi} 
+\int_{\Omega_0^-} \! \! \int_\bT \uq^{\ell_1,-} (t,y,\theta) \, \cJ^{\ell_2} \, (t,y,\theta) \, {\rm d}y \, \dfrac{{\rm d}\theta}{2\, \pi} \, .
\end{multline}
\bigskip

We now collect the contributions \eqref{expression_Jm(t)-1} and \eqref{expression_Jm(t)-2} to get the expression of $\cI^m (t)$ in 
\eqref{BKW_norm_pression}. We have just derived the expression:
\begin{equation}
\label{expression_Im(t)}
\forall \, m \ge 1 \, ,\quad \cI^m (t) \, = \, \int_{\Omega_0^+} \widehat{\uq}^{\, m,+}(t,y,0) \, {\rm d}y \, + \, 
\int_{\Omega_0^-} \widehat{\uq}^{\, m,-}(t,y,0) \, {\rm d}y \, + \,  \bI^{\, m-1} (t) \, ,
\end{equation}
where $\bI^{\, m-1} (t)$ is computed from the `previous' profiles $q^{\, 1,\pm},\dots,q^{\, m-1,\pm}$, $\psi^2,\dots,\psi^m$ by setting:
\begin{align}
\bI^{\, m-1} (t) \, :=& \, \sum_{\substack{\ell_1+\ell_2=m-1 \\ \ell_1 \ge 1}} \dfrac{1}{\ell_2 \, !} \, \left\{ 
\int_{\bT^2 \times \R^+} Y_3^{\ell_2} \, \p_{y_3}^{\ell_2} \widehat{q}_\star^{\, \ell_1,+} (t,y',0,Y_3,0) \, 
{\rm d}y' \, {\rm d}Y_3 \right. \notag \\
& \qquad \qquad \qquad \qquad \qquad \left. +\int_{\bT^2 \times \R^-} 
Y_3^{\ell_2} \, \p_{y_3}^{\ell_2} \widehat{q}_\star^{\, \ell_1,-} (t,y',0,Y_3,0) \, {\rm d}y' \, {\rm d}Y_3 \right\} 
\label{expression_Im(t)-bis} \\
& \, +\sum_{\substack{\ell_1+\ell_2=m \\ \ell_1 \ge 1,\ell_2 \ge 2}} 
\int_{\Omega_0^+} \! \! \int_\bT \uq^{\ell_1,+} (t,y,\theta) \, \cJ^{\ell_2} \, (t,y,\theta) \, 
{\rm d}y \, \dfrac{{\rm d}\theta}{2\, \pi} 
+\int_{\Omega_0^-} \! \! \int_\bT \uq^{\ell_1,-} (t,y,\theta) \, \cJ^{\ell_2} \, (t,y,\theta) \, 
{\rm d}y \, \dfrac{{\rm d}\theta}{2\, \pi} \, .\notag
\end{align}
Since we wish each term $\cI^m(t)$ in \eqref{BKW_norm_pression} to vanish, this will lead from the expression \eqref{expression_Im(t)} 
to determine inductively the slow mean $(\widehat{\uq}^{\, m,\pm}(0))_{m \ge 1}$ of the total pressure by setting:
\begin{equation}
\label{BKW_norm_pression_m}
\int_{\Omega_0^+} \widehat{\uq}^{\, m,+}(t,y,0) \, {\rm d}y \, + \, 
\int_{\Omega_0^-} \widehat{\uq}^{\, m,-}(t,y,0) \, {\rm d}y \, = \, -\bI^{\, m-1} (t) \, ,
\end{equation}
with $\bI^{\, m-1} (t)$ given in \eqref{expression_Im(t)-bis}. In particular, examining the relation \eqref{BKW_norm_pression_m} shows that 
the first three slow means $\widehat{\uq}^{\, m,\pm}(0)$, $m=1,2,3$, should satisfy:
\begin{subequations}
\begin{align}
\int_{\Omega_0^+} \widehat{\uq}^{\, 1,+}(t,y,0) \, {\rm d}y \, + \, 
\int_{\Omega_0^-} \widehat{\uq}^{\, 1,-}(t,y,0) \, {\rm d}y \, =& \, 0 \, ,\label{expression_I1(t)} \\
\int_{\Omega_0^+} \widehat{\uq}^{\, 2,+}(t,y,0) \, {\rm d}y \, + \, 
\int_{\Omega_0^-} \widehat{\uq}^{\, 2,-}(t,y,0) \, {\rm d}y \, =& -\int_{\bT^2 \times \R^+} 
\widehat{q}_\star^{\, 1,+} (t,y',0,Y_3,0) \, {\rm d}y' \, {\rm d}Y_3 \label{expression_I2(t)} \\
& \, -\int_{\bT^2 \times \R^-} \widehat{q}_\star^{\, 1,-} (t,y',0,Y_3,0) \, {\rm d}y' \, {\rm d}Y_3 \, ,\notag \\
\int_{\Omega_0^+} \widehat{\uq}^{\, 3,+}(t,y,0) \, {\rm d}y \, + \, 
\int_{\Omega_0^-} \widehat{\uq}^{\, 3,-}(t,y,0) \, {\rm d}y \, =& -\int_{\bT^2 \times \R^+} 
\widehat{q}_\star^{\, 2,+} (t,y',0,Y_3,0) \, {\rm d}y' \, {\rm d}Y_3 \label{expression_I3(t)} \\
& \, -\int_{\bT^2 \times \R^-} \widehat{q}_\star^{\, 2,-} (t,y',0,Y_3,0) \, {\rm d}y' \, {\rm d}Y_3 \, ,\notag \\
& \, -\int_{\bT^2 \times \R^+} Y_3 \, \p_{Y_3} \widehat{q}_\star^{\, 1,+} (t,y',0,Y_3,0) \, {\rm d}y' \, {\rm d}Y_3 \notag \\
& \, -\int_{\bT^2 \times \R^-} Y_3 \, \p_{Y_3} \widehat{q}_\star^{\, 1,-} (t,y',0,Y_3,0) \, {\rm d}y' \, {\rm d}Y_3 \notag \\
&\, - \int_{\Omega_0^+} \! \! \int_\bT \uq^{\, 1,+} (t,y,\theta) \, \chi'(y_3) \, \psi^2 (t,y',\theta) \, {\rm d}y \, \dfrac{{\rm d}\theta}{2\, \pi} \notag \\
&\, - \int_{\Omega_0^-} \! \! \int_\bT \uq^{\, 1,-} (t,y,\theta) \, \chi'(y_3) \, \psi^2 (t,y',\theta) \, {\rm d}y \, \dfrac{{\rm d}\theta}{2\, \pi} \, .\notag
\end{align}
\end{subequations}
It will turn out that with our choice of initial data, the whole residual total pressure $\uq^{\, 1,\pm}$, including the slow mean 
$\widehat{\uq}^{\, 1,\pm}(0)$ will vanish, and the fast mean $\widehat{q}_\star^{\, 1,\pm}(0)$ will also vanish. Hence the above 
relations \eqref{expression_I2(t)} and \eqref{expression_I3(t)} will reduce to:
\begin{align*}
\int_{\Omega_0^+} \widehat{\uq}^{\, 2,+}(t,y,0) \, {\rm d}y \, + \, 
\int_{\Omega_0^-} \widehat{\uq}^{\, 2,-}(t,y,0) \, {\rm d}y \, =& \, 0 \, , \\
\int_{\Omega_0^+} \widehat{\uq}^{\, 3,+}(t,y,0) \, {\rm d}y \, + \, 
\int_{\Omega_0^-} \widehat{\uq}^{\, 3,-}(t,y,0) \, {\rm d}y \, =& \, 
-\int_{\bT^2 \times \R^+} \widehat{q}_\star^{\, 2,+} (t,y',0,Y_3,0) \, {\rm d}y' \, {\rm d}Y_3 \\
& \, -\int_{\bT^2 \times \R^-} \widehat{q}_\star^{\, 2,-} (t,y',0,Y_3,0) \, {\rm d}y' \, {\rm d}Y_3 \, .
\end{align*}

\section{Summary}

Let us now summarize the equations to be solved. We wish to determine some profiles $(U^{\, m,\pm},\psi^{\, m+1})_{m \ge 0}$ that satisfy:
\begin{itemize}
 \item the fast system \eqref{dev_BKW_int} together with the fast divergence constraint \eqref{s3-cascade_div_H^m+1,pm'} on the 
 magnetic field (the source term $F^{\, m,\pm}$ in \eqref{dev_BKW_int} is explicitly given in \eqref{s3-def_terme_source_F^m,pm} and 
 the source term $F_8^{\, m,\pm}$ in \eqref{s3-cascade_div_H^m+1,pm'} denotes the right hand side of \eqref{s3-cascade_div_H^m+1,pm}),
 \item the jump conditions \eqref{s3-cascade_BKW_bord_matriciel} for the double traces on $\{ y_3=Y_3=0 \}$ (the source terms 
 $G_1^{\, m,\pm},G_2^{\, m,\pm}$ in \eqref{s3-cascade_BKW_bord_matriciel} are explicitly given in  \eqref{s3-def_G_1^m,pm}, 
 \eqref{s3-def_G_2^m,pm}),
 \item the top and bottom boundary conditions \eqref{s3-cond_bords_fixes_uH_3^m,pm} for all Fourier modes of the residual 
 components $\underline{u}_3^{\, m,\pm},\underline{H}_3^{\, m,\pm}$ of the normal velocity and normal magnetic field,
 \item the normalization constraints ${\mathcal I}^m(t)=0$ for the slow mean of the total pressure $q^{\, m,\pm}$, with ${\mathcal I}^m$ 
 given by \eqref{expression_Im(t)}, \eqref{expression_Im(t)-bis}.
\end{itemize}

\noindent All these equations are supplemented with initial conditions in agreement with \eqref{s3-def_cond_init_oscil_psi}. The initial data 
for the front profiles $(\psi^m)_{m \ge 2}$ are\footnote{Recall that in \eqref{s3-cond_init_profils_psi_0^m}, $\psi^2_0$ is assumed to have 
zero mean with respect to $\theta$.}:
\begin{equation}
\label{s3-cond_init_profils_psi_0^m}
\forall \, (x',\theta) \in \bT^3 \, ,\quad \psi^2(0,y',\theta) \, = \, \psi_0^2 (y',\theta) \, ,\quad \text{ and } \quad 
\forall \, m \ge 3 \, , \quad \psi^m(0,y',\theta) \, = \, 0 \, ,
\end{equation}
and the initial data for the fast mean of the tangential components of each $U^{\, m,\pm}$ will be zero:
\begin{equation}
\label{s3-cond_moyenne_initiale_U^m,pm}
\forall \, m \ge 1 \, ,\quad \Pi \, \widehat{U}_\star^{\, m,\pm}(0,y,0) \, = \, 0 \, .
\end{equation}
We shall go back later on in Chapters \ref{chapter4} and \ref{chapter5} to the problem of determining the initial data for the slow mean 
of the profiles $U^{\, m,\pm}$. There is some flexibility there too, but the data have to be compatible with some divergence constraints so 
we have thought it more convenient to examine the determination of initial data for the slow mean $\widehat{\uU}^{\, m,\pm}(0)$, $m \ge 1$, 
when it arises in the analysis.

The cornerstone of the whole program consists in solving the so-called {\it fast} problem \eqref{fast_problem} below in which we focus on 
the fast equations \eqref{dev_BKW_int}, \eqref{s3-cascade_div_H^m+1,pm'} and the jump conditions \eqref{s3-cascade_BKW_bord_matriciel}. 
Depending on whether $m=0$ or $m \ge 1$, the fast problem has to be solved either in the homogeneous or nonhomogeneous case, which 
will be done in Chapter \ref{chapter3} hereafter. We shall go back later in Chapters \ref{chapter4} and \ref{chapter5} to the top and bottom 
boundary conditions and to the total pressure normalization.

\chapter{Analysis of the fast problem}
\label{chapter3}

Constructing a solution to the WKB cascade \eqref{dev_BKW_int}, \eqref{s3-cascade_div_H^m+1,pm'}, \eqref{s3-cascade_BKW_bord_matriciel}, 
\eqref{s3-cond_bords_fixes_uH_3^m,pm} is done inductively. At each step of the induction process, one point in the analysis is to solve a system 
of equations of the form:
\begin{equation}
\label{fast_problem}
\begin{cases}
\cL_f^\pm(\partial) \, U^\pm \, = \, F^\pm \, ,& y_3 \in I^\pm \, ,\, \pm Y_3 >0 \, ,\\
\p_{Y_3} H_3^\pm +\xi_j \, \p_\theta H_j^\pm \, = \, F_8^\pm \, ,& y_3 \in I^\pm \, ,\, \pm Y_3 >0 \, ,\\
B^+ \, U^+|_{y_3=Y_3=0} +B^- \, U^-|_{y_3=Y_3=0} +\partial_\theta \psi \, \ub \, = \, G \, ,
\end{cases}
\end{equation}
where the fast operators $\cL_f^\pm(\partial)$ have ben defined in \eqref{s3-def_L(d)^pm}, the matrices $B^\pm$ have been defined in 
\eqref{s3-def_B^pm} and the vector $\ub$ is given by \eqref{s3-def_cal_b}. We forget temporarily the top and bottom boundary conditions 
on $\Gamma^\pm$ since they can be dealt with more easily than all remaining equations. Both the source terms and the solution in 
\eqref{fast_problem} are real valued.

It is important to observe that in \eqref{fast_problem}, the slow variables $(t,y')$ enter as parameters. The slow normal variable $y_3$ also 
enters as a parameter in the two first equations, but the boundary conditions on $\Gamma_0$ only bear on the trace on $\{ y_3=0 \}$. However, 
for later use, it is useful to consider interior source terms $F^\pm,F_8^\pm$ that also depend on $(t,y)$ and a boundary source term $G$ that 
depends on $(t,y')$. The main purpose is to clarify the functional framework in which \eqref{fast_problem} can be solved. Since the unknown 
front profile $\psi$ in \eqref{fast_problem} only appears through its $\theta$-derivative, it will of course be defined only up to its mean with 
respect to $\theta$ (this mean being a function of $(t,y')$).

We shall refer from now on to \eqref{fast_problem} as the `fast problem'. In this Section, we fix a time $T>0$. The functional spaces $\uS^\pm$, 
$S_\star^\pm$, $S^\pm$ are defined accordingly, see Definition \ref{s3-def_espaces_fonctionnels}. Our main result in this Chapter is the following.

\begin{theorem}
\label{theorem_fast_problem}
Let Assumptions \eqref{s3-hyp_stab_nappe_plane}, \eqref{s3-hyp_xi'_rationnelle}, \eqref{s3-hyp_delta_neq_0}, \eqref{s3-hyp_tau_racine_det_lop} 
be satisfied together with $\tau\neq 0$. Let $F^\pm, F_8^\pm \in S^\pm$ and let $G \in H^\infty ([0,T] \times \bT^2 \times \bT)$. Then the fast problem 
\eqref{fast_problem} has a solution $(U^\pm,\psi) \in S^\pm \times H^\infty ([0,T] \times \bT^2 \times \bT)$ if and only if the following conditions are 
satisfied:
\begin{subequations}
\label{compatibilite_pb_rapide}
\begin{equation}
\label{compatibilite_pb_rapide_a}
\begin{cases}
F_6^+|_{y_3=Y_3=0} \, = \, -b^+ \, \p_\theta G_1 +c^+ \, \p_\theta G_2 \, ,& \\
F_6^-|_{y_3=Y_3=0} \, = \, -b^- \, \p_\theta G_3 +c^- \, \p_\theta G_4 \, ,& 
\end{cases}
\quad \text{\rm (compatibility at the boundary)}
\end{equation}
\begin{equation}
\label{compatibilite_pb_rapide_b}
\widehat{\uF}^\pm(t,y,0) \, = \, 0 \, ,\quad \widehat{\uF}_8^\pm(t,y,0) \, = \, 0 \, ,\quad \text{\rm (compatibility for the slow means)}
\end{equation}
\begin{equation}
\label{compatibilite_pb_rapide_c}
\begin{cases}
u_1^{0,\pm} \, \widehat{F}_{7,\star}^\pm (0) -H_1^{0,\pm} \, \widehat{F}_{8,\star}^\pm (0) 
\, = \, \widehat{F}_{1,\star}^\pm (0) \, , & \\
u_2^{0,\pm} \, \widehat{F}_{7,\star}^\pm (0) -H_2^{0,\pm} \, \widehat{F}_{8,\star}^\pm (0) 
\, = \, \widehat{F}_{2,\star}^\pm (0) \, , & \\
H_1^{0,\pm} \, \widehat{F}_{7,\star}^\pm (0) -u_1^{0,\pm} \, \widehat{F}_{8,\star}^\pm (0) 
\, = \, \widehat{F}_{4,\star}^\pm (0) \, , & \\
H_2^{0,\pm} \, \widehat{F}_{7,\star}^\pm (0) -u_2^{0,\pm} \, \widehat{F}_{8,\star}^\pm (0) 
\, = \, \widehat{F}_{5,\star}^\pm (0) \, , &
\end{cases}
\quad \text{\rm (compatibility for the fast means)}
\end{equation}
\begin{equation}
\label{compatibilite_pb_rapide_d}
\p_{Y_3} F_6^\pm +\xi_j \, \p_\theta F_{3+j}^\pm -\tau \, \p_\theta F_8^\pm \, = \, 0 \, ,\quad 
\text{\rm (compatibility for the divergence of the magnetic field)}
\end{equation}
\begin{multline}
\label{compatibilite_pb_rapide_e}
\forall \, k \neq 0 \, ,\quad 
\int_{\R^+} \mathrm{e}^{-|k| \, Y_3} \, \cL^+(k) \sbt \widehat{F}^+ (t,y',0,Y_3,k) \,  {\rm d}Y_3 \, - \, 
\int_{\R^-} \mathrm{e}^{|k| \, Y_3} \, \cL^-(k) \sbt \widehat{F}^- (t,y',0,Y_3,k) \,  {\rm d}Y_3 \\
+ \, \ell_1^+ \, \widehat{G}_1(t,y',k) \, + \, \ell_2^+ \, \widehat{G}_2(t,y',k) \, + \, \ell_1^- \, \widehat{G}_3(t,y',k) 
\, + \, \ell_2^- \, \widehat{G}_4(t,y',k) \, - \, i \, \tau \, \text{\rm sgn}(k) \, \widehat{G}_5(t,y',k) \, = \, 0 \, ,
\end{multline}
\end{subequations}
where in \eqref{compatibilite_pb_rapide_e}, the vectors $\cL^\pm(k)$ are explicitly defined in \eqref{appA-defRLkpm}, the notation `$\sbt$' 
stands for the Hermitian product between two vectors, and the quantities $\ell_{1,2}^\pm$ are defined in \eqref{s3-def_l1_l2} below. The last 
solvability condition \eqref{compatibilite_pb_rapide_e} will be referred to as an \emph{orthogonality condition} (for the nonzero Fourier modes).

If the solvability conditions \eqref{compatibilite_pb_rapide} are satisfied by the source terms of \eqref{fast_problem}, then \eqref{fast_problem} 
has a solution of the form $(\bU^\pm,0)$, where $\bU^\pm \in S^\pm$ can be chosen such that:
$$
\Pi \, \widehat{\bU}^\pm (0) \, = \, 0 \, ,\quad \text{\rm and} \quad \widehat{\underline{\bU}}^\pm (0) \big|_{y_3=\pm 1} \, = \, 0 \, .
$$
Furthermore, any solution to \eqref{fast_problem} then reads\footnote{The subscript $h$ stands for `homogeneous'.}:
\begin{align*}
U^\pm (t,y,Y_3,\theta) \, = \, \bU^\pm (t,y,Y_3,\theta) \, + \, \uU_h^\pm (t,y) \, & + \, 
\Big( u_{1,\star}^\pm,u_{2,\star}^\pm,0,H_{1,\star}^\pm,H_{2,\star}^\pm,0,0 \Big)^T (t,y,Y_3) \\
&+\sum_{k \in \Z \setminus \{ 0\}} \gamma^\pm (t,y,k) \, {\rm e}^{\mp |k| \, Y_3 +i\, k \, \theta} \, \cR^\pm (k) \, ,
\end{align*}
where $\uU_h^\pm \in \uS^\pm$ is independent of $\theta$, $u_{1,\star}^\pm,u_{2,\star}^\pm,H_{1,\star}^\pm,H_{2,\star}^\pm \in S_\star^\pm$, 
the vectors $\cR^\pm (k)$ are explicitly given in \eqref{appA-defRLkpm}, the coefficients $\gamma^\pm$ satisfy $\gamma^\pm (t,y',0,k) =\pm |k| 
\, \widehat{\psi}(t,y',k)$ for all $(t,y',k) \in [0,T] \times \bT^2 \times \bZ^*$ with the \emph{reality} condition:
$$
\forall \, k \in \Z^* \, ,\quad \gamma^\pm (t,y,-k) \, = \, \overline{\gamma^\pm (t,y,k)} \, ,
$$
and, eventually, the slow mean $\uU_h^\pm$ satisfies the boundary conditions on $\Gamma_0$:
$$
\uu_{h,3}^\pm|_{y_3=0} \, = \, \uH_{h,3}^\pm|_{y_3=0} \, = \, 0 \, ,\quad \uq_h^+|_{y_3=0} \, = \, \uq_h^-|_{y_3=0} \, .
$$
Here we have used the notation $\uU_h^\pm = (\uu_{h,1}^\pm, \uu_{h,2}^\pm, \uu_{h,3}^\pm, \uH_{h,1}^\pm, \uH_{h,2}^\pm, 
\uH_{h,3}^\pm, \uq_h^\pm)^T$.
\end{theorem}

\noindent Let us observe immediately that \eqref{compatibilite_pb_rapide} does not involve any condition on the mean of the boundary source 
term $\widehat{G}(0)$. In the induction process of Chapter \ref{chapter5}, the mean of the front profiles $\widehat{\psi}^m(0)$, $m \ge 2$, will 
be determined by enforcing a solvability condition for the Laplace type problem that will determine the slow mean of the total pressure correctors. 
This solvability condition will be examined separately since it does not enter the analysis of the fast problem. The analysis of the Laplace problem, 
which is reminiscent of the analysis in \cite{SWZ}, will lead to a second order wave type equation for the slow mean of the front $\widehat{\psi}^m(0)$.

The following sections are devoted to the proof of Theorem \ref{theorem_fast_problem}. In order to understand better why \eqref{fast_problem} 
admits nonzero solutions in the homogeneous case, we first go back for a while to the linear stability problem of incompressible current vortex 
sheets and recall why Assumption \eqref{s3-hyp_stab_nappe_plane} yields the existence of \emph{linear} surface waves that are exponentially 
decaying with respect to the normal variable to the current vortex sheet. This will be the opportunity to recall the expression of the so-called 
Lopatinskii determinant, which will be useful later on. The various matrices, eigenvectors and so on involved in the normal mode analysis also 
enter the explicit description of the (leading amplitude in the) weakly nonlinear ansatz \eqref{s3-def_dvlpt_BKW}.

\section{A reminder on the normal mode analysis}

The analysis of the fast problem \eqref{fast_problem} in the \emph{homogeneous} case, that is when $F^\pm \equiv 0$, $F_8^\pm \equiv 0$, 
$G \equiv 0$, is more or less equivalent to the normal mode analysis that is performed for testing the linear stability of the piecewise constant 
solution \eqref{s3-def_U^0,pm} to \eqref{s3-equations_nappes_MHD}. To highlight this, we thus go back for a while to the original (quasilinear) 
system \eqref{s3-equations_nappes_MHD} and linearize the equations around the piecewise constant solution \eqref{s3-def_U^0,pm}. Forgetting 
in this Section about the top and bottom boundaries $\Gamma^\pm$, the linearized problem reads:
$$
\left\{
\begin{array}{r l}
\p_t \dot{u}^\pm \, +(u^{0,\pm} \sg) \, \dot{u}^\pm \, -(H^{0,\pm} \sg) \, \dot{H}^\pm \, +\nabla \dot{q}^\pm \, = \, 0 \, , & 
\text{ in } \bT^2 \times \bR^\pm \, ,\quad t\in [0,T] \, ,\\[0.5ex]
\p_t \dot{H}^\pm \, +(u^{0,\pm} \sg) \, \dot{H}^\pm \, -(H^{0,\pm} \sg) \, \dot{u}^\pm \, = \, 0 \, , & 
\text{ in } \bT^2 \times \bR^\pm \, ,\quad t\in [0,T] \, ,\\[0.5ex]
\dv \dot{u}^\pm(t) \, = \, \dv \dot{H}^\pm(t) \, = \, 0 \, , & \text{ in } \bT^2 \times \bR^\pm \, , \quad t\in [0,T] \, ,\\[0.5ex]
\p_t \dot{\psi} +u_j^{0,\pm} \, \p_{y_j} \dot{\psi} \, = \, \dot{u}_3^\pm \, ,\quad 
H_j^{0,\pm} \, \p_{y_j} \dot{\psi} \, = \, \dot{H}_3^\pm \, ,\quad [ \, \dot{q} \, ] \, = \, 0 \, , & \text{ on } \bT^2 \, ,\quad t\in [0,T] \, .
\end{array}
\right.
$$
For this Section only, we consider the normal variable $y_3 \in \R^\pm$, but the tangential variables $y'$ still lie in $\bT^2$. We recall that 
$j$ and $j'$ refer to tangential spatial coordinates and that we use Einstein summation convention.

We perform a Laplace transform with respect to the time variable ($z:= \tau -i \, \gamma$ is the associated dual variable) and a decomposition 
in Fourier series with respect to the tangential space coordinates $y' \in \bT^2$ (the associated frequencies are $k \in \bZ^2$). This yields the 
problem:
$$
\left\{
\begin{array}{r l}
(z+k_j \, u_j^{0,\pm}) \, \hat{u}_{j'}^\pm \, -(k_j \, H_j^{0,\pm}) \, \hat{H}_{j'}^\pm \, 
+k_{j'} \, \hat{q}^\pm \, = \, 0 \, , & y_3 \in \bR^\pm \, ,\quad j'=1,2 \, ,\\[0.5ex]
i \, (z+k_j \, u_j^{0,\pm}) \, \hat{u}_3^\pm \, -i \, (k_j \, H_j^{0,\pm}) \, \hat{H}_3 \, 
+\p_{y_3} \hat{q}^\pm \, = \, 0 \, , & y_3 \in \bR^\pm \, ,\\[0.5ex]
(z+k_j \, u_j^{0,\pm}) \, \hat{H}^\pm \, -(k_j \, H_j^{0,\pm}) \, \hat{u}^\pm \, = \, 0 \, , & y_3 \in \bR^\pm \, ,\\[0.5ex]
i \, k_j \, \hat{u}_j^\pm +\p_{y_3} \hat{u}_3^\pm \, = i \, k_j \, \hat{H}_j^\pm +\p_{y_3} \hat{H}_3^\pm \, = \, 0 \, , & 
y_3 \in \bR^\pm \, ,\\[0.5ex]
i \, (z+k_j \, u_j^{0,\pm}) \, \hat{\psi} \, = \, \hat{u}_3^\pm \, ,\quad 
i \, k_j \, H_j^{0,\pm} \, \hat{\psi} \, = \, \hat{H}_3^\pm \, ,\quad [ \, \hat{q} \, ] \, = \, 0 \, , & y_3 =0 \, ,
\end{array}
\right.
$$
with what are -at least we hope so- self-explanatory notations. The analysis below follows what has been performed in 
\cite{Syro,Axford1,Chandra,BT,MTT}. We first look for exploding modes in time, that is we assume $z$ to be of negative imaginary 
part ($\gamma>0$). With some easy manipulations, we first use the divergence constraints on $\hat{u}^\pm, \hat{H}^\pm$ in the 
two first equations and derive the second order differential equation for the total pressure\footnote{On the Laplace-Fourier side, 
this is of course the analogue of the Laplace problem that determines the total pressure as the Lagrange multiplier associated 
with the divergence constraints on the velocity field.}:
$$
\left\{
\begin{array}{r l}
\p_{y_3}^2 \hat{q}^\pm \, - \, |k|^2 \, \hat{q}^\pm \, = \, 0 \, , & y_3 \in \bR^\pm \, ,\\[0.5ex]
[ \, \hat{q} \, ] \, = \, 0 \, , & y_3 =0 \, ,\\[0.5ex]
\p_{y_3} \hat{q}^\pm \, = \, \big( (z+k_j \, u_j^{0,\pm})^2 \, -(k_j \, H_j^{0,\pm})^2 \big) \, \hat{\psi} \, , & y_3 =0 \, .
\end{array}
\right.
$$
For a nonzero frequency vector $k$, this problem has a nontrivial solution that is (exponentially) decaying at $|y_3|= +\infty$ if and only if:
\begin{equation}
\label{s3-def_det_lop'}
(z+k_j \, u_j^{0,+})^2 \, +(z+k_j \, u_j^{0,-})^2 \, -(k_j \, H_j^{0,+})^2 \, -(k_j \, H_j^{0,-})^2 =0 \, .
\end{equation}
Provided that we can find the expression of the total pressure, all other quantities are then determined by solving linear equations, so the 
vanishing of the quantity on the left hand side of \eqref{s3-def_det_lop'} determines whether, on the Laplace-Fourier side, the linearized 
equations admit a nontrivial solution. To make the analogy with the theory of hyperbolic initial boundary value problems \cite{K,sakamoto,BS}, 
the polynomial expression in \eqref{s3-def_det_lop'} will be referred to as the Lopatinskii determinant for the current vortex sheet problem.

Under the stability condition \eqref{s3-hyp_stab_nappe_plane} on the reference planar current vortex sheet, the second degree polynomial 
equation \eqref{s3-def_det_lop'} in $z$ has two \emph{simple real} roots. This means that exploding modes in time ($\gamma>0$) do not 
occur, which is good for stability, but the linearized problem admits nontrivial oscillating waves in time ($\gamma=0$). Since these waves 
decay exponentially with respect to the normal variable $y_3$, they are usually referred to as \emph{surface waves}. More algebraic details 
that fit into our framework for weakly nonlinear geometric optics are given in the following Section.

\section{The homogeneous case}

In this Section, we are going to characterize the solutions $(U^\pm,\psi) \in S^\pm \times H^\infty ([0,T] \times \bT^2 \times \bT)$ to the fast 
problem \eqref{fast_problem} in the homogeneous case. We thus wish to determine all solutions to the system:
\begin{equation}
\label{fast_homogeneous}
\begin{cases}
\cL_f^\pm(\partial) \, U^\pm \, = \, 0 \, ,& y_3 \in I^\pm \, ,\, \pm Y_3 >0 \, ,\\
\p_{Y_3} H_3^\pm +\xi_j \, \p_\theta H_j^\pm \, = \, 0 \, ,& y_3 \in I^\pm \, ,\, \pm Y_3 >0 \, ,\\
B^+ \, U^+|_{y_3=Y_3=0} +B^- \, U^-|_{y_3=Y_3=0} +\partial_\theta \psi \, \ub \, = \, 0 \, .
\end{cases}
\end{equation}
Once again, we are only interested in real valued solutions. The result is summarized in the following Proposition.

\begin{proposition}
\label{prop_fast_homogeneous}
The functions $(U^\pm,\psi) \in S^\pm \times H^\infty ([0,T] \times \bT^2 \times \bT)$ that satisfy \eqref{fast_homogeneous} are exactly those 
profiles of the form:
\begin{align*}
U^\pm (t,y,Y_3,\theta) \, = \, \uU^\pm (t,y) & \, 
+\Big( u_{1,\star}^\pm,u_{2,\star}^\pm,0,H_{1,\star}^\pm,H_{2,\star}^\pm,0,0 \Big)^T (t,y,Y_3) \\[0.5ex]
& \, +\sum_{k \in \Z \setminus \{ 0\}} \gamma^\pm (t,y,k) \, {\rm e}^{\mp |k| \, Y_3 +i\, k \, \theta} \, \cR^\pm (k) \, ,
\end{align*}
where $\uU^\pm \in \uS^\pm$ is independent of $\theta$, $u_{1,\star}^\pm,u_{2,\star}^\pm,H_{1,\star}^\pm,H_{2,\star}^\pm \in S_\star^\pm$, 
the vectors $\cR^\pm (k)$ are explicitly given in \eqref{appA-defRLkpm}, the coefficients $\gamma^\pm$ satisfy $\gamma^\pm (t,y',0,k) =\pm 
|k| \, \widehat{\psi}(t,y',k)$ for all $(t,y',k) \in [0,T] \times \bT^2 \times \bZ^*$ and the reality condition:
$$
\forall \, k \in \Z^* \, ,\quad \gamma^\pm (t,y,-k) \, = \, \overline{\gamma^\pm (t,y,k)} \, ,
$$
and, eventually, the slow mean $\uU^\pm$ satisfies the boundary conditions on $\Gamma_0$:
$$
\uu_3^\pm|_{y_3=0} \, = \, \uH_3^\pm|_{y_3=0} \, = \, 0 \, ,\quad \uq^+|_{y_3=0} \, = \, \uq^-|_{y_3=0} \, .
$$

In particular, given any $\psi \in H^\infty ([0,T] \times \bT^2 \times \bT)$, we can construct a pair of profiles $U^\pm \in S^\pm$ such that 
$(U^\pm,\psi)$ satisfies \eqref{fast_homogeneous}.
\end{proposition}

\noindent Proposition \ref{prop_fast_homogeneous} extends to the three-dimensional case some of the calculations performed in \cite{AliHunter}. 
We now give the proof for the sake of completeness.

\begin{proof}[Proof of Proposition \ref{prop_fast_homogeneous}]
The proof is rather elementary and follows from mere algebraic manipulations. We first assume that $(U^\pm,\psi) \in S^\pm \times H^\infty ([0,T] 
\times \bT^2 \times \bT)$ is a solution to \eqref{fast_homogeneous}, try to derive the expression of $(U^\pm,\psi)$, and eventually verify that such 
expressions provide indeed with the only possible solutions to \eqref{fast_homogeneous}. Let us therefore assume that $(U^\pm,\psi) \in S^\pm 
\times H^\infty ([0,T] \times \bT^2 \times \bT)$ is a solution to \eqref{fast_homogeneous}.
\bigskip

$\bullet$ \emph{The residual component}. We first take the limit $|Y_3| \rightarrow +\infty$ and obtain the fast equations:
$$
\begin{cases}
{\mathcal A}^\pm \, \p_\theta \, \uU^\pm \, = \, 0 \, ,& y_3 \in I^\pm \, ,\\
\xi_j \, \p_\theta \uH_j^\pm \, = \, 0 \, ,& y_3 \in I^\pm \, .\\
\end{cases}
$$
The explicit expression of the matrices ${\mathcal A}^\pm$ is given in Appendix \ref{appendixA}. In particular, under Assumptions 
\eqref{s3-hyp_stab_nappe_plane}, \eqref{s3-hyp_xi'_rationnelle}, \eqref{s3-hyp_delta_neq_0}, \eqref{s3-hyp_tau_racine_det_lop}, together with 
$\tau\neq 0$, it is proved in Appendix \ref{appendixA} that the matrices ${\mathcal A}^\pm$ are invertible. This means that, even without using the 
fast divergence constraint $\xi_j \, \p_\theta \uH_j^\pm =0$ on the magnetic field, we necessarily have $\p_\theta \uU^\pm =0$, or in other words:
$$
\uU^\pm (t,y,\theta) \, = \, \uU^\pm (t,y) \, ,
$$
for some functions $\uU^\pm \in H^\infty ([0,T] \times \bT^2 \times I^\pm)$. The divergence constraint on the residual magnetic field is a 
consequence of the remaining seven equations. We shall examine the boundary conditions that must be satisfied by the slow means 
$\uU^\pm$ on $\{ y_3 =0 \}$ later.
\bigskip

$\bullet$ \emph{The surface wave component. Zero Fourier mode}. Subtracting the residual component at $Y_3=\pm\infty$ and taking the average 
with respect to $\theta$ of the fast equations in \eqref{fast_homogeneous}, we find that the fast mean $\widehat{U}_\star^\pm (0)$ must satisfy:
$$
A_3^\pm \, \p_{Y_3} \widehat{U}_\star^\pm (0) \, = \, 0 \, ,\quad \p_{Y_3} \widehat{H}_{3,\star}^\pm (0) \, = \, 0 \, ,
$$
where $(t,y) \in [0,T] \times \bT^2 \times I^\pm$ enter here as parameters. From the explicit expression of the matrices $A_3^\pm$ given in Appendix 
\ref{appendixA} and the exponential decay at $Y_3 = \pm \infty$ of functions in $S_\star^\pm$, we find that the latter equations are satisfied if and only 
if
$$
\widehat{u}_{3,\star}^\pm (0) \, = \, \widehat{H}_{3,\star}^\pm (0) \, = \, \widehat{q}_\star^\pm (0) \, \equiv 0 \, .
$$
In other words, the fast mean $\widehat{U}_\star^\pm (0)$ of $U^\pm$ has the form:
$$
\Big( u_{1,\star}^\pm,u_{2,\star}^\pm,0,H_{1,\star}^\pm,H_{2,\star}^\pm,0,0 \Big)^T (t,y,Y_3) \, ,
$$
as claimed in Proposition \ref{prop_fast_homogeneous}. Since the tangential components of the velocity and magnetic field do not enter the boundary 
conditions in \eqref{fast_homogeneous}, the four scalar functions that define the fast mean are completely free.
\bigskip

Since the normal velocity, normal magnetic field and total pressure of the fast mean $\widehat{U}_\star^\pm (0)$ are zero, 
it is easy to see that the slow mean $\uU^\pm$ must satisfy the boundary conditions on $\Gamma_0$:
$$
\uu_3^\pm|_{y_3=0} \, = \, \uH_3^\pm|_{y_3=0} \, = \, 0 \, ,\quad \uq^+|_{y_3=0} \, = \, \uq^-|_{y_3=0} \, .
$$
These conditions are obtained by computing the mean with respect to $\theta$ on $\bT$ of the boundary conditions in 
\eqref{fast_homogeneous}.
\bigskip

$\bullet$ \emph{The surface wave component. Nonzero Fourier modes}. We now consider the oscillating Fourier modes 
of the surface wave component of $U^\pm$. We consider a nonzero $k \in \bZ$ and compute the $k$-th Fourier coefficient 
with respect to $\theta$ of all equations in \eqref{fast_homogeneous}. We get:
\begin{equation}
\label{fast_homogeneous_k}
\begin{cases}
A_3^\pm \, \p_{Y_3} \widehat{U}_\star^\pm (k) +i \, k \, {\mathcal A}^\pm \, \widehat{U}_\star^\pm (k) \, = \, 0 \, ,& 
y_3 \in I^\pm \, ,\, \pm Y_3 >0 \, ,\\
\p_{Y_3} \widehat{H}_{3,\star}^\pm (k) +i\, k \, \xi_j \, \widehat{H}_{j,\star}^\pm (k) \, = \, 0 \, ,& y_3 \in I^\pm \, ,\, \pm Y_3 >0 \, ,\\
B^+ \, \widehat{U}_\star^+|_{y_3=Y_3=0} (k) +B^- \, \widehat{U}_\star^-|_{y_3=Y_3=0} (k) +i \, k \, \widehat{\psi} (k) \, \ub \, = \, 0 \, .
\end{cases}
\end{equation}
Observe that here, we have already used the fact that the residual component $\uU^\pm$ does not depend on $\theta$. 
The first two equations in \eqref{fast_homogeneous_k} are reminiscent of the normal mode analysis performed to study 
the linear stability of the planar current vortex sheet \eqref{s3-def_U^0,pm}. Indeed, if we write down the equations satisfied 
by the Fourier coefficient $\widehat{U}_\star^\pm (k)$ and use the divergence constraints for the velocity and magnetic field, we 
get\footnote{This system is similar to the one obtained when performing the normal mode analysis in the previous Section, 
with corresponding frequencies $z=\tau$ and $k_j =k\, \xi_j$.}:
\begin{equation}
\label{s3-sys_equiv_pb_lin}
\begin{cases}
c^\pm \, \hatu_{j,\star}^\pm - b^\pm \, \hatH_{j,\star}^\pm + \xi_j \, \hatq_\star^{\, \pm} \, = \, 0 \, , & j=1,2, \\
\p_{Y_3} \hatq_\star^{\, \pm} + i \, k \, c^\pm \, \hatu_{3,\star}^\pm - i\, k \, b^\pm \, \hatH_{3,\star}^\pm \, = \, 0 \, , & \\
-b^\pm \, \hatu_\star^\pm + c^\pm \, \hatH_\star^\pm \, = \, 0 \, , & 
\end{cases}
\end{equation}
together with the divergence constraints
\begin{subequations}\label{s3-contrainte_div_k}
\begin{align}
& \p_{Y_3} \hatu_{3,\star}^1 \, + \, i \, k \, \xi_j \, \hatu_{j,\star}^1 \, = \, 0 \, ,\label{s3-contrainte_div_u_k} \\
& \p_{Y_3} \hatH_{3,\star}^1 \, + \, i \, k \, \xi_j \, \hatH_{j,\star}^1 \, = \, 0 \, .\label{s3-contrainte_div_H_k}
\end{align}
\end{subequations}

Now, using \eqref{s3-contrainte_div_k} in \eqref{s3-sys_equiv_pb_lin}, and recalling the boundary conditions in 
\eqref{fast_homogeneous_k}, we find that the total pressure must satisfy the (seemingly overdetermined) elliptic 
system\footnote{We recall that the tangential frequency vector $\xi'$ has been chosen with norm $1$.}:
\begin{equation}
\label{s3-pb_laplace_q^1,pm_k}
\begin{cases}
\p_{Y_3}^2 \hatq_\star^{\, \pm} - k^2 \, \hatq_\star^{\, \pm} \, = \, 0 \, , & y_3 \in I^\pm \, ,\quad Y_3\gtrless 0 \, , \\[0.5ex]
\hatq_\star^{\, +}|_{y_3=Y_3=0} \, =  \, \hatq_\star^{\, -}|_{y_3=Y_3=0} \, , & \\[0.5ex]
\p_{Y_3} \hatq_\star^{\, \pm}|_{y_3=Y_3=0} \, = \, k^2 \, \big( (c^\pm)^2 - (b^\pm)^2 \big) \, \hatpsi \, . & 
\end{cases}
\end{equation}
Since we are looking for exponentially decaying profiles (in $Y_3$), the second order differential equation for $\hatq_\star^{\, \pm}$ 
gives the general expression\footnote{Here we use the fact that both quantities $(c^\pm)^2 - (b^\pm)^2$ are nonzero, see Appendix 
\ref{appendixA} for the verification of this property.}:
\begin{equation}
\label{s3-expr_hatq^1,pm_y3_qcq}
\hatq_\star^{\, \pm} (t,y,Y_3,k) \, = \, \big( (b^\pm)^2 - (c^\pm)^2 \big) \, \gamma^\pm(t,y,k) \, \mathrm{e}^{\mp |k| \, Y_3} \, , \quad 
\forall \, (t,y,Y_3) \in [0,T] \times \bT^2 \times I^\pm \times \bR^\pm \, .
\end{equation}
Let us notice that $y_3$ is a parameter in \eqref{s3-expr_hatq^1,pm_y3_qcq}. Hence the only remaining problem is to 
take the double trace $y_3=Y_3=0$ in \eqref{s3-expr_hatq^1,pm_y3_qcq} and to determine whether we can solve the 
boundary conditions in \eqref{s3-pb_laplace_q^1,pm_k}. The first boundary condition in \eqref{s3-pb_laplace_q^1,pm_k} 
gives:
\begin{equation*}
\gamma^+ (t,y',0,k) \, = \, \gamma^-(t,y',0,k) \, .
\end{equation*}
Therefore the remaining boundary conditions in \eqref{s3-pb_laplace_q^1,pm_k} will be satisfied if and only if (we recall $k \neq 0$ 
and we are looking for nontrivial solutions, \emph{i.e.} $\hatq_\star^{\, \pm} \not\equiv 0$ and $\hatpsi \not\equiv 0$ for at least one 
nonzero Fourier mode):
\begin{equation}
\label{s3-egalite_b^2_c^2}
(c^+)^2 \, + \, (c^-)^2 \, = \, (b^+)^2 \, + \, (b^-)^2,
\end{equation}
hence Assumption \eqref{s3-hyp_tau_racine_det_lop}. As in the theory of initial boundary value problems for hyperbolic systems 
\cite{K,sakamoto,BS}, the identity \eqref{s3-egalite_b^2_c^2} is similar to the cancellation of some Lopatinskii determinant (see 
\cite{MTT}) that we can define as follows:
\begin{equation}\label{s3-def_det_lop}
\begin{aligned}
\Delta(\tau,\xi_1,\xi_2) \, :=& \, \, (c^+)^2 + (c^-)^2 - (b^+)^2 - (b^-)^2 \\
=& \, \, 2 \, \tau^2 \, + \, 2 \, (a^+ + a^-) \, \tau + (a^+)^2 + (a^-)^2 - (b^+)^2 - (b^-)^2 \, .
\end{aligned}
\end{equation}
This is nothing but a polynomial of degree 2 in $\tau$, so the roots can be easily computed. Under the stability assumption 
\eqref{s3-hyp_stab_nappe_plane}, we can show that $\Delta$ has exactly two \emph{simple real} roots $\tau_1 = \tau_1 (\xi_1,\xi_2)$ 
and $\tau_2 = \tau_2(\xi_1,\xi_2)$; only the \emph{weak} Lopatinskii condition is fulfilled. Observe that there is no reason why 
$\tau = 0$ could not be a root to \eqref{s3-def_det_lop}, but since the case $\tau=0$ induces a different parametrization of 
some eigenspaces (for instance ${\mathcal A}^\pm$ is not invertible if $\tau=0$), we need to exclude this possibility.

We refer to \cite{BS} for more details on the possible degeneracies of the Uniform Lopatinskii Condition for hyperbolic initial 
boundary value problems. Any of the two triples of frequencies $(\tau_{1,2}(\xi),\xi_1,\xi_2)$ will be called \emph{elliptic} because 
it corresponds to a case where the nontrivial solutions to the linearized equations have exponential decay with respect to the 
normal variable to the current vortex sheet. For `standard' hyperbolic initial boundary value problems, this type of situation has 
been studied in \cite{S-T}.
\bigskip

Once we have the expression \eqref{s3-expr_hatq^1,pm_y3_qcq} for the total pressure, it is a mere algebra exercise to verify 
that the solution to \eqref{fast_homogeneous_k} is given by:
$$
\widehat{U}^\pm (t,y,Y_3,k) \, = \, \gamma^\pm (t,y,k) \, \mathrm{e}^{\mp |k| \, Y_3} \, {\mathcal R}^\pm (k) \, ,
$$
with the vector ${\mathcal R}^\pm (k)$ defined in \eqref{appA-defRLkpm}. Our normalization \eqref{s3-expr_hatq^1,pm_y3_qcq} 
for the coefficients $\gamma^\pm$ yields the (rather simple) relation $\gamma^\pm (t,y',0,k) =\pm |k| \, \widehat{\psi} (t,y',k)$ on 
$\Gamma_0$. The reality condition for $\gamma^\pm$ comes from our choice of ${\mathcal R}^\pm (k)$ that also satisfies, see 
\eqref{appA-defRLkpm}:
$$
{\mathcal R}^\pm (-k) \, = \, \overline{{\mathcal R}^\pm (k)} \, .
$$

Conversely, it is elementary to verify that any function of the form given in Proposition \ref{prop_fast_homogeneous} 
does indeed satisfy the homogeneous fast system \eqref{fast_homogeneous}, which completes the proof of Proposition 
\ref{prop_fast_homogeneous}.
\end{proof}

Proposition \ref{prop_fast_homogeneous} already yields part of the results claimed in Theorem \ref{theorem_fast_problem}. 
Namely, since we now have a complete parametrization of the solutions to \eqref{fast_problem} in the homogeneous case, 
by linearity it is sufficient to prove that the conditions \eqref{compatibilite_pb_rapide} are necessary and sufficient for the 
existence of \emph{one} solution to \eqref{fast_problem}. Let us also observe that if, for some given source terms 
$F^\pm,F_8^\pm,G$ in \eqref{fast_problem}, there exists a solution $(U^\pm,\psi) \in S^\pm \times H^\infty$, then we can 
always subtract a solution $(V^\pm,\psi)$ to the homogeneous problem \eqref{fast_homogeneous} with the \emph{same} 
function $\psi$. Consequently, showing that \eqref{fast_problem} admits one solution $(U^\pm,\psi)$ amounts to showing 
that \eqref{fast_problem} admits one solution of the form $(\bU^\pm,0)$ with $\bU^\pm \in S^\pm$.

\section{The inhomogeneous case}

In this Section, we complete the proof of Theorem \ref{theorem_fast_problem} by using the reduction to the case $\psi \equiv 0$ 
as explained above. We thus focus on the fast problem, whose unknown is now denoted $\bU^\pm =(\cU^\pm,\cH^\pm,\cQ^\pm)^T$:
\begin{equation}
\label{fast_problem'}
\begin{cases}
\cL_f^\pm(\partial) \, \bU^\pm \, = \, F^\pm \, ,& y_3 \in I^\pm \, ,\, \pm Y_3 >0 \, ,\\
\p_{Y_3} \cH_3^\pm +\xi_j \, \p_\theta \cH_j^\pm =F_8^\pm \, ,& y_3 \in I^\pm \, ,\, \pm Y_3 >0 \, ,\\
B^+ \, \bU^+|_{y_3=Y_3=0} +B^- \, \bU^-|_{y_3=Y_3=0} \, = \, G \, ,
\end{cases}
\end{equation}
and we try to determine necessary and sufficient conditions on the source terms such that \eqref{fast_problem'} has (at least) 
one solution.

\subsection{Necessary conditions for solvability}

We first exhibit several more or less obvious necessary solvability conditions for \eqref{fast_problem'}. Let us first note 
that the sixth equation in the system $\cL_f^\pm(\partial) \, \bU^\pm =F^\pm$ reads\footnote{We shall repeatedly use in this 
Section the expression of the matrices $A_3^\pm$, ${\mathcal A}^\pm$ that are given in Appendix \ref{appendixA}.}:
$$
-b^\pm \, \p_\theta \cU_3^\pm +c^\pm \, \p_\theta \cH_3^\pm \, = \, F_6^\pm \, .
$$
Since the slow and fast normal variables $y_3$ and $Y_3$ enter as parameters here, we can take the trace of this equation 
on $y_3=Y_3=0$, and use the boundary conditions in \eqref{fast_problem'}:
$$
\cU_3^+|_{y_3=Y_3=0} \, = \, G_1 \, ,\quad \cH_3^+|_{y_3=Y_3=0} \, = \, G_2 \, ,\quad 
\cU_3^-|_{y_3=Y_3=0} \, = \, G_3 \, ,\quad \cH_3^-|_{y_3=Y_3=0} \, = \, G_4 \, .
$$
We get the system \eqref{compatibilite_pb_rapide_a}, which is a compatibility condition between some of the boundary and 
interior source terms in \eqref{fast_problem'}.
\bigskip

We now take the limit $Y_3=\pm \infty$ in the `interior' equations of \eqref{fast_problem'}, and then take the mean with respect to 
$\theta \in \bT$. We get the relations \eqref{compatibilite_pb_rapide_b}, which  mean that the residual components of the interior 
source terms $\uF^\pm,\uF_8^\pm$ should be `purely oscillating' in $\theta$.
\bigskip

Let us now project the interior equations of \eqref{fast_problem'} on the surface wave components $S_\star^\pm$ (meaning that 
we subtract the limit of all quantities at $Y_3=\pm \infty$), and take the mean with respect to $\theta$ on $\bT$. We get:
$$
A_3^\pm \, \p_{Y_3} \widehat{\bU}_\star^\pm (0) \, = \, \widehat{F}_\star^\pm (0) \, ,\quad 
\p_{Y_3} \widehat{\cH}_{3,\star}^\pm (0) \, = \, \widehat{F}_{8,\star}^\pm (0) \, .
$$
Using the explicit expression of the matrices $A_3^\pm$ given in Appendix \ref{appendixA}, we find that the noncharacteristic 
components of the vector $\widehat{\cU}_\star^\pm (0)$, that is, the normal velocity, normal magnetic field and total pressure, 
are given by solving:
$$
\p_{Y_3} \widehat{\cQ}_\star^\pm (0) \, = \, \widehat{F}_{3,\star}^\pm (0) \, ,\quad 
\p_{Y_3} \widehat{\cU}_{3,\star}^\pm (0) \, = \, \widehat{F}_{7,\star}^\pm (0) \, ,\quad 
\p_{Y_3} \widehat{\cH}_{3,\star}^\pm (0) \, = \, \widehat{F}_{8,\star}^\pm (0) \, ,
$$
with zero `boundary conditions' at $Y_3=\pm \infty$, and it then remains to verify the five algebraic relations:
\begin{align*}
0 & \, = \,  \widehat{F}_{6,\star}^\pm (0) \, ,& \\
u_j^{0,\pm} \, \p_{Y_3} \widehat{\cU}_{3,\star}^\pm (0) -H_j^{0,\pm} \, \p_{Y_3} \widehat{\cH}_{3,\star}^\pm (0) 
& \, = \, \widehat{F}_{j,\star}^\pm (0) \, ,\quad & j=1,2 \, ,\\
H_j^{0,\pm} \, \p_{Y_3} \widehat{\cU}_{3,\star}^\pm (0) -u_j^{0,\pm} \, \p_{Y_3} \widehat{\cH}_{3,\star}^\pm (0) 
& \, = \, \widehat{F}_{3+j,\star}^\pm (0) \, ,\quad & j=1,2 \, .
\end{align*}
The necessary solvability conditions for determining the fast mean $\widehat{\bU}_\star^\pm (0)$ therefore read:
\begin{equation}
\label{comp_fast_2}
\widehat{F}_{6,\star}^\pm (0) \, = \, 0 \, ,\quad 
u_j^{0,\pm} \, \widehat{F}_{7,\star}^\pm (0) -H_j^{0,\pm} \, \widehat{F}_{8,\star}^\pm (0) \, = \, \widehat{F}_{j,\star}^\pm (0) \, ,\quad 
H_j^{0,\pm} \, \widehat{F}_{7,\star}^\pm (0) -u_j^{0,\pm} \, \widehat{F}_{8,\star}^\pm (0) \, = \, \widehat{F}_{3+j,\star}^\pm (0) \, .
\end{equation}
The three remaining components $\widehat{F}_{3,\star}^\pm (0)$, $\widehat{F}_{7,\star}^\pm (0)$, $\widehat{F}_{8,\star}^\pm (0)$ 
are free. Let us observe that \eqref{comp_fast_2} is more restrictive than the conditions \eqref{compatibilite_pb_rapide_c} 
given in Theorem \ref{theorem_fast_problem}. We shall explain why \eqref{compatibilite_pb_rapide_c}, together with 
\eqref{compatibilite_pb_rapide_d} are actually sufficient conditions later on.

The solvability conditions \eqref{compatibilite_pb_rapide_b}, \eqref{comp_fast_2} bear on the slow and fast means with 
respect to $\theta$. The following compatibility conditions will rather bear on the oscillating part with respect to $\theta$ 
(as for \eqref{compatibilite_pb_rapide_a}).
\bigskip

Let us now examine more closely the fast equations in \eqref{fast_problem'}. Using the expressions of the matrices 
${\mathcal A}^\pm$ and $A_3^\pm$ in Appendix \ref{appendixA}, the fourth, fifth, sixth and seventh equations in the 
system $\cL_f^\pm(\partial) \, \bU^\pm =F^\pm$ equivalently read:
$$
\begin{cases}
-u_1^{0,\pm} \, \p_{Y_3} \cH_3^\pm -b^\pm \, \p_\theta \cU_1^\pm +(\tau +\xi_2 \, u_2^{0,\pm}) \, \p_\theta \cH_1^\pm 
-\xi_1 \, u_1^{0,\pm} \, \p_\theta \cH_2^\pm & \, = \, F_4^\pm -H_1^{0,\pm} \, F_7^\pm \, , \\
-u_2^{0,\pm} \, \p_{Y_3} \cH_3^\pm -b^\pm \, \p_\theta \cU_2^\pm -\xi_1 \, u_2^{0,\pm} \, \p_\theta \cH_1^\pm 
+(\tau +\xi_1 \, u_1^{0,\pm}) \, \p_\theta \cH_2^\pm  & \, = \, F_5^\pm -H_2^{0,\pm} \, F_7^\pm \, , \\
-b^\pm \, \p_\theta \, \cU_3^\pm +c^\pm \, \p_\theta \cH_3^\pm & \, = \, F_6^\pm \, , \\
\p_{Y_3} \cU_3^\pm +\xi_j \, \p_\theta \cU_j^\pm & \, = \, F_7^\pm \, .
\end{cases}
$$
Applying $\xi_1 \, \p_\theta$ to the first equation, $\xi_2 \, \p_\theta$ to the second equation, $\p_{Y_3}$ to the third 
one and adding all three quantities, we end up with:
$$
\tau \, \p_\theta \Big( \p_{Y_3} \cH_3^\pm +\xi_j \, \p_\theta \cH_j^\pm \Big) \, = \, 
\p_{Y_3} F_6^\pm +\xi_j \, \p_\theta F_{3+j}^\pm \, .
$$
This means that the source terms in \eqref{fast_problem'} must satisfy:
\begin{equation}
\label{comp_fast_3}
\p_{Y_3} F_6^\pm +\xi_j \, \p_\theta F_{3+j}^\pm -\tau \, \p_\theta F_8^\pm \, = \, 0 \, ,
\end{equation}
which includes, by taking the mean with respect to $\theta$, one of the conditions in \eqref{comp_fast_2}, namely 
$\widehat{F}_{6,\star}^\pm (0)=0$. The condition \eqref{comp_fast_3} is a compatibility between the fast divergence 
of the magnetic field and the source term in the fast equations for the magnetic field. Such a compatibility condition 
arises because there is no Lagrange multiplier associated with the divergence constraint in the evolution equation 
for the magnetic field (this constraint is meant to be propagated in time due to the `curl' form of the original equation).
\bigskip

Let us now exhibit a convenient duality formula, which is some kind of a Fredholm alternative for the solvability 
of \eqref{fast_problem'}. The analysis below is reminiscent of what has been done in \cite{BC1,BC2,CW}. We 
consider a \emph{nonzero} $k \in \bZ$ and compute the $k$-th Fourier coefficient of all equations in \eqref{fast_problem'}. 
We obtain\footnote{Here we forget temporarily about the divergence constraint on the magnetic field since it is 
meant to be recovered in the end by the compatibility condition \eqref{comp_fast_3} for nonzero Fourier modes.}:
\begin{equation}
\label{fast_problem_k}
\begin{cases}
(A_3^\pm \, \p_{Y_3} +i \, k \, {\mathcal A}^\pm) \, \widehat{\bU}^\pm(k) \, = \, \widehat{F}^\pm(k) \, ,& 
y_3 \in I^\pm \, ,\, \pm Y_3 >0 \, ,\\
B^+ \, \widehat{\bU}^+(k)|_{y_3=Y_3=0} +B^- \, \widehat{\bU}^-(k)|_{y_3=Y_3=0} \, = \, \widehat{G}(k) \, .
\end{cases}
\end{equation}
In particular, taking the trace on $\Gamma_0$ of the fast equation in \eqref{fast_problem_k}, we find that there exists 
a smooth and bounded (with respect to $Y_3$) solution $W^\pm$ to the system:
\begin{equation}
\label{fast_problem_k'}
\begin{cases}
(A_3^\pm \, \p_{Y_3} +i \, k \, {\mathcal A}^\pm) \, W^\pm \, = \, \widehat{F}^\pm(k)|_{y_3=0} \, ,& \pm Y_3 >0 \, ,\\
B^+ \, W^+(0) +B^- \, W^-(0) \, = \, \widehat{G}(k) \, .
\end{cases}
\end{equation}
Observe that in \eqref{fast_problem_k'}, the only remaining parameters, which we have omitted to write explicitly, 
are $(t,y')$. For convenience, we introduce the notation $L_k^\pm :=A_3^\pm \, \p_{Y_3} +i \, k \, {\mathcal A}^\pm$, 
which corresponds to the action of the fast operator $\cL_f^\pm (\p)$ on the $k$-th Fourier mode with respect to 
$\theta$. We also use the notation $(L_k^\pm)^*$ to denote the (formal) adjoint operator $-(A_3^\pm)^T \, \p_{Y_3} 
-i \, k \, ({\mathcal A}^\pm)^T$. We use at last the notation `$\, \sbt \,$' for the Hermitian product in $\bC^7$:
$$
\forall \, X, Y \in \bC^7 \, ,\quad X \sbt Y \, := \, \sum_{j=1}^7 \overline{X_j} \, Y_j \, .
$$

Let us start with the (bounded) solution $W^\pm$ to the system \eqref{fast_problem_k'}. For a pair of sufficiently smooth 
and (exponentially) decaying at infinity test functions $V^\pm$, we have:
\begin{multline}
\label{dualite_1}
\int_0^{+\infty} V^+ \sbt (L_k^+ W^+) \, {\rm d}Y_3 \, - \, \int_{-\infty}^0 V^- \sbt (L_k^- W^-) \, {\rm d}Y_3 \\
= \, \int_0^{+\infty} (L_k^+)^* \, V^+ \sbt W^+ \, {\rm d}Y_3 \, - \, \int_{-\infty}^0 (L_k^-)^* \, V^- \sbt W^- \, {\rm d}Y_3 
\, - \, V^+(0) \sbt A_3^+ \, W^+(0) \, - \, V^-(0) \sbt A_3^- \, W^-(0) \, .
\end{multline}
We make the boundary terms at $Y_3=0$ explicit thanks to the expression of $A_3^\pm$ (see Appendix \ref{appendixA}), and use 
the fact that $W^\pm$ satisfies \eqref{fast_problem_k'}. This yields:
\begin{align}
\int_0^{+\infty} V^+ \sbt \widehat{F}^+(k)|_{y_3=0} \, {\rm d}Y_3 \, -& \, \int_{-\infty}^0 V^- \sbt \widehat{F}^-(k)|_{y_3=0} \, {\rm d}Y_3 
\notag \\
=& \, {\color{blue} \int_0^{+\infty} (L_k^+)^* \, V^+ \sbt W^+ \, {\rm d}Y_3} \, 
- \, {\color{blue} \int_{-\infty}^0 (L_k^-)^* \, V^- \sbt W^- \, {\rm d}Y_3} \notag \\
& \, -{\color{blue} \Big( \overline{v_3^+(0)} \, + \, \overline{v_3^-(0)} \Big) \, q^-(0)} \notag \\
& \, -\Big( u_j^{0,+} \, \overline{v_j^+(0)} +H_j^{0,+} \, \overline{B_j^+(0)} +\overline{p^+(0)} \Big) \, \widehat{G}_1(k) \label{dualite_1'} \\
& \, -\Big( u_j^{0,-} \, \overline{v_j^-(0)} +H_j^{0,-} \, \overline{B_j^-(0)} +\overline{p^-(0)} \Big) \, \widehat{G}_3(k) \notag \\
& \, +\Big( H_j^{0,+} \, \overline{v_j^+(0)} +u_j^{0,+} \, \overline{B_j^+(0)} \Big) \, \widehat{G}_2(k) 
+\Big( H_j^{0,-} \, \overline{v_j^-(0)} +u_j^{0,-} \, \overline{B_j^-(0)} \Big) \, \widehat{G}_4(k) \notag \\
& \, -\overline{v_3^+(0)} \, \widehat{G}_5(k) \, ,\notag 
\end{align}
where the coordinates of $V^\pm$ are denoted $(v^\pm,B^\pm,p^\pm)$ and the seventh coordinate of $W^\pm$ is denoted 
$q^\pm$. The goal is now to choose the test functions $V^\pm$ such that all the blue terms in \eqref{dualite_1'} -those where 
$W^\pm$ still appears- vanish. This leads us to introducing the so-called \emph{dual} problem:
\begin{equation}
\label{s3-sys_L_k^pm,*}
\begin{cases}
(L_k^\pm)^* \, V^\pm \, = \, 0, & \pm Y_3>0 \, , \\
v_3^+ (0) \, + \, v_3^- (0) \, = \, 0 \, . & 
\end{cases}
\end{equation}
After several calculations which are quite similar to those that have been done in the proof of Proposition \ref{prop_fast_homogeneous}, 
we find that there is a one-dimensional space of solutions to \eqref{s3-sys_L_k^pm,*}, that is spanned by the pair:
\begin{equation}
\label{s3-def_V^pm}
V^\pm(Y_3) \, := \, \mathrm{e}^{\mp |k| \, Y_3} \, \cL^\pm(k) \, ,
\end{equation}
where the vectors $\cL^\pm(k) \in \bC^7$ are explicitly defined by \eqref{appA-defRLkpm}. Plugging the expression \eqref{s3-def_V^pm} 
of $V^\pm$ in \eqref{dualite_1'}, and defining the quantities:
\begin{equation}
\label{s3-def_l1_l2}
\ell_1^\pm \, := \, 2 \, (b^\pm)^2 - \tau \, c^\pm \, ,\quad \text{ and } \quad \ell_2^\pm \, := \, - (a^\pm +c^\pm)\, b^\pm \, ,
\end{equation}
we find that a necessary condition for \eqref{fast_problem_k'} to have a solution is:
\begin{multline}
\label{s3-cond_ortho}
\int_0^{+\infty} \mathrm{e}^{-|k| \, Y_3} \, \cL^+(k) \sbt \widehat{F}^+(k)|_{y_3=0} \, {\rm d}Y_3 
\, - \, \int_{-\infty}^0 \mathrm{e}^{|k| \, Y_3} \, \cL^-(k) \sbt \widehat{F}^-(k)|_{y_3=0} \, {\rm d}Y_3 \\
+ \, \ell_1^+ \, \widehat{G}_1(k) \, + \, \ell_2^+ \, \widehat{G}_2(k) \, + \, \ell_1^- \, \widehat{G}_3(k) 
\, + \, \ell_2^- \, \widehat{G}_4(k) \, - \, i \, \tau \, \text{sgn}(k) \, \widehat{G}_5(k) \, = \, 0 \, .
\end{multline}

At this stage, we have shown that if \eqref{fast_problem'} has a solution (which is equivalent to assuming that the original fast 
problem \eqref{fast_problem} has a solution), then the conditions \eqref{compatibilite_pb_rapide} must be satisfied by the source 
terms.

We assume from now on that the source terms in \eqref{fast_problem'} satisfy \eqref{compatibilite_pb_rapide} and wish to show 
that there exists a solution $\bU^\pm \in S^\pm$ to \eqref{fast_problem'}. This is done by separating the zero Fourier mode from 
the nonzero modes and by constructing an explicit solution to \eqref{fast_problem'} in each case.

\subsection{Solvability I. Zero Fourier mode}

We first show that we can solve the projection of \eqref{fast_problem'} on the zero Fourier mode. Our goal is to construct a 
pair of functions $\widehat{\bU}^\pm (0) \in S^\pm$ that does not depend on $\theta$, and that satisfies:
\begin{equation}
\label{fast_problem_0}
\begin{cases}
A_3^\pm \, \p_{Y_3} \widehat{\bU}^\pm (0) \, = \, \widehat{F}^\pm(0) \, ,& y_3 \in I^\pm \, ,\, \pm Y_3 >0 \, ,\\
\p_{Y_3} \widehat{\cH}_3^\pm (0) \, = \, \widehat{F}_8^\pm(0) \, ,& y_3 \in I^\pm \, ,\, \pm Y_3 >0 \, ,\\
B^+ \, \widehat{\bU}^+ (0)|_{y_3=Y_3=0} +B^- \, \widehat{\bU}^- (0)|_{y_3=Y_3=0} \, = \, \widehat{G}(0) \, .
\end{cases}
\end{equation}
Here and below, the coordinates of $\widehat{\bU}^\pm (0)$ are denoted $\cU^\pm,\cH^\pm,\cQ^\pm$ for the velocity, 
magnetic field and total pressure respectively (we omit to recall the hat notation and the reference to the zero Fourier 
mode for clarity). We also wish to show that solving \eqref{fast_problem_0} is possible with the additional constraint 
$\Pi \, \widehat{\bU}^\pm (0) =0$ (zero tangential components for the velocity and magnetic field) and 
$\widehat{\underline{\bU}}^\pm (0)|_{y_3 =\pm 1} =0$ (zero trace for the residual component on $\Gamma^\pm$).

Solving \eqref{fast_problem_0} is rather easy. Using the solvability condition \eqref{compatibilite_pb_rapide_b}, 
we already know that $\widehat{F}^\pm(0)=\widehat{F}_\star^\pm(0)$ and $\widehat{F}_8^\pm(0)=\widehat{F}_{8,\star}^\pm(0)$. 
We thus define the fast normal velocity, normal magnetic field and total pressure by:
\begin{align*}
&\cU_{3,\star}^-(t,y,Y_3) \, := \, \int_{-\infty}^{Y_3} \widehat{F}_{7,\star}^-(t,y,Y,0) \, {\rm d}Y \, ,\quad 
&\cU_{3,\star}^+(t,y,Y_3) \, := \, -\int^{+\infty}_{Y_3} \widehat{F}_{7,\star}^+(t,y,Y,0) \, {\rm d}Y \, ,\\
&\cH_{3,\star}^-(t,y,Y_3) \, := \, \int_{-\infty}^{Y_3} \widehat{F}_{8,\star}^-(t,y,Y,0) \, {\rm d}Y \, ,\quad 
&\cH_{3,\star}^+(t,y,Y_3) \, := \, -\int^{+\infty}_{Y_3} \widehat{F}_{8,\star}^+(t,y,Y,0) \, {\rm d}Y \, ,\\
&\cQ_\star^-(t,y,Y_3) \, := \, \int_{-\infty}^{Y_3} \widehat{F}_{3,\star}^-(t,y,Y,0) \, {\rm d}Y \, ,\quad 
&\cQ_\star^+(t,y,Y_3) \, := \, -\int^{+\infty}_{Y_3} \widehat{F}_{3,\star}^+(t,y,Y,0) \, {\rm d}Y \, .
\end{align*}
The above formulas define some functions $\cU_{3,\star}^\pm,\cH_{3,\star}^\pm,\cQ_\star^\pm \in S_\star^\pm$ 
that do not depend on $\theta$. Because of the compatibility conditions \eqref{compatibilite_pb_rapide_c} and 
\eqref{compatibilite_pb_rapide_d} (which implies $\widehat{F}_{6,\star}^\pm \equiv 0$), we find that the vectors:
$$
\widehat{\bU}_\star^\pm (0) \, := \, \Big( 0,0,\cU_{3,\star}^\pm,0,0,\cH_{3,\star}^\pm,\cQ_\star^\pm \Big)^T \, ,
$$
satisfy the fast equations in \eqref{fast_problem_0}, namely:
$$
A_3^\pm \, \p_{Y_3} \widehat{\bU}_\star^\pm (0) \, = \, \widehat{F}_\star^\pm(0) \, = \, \widehat{F}^\pm(0) \, ,\quad 
\p_{Y_3} \widehat{\cH}_{3,\star}^\pm (0) \, = \, \widehat{F}_{8,\star}^\pm(0) \, = \, \widehat{F}_8^\pm(0) \, ,
$$
for all $(t,y,Y_3) \in [0,T] \times \bT^2 \times I^\pm \times \R^\pm$. It remains to add some slow functions 
$\widehat{\underline{\bU}}^\pm (0) (t,y)$ in order to satisfy the boundary conditions in \eqref{fast_problem_0}. Indeed, 
any choice of functions $\widehat{\underline{\bU}}^\pm (0)$ that are independent of $Y_3$ will not modify the fulfillment 
of the fast equations in \eqref{fast_problem_0}. We thus only need to determine $\widehat{\underline{\bU}}^\pm (0)$ 
such that:
\begin{align*}
\underline{\cU}_3^+ (t,y',0) \, &= \, \widehat{G}_1 (t,y',0) \, - \, \cU_{3,\star}^+ (t,y',0,0) \, ,& \\
\underline{\cH}_3^+ (t,y',0) \, &= \, \widehat{G}_2 (t,y',0) \, - \, \cH_{3,\star}^+ (t,y',0,0) \, ,& \\
\underline{\cU}_3^- (t,y',0) \, &= \, \widehat{G}_3 (t,y',0) \, - \, \cU_{3,\star}^- (t,y',0,0) \, ,& \\
\underline{\cH}_3^- (t,y',0) \, &= \, \widehat{G}_4 (t,y',0) \, - \, \cH_{3,\star}^- (t,y',0,0) \, ,& \\
\underline{\cQ}^+ (t,y',0) -\underline{\cQ}^- (t,y',0) \, &= \, \widehat{G}_5 (t,y',0) \, - \, 
\cQ_\star^+ (t,y',0,0) \, + \, \cQ_\star^- (t,y',0,0) \, ,& 
\end{align*}
which is always possible. Satisfying the additional constraint $\widehat{\underline{\bU}}^\pm (0)|_{y_3 =\pm 1} =0$ is 
always made possible by multiplying, for instance, by the cut-off function $\chi(y_3)$ that vanishes outside $[-2/3,2/3]$.

The particular solution $\widehat{\bU}^\pm (0)$ which we have constructed for the problem \eqref{fast_problem_0} 
obviously satisfies $\widehat{\bU}^\pm (0) \in S^\pm$. Moreover, we have shown that it is always possible to choose 
$\widehat{\bU}^\pm (0)$ such that $\Pi \, \widehat{\bU}^\pm (0) =0$ and $\widehat{\underline{\bU}}^\pm (0)|_{y_3 =\pm 1} 
=0$. In order to prove Theorem \ref{theorem_fast_problem}, we can now reduce to the case where all source terms in 
\eqref{fast_problem'} have zero mean with respect to $\theta$.

\subsection{Solvability II. Nonzero Fourier modes}

Thanks to the analysis in the previous Paragraph, the only point left in the proof of Theorem \ref{theorem_fast_problem} 
is to show that \eqref{fast_problem'} admits one solution $\bU^\pm$ when the source terms satisfy 
\eqref{compatibilite_pb_rapide_a}, \eqref{compatibilite_pb_rapide_d}, \eqref{compatibilite_pb_rapide_e}, together 
with $\widehat{F}^\pm (0)=0$, $\widehat{F}_8^\pm (0)=0$ and $\widehat{G} (0)=0$. (Let us observe that in this 
case, \eqref{compatibilite_pb_rapide_b}, \eqref{compatibilite_pb_rapide_c} are trivially satisfied.) In order to solve 
\eqref{fast_problem'}, we shall first construct an explicit solution for its projection on each nonzero Fourier mode 
with respect to $\theta$. We shall then study the summability properties of the corresponding Fourier series. Let 
us therefore compute the $k$-th Fourier coefficient of each equation in \eqref{fast_problem'} and thus introduce 
the problem:
\begin{equation}
\label{fast_problem_k''}
\begin{cases}
(A_3^\pm \, \p_{Y_3} +i \, k \, {\mathcal A}^\pm) \, \widehat{\bU}^\pm (k) \, = \, \widehat{F}^\pm (k) \, ,& 
y_3 \in I^\pm \, ,\quad \pm Y_3 >0 \, ,\\
\p_{Y_3} \widehat{\cH}_3^\pm (k) +i \, k \, \xi_j \, \widehat{\cH}_j^\pm (k) \, = \, \widehat{F}_8^\pm (k) \, ,& 
y_3 \in I^\pm \, ,\quad \pm Y_3 >0 \, ,\\
B^+ \, \widehat{\bU}^\pm (k)|_{y_3=Y_3=0} +B^- \, \widehat{\bU}^\pm (k)|_{y_3=Y_3=0} \, = \, \widehat{G} (k) \, .
\end{cases}
\end{equation}

We focus on the system \eqref{fast_problem_k''}, of which we first take the limit at $|Y_3|=+\infty$ in order to try 
to determine the residual component of $\widehat{\bU}^\pm (k)$. We wish to construct a solution to:
\begin{equation}
\label{residual_pb_k}
i\, k \, {\mathcal A}^\pm \, \widehat{\underline{\bU}}^\pm (k) \, = \, \widehat{\uF}^\pm (k) \, ,\quad 
y_3 \in I^\pm \, , \quad Y_3 \in \bR^\pm \, .
\end{equation}
Since we already know that the matrices ${\mathcal A}^\pm$ are invertible, there is no choice for the residual 
component, and we must set:
$$
\widehat{\underline{\bU}}^\pm (k) \, := \, \dfrac{1}{i\, k} \, ({\mathcal A}^\pm)^{-1} \, \widehat{\uF}^\pm (k) \, .
$$
Performing some manipulations on the fourth, fifth, sixth and seventh equations of system \eqref{residual_pb_k}, 
we find that the residual component which we have just defined satisfies the additional relation:
$$
i\, k \, \tau \, \Big( i\, k \, \xi_j \, \widehat{\underline{\cH}_j}^\pm (k) \Big) 
\, = \, i \, k \, \xi_j \, \widehat{\uF}_{3+j}^\pm (k) \, = \, i\, k \, \tau \, \widehat{\uF}_8^\pm (k) \, ,
$$
where the final equality comes from \eqref{compatibilite_pb_rapide_d}. Since $k$ and $\tau$ are nonzero, this 
means that the residual components $\widehat{\underline{\bU}}^\pm (k)$ satisfy:
\begin{equation*}
i\, k \, \xi_j \, \widehat{\underline{\cH}_j}^\pm (k) \, = \, \widehat{\uF}_8^\pm (k) \, ,
\end{equation*}
which is the $k$-th Fourier coefficient projection of the fast divergence constraint $\xi_j \, \p_\theta \underline{\cH}_j^\pm =\uF_8^\pm$.

We now define the boundary source term:
$$
\underline{\widehat{G}} (k) \, := \, 
B^+ \, \widehat{\underline{\bU}}^+ (k)|_{y_3=0} +B^- \, \widehat{\underline{\bU}}^- (k)|_{y_3=0} \, ,
$$
in such a way that the pair of profiles $\widehat{\underline{\bU}}^\pm (k)$ is a solution to:
\begin{equation}
\label{fast_problem_k_res}
\begin{cases}
i \, k \, {\mathcal A}^\pm \, \widehat{\underline{\bU}}^\pm (k) \, = \, \widehat{\uF}^\pm (k) \, ,& 
y_3 \in I^\pm \, ,\, \pm Y_3 >0 \, ,\\
i \, k \, \xi_j \, \widehat{\underline{\cH}_j}^\pm (k) \, = \, \widehat{\uF}_8^\pm (k) \, ,& y_3 \in I^\pm \, ,\, \pm Y_3 >0 \, ,\\
B^+ \, \widehat{\underline{\bU}}^+ (k)|_{y_3=0} +B^- \, \widehat{\underline{\bU}}^- (k)|_{y_3=0} \, = \, \widehat{\uG} (k) \, .
\end{cases}
\end{equation}
This is nothing but the residual component of \eqref{fast_problem_k''} except that the boundary source term 
$\widehat{\uG} (k)$ in \eqref{fast_problem_k_res} is not the same as in \eqref{fast_problem_k''}.

The above derivation of necessary solvability conditions for \eqref{fast_problem} implies that the source terms in 
\eqref{fast_problem_k_res} satisfies the orthogonality condition:
\begin{multline}
\label{compatibilite_pb_rapide_e_res}
\int_{\R^+} \mathrm{e}^{-|k| \, Y_3} \, \cL^+(k) \sbt \widehat{\uF}^+ (t,y',0,k) \,  {\rm d}Y_3 \, - \, 
\int_{\R^-} \mathrm{e}^{|k| \, Y_3} \, \cL^-(k) \sbt \widehat{\uF}^- (t,y',0,k) \,  {\rm d}Y_3 \\
+ \, \ell_1^+ \, \widehat{\uG}_1(t,y',k) \, + \, \ell_2^+ \, \widehat{\uG}_2(t,y',k) \, + \, \ell_1^- \, \widehat{\uG}_3(t,y',k) 
\, + \, \ell_2^- \, \widehat{\uG}_4(t,y',k) \, - \, i \, \tau \, \text{\rm sgn}(k) \, \widehat{\uG}_5(t,y',k) \, = \, 0 \, ,
\end{multline}
\bigskip

From the definition \eqref{residual_pb_k} of the Fourier coefficients $\widehat{\underline{\bU}}^\pm (k)$, $k \neq 0$, and the fact 
that the source term $\uF^\pm$ belongs to $\uS^\pm = H^\infty ([0,T] \times \bT^2 \times I^\pm \times \bT)$, it is quite clear that 
the formula:
$$
\underline{\bU}^\pm (t,y,\theta) \, := \, \sum_{k \neq 0} \, \widehat{\underline{\bU}}^\pm (t,y,k) \, {\rm e}^{i \, k \, \theta} 
\, = \, \sum_{k \neq 0} \, \dfrac{1}{i\, k} \, ({\mathcal A}^\pm)^{-1} \, \widehat{\uF}^\pm (t,y,k) \, {\rm e}^{i \, k \, \theta}  \, ,
$$
defines a pair of functions $\underline{\cU}^\pm \in \uS^\pm$, and these functions satisfy\footnote{Of course all $Y_3$-derivatives 
vanish in \eqref{fast_problem''} since residual functions are independent of $Y_3$.}:
\begin{equation}
\label{fast_problem'''}
\begin{cases}
\cL_f^\pm(\partial) \, \underline{\bU}^\pm \, = \, \uF^\pm \, ,& y_3 \in I^\pm \, ,\, \pm Y_3 >0 \, ,\\
\xi_j \, \p_\theta \underline{\cH}_j^\pm \, = \, \uF_8^\pm \, ,& y_3 \in I^\pm \, ,\, \pm Y_3 >0 \, ,\\
B^+ \, \underline{\bU}^+|_{y_3=0} +B^- \, \underline{\bU}^-|_{y_3=0} \, = \, \uG \, ,
\end{cases}
\end{equation}
with an obvious definition for the boundary source term $\uG$.
\bigskip

Since we have already constructed a solution to \eqref{fast_problem'''}, we can subtract \eqref{fast_problem'} and 
\eqref{fast_problem'''}, and it remains to construct a solution to the fast problem:
\begin{equation}
\label{fast_problem''}
\begin{cases}
\cL_f^\pm(\partial) \, \bU^\pm \, = \, F_\star^\pm \, ,& y_3 \in I^\pm \, ,\, \pm Y_3 >0 \, ,\\
\p_{Y_3} \cH_3^\pm +\xi_j \, \p_\theta \cH_j^\pm \, = \, F_{8,\star}^\pm \, ,& y_3 \in I^\pm \, ,\, \pm Y_3 >0 \, ,\\
B^+ \, \bU^+|_{y_3=Y_3=0} +B^- \, \bU^-|_{y_3=Y_3=0} \, = \, \underbrace{G-\uG}_{=: \, G_\star} \, ,&
\end{cases}
\end{equation}
where the main gain with respect to \eqref{fast_problem'} is that now all interior source terms are exponentially decaying 
at $Y_3 =\pm \infty$. Subtracting \eqref{compatibilite_pb_rapide_e} and \eqref{compatibilite_pb_rapide_e_res}, we know 
that the source terms in \eqref{fast_problem''} satisfy the orthogonality condition:
\begin{multline}
\label{compatibilite_pb_rapide_e'}
\int_{\R^+} \mathrm{e}^{-|k| \, Y_3} \, \cL^+(k) \sbt \widehat{F}_\star^+ (t,y',0,k) \,  {\rm d}Y_3 \, - \, 
\int_{\R^-} \mathrm{e}^{|k| \, Y_3} \, \cL^-(k) \sbt \widehat{F}_\star^- (t,y',0,k) \,  {\rm d}Y_3 \\
+ \, \ell_1^+ \, \widehat{G}_{1,\star} (t,y',k) \, + \, \ell_2^+ \, \widehat{G}_{2,\star}(t,y',k) \, 
+ \, \ell_1^- \, \widehat{G}_{3,\star} (t,y',k) \, + \, \ell_2^- \, \widehat{G}_{4,\star} (t,y',k) \, 
- \, i \, \tau \, \text{\rm sgn}(k) \, \widehat{G}_{5,\star} (t,y',k) \, = \, 0 \, .
\end{multline}
By linearity, we also know that the source terms in \eqref{fast_problem''} satisfy the compatibility condition on $\Gamma_0$:
\begin{equation}
\label{compatibilite_pb_rapide_a'}
\begin{cases}
F_{6,\star}^+|_{y_3=Y_3=0} \, = \, -b^+ \, \p_\theta G_{1,\star} +c^+ \, \p_\theta G_{2,\star} \, ,& \\
F_{6,\star}^-|_{y_3=Y_3=0} \, = \, -b^- \, \p_\theta G_{3,\star} +c^- \, \p_\theta G_{4,\star} \, .& 
\end{cases}
\end{equation}
We shall also use the compatibility condition for the divergence of the magnetic field, which corresponds to the projection 
on $S_\star^\pm$ of \eqref{compatibilite_pb_rapide_d}, namely:
\begin{equation}
\label{compatibilite_pb_rapide_d'}
\p_{Y_3} F_{6,\star}^\pm +\xi_j \, \p_\theta F_{3+j,\star}^\pm -\tau \, \p_\theta F_{8,\star}^\pm \, = \, 0 \, .
\end{equation}

In order to construct a solution to \eqref{fast_problem''}, we are going to find the expression for the $k$-th Fourier 
coefficient of $\bU^\pm$, which must be a solution to:
\begin{equation}
\label{fast_problem''_k}
\begin{cases}
L_k^\pm \, \widehat{\bU}^\pm (k) \, = \, \widehat{F}_\star^\pm \, ,& y_3 \in I^\pm \, ,\, \pm Y_3 >0 \, ,\\
\p_{Y_3} \widehat{\cH}_3^\pm (k) +i \, k \, \xi_j \, \widehat{\cH}_j^\pm (k) \, = \, \widehat{F}_{8,\star}^\pm (k)\, ,& 
y_3 \in I^\pm \, ,\, \pm Y_3 >0 \, ,\\
B^+ \, \widehat{\bU}^+ (k)|_{y_3=Y_3=0} +B^- \, \widehat{\bU}^- (k)|_{y_3=Y_3=0} \, = \, \widehat{G}_\star (k) \, .&
\end{cases}
\end{equation}
Making the fast equations explicit, we must solve (forgetting from now on the hat notation and the reference to the index $k$ 
for the solution):
\begin{equation}
\label{fast_problem'''_k}
\begin{cases}
i\, k \, c^\pm \, \cU_j^\pm -i \, k \, b^\pm \, \cH_j^\pm +i\, k \, \xi_j \, \cQ^\pm 
& \, = \, \widehat{F}_{j,\star}^\pm -u_j^{0,\pm} \, \widehat{F}_{7,\star}^\pm +H_j^{0,\pm} \, \widehat{F}_{8,\star}^\pm \, ,\quad j=1,2 \, ,\\
i\, k \, c^\pm \, \cU_3^\pm -i \, k \, b^\pm \, \cH_3^\pm +\p_{Y_3} \cQ^\pm & \, = \, \widehat{F}_{3,\star}^\pm \, ,\\
-i\, k \, b^\pm \, \cU_j^\pm +i \, k \, c^\pm \, \cH_j^\pm & \, = \, \widehat{F}_{3+j,\star}^\pm -H_j^{0,\pm} \, \widehat{F}_{7,\star}^\pm 
+u_j^{0,\pm} \, \widehat{F}_{8,\star}^\pm \, ,\quad j=1,2 \, ,\\
-i\, k \, b^\pm \, \cU_3^\pm +i \, k \, c^\pm \, \cH_3^\pm & \, = \, \widehat{F}_{6,\star}^\pm \, ,\\
\p_{Y_3} \cU_3^\pm +i \, k \, \xi_j \, \cU_j^\pm & \, = \, \widehat{F}_{7,\star}^\pm \, ,\\
\p_{Y_3} \cH_3^\pm +i \, k \, \xi_j \, \cH_j^\pm & \, = \, \widehat{F}_{8,\star}^\pm \, .
\end{cases}
\end{equation}
Using once again the divergence constraints on the velocity and the magnetic field, we find that the total pressure 
must satisfy the differential equation:
\begin{align}
\p_{Y_3}^2 \cQ^\pm -k^2 \, \cQ^\pm \, &= \, \p_{Y_3} \widehat{F}_{3,\star}^\pm -i\, k \, (a^\pm+c^\pm) \, \widehat{F}_{7,\star}^\pm 
+2\, i\, k \, b^\pm \, \widehat{F}_{8,\star}^\pm +i\, k \, \xi_j \, \widehat{F}_{j,\star}^\pm \, ,\notag \\
&= \, \underbrace{\p_{Y_3} \widehat{F}_{3,\star}^\pm +\dfrac{2\, b^\pm}{\tau} \, \p_{Y_3} \widehat{F}_{6,\star}^\pm 
-i\, k \, (a^\pm+c^\pm) \, \widehat{F}_{7,\star}^\pm 
+\dfrac{2\, i\, k \, b^\pm}{\tau} \, \xi_j \, \widehat{F}_{3+j,\star}^\pm +i\, k \, \xi_j \, \widehat{F}_{j,\star}^\pm}_{=: \, \cF^\pm} 
\, ,\label{eq_diff_qk}
\end{align}
where the second equality comes from \eqref{compatibilite_pb_rapide_d'}. The total pressure should also satisfy the 
boundary conditions:
\begin{subequations}
\label{bound_cond_qk}
\begin{equation}
\label{bound_cond_qk_a}
\cQ^+|_{y_3=Y_3=0} \, - \, \cQ^-|_{y_3=Y_3=0} \, = \, \widehat{G}_{5,\star}\, ,
\end{equation}
\begin{equation}
\label{bound_cond_qk_b}
\p_{Y_3} \cQ^+|_{y_3=Y_3=0} \, = \, \widehat{F}_{3,\star}^+|_{y_3=Y_3=0} 
-i\, k\, c^+ \, \widehat{G}_{1,\star} +i\, k\, b^+ \, \widehat{G}_{2,\star} \, ,
\end{equation}
\begin{equation}
\label{bound_cond_qk_c}
\p_{Y_3} \cQ^-|_{y_3=Y_3=0} \, = \, \widehat{F}_{3,\star}^-|_{y_3=Y_3=0} 
-i\, k\, c^- \, \widehat{G}_{3,\star} +i\, k\, b^- \, \widehat{G}_{4,\star} \, ,
\end{equation}
\end{subequations}
A solution $(\cQ^+,\cQ^-)$ to \eqref{eq_diff_qk}, \eqref{bound_cond_qk_b}, \eqref{bound_cond_qk_c} is given by:
\begin{equation}
\label{expression_qk+}
\cQ^+ \, := \, \kappa^+ \, {\rm e}^{-|k| \, Y_3} \, - \, \dfrac{1}{2\, |k|} \, \left( \int_0^{Y_3} {\rm e}^{-|k| \, (Y_3-Y)} \, \cF^+(Y) \, {\rm d}Y 
+\int_{Y_3}^{+\infty} {\rm e}^{-|k| \, (Y-Y_3)} \, \cF^+(Y) \, {\rm d}Y \right) \, ,
\end{equation}
\begin{equation}
\label{expression_qk-}
\cQ^- \, := \, \kappa^- \, {\rm e}^{|k| \, Y_3} \, - \, \dfrac{1}{2\, |k|} \, \left( \int_{-\infty}^{Y_3} {\rm e}^{-|k| \, (Y_3-Y)} \, \cF^-(Y) \, {\rm d}Y 
+\int_{Y_3}^0 {\rm e}^{-|k| \, (Y-Y_3)} \, \cF^-(Y) \, {\rm d}Y \right) \, ,
\end{equation}
where the coefficients $\kappa^\pm$, which are actually functions of $(t,y)$, must satisfy at this stage:
\begin{equation}
\label{def_kappa+}
\kappa^+|_{y_3=0} \, = \, -\dfrac{1}{2 \, |k|} \, \int_0^{+\infty} {\rm e}^{-|k| \, Y} \, \cF^+|_{y_3=0}(Y) \, {\rm d}Y 
-\dfrac{1}{|k|} \, \widehat{F}_{3,\star}^+|_{y_3=Y_3=0} +i\, \text{\rm sgn}(k) \, c^+ \, \widehat{G}_{1,\star} 
-i\, \text{\rm sgn}(k) \, b^+ \, \widehat{G}_{2,\star} \, ,
\end{equation}
\begin{equation}
\label{def_kappa-}
\kappa^-|_{y_3=0} \, = \, -\dfrac{1}{2 \, |k|} \, \int_{-\infty}^0 {\rm e}^{|k| \, Y} \, \cF^-|_{y_3=0}(Y) \, {\rm d}Y 
+\dfrac{1}{|k|} \, \widehat{F}_{3,\star}^-|_{y_3=Y_3=0} -i\, \text{\rm sgn}(k) \, c^- \, \widehat{G}_{3,\star} 
+i\, \text{\rm sgn}(k) \, b^- \, \widehat{G}_{4,\star} \, .
\end{equation}
The choice of the dependence of $\kappa^\pm$ on the slow normal variable $y_3$ is completely free at this point 
since it does not affect the fulfillment of \eqref{eq_diff_qk}. The question now is to determine whether this solution 
to \eqref{eq_diff_qk}, \eqref{bound_cond_qk_b}, \eqref{bound_cond_qk_c} also satisfies \eqref{bound_cond_qk_a} 
which will follow from the orthogonality condition \eqref{compatibilite_pb_rapide_e'} and the compatibility condition 
at the boundary \eqref{compatibilite_pb_rapide_a'}. Indeed, we compute from \eqref{expression_qk+}, \eqref{expression_qk-} 
and \eqref{def_kappa+}, \eqref{def_kappa-}:
\begin{align*}
\cQ^+|_{y_3=Y_3=0} \, -& \, \cQ^-|_{y_3=Y_3=0} \\
= \, & \, \kappa^+|_{y_3=0} -\dfrac{1}{2 \, |k|} \, \int_0^{+\infty} {\rm e}^{-|k| \, Y} \, \cF^+|_{y_3=0}(Y) \, {\rm d}Y 
-\kappa^-|_{y_3=0} +\dfrac{1}{2 \, |k|} \, \int_{-\infty}^0 {\rm e}^{|k| \, Y} \, \cF^-|_{y_3=0}(Y) \, {\rm d}Y \\
= \, & \, -\dfrac{1}{|k|} \, \int_0^{+\infty} {\rm e}^{-|k| \, Y} \, \cF^+|_{y_3=0}(Y) \, {\rm d}Y 
+\dfrac{1}{|k|} \, \int_{-\infty}^0 {\rm e}^{|k| \, Y} \, \cF^-|_{y_3=0}(Y) \, {\rm d}Y \\
& \, -\dfrac{1}{|k|} \, \Big( \widehat{F}_{3,\star}^+|_{y_3=Y_3=0} +\widehat{F}_{3,\star}^-|_{y_3=Y_3=0} \Big) 
+i\, \text{\rm sgn}(k) \, \Big( c^+ \, \widehat{G}_{1,\star} +c^- \, \widehat{G}_{3,\star} 
-b^+ \, \widehat{G}_{2,\star} -b^- \, \widehat{G}_{4,\star} \Big) \, .
\end{align*}
At this stage, we go back to the expression of the source terms $\cF^\pm$ in the differential equations 
\eqref{eq_diff_qk} and we integrate by parts those two terms in the expression of $\cF^\pm$ that involve 
a $Y_3$-derivative. Using the expression \eqref{appA-defRLkpm} of the vectors $\cL^\pm(k)$, we get:
\begin{align*}
\cQ^+ &|_{y_3=Y_3=0} \, - \, \cQ^-|_{y_3=Y_3=0} \\
= \, & \, -\dfrac{i\, \text{\rm sgn}(k)}{\tau} \, 
\int_0^{+\infty} {\rm e}^{-|k| \, Y} \, \cL^+(k) \sbt \widehat{F}_\star^+|_{y_3=0}(Y,k) \, {\rm d}Y  
+\dfrac{i\, \text{\rm sgn}(k)}{\tau} \, 
\int_{-\infty}^0 {\rm e}^{|k| \, Y} \, \cL^-(k) \sbt \widehat{F}_\star^-|_{y_3=0}(Y,k) \, {\rm d}Y  \\
& \, +\dfrac{2}{|k| \, \tau} \, \Big( b^+ \, \widehat{F}_{6,\star}^+|_{y_3=Y_3=0} 
+b^- \, \widehat{F}_{6,\star}^-|_{y_3=Y_3=0} \Big) 
+i\, \text{\rm sgn}(k) \, \Big( c^+ \, \widehat{G}_{1,\star} +c^- \, \widehat{G}_{3,\star} 
-b^+ \, \widehat{G}_{2,\star} -b^- \, \widehat{G}_{4,\star} \Big) \\
= \, & \, i\, \text{\rm sgn}(k) \, \left(  \left( \dfrac{\ell_1^+}{\tau} +c^+ \right) \, \widehat{G}_{1,\star} 
+ \left( \dfrac{\ell_2^+}{\tau} -b^+ \right) \, \widehat{G}_{2,\star} 
+ \left( \dfrac{\ell_1^-}{\tau} +c^- \right) \, \widehat{G}_{3,\star} 
+ \left( \dfrac{\ell_2^-}{\tau} -b^- \right) \, \widehat{G}_{4,\star} \right) +\widehat{G}_{5,\star} \\
& \, +\dfrac{2}{|k| \, \tau} \, \Big( b^+ \, \widehat{F}_{6,\star}^+|_{y_3=Y_3=0} 
+b^- \, \widehat{F}_{6,\star}^-|_{y_3=Y_3=0} \Big) \\
= \, & \, i\, \text{\rm sgn}(k) \, \left(  \dfrac{2 \, (b^+)^2}{\tau} \, \widehat{G}_{1,\star} 
-\dfrac{2 \, c^+ \, b^+}{\tau} \, \widehat{G}_{2,\star} 
+\dfrac{2 \, (b^-)^2}{\tau} \, \widehat{G}_{3,\star} 
-\dfrac{2 \, c^- \, b^-}{\tau} \, \widehat{G}_{4,\star} \right) +\widehat{G}_{5,\star} \\
& \, +\dfrac{2}{|k| \, \tau} \, \Big( b^+ \, \widehat{F}_{6,\star}^+|_{y_3=Y_3=0} 
+b^- \, \widehat{F}_{6,\star}^-|_{y_3=Y_3=0} \Big) \, ,
\end{align*}
where we have used the orthogonality condition \eqref{compatibilite_pb_rapide_e'} and the expression  
\eqref{s3-def_l1_l2} of the coefficients $\ell_1^\pm,\ell_2^\pm$. Using now \eqref{compatibilite_pb_rapide_a'}, 
we obtain that the total pressure defined in \eqref{expression_qk+}, \eqref{expression_qk-} satisfies 
\eqref{bound_cond_qk_a}.

For convenience, we extend $\kappa^\pm$ to $I^\pm$ by choosing them to be independent of the slow normal 
variable $y_3$. This has no consequence as far as regularity and integrability are concerned, and we shall 
see in a moment why this choice does not affect the solvability of \eqref{fast_problem''_k}.
\bigskip

Up to now, we have constructed the total pressure $\cQ^\pm$ as a solution to \eqref{eq_diff_qk} and 
\eqref{bound_cond_qk}. We now construct the tangential components of the velocity and the magnetic field. 
Namely, we set:
\begin{align*}
\cU_j^\pm \, &:= \, -\dfrac{c^\pm \, \xi_j}{(c^\pm)^2 -(b^\pm)^2} \, \cQ^\pm 
+\dfrac{c^\pm \, (\widehat{F}_{j,\star}^\pm -u_j^{0,\pm} \, \widehat{F}_{7,\star}^\pm)}{i\, k \, ((c^\pm)^2 -(b^\pm)^2)} 
+\dfrac{b^\pm \, (\widehat{F}_{3+j,\star}^\pm -H_j^{0,\pm} \, \widehat{F}_{8,\star}^\pm)}{i\, k \, ((c^\pm)^2 -(b^\pm)^2)} 
+\dfrac{c^\pm \, H_j^{0,\pm} +b^\pm \, u_j^{0,\pm}}{i\, k \, ((c^\pm)^2 -(b^\pm)^2)} \, \widehat{F}_{8,\star}^\pm \, ,\\
\cH_j \, &:= \, -\dfrac{b^\pm \, \xi_j}{(c^\pm)^2 -(b^\pm)^2} \, \cQ^\pm 
+\dfrac{b^\pm \, (\widehat{F}_{j,\star}^\pm -u_j^{0,\pm} \, \widehat{F}_{7,\star}^\pm)}{i\, k \, ((c^\pm)^2 -(b^\pm)^2)} 
+\dfrac{c^\pm \, (\widehat{F}_{3+j,\star}^\pm -H_j^{0,\pm} \, \widehat{F}_{7,\star}^\pm)}{i\, k \, ((c^\pm)^2 -(b^\pm)^2)} 
+\dfrac{c^\pm \, u_j^{0,\pm} +b^\pm \, H_j^{0,\pm}}{i\, k \, ((c^\pm)^2 -(b^\pm)^2)} \, \widehat{F}_{8,\star}^\pm \, ,
\end{align*}
with $j=1,2$. In this way, we already satisfy the first and third equations in \eqref{fast_problem'''_k}, namely the `tangential' 
equations. In order to verify the divergence constraint, we have no choice but to set:
\begin{align*}
&\cU_3^- \, := \, \int_{-\infty}^{Y_3} \widehat{F}_{7,\star}^- -\, i \, k \, \xi_j \, \cU_j^- \, {\rm d}Y \, ,\quad 
&\cU_3^+ \, := \, -\int^{+\infty}_{Y_3} \widehat{F}_{7,\star}^+ -\, i \, k \, \xi_j \, \cU_j^+ \, {\rm d}Y \, ,\\
&\cH_3^- \, := \, \int_{-\infty}^{Y_3} \widehat{F}_{8,\star}^- -\, i \, k \, \xi_j \, \cH_j^- \, {\rm d}Y \, ,\quad 
&\cH_3^+ \, := \, -\int^{+\infty}_{Y_3} \widehat{F}_{8,\star}^+ -\, i \, k \, \xi_j \, \cH_j^+ \, {\rm d}Y \, .
\end{align*}
We thus ensure at the same time the fifth and sixth equations in \eqref{fast_problem'''_k} and exponential decay 
with respect to $Y_3$ at infinity. It remains to verify that the second and fourth equations in \eqref{fast_problem'''_k} 
are satisfied (the `normal' equations). At this stage we have defined all coordinates of the velocity and magnetic field 
so we can only hope that our previous definitions will automatically yield the missing relations for verifying 
\eqref{fast_problem'''_k} and the boundary conditions in \eqref{fast_problem''_k}.

Let us therefore verify that the second and fourth equations in \eqref{fast_problem'''_k} are satisfied. We compute:
\begin{align*}
\p_{Y_3} \Big( i\, k \, b^\pm \, \cU_3^\pm & \, -i \, k \, c^\pm \, \cH_3^\pm +\widehat{F}_{6,\star}^\pm \Big) \\
= \, & \, \p_{Y_3} \widehat{F}_{6,\star}^\pm +i\, k \, \Big( b^\pm \, \widehat{F}_{7,\star}^\pm -c^\pm \, \widehat{F}_{8,\star}^\pm \Big) 
+i\, k \, \Big( c^\pm \, i\, k \, \xi_j \, \cH_j^\pm -b^\pm \, i\, k \, \xi_j \, \cU_j^\pm \Big) \\
= \, & \, \p_{Y_3} \widehat{F}_{6,\star}^\pm +i\, k \,  \xi_j \, \widehat{F}_{3+j,\star}^\pm -i\, k \, \tau \, \widehat{F}_{8,\star}^\pm 
\, = \, 0 \, .
\end{align*}
Here we have used the compatibility condition \eqref{compatibilite_pb_rapide_d'}. Since all functions $\cU_3^\pm$, 
$\cH_3^\pm$, $\widehat{F}_{6,\star}^\pm$ decay exponentially with respect to $Y_3$ at infinity, this means that we have:
$$
i\, k \, b^\pm \, \cU_3^\pm -i \, k \, c^\pm \, \cH_3^\pm +\widehat{F}_{6,\star}^\pm \equiv 0 \, ,
$$
meaning that the fourth equation in \eqref{fast_problem'''_k} is satisfied. By the same argument (differentiation with 
respect to $Y_3$ and limit at infinity), we can also show that the second equation in \eqref{fast_problem'''_k} is 
satisfied. We now explain why the boundary conditions in \eqref{fast_problem''_k} are satisfied. Taking the double 
trace of the fourth equation in \eqref{fast_problem'''_k}, we first get:
\begin{equation}
\label{bound_cond_k_1}
\begin{cases}
-i\, k \, b^+ \, \cU_3^+|_{y_3=Y_3=0} +i \, k \, c^+ \, \cH_3^+|_{y_3=Y_3=0} \, = \, \widehat{F}_{6,\star}^+|_{y_3=Y_3=0} 
\, = \, -i\, k \, b^+ \, G_{1,\star} +i \, k \, c^+ \, G_{2,\star} \, ,& \\
-i\, k \, b^- \, \cU_3^-|_{y_3=Y_3=0} +i \, k \, c^- \, \cH_3^-|_{y_3=Y_3=0} \, = \, \widehat{F}_{6,\star}^-|_{y_3=Y_3=0} 
\, = \, -i\, k \, b^- \, G_{3,\star} +i \, k \, c^- \, G_{4,\star} \, ,& 
\end{cases}
\end{equation}
where we have used \eqref{compatibilite_pb_rapide_a'}. We now take the double trace of the second equation in 
\eqref{fast_problem'''_k}, and use \eqref{bound_cond_qk_b}, \eqref{bound_cond_qk_c}:
\begin{equation}
\label{bound_cond_k_2}
\begin{cases}
i\, k \, c^+ \, \cU_3^+|_{y_3=Y_3=0} -i \, k \, b^+ \, \cH_3^+|_{y_3=Y_3=0} \, = \, \widehat{F}_{3,\star}^+|_{y_3=Y_3=0} 
-\p_{Y_3} \cQ^+|_{y_3=Y_3=0} \, = \, i\, k \, c^+ \, G_{1,\star} -i \, k \, b^+ \, G_{2,\star} \, ,& \\
i\, k \, c^- \, \cU_3^-|_{y_3=Y_3=0} -i \, k \, b^- \, \cH_3^-|_{y_3=Y_3=0} \, = \, \widehat{F}_{3,\star}^-|_{y_3=Y_3=0} 
-\p_{Y_3} \cQ^-|_{y_3=Y_3=0} \, = \, i\, k \, c^- \, G_{3,\star} -i \, k \, b^- \, G_{4,\star} \, .& 
\end{cases}
\end{equation}
Combining \eqref{bound_cond_k_1} and \eqref{bound_cond_k_2}, and $(c^\pm)^2-(b^\pm)^2 \neq 0$, we obtain:
$$
\cU_3^+|_{y_3=Y_3=0} \, = \, G_{1,\star} \, ,\quad \cH_3^+|_{y_3=Y_3=0} \, = \, G_{2,\star} \, ,\quad 
\cU_3^-|_{y_3=Y_3=0} \, = \, G_{3,\star} \, ,\quad \cH_3^-|_{y_3=Y_3=0} \, = \, G_{4,\star} \, .
$$
Together with \eqref{bound_cond_qk_a}, which we have shown to be valid, this proves that the boundary conditions in 
\eqref{fast_problem''_k} are satisfied. In other words, we have completed the construction of a solution to the fast problem 
\eqref{fast_problem_k''}. It remains to verify that the sum
$$
\sum_{k \neq 0} \widehat{\bU}^\pm (k) \, {\rm e}^{i \, k \, \theta} \, ,
$$
defines a function in $S_\star^\pm$. This question is addressed in \cite[Lemma 4.2]{Marcou} in full details, so rather than 
repeating the same arguments, we refer the reader to that reference (the integrals in \eqref{expression_qk+}, \eqref{expression_qk-} 
are exactly the quantities to which the result of \cite[Lemma 4.2]{Marcou} applies). This completes the proof of Theorem 
\ref{theorem_fast_problem}.

\section{A simplified version of Theorem \ref{theorem_fast_problem}}

At the beginning of this Chapter, we have stated the solvability result of Theorem \ref{theorem_fast_problem} for the fast problem 
\eqref{fast_problem} in its most general form. In the inductive corrector construction of Chapter \ref{chapter5}, we shall use Theorem 
\ref{theorem_fast_problem} to enforce solvability conditions for the WKB cascade. Part of the corrector construction also relies in 
finding solutions to fast problems of the form \eqref{fast_problem} but, for some reason that will be made clear in Chapter \ref{chapter5}, 
it will be useful to solve fast problems where the source terms are `purely oscillating' in $\theta$, meaning that their mean with respect 
to $\theta$ on $\bT$ is zero. We therefore state the following result whose proof is a straightforward consequence of Theorem 
\ref{theorem_fast_problem} and is therefore omitted.

\begin{corollary}
\label{cor_fast_problem}
Let Assumptions \eqref{s3-hyp_stab_nappe_plane}, \eqref{s3-hyp_xi'_rationnelle}, \eqref{s3-hyp_delta_neq_0}, \eqref{s3-hyp_tau_racine_det_lop} 
be satisfied together with $\tau\neq 0$. Let $F^\pm, F_8^\pm \in S^\pm$ and let $G \in H^\infty ([0,T] \times \bT^2 \times \bT)$ satisfy
$$
\widehat{F}^{\pm}(0) \, = \, 0 \, ,\quad \widehat{F}_8^{\pm}(0) \, = \, 0 \, ,\quad \widehat{G}(0) \, = \, 0 \, ,
$$
as well as
\begin{subequations}
\label{cor_compatibilite_pb_rapide}
\begin{equation}
\label{ccor_compatibilite_pb_rapide_a}
\begin{cases}
F_6^+|_{y_3=Y_3=0} \, = \, -b^+ \, \p_\theta G_1 +c^+ \, \p_\theta G_2 \, ,& \\
F_6^-|_{y_3=Y_3=0} \, = \, -b^- \, \p_\theta G_3 +c^- \, \p_\theta G_4 \, ,& 
\end{cases}
\quad \text{\rm (compatibility at the boundary)}
\end{equation}
\begin{equation}
\label{cor_compatibilite_pb_rapide_d}
\p_{Y_3} F_6^\pm +\xi_j \, \p_\theta F_{3+j}^\pm -\tau \, \p_\theta F_8^\pm \, = \, 0 \, ,\quad 
\text{\rm (compatibility for the divergence of the magnetic field)}
\end{equation}
\begin{multline}
\label{cor_compatibilite_pb_rapide_e}
\forall \, k \neq 0 \, ,\quad 
\int_{\R^+} \mathrm{e}^{-|k| \, Y_3} \, \cL^+(k) \sbt \widehat{F}^+ (t,y',0,Y_3,k) \,  {\rm d}Y_3 \, - \, 
\int_{\R^-} \mathrm{e}^{|k| \, Y_3} \, \cL^-(k) \sbt \widehat{F}^- (t,y',0,Y_3,k) \,  {\rm d}Y_3 \\
+ \, \ell_1^+ \, \widehat{G}_1(t,y',k) \, + \, \ell_2^+ \, \widehat{G}_2(t,y',k) \, + \, \ell_1^- \, \widehat{G}_3(t,y',k) 
\, + \, \ell_2^- \, \widehat{G}_4(t,y',k) \, - \, i \, \tau \, \text{\rm sgn}(k) \, \widehat{G}_5(t,y',k) \, = \, 0 \, .
\end{multline}
\end{subequations}
Then the fast problem \eqref{fast_problem} has a solution $(\bU^\pm,0) \in S^\pm \times H^\infty ([0,T] \times \bT^2 \times \bT)$ that satisfies 
$\widehat{\bU}^{\pm}(0)=0$.
\end{corollary}

\chapter{Solving the WKB cascade I: the leading amplitude}
\label{chapter4}

In this Chapter, we start solving the WKB cascade \eqref{dev_BKW_int}, \eqref{s3-cascade_div_H^m+1,pm'}, 
\eqref{s3-cascade_BKW_bord_matriciel}, \eqref{s3-cond_bords_fixes_uH_3^m,pm} together with the normalization 
conditions for the total pressure. Namely, we are going to construct the leading amplitude $(U^{\, 1,\pm},\psi^2)$ in the 
WKB ansatz \eqref{s3-def_dvlpt_BKW}, \eqref{s3-def_dvlpt_BKW_psi} by first identifying the degrees of freedom that 
we have, and then by determining all functions at our disposal by imposing (some of) the necessary solvability conditions 
\eqref{compatibilite_pb_rapide} for the fast problem that must be satisfied by the first corrector $(U^{\, 2,\pm},\psi^3)$. This 
is the standard procedure in geometric optics, be it linear or weakly nonlinear, see, \emph{e.g.}, \cite{Rauch}, with the main 
feature here that some of the solvability conditions \eqref{compatibilite_pb_rapide} must \emph{come for free} in the WKB 
cascade since we shall have more constraints to satisfy than we have degrees of freedom at our disposal. These compatibility 
-rather than solvability- conditions will not be examined in this Chapter but will play a key role in the inductive corrector 
construction of Chapter \ref{chapter5}.
\bigskip

Let us therefore begin with solving the WKB cascade. We focus in this Chapter on:
\begin{itemize}
 \item the interior equations \eqref{dev_BKW_int} for $m=0$ and $m=1$, where the source terms $F^{0,\pm}$, $F^{\, 1,\pm}$ 
 are explicitly given by \eqref{s3-def_F^1,pm},
 \item the fast divergence constraint \eqref{s3-cascade_div_H^m+1,pm'} on the magnetic field for $m=0$ and $m=1$, where 
 the source terms $F_8^{0,\pm}$, $F_8^{\, 1,\pm}$ are given by evaluating the right hand side of \eqref{s3-cascade_div_H^m+1,pm} 
 for $m=0$ and $m=1$, namely:
\begin{subequations}
\label{s3-def_F_8^1,pm}
\begin{align}
F_8^{0,\pm} & \, = \, 0 \, ,\label{s3-def_terme_source_F_8^0,pm} \\
F_8^{\, 1,\pm} & \, = \, -\nabla \cdot H^{\, 1,\pm} \, + \, \xi_j \, \p_\theta \psi^2 \, \p_{Y_3} H_j^{\, 1,\pm} 
\, ,\label{s3-def_terme_source_F_8^1,pm}
\end{align}
\end{subequations}
 \item the boundary conditions \eqref{s3-cascade_BKW_bord_matriciel} for $m=0$ and $m=1$, which read explicitly 
 \eqref{s3-cond_bord_U^1,pm}, \eqref{s3-cascade_BKW_bord_simplifiee_m=1},
 \item the top and bottom boundary conditions \eqref{s3-cond_bords_fixes_uH_3^m,pm} for $m=0$,
 \item the normalization condition \eqref{expression_I1(t)} for the slow mean $\widehat{\uq}^{\, 1,\pm} (0)$ of the total pressure.
\end{itemize}

Collecting \eqref{dev_BKW_int}, \eqref{s3-cascade_div_H^m+1,pm'} for $m=0$, and \eqref{s3-cond_bord_U^1,pm}, we first 
observe that the leading amplitude $(U^{\, 1,\pm},\psi^2)$ must satisfy the homogeneous fast problem:
\begin{equation}
\label{fast_homogeneous_lead}
\begin{cases}
\cL_f^\pm(\partial) \, U^{\, 1,\pm} \, = \, 0 \, ,& y_3 \in I^\pm \, ,\, \pm Y_3 >0 \, ,\\
\p_{Y_3} H_3^{\, 1,\pm} +\xi_j \, \p_\theta H_j^{\, 1,\pm} \, = \, 0 \, ,& y_3 \in I^\pm \, ,\, \pm Y_3 >0 \, ,\\
B^+ \, U^{\, 1,+}|_{y_3=Y_3=0} +B^- \, U^{\, 1,-}|_{y_3=Y_3=0} +\partial_\theta \psi^2 \, \ub \, = \, 0 \, .
\end{cases}
\end{equation}
Let us recall that we denote $U^\pm =(u^\pm,H^\pm,q^\pm)^T$ and that $(t,y')$ enter \eqref{fast_homogeneous_lead} as parameters 
only. We apply Proposition \ref{prop_fast_homogeneous} and deduce that the leading profile can be decomposed as:
\begin{align}
U^{\, 1,\pm} (t,y,Y_3,\theta) \, = \, \uU^{\, 1,\pm} (t,y) &+ 
\Big( u_{1,\star}^{\, 1,\pm},u_{2,\star}^{\, 1,\pm},0,H_{1,\star}^{\, 1,\pm},H_{2,\star}^{\, 1,\pm},0,0 \Big)^T (t,y,Y_3) \notag \\[0.5ex]
&+\sum_{k \in \Z \setminus \{ 0\}} \gamma^{\, 1,\pm} (t,y,k) \, {\rm e}^{\mp |k| \, Y_3 +i\, k \, \theta} \, \cR^\pm (k) \, ,
\label{decomp_leading_profile}
\end{align}
where $\uU^{\, 1,\pm} \in \uS^\pm$ is independent of $\theta$, $u_{1,\star}^{\, 1,\pm},u_{2,\star}^{\, 1,\pm},H_{1,\star}^{\, 1,\pm},
H_{2,\star}^{\, 1,\pm} \in S_\star^\pm$, the vectors $\cR^\pm (k)$ are given in \eqref{appA-defRLkpm}, the coefficients 
$\gamma^{\, 1,\pm}$ satisfy $\gamma^{\, 1,\pm} (t,y',0,k) =\pm |k| \, \widehat{\psi}^{\, 2}(t,y',k)$ for all $(t,y',k) \in [0,T] \times \bT^2 
\times \bZ^*$, and
$$
\gamma^{\, 1,\pm} (t,y,-k) \, = \, \overline{\gamma^{\, 1,\pm} (t,y,k)} \, ,
$$
which ensures that the leading profile $U^{\, 1,\pm}$ is real-valued. The slow mean $\uU^{\, 1,\pm}$ should also satisfy the boundary 
conditions on $\Gamma_0$:
\begin{equation}
\label{cond_bord_moyenne_lead}
\uu_3^{\, 1,\pm}|_{y_3=0} \, = \, \uH_3^{\, 1,\pm}|_{y_3=0} \, = \, 0 \, ,\quad \uq^{\, 1,+}|_{y_3=0} \, = \, \uq^{\, 1,-}|_{y_3=0} \, .
\end{equation}

One point to keep in mind here is that we have not determined the final time $T>0$ so far. The time $T>0$ appears in Definition 
\ref{s3-def_espaces_fonctionnels} of the functional spaces $\uS^\pm$, $S_\star^\pm$. Fixing the time $T>0$ will be done once 
and for all later on in this Chapter when we prove the solvability of the leading amplitude equation \eqref{s3-edp_hatpsi^2-final} 
below, together with a property of propagation of regularity in time. Hence the first part of this Chapter is mostly a matter of using 
necessary conditions in order to determine the form of the leading profile. At the very end, we shall show how to construct the 
leading profile and we shall clearly list which (solvability) conditions for the first corrector construction are satisfied. We split 
the identification of the various functions in the decomposition \eqref{decomp_leading_profile} in several Sections below. This 
splitting, in the exact same order, will be used when constructing the correctors in the WKB ansatz \eqref{s3-def_dvlpt_BKW}, 
\eqref{s3-def_dvlpt_BKW_psi}. This explains why we give all details for the derivation of the leading amplitude equation 
\eqref{s3-edp_hatpsi^2-final}, and we shall feel free to shorten similar arguments in Chapter \ref{chapter5}.

Another key point to keep in mind is that the mean of the leading front $\widehat{\psi}^{\, 2}(0)$ does not appear in 
\eqref{fast_homogeneous_lead}, nor will it appear in the leading amplitude equation \eqref{s3-edp_hatpsi^2-final} that only 
involves the oscillating modes of $\psi^2$. The mean of the leading front will appear later on though as an extra 
degree of freedom that we shall use to determine the slow mean of the first corrector $\widehat{U}^{\, 2,\pm}(0)$.

\section{The slow mean of the leading profile}

As identified in the decomposition \eqref{decomp_leading_profile}, we have already seen that the residual component 
of the leading amplitude does not depend on the fast variable $\theta$, that is $\uU^{\, 1,\pm} =\uU^{\, 1,\pm}(t,y)$. To avoid 
overloaded notation, we feel free to omit from now on the reference to the zero Fourier mode in $\theta$ and keep the 
notation $\uU^{\, 1,\pm}$ rather than $\widehat{\uU}^{\, 1,\pm} (0)$. The decomposition \eqref{decomp_leading_profile} has 
a first important consequence for the top and bottom boundary conditions. Indeed, we already see that in order to satisfy 
the conditions \eqref{s3-cond_bords_fixes_uH_3^m,pm} for $m=0$, which are conditions on \emph{all} Fourier modes 
with respect to $\theta$, it is now sufficient to verify:
$$
\uu_3^{\, 1,\pm}|_{\Gamma^\pm} \, = \, \uH_3^{\, 1,\pm}|_{\Gamma^\pm} \, = \, 0 \, ,
$$
where, keeping the notation of \eqref{decomp_leading_profile}, $\uu_3^{\, 1,\pm},\uH_3^{\, 1,\pm}$ are functions of 
$(t,y)$ only. The validity of \eqref{s3-cond_bords_fixes_uH_3^m,pm} -with $m=0$- for the nonzero Fourier modes 
is \emph{automatic}. (See Chapter \ref{chapter5} for a similar statement for the correctors in the WKB ansatz.) The 
slow mean $\uU^{\, 1,\pm}$ should also satisfy the boundary conditions \eqref{cond_bord_moyenne_lead} on $\Gamma_0$. 
Let us eventually recall that we wish the total pressure to satisfy the normalization condition \eqref{expression_I1(t)} which, 
in the notation of \eqref{decomp_leading_profile}, now reads:
\begin{equation}
\label{normalization_pressure_lead}
\forall \, t \, ,\quad \int_{\Omega_0^+} \uq^{\, 1,+}(t,y) \, {\rm d}y \, + \, \int_{\Omega_0^-} \uq^{\, 1,-}(t,y) \, {\rm d}y \, = \, 0 \, .
\end{equation}
Recall that $\Omega_0^\pm$ in \eqref{normalization_pressure_lead} denote the sets $\bT^2 \times I^\pm$.

The evolution equations that determine the slow mean of the leading profile are obtained by imposing the solvability condition 
\eqref{compatibilite_pb_rapide_b} on the fast problem that must be satisfied by the first corrector. Let us note indeed that, 
among several equations, the first corrector $(U^{\, 2,\pm},\psi^3)$ must satisfy:
\begin{equation}
\label{fast_first_corrector}
\begin{cases}
\cL_f^\pm(\partial) \, U^{\, 2,\pm} \, = \, F^{\, 1,\pm} \, ,& y \in \Omega_0^\pm \, ,\, \pm Y_3 >0 \, ,\\
\p_{Y_3} H_3^{\, 2,\pm} +\xi_j \, \p_\theta H_j^{\, 2,\pm} \, = \, F_8^{\, 1,\pm} \, ,& y \in \Omega_0^\pm \, ,\, \pm Y_3 >0 \, ,\\
B^+ \, U^{\, 2,+}|_{y_3=Y_3=0} +B^- \, U^{\, 2,-}|_{y_3=Y_3=0} +\partial_\theta \psi^3 \, \ub \, = \, G^1 \, ,
\end{cases}
\end{equation}
where $F^{\, 1,\pm}$ is given by \eqref{s3-def_terme_source_F^1,pm}, $F_8^{\, 1,\pm}$ is given by 
\eqref{s3-def_terme_source_F_8^1,pm}, and $G^1$ is given by \eqref{terme_source_bord_BKW} with $m=1$. 
In order to solve \eqref{fast_first_corrector}, by Theorem \ref{theorem_fast_problem}, we must necessarily have:
$$
\widehat{\uF}^{\, 1,\pm}(0) \, \equiv \, 0 \, ,\quad \widehat{\uF}_8^{\, 1,\pm}(0) \, \equiv \, 0 \, .
$$
Using \eqref{s3-def_terme_source_F^1,pm}, \eqref{s3-def_terme_source_F_8^1,pm}, of which we first compute the limit 
as $Y_3$ tends to infinity and then take the mean with respect to $\theta$ on $\bT$, we find\footnote{We see here why 
it was useful to start with the conservative form of the MHD system in order to use the symmetry of the bilinear mappings 
$\bA_\alpha$. Indeed this symmetry property allows to write some of the terms as partial derivatives with respect to $\theta$, 
which automatically have zero mean on $\bT$.} that the slow mean $\uU^{\, 1,\pm}$ should satisfy the linearized current vortex 
sheet problem:
\begin{equation}
\label{s3-sys_uuU^1,pm}
\left\{
\begin{array}{r l}
\p_t \uu^{\, 1,\pm} +u_j^{0,\pm} \, \p_{y_j} \uu^{\, 1,\pm} -H_j^{0,\pm} \, \p_{y_j} \uH^{\, 1,\pm} +\nabla \uq^{\, 1,\pm} \, = \, 0 \, , &  \\ [0.5ex]
\p_t \uH^{\, 1,\pm} +u_j^{0,\pm} \, \p_{y_j} \uH^{\, 1,\pm} -H_j^{0,\pm} \, \p_{y_j} \uu^{\, 1,\pm} \, = \, 0 \, , &  \\ [0.5ex]
\dv \uu^{\, 1,\pm} \, = \, \dv \uH^{\, 1,\pm} \, = \, 0 \, , & (t,y) \in [0,T] \times \Omega_0^\pm \, , \\ [0.5ex]
\uu_3^{\, 1,\pm} \, = \, \uH_3^{\, 1,\pm} \, = \, \big[ \, \uq^1 \, \big] \, = \, 0 \, , & y_3 = 0 \, , \\ [0.5ex]
\uu_3^{\, 1,\pm} \, = \, \uH_3^{\, 1,\pm} \, = \, 0 \, , & y_3 = \pm 1 \, . 
\end{array}
\right.
\end{equation}

We are now going to prove that the \emph{homogeneous} system \eqref{s3-sys_uuU^1,pm} admits a unique smooth solution, 
which is identically zero if we choose to prescribe zero initial data\footnote{Initial data for the total pressure should not be 
prescribed since the total pressure is defined, as we shall see below, as the Lagrange multiplier associated with the divergence 
constraint on the velocity.}:
$$
(\uu^{\, 1,\pm},\uH^{\, 1,\pm})|_{t=0} \, \equiv \, 0 \, .
$$
This choice for the initial conditions of \eqref{s3-sys_uuU^1,pm} is free, up to the obvious compatibility conditions for the divergence 
and boundary conditions that need to be satisfied at the initial time. Consistently with our choice \eqref{s3-cond_moyenne_initiale_U^m,pm} 
for the tangential components of the fast mean, we are going to choose the easiest possible initial conditions for the slow means. 
In the case of the leading profile, the easiest possible choice of initial data for \eqref{s3-sys_uuU^1,pm} corresponds to prescribing 
zero initial data. This choice is indeed compatible with the divergence constraints on $\uu^{\, 1,\pm}$ and $\uH^{\, 1,\pm}$ and with the 
boundary conditions in \eqref{s3-sys_uuU^1,pm}. We refer to Chapter \ref{chapter5} for the choice of initial data for the slow mean 
of the correctors (in that case, we shall not be able to prescribe zero initial data any longer due to the divergence constraints and 
boundary conditions).

We can refer to Catania \cite{Catania} for an existence and uniqueness result in the space $H^1$ for the linearized incompressible 
current vortex sheet system, without the fixed top and bottom boundaries $\Gamma^\pm$ but with a linearized front appearing in 
the boundary conditions on $\Gamma_0$. The analysis below of the linearized system \eqref{s3-sys_uuU^1,pm} differs from 
\cite{Catania} and is somehow much simpler so we choose to give it here with all details. Let us start by determining the total 
pressure $(\uq^{\, 1,+}, \uq^{\, 1,-})$ which must satisfy the following homogeneous linear Laplace problem (where the time $t \in 
[0,T]$ plays the role of a parameter):\begin{equation}\label{s3-pb_laplace_uuq^1,pm}
\left\{
\begin{array}{r l}
-\Delta \, \uq^{\, 1,\pm} \, = \, 0 \, , & \text{ in } \Omega_0^\pm \, , \\ [0.5ex]
\big[ \, \uq^1 \, \big] \, = \, 0 \, , & \\
\big[ \, \p_{y_3} \uq^1 \, \big] \, = \, 0 \, , & \text{ on } \Gamma_0 \, , \\ [0.5ex]
\p_{y_3} \uq^{\, 1,\pm} \, = \, 0, & \text{ on } \Gamma^\pm \, .
\end{array}
\right.
\end{equation}
The total pressure should also satisfy the normalization condition \eqref{normalization_pressure_lead}. Using the Lax-Milgram 
Theorem \cite{Evans} in the Hilbert space:
$$
\cH \, := \, \Big\{ (q^+,q^-) \in H^1(\Omega_0^+) \times H^1(\Omega_0^-) \, \Big| \, \, \, [ \, q \, ] \, = \, 0 \quad \text{ and } \quad 
\int_{\Omega_0^+} q^+ \, {\rm d}y \, + \, \int_{\Omega_0^-} q^- \, {\rm d}y \, = \, 0 \Big\} \, ,
$$
we can prove that there exists a unique solution $(\uq^{\, 1,+},\uq^{\, 1,-})$ to \eqref{s3-pb_laplace_uuq^1,pm} in $\cH$, 
and since zero is an obvious solution, we necessarily have $\uq^{\, 1,\pm} \equiv 0$ in \eqref{s3-sys_uuU^1,pm}.

To finish with system \eqref{s3-sys_uuU^1,pm}, it remains to determine $(\uu^{\, 1,\pm},\uH^{\, 1,\pm})$. Since the total pressure 
vanishes, the evolution equations in \eqref{s3-sys_uuU^1,pm} can be rewritten under the form of the symmetric hyperbolic 
system:
\begin{equation*}
\left\{
\begin{array}{r l}
\p_t \uu^{\, 1,\pm} +u_j^{0,\pm} \, \p_{y_j} \uu^{\, 1,\pm} -H_j^{0,\pm} \, \p_{y_j} \uH^{\, 1,\pm} \, = \, 0 \, , &  \\ [0.5ex]
\p_t \uH^{\, 1,\pm} +u_j^{0,\pm} \, \p_{y_j} \uH^{\, 1,\pm} -H_j^{0,\pm} \, \p_{y_j} \uu^{\, 1,\pm} \, = \, 0 \, . &
\end{array}
\right.
\end{equation*}
Since we impose zero initial conditions for the velocity and magnetic field, the energy method \cite{BS} immediately implies 
$\uu^{\, 1,\pm},\uH^{\, 1,\pm} \equiv 0$. No boundary condition is needed on $\Gamma_0$ or $\Gamma^\pm$ since we have 
$u^{0,\pm}_3=H^{0,\pm}_3=0$ (this means actually that $y_3$ is a parameter in the latter system). As a matter of fact, the 
boundary conditions in \eqref{s3-sys_uuU^1,pm} rather arise as compatibility conditions which are automatically satisfied 
here since $\uu^{\, 1,\pm}$ and $\uH^{\, 1,\pm}$ vanish.

Let us recall that prescribing zero initial conditions for the slow mean of the leading profile is only done for the sake of 
simplicity, in order to focus on the surface wave component $U_\star^{\, 1,\pm}$. We shall explain in Chapter \ref{chapter5} 
why the compatibility conditions that are necessary for solving the WKB cascade are independent of this specific choice 
of initial conditions.

\section{The fast mean of the leading profile}

Another necessary condition for solving the inhomogeneous fast problem \eqref{fast_first_corrector} is 
\eqref{compatibilite_pb_rapide_c}. Using the expression \eqref{s3-def_terme_source_F^1,pm}, we have:
$$
\widehat{F}_\star^{\, 1,\pm} (0) \, = \, -L_s^\pm (\p) \widehat{U}_\star^{\, 1,\pm} (0) \, + \, 
\widehat{ \Big( \p_\theta\psi^2 \, \cA^\pm \, \p_{Y_3}U_\star^{\, 1,\pm} \Big)} (0) \, - \, 
\dfrac{1}{2} \, \p_{Y_3} \widehat{ \Big( \bA_3(U^{\, 1,\pm},U^{\, 1,\pm}) \Big)} (0) \, ,
$$
and since we already know that our choice of initial data \eqref{s3-cond_moyenne_initiale_U^m,pm} implies $\uU^{\, 1,\pm} =0$, 
we get\footnote{The more general case $\uU^{\, 1,\pm} \neq 0$ can be dealt with the same arguments as the one we use. One 
just needs to take care of a few extra terms but there is no new difficulty.}:
\begin{equation}
\label{decomp_comp_fast_mean_lead_1}
\widehat{F}_\star^{\, 1,\pm} (0) \, = \, -L_s^\pm (\p) \widehat{U}_\star^{\, 1,\pm} (0) \, + \, 
\widehat{\Big( \p_\theta\psi^2 \, \cA^\pm \, \p_{Y_3} (U_\star^{\, 1,\pm}-\widehat{U}_\star^{\, 1,\pm} (0)) \Big)} (0) \, - \, 
\dfrac{1}{2} \, \p_{Y_3} \widehat{\Big( \bA_3(U_\star^{\, 1,\pm},U_\star^{\, 1,\pm}) \Big)} (0) \, .
\end{equation}
In a similar way, starting from \eqref{s3-def_terme_source_F_8^1,pm}, we compute:
\begin{equation}
\label{decomp_comp_fast_mean_lead_2}
\widehat{F}_{8,\star}^{\, 1,\pm} (0) \, = \, -\nabla \cdot \widehat{H}_\star^{\, 1,\pm} (0) \, + \, 
\widehat{\Big( \xi_j \, \p_\theta\psi^2 \, \p_{Y_3} (H_{j,\star}^{\, 1,\pm}-\widehat{H}_{j,\star}^{\, 1,\pm} (0)) \Big)} (0) \, .
\end{equation}

Let us split the expressions in \eqref{decomp_comp_fast_mean_lead_1}, \eqref{decomp_comp_fast_mean_lead_2} in 
several pieces. We first introduce:
\begin{equation}
\label{decomp_comp_fast_mean_lead_3}
\mathfrak{F}^{\, 1,\pm} \, := \, \dfrac{1}{2} \, \p_{Y_3} \widehat{\Big( \bA_3(U_\star^{\, 1,\pm},U_\star^{\, 1,\pm}) \Big)} (0) \, ,\quad 
\mathfrak{F}_8^{\, 1,\pm} \, := \, 0 \, .
\end{equation}
Using the decomposition \eqref{decomp_leading_profile} and the explicit expression of the Hessian mapping $\bA_3$ 
given in Appendix \ref{appendixA}, we have:
$$
\bA_3 \Big( \widehat{U}_\star^{\, 1,\pm} (0),\widehat{U}_\star^{\, 1,\pm} (0) \Big) \, = \, 0 \, ,
$$
independently of the determination of the tangential components $u_{j,\star}^{\, 1,\pm},H_{j,\star}^{\, 1,\pm}$ in 
\eqref{decomp_leading_profile}. We thus compute:
\begin{align*}
\mathfrak{F}^{\, 1,\pm} \, & \, = \, \dfrac{1}{2} \, \p_{Y_3} \left( \sum_{k \neq 0} {\rm e}^{\mp 2\, |k| \, Y_3} \, 
\bA_3 \Big( \gamma^{\, 1,\pm} (t,y,k) \, \cR^\pm (k),\gamma^{\, 1,\pm} (t,y,-k) \, \cR^\pm (-k) \Big) \right) \\
& \, = \, \mp \sum_{k \neq 0} {\rm e}^{\mp 2\, |k| \, Y_3} \, |k| \, \left| \gamma^{\, 1,\pm} (t,y,k) \right|^2 \, 
\bA_3 \Big( {\bf R}^\pm,\overline{{\bf R}^\pm} \Big) \, ,
\end{align*}
where we have used \eqref{appA-defRLkpm}, the symmetry of $\bA_3$, and the reality condition:
$$
\gamma^{\, 1,\pm} (t,y,-k) \, = \, \overline{\gamma^{\, 1,\pm} (t,y,k)} \, .
$$
Now using \eqref{appA-defeigenvectorsRpm}, we find that the vector $\bA_3 \big( {\bf R}^\pm,\overline{{\bf R}^\pm} \big)$ 
is proportional to the third vector of the canonical basis of $\bR^7$:
$$
\bA_3 \Big( {\bf R}^\pm,\overline{{\bf R}^\pm} \Big) \, = \, 2 \, \Big( 0,0,(c^\pm)^2-(b^\pm)^2,0,0,0,0 \Big)^T \, ,
$$
which means that the source terms $\mathfrak{F}^{\, 1,\pm},\mathfrak{F}_8^{\, 1,\pm}$ defined in 
\eqref{decomp_comp_fast_mean_lead_3} automatically satisfy the linear system (compare with \eqref{compatibilite_pb_rapide_c}):
\begin{equation*}
\begin{cases}
u_j^{0,\pm} \, \mathfrak{F}_7^{\, 1,\pm} -H_j^{0,\pm} \, \mathfrak{F}_8^{\, 1,\pm} \, = \, \mathfrak{F}_j^{\, 1,\pm} \, , & j=1,2 \, ,\\
H_j^{0,\pm} \, \mathfrak{F}_7^{\, 1,\pm} -u_j^{0,\pm} \, \mathfrak{F}_8^{\, 1,\pm} \, = \, \mathfrak{F}_{3+j}^{\, 1,\pm} \, , & j=1,2 \, ,
\end{cases}
\end{equation*}
independently of the choice we can make for $u_{j,\star}^{\, 1,\pm},H_{j,\star}^{\, 1,\pm}$ and $\gamma^{\, 1,\pm}$ in the 
decomposition \eqref{decomp_leading_profile}.

Let us go on with the splitting of the expressions in \eqref{decomp_comp_fast_mean_lead_1} and 
\eqref{decomp_comp_fast_mean_lead_2}. We now introduce:
\begin{equation}
\label{decomp_comp_fast_mean_lead_4}
\mathcal{F}^{\, 1,\pm} \, := \, \widehat{\Big( \p_\theta\psi^2 \, \cA^\pm \, \p_{Y_3} (U_\star^{\, 1,\pm}-\widehat{U}_\star^{\, 1,\pm} (0)) \Big)} (0) \, ,\quad 
\mathcal{F}_8^{\, 1,\pm} \, := \, \widehat{\Big( \xi_j \, \p_\theta\psi^2 \, \p_{Y_3} (H_{j,\star}^{\, 1,\pm}-\widehat{H}_{j,\star}^{\, 1,\pm} (0)) \Big)} (0) \, ,
\end{equation}
and we compute:
\begin{align*}
\mathcal{F}^{\, 1,\pm} \, & \, = \, \pm \sum_{k \neq 0} i\, k \, |k| \, \widehat{\psi}^{\, 2} (t,y',-k) \, \gamma^{\, 1,\pm}(t,y,k) 
\, {\rm e}^{\mp |k| \, Y_3} \, \cA^\pm \, \cR^\pm (k) \\
& \, = \, \sum_{k \neq 0} k^2 \, \widehat{\psi}^{\, 2} (t,y',-k) \, \gamma^{\, 1,\pm}(t,y,k) \, {\rm e}^{\mp |k| \, Y_3} \, A_3^\pm \, \cR^\pm (k) \, ,
\end{align*}
with $\cR^\pm (k)$ given by \eqref{appA-defRLkpm} (here we have used the relation $\cA^\pm \, \cR^\pm (k)=\mp i\, \text{\rm sgn} (k) 
\, A_3^\pm \, \cR^\pm (k)$). We also compute:
$$
\mathcal{F}_8^{\, 1,\pm} \, = \, \pm \sum_{k \neq 0} i\, \text{\rm sgn} (k) \, b^\pm \, k^2 \, \widehat{\psi}^{\, 2} (t,y',-k) \, \gamma^{\, 1,\pm}(t,y,k) 
\, {\rm e}^{\mp |k| \, Y_3} \, ,
$$
and it is now a rather straightforward exercise to verify that the source terms in \eqref{decomp_comp_fast_mean_lead_4} 
automatically satisfy the linear system (compare again with \eqref{compatibilite_pb_rapide_c}):
\begin{equation*}
\begin{cases}
u_j^{0,\pm} \, \mathcal{F}_7^{\, 1,\pm} -H_j^{0,\pm} \, \mathcal{F}_8^{\, 1,\pm} \, = \, \mathcal{F}_j^{\, 1,\pm} \, , & j=1,2 \, ,\\
H_j^{0,\pm} \, \mathcal{F}_7^{\, 1,\pm} -u_j^{0,\pm} \, \mathcal{F}_8^{\, 1,\pm} \, = \, \mathcal{F}_{3+j}^{\, 1,\pm} \, , & j=1,2 \, ,
\end{cases}
\end{equation*}
independently of the choice we can make for $\psi^2$ and $\gamma^{\, 1,\pm}$ in \eqref{decomp_leading_profile}.

Summarizing, the source terms given in \eqref{decomp_comp_fast_mean_lead_1} and \eqref{decomp_comp_fast_mean_lead_2} 
will satisfy the solvability condition \eqref{compatibilite_pb_rapide_c} as long as there holds:
\begin{equation}
\label{decomp_comp_fast_mean_lead_5}
\begin{cases}
u_j^{0,\pm} \, \big( -\nabla \cdot \widehat{u}_\star^{\, 1,\pm}(0) \big) -H_j^{0,\pm} \, \big( -\nabla \cdot \widehat{H}_\star^{\, 1,\pm}(0) \big) 
\, = \, \big( -L_s^\pm (\p) \, \widehat{U}_\star^{\, 1,\pm} (0) \big)_j \, , & j=1,2 \, ,\\
H_j^{0,\pm} \, \big( -\nabla \cdot \widehat{u}_\star^{\, 1,\pm}(0) \big) -u_j^{0,\pm} \, \big( -\nabla \cdot \widehat{H}_\star^{\, 1,\pm}(0) \big) 
\, = \, \big( -L_s^\pm (\p) \, \widehat{U}_\star^{\, 1,\pm} (0) \big)_{3+j} \, , & j=1,2 \, .
\end{cases}
\end{equation}
The latter condition \eqref{decomp_comp_fast_mean_lead_5} is independent of the choice we can make for the leading front $\psi^2$ 
and the coefficients $\gamma^{\, 1,\pm}$. The evolution equations \eqref{decomp_comp_fast_mean_lead_5} can be equivalently rewritten 
as the symmetric hyperbolic system:
\begin{equation}
\label{decomp_comp_fast_mean_lead_6}
\p_t \begin{bmatrix}
u_{1,\star}^{\, 1,\pm} \\
u_{2,\star}^{\, 1,\pm} \\
H_{1,\star}^{\, 1,\pm} \\
H_{2,\star}^{\, 1,\pm} \end{bmatrix} 
+\begin{bmatrix}
u_j^{0,\pm} & 0 & -H_j^{0,\pm} & 0 \\
0 & u_j^{0,\pm} & 0 & -H_j^{0,\pm} \\
-H_j^{0,\pm} & 0 & u_j^{0,\pm} & 0 \\
0 & -H_j^{0,\pm} & 0 & u_j^{0,\pm} \end{bmatrix} \, \p_{y_j} \begin{bmatrix}
u_{1,\star}^{\, 1,\pm} \\
u_{2,\star}^{\, 1,\pm} \\
H_{1,\star}^{\, 1,\pm} \\
H_{2,\star}^{\, 1,\pm} \end{bmatrix} \, = \, 0 \, .
\end{equation}
Let us observe that in \eqref{decomp_comp_fast_mean_lead_6}, the slow and fast normal variables $y_3 \in I^\pm$, $Y_3 \in \bR^\pm$ 
enter as parameters since only tangential differentiation in $t,y_1,y_2$ occurs. Recalling now that we prescribe zero initial data for the 
fast mean of the tangential components of the velocity and magnetic field, see \eqref{s3-cond_moyenne_initiale_U^m,pm}, the system 
\eqref{decomp_comp_fast_mean_lead_6} is assigned with zero initial data and we therefore have $u_{1,\star}^{\, 1,\pm},u_{2,\star}^{\, 1,\pm}, 
H_{1,\star}^{\, 1,\pm},H_{2,\star}^{\, 1,\pm} \equiv 0$ by applying again the energy method.

At this stage, we have already shown that prescribing zero initial data for the slow mean and for the fast mean of the tangential components 
of the velocity and magnetic field simplifies the decomposition \eqref{decomp_leading_profile} into:
\begin{equation}
\label{decomp_leading_profile'}
U^{\, 1,\pm} (t,y,Y_3,\theta) \, = \,  \sum_{k \in \Z \setminus \{ 0\}} \gamma^{\, 1,\pm} (t,y,k) \, {\rm e}^{\mp |k| \, Y_3 +i\, k \, \theta} 
\, \cR^\pm (k) \, ,
\end{equation}
with the boundary condition $\gamma^{\, 1,\pm} (t,y',0,k) =\pm |k| \, \widehat{\psi}^{\, 2}(t,y',k)$ on $\Gamma_0$. In the following 
Section, we are going to show that imposing the solvability condition \eqref{compatibilite_pb_rapide_e} for the fast problem 
\eqref{fast_first_corrector} verified by the first corrector fully determines the evolution of the leading front $\psi^2$ (or, more 
precisely, the evolution of its oscillating modes).

\section{The nonlocal Hamilton-Jacobi equation for the leading front}

\subsection{Derivation of the equation}

Let us focus again on the inhomogeneous fast problem \eqref{fast_first_corrector} that must be satisfied by the first corrector 
$(U^{\, 2,\pm},\psi^3)$. Applying Theorem \ref{theorem_fast_problem}, we know that a necessary condition for \eqref{fast_first_corrector} 
to have a solution in $S^\pm \times H^\infty$ is the orthogonality condition \eqref{compatibilite_pb_rapide_e} for all nonzero Fourier 
modes. Recalling the general form \eqref{terme_source_bord_BKW} of the boundary source term in the WKB cascade, a necessary 
condition for the validity of the asymptotic expansion \eqref{s3-def_dvlpt} is therefore:
\begin{multline}
\label{s3-cond_ortho_psi^2}
\forall \, k \neq 0 \, ,\quad 
\int_{\bR^+} \mathrm{e}^{-|k|Y_3} \, \cL^+(k) \sbt \hatF^{\, 1,+}(k)|_{y_3=0} \, {\rm d}Y_3 \, 
- \, \int_{\bR^-} \mathrm{e}^{|k|Y_3} \, \cL^-(k) \sbt \hatF^{\, 1,-}(k)|_{y_3=0} \, {\rm d}Y_3 \\
+ \, \ell_1^+ \, \hatG_1^{\, 1,+}(k) \, + \, \ell_2^+ \, \hatG_2^{\, 1,+}(k) \, + \, \ell_1^- \, \hatG_1^{\, 1,-}(k) \, + \, \ell_2^- \, \hatG_2^{\, 1,-}(k) \, = \, 0 \, ,
\end{multline}
where, for simplicity, we omit to recall the slow variables $(t,y')$, which enter \eqref{s3-cond_ortho_psi^2} as parameters. 
Let us recall here that the expression of the functions $G_1^{\, 1,\pm}$, $G_2^{\, 1,\pm}$ is given in \eqref{s3-def_G_1^m,pm}, 
\eqref{s3-def_G_2^m,pm} (with $m=1$), and that the interior source term $F^{\, 1,\pm}$ is given in \eqref{s3-def_terme_source_F^1,pm}. 
There is a subtlety here, which was absent from \cite{AliHunter} but already occurred in the related work \cite{CW}. Namely, 
from the expression \eqref{decomp_leading_profile'}, we can compute the trace of the leading profile $U^{\, 1,\pm}$ on 
$\Gamma_0$ in terms of the leading front $\psi^2$:
\begin{equation}
\label{s3-U^1,pm_k_y3=0}
U^{\, 1,\pm} (t,y',0,Y_3,\theta) \, = \,  \pm \sum_{k \in \Z \setminus \{ 0\}} |k| \, \widehat{\psi}^{\, 2} (t,y',k) \, 
{\rm e}^{\mp |k| \, Y_3 +i\, k \, \theta} \, \cR^\pm (k) \, .
\end{equation}
This allows us to write the source terms $G_1^{\, 1,\pm},G_2^{\, 1,\pm}$ in \eqref{s3-cond_ortho_psi^2} in terms of the leading front 
$\psi^2$, see below for explicit calculations. From the expression \eqref{s3-def_terme_source_F^1,pm}, we also see that the 
trace $F^{\, 1,\pm}|_{y_3=0}$ can be expressed in terms of $U^{\, 1,\pm}|_{y_3=0}$, and consequently in terms of $\psi^2$, for 
almost all terms on the right hand side of \eqref{s3-def_terme_source_F^1,pm} but \emph{one} (!). Indeed, the trace on 
$\Gamma_0$ of the expression in \eqref{s3-def_terme_source_F^1,pm} involves the normal derivative $\p_{y_3} U^{\, 1,\pm} 
|_{y_3=0}$ and there is no reason why this term can be expressed in terms of $\psi^2$ (unless we make a specific choice 
for lifting $\psi^2$ on $\Gamma_0$ to $\gamma^{\, 1,\pm}$ in $\Omega_0^\pm$). However, we are going to see now that the 
contribution of these normal derivatives in \eqref{s3-cond_ortho_psi^2} is zero, which does not seem obvious at first sight 
(see \cite{CW} for a similar cancelation property in the context of elastodynamics). Once this property is clarified, it will remain 
to compute all other terms in \eqref{s3-cond_ortho_psi^2} as functions of $\psi^2$.

\paragraph{The orthogonality condition \eqref{s3-cond_ortho_psi^2} is a closed equation on the leading front.}

The goal in this Paragraph is to examine more closely the quantity
$$
- \, \int_{\bR^+} \mathrm{e}^{-|k|Y_3} \, \cL^+(k) \sbt A_3^+ \, \p_{y_3} \hatU^{\, 1,+}(k)|_{y_3=0} \, {\rm d}Y_3 
\, + \, \int_{\bR^-} \mathrm{e}^{|k|Y_3} \, \cL^-(k) \sbt A_3^- \, \p_{y_3} \hatU^{\, 1,-}(k)|_{y_3=0} \, {\rm d}Y_3 \, ,
$$
with the aim of showing that each of these two integrals vanishes. Since all other quantities in $F^{\, 1,\pm}$ involve partial 
derivatives that are tangent to the boundary $\Gamma_0$, this will imply that \eqref{s3-cond_ortho_psi^2} is a closed 
equation on the leading front $\psi^2$. We shall even see below that it is a closed equation on the oscillating modes 
of $\psi^2$ only.

Let us therefore consider some $y_3 \in I^+$ and $Y_3 \in \bR^+$. We recall that the leading profile $U^{\, 1,+}$ 
satisfies the homogeneous fast problem \eqref{fast_homogeneous_lead}, which means that each oscillating mode 
$\hatU^{\, 1,+} (k)$ satisfies the differential equation:
$$
L_k^+ \, \hatU^{\, 1,+} (k) \, = \, 0 \, ,\quad L_k^+ \, = \, A_3^+ \, \p_{Y_3} +i \, k \, \cA^+ \, .
$$
We perform the same integration by parts argument as the one that led us to \eqref{dualite_1}, and obtain:
\begin{equation}
\label{dualite_2}
0 \, = \, \int_{Y_3}^{+\infty} V^+ \sbt L_k^+ \, \hatU^{\, 1,+} (k) \, {\rm d}Y \, = \, 
\int_{Y_3}^{+\infty} (L_k^+)^* \, V^+ \sbt \hatU^{\, 1,+} (k) \, {\rm d}Y \, - \, V^+(Y_3) \sbt A_3^+ \, \hatU^{\, 1,+} (Y_3,k) \, .
\end{equation}
If we specify furthermore to the choice \eqref{s3-def_V^pm} for the test function $V^+$, the term $(L_k^+)^* \, V^+$ in 
\eqref{dualite_2} vanishes and we are led to:
$$
\cL^+(k) \sbt A_3^+ \, \hatU^{\, 1,+}(Y_3,k) \, = \, 0 \, ,
$$
which could have also been obtained -in our rather simple context- by just verifying the orthogonality condition $\cL^+(k) \sbt 
A_3^+ \, \cR^+(k)=0$ and using the decomposition \eqref{decomp_leading_profile'}. Differentiating the latter equality with 
respect to $y_3$ and taking the trace on $\Gamma_0$, we obtain:
$$
\int_{\bR^+} \mathrm{e}^{-|k|Y_3} \, \cL^+(k) \sbt A_3^+ \, \p_{y_3} \hatU^{\, 1,+}(k)|_{y_3=0} \, {\rm d}Y_3 \, = \, 0 \, ,
$$
and the same argument applies on the opposite side of the current vortex sheet.

In other words, we can get rid of the terms $A_3^\pm \, \p_{y_3} U^{\, 1,\pm}$ in \eqref{s3-cond_ortho_psi^2} and it remains to 
compute all the contributions in \eqref{s3-cond_ortho_psi^2} in terms of the Fourier coefficients of $\psi^2$. This computation 
is split in five different cases, which we shall eventually collect to derive the leading amplitude equation for $\psi^2$. Unsurprisingly, 
we shall find the same amplitude equation as in \cite{AliHunter} with however the incorporation of the group velocity transport 
with respect to the slow spatial variables $y'$ (slow modulation with respect to the space variables was not considered in 
\cite{AliHunter} and the original MHD equations in that earlier reference were two-dimensional).

\paragraph{The boundary terms in \eqref{s3-cond_ortho_psi^2}.}

Let us recall the expression \eqref{s3-U^1,pm_k_y3=0} for the trace $U^{\, 1,\pm}|_{y_3=0}$, which gives\footnote{Recall that 
the residual oscillating modes of $U^{\, 1,\pm}$ vanish because the matrix $\cA^\pm$ is invertible.}:
\begin{equation}
\label{s3-U^1,pm_k_y3=0-bis}
\forall \, k \neq 0 \, ,\quad \hatU^{\, 1,\pm} (t,y',0,Y_3,k) \, = \, \hatU_\star^{\, 1,\pm} (t,y',0,Y_3,k) \, = \, 
\pm \, |k| \, \widehat{\psi}^{\, 2} (t,y',k) \, {\rm e}^{\mp \, |k| \, Y_3} \, \cR^\pm(k) \, ,
\end{equation}
where the vectors $\cR^\pm(k)$ are defined in \eqref{appA-defRLkpm}. In the following, we shall omit the variables $(t,y')$ which play a 
role of parameters. We also recall that $k$ is a nonzero integer since the orthogonality condition \eqref{compatibilite_pb_rapide_e} 
only bears on nonzero Fourier modes.

From the definition \eqref{s3-def_G_1^m,pm} and using \eqref{s3-U^1,pm_k_y3=0-bis}, we compute:
\begin{align}
\hatG_1^{\, 1,\pm} (k) \, & \,= \, \p_t \widehat{\psi}^{\, 2} (k) \, + \, u_j^{0,\pm} \, \p_{y_j} \widehat{\psi}^{\, 2} (k) 
\, + \, \sum_{k_1 +k_2 =k} i \, k_2 \, \xi_j \, \widehat{u_j}^{\, 1,\pm}|_{y_3=Y_3=0} (k_1) \, \widehat{\psi}^{\, 2} (k_2) \notag \\
& \, = \, \p_t \widehat{\psi}^{\, 2}(k) \, + \, u_j^{0,\pm} \, \p_{y_j} \widehat{\psi}^{\, 2}(k) 
\, \pm \, i \, c^\pm \, \sum_{k_1 +k_2 =k} \dfrac{k_1 \, |k_2|+k_2 \, |k_1|}{2} \, \widehat{\psi}^{\, 2}(k_1) \, \widehat{\psi}^{\, 2}(k_2) \, .
\label{s3-calcul_hatG_1^pm_1}
\end{align}
In the latter equality, we have \emph{symmetrized} the kernel that arises in the bilinear Fourier multiplier that acts on 
$(\psi^2,\psi^2)$. We shall repeat this symmetrization argument in all bilinear expressions on the Fourier modes of 
$\psi^2$. With similar computations, we obtain from \eqref{s3-def_G_2^m,pm} the expression:
\begin{equation}
\label{s3-calcul_hatG_2^pm_1}
\hatG_2^{\, 1,\pm} (k) \, = \, H_j^{0,\pm} \, \p_{y_j} \widehat{\psi}^{\, 2} (k) \, \pm \, i \, b^\pm \, \sum_{k_1 +k_2 =k} 
\dfrac{k_1 \, |k_2|+k_2 \, |k_1|}{2} \, \widehat{\psi}^{\, 2} (k_1) \, \widehat{\psi}^{\, 2} (k_2) \, .
\end{equation}

Using the definition \eqref{s3-def_l1_l2} of the coefficients $\ell_1^\pm$, $\ell_2^\pm$ together with \eqref{s3-calcul_hatG_1^pm_1} 
and \eqref{s3-calcul_hatG_2^pm_1}, the contribution of the boundary source terms in \eqref{s3-cond_ortho_psi^2} reads:
\begin{multline}
\label{s3-calcul_termes_bord_relation_dualité}
\big( \ell_1^+ \, \hatG_1^{\, 1,+} + \ell_2^+ \, \hatG_2^{\, 1,+} + \ell_1^- \, \hatG_1^{\, 1,-} + \ell_2^- \, \hatG_2^{\, 1,-} \big) (k) 
\, = \, \Big( 2 \, (b^+)^2 +2 \, (b^-)^2\big) - \tau \, (c^+ + c^-) \Big) \, \p_t \widehat{\psi}^{\, 2} (k) \\
+ \, \Big[ \Big( 2 \, (b^+)^2 - \tau \, c^+\Big) \, u_j^{0,+} \, + \, \Big( 2 \, (b^-)^2 - \tau \, c^-\Big) \, u_j^{0,-} \, 
- \, (a^+ +c^+) \, b^+ \, H_j^{0,+} - \, (a^- +c^-) \, b^- \, H_j^{0,-} \Big] \, \p_{y_j} \widehat{\psi}^{\, 2} (k) \\
+ \, \dfrac{i \, \tau}{2} \, \Big( (b^+)^2 - (b^-)^2 - (c^+)^2 + (c^-)^2 \Big) \, \sum_{k_1 +k_2 =k} 
\big( \, k_1 \, |k_2|+k_2 \, |k_1| \, \big) \, \widehat{\psi}^{\, 2} (k_1) \, \widehat{\psi}^{\, 2} (k_2) \, .
\end{multline}

\paragraph{Computation of the linear terms in the integrals of \eqref{s3-cond_ortho_psi^2}.}

Recalling the expression \eqref{s3-def_terme_source_F^1,pm} of $F^{\, 1,\pm}$, we see that the trace $U^{\, 1,\pm}|_{y_3=0}$ 
enters either linearly or quadratically in it. This means that as in \eqref{s3-calcul_termes_bord_relation_dualité}, the integrals 
in \eqref{s3-cond_ortho_psi^2} will contribute to the amplitude equation in the form of either linear or quadratic terms with 
respect to $\psi^2$. In this Paragraph, we compute all expressions in the integrals of \eqref{s3-cond_ortho_psi^2} that 
contribute for linear terms in $\psi^2$. Going back to \eqref{s3-def_terme_source_F^1,pm}, we need to compute:
$$
-\int_{\bR^+} \mathrm{e}^{-|k|Y_3} \cL^+(k) \sbt (A_0 \, \p_t +A_j^\pm \, \p_{y_j} ) \, \hatU^{\, 1,+}(k)|_{y_3=0} \, {\rm d}Y_3 
+\int_{\bR^-} \mathrm{e}^{|k|Y_3} \cL^-(k) \sbt (A_0 \, \p_t +A_j^\pm \, \p_{y_j} )\, \hatU^{\, 1,-}(k)|_{y_3=0} \, {\rm d}Y_3 
$$
with $\hatU^{\, 1,+}(k)|_{y_3=0}$ given by \eqref{s3-U^1,pm_k_y3=0-bis}. Computing the latter integrals explicitly, we get 
the expression:
\begin{align*}
&- \, \dfrac{1}{2} \, \Big( \cL^+(k) \sbt A_0 \, \cR^+(k) \, + \, \cL^-(k) \sbt A_0 \, \cR^-(k) \Big) \, \p_t \widehat{\psi}^{\, 2} (k) \\
&- \, \dfrac{1}{2} \, \Big( \cL^+(k) \sbt A_j^+ \, \cR^+(k) \, + \, \cL^-(k) \sbt A_j^- \, \cR^-(k) \Big) \, \p_{y_j} \widehat{\psi}^{\, 2} (k) \, ,
\end{align*}
which can be further simplified by using the expressions \eqref{appA-produits_hermitiens-1} for the various Hermitian 
products arising above. Eventually, the collection of all linear terms in $\psi^2$ in the integrals on the left hand side of 
\eqref{s3-cond_ortho_psi^2} reads:
\begin{multline}
\label{s3-calcul_hatF_1^pm}
-\Big( 2 \, (b^+)^2 +2 \, (b^-)^2 +\tau \, (c^+ +c^-) \Big) \, \p_t \widehat{\psi}^{\, 2} (k) \\
-\Big[ \Big( 2 \, (b^+)^2 +\tau \, c^+\Big) \, u_j^{0,+} + \Big( 2 \, (b^-)^2 +\tau \, c^-\Big) \, u_j^{0,-} 
- (2\, \tau +a^+ +c^+) \, b^+ \, H_j^{0,+} - (2\, \tau +a^- +c^-) \, b^- \, H_j^{0,-} \Big] \, \p_{y_j} \widehat{\psi}^{\, 2} (k) \, .
\end{multline}
To derive the expression \eqref{s3-calcul_hatF_1^pm}, we have also used the relation \eqref{s3-hyp_tau_racine_det_lop} 
to cancel part of the coefficient of $\p_{y_j} \widehat{\psi}^{\, 2} (k)$. At this stage, the sum of the contributions in 
\eqref{s3-calcul_termes_bord_relation_dualité} and \eqref{s3-calcul_hatF_1^pm} reads:
\begin{multline}
\label{s3-calcul_termes_bord_relation_dualité'}
-2 \, \tau \, (c^+ + c^-) \, \p_t \widehat{\psi}^{\, 2} (k) 
-2 \, \tau \, \Big( c^+ \, u_j^{0,+} \, + \, c^- \, u_j^{0,-} \, - \, b^+ \, H_j^{0,+} \, - \, b^- \, H_j^{0,-} \Big) \, \p_{y_j} \widehat{\psi}^{\, 2} (k) \\
+ \, \dfrac{i \, \tau}{2} \, \Big( (b^+)^2 - (b^-)^2 - (c^+)^2 + (c^-)^2 \Big) \, \sum_{k_1 +k_2 =k} 
\big( \, k_1 \, |k_2|+k_2 \, |k_1| \, \big) \, \widehat{\psi}^{\, 2} (k_1) \, \widehat{\psi}^{\, 2} (k_2) \, .
\end{multline}
We now turn to the computation of all quadratic expressions in the integrals of \eqref{s3-cond_ortho_psi^2}.

\paragraph{The quadratic terms in the integrals of \eqref{s3-cond_ortho_psi^2}. I.}

In the decomposition \eqref{s3-def_terme_source_F^1,pm} of $F^{\, 1,\pm}$, we isolate the term with a $\p_{Y_3}$ derivative, 
which contributes in \eqref{s3-cond_ortho_psi^2} for:
\begin{align*}
&- \, \dfrac{1}{2} \, \int_{\bR^+} \mathrm{e}^{-|k|Y_3} \, \cL^+(k) \sbt \big( \widehat{\p_{Y_3} \bA_3(U^{\, 1,+},U^{\, 1,+})|_{y_3=0}} \big) (k) \, {\rm d}Y_3 \\
&+ \, \dfrac{1}{2} \, \int_{\bR^-} \mathrm{e}^{|k|Y_3} \, \cL^-(k) \sbt \big( \widehat{\p_{Y_3} \bA_3(U^{\, 1,-},U^{\, 1,-})|_{y_3=0}} \big) (k) \, {\rm d}Y_3 \, ,
\end{align*}
and by integrating by parts, this contribution can be rewritten as:
\begin{align*}
& \, \dfrac{1}{2} \, \cL^+(k) \sbt \big( \widehat{\bA_3(U^{\, 1,+},U^{\, 1,+}) |_{y_3=Y_3=0}} \big) (k) \, + \, 
\dfrac{1}{2} \,  \cL^-(k) \sbt \big( \widehat{\bA_3(U^{\, 1,-},U^{\, 1,-}) |_{y_3=Y_3=0}} \big) (k) \\
& \, - \, \dfrac{|k|}{2} \, \int_{\bR^+} \mathrm{e}^{-|k|Y_3} \, \cL^+(k) \sbt \big( \widehat{\bA_3(U^{\, 1,+},U^{\, 1,+})|_{y_3=0}} \big) (k) \, {\rm d}Y_3 \\
& \, - \, \dfrac{|k|}{2} \, \int_{\bR^-} \mathrm{e}^{|k|Y_3} \, \cL^-(k) \sbt \big( \widehat{\bA_3(U^{\, 1,-},U^{\, 1,-})|_{y_3=0}} \big) (k) \, {\rm d}Y_3 \\
= \, & \, \dfrac{1}{2} \, \sum_{k_1+k_2=k} |k_1| \, |k_2| \, \cL^+(k) \sbt \bA_3 (\cR^+(k_1),\cR^+(k_2)) \, 
\widehat{\psi}^{\, 2} (k_1) \, \widehat{\psi}^{\, 2} (k_2) \\
& \, + \, \dfrac{1}{2} \, \sum_{k_1+k_2=k} |k_1| \, |k_2| \, \cL^-(k) \sbt \bA_3 (\cR^-(k_1),\cR^-(k_2)) \, 
\widehat{\psi}^{\, 2} (k_1) \, \widehat{\psi}^{\, 2} (k_2) \\
& \, - \, \dfrac{1}{2} \, \sum_{k_1+k_2=k} \dfrac{|k| \, |k_1| \, |k_2|}{|k| +|k_1| +|k_2|} \, \cL^+(k) \sbt \bA_3 (\cR^+(k_1),\cR^+(k_2)) \, 
\widehat{\psi}^{\, 2} (k_1) \, \widehat{\psi}^{\, 2} (k_2) \\
& \, - \, \dfrac{1}{2} \, \sum_{k_1+k_2=k} \dfrac{|k| \, |k_1| \, |k_2|}{|k| +|k_1| +|k_2|} \, \cL^-(k) \sbt \bA_3 (\cR^-(k_1),\cR^-(k_2)) \, 
\widehat{\psi}^{\, 2} (k_1) \, \widehat{\psi}^{\, 2} (k_2) \\
= \, & \, \dfrac{1}{2} \, \sum_{k_1+k_2=k} 
\dfrac{|k_1| \, |k_2| \, (|k_1|+|k_2|)}{|k| +|k_1| +|k_2|} \, \cL^+(k) \sbt \bA_3 (\cR^+(k_1),\cR^+(k_2)) \, 
\widehat{\psi}^{\, 2} (k_1) \, \widehat{\psi}^{\, 2} (k_2) \\
& \, + \, \dfrac{1}{2} \, \sum_{k_1+k_2=k} 
\dfrac{|k_1| \, |k_2| \, (|k_1|+|k_2|)}{|k| +|k_1| +|k_2|} \, \cL^-(k) \sbt \bA_3 (\cR^-(k_1),\cR^-(k_2)) \, 
\widehat{\psi}^{\, 2} (k_1) \, \widehat{\psi}^{\, 2} (k_2) \, .
\end{align*}
Using the expressions \eqref{appA-produits_hermitiens-2}, we end up with the contribution:
\begin{multline}
\label{s3-calcul_termes_bord_relation_dualité''}
- \, \dfrac{i \, \tau}{2} \, \Big( (b^+)^2 - (b^-)^2 - (c^+)^2 + (c^-)^2 \Big) \\
\sum_{k_1 +k_2 =k} 
\dfrac{|k_1| \, |k_2| \, (|k_1| +|k_2|)}{|k| +|k_1| +|k_2|} \, \big( \text{\rm sgn} (k_1) +\text{\rm sgn} (k_2) 
+2 \, \text{\rm sgn} (k) \, \text{\rm sgn} (k_1) \, \text{\rm sgn} (k_2) \big) \, \widehat{\psi}^{\, 2} (k_1) \, \widehat{\psi}^{\, 2} (k_2) \, .
\end{multline}
The sum of \eqref{s3-calcul_termes_bord_relation_dualité'} and \eqref{s3-calcul_termes_bord_relation_dualité''} can be simplified and now reads:
\begin{multline}
\label{s3-calcul_termes_bord_relation_dualité'-bis}
-2 \, \tau \, (c^+ + c^-) \, \p_t \widehat{\psi}^{\, 2} (k) 
-2 \, \tau \, \Big( c^+ \, u_j^{0,+} \, + \, c^- \, u_j^{0,-} \, - \, b^+ \, H_j^{0,+} \, - \, b^- \, H_j^{0,-} \Big) \, \p_{y_j} \widehat{\psi}^{\, 2} (k) \\
- \, \dfrac{i \, \tau}{2} \, \Big( (b^+)^2 - (b^-)^2 - (c^+)^2 + (c^-)^2 \Big) \, \text{\rm sgn} (k) \, \times \\
\sum_{k_1 +k_2 =k} \left( 2 \, k_1 \, k_2 -\dfrac{|k_1| \, |k_2| \, |k|}{|k| +|k_1| +|k_2|} \, 
\Big( \text{\rm sgn} (k) \, (\text{\rm sgn} (k_1) +\text{\rm sgn} (k_2)) 
+2 \, \text{\rm sgn} (k_1) \, \text{\rm sgn} (k_2) \Big) \right) \, \widehat{\psi}^{\, 2} (k_1) \, \widehat{\psi}^{\, 2} (k_2) \, .
\end{multline}

\paragraph{The quadratic terms in the integrals of \eqref{s3-cond_ortho_psi^2}. II.}

With rather similar calculations, we can compute the contribution in \eqref{s3-cond_ortho_psi^2} of the terms in 
\eqref{s3-def_terme_source_F^1,pm} with a $\p_\theta$ derivative:
\begin{align*}
&- \, \dfrac{1}{2} \, \int_{\bR^+} \mathrm{e}^{-|k|Y_3} \, \cL^+(k) \sbt 
\big( \widehat{\xi_j \, \p_\theta \bA_j(U^{\, 1,+},U^{\, 1,+})|_{y_3=0}} \big) (k) \, {\rm d}Y_3 \\
&+ \, \dfrac{1}{2} \, \int_{\bR^-} \mathrm{e}^{|k|Y_3} \, \cL^-(k) \sbt 
\big( \widehat{\xi_j \, \p_\theta \bA_j(U^{\, 1,-},U^{\, 1,-})|_{y_3=0}} \big) (k) \, {\rm d}Y_3 \, .
\end{align*}
Using once again the expressions \eqref{appA-produits_hermitiens-2}, we get the contribution:
\begin{multline}
\label{s3-calcul_termes_bord_relation_dualité'''}
\dfrac{i \, \tau}{2} \, \Big( (b^+)^2 - (b^-)^2 - (c^+)^2 + (c^-)^2 \Big) \, \text{\rm sgn} (k) \, \times \\
\sum_{k_1 +k_2 =k} \left( 2 \, \dfrac{|k_1| \, |k_2| \, |k|}{|k| +|k_1| +|k_2|} 
+\dfrac{|k_1| \, |k_2| \, k}{|k| +|k_1| +|k_2|} \, \big( \text{\rm sgn} (k_1) +\text{\rm sgn} (k_2) \big) \right) \, 
\widehat{\psi}^{\, 2} (k_1) \, \widehat{\psi}^{\, 2} (k_2) \, .
\end{multline}
Adding \eqref{s3-calcul_termes_bord_relation_dualité'''} with \eqref{s3-calcul_termes_bord_relation_dualité'-bis}, we get the expression:
\begin{align}
& -2 \, \tau \, (c^+ + c^-) \, \p_t \widehat{\psi}^{\, 2} (k) 
-2 \, \tau \, \Big( c^+ \, u_j^{0,+} \, + \, c^- \, u_j^{0,-} \, - \, b^+ \, H_j^{0,+} \, - \, b^- \, H_j^{0,-} \Big) \, \p_{y_j} \widehat{\psi}^{\, 2} (k) 
\label{s3-calcul_termes_bord_relation_dualité'-ter} \\
& +i \, \tau \, \Big( (b^+)^2 - (b^-)^2 - (c^+)^2 + (c^-)^2 \Big) \, \text{\rm sgn} (k) 
\sum_{k_1 +k_2 =k} \dfrac{|k_1| \, |k_2| \, |k|}{|k| +|k_1| +|k_2|} \, \widehat{\psi}^{\, 2} (k_1) \, \widehat{\psi}^{\, 2} (k_2) \notag \\
& -i \, \tau \, \Big( (b^+)^2 - (b^-)^2 - (c^+)^2 + (c^-)^2 \Big) \, \text{\rm sgn} (k) \, \times \notag \\
& \sum_{k_1 +k_2 =k} \left( k_1 \, k_2 -\dfrac{|k_1| \, |k_2| \, |k|}{|k| +|k_1| +|k_2|} \, 
\Big( \text{\rm sgn} (k) \, (\text{\rm sgn} (k_1) +\text{\rm sgn} (k_2)) 
+\text{\rm sgn} (k_1) \, \text{\rm sgn} (k_2) \Big) \right) \, \widehat{\psi}^{\, 2} (k_1) \, \widehat{\psi}^{\, 2} (k_2) \, .\notag
\end{align}

\paragraph{The quadratic terms in the integrals of \eqref{s3-cond_ortho_psi^2}. III. Conclusion.}

The last term to take into account in \eqref{s3-def_terme_source_F^1,pm} for computing all the contributions in the 
integrals of \eqref{s3-cond_ortho_psi^2} is $\p_\theta \psi^2 \, \cA^\pm \, \p_{Y_3} U^{\, 1,\pm}$. We thus need to compute:
\begin{equation*}
\int_{\bR^+} \mathrm{e}^{-|k|Y_3} \, \cL^+(k) \sbt 
\big( \widehat{\p_\theta \psi^2 \, \cA^\pm \, \p_{Y_3} U^{\, 1,+}|_{y_3=0}} \big) (k) \, {\rm d}Y_3 
- \, \int_{\bR^-} \mathrm{e}^{|k|Y_3} \, \cL^-(k) \sbt 
\big( \widehat{\p_\theta \psi^2 \, \cA^\pm \, \p_{Y_3} U^{\, 1,+}|_{y_3=0}} \big) (k) \, {\rm d}Y_3 \, .
\end{equation*}
Using the last equality in \eqref{appA-produits_hermitiens-1}, we get the contribution:
\begin{multline}
\label{s3-calcul_termes_bord_relation_dualité''''}
-\, \dfrac{i \, \tau}{2} \, \Big( (b^+)^2 - (b^-)^2 - (c^+)^2 + (c^-)^2 \Big) \, \text{\rm sgn} (k) \, \times \\
\sum_{k_1 +k_2 =k} \left( \dfrac{k_1^2 \, k_2}{|k| +|k_1|} \, \big( \text{\rm sgn} (k)-\text{\rm sgn} (k_1) \big) 
+\dfrac{k_1 \, k_2^2}{|k| +|k_2|} \, \big( \text{\rm sgn} (k)-\text{\rm sgn} (k_2) \big) \right) \, 
\widehat{\psi}^{\, 2} (k_1) \, \widehat{\psi}^{\, 2} (k_2) \, ,
\end{multline}
where we have symmetrized the kernel acting on $(\psi^2,\psi^2)$ once again. Adding \eqref{s3-calcul_termes_bord_relation_dualité''''} 
with \eqref{s3-calcul_termes_bord_relation_dualité'-ter} and dividing by $-2 \, \tau$, we end up -after a little bit of further 
simplifications- with the leading amplitude equation that must be satisfied by the nonzero Fourier modes of $\psi^2$:
\begin{multline}
\label{s3-edp_hatpsi^2}
\big( c^+ + c^- \big) \, \p_t \widehat{\psi}^{\, 2} (k) 
+\big( c^+ \, u_j^{0,+} \, + \, c^- \, u_j^{0,-} \, - \, b^+ \, H_j^{0,+} \, - \, b^- \, H_j^{0,-} \big) \, \p_{y_j} \widehat{\psi}^{\, 2} (k) \\
+i\, \Big( (c^+)^2 - (c^-)^2 - (b^+)^2 + (b^-)^2 \Big) \, \text{\rm sgn} (k) \, 
\sum_{k_1 +k_2 =k} \dfrac{|k_1| \, |k_2| \, |k_1 +k_2|}{|k_1| + |k_2| + |k_1 +k_2|} \, \widehat{\psi}^{\, 2} (k_1) \, \widehat{\psi}^{\, 2} (k_2) \, = \, 0 \, .
\end{multline}

The main novelty here with respect to \cite{AliHunter} is the slow modulation with respect to the spatial tangential variables 
$y_1,y_2$. It is evidenced by the transport term in the first line of \eqref{s3-edp_hatpsi^2}. Let us observe right away that 
the coefficients in front of $\p_t \widehat{\psi}^{\, 2} (k)$ and $\p_{y_j} \widehat{\psi}^{\, 2} (k)$ are real, which means that provided 
that $c^+ +c^- \neq 0$, we shall have indeed a (constant coefficient) transport operator with respect to $y_1,y_2$. Let 
us also observe that the `kernel' :
\begin{equation}
\label{def_kernel}
(k_1,k_2) \longmapsto \dfrac{|k_1| \, |k_2| \, |k_1 +k_2|}{|k_1| + |k_2| + |k_1 +k_2|} \, ,
\end{equation}
is well-defined for all relevant values of $(k_1,k_2)$, that is for $k_1+k_2 \neq 0$. In particular, the kernel vanishes on all 
couples of the form $(0,k_2)$ and $(k_1,0)$, which means that the equation \eqref{s3-edp_hatpsi^2} is an evolution equation 
for the nonzero Fourier modes of $\psi^2$ only. The kernel \eqref{def_kernel} is the same as the one arising in a simplified 
model for weakly nonlinear Rayleigh wave modulation in elastodynamics, see \cite{HIZ}, and in a related problem of 
magnetohydrodynamics, namely the plasma-vacuum interface problem \cite{Secchi}. Further consideration on amplitude 
equations for weakly nonlinear surface waves may be found in \cite{AliHunterParker,AustriaHunter}. The kernel in 
\eqref{def_kernel} coincides of course with the one derived in \cite{AliHunter} in the two-dimensional case. This is not 
surprising due to the isotropy of the MHD equations.

Let us also observe that the kernel in \eqref{def_kernel} can be written as:
$$
\Lambda (k_1,k_2,-k_1-k_2) \, ,\quad \Lambda (k_1,k_2,k_3) := \dfrac{|k_1| \, |k_2| \, |k_3|}{|k_1| + |k_2| + |k_3|} \, ,
$$
where $\Lambda$ is symmetric with respect to $(k_1,k_2,k_3)$. As evidenced in \cite{AliHunterParker}, this symmetry 
property is linked with a (formal) Hamiltonian structure for \eqref{s3-edp_hatpsi^2}, see further investigation of this property 
in \cite{BC1,BC2,AustriaHunter}.

\subsection{Solvability of the leading amplitude equation}

In order to prove a (local) well-posedness result for \eqref{s3-edp_hatpsi^2}, we first perform some reductions. Let us first 
show that the coefficient $c^+ + c^-$ in front of $\p_t \widehat{\psi}^{\, 2} (k)$ is nonzero, which means that the equation 
\eqref{s3-edp_hatpsi^2} is of evolutionary type. Indeed, we recall that the Lopatinskii determinant for the current vortex 
sheet problem is defined in \eqref{s3-def_det_lop}. Under Assumption \eqref{s3-hyp_stab_nappe_plane}, $\Delta$ has 
two \emph{simple} real roots, and since we know by Assumption \eqref{s3-hyp_tau_racine_det_lop} that $\tau$ is one 
of these two roots, there holds $\partial_\tau \Delta (\tau,\xi_1,\xi_2) =2\, (c^+ + c^-) \neq 0$. We also compute
$$
\partial_{\xi_j} \Delta (\tau,\xi_1,\xi_2) \, = \, 
2 \, \big( c^+ \, u_j^{0,+} \, + \, c^- \, u_j^{0,-} \, - \, b^+ \, H_j^{0,+} \, - \, b^- \, H_j^{0,-} \big) \, ,
$$
which means that \eqref{s3-edp_hatpsi^2} can be rewritten as:
\begin{multline}
\label{s3-edp_hatpsi^2'}
\p_t \widehat{\psi}^{\, 2} (k) 
+\dfrac{\partial_{\xi_j} \Delta (\tau,\xi_1,\xi_2)}{\partial_\tau \Delta (\tau,\xi_1,\xi_2)} \, \p_{y_j} \widehat{\psi}^{\, 2} (k) \\
+i\, \dfrac{(c^+)^2 - (c^-)^2 - (b^+)^2 + (b^-)^2}{c^+ + c^-} \, \text{\rm sgn} (k) \, 
\sum_{k_1 +k_2 =k} \dfrac{|k_1| \, |k_2| \, |k_1 +k_2|}{|k_1| + |k_2| + |k_1 +k_2|} \, \widehat{\psi}^{\, 2} (k_1) \, \widehat{\psi}^{\, 2} (k_2) 
\, = \, 0 \, .
\end{multline}
The transport operator in \eqref{s3-edp_hatpsi^2'} is governed by the \emph{group velocity} associated with the manifold 
along which the Lopatinskii determinant vanishes. This observation is a general fact that we can directly check for our 
particular problem, see \cite{Marcou,CW,WW}. Let us now define a bilinear symmetric operator $\cB$ that acts on 
$\cC^\infty$, $2\pi-$periodic functions in $\theta$ as follows:
\begin{equation*}
\forall \, k \in \bZ \setminus \{ 0 \} \, ,\quad 
\widehat{\cB (\varphi,\psi)} (k) \, := \, \sum_{k_1+k_2=k} \dfrac{|k_1| \, |k_2| \, |k_1 +k_2|}{|k_1| + |k_2| + |k_1 +k_2|} \, 
\widehat{\varphi} (k_1) \, \widehat{\psi} (k_2) \, .
\end{equation*}
The variables $(t,y')$ play here a role of parameters in $\varphi,\psi$. We do not need a definition for $\widehat{\cB (\varphi,\psi)} (0)$. 
Let us observe that the operator $\cB$ preserves real valued functions. This property is used below since the leading front $\psi^2$ 
is meant to be real valued. Then \eqref{s3-edp_hatpsi^2'} can be recast in the more compact form:
\begin{multline}
\label{s3-edp_hatpsi^2-final}
\forall \, k \in \bZ \setminus \{ 0 \} \, ,\\
\p_t \widehat{\psi}^{\, 2} (k) 
+ \dfrac{\partial_{\xi_j} \Delta (\tau,\xi_1,\xi_2)}{\partial_\tau \Delta (\tau,\xi_1,\xi_2)} \, \p_{y_j} \widehat{\psi}^{\, 2} (k) 
- \dfrac{(b^+)^2 - (b^-)^2 - (c^+)^2 + (c^-)^2}{c^+ + c^-} \, i \, \text{\rm sgn} (k) \, \widehat{\cB (\psi^2,\psi^2)} (k) \, = \, 0 \, ,
\end{multline}
where we recognize in the bilinear term the action of the so-called Hilbert transform $\cH$. In other words, we can also 
rewrite \eqref{s3-edp_hatpsi^2-final} as:
\begin{multline*}
\forall \, k \in \bZ \setminus \{ 0 \} \, ,\\
\p_t \widehat{\psi}^{\, 2} (k) + \dfrac{\partial_{\xi_j} \Delta (\tau,\xi_1,\xi_2)}{\partial_\tau \Delta (\tau,\xi_1,\xi_2)} \, \p_{y_j} \widehat{\psi}^{\, 2} (k) 
+ \dfrac{(b^+)^2 - (b^-)^2 - (c^+)^2 + (c^-)^2}{c^+ + c^-} \, \widehat{\cH \Big( \cB (\psi^2,\psi^2) \Big)} (k) \, = \, 0 \, ,
\end{multline*}
Let us observe that because of Assumption \eqref{s3-hyp_delta_neq_0}, the coefficient in front of the bilinear nonlocal operator 
is nonzero, for otherwise, combining with \eqref{s3-hyp_tau_racine_det_lop}, we would get $(c^\pm)^2 =(b^\pm)^2$ and it 
is proved in Appendix \ref{appendixA} that $(c^\pm)^2 \neq (b^\pm)^2$. Hence \eqref{s3-edp_hatpsi^2-final} is genuinely 
a nonlinear nonlocal equation of Hamilton-Jacobi type (since the kernel defining $\cB$ is homogeneous degree $2$). The 
well-posedness of nonlocal equations of the type \eqref{s3-edp_hatpsi^2-final} has been systematically studied in 
\cite{Hunter2006,B,Marcou,CW,WW} in either the pulse or wavetrain framework ($k \in \R$ or $k \in \Z$). The most convenient 
references for our purpose here are \cite{Hunter2006,WW} where the following result is proved\footnote{One can even consider 
more general kernels than the one in \eqref{s3-edp_hatpsi^2-final} but we rather refer the interested reader to the references 
instead of quoting the corresponding results in full generality.}. For notational convenience, we introduce $H^s_\sharp$, 
$s \ge 0$, as the Sobolev space of functions $g \in H^s (\bT_{y'}^2\times\bT_\theta)$ with zero mean with respect to $\theta$, 
and $H^\infty_\sharp := \cap_s H^s_\sharp$. The space $H^s_\sharp$ is equipped with the obvious norm defined with the help 
of Fourier coefficients in $(y',\theta)$, see, \emph{e.g.}, \cite{BCD}.

\begin{theorem}[\cite{Hunter2006}]
\label{s3-thm_hunter}
Let $s>7/2$ and let $g\in H^s_\sharp$. Then there exist a time $T > 0$ and a unique solution $\varphi \in \cC ([0,T];H^s_\sharp)$ to 
the Cauchy problem:
\begin{equation}
\label{s3-pb_cauchy_HJ}
\forall \, k \neq 0 \, ,\quad \p_t \widehat{\varphi} (k) 
+ \dfrac{\partial_{\xi_j} \Delta (\tau,\xi_1,\xi_2)}{\partial_\tau \Delta (\tau,\xi_1,\xi_2)} \, \p_{y_j} \widehat{\varphi} (k) 
- \dfrac{(b^+)^2 - (b^-)^2 - (c^+)^2 + (c^-)^2}{c^+ + c^-} \, i \, \text{\rm sgn} (k) \, \widehat{\cB (\varphi,\varphi)} (k) \, = \, 0 \, ,
\end{equation}
with $\varphi|_{t=0} =g$.

Moreover, if $g\in H_\sharp^\infty (\bT^3)$, then the unique zero mean solution $\varphi$ to \eqref{s3-pb_cauchy_HJ} belongs to 
$\cC^\infty\big( [0,T] \, ; \, H^\infty_\sharp \big)$, where $T>0$ is given by the previous result with, for instance, $s=4$.
\end{theorem}

The fact that the time $T > 0$ can be chosen to be independent of the Sobolev index $s$ follows from a \emph{tame} estimate 
for the solutions to \eqref{s3-pb_cauchy_HJ}. Such a tame estimate is somehow hidden in \cite{Hunter2006} but a detailed 
proof is given in \cite{WW} with even the incorporation of spatial tangential variables. The problem of global existence of solutions 
to \eqref{s3-pb_cauchy_HJ}, either in a weak or strong sense, is still open (see \cite{Hunter2006}); numerical simulations in 
\cite{AliHunter} seem to reveal that some smooth solutions to \eqref{s3-pb_cauchy_HJ} develop singularities in finite time.

\section{Construction of the leading profile}

At this stage, Theorem \ref{s3-thm_hunter} is sufficient to fully determine the oscillating modes of the leading profile $\psi^2$. 
For future use, let us indeed introduce the decomposition:
$$
\psi^2 (t,y',\theta) \, = \, \widehat{\psi}^{\, 2}(t,y',0) +\psi^2_\sharp (t,y',\theta) \, ,
$$
where $\psi^2_\sharp$ has zero mean with respect to the fast variable $\theta$ for all $(t,y')$. As follows from the expression 
\eqref{s3-edp_hatpsi^2'}, the leading amplitude equation \eqref{s3-edp_hatpsi^2'} only involves $\psi^2_\sharp$. Furthermore, 
in order to fulfill the initial condition \eqref{s3-def_cond_init_oscil_psi}, we impose:
$$
\psi^2_\sharp|_{t=0} \, = \, \psi^2_0 \, ,\quad \widehat{\psi}^{\, 2}(0)|_{t=0} \, = \, 0 \, .
$$
Consequently, the oscillating part $\psi^2_\sharp$ of the leading profile $\psi^2$ for the front is given by Theorem \ref{s3-thm_hunter} 
as the only solution, with zero mean with respect to $\theta$, to \eqref{s3-pb_cauchy_HJ} with initial condition $\psi^2_0$. 
This fixes the time $T > 0$, this final time depending only on a fixed Sobolev norm of the initial condition $\psi^2_0$. Once 
we have determined $\psi^2_\sharp$, we get thanks to \eqref{s3-U^1,pm_k_y3=0-bis} the expression of the leading profile 
$U^{\, 1,\pm}$ at $y_3=0$, that is on the boundary $\Gamma_0$. It then remains to lift $U^{\, 1,\pm}$ to any value of $y_3 
\in I^\pm$, with the only constraint that $U^{\, 1,\pm}$ should be of the general form \eqref{decomp_leading_profile} (meaning 
that we have only some scalar components $\gamma^{\, 1,\pm}(k)$, $k \neq 0$, at our disposal to lift $\pm |k| \, \widehat{\psi}^{\, 2}(k)$). 
For simplicity, we lift the trace of $\gamma^{\, 1,\pm}$ in the most simple way, that is, we set:
\begin{equation}
\label{s3-U^1,pm_k_final}
U^{\, 1,\pm} (t,y,Y_3,\theta) \, := \, \pm \, 
\sum_{k \neq 0} |k| \, \widehat{\psi}^{\, 2} (t,y',k) \, \chi (y_3) \, {\rm e}^{\mp \, |k| \, Y_3 +i \, k \, \theta} \, \cR^\pm(k) \, ,
\end{equation}
where the vectors $\cR^\pm(k)$ are defined in \eqref{appA-defRLkpm}. Let us recall again that $\chi$ is a fixed cut-off function that 
equals $1$ on $[-1/3,1/3]$ and that vanishes outside of $[-2/3,2/3]$. In particular, the form \eqref{s3-U^1,pm_k_final} is compatible 
with \eqref{decomp_leading_profile}, and there holds $U^{\, 1,\pm} =U^{\, 1,\pm}_\star \in S_\star^\pm$. The lifting procedure 
\eqref{s3-U^1,pm_k_final} is the same as in \cite{Marcou}.
\bigskip

\noindent Let us now summarize how we have constructed the leading profile $(U^{\, 1,\pm},\psi^2)$ and which properties it satisfies:
\begin{itemize}
 \item Since the leading profile $(U^{\, 1,\pm},\psi^2)$ must satisfy the homogeneous fast problem \eqref{fast_homogeneous_lead}, we 
 have obtained its general decomposition \eqref{decomp_leading_profile}.
 \item By imposing the (necessary) solvability conditions \eqref{compatibilite_pb_rapide_b}, \eqref{compatibilite_pb_rapide_c} on the 
 inhomogeneous fast problem \eqref{fast_first_corrector} satisfied by the first corrector $(U^{\, 2,\pm},\psi^3)$, we have shown that the 
 slow and fast means in \eqref{decomp_leading_profile} vanish (this property is linked to our choice 
 \eqref{s3-cond_moyenne_initiale_U^m,pm} of initial data for the fast mean and to the easiest possible choice of initial data for the 
 slow mean).
 \item By imposing the (necessary) solvability condition \eqref{compatibilite_pb_rapide_e} on the  inhomogeneous fast problem 
 \eqref{fast_first_corrector} satisfied by the first corrector $(U^{\, 2,\pm},\psi^3)$, we have determined the evolution of the nonzero 
 Fourier modes of the leading front $\psi^2$.
 \item The expression of the leading profile $(U^{\, 1,\pm},\psi^2)$ for any value of the slow normal variable $y_3$ is defined by the 
 simple lifting procedure \eqref{s3-U^1,pm_k_final}.
\end{itemize}

We have thus solved the homogeneous fast problem \eqref{fast_homogeneous_lead}, and enforced the solvability conditions 
\eqref{compatibilite_pb_rapide_b}, \eqref{compatibilite_pb_rapide_c}, \eqref{compatibilite_pb_rapide_e} on the fast problem 
\eqref{fast_first_corrector}. We have also satisfied the top and bottom boundary conditions \eqref{s3-cond_bords_fixes_uH_3^m,pm} 
for $m=0$, and the normalization condition \eqref{expression_I1(t)} for the slow mean of the total pressure. The expression 
\eqref{s3-U^1,pm_k_final}, as well as the fulfillment of \eqref{compatibilite_pb_rapide_b}, \eqref{compatibilite_pb_rapide_c}, 
\eqref{compatibilite_pb_rapide_e} on the system \eqref{fast_first_corrector}, are independent of the slow mean $\widehat{\psi}^{\, 2}(0)$ 
of the leading front. The latter function will be determined in the following Chapter when constructing the slow mean of the corrector 
$U^{\, 2,\pm}$.

As a concluding remark, let us observe that the solvability conditions \eqref{compatibilite_pb_rapide_a}, \eqref{compatibilite_pb_rapide_d} 
are not yet satisfied at this stage for the inhomogeneous fast problem \eqref{fast_first_corrector}. Namely, we shall have to check the 
relations:
\begin{align*}
& F_6^{\, 1,\pm}|_{y_3=Y_3=0}  \, = \, -b^\pm \, \p_\theta G_1^{\, 1,\pm} +c^\pm \, \p_\theta G_2^{\, 1,\pm} \, ,\\
& \p_{Y_3} F_6^{\, 1,\pm} +\xi_j \, \p_\theta F_{3+j}^{\, 1,\pm} -\tau \, \p_\theta F_8^{\, 1,\pm} \, = 0 \, ,
\end{align*}
see Lemma \ref{lem_compatibilite_bord} and Lemma \ref{lem_compatibilite_div} below in Chapter \ref{chapter5}. We now only have 
the mean $\widehat{\psi}^{\, 2}(0)$ of the leading front which has not been fixed yet, but this function does not enter the quantities 
involved in the previous two relations, so the solvability conditions \eqref{compatibilite_pb_rapide_a}, \eqref{compatibilite_pb_rapide_d} 
will need to come `for free'.

\chapter{Solving the WKB cascade II: the correctors}
\label{chapter5}

In this Chapter, we complete the construction of a solution to the WKB cascade, namely we are going to construct iteratively the 
corrector $(U^{\, m+1,\pm},\psi^{\, m+2})$, being given the collection of profiles $(U^{\, 1,\pm},\psi^2)$, $\dots$, $(U^{\, m,\pm},\psi^{\, m+1})$. 
This construction is based on an induction assumption stated as $H(m)$ below. This induction assumption includes several items, 
which we list as \eqref{inductionHm1}, \dots, \eqref{inductionHm7} below.

The verification of the initial step $H(1)$ has been performed in Chapter \ref{chapter4}, as we shall recall below, and our goal in this 
Chapter is to verify that $H(m)$ implies $H(m+1)$ for any integer $m \ge 1$. At the very end of this Chapter, we shall explain why the 
sequence of profiles $(U^{\, m,\pm},\psi^{\, m+1})$ allows for the construction of an approximate solution to \eqref{s3-equations_nappes_MHD} 
with highly oscillating data.
\bigskip

Our induction assumption is the following: there exists a time $T > 0$ and a collection of profiles $(U^{\, 1,\pm},\psi^2)$, $\dots$, 
$(U^{\, m-1,\pm},\psi^m)$, $(U^{\, m,\pm},\psi^{\, m+1}_\sharp)$ that satisfy the following seven properties:
\begin{equation}\tag{$H(m)-1$}
\label{inductionHm1}
(U^{\, 1,\pm},\psi^2) \, ,\dots \, , \, (U^{\, m-1,\pm},\psi^m) \, , \, (U^{\, m,\pm},\psi^{\, m+1}_\sharp) \in S^\pm \times 
H^\infty ([0,T] \times \bT^2 \times \bT) \, ,
\end{equation}
where the functional space $S^\pm$ in \eqref{inductionHm1} is given by Definition \ref{s3-def_espaces_fonctionnels},
\begin{equation}\tag{$H(m)-2$}
\label{inductionHm2}
\forall \, \mu \, = \, 1,\dots,m \, ,\quad \begin{cases}
\cL_f^\pm(\partial) \, U^{\mu,\pm} \, = \, F^{\mu-1,\pm} \, ,& y \in \Omega_0^\pm \, ,\, \pm Y_3 >0 \, ,\\
\p_{Y_3} H_3^{\mu,\pm} +\xi_j \, \p_\theta H_j^{\mu,\pm} \, = \, F_8^{\mu-1,\pm} \, ,& y \in \Omega_0^\pm \, ,\, \pm Y_3 >0 \, ,\\
B^+ \, U^{\mu,+}|_{y_3=Y_3=0} +B^- \, U^{\mu,-}|_{y_3=Y_3=0} +\partial_\theta \psi^{\mu+1} \, \ub \, = \, G^{\mu-1} \, ,
\end{cases}
\end{equation}
where the source term $F^{\mu-1,\pm}$, resp. $F_8^{\mu-1,\pm}$, is defined by \eqref{s3-def_terme_source_F^m,pm} (with 
$\mu-1$ rather than $m$ as the index), resp. the right hand side of \eqref{s3-cascade_div_H^m+1,pm}, and the boundary 
source term $G^{\mu-1}$ is defined by \eqref{terme_source_bord_BKW} (with $\mu-1$ rather than $m$ as the index),
\begin{equation}\tag{$H(m)-3$}
\label{inductionHm3}
\forall \, \mu \, = \, 1,\dots,m \, ,\quad 
\underline{u}_3^{\mu,\pm}|_{y_3=\pm 1} \, = \, \underline{H}_3^{\mu,\pm}|_{y_3=\pm 1} \, = \, 0 \, .
\end{equation}
\begin{equation}\tag{$H(m)-4$}
\label{inductionHm4}
\forall \, \mu \, = \, 1,\dots,m \, ,\quad \forall \, t \in [0,T] \, ,\quad {\mathcal I}^\mu(t) \, = \, 0 \, ,
\end{equation}
with ${\mathcal I}^\mu$ defined by \eqref{expression_Im(t)}, \eqref{expression_Im(t)-bis} (once again, with $\mu$ rather than $m$ 
as the index),
\begin{equation}\tag{$H(m)-5$}
\label{inductionHm5}
\widehat{\uF}^{\, m,\pm}(t,y,0) \, = \, 0 \, ,\quad \widehat{\uF}_8^{\, m,\pm}(t,y,0) \, = \, 0 \, ,
\end{equation}
\begin{equation}\tag{$H(m)-6$}
\label{inductionHm6}
\begin{cases}
u_j^{0,\pm} \, \widehat{F}_{7,\star}^{\, m,\pm} (0) -H_j^{0,\pm} \, \widehat{F}_{8,\star}^{\, m,\pm} (0) 
\, = \, \widehat{F}_{j,\star}^{\, m,\pm} (0) \, , & j=1,2 \, ,\\[1ex]
H_j^{0,\pm} \, \widehat{F}_{7,\star}^{\, m,\pm} (0) -u_j^{0,\pm} \, \widehat{F}_{8,\star}^{\, m,\pm} (0) 
\, = \, \widehat{F}_{3+j,\star}^{\, m,\pm} (0) \, , & j=1,2 \, ,
\end{cases}
\end{equation}
\begin{multline}\tag{$H(m)-7$}
\label{inductionHm7}
\forall \, k \neq 0 \, ,\quad 
\int_{\R^+} \mathrm{e}^{-|k| \, Y_3} \, \cL^+(k) \sbt \widehat{F}^{\, m,+} (t,y',0,Y_3,k) \,  {\rm d}Y_3 \, - \, 
\int_{\R^-} \mathrm{e}^{|k| \, Y_3} \, \cL^-(k) \sbt \widehat{F}^{\, m,-} (t,y',0,Y_3,k) \,  {\rm d}Y_3 \\
+ \, \ell_1^+ \, \widehat{G}_1^{\, m,+}(t,y',k) \, + \, \ell_2^+ \, \widehat{G}_2^{\, m,+}(t,y',k) \, + \, \ell_1^- \, \widehat{G}_1^{\, m,-}(t,y',k) 
\, + \, \ell_2^- \, \widehat{G}_2^{\, m,-}(t,y',k) \, = \, 0 \, .
\end{multline}
\bigskip

Several points should be emphasized. First of all, the time $T$ in the induction assumption $H(m)$ should not depend on $m$. 
In the analysis below, we shall prove that if $H(m)$ holds for some time $T>0$, then $H(m+1)$ holds for the \emph{same} time 
$T>0$. This will give eventually a uniform positive lifespan for all profiles. The lifespan is dictated by the solvability of the leading 
profile nonlinear equation \eqref{s3-edp_hatpsi^2'}. (Let us recall that global existence for \eqref{s3-edp_hatpsi^2'} is an open 
issue so far.) Another important point to notice is that, in the induction assumption $H(m)$, the profiles for the front $\psi^2$, \dots, 
$\psi^m$ are completely given, namely both their oscillating modes in $\theta$ and their mean with respect to the fast variable 
$\theta$, but the very last profile $\psi^{\, m+1}$ is given only through its oscillating modes. The mean of $\psi^{\, m+1}$ with 
respect to $\theta$ does not enter the definition of the source terms $F^{\, m,\pm},F_8^{\, m,\pm},G^m_\sharp$ in \eqref{inductionHm5}, 
\eqref{inductionHm6}, \eqref{inductionHm7}. The determination of the mean of $\psi^{\, m+1}$ with respect to $\theta$ will be 
one of the main points in the analysis below.

In the particular case $m=1$, $(H(1)-1)$ should be understood as $(U^{\, 1,\pm},\psi^2_\sharp) \in S^\pm \times H^\infty$, which is 
consistent with the analysis of Chapter \ref{chapter4} where we have only determined $\psi^2_\sharp$. We now briefly recall why 
the analysis of Chapter \ref{chapter4} implies that $H(1)$ is satisfied for some time $T>0$.

\section{The initial step of the induction}

From Theorem \ref{s3-thm_hunter}, we know that the leading front $\psi^2_\sharp$ belongs to $H^\infty ([0,T] \times \bT^2 \times \bT)$ 
for some time $T>0$, and consequently, from the expression \eqref{s3-U^1,pm_k_final}, the leading profile $U^{\, 1,\pm}$ belongs to $S^\pm$. 
(It even belongs to $S_\star^\pm$.) Furthermore, the expression \eqref{s3-U^1,pm_k_final} shows that the leading profile $(U^{\, 1,\pm},\psi^2)$ 
satisfies the homogeneous fast problem \eqref{fast_homogeneous_lead} (independently of the determination of the mean $\widehat{\psi}^{\, 2} 
(0)$ with respect to $\theta$). This means that for some time $T>0$, $(H(1)-1)$ and $(H(1)-2)$ are satisfied.

The expression \eqref{s3-U^1,pm_k_final} shows that $\uU^{\, 1,\pm}$ vanish, so $(H(1)-3)$ is also clearly satisfied, as well as $(H(1)-4)$ since, 
recalling \eqref{expression_Im(t)}, \eqref{expression_Im(t)-bis}, we have:
$$
{\mathcal I}^1(t) \, = \, \int_{\Omega_0^+} \widehat{\uq}^{\, 1,+}(t,y,0) \, {\rm d}y \, + \, 
\int_{\Omega_0^-} \widehat{\uq}^{\, 1,-}(t,y,0) \, {\rm d}y \, = \, 0 \, .
$$

In Chapter \ref{chapter4}, we have enforced the necessary solvability conditions \eqref{compatibilite_pb_rapide_b}, 
\eqref{compatibilite_pb_rapide_c} and \eqref{compatibilite_pb_rapide_e} on the inhomogeneous fast problem \eqref{fast_first_corrector}. 
This means equivalently that $(H(1)-5)$, $(H(1)-6)$ and $(H(1)-7)$ are satisfied for $m=1$. With our choice of initial conditions for the 
slow and fast means, conditions \eqref{compatibilite_pb_rapide_b} and \eqref{compatibilite_pb_rapide_c} even allowed us to conclude 
that the slow and fast means of $U^{\, 1,\pm}$ had to vanish. We have thus proved so far that there exists a time $T>0$ such that $H(1)$ 
is satisfied. All partial differential equations to be solved later on will be linear, which is the reason why there will be no further restriction 
on the final time $T$ (see \cite{Rauch} for the case of one phase weakly nonlinear geometric optics in the whole space).
\bigskip

In what remains of this Chapter, we assume that $H(m)$ is satisfied for some time $T>0$ and some integer $m \ge 1$. We are going 
to show that $H(m+1)$ is satisfied for the \emph{same} time $T>0$. Let us quickly observe that, because of \eqref{inductionHm2}, the 
fast problems:
\begin{equation*}
\begin{cases}
\cL_f^\pm(\partial) \, U^{\mu,\pm} \, = \, F^{\mu-1,\pm} \, ,& \\
\p_{Y_3} H_3^{\mu,\pm} +\xi_j \, \p_\theta H_j^{\mu,\pm} \, = \, F_8^{\mu-1,\pm} \, ,& \\
B^+ \, U^{\mu,+}|_{y_3=Y_3=0} +B^- \, U^{\mu,-}|_{y_3=Y_3=0} +\partial_\theta \psi^{\mu+1} \, \ub \, = \, G^{\mu-1} \, ,
\end{cases}
\quad \mu \, = \, 1,\dots,m \, ,
\end{equation*}
have a solution in $S^\pm \times H^\infty ([0,T] \times \bT^2 \times \bT)$. Using Theorem \ref{theorem_fast_problem}, this means that 
the source terms $(F^{0,\pm},G^0)$, \dots, $(F^{\, m-1,\pm},G^{\, m-1})$ satisfy the solvability conditions \eqref{compatibilite_pb_rapide}. 
This will be used in some calculations below. Observe that at this point the source terms $(F^{\, m,\pm},G^m)$ do not satisfy all solvability 
conditions \eqref{compatibilite_pb_rapide}. They only satisfy \eqref{compatibilite_pb_rapide_b} (because of \eqref{inductionHm5}), 
\eqref{compatibilite_pb_rapide_c} (because of \eqref{inductionHm6}) and \eqref{compatibilite_pb_rapide_e} (because of 
\eqref{inductionHm7}).

In order to prove that $H(m+1)$ is satisfied, we first need to show that we can solve the fast problem \eqref{fast_problem_m+1} below, 
which corresponds to enforcing $(H(m+1)-2)$. Using Theorem \ref{theorem_fast_problem} as well as the above remark, being able to 
solve \eqref{fast_problem_m+1} reduces to verifying that the source terms $(F^{\, m,\pm},G^m)$ satisfy the solvability conditions 
\eqref{compatibilite_pb_rapide_a} (compatibility at the boundary) and \eqref{compatibilite_pb_rapide_d} (compatibility for the divergence 
of the magnetic field). The verification of these two remaining solvability conditions is the first task to achieve in the induction process, 
and this will already determine a large part of the corrector $U^{\, m+1,\pm}$. Note that verifying \eqref{compatibilite_pb_rapide_a} and 
\eqref{compatibilite_pb_rapide_d} for $(F^{\, m,\pm},G^m)$ is independent of the mean $\widehat{\psi}^{\, m+1}(0)$. This function as well 
as the remaining degrees of freedom in $U^{\, m+1,\pm}$ will be used to enforce the relations $(H(m+1)-3)$, $(H(m+1)-5)$, $(H(m+1)-6)$, 
$(H(m+1)-7)$.

\section{Reduction to (almost) homogeneous equations}

In order to show that $H(m)$ implies $H(m+1)$, the first task is to solve the fast problem:
\begin{equation}
\label{fast_problem_m+1}
\begin{cases}
\cL_f^\pm(\partial) \, \bU^{\, m+1,\pm} \, = \, F^{\, m,\pm} \, ,& y \in \Omega_0^\pm \, ,\, \pm Y_3 >0 \, ,\\
\p_{Y_3} \cH_3^{\, m+1,\pm} +\xi_j \, \p_\theta \cH_j^{\, m+1,\pm} \, = \, F_8^{\, m,\pm} \, ,& y \in \Omega_0^\pm \, ,\, \pm Y_3 >0 \, ,\\
B^+ \, \bU^{\, m+1,+}|_{y_3=Y_3=0} +B^- \, \bU^{\, m+1,-}|_{y_3=Y_3=0} \, = \, G^m \, ,
\end{cases}
\end{equation}
which is possible if and only if the source terms satisfy the solvability conditions \eqref{compatibilite_pb_rapide}. Here we keep 
the notation of Theorem \ref{theorem_fast_problem} and look for a particular solution to the fast problem \eqref{fast_problem_m+1} 
of the form $\bU^{\, m+1,\pm} =(\cU^{\, m+1,\pm},\cH^{\, m+1,\pm},\cQ^{\, m+1,\pm})^T$ with no front in the jump conditions on $\Gamma_0$. 
We recall that the verification of \eqref{compatibilite_pb_rapide} is partly included in \eqref{inductionHm5}, \eqref{inductionHm6} 
and \eqref{inductionHm7}, but we still need to verify that the source terms in \eqref{fast_problem_m+1} satisfy the conditions 
\eqref{compatibilite_pb_rapide_a} and \eqref{compatibilite_pb_rapide_d}.

There is one point to keep in mind for later. Though the verification of \eqref{compatibilite_pb_rapide_a} and \eqref{compatibilite_pb_rapide_d} 
for the fast problem \eqref{fast_problem_m+1} will be independent of the mean $\widehat{\psi}^{\, m+1}(0)$ (because it only involves 
$\p_\theta G^m$), the source term $G^m$ does involve $\widehat{\psi}^{\, m+1}(0)$ and therefore the solution $\bU^{\, m+1,\pm}$ to 
\eqref{fast_problem_m+1} must depend on $\widehat{\psi}^{\, m+1}(0)$. We shall go back to this after proving Lemma 
\ref{lem_compatibilite_bord} and Lemma \ref{lem_compatibilite_div} below.

We start with the compatibility condition \eqref{compatibilite_pb_rapide_a} at the boundary $\Gamma_0$.

\begin{lemma}[Compatibility of the source terms at the boundary]
\label{lem_compatibilite_bord}
Under the induction assumption $H(m)$, there holds:
\begin{equation}
\label{lem_compatibilite_relation}
F_6^{\, m,\pm}|_{y_3=Y_3=0} \, = \, -b^\pm \, \p_\theta G_1^{\, m,\pm} +c^\pm \, \p_\theta G_2^{\, m,\pm} \, .
\end{equation}
\end{lemma}

\begin{proof}[Proof of Lemma \ref{lem_compatibilite_bord}]
The proof of Lemma \ref{lem_compatibilite_bord} follows from rather elementary algebraic manipulations, which 
we make explicit for the sake of completeness. We recall that by the induction assumption $H(m)$, the profiles 
$(U^{\, 1,\pm},\psi^2)$, $\dots$, $(U^{\, m,\pm},\psi^{\, m+1}_\sharp)$ are given, and we then compute the source term 
$F^{\, m,\pm}$ from the definition \eqref{s3-def_terme_source_F^m,pm}. The boundary source term $G^m$ is 
defined by \eqref{terme_source_bord_BKW}. Once again, the expressions of $G_1^{\, m,\pm}$, $G_2^{\, m,\pm}$ 
do depend on the mean of $\psi^{\, m+1}$ with respect to $\theta$, but the partial derivatives $\p_\theta G_1^{\, m,\pm}$ 
and $\p_\theta G_2^{\, m,\pm}$ do not. It follows from inspection of \eqref{s3-def_terme_source_F^m,pm} that the 
source term $F^{\, m,\pm}$ is uniquely defined in terms of $\psi^{\, m+1}_\sharp$. Hence the fulfillment of Lemma 
\ref{lem_compatibilite_bord} is independent of how we shall determine the mean of $\psi^{\, m+1}$ (which will be 
done in the following Section).

The verification of \eqref{lem_compatibilite_relation} is split in several steps. We first simplify the expression of 
the double trace $F_6^{\, m,\pm}|_{y_3=Y_3=0}$ by using the fast divergence constraints that follow from the previous 
steps in the induction. After these reductions, the expression of $F_6^{\, m,\pm}|_{y_3=Y_3=0}$ will only involve 
tangential derivatives of the traces of $u_3^{\mu,\pm},H_3^{\mu,\pm}$, $\mu=1,\dots,m$, with respect to $\Gamma_0$. 
The verification of \eqref{lem_compatibilite_relation} will then follow from differentiation of the jump conditions on 
$\Gamma_0$ at the previous steps of the induction. To make the expressions below easier to read, we omit in all 
the proof of Lemma \ref{lem_compatibilite_bord} the superscripts $\pm$. This is of no consequence since the 
expressions are absolutely identical on either side of the boundary $\Gamma_0$.
\bigskip

$\bullet$ \underline{Preliminary reductions}. We start from the general definition \eqref{s3-def_terme_source_F^m,pm}, of which 
we compute the sixth coordinate (see Appendix \ref{appendixA} for details). We then compute the double trace on $y_3=Y_3=0$, 
the main effect of which being that the cut-off functions $\chi^{[\ell]}$ vanish if $\ell \ge 1$, and $\chi^{[0]}$ equals $1$ at $y_3=0$. 
We also have $\dot{\chi}^{[\ell]}|_{y_3=0}=0$ for all $\ell \ge 0$, see \eqref{formule_chim} and \eqref{formule_chipointm} in 
Appendix \ref{appendixB}. We then get the expression:
\begin{equation}
\label{terme_source_F_6^m,pm-1}
\begin{aligned}
-F_6^m|_{y_3=Y_3=0} \, = \, & \, \big( \p_t +u_j^0 \, \p_{y_j} \big) H_3^m -H_j^0 \, \p_{y_j} u_3^m 
+\sum_{\ell_1+\ell_2=m+2} \p_\theta \psi^{\ell_1} \, \big( b \, \p_{Y_3} u_3^{\ell_2} -c \, \p_{Y_3} H_3^{\ell_2} \big) \\
& \, +\sum_{\ell_1+\ell_2=m+1} \p_\theta \psi^{\ell_1} \, \big( b \, \p_{y_3} u_3^{\ell_2} -c \, \p_{y_3} H_3^{\ell_2} \big) \\
& \, +\sum_{\ell_1+\ell_2=m+1} 
H_j^0 \, \p_{y_j} \psi^{\ell_1} \, \p_{Y_3} u_3^{\ell_2} -\big( \p_t +u_j^0 \, \p_{y_j} \big) \psi^{\ell_1} \, \p_{Y_3} H_3^{\ell_2} \\
& \, +\sum_{\ell_1+\ell_2=m} 
H_j^0 \, \p_{y_j} \psi^{\ell_1} \, \p_{y_3} u_3^{\ell_2} -\big( \p_t +u_j^0 \, \p_{y_j} \big) \psi^{\ell_1} \, \p_{y_3} H_3^{\ell_2} \\
& \, +\sum_{\substack{\ell_1+\ell_2=m+1 \\ \ell_1 \ge 1}} 
\xi_j \, u_j^{\ell_1} \, \p_\theta H_3^{\ell_2} +{\color{blue} H_3^{\ell_1} \, \xi_j \, \p_\theta u_j^{\ell_2}} 
-\xi_j \, H_j^{\ell_1} \, \p_\theta u_3^{\ell_2} -{\color{red} u_3^{\ell_1} \, \xi_j \, \p_\theta H_j^{\ell_2}} \\
& \, +\sum_{\substack{\ell_1+\ell_2=m \\ \ell_1 \ge 1}} 
u_j^{\ell_1} \, \p_{y_j} H_3^{\ell_2} +{\color{blue} H_3^{\ell_1} \, \p_{y_j} u_j^{\ell_2}} 
-H_j^{\ell_1} \, \p_{y_j} u_3^{\ell_2} -{\color{red} u_3^{\ell_1} \, \p_{y_j} H_j^{\ell_2}} \\
& \, +\sum_{\substack{\ell_1+\ell_2+\ell_3=m+2 \\ \ell_2 \ge 1}} \p_\theta \psi^{\ell_1} \, \big( 
{\color{red} u_3^{\ell_2} \, \xi_j \, \p_{Y_3} H_j^{\ell_3}} +\xi_j \, H_j^{\ell_2} \, \p_{Y_3} u_3^{\ell_3} 
-{\color{blue} H_3^{\ell_2} \, \xi_j \, \p_{Y_3} u_j^{\ell_3}} -\xi_j \, u_j^{\ell_2} \, \p_{Y_3} H_3^{\ell_3} \big) \\
& \, +\sum_{\substack{\ell_1+\ell_2+\ell_3=m+1 \\ \ell_2 \ge 1}} \p_{y_j} \psi^{\ell_1} \, \big( 
{\color{red} u_3^{\ell_2} \, \p_{Y_3} H_j^{\ell_3}} +H_j^{\ell_2} \, \p_{Y_3} u_3^{\ell_3} 
-{\color{blue} H_3^{\ell_2} \, \p_{Y_3} u_j^{\ell_3}} -u_j^{\ell_2} \, \p_{Y_3} H_3^{\ell_3} \big) \\
& \, +\sum_{\substack{\ell_1+\ell_2+\ell_3=m+1 \\ \ell_2 \ge 1}} \p_\theta \psi^{\ell_1} \, \big( 
{\color{red} u_3^{\ell_2} \, \xi_j \, \p_{y_3} H_j^{\ell_3}} +\xi_j \, H_j^{\ell_2} \, \p_{y_3} u_3^{\ell_3}
-{\color{blue} H_3^{\ell_2} \, \xi_j \, \p_{y_3} u_j^{\ell_3}} -\xi_j \, u_j^{\ell_2} \, \p_{y_3} H_3^{\ell_3} \big) \\
& \, +\sum_{\substack{\ell_1+\ell_2+\ell_3=m \\ \ell_2 \ge 1}} \p_{y_j} \psi^{\ell_1} \, \big( 
{\color{red} u_3^{\ell_2} \, \p_{y_3} H_j^{\ell_3}} +H_j^{\ell_2} \, \p_{y_3} u_3^{\ell_3} 
-{\color{blue} H_3^{\ell_2} \, \p_{y_3} u_j^{\ell_3}} -u_j^{\ell_2} \, \p_{y_3} H_3^{\ell_3} \big) \, ,
\end{aligned}
\end{equation}
where it is understood that all functions on the right hand side of \eqref{terme_source_F_6^m,pm-1} are evaluated at $y_3=Y_3=0$. 
We use the induction assumption $H(m)$, and more precisely the fast divergence constraints on the velocity and magnetic field 
included in \eqref{inductionHm2}. This simplifies accordingly the blue and red terms in \eqref{terme_source_F_6^m,pm-1}, and 
we get:
\begin{equation}
\label{terme_source_F_6^m,pm-2}
\begin{aligned}
-F_6^m|_{y_3=Y_3=0} \, = \, & \, \big( \p_t +u_j^0 \, \p_{y_j} \big) H_3^m -H_j^0 \, \p_{y_j} u_3^m 
+\sum_{\substack{\ell_1+\ell_2=m+1 \\ \ell_1 \ge 1}} 
\xi_j \, u_j^{\ell_1} \, \p_\theta H_3^{\ell_2} -\xi_j \, H_j^{\ell_1} \, \p_\theta u_3^{\ell_2} \\
& \, +\sum_{\substack{\ell_1+\ell_2=m \\ \ell_1 \ge 1}} 
u_j^{\ell_1} \, \p_{y_j} H_3^{\ell_2} -H_j^{\ell_1} \, \p_{y_j} u_3^{\ell_2} \\
& \, +\sum_{\ell_1+\ell_2=m+2} 
\p_\theta \psi^{\ell_1} \, \big( b \, {\color{ForestGreen} \p_{Y_3} u_3^{\ell_2}} -c \, {\color{magenta} \p_{Y_3} H_3^{\ell_2}} \big) 
+\sum_{\ell_1+\ell_2=m+1} \p_\theta \psi^{\ell_1} \, \big( b \, {\color{blue} \p_{y_3} u_3^{\ell_2}} -c \, {\color{orange} \p_{y_3} H_3^{\ell_2}} \big) \\
& \, +\sum_{\ell_1+\ell_2=m+1} 
H_j^0 \, \p_{y_j} \psi^{\ell_1} \, {\color{ForestGreen} \p_{Y_3} u_3^{\ell_2}} 
-\big( \p_t +u_j^0 \, \p_{y_j} \big) \psi^{\ell_1} \, {\color{magenta} \p_{Y_3} H_3^{\ell_2}} \\
& \, +\sum_{\ell_1+\ell_2=m} H_j^0 \, \p_{y_j} \psi^{\ell_1} \, {\color{blue} \p_{y_3} u_3^{\ell_2}} 
-\big( \p_t +u_j^0 \, \p_{y_j} \big) \psi^{\ell_1} \, {\color{orange} \p_{y_3} H_3^{\ell_2}} \\
& \, +\sum_{\substack{\ell_1+\ell_2=m+1 \\ \ell_1 \ge 1}} 
u_3^{\ell_1} \, {\color{magenta} \p_{Y_3} H_3^{\ell_2}} -H_3^{\ell_1} \, {\color{ForestGreen} \p_{Y_3} u_3^{\ell_2}} 
+\sum_{\substack{\ell_1+\ell_2=m \\ \ell_1 \ge 1}} 
u_3^{\ell_1} \, {\color{orange} \p_{y_3} H_3^{\ell_2}} -H_3^{\ell_1} \, {\color{blue} \p_{y_3} u_3^{\ell_2}} \\
& \, +\sum_{\substack{\ell_1+\ell_2+\ell_3=m+2 \\ \ell_2 \ge 1}} \p_\theta \psi^{\ell_1} \, \big( 
\xi_j \, H_j^{\ell_2} \, {\color{ForestGreen} \p_{Y_3} u_3^{\ell_3}} -\xi_j \, u_j^{\ell_2} \, {\color{magenta} \p_{Y_3} H_3^{\ell_3}} \big) \\
& \, +\sum_{\substack{\ell_1+\ell_2+\ell_3=m+1 \\ \ell_2 \ge 1}} \p_{y_j} \psi^{\ell_1} \, \big( 
H_j^{\ell_2} \, {\color{ForestGreen} \p_{Y_3} u_3^{\ell_3}} -u_j^{\ell_2} \, {\color{magenta} \p_{Y_3} H_3^{\ell_3}} \big) \\
& \, +\sum_{\substack{\ell_1+\ell_2+\ell_3=m+1 \\ \ell_2 \ge 1}} \p_\theta \psi^{\ell_1} \, \big( 
\xi_j \, H_j^{\ell_2} \, {\color{blue} \p_{y_3} u_3^{\ell_3}} -\xi_j \, u_j^{\ell_2} \, {\color{orange} \p_{y_3} H_3^{\ell_3}} \big) \\
& \, +\sum_{\substack{\ell_1+\ell_2+\ell_3=m \\ \ell_2 \ge 1}} \p_{y_j} \psi^{\ell_1} \, \big( 
H_j^{\ell_2} \, {\color{blue} \p_{y_3} u_3^{\ell_3}} -u_j^{\ell_2} \, {\color{orange} \p_{y_3} H_3^{\ell_3}} \big) \, .
\end{aligned}
\end{equation}
At this stage, we may use the boundary condition in \eqref{inductionHm2} to further simplify all normal derivatives. For instance 
we have highlighted in green all the fast normal derivatives $\p_{Y_3} u_3^{\ell_2}|_{y_3=Y_3=0}$, which are simplified by 
using $H_3^{\ell_1}|_{y_3=Y_3=0}=G_2^{\ell_1-1}$ for $\ell_1 \le m$. Other groups of terms that simplify are highlighted in 
pink, blue and orange. Relabeling the terms when necessary, we find that all normal derivatives in \eqref{terme_source_F_6^m,pm-2} 
disappear thanks to the fulfillment of the boundary conditions in \eqref{inductionHm2}. In other words, we are reduced to the 
amazingly simple expression:
\begin{align}
-F_6^m|_{y_3=Y_3=0} \, = \, & \, \big( \p_t +u_j^0 \, \p_{y_j} \big) H_3^m -H_j^0 \, \p_{y_j} u_3^m \notag \\
& \, +\sum_{\substack{\ell_1+\ell_2=m+1 \\ \ell_1 \ge 1}} 
\xi_j \, u_j^{\ell_1} \, \p_\theta H_3^{\ell_2} -\xi_j \, H_j^{\ell_1} \, \p_\theta u_3^{\ell_2} 
+\sum_{\substack{\ell_1+\ell_2=m \\ \ell_1 \ge 1}} 
u_j^{\ell_1} \, \p_{y_j} H_3^{\ell_2} -H_j^{\ell_1} \, \p_{y_j} u_3^{\ell_2} \, ,\label{terme_source_F_6^m,pm-3}
\end{align}
and in \eqref{terme_source_F_6^m,pm-3}, there are only tangential derivatives ($\p_t$, $\p_{y_j}$ or $\p_\theta$) with respect to 
the boundary $\{ y_3=Y_3=0 \}$. This means that we can differentiate the boundary conditions in \eqref{inductionHm2} and then 
substitute in \eqref{terme_source_F_6^m,pm-3} in order to further transform the expression of $F_6^m|_{y_3=Y_3=0}$. This 
substitution process is performed below.
\bigskip

$\bullet$ \underline{Step 1}. The linear terms. We split the right hand side of \eqref{terme_source_F_6^m,pm-3} in three pieces. 
The first piece gathers the `linear' terms. Using the boundary conditions in \eqref{inductionHm2}, we get:
\begin{align}
\cF_1 \, := \, \big( \p_t +u_j^0 \, \p_{y_j} \big) H_3^m -H_j^0 \, \p_{y_j} u_3^m \, = \, 
& \, {\color{magenta} \big( \p_t +u_j^0 \, \p_{y_j} \big) (b \, \p_\theta \psi^{\, m+1}) -H_j^0 \, \p_{y_j} (c \, \p_\theta \psi^{\, m+1})} \notag \\
& +\sum_{\ell_1+\ell_2=m+1} 
\xi_{j'} \, \p_\theta \psi^{\ell_1} \, \Big( \big( \p_t +u_j^0 \, \p_{y_j} \big) H_{j'}^{\ell_2} -H_j^0 \, \p_{y_j} u_{j'}^{\ell_2} \Big) \notag \\
& \, +\sum_{\ell_1+\ell_2=m} 
\p_{y_{j'}} \psi^{\ell_1} \, \Big( \big( \p_t +u_j^0 \, \p_{y_j} \big) H_{j'}^{\ell_2} -H_j^0 \, \p_{y_j} u_{j'}^{\ell_2} \Big) \label{terme_source_F_6^m,pm-4} \\
& \, +{\color{blue} \sum_{\substack{\ell_1+\ell_2=m+1 \\ \ell_2 \ge 1}} 
\big( \p_t +u_j^0 \, \p_{y_j} \big) \p_\theta \psi^{\ell_1} \, \xi_{j'} \, H_{j'}^{\ell_2} -H_j^0 \, \p_{y_j} \p_\theta \psi^{\ell_1} \, \xi_{j'} \, u_{j'}^{\ell_2}} \notag \\
& \, +\sum_{\substack{\ell_1+\ell_2=m \\ \ell_2 \ge 1}} 
\big( \p_t +u_j^0 \, \p_{y_j} \big)\p_{y_{j'}} \psi^{\ell_1} \, H_{j'}^{\ell_2} \, -H_j^0 \, \p_{y_j}\p_{y_{j'}} \psi^{\ell_1} \, u_{j'}^{\ell_2} \, .\notag
\end{align}
The first (pink) line in the above decomposition \eqref{terme_source_F_6^m,pm-4} of $\cF_1$ will directly contribute to 
$b \, \p_\theta G_1^m -c \, \p_\theta G_2^m$. The blue terms are highlighted in view of future cancellation. We now focus 
on the quadratic terms in the second line of \eqref{terme_source_F_6^m,pm-3}, which we are going to simplify with the blue 
term on the right hand side of \eqref{terme_source_F_6^m,pm-4}.
\bigskip

$\bullet$ \underline{Step 2}. The quadratic terms. (Fast derivatives.) Let us go on splitting the right hand side of 
\eqref{terme_source_F_6^m,pm-3} and use the boundary conditions in \eqref{inductionHm2} to get:
\begin{align}
\cF_2 \, := \, \sum_{\substack{\ell_1+\ell_2=m+1 \\ \ell_1 \ge 1}} 
\xi_j \, u_j^{\ell_1} \, \p_\theta H_3^{\ell_2} -\xi_j \, H_j^{\ell_1} \, \p_\theta u_3^{\ell_2} \, = \, 
& \, {\color{magenta} \sum_{\substack{\ell_1+\ell_2=m+2 \\ \ell_2 \ge 1}} \p_\theta^2 \psi^{\ell_1} \, 
\big( b \, \xi_j \, u_j^{\ell_2} -c \, \xi_j \, H_j^{\ell_2} \big)} \notag \\
+ \, & \, \sum_{\substack{\ell_1+\ell_2+\ell_3=m+2 \\ \ell_2 \ge 1}} \xi_{j'} \, \p_\theta \psi^{\ell_1} \, 
\big( \xi_j \, u_j^{\ell_2} \, \p_\theta H_{j'}^{\ell_3} -\xi_j \, H_j^{\ell_2} \, \p_\theta u_{j'}^{\ell_3} \big) \notag \\
+ \, & \, \sum_{\substack{\ell_1+\ell_2+\ell_3=m+1 \\ \ell_2 \ge 1}}\p_{y_{j'}} \psi^{\ell_1} \, 
\big( \xi_j \, u_j^{\ell_2} \, \p_\theta H_{j'}^{\ell_3} -\xi_j \, H_j^{\ell_2} \, \p_\theta u_{j'}^{\ell_3} \big) \label{terme_source_F_6^m,pm-5} \\
+ \, & \, \sum_{\substack{\ell_1+\ell_2+\ell_3=m+1 \\ \ell_2,\ell_3 \ge 1}}\p_{y_{j'}} \p_\theta \psi^{\ell_1} \, 
\big( \xi_j \, u_j^{\ell_2} \, H_{j'}^{\ell_3} -\xi_j \, H_j^{\ell_2} \, u_{j'}^{\ell_3} \big) \notag \\
+ \, & {\color{blue} \sum_{\substack{\ell_1+\ell_2=m+1 \\ \ell_2 \ge 1}} 
H_j^0 \, \p_{y_j} \p_\theta \psi^{\ell_1} \, \xi_{j'} \, u_{j'}^{\ell_2} 
-\big( \p_t +u_j^0 \, \p_{y_j} \big) \p_\theta \psi^{\ell_1} \, \xi_{j'} \, H_{j'}^{\ell_2}}. \notag 
\end{align}
The blue term in \eqref{terme_source_F_6^m,pm-4} simplifies with the last line in the decomposition \eqref{terme_source_F_6^m,pm-5} 
(also highlighted in blue). We thus get:
\begin{equation}
\label{terme_source_F_6^m,pm-6}
\begin{aligned}
\cF_1 +\cF_2 \, = \, 
& \, {\color{magenta} \big( \p_t +u_j^0 \, \p_{y_j} \big) (b \, \p_\theta \psi^{\, m+1}) -H_j^0 \, \p_{y_j} (c \, \p_\theta \psi^{\, m+1})} 
+{\color{magenta} \sum_{\substack{\ell_1+\ell_2=m+2 \\ \ell_2 \ge 1}} \p_\theta^2 \psi^{\ell_1} \, 
\big( b \, \xi_j \, u_j^{\ell_2} -c \, \xi_j \, H_j^{\ell_2} \big)} \\
& \, +\sum_{\ell_1+\ell_2=m+1} 
\xi_{j'} \, \p_\theta \psi^{\ell_1} \, \Big( \big( \p_t +u_j^0 \, \p_{y_j} \big) H_{j'}^{\ell_2} -H_j^0 \, \p_{y_j} u_{j'}^{\ell_2} \Big) \\
& +\sum_{\ell_1+\ell_2=m} 
\p_{y_{j'}} \psi^{\ell_1} \, \Big( \big( \p_t +u_j^0 \, \p_{y_j} \big) H_{j'}^{\ell_2} -H_j^0 \, \p_{y_j} u_{j'}^{\ell_2} \Big) \\
& \, +{\color{ForestGreen} \sum_{\substack{\ell_1+\ell_2=m \\ \ell_2 \ge 1}} 
\big( \p_t +u_j^0 \, \p_{y_j} \big)\p_{y_{j'}} \psi^{\ell_1} \, H_{j'}^{\ell_2} -H_j^0 \, \p_{y_j}\p_{y_{j'}} \psi^{\ell_1} \, u_{j'}^{\ell_2}} \\
& \, +\sum_{\substack{\ell_1+\ell_2+\ell_3=m+2 \\ \ell_2 \ge 1}} \xi_{j'} \, \p_\theta \psi^{\ell_1} \, 
\big( \xi_j \, u_j^{\ell_2} \, \p_\theta H_{j'}^{\ell_3} -\xi_j \, H_j^{\ell_2} \, \p_\theta u_{j'}^{\ell_3} \big) \\
& \, +\sum_{\substack{\ell_1+\ell_2+\ell_3=m+1 \\ \ell_2 \ge 1}}\p_{y_{j'}} \psi^{\ell_1} \, 
\big( \xi_j \, u_j^{\ell_2} \, \p_\theta H_{j'}^{\ell_3} -\xi_j \, H_j^{\ell_2} \, \p_\theta u_{j'}^{\ell_3} \big) \\
& \, +{\color{ForestGreen} \sum_{\substack{\ell_1+\ell_2+\ell_3=m+1 \\ \ell_2,\ell_3 \ge 1}}\p_{y_{j'}} \p_\theta \psi^{\ell_1} \, 
\big( \xi_j \, u_j^{\ell_2} \, H_{j'}^{\ell_3} -\xi_j \, H_j^{\ell_2} \, u_{j'}^{\ell_3} \big)} \, .
\end{aligned}
\end{equation}
Some terms in $\cF_1 +\cF_2$ are highlighted in green in view of further simplification with what remains of the right hand side 
of \eqref{terme_source_F_6^m,pm-3}.
\bigskip

$\bullet$ \underline{Step 3}. The quadratic terms. (Slow derivatives.) We now consider the last term in \eqref{terme_source_F_6^m,pm-3}, 
that is:
\begin{equation*}
\cF_3 \, := \, \sum_{\substack{\ell_1+\ell_2=m \\ \ell_1 \ge 1}} 
u_j^{\ell_1} \, \p_{y_j} H_3^{\ell_2} -H_j^{\ell_1} \, \p_{y_j} u_3^{\ell_2} \, ,
\end{equation*}
so that with our previous definitions, \eqref{terme_source_F_6^m,pm-3} reads:
$$
-F_6^m|_{y_3=Y_3=0} \, = \, \cF_1 +\cF_2 +\cF_3 \, .
$$
We simplify the expression of $\cF_3$ by adding it with the green terms on the right hand side of \eqref{terme_source_F_6^m,pm-6} 
and by using again the boundary conditions in \eqref{inductionHm2}. We get:
\begin{equation}
\label{terme_source_F_6^m,pm-7}
\begin{aligned}
\cF_3 \, + \, & \, \sum_{\substack{\ell_1+\ell_2=m \\ \ell_2 \ge 1}} 
\big( \p_t +u_j^0 \, \p_{y_j} \big)\p_{y_{j'}} \psi^{\ell_1} \, H_{j'}^{\ell_2} -H_j^0 \, \p_{y_j}\p_{y_{j'}} \psi^{\ell_1} \, u_{j'}^{\ell_2} \\
+ \, & \, \sum_{\substack{\ell_1+\ell_2+\ell_3=m+1 \\ \ell_2,\ell_3 \ge 1}}\p_{y_{j'}} \p_\theta \psi^{\ell_1} \, 
\big( \xi_j \, u_j^{\ell_2} \, H_{j'}^{\ell_3} -\xi_j \, H_j^{\ell_2} \, u_{j'}^{\ell_3} \big) \\
= \, & \, {\color{magenta} \sum_{\substack{\ell_1+\ell_2=m+1 \\ \ell_2 \ge 1}} \p_{y_j} \p_\theta \psi^{\ell_1} \, 
\big( b \, u_j^{\ell_2} -c \, H_j^{\ell_2} \big)} \\
+ \, & \, \sum_{\substack{\ell_1+\ell_2+\ell_3=m+1 \\ \ell_2 \ge 1}} \xi_{j'} \, \p_\theta \psi^{\ell_1} \, 
\big( u_j^{\ell_2} \, \p_{y_j} H_{j'}^{\ell_3} -H_j^{\ell_2} \, \p_{y_j} u_{j'}^{\ell_3} \big) \\
+ \, & \, \sum_{\substack{\ell_1+\ell_2+\ell_3=m \\ \ell_2 \ge 1}}\p_{y_{j'}} \psi^{\ell_1} \, 
\big( u_j^{\ell_2} \, \p_{y_j} H_{j'}^{\ell_3} -H_j^{\ell_2} \, \p_{y_j} u_{j'}^{\ell_3} \big) \, .
\end{aligned}
\end{equation}

Collecting \eqref{terme_source_F_6^m,pm-6} and \eqref{terme_source_F_6^m,pm-7}, we have derived so far the expression:
\begin{equation}
\label{terme_source_F_6^m,pm-8}
\begin{aligned}
-F_6^m|_{y_3=Y_3=0} \, = \, & \, {\color{magenta} \big( \p_t +u_j^0 \, \p_{y_j} \big) (b \, \p_\theta \psi^{\, m+1}) 
-H_j^0 \, \p_{y_j} (c \, \p_\theta \psi^{\, m+1})} \\
& \, + {\color{magenta} \sum_{\substack{\ell_1+\ell_2=m+2 \\ \ell_2 \ge 1}} \p_\theta^2 \psi^{\ell_1} \, 
\big( b \, \xi_j \, u_j^{\ell_2} -c \, \xi_j \, H_j^{\ell_2} \big)} 
+{\color{magenta} \sum_{\substack{\ell_1+\ell_2=m+1 \\ \ell_2 \ge 1}} \p_{y_j} \p_\theta \psi^{\ell_1} \, 
\big( b \, u_j^{\ell_2} -c \, H_j^{\ell_2} \big)} \\
& \, +\sum_{\ell_1+\ell_2=m+1} 
\xi_{j'} \, \p_\theta \psi^{\ell_1} \, \Big( \big( \p_t +u_j^0 \, \p_{y_j} \big) H_{j'}^{\ell_2} -H_j^0 \, \p_{y_j} u_{j'}^{\ell_2} \Big) \\
& \, +\sum_{\ell_1+\ell_2=m} 
\p_{y_{j'}} \psi^{\ell_1} \, \Big( \big( \p_t +u_j^0 \, \p_{y_j} \big) H_{j'}^{\ell_2} -H_j^0 \, \p_{y_j} u_{j'}^{\ell_2} \Big) \\
& \, +\sum_{\substack{\ell_1+\ell_2+\ell_3=m+2 \\ \ell_2 \ge 1}} \xi_{j'} \, \p_\theta \psi^{\ell_1} \, 
\big( \xi_j \, u_j^{\ell_2} \, \p_\theta H_{j'}^{\ell_3} -\xi_j \, H_j^{\ell_2} \, \p_\theta u_{j'}^{\ell_3} \big) \\
& \, +\sum_{\substack{\ell_1+\ell_2+\ell_3=m+1 \\ \ell_2 \ge 1}}\p_{y_{j'}} \psi^{\ell_1} \, 
\big( \xi_j \, u_j^{\ell_2} \, \p_\theta H_{j'}^{\ell_3} -\xi_j \, H_j^{\ell_2} \, \p_\theta u_{j'}^{\ell_3} \big) \\
& \, +\sum_{\substack{\ell_1+\ell_2+\ell_3=m+1 \\ \ell_2 \ge 1}} \xi_{j'} \, \p_\theta \psi^{\ell_1} \, 
\big( u_j^{\ell_2} \, \p_{y_j} H_{j'}^{\ell_3} -H_j^{\ell_2} \, \p_{y_j} u_{j'}^{\ell_3} \big) \\
& \, +\sum_{\substack{\ell_1+\ell_2+\ell_3=m \\ \ell_2 \ge 1}}\p_{y_{j'}} \psi^{\ell_1} \, 
\big( u_j^{\ell_2} \, \p_{y_j} H_{j'}^{\ell_3} -H_j^{\ell_2} \, \p_{y_j} u_{j'}^{\ell_3} \big) \, .
\end{aligned}
\end{equation}
\bigskip

$\bullet$ \underline{Step 4}. Conclusion. Going back to the definitions \eqref{s3-def_G_1^m,pm}, \eqref{s3-def_G_2^m,pm}, we compute:
\begin{align}
b \, \p_\theta G_1^m -c \, \p_\theta G_2^m \, = \, & \, {\color{magenta} 
\big( \p_t +u_j^0 \, \p_{y_j} \big) (b \, \p_\theta \psi^{\, m+1}) -H_j^0 \, \p_{y_j} (c \, \p_\theta \psi^{\, m+1})} \notag \\
& \, +{\color{magenta} \sum_{\substack{\ell_1+\ell_2=m+2 \\ \ell_2 \ge 1}} \p_\theta^2 \psi^{\ell_1} \, 
\big( b \, \xi_j \, u_j^{\ell_2} -c \, \xi_j \, H_j^{\ell_2} \big)} +\sum_{\ell_1+\ell_2=m+2} \p_\theta \psi^{\ell_1} \, 
\big( b \, \xi_j \, \p_\theta u_j^{\ell_2} -c \, \xi_j \, \p_\theta H_j^{\ell_2} \big) \notag \\
& \, +{\color{magenta} \sum_{\substack{\ell_1+\ell_2=m+1 \\ \ell_2 \ge 1}} \p_{y_j} \p_\theta \psi^{\ell_1} \, 
\big( b \, u_j^{\ell_2} -c \, H_j^{\ell_2} \big)} 
+\sum_{\ell_1+\ell_2=m+1} \p_{y_j} \psi^{\ell_1} \, \big( b \, \p_\theta u_j^{\ell_2} -c \, \p_\theta H_j^{\ell_2} \big) \, .\label{terme_source_F_6^m,pm-9}
\end{align}
Comparing the decomposition \eqref{terme_source_F_6^m,pm-9} with the first two lines in \eqref{terme_source_F_6^m,pm-8} (identical 
terms are highlighted in pink), we see that the proof of the relation \eqref{lem_compatibilite_relation} reduces to showing the identity:
\begin{align}
{\color{Cyan} \sum_{\ell_1+\ell_2=m+2} \xi_j \, \p_\theta \psi^{\ell_1} \, \big( b \, \p_\theta u_j^{\ell_2} -c \, \p_\theta H_j^{\ell_2} \big)} \, 
+ \, & \, \sum_{\ell_1+\ell_2=m+1} \p_{y_j} \psi^{\ell_1} \, \big( b \, \p_\theta u_j^{\ell_2} -c \, \p_\theta H_j^{\ell_2} \big) \notag \\
= \, & \, {\color{Cyan} \sum_{\ell_1+\ell_2=m+1} \xi_j \, \p_\theta \psi^{\ell_1} \, \Big( \big( \p_t +u_{j'}^0 \,\p_{y_{j'}} \big) H_j^{\ell_2} 
-H_{j'}^0 \,\p_{y_{j'}} u_j^{\ell_2} \Big)} \notag \\
& \, +\sum_{\ell_1+\ell_2=m} \p_{y_j} \psi^{\ell_1} \, \Big( \big( \p_t +u_{j'}^0 \,\p_{y_{j'}} \big) H_j^{\ell_2} -H_{j'}^0 \,\p_{y_{j'}} u_j^{\ell_2} \Big) \notag \\
& \, +{\color{Cyan} \sum_{\substack{\ell_1+\ell_2+\ell_3=m+2 \\ \ell_2 \ge 1}} \xi_j \, \xi_{j'} \, \p_\theta \psi^{\ell_1} \, 
\big( u_j^{\ell_2} \, \p_\theta H_{j'}^{\ell_3} -H_j^{\ell_2} \, \p_\theta u_{j'}^{\ell_3} \big)} \label{terme_source_F_6^m,pm-final} \\
& \, +\sum_{\substack{\ell_1+\ell_2+\ell_3=m+1 \\ \ell_2 \ge 1}} \p_{y_j} \psi^{\ell_1} \, 
\big( \xi_{j'} \, u_{j'}^{\ell_2} \, \p_\theta H_j^{\ell_3} -\xi_{j'} \, H_{j'}^{\ell_2} \, \p_\theta u_j^{\ell_3} \big) \notag \\
& \, +{\color{Cyan} \sum_{\substack{\ell_1+\ell_2+\ell_3=m+1 \\ \ell_2 \ge 1}} \xi_j \, \p_\theta \psi^{\ell_1} \, 
\big( u_{j'}^{\ell_2} \,\p_{y_{j'}} H_j^{\ell_3} -H_{j'}^{\ell_2} \,\p_{y_{j'}} u_j^{\ell_3} \big)} \notag \\
& \, +\sum_{\substack{\ell_1+\ell_2+\ell_3=m \\ \ell_2 \ge 1}} \p_{y_j} \psi^{\ell_1} \, 
\big( u_{j'}^{\ell_2} \,\p_{y_{j'}} H_j^{\ell_3} -H_{j'}^{\ell_2} \,\p_{y_{j'}} u_j^{\ell_3} \big) \, .\notag
\end{align}

As we are now going to show, the latter identity \eqref{terme_source_F_6^m,pm-final} is a consequence of the relation:
\begin{align}
\forall \, j \, = \, 1,2 \, ,\quad \forall \, \mu \, = \, 1,\dots,m \, ,\quad b \, \p_\theta u_j^\mu -c \, \p_\theta H_j^\mu \, = \, 
& \, \big( \p_t +u_{j'}^0 \,\p_{y_{j'}} \big) H_j^{\mu-1} -H_{j'}^0 \,\p_{y_{j'}} u_j^{\mu-1} \notag \\
& \, +\sum_{\substack{\ell_1+\ell_2=\mu \\ \ell_1 \ge 1}} 
\xi_{j'} \, u_{j'}^{\ell_1} \, \p_\theta H_j^{\ell_2} -\xi_{j'} \, H_{j'}^{\ell_1} \, \p_\theta u_j^{\ell_2} \label{terme_source_F_6^m,pm-10} \\
& \, +\sum_{\substack{\ell_1+\ell_2=\mu-1 \\ \ell_1 \ge 1}} 
u_{j'}^{\ell_1} \,\p_{y_{j'}} H_j^{\ell_2} -H_{j'}^{\ell_1} \,\p_{y_{j'}} u_j^{\ell_2} \, ,\notag
\end{align}
which holds on the boundary $\{ y_3=Y_3=0 \}$. Indeed, the relation \eqref{terme_source_F_6^m,pm-10} can be proved in 
a similar way as we have derived the relation \eqref{terme_source_F_6^m,pm-3}. (Observe the similarity between the right 
hand sides in those formulas.) Namely, we consider the fourth and fifth equations in the fast problems \eqref{inductionHm2}, 
take the double trace $y_3=Y_3=0$ and use the fast divergence constraints on the velocity and magnetic field. We then use 
the boundary conditions in \eqref{inductionHm2}, which make again all normal derivatives disappear. We feel free to skip the 
details of these calculations since they are really similar to those that can be found above for deriving 
\eqref{terme_source_F_6^m,pm-3}. Eventually, we get the relation \eqref{terme_source_F_6^m,pm-10} for the tangential 
components of the velocity and magnetic field.

Substituting the expression of $b \, \p_\theta u_j^{\ell_2} -c \, \p_\theta H_j^{\ell_2}$ in \eqref{terme_source_F_6^m,pm-10} 
on the left hand side of \eqref{terme_source_F_6^m,pm-final}, we see for instance that the light blue term  on the left of 
\eqref{terme_source_F_6^m,pm-final} can be decomposed as the sum of all three light blue terms on the right. The black 
term on the left of \eqref{terme_source_F_6^m,pm-final} can be also decomposed as the sum of the three black terms on 
the right, which completes the proof of the validity of \eqref{terme_source_F_6^m,pm-final}, and therefore of 
\eqref{lem_compatibilite_relation}. This means that the source terms in \eqref{fast_problem_m+1} satisfy the solvability 
condition \eqref{compatibilite_pb_rapide_a}.
\end{proof}

\noindent We are now going to prove the validity of the condition \eqref{compatibilite_pb_rapide_d} for the source terms of the 
fast problem \eqref{fast_problem_m+1}.

\begin{lemma}[Compatibility of the divergence source term]
\label{lem_compatibilite_div}
Under the induction assumption $H(m)$, there holds:
\begin{equation}
\label{lem_compatibilite_relation'}
\p_{Y_3} F_6^{\, m,\pm} +\xi_j \, \p_\theta F_{3+j}^{\, m,\pm} -\tau \, \p_\theta F_8^{\, m,\pm} \, = \, 0 \, .
\end{equation}
\end{lemma}

\begin{proof}[Proof of Lemma \ref{lem_compatibilite_div}]
We start by recalling the expressions of the source terms $F_{3+\alpha}^{\, m,\pm}$ and $F_8^{\, m,\pm}$. In all the proof, we omit 
to recall the superscripts $\pm$ in order to clarify the expressions below. It should be understood of course that the calculations 
are done on either side of the interface $\Gamma_0$. The source term $F_8^m$ associated with the fast divergence constraint 
on the magnetic field reads:
\begin{align}
F_8^m \, = \, -\nabla \cdot H^m \, 
& \, +\sum_{\ell_1+\ell_2=m+2} \p_\theta \psi^{\ell_1} \, {\color{ForestGreen} \p_{Y_3} \xi_j \, H_j^{\ell_2}} 
+\sum_{\ell_1+\ell_2=m+1} \p_{y_j} \psi^{\ell_1} \, \p_{Y_3} H_j^{\ell_2} \notag \\
& \, +\sum_{\ell_1+\ell_2+\ell_3=m+1} \chi^{[\ell_1]} \, \p_\theta \psi^{\ell_2} \, {\color{ForestGreen} \p_{y_3} \xi_j \, H_j^{\ell_3}} 
+\sum_{\ell_1+\ell_2+\ell_3=m} \chi^{[\ell_1]} \, \p_{y_j} \psi^{\ell_2} \, \p_{y_3} H_j^{\ell_3} \label{expressionF8m} \\
& \, +\sum_{\ell_1+\ell_2+\ell_3=m} \dot{\chi}^{[\ell_1]} \, \psi^{\ell_2} \, \p_{y_3} H_3^{\ell_3} \, ,\notag
\end{align}
while the tangential source terms $F_{3+j}^m$ read:
\begin{align}
-F_{3+j}^m \, = \, & \, 
\big( \p_t +u_{j'}^0 \,\p_{y_{j'}} \big) H_j^m -H_{j'}^0 \,\p_{y_{j'}} u_j^m -u_j^0 \, \nabla \cdot H^m +H_j^0 \, \nabla \cdot u^m \notag \\
& \, +\sum_{\ell_1+\ell_2=m+2} \p_\theta \psi^{\ell_1} \, \big( {\color{ForestGreen} b \, \p_{Y_3} u_j^{\ell_2} -c \, \p_{Y_3} H_j^{\ell_2} 
-H_j^0 \, \p_{Y_3} \xi_{j'} \, u_{j'}^{\ell_2} +u_j^0 \, \p_{Y_3} \xi_{j'} \, H_{j'}^{\ell_2}} \big) \notag \\
& \, +\sum_{\ell_1+\ell_2=m+1}\p_{y_{j'}} \psi^{\ell_1} \, \big( H_{j'}^0 \, \p_{Y_3} u_j^{\ell_2} -u_{j'}^0 \, \p_{Y_3} H_j^{\ell_2} 
-H_j^0 \, \p_{Y_3} u_{j'}^{\ell_2} +u_j^0 \, \p_{Y_3} H_{j'}^{\ell_2} \big) \notag \\
& \, -\sum_{\ell_1+\ell_2=m+1} \p_t \psi^{\ell_1} \, \p_{Y_3} H_j^{\ell_2} \notag \\
& \, +\sum_{\ell_1+\ell_2+\ell_3=m+1} \chi^{[\ell_1]} \, \p_\theta \psi^{\ell_2} \, \big( {\color{ForestGreen} b \, \p_{y_3} u_j^{\ell_3} 
-c \, \p_{y_3} H_j^{\ell_3} -H_j^0 \, \p_{y_3} \xi_{j'} \, u_{j'}^{\ell_3} +u_j^0 \, \p_{y_3} \xi_{j'} \, H_{j'}^{\ell_3}} \big) \notag \\
& \, +\sum_{\ell_1+\ell_2+\ell_3=m} \chi^{[\ell_1]} \,\p_{y_{j'}} \psi^{\ell_2} \, \big( H_{j'}^0 \, \p_{y_3} u_j^{\ell_3} 
-u_{j'}^0 \, \p_{y_3} H_j^{\ell_3} -H_j^0 \, \p_{y_3} u_{j'}^{\ell_3} +u_j^0 \, \p_{y_3} H_{j'}^{\ell_3} \big) \notag \\
& \, -\sum_{\ell_1+\ell_2+\ell_3=m} \chi^{[\ell_1]} \,\p_t \psi^{\ell_2} \, \p_{y_3} H_j^{\ell_3} 
+\sum_{\ell_1+\ell_2+\ell_3=m} \dot{\chi}^{[\ell_1]} \, \psi^{\ell_2} \, 
\big( u_j^0 \, \p_{y_3} H_3^{\ell_3} -H_j^0 \, \p_{y_3} u_3^{\ell_3} \big) \notag \\
& \, +\sum_{\substack{\ell_1+\ell_2=m+1 \\ \ell_1,\ell_2 \ge 1}} 
{\color{red} \p_\theta \big( \xi_{j'} \, u_{j'}^{\ell_1} \, H_j^{\ell_2} -\xi_{j'} \, H_{j'}^{\ell_1} \, u_j^{\ell_2} \big)} 
+{\color{blue} \p_{Y_3} \big( u_3^{\ell_1} \, H_j^{\ell_2} -H_3^{\ell_1} \, u_j^{\ell_2} \big)} \label{expressionF3+jm} \\
& \, +\sum_{\substack{\ell_1+\ell_2=m \\ \ell_1,\ell_2 \ge 1}} 
\nabla \cdot \big( H_j^{\ell_1} \, u^{\ell_2} -u_j^{\ell_1} \, H^{\ell_2} \big) \notag \\
& \, +\sum_{\substack{\ell_1+\ell_2+\ell_3=m+2 \\ \ell_2,\ell_3 \ge 1}} \p_\theta \psi^{\ell_1} \, 
{\color{red} \xi_{j'} \, \p_{Y_3} \big( u_j^{\ell_2} \, H_{j'}^{\ell_3} -H_j^{\ell_2} \, u_{j'}^{\ell_3} \big)} 
+\sum_{\substack{\ell_1+\ell_2+\ell_3=m+1 \\ \ell_2,\ell_3 \ge 1}}\p_{y_{j'}} \psi^{\ell_1} \, 
\p_{Y_3} \big( u_j^{\ell_2} \, H_{j'}^{\ell_3} -H_j^{\ell_2} \, u_{j'}^{\ell_3} \big) \notag \\
& \, +\sum_{\substack{\ell_1+\cdots+\ell_4=m+1 \\ \ell_3,\ell_4 \ge 1}} \chi^{[\ell_1]} \, \p_\theta \psi^{\ell_2} 
\, {\color{red} \p_{y_3} \big( u_j^{\ell_3} \, \xi_{j'} \, H_{j'}^{\ell_4} -H_j^{\ell_3} \, \xi_{j'} \, u_{j'}^{\ell_4} \big)} \notag \\
& \, +\sum_{\substack{\ell_1+\cdots+\ell_4=m \\ \ell_3,\ell_4 \ge 1}} \chi^{[\ell_1]} \,\p_{y_{j'}} \psi^{\ell_2} \, 
\p_{y_3} \big( u_j^{\ell_3} \, H_{j'}^{\ell_4} -H_j^{\ell_3} \, u_{j'}^{\ell_4} \big) \notag \\
& \, +\sum_{\substack{\ell_1+\cdots+\ell_4=m \\ \ell_3,\ell_4 \ge 1}} \dot{\chi}^{[\ell_1]} \, \psi^{\ell_2} \, 
\p_{y_3} \big( u_j^{\ell_3} \, H_3^{\ell_4} -H_j^{\ell_3} \, u_3^{\ell_4} \big) \, .\notag
\end{align}
In \eqref{expressionF3+jm} and in all what follows, the symbol $\nabla \cdot$ refers to the divergence with respect to the $y$ variable. 
Eventually, the normal source term $F_6^m$ reads:
\begin{align}
-F_6^m \, = \, & \, \big( \p_t +u_j^0 \, \p_{y_j} \big) H_3^m -H_j^0 \, \p_{y_j} u_3^m 
+\sum_{\ell_1+\ell_2=m+2} \p_\theta \psi^{\ell_1} \, \big( b \, \p_{Y_3} u_3^{\ell_2} -c \, \p_{Y_3} H_3^{\ell_2} \big) \notag \\
& \, +\sum_{\ell_1+\ell_2=m+1} H_j^0 \, \p_{y_j} \psi^{\ell_1} \, \p_{Y_3} u_3^{\ell_2} 
-\big( \p_t +u_j^0 \, \p_{y_j} \big) \psi^{\ell_1} \, \p_{Y_3} H_3^{\ell_2} \notag \\
& \, +\sum_{\ell_1+\ell_2+\ell_3=m+1}  \chi^{[\ell_1]} \, \p_\theta \psi^{\ell_2} \, 
\big( b \, \p_{y_3} u_3^{\ell_3} -c \, \p_{y_3} H_3^{\ell_3} \big) \notag \\
& \, +\sum_{\ell_1+\ell_2+\ell_3=m} \chi^{[\ell_1]} \, \Big( H_j^0 \, \p_{y_j} \psi^{\ell_2} \, \p_{y_3} u_3^{\ell_3} 
-\big( \p_t +u_j^0 \, \p_{y_j} \big) \psi^{\ell_2} \, \p_{y_3} H_3^{\ell_3} \Big) \notag \\
& \, +\sum_{\substack{\ell_1+\ell_2=m+1 \\ \ell_1,\ell_2 \ge 1}} {\color{blue} \xi_j \, 
\p_\theta \big( u_j^{\ell_1} \, H_3^{\ell_2} -H_j^{\ell_1} \, u_3^{\ell_2} \big)} 
+\sum_{\substack{\ell_1+\ell_2=m \\ \ell_1,\ell_2 \ge 1}} 
\p_{y_j} \big( u_j^{\ell_1} \, H_3^{\ell_2} -H_j^{\ell_1} \, u_3^{\ell_2} \big) \label{expressionF6m} \\
& \, +\sum_{\substack{\ell_1+\ell_2+\ell_3=m+2 \\ \ell_2,\ell_3 \ge 1}} \p_\theta \psi^{\ell_1} \, \xi_j \, 
\p_{Y_3} \big( u_3^{\ell_2} \, H_j^{\ell_3} -H_3^{\ell_2} \, u_j^{\ell_3} \big) 
+\sum_{\substack{\ell_1+\ell_2+\ell_3=m+1 \\ \ell_2,\ell_3 \ge 1}} \p_{y_j} \psi^{\ell_1} \, 
\p_{Y_3} \big( u_3^{\ell_2} \, H_j^{\ell_3} -H_3^{\ell_2} \, u_j^{\ell_3} \big) \notag \\
& \, +\sum_{\substack{\ell_1+\cdots+\ell_4=m+1 \\ \ell_3,\ell_4 \ge 1}} \chi^{[\ell_1]} \, \p_\theta \psi^{\ell_2} 
\, \xi_j \, \p_{y_3} \big( u_3^{\ell_3} \, H_j^{\ell_4} -H_3^{\ell_3} \, u_j^{\ell_4} \big) \notag \\
& \, +\sum_{\substack{\ell_1+\cdots+\ell_4=m \\ \ell_3,\ell_4 \ge 1}} \chi^{[\ell_1]} \, \p_{y_j} \psi^{\ell_2} \, 
\p_{y_3} \big( u_3^{\ell_3} \, H_j^{\ell_4} -H_3^{\ell_3} \, u_j^{\ell_4} \big) \, .\notag
\end{align}

We now define the quantity:
$$
\cD^m \, := \, \p_{Y_3} F_6^m +\xi_j \, \p_\theta F_{3+j}^m -\tau \, \p_\theta F_8^m \, ,
$$
and the goal of course is to show that $\cD^m$ vanishes under the induction assumption $H(m)$. The calculation 
splits again in several steps.
\bigskip

$\bullet$ \underline{Step 1}. Collecting the terms.

Let us first observe that the above highlighted blue terms in $F_{3+\alpha}^m$ contribute to zero in the calculation of 
$\cD^m$ (since they correspond to a fast curl of which we take the fast divergence). We also note that the red terms 
in $F_{3+j}^m$ drop out for symmetry reasons when we compute the linear combination $\xi_j \, F_{3+j}^m$, and the 
green terms also drop out when we compute $\xi_j \, F_{3+j}^m-\tau \, F_8^m$. For the following calculations, it is also 
useful to recall that the profiles $\psi^\mu$ only depend on $(t,y',\theta)$, the functions $\chi^{[\ell]}$ depend on $(t,y,\theta)$, 
and the profiles $U^{\ell,\pm}$ depend on all variables $(t,y,Y_3,\theta)$. Using the fast problems \eqref{inductionHm2} 
that are satisfied at the previous steps of the induction, we are left (!) with the following expression for $\cD^m$:
\begin{align}
\cD^m \, & = \, {\color{VioletRed} -\p_t F_8^{\, m-1} +\p_{y_\alpha} F_{3+\alpha}^{\, m-1}} 
-{\color{ForestGreen} \sum_{\ell_1+\ell_2=m+1} \p_\theta \psi^{\ell_1} \, \p_{Y_3} \big( \xi_j \, F_{3+j}^{\ell_2} -\tau \, F_8^{\ell_2} \big)} \notag \\
& \, +{\color{blue} \sum_{\ell_1+\ell_2=m+1} \p_t \p_\theta \psi^{\ell_1} \, \xi_j \, \p_{Y_3} H_j^{\ell_2}} 
+{\color{Orange} \sum_{\ell_1+\ell_2=m} \p_t \psi^{\ell_1} \, \p_{Y_3} F_8^{\ell_2}} \notag \\
& \, +{\color{blue} \sum_{\ell_1+\ell_2 =m+1} \p_{y_j} \p_\theta \psi^{\ell_1} \, \big( u_j^0 \, \xi_{j'} \, \p_{Y_3} H_{j'}^{\ell_2} 
-H_j^0 \, \xi_{j'} \, \p_{Y_3} u_{j'}^{\ell_2} +b \, \p_{Y_3} u_j^{\ell_2} -c \, \p_{Y_3} H_j^{\ell_2} \big)} \notag \\
& \, -{\color{Mahogany} \sum_{\ell_1+\ell_2 =m} \p_{y_j} \psi^{\ell_1} \, \p_{Y_3} F_{3+j}^{\ell_2}} 
-\sum_{\ell_1+\ell_2+\ell_3=m} \chi^{[\ell_1]} \, \p_\theta \psi^{\ell_2} \, 
\p_{y_3} \big( \xi_j \, F_{3+j}^{\ell_3} -\tau \, F_8^{\ell_3} \big) \notag \\
& \, +{\color{blue} \sum_{\ell_1+\ell_2+\ell_3=m} \chi^{[\ell_1]} \, \p_t \p_\theta \psi^{\ell_2} \, \xi_j \, \p_{y_3} H_j^{\ell_3}} 
+\sum_{\ell_1+\ell_2+\ell_3=m-1} \chi^{[\ell_1]} \, \p_t \psi^{\ell_2} \, \p_{y_3} F_8^{\ell_3} \notag \\
& \, +{\color{blue} \sum_{\ell_1+\ell_2+\ell_3=m} \chi^{[\ell_1]} \, \p_{y_j} \p_\theta \psi^{\ell_2} \, \Big( 
u_j^0 \, \xi_{j'} \, \p_{y_3} H_{j'}^{\ell_3} -H_j^0 \, \xi_{j'} \, \p_{y_3} u_{j'}^{\ell_3} 
+b \, \p_{y_3} u_j^{\ell_3} -c \, \p_{y_3} H_j^{\ell_3} \Big)} \notag \\
& \, -\sum_{\ell_1+\ell_2+\ell_3=m-1} \chi^{[\ell_1]} \, \p_{y_j} \psi^{\ell_2} \, \p_{y_3} F_{3+j}^{\ell_3} \notag \\
& \, +{\color{VioletRed} \sum_{\substack{\ell_1+\ell_2=m \\ \ell_1,\ell_2 \ge 1}} \p_{y_j} \p_{Y_3} \big( 
H_j^{\ell_1} \, u_3^{\ell_2} -u_j^{\ell_1} \, H_3^{\ell_2} \big)} 
+{\color{VioletRed} \sum_{\substack{\ell_1+\ell_2=m \\ \ell_1,\ell_2 \ge 1}} \nabla \cdot \p_\theta \big( 
\xi_j \, u_j^{\ell_1} \, H^{\ell_2} -\xi_j \, H_j^{\ell_1} \, u^{\ell_2} \big)} \notag \\
& \, +{\color{ForestGreen} \sum_{\substack{\ell_1+\ell_2+\ell_3=m+2 \\ \ell_2,\ell_3 \ge 1}} \p_\theta \psi^{\ell_1} \, 
\p_{Y_3}^2 \big( H_3^{\ell_2} \, \xi_j \, u_j^{\ell_3} -u_3^{\ell_2} \, \xi_j \, H_j^{\ell_3} \big)} 
+{\color{Mahogany} \sum_{\substack{\ell_1+\ell_2+\ell_3=m+1 \\ \ell_2,\ell_3 \ge 1}} \p_{y_j} \psi^{\ell_1} \, 
\p_{Y_3}^2 \big( H_3^{\ell_2} \, u_j^{\ell_3} -u_3^{\ell_2} \, H_j^{\ell_3} \big)} \notag \\
& \, +{\color{Mahogany} \sum_{\substack{\ell_1+\ell_2+\ell_3=m+1 \\ \ell_2,\ell_3 \ge 1}} \p_{y_j} \psi^{\ell_1} \, 
\p_\theta \p_{Y_3} \big( u_j^{\ell_2} \, \xi_{j'} \, H_{j'}^{\ell_3} -H_j^{\ell_2} \, \xi_{j'} \, u_{j'}^{\ell_3} \big)} \notag \\
& \, +{\color{blue} \sum_{\substack{\ell_1+\ell_2+\ell_3=m+1 \\ \ell_2,\ell_3 \ge 1}} \p_{y_j} \p_\theta \psi^{\ell_1} \, 
\p_{Y_3} \big( u_j^{\ell_2} \, \xi_{j'} \, H_{j'}^{\ell_3} -H_j^{\ell_2} \, \xi_{j'} \, u_{j'}^{\ell_3} \big)} \label{groscalcul-1} \\
& \, +\sum_{\substack{\ell_1+\cdots+\ell_4=m+1 \\ \ell_3,\ell_4 \ge 1}} \chi^{[\ell_1]} \, \p_\theta \psi^{\ell_2} \, 
\p_{y_3} \p_{Y_3} \big( H_3^{\ell_3} \, \xi_j \, u_j^{\ell_4} -u_3^{\ell_3} \, \xi_j \, H_j^{\ell_4} \big) \notag \\
& \, +\sum_{\substack{\ell_1+\cdots+\ell_4=m \\ \ell_3,\ell_4 \ge 1}} \chi^{[\ell_1]} \, \p_{y_j} \psi^{\ell_2} \, 
\p_{y_3} \p_{Y_3} \big( H_3^{\ell_3} \, u_j^{\ell_4} -u_3^{\ell_3} \, H_j^{\ell_4} \big) \notag \\
& \, +\sum_{\substack{\ell_1+\cdots+\ell_4=m \\ \ell_3,\ell_4 \ge 1}} \chi^{[\ell_1]} \, \p_{y_j} \psi^{\ell_2} \, 
\p_{y_3} \p_\theta \big( u_j^{\ell_3} \, \xi_{j'} \, H_{j'}^{\ell_4} -H_j^{\ell_3} \, \xi_{j'} \, u_{j'}^{\ell_4} \big) \notag \\
& \, +{\color{blue} \sum_{\substack{\ell_1+\cdots+\ell_4=m \\ \ell_3,\ell_4 \ge 1}} \chi^{[\ell_1]} \, \p_{y_j} \p_\theta \psi^{\ell_2} \, 
\p_{y_3} \big( u_j^{\ell_3} \, \xi_{j'} \, H_{j'}^{\ell_4} -H_j^{\ell_3} \, \xi_{j'} \, u_{j'}^{\ell_4} \big)} 
+\sum_{\ell_1+\ell_2+\ell_3=m} \p_\theta \chi^{[\ell_1]} \, \p_t \psi^{\ell_2} \, \xi_j \, \p_{y_3} H_j^{\ell_3} \notag \\
& \, +\sum_{\ell_1+\ell_2+\ell_3=m} \p_\theta \chi^{[\ell_1]} \, \p_{y_j} \psi^{\ell_2} \, \big( 
u_j^0 \, \xi_{j'} \, \p_{y_3} H_{j'}^{\ell_3} -H_j^0 \, \xi_{j'} \, \p_{y_3} u_{j'}^{\ell_3} 
+b \, \p_{y_3} u_j^{\ell_3} -c \, \p_{y_3} H_j^{\ell_3} \big) \notag \\
& \, +\sum_{\substack{\ell_1+\cdots+\ell_4=m \\ \ell_3,\ell_4 \ge 1}} \p_\theta \chi^{[\ell_1]} \, \p_{y_j} \psi^{\ell_2} \, 
\p_{y_3} \big( u_j^{\ell_3} \, \xi_{j'} \, H_{j'}^{\ell_4} -H_j^{\ell_3} \, \xi_{j'} \, u_{j'}^{\ell_4} \big) \notag \\
& \, +\p_\theta \Big\{ \sum_{\ell_1+\ell_2+\ell_3=m} \dot{\chi}^{[\ell_1]} \, \psi^{\ell_2} \, 
\big( b \, \p_{y_3} u_3^{\ell_3} -c \, \p_{y_3} H_3^{\ell_3} \big) 
+\sum_{\substack{\ell_1+\cdots+\ell_4=m \\ \ell_3,\ell_4 \ge 1}} \dot{\chi}^{[\ell_1]} \, \psi^{\ell_2} \, 
\p_{y_3} \big( \xi_j \, H_j^{\ell_3} \, u_3^{\ell_4} -\xi_j \, u_j^{\ell_3} \, H_3^{\ell_4} \big) \Big\} \, .\notag 
\end{align}
\bigskip

$\bullet$ \underline{Step 2}. Substituting in the fast derivatives.

There is now a series of substitutions that we need to make in order to use the fast problems that have already 
been solved at the previous steps of the induction. Namely, we proceed with the following manipulations in the 
decomposition \eqref{groscalcul-1} of $\cD^m$ by using \eqref{expressionF3+jm} and \eqref{expressionF8m}:
\begin{itemize}
 \item we collect the green terms and substitute the value of $\xi_j \, F_{3+j}^{\ell_2} -\tau \, F_8^{\ell_2}$,
 \item we collect the brown terms and substitute the value of $F_{3+j}^{\ell_2}$,
 \item we substitute the value of $F_8^{\ell_2}$ in the orange term.
\end{itemize}
For possible intermediate verification of the interested reader, let us give the detailed expression of these first 
manipulations. We isolate the green, brown and orange terms in $\cD^m$, and therefore introduce the quantity:
\begin{align*}
\cD_1^m \, := \, & \, -\sum_{\ell_1+\ell_2=m+1} \p_\theta \psi^{\ell_1} \, \p_{Y_3} \big( \xi_j \, F_{3+j}^{\ell_2} -\tau \, F_8^{\ell_2} \big) 
-\sum_{\ell_1+\ell_2 =m} \p_{y_j} \psi^{\ell_1} \, \p_{Y_3} F_{3+j}^{\ell_2} -\p_t \psi^{\ell_1} \, \p_{Y_3} F_8^{\ell_2} \\
& \, +\sum_{\substack{\ell_1+\ell_2+\ell_3=m+2 \\ \ell_2,\ell_3 \ge 1}} \p_\theta \psi^{\ell_1} \, 
\p_{Y_3}^2 \big( H_3^{\ell_2} \, \xi_j \, u_j^{\ell_3} -u_3^{\ell_2} \, \xi_j \, H_j^{\ell_3} \big) \\
& \, +\sum_{\substack{\ell_1+\ell_2+\ell_3=m+1 \\ \ell_2,\ell_3 \ge 1}} \p_{y_j} \psi^{\ell_1} \, 
\p_{Y_3}^2 \big( H_3^{\ell_2} \, u_j^{\ell_3} -u_3^{\ell_2} \, H_j^{\ell_3} \big) \\
& \, +\sum_{\substack{\ell_1+\ell_2+\ell_3=m+1 \\ \ell_2,\ell_3 \ge 1}} \p_{y_j} \psi^{\ell_1} \, 
\p_\theta \p_{Y_3} \big( u_j^{\ell_2} \, \xi_{j'} \, H_{j'}^{\ell_3} -H_j^{\ell_2} \, \xi_{j'} \, u_{j'}^{\ell_3} \big) \, .
\end{align*}
After substituting, we get:
\begin{align}
\cD_1^m \, = \, & \, -{\color{red} \sum_{\ell_1+\ell_2=m} \p_t \psi^{\ell_1} \, \nabla \cdot \p_{Y_3} H^{\ell_2}} 
+{\color{red} \sum_{\ell_1+\ell_2=m+1} \p_\theta \psi^{\ell_1} \, \xi_j \, \p_t \p_{Y_3} H_j^{\ell_2}} 
+{\color{red} \sum_{\ell_1+\ell_2=m} \p_{y_j} \psi^{\ell_1} \, \p_t \p_{Y_3} H_j^{\ell_2}} \notag \\
& \, +{\color{red} \sum_{\ell_1+\ell_2=m+1} \p_\theta \psi^{\ell_1} \, \big( u_j^0 \, \xi_{j'} \, \p_{y_j} \p_{Y_3} H_{j'}^{\ell_2} 
-H_j^0 \, \xi_{j'} \, \p_{y_j} \p_{Y_3} u_{j'}^{\ell_2} 
+b \, \nabla \cdot \p_{Y_3} u^{\ell_2} -c \, \nabla \cdot \p_{Y_3} H^{\ell_2} \big)} \notag \\
& \, +{\color{red} \sum_{\ell_1+\ell_2=m} \p_{y_j} \psi^{\ell_1} \, \big( u_{j'}^0 \,\p_{y_{j'}} \p_{Y_3} H_j^{\ell_2} 
-H_{j'}^0 \,\p_{y_{j'}} \p_{Y_3} u_j^{\ell_2} 
+H_j^0 \, \nabla \cdot \p_{Y_3} u^{\ell_2} -u_j^0 \, \nabla \cdot \p_{Y_3} H^{\ell_2} \big)} \notag \\
& \, +{\color{red} \sum_{\substack{\ell_1+\ell_2+\ell_3=m+1 \\ \ell_2,\ell_3 \ge 1}} \p_\theta \psi^{\ell_1} \, 
\nabla \cdot \p_{Y_3} \big( \xi_j \, H_j^{\ell_2} \, u^{\ell_3} -\xi_j \, u_j^{\ell_2} \, H^{\ell_3} \big)} \notag \\
& \, +{\color{red} \sum_{\substack{\ell_1+\ell_2+\ell_3=m \\ \ell_2,\ell_3 \ge 1}} \p_{y_j} \psi^{\ell_1} \, 
\nabla \cdot \p_{Y_3} \big( H_j^{\ell_2 } \, u^{\ell_3} -u_j^{\ell_2} \, H^{\ell_3} \big)} \label{expressionD1m} \\
& \, +\sum_{\ell_1+\cdots+\ell_4=m+1} \dot{\chi}^{[\ell_1]} \, \psi^{\ell_2} \, \p_\theta \psi^{\ell_3} \, 
\big( c \, \p_{y_3} \p_{Y_3} H_3^{\ell_4} -b \, \p_{y_3} \p_{Y_3} u_3^{\ell_4} \big) \notag \\
& \, +\sum_{\ell_1+\cdots+\ell_4=m} \dot{\chi}^{[\ell_1]} \, \psi^{\ell_2} \, \Big( \big( \p_t +u_j^0 \, \p_{y_j} \big) \psi^{\ell_3} \, 
\p_{y_3} \p_{Y_3} H_3^{\ell_4} -H_j^0 \, \p_{y_j} \psi^{\ell_3} \, \p_{y_3} \p_{Y_3} u_3^{\ell_4} \Big) \notag \\
& \, +\sum_{\substack{\ell_1+\cdots+\ell_5=m+1 \\ \ell_4,\ell_5 \ge 1}} \dot{\chi}^{[\ell_1]} \, \psi^{\ell_2} \, \p_\theta \psi^{\ell_3} \, 
\p_{y_3} \p_{Y_3} \big( H_3^{\ell_4} \, \xi_j \, u_j^{\ell_5} -u_3^{\ell_4} \, \xi_j \, H_j^{\ell_5} \big) \notag \\
& \, +\sum_{\substack{\ell_1+\cdots+\ell_5=m \\ \ell_4,\ell_5 \ge 1}} \dot{\chi}^{[\ell_1]} \, \psi^{\ell_2} \, \p_{y_j} \psi^{\ell_3} \, 
\p_{y_3} \p_{Y_3} \big( H_3^{\ell_4} \, u_j^{\ell_5} -u_3^{\ell_4} \, H_j^{\ell_5} \big) \, .\notag
\end{align}
Some terms in the decomposition \eqref{expressionD1m} have been highlighted in red in order to exhibit 
as clearly as possible a cancellation with the decomposition for the next term $\cD_2^m$ which we are just 
going to introduce now.
\bigskip

$\bullet$ \underline{Step 3}. Substituting in the slow derivatives.

In the previous step, we have made all substitutions in the terms $\p_{Y_3} F^\ell$, and collected various other terms that cancel 
after making these substitutions. Once this is done, we now need to make substitutions in some slow derivatives. The first, and 
actually longest, task is to substitute the expressions of $F_{3+\alpha}^{\, m-1}$ and $F_8^{\, m-1}$ in the pink terms of \eqref{groscalcul-1}, 
using again \eqref{expressionF8m}, \eqref{expressionF3+jm}, \eqref{expressionF6m}. Since some cancellations are going to appear, 
it is useful to incorporate the blue terms of \eqref{groscalcul-1} with the pink. In other words, we introduce the quantity:
\begin{align*}
\cD_2^m \, := \, & \, -\p_t F_8^{\, m-1} +\p_{y_\alpha} F_{3+\alpha}^{\, m-1} 
+\sum_{\ell_1+\ell_2=m+1} \p_t \p_\theta \psi^{\ell_1} \, \xi_j \, \p_{Y_3} H_j^{\ell_2} 
+\sum_{\ell_1+\ell_2+\ell_3=m} \chi^{[\ell_1]} \, \p_t \p_\theta \psi^{\ell_2} \, \xi_j \, \p_{y_3} H_j^{\ell_3} \\
& \, +\sum_{\ell_1+\ell_2 =m+1} \p_{y_j} \p_\theta \psi^{\ell_1} \, \big( u_j^0 \, \xi_{j'} \, \p_{Y_3} H_{j'}^{\ell_2} 
-H_j^0 \, \xi_{j'} \, \p_{Y_3} u_{j'}^{\ell_2} +b \, \p_{Y_3} u_j^{\ell_2} -c \, \p_{Y_3} H_j^{\ell_2} \big) \\
& \, +\sum_{\ell_1+\ell_2+\ell_3=m} \chi^{[\ell_1]} \, \p_{y_j} \p_\theta \psi^{\ell_2} \, \Big( 
u_j^0 \, \xi_{j'} \, \p_{y_3} H_{j'}^{\ell_3} -H_j^0 \, \xi_{j'} \, \p_{y_3} u_{j'}^{\ell_3} 
+b \, \p_{y_3} u_j^{\ell_3} -c \, \p_{y_3} H_j^{\ell_3} \Big) \\
& \, +\sum_{\substack{\ell_1+\ell_2=m \\ \ell_1,\ell_2 \ge 1}} \p_{y_j} \p_{Y_3} \big( 
H_j^{\ell_1} \, u_3^{\ell_2} -u_j^{\ell_1} \, H_3^{\ell_2} \big) 
+\sum_{\substack{\ell_1+\ell_2=m \\ \ell_1,\ell_2 \ge 1}} \nabla \cdot \p_\theta \big( 
\xi_j \, u_j^{\ell_1} \, H^{\ell_2} -\xi_j \, H_j^{\ell_1} \, u^{\ell_2} \big) \\
& \, +\sum_{\substack{\ell_1+\ell_2+\ell_3=m+1 \\ \ell_2,\ell_3 \ge 1}} \p_{y_j} \p_\theta \psi^{\ell_1} \, 
\p_{Y_3} \big( u_j^{\ell_2} \, \xi_{j'} \, H_{j'}^{\ell_3} -H_j^{\ell_2} \, \xi_{j'} \, u_{j'}^{\ell_3} \big) \\
& \, +\sum_{\substack{\ell_1+\cdots+\ell_4=m \\ \ell_3,\ell_4 \ge 1}} \chi^{[\ell_1]} \, \p_{y_j} \p_\theta \psi^{\ell_2} \, 
\p_{y_3} \big( u_j^{\ell_3} \, \xi_{j'} \, H_{j'}^{\ell_4} -H_j^{\ell_3} \, \xi_{j'} \, u_{j'}^{\ell_4} \big) \, .
\end{align*}
Substituting the expressions of $F_{3+\alpha}^{\, m-1}$ and $F_8^{\, m-1}$ in the above definition of $\cD_2^m$ and rearranging, $\cD_2^m$ 
decomposes as follows:
\begin{align}
\cD_2^m \, = \, & \, {\color{red} \sum_{\ell_1+\ell_2=m} \p_t \psi^{\ell_1} \, \nabla \cdot \p_{Y_3} H^{\ell_2}} 
-{\color{red} \sum_{\ell_1+\ell_2=m+1} \p_\theta \psi^{\ell_1} \, \xi_j \, \p_t \p_{Y_3} H_j^{\ell_2}} 
-{\color{red} \sum_{\ell_1+\ell_2=m} \p_{y_j} \psi^{\ell_1} \, \p_t \p_{Y_3} H_j^{\ell_2}} \notag \\
 & \, +{\color{red} \sum_{\ell_1+\ell_2=m+1} \p_\theta \psi^{\ell_1} \, 
\big( H_j^0 \, \xi_{j'} \, \p_{y_j} \p_{Y_3} u_{j'}^{\ell_2} -u_j^0 \, \xi_{j'} \, \p_{y_j} \p_{Y_3} H_{j'}^{\ell_2} 
+c \, \nabla \cdot \p_{Y_3} H^{\ell_2} -b \, \nabla \cdot \p_{Y_3} u^{\ell_2} \big)} \notag \\
 & \, +{\color{blue} \sum_{\ell_1+\ell_2+\ell_3=m} \chi^{[\ell_1]} \, \p_\theta \psi^{\ell_2} \, 
\big( H_j^0 \, \xi_{j'} \, \p_{y_j} \p_{y_3} u_{j'}^{\ell_3} -u_j^0 \, \xi_{j'} \, \p_{y_j} \p_{y_3} H_{j'}^{\ell_3} 
+c \, \nabla \cdot \p_{y_3} H^{\ell_3} -b \, \nabla \cdot \p_{y_3} u^{\ell_3} \big)} \notag \\
 & \, +{\color{red} \sum_{\ell_1+\ell_2=m} \p_{y_j} \psi^{\ell_1} \, 
\big( H_{j'}^0 \,\p_{y_{j'}} \p_{Y_3} u_j^{\ell_2} -u_{j'}^0 \,\p_{y_{j'}} \p_{Y_3} H_j^{\ell_2} 
+u_j^0 \, \nabla \cdot \p_{Y_3} H^{\ell_2} -H_j^0 \, \nabla \cdot \p_{Y_3} u^{\ell_2} \big)} \notag \\
 & \, +{\color{ForestGreen} \sum_{\ell_1+\ell_2+\ell_3=m-1} \chi^{[\ell_1]} \, \p_{y_j} \psi^{\ell_2} \, 
\big( H_{j'}^0 \,\p_{y_{j'}} \p_{y_3} u_j^{\ell_3} -u_{j'}^0 \,\p_{y_{j'}} \p_{y_3} H_j^{\ell_3} 
+u_j^0 \, \nabla \cdot \p_{y_3} H^{\ell_3} -H_j^0 \, \nabla \cdot \p_{y_3} u^{\ell_3} \big)} \notag \\
 & \, +{\color{orange} \sum_{\ell_1+\ell_2+\ell_3=m-1} \chi^{[\ell_1]} \, \p_t \psi^{\ell_2} \, \nabla \cdot \p_{y_3} H^{\ell_3}} 
+\sum_{\ell_1+\ell_2+\ell_3=m-1} \p_{y_\alpha} \chi^{[\ell_1]} \, \p_t \psi^{\ell_2} \, \p_{y_3} H_\alpha^{\ell_3} \notag \\
 & \, +\sum_{\ell_1+\ell_2+\ell_3=m} \p_{y_j} \chi^{[\ell_1]} \, \p_\theta \psi^{\ell_2} \, \big( H_j^0 \, \xi_{j'} \, \p_{y_3} u_{j'}^{\ell_3} 
-u_j^0 \, \xi_{j'} \, \p_{y_3} H_{j'}^{\ell_3} +c \, \p_{y_3} H_j^{\ell_3} -b \, \p_{y_3} u_j^{\ell_3} \big) \notag \\
 & \, +\sum_{\ell_1+\ell_2+\ell_3=m-1}\p_{y_{j'}} \chi^{[\ell_1]} \, \p_{y_j} \psi^{\ell_2} \, \big( H_{j'}^0 \, \p_{y_3} u_j^{\ell_3} 
-u_{j'}^0 \, \p_{y_3} H_j^{\ell_3} +u_j^0 \, \p_{y_3} H_{j'}^{\ell_3} -H_j^0 \, \p_{y_3} u_{j'}^{\ell_3} \big) \notag \\
 & \, +\sum_{\ell_1+\ell_2+\ell_3=m} \p_{y_3} \chi^{[\ell_1]} \, \p_\theta \psi^{\ell_2} \, \big( 
c \, \p_{y_3} H_3^{\ell_3} -b \, \p_{y_3} u_3^{\ell_3} \big) \notag \\
 & \, +\sum_{\ell_1+\ell_2+\ell_3=m-1} \p_{y_3} \chi^{[\ell_1]} \, \p_{y_j} \psi^{\ell_2} \, 
\big( u_j^0 \, \p_{y_3} H_3^{\ell_3} -H_j^0 \, \p_{y_3} u_3^{\ell_3} \big) \notag \\
 & \, +{\color{red} \sum_{\substack{\ell_1+\ell_2+\ell_3=m+1 \\ \ell_2,\ell_3 \ge 1}} \p_\theta \psi^{\ell_1} \, 
\nabla \cdot \p_{Y_3} \big( \xi_j \, u_j^{\ell_2} \, H^{\ell_3} -\xi_j \, H_j^{\ell_2} \, u^{\ell_3} \big)} \notag \\
 & \, +{\color{red} \sum_{\substack{\ell_1+\ell_2+\ell_3=m \\ \ell_2,\ell_3 \ge 1}} \p_{y_j} \psi^{\ell_1} \, 
\nabla \cdot \p_{Y_3} \big( u_j^{\ell_2} \, H^{\ell_3} -H_j^{\ell_2} \, u^{\ell_3} \big)} \label{expressionD2m} \\
 & \, +{\color{blue} \sum_{\substack{\ell_1+\cdots+\ell_4=m \\ \ell_3,\ell_4 \ge 1}} \chi^{[\ell_1]} \, \p_\theta \psi^{\ell_2} \, 
\nabla \cdot \p_{y_3} \big( \xi_j \, u_j^{\ell_3} \, H^{\ell_4} -\xi_j \, H_j^{\ell_3} \, u^{\ell_4} \big) 
-\sum_{\ell_1+\ell_2+\ell_3=m} \chi^{[\ell_1]} \, \p_\theta \psi^{\ell_2} \, \xi_j \, \p_t \p_{y_3} H_j^{\ell_3}} \notag \\
 & \, +{\color{ForestGreen} \sum_{\substack{\ell_1+\cdots+\ell_4=m-1 \\ \ell_2,\ell_3 \ge 1}} \chi^{[\ell_1]} \, \p_{y_j} \psi^{\ell_2} \, 
\nabla \cdot \p_{y_3} \big( u_j^{\ell_3} \, H^{\ell_4} -H_j^{\ell_3} \, u^{\ell_4} \big) 
-\sum_{\ell_1+\ell_2+\ell_3=m-1} \chi^{[\ell_1]} \, \p_{y_j} \psi^{\ell_2} \, \p_t \p_{y_3} H_j^{\ell_3}} \notag \\
 & \, +\sum_{\substack{\ell_1+\cdots+\ell_4=m \\ \ell_3,\ell_4 \ge 1}} \p_{y_\alpha} \chi^{[\ell_1]} \, \p_\theta \psi^{\ell_2} \, 
 \p_{y_3} \big( \xi_j \, u_j^{\ell_3} \, H_\alpha^{\ell_4} -\xi_j \, H_j^{\ell_3} \, u_\alpha^{\ell_4} \big) \notag \\
 & \, +\sum_{\substack{\ell_1+\cdots+\ell_4=m-1 \\ \ell_2,\ell_3 \ge 1}} \p_{y_\alpha} \chi^{[\ell_1]} \, \p_{y_j} \psi^{\ell_2} \, 
 \p_{y_3} \big( u_j^{\ell_3} \, H_\alpha^{\ell_4} -H_j^{\ell_3} \, u_\alpha^{\ell_4} \big) \notag \\
 & \, +\p_{y_j} \, \Big\{ \sum_{\ell_1+\ell_2+\ell_3=m-1} \dot{\chi}^{[\ell_1]} \, \psi^{\ell_2} \, 
\big( H_j^0 \, \p_{y_3} u_3^{\ell_3} -u_j^0 \, \p_{y_3} H_3^{\ell_3} \big) \Big\} \notag \\
 & \, +\p_{y_j} \, \Big\{ \sum_{\substack{\ell_1+\cdots+\ell_4=m-1 \\ \ell_3,\ell_4 \ge 1}} \dot{\chi}^{[\ell_1]} \, \psi^{\ell_2} \, 
\p_{y_3} \big( H_j^{\ell_3} \, u_3^{\ell_4} -u_j^{\ell_3} \, H_3^{\ell_4} \big) \Big\} 
-\sum_{\ell_1+\ell_2+\ell_3=m} \p_t \chi^{[\ell_1]} \, \p_\theta \psi^{\ell_2} \, \xi_j \, \p_{y_3} H_j^{\ell_3} \notag \\
 & \, -\sum_{\ell_1+\ell_2+\ell_3=m-1} \p_t \chi^{[\ell_1]} \, \p_{y_j} \psi^{\ell_2} \, \p_{y_3} H_j^{\ell_3} 
-\p_t \, \Big\{ \sum_{\ell_1+\ell_2+\ell_3=m-1} \dot{\chi}^{[\ell_1]} \, \psi^{\ell_2} \, \p_{y_3} H_3^{\ell_3} \Big\} \, .\notag
\end{align}
Those terms in the decomposition \eqref{expressionD2m} that have been highlighted in red cancel exactly with their opposite 
which has also been highlighted in red in the decomposition \eqref{expressionD1m}. Other colors used to highlight some 
of the terms in \eqref{expressionD2m} will be explained in the following step of the proof.
\bigskip

$\bullet$ \underline{Step 4}. Simplifying and rearranging.

Up to now, we have decomposed the quantity $\cD^m$ into the sum of three quantities, namely we have written:
$$
\cD^m \, = \, \cD_1^m +\cD_2^m +\cD_3^m \, ,
$$
where $\cD_1^m$ is given by \eqref{expressionD1m}, $\cD_2^m$ is given by \eqref{expressionD2m}, and $\cD_3^m$ 
is the sum of all black terms in \eqref{groscalcul-1}, that is:
\begin{align}
\cD_3^m \, := \, & \, -{\color{blue} \sum_{\ell_1+\ell_2+\ell_3=m} \chi^{[\ell_1]} \, \p_\theta \psi^{\ell_2} \, 
\p_{y_3} \big( \xi_j \, F_{3+j}^{\ell_3} -\tau \, F_8^{\ell_3} \big)} 
+{\color{orange} \sum_{\ell_1+\ell_2+\ell_3=m-1} \chi^{[\ell_1]} \, \p_t \psi^{\ell_2} \, \p_{y_3} F_8^{\ell_3}} \notag \\
& \, -{\color{ForestGreen} \sum_{\ell_1+\ell_2+\ell_3=m-1} \chi^{[\ell_1]} \, \p_{y_j} \psi^{\ell_2} \, \p_{y_3} F_{3+j}^{\ell_3}} 
+{\color{blue} \sum_{\substack{\ell_1+\cdots+\ell_4=m+1 \\ \ell_3,\ell_4 \ge 1}} \chi^{[\ell_1]} \, \p_\theta \psi^{\ell_2} \, 
\p_{y_3} \p_{Y_3} \big( H_3^{\ell_3} \, \xi_j \, u_j^{\ell_4} -u_3^{\ell_3} \, \xi_j \, H_j^{\ell_4} \big)} \notag \\
& \, +{\color{ForestGreen} \sum_{\substack{\ell_1+\cdots+\ell_4=m \\ \ell_3,\ell_4 \ge 1}} \chi^{[\ell_1]} \, \p_{y_j} \psi^{\ell_2} \, 
\p_{y_3} \p_{Y_3} \big( H_3^{\ell_3} \, u_j^{\ell_4} -u_3^{\ell_3} \, H_j^{\ell_4} \big)} \notag \\
& \, +{\color{ForestGreen} \sum_{\substack{\ell_1+\cdots+\ell_4=m \\ \ell_3,\ell_4 \ge 1}} \chi^{[\ell_1]} \, \p_{y_j} \psi^{\ell_2} \, 
\p_{y_3} \p_\theta \big( u_j^{\ell_3} \, \xi_{j'} \, H_{j'}^{\ell_4} -H_j^{\ell_3} \, \xi_{j'} \, u_{j'}^{\ell_4} \big)} 
+\sum_{\ell_1+\ell_2+\ell_3=m} \p_\theta \chi^{[\ell_1]} \, \p_t \psi^{\ell_2} \, \xi_j \, \p_{y_3} H_j^{\ell_3} \label{groscalcul-2} \\
& \, +\sum_{\ell_1+\ell_2+\ell_3=m} \p_\theta \chi^{[\ell_1]} \, \p_{y_j} \psi^{\ell_2} \, \big( 
u_j^0 \, \xi_{j'} \, \p_{y_3} H_{j'}^{\ell_3} -H_j^0 \, \xi_{j'} \, \p_{y_3} u_{j'}^{\ell_3} 
+b \, \p_{y_3} u_j^{\ell_3} -c \, \p_{y_3} H_j^{\ell_3} \big) \notag \\
& \, +\sum_{\substack{\ell_1+\cdots+\ell_4=m \\ \ell_3,\ell_4 \ge 1}} \p_\theta \chi^{[\ell_1]} \, \p_{y_j} \psi^{\ell_2} \, 
\p_{y_3} \big( u_j^{\ell_3} \, \xi_{j'} \, H_{j'}^{\ell_4} -H_j^{\ell_3} \, \xi_{j'} \, u_{j'}^{\ell_4} \big) \notag \\
& \, +\p_\theta \Big\{ \sum_{\ell_1+\ell_2+\ell_3=m} \dot{\chi}^{[\ell_1]} \, \psi^{\ell_2} \, 
\big( b \, \p_{y_3} u_3^{\ell_3} -c \, \p_{y_3} H_3^{\ell_3} \big) \Big\} \notag \\
& \, +\p_\theta \Big\{ \sum_{\substack{\ell_1+\cdots+\ell_4=m \\ \ell_3,\ell_4 \ge 1}} \dot{\chi}^{[\ell_1]} \, \psi^{\ell_2} \, 
\p_{y_3} \big( \xi_j \, H_j^{\ell_3} \, u_3^{\ell_4} -\xi_j \, u_j^{\ell_3} \, H_3^{\ell_4} \big) \Big\} \, .\notag 
\end{align}

It is now time to regroup some of the terms from $\cD_2^m$ and $\cD_3^m$ together in order to further substitute and simplify 
$\cD^m$. Observe that all terms in the expressions \eqref{expressionD1m}, \eqref{expressionD2m} and \eqref{groscalcul-2} involve 
the functions $\chi^{[\ell]}$ or $\dot{\chi}^{[\ell]}$, whose symmetry properties will be crucial for the final argument of the proof. But 
before exhibiting those symmetry properties, let us rearrange the expression of $\cD^m$. We are now going to decompose the 
quantity $\cD^m$ under the form:
$$
\cD^m \, = \, \cE_1^m +\cE_2^m +\cE_3^m +\cE_4^m \, ,
$$
where:
\begin{itemize}
 \item $\cE_1^m$ corresponds to the collection of the blue terms in $\cD_2^m$ and $\cD_3^m$,
 \item $\cE_2^m$ corresponds to the collection of the green terms in $\cD_2^m$ and $\cD_3^m$,
 \item $\cE_3^m$ corresponds to the collection of the orange terms in $\cD_2^m$ and $\cD_3^m$,
 \item $\cE_4^m$ corresponds to the collection of all remaining (black) terms in $\cD_1^m$, $\cD_2^m$ 
and $\cD_3^m$.
\end{itemize}
In other words, we set:
\begin{align*}
\cE_1^m \, := \, & \, -\sum_{\ell_1+\ell_2+\ell_3=m} \chi^{[\ell_1]} \, \p_\theta \psi^{\ell_2} \, 
\p_{y_3} \big( \xi_j \, F_{3+j}^{\ell_3} -\tau \, F_8^{\ell_3} \big) 
-\sum_{\ell_1+\ell_2+\ell_3=m} \chi^{[\ell_1]} \, \p_\theta \psi^{\ell_2} \, \xi_j \, \p_t \p_{y_3} H_j^{\ell_3} \\
 & \, \sum_{\ell_1+\ell_2+\ell_3=m} \chi^{[\ell_1]} \, \p_\theta \psi^{\ell_2} \, 
\big( H_j^0 \, \xi_{j'} \, \p_{y_j} \p_{y_3} u_{j'}^{\ell_3} -u_j^0 \, \xi_{j'} \, \p_{y_j} \p_{y_3} H_{j'}^{\ell_3} 
+c \, \nabla \cdot \p_{y_3} H^{\ell_3} -b \, \nabla \cdot \p_{y_3} u^{\ell_3} \big) \\
 & \, +\sum_{\substack{\ell_1+\cdots+\ell_4=m+1 \\ \ell_3,\ell_4 \ge 1}} \chi^{[\ell_1]} \, \p_\theta \psi^{\ell_2} \, 
\p_{y_3} \p_{Y_3} \big( H_3^{\ell_3} \, \xi_j \, u_j^{\ell_4} -u_3^{\ell_3} \, \xi_j \, H_j^{\ell_4} \big) \\
 & \, +\sum_{\substack{\ell_1+\cdots+\ell_4=m \\ \ell_3,\ell_4 \ge 1}} \chi^{[\ell_1]} \, \p_\theta \psi^{\ell_2} \, 
\nabla \cdot \p_{y_3} \big( \xi_j \, u_j^{\ell_3} \, H^{\ell_4} -\xi_j \, H_j^{\ell_3} \, u^{\ell_4} \big) \, ,
\end{align*}
\begin{align*}
\cE_2^m \, := \, & \, -\sum_{\ell_1+\ell_2+\ell_3=m-1} \chi^{[\ell_1]} \, \p_{y_j} \psi^{\ell_2} \, \p_{y_3} F_{3+j}^{\ell_3} 
-\sum_{\ell_1+\ell_2+\ell_3=m-1} \chi^{[\ell_1]} \, \p_{y_j} \psi^{\ell_2} \, \p_t \p_{y_3} H_j^{\ell_3} \\
 & \, +\sum_{\ell_1+\ell_2+\ell_3=m-1} \chi^{[\ell_1]} \, \p_{y_j} \psi^{\ell_2} \, 
\big( H_{j'}^0 \,\p_{y_{j'}} \p_{y_3} u_j^{\ell_3} -u_{j'}^0 \,\p_{y_{j'}} \p_{y_3} H_j^{\ell_3} 
+u_j^0 \, \nabla \cdot \p_{y_3} H^{\ell_3} -H_j^0 \, \nabla \cdot \p_{y_3} u^{\ell_3} \big) \\
 & \, +\sum_{\substack{\ell_1+\cdots+\ell_4=m \\ \ell_3,\ell_4 \ge 1}} \chi^{[\ell_1]} \, \p_{y_j} \psi^{\ell_2} \, 
\p_{y_3} \p_{Y_3} \big( H_3^{\ell_3} \, u_j^{\ell_4} -u_3^{\ell_3} \, H_j^{\ell_4} \big) \\
 & \, +\sum_{\substack{\ell_1+\cdots+\ell_4=m \\ \ell_3,\ell_4 \ge 1}} \chi^{[\ell_1]} \, \p_{y_j} \psi^{\ell_2} \, 
\p_{y_3} \p_\theta \big( u_j^{\ell_3} \, \xi_{j'} \, H_{j'}^{\ell_4} -H_j^{\ell_3} \, \xi_{j'} \, u_{j'}^{\ell_4} \big) \\
 & \, +\sum_{\substack{\ell_1+\cdots+\ell_4=m-1 \\ \ell_2,\ell_3 \ge 1}} \chi^{[\ell_1]} \, \p_{y_j} \psi^{\ell_2} \, 
\nabla \cdot \p_{y_3} \big( u_j^{\ell_3} \, H^{\ell_4} -H_j^{\ell_3} \, u^{\ell_4} \big) \, ,
\end{align*}
\begin{equation*}
\cE_3^m \, := \, 
\sum_{\ell_1+\ell_2+\ell_3=m-1} \chi^{[\ell_1]} \, \p_t \psi^{\ell_2} \, \p_{y_3} \big( F_8^{\ell_3} +\nabla \cdot H^{\ell_3} \big) \, ,
\end{equation*}
and according to the expressions \eqref{expressionD1m}, \eqref{expressionD2m} and \eqref{groscalcul-2}, the final 
term $\cE_4^m$ reads:
\begin{align}
\cE_4^m \, := \, & \, {\color{blue} \sum_{\ell_1+\ell_2+\ell_3=m-1} 
\big( \p_{y_j} \chi^{[\ell_1]} \, \p_t \psi^{\ell_2} -\p_t \chi^{[\ell_1]} \, \p_{y_j} \psi^{\ell_2} \big) \, \p_{y_3} H_j^{\ell_3}} 
+\sum_{\ell_1+\ell_2+\ell_3=m-1} \p_{y_3} \chi^{[\ell_1]} \, \p_t \psi^{\ell_2} \, \p_{y_3} H_3^{\ell_3} \notag \\
 & \, +{\color{blue} \sum_{\ell_1+\ell_2+\ell_3=m} \big( \p_\theta \chi^{[\ell_1]} \, \p_t \psi^{\ell_2} 
 -\p_t \chi^{[\ell_1]} \, \p_\theta \psi^{\ell_2} \big) \, \xi_j \, \p_{y_3} H_j^{\ell_3}} \notag \\
 & \, +{\color{blue} \sum_{\ell_1+\ell_2+\ell_3=m} \big( \p_{y_j} \chi^{[\ell_1]} \, \p_\theta \psi^{\ell_2} 
-\p_\theta \chi^{[\ell_1]} \, \p_{y_j} \psi^{\ell_2} \big) \, \big( H_j^0 \, \xi_{j'} \, \p_{y_3} u_{j'}^{\ell_3} 
-u_j^0 \, \xi_{j'} \, \p_{y_3} H_{j'}^{\ell_3} +c \, \p_{y_3} H_j^{\ell_3} -b \, \p_{y_3} u_j^{\ell_3} \big)} \notag \\
 & \, +\sum_{\ell_1+\ell_2+\ell_3=m} \p_{y_3} \chi^{[\ell_1]} \, \p_\theta \psi^{\ell_2} \, \big( 
c \, \p_{y_3} H_3^{\ell_3} -b \, \p_{y_3} u_3^{\ell_3} \big) \notag \\
 & \, +{\color{blue} \sum_{\ell_1+\ell_2+\ell_3=m-1}\p_{y_{j'}} \chi^{[\ell_1]} \, \p_{y_j} \psi^{\ell_2} \, \big( 
H_{j'}^0 \, \p_{y_3} u_j^{\ell_3} -u_{j'}^0 \, \p_{y_3} H_j^{\ell_3} 
+u_j^0 \, \p_{y_3} H_{j'}^{\ell_3} -H_j^0 \, \p_{y_3} u_{j'}^{\ell_3} \big)} \notag \\
 & \, +\sum_{\ell_1+\ell_2+\ell_3=m-1} \p_{y_3} \chi^{[\ell_1]} \, \p_{y_j} \psi^{\ell_2} \, \big( 
u_j^0 \, \p_{y_3} H_3^{\ell_3} -H_j^0 \, \p_{y_3} u_3^{\ell_3} \big) \notag \\
 & \, +{\color{blue} \sum_{\substack{\ell_1+\cdots+\ell_4=m \\ \ell_3,\ell_4 \ge 1}} 
\big( \p_{y_j} \chi^{[\ell_1]} \, \p_\theta \psi^{\ell_2} -\p_\theta \chi^{[\ell_1]} \, \p_{y_j} \psi^{\ell_2} \big) \, 
\p_{y_3} \big( H_j^{\ell_3} \, \xi_{j'} \, u_{j'}^{\ell_4} -u_j^{\ell_3} \, \xi_{j'} \, H_{j'}^{\ell_4} \big)} \notag \\
 & \, +\sum_{\substack{\ell_1+\cdots+\ell_4=m \\ \ell_3,\ell_4 \ge 1}} \p_{y_3} \chi^{[\ell_1]} \, \p_\theta \psi^{\ell_2} \, 
\p_{y_3} \big( \xi_j \, u_j^{\ell_3} \, H_3^{\ell_4} -\xi_j \, H_j^{\ell_3} \, u_3^{\ell_4} \big) \notag \\
 & \, +{\color{blue} \sum_{\substack{\ell_1+\cdots+\ell_4=m-1 \\ \ell_3,\ell_4 \ge 1}}\p_{y_{j'}} \chi^{[\ell_1]} \, \p_{y_j} \psi^{\ell_2} \, 
\p_{y_3} \big( u_j^{\ell_3} \, H_{j'}^{\ell_4} -H_j^{\ell_3} \, u_{j'}^{\ell_4} \big)} \label{decompositionE4m} \\
 & \, +\sum_{\substack{\ell_1+\cdots+\ell_4=m-1 \\ \ell_3,\ell_4 \ge 1}} \p_{y_3} \chi^{[\ell_1]} \, \p_{y_j} \psi^{\ell_2} \, 
\p_{y_3} \big( u_j^{\ell_3} \, H_3^{\ell_4} -H_j^{\ell_3} \, u_3^{\ell_4} \big) \notag \\
 & \, +\p_{y_j} \, \Big\{ \sum_{\ell_1+\ell_2+\ell_3=m-1} \dot{\chi}^{[\ell_1]} \, \psi^{\ell_2} \, 
\big( H_j^0 \, \p_{y_3} u_3^{\ell_3} -u_j^0 \, \p_{y_3} H_3^{\ell_3} \big) \Big\} \notag \\
 & \, +\p_{y_j} \, \Big\{ \sum_{\substack{\ell_1+\cdots+\ell_4=m-1 \\ \ell_3,\ell_4 \ge 1}} \dot{\chi}^{[\ell_1]} \, \psi^{\ell_2} \, 
\p_{y_3} \big( H_j^{\ell_3} \, u_3^{\ell_4} -u_j^{\ell_3} \, H_3^{\ell_4} \big) \Big\} 
-\p_t \Big\{ \sum_{\ell_1+\ell_2+\ell_3=m-1} \dot{\chi}^{[\ell_1]} \, \psi^{\ell_2} \, \p_{y_3} H_3^{\ell_3} \Big\} \notag \\
 & \, +\p_\theta \Big\{ \sum_{\ell_1+\ell_2+\ell_3=m} \dot{\chi}^{[\ell_1]} \, \psi^{\ell_2} \, 
\big( b \, \p_{y_3} u_3^{\ell_3} -c \, \p_{y_3} H_3^{\ell_3} \big) 
+\sum_{\substack{\ell_1+\cdots+\ell_4=m \\ \ell_3,\ell_4 \ge 1}} \dot{\chi}^{[\ell_1]} \, \psi^{\ell_2} \, 
\p_{y_3} \big( \xi_j \, H_j^{\ell_3} \, u_3^{\ell_4} -\xi_j \, u_j^{\ell_3} \, H_3^{\ell_4} \big) \Big\} \notag \\
 & \, +\sum_{\ell_1+\cdots+\ell_4=m+1} \dot{\chi}^{[\ell_1]} \, \psi^{\ell_2} \, \p_\theta \psi^{\ell_3} \, 
\big( c \, \p_{y_3} \p_{Y_3} H_3^{\ell_4} -b \, \p_{y_3} \p_{Y_3} u_3^{\ell_4} \big) \notag \\
 & \, +\sum_{\ell_1+\cdots+\ell_4=m} \dot{\chi}^{[\ell_1]} \, \psi^{\ell_2} \, \Big( \big( \p_t +u_j^0 \, \p_{y_j} \big) \psi^{\ell_3} \, 
\p_{y_3} \p_{Y_3} H_3^{\ell_4} -H_j^0 \, \p_{y_j} \psi^{\ell_3} \, \p_{y_3} \p_{Y_3} u_3^{\ell_4} \Big) \notag \\
 & \, +\sum_{\substack{\ell_1+\cdots+\ell_5=m+1 \\ \ell_4,\ell_5 \ge 1}} \dot{\chi}^{[\ell_1]} \, \psi^{\ell_2} \, \p_\theta \psi^{\ell_3} \, 
\p_{y_3} \p_{Y_3} \big( H_3^{\ell_4} \, \xi_j \, u_j^{\ell_5} -u_3^{\ell_4} \, \xi_j \, H_j^{\ell_5} \big) \notag \\
 & \, +\sum_{\substack{\ell_1+\cdots+\ell_5=m \\ \ell_4,\ell_5 \ge 1}} \dot{\chi}^{[\ell_1]} \, \psi^{\ell_2} \, \p_{y_j} \psi^{\ell_3} \, 
\p_{y_3} \p_{Y_3} \big( H_3^{\ell_4} \, u_j^{\ell_5} -u_3^{\ell_4} \, H_j^{\ell_5} \big) \, .\notag
\end{align}
Some terms in the expression of $\cE_4^m$ have been highlighted in blue in order to exhibit some symmetry 
properties. These terms will be dealt with later on thanks to some specific structure of the functions $\chi^{[\ell]}$.
\bigskip

$\bullet$ \underline{Step 5}. Substituting in the slow derivatives, and (partially) simplifying.

We use again the expressions \eqref{expressionF3+jm}, \eqref{expressionF8m} of $F_{3+j}^\ell$ and $F_8^\ell$ 
to simplify the expressions of $\cE_1^m$ and $\cE_2^m$. Going straight to the final result, we end up with:
\begin{align*}
\cE_1^m \, = \, & \, -{\color{ForestGreen} \sum_{\ell_1+\cdots+\ell_4=m+1} 
\chi^{[\ell_1]} \, \p_t \psi^{\ell_2} \, \p_\theta \psi^{\ell_3} \, \xi_j \, \p_{y_3} \p_{Y_3} H_j^{\ell_4}} \\
 & \, +{\color{VioletRed} \sum_{\ell_1+\cdots+\ell_4=m+1} \chi^{[\ell_1]} \, \p_{y_j} \psi^{\ell_2} \, \p_\theta \psi^{\ell_3} \, \big( 
H_j^0 \, \xi_{j'} \, \p_{y_3} \p_{Y_3} u_{j'}^{\ell_4} -u_j^0 \, \xi_{j'} \, \p_{y_3} \p_{Y_3} H_{j'}^{\ell_4} 
+c \, \p_{y_3} \p_{Y_3} H_j^{\ell_4} -b \, \p_{y_3} \p_{Y_3} u_j^{\ell_4} \big)} \\
 & \, -{\color{ForestGreen} \sum_{\ell_1+\cdots+\ell_5=m} \chi^{[\ell_1]} \, \p_t \psi^{\ell_2} \, \p_\theta \psi^{\ell_3} \, 
\p_{y_3} \Big( \chi^{[\ell_4]} \, \p_{y_3} \big( \xi_j \, H_j^{\ell_5} \big) \Big)} \\
 & \, +{\color{VioletRed} \sum_{\ell_1+\cdots+\ell_5=m} \chi^{[\ell_1]} \, \p_{y_j} \psi^{\ell_2} \, \p_\theta \psi^{\ell_3} \, \p_{y_3} \Big( 
\chi^{[\ell_4]} \, \big( H_j^0 \, \xi_{j'} \, \p_{y_3} u_{j'}^{\ell_5} -u_j^0 \, \xi_{j'} \, \p_{y_3} H_{j'}^{\ell_5} 
+c \, \p_{y_3} H_j^{\ell_5} -b \, \p_{y_3} u_j^{\ell_5} \big) \Big)} \\
 & \, +{\color{VioletRed} \sum_{\substack{\ell_1+\cdots+\ell_5=m+1 \\ \ell_4,\ell_5 \ge 1}} \chi^{[\ell_1]} \, \p_{y_j} \psi^{\ell_2} \, 
\p_\theta \psi^{\ell_3} \, \p_{y_3} \p_{Y_3} \big( H_j^{\ell_4} \, \xi_{j'} \, u_{j'}^{\ell_5} -u_j^{\ell_4} \, \xi_{j'} \, H_{j'}^{\ell_5} \big)} \\
 & \, +{\color{VioletRed} \sum_{\substack{\ell_1+\cdots+\ell_6=m \\ \ell_5,\ell_6 \ge 1}} \chi^{[\ell_1]} \, \p_{y_j} \psi^{\ell_2} \, 
\p_\theta \psi^{\ell_3} \, \p_{y_3} \Big( \chi^{[\ell_4]} \, 
\p_{y_3} \big( H_j^{\ell_5} \, \xi_{j'} \, u_{j'}^{\ell_6} -u_j^{\ell_5} \, \xi_{j'} \, H_{j'}^{\ell_6} \big) \Big)} \\
 & \, +\sum_{\ell_1+\cdots+\ell_5=m} \chi^{[\ell_1]} \, \psi^{\ell_2} \, \p_\theta \psi^{\ell_3} \, 
\p_{y_3} \Big( \dot{\chi}^{[\ell_4]} \, \big( c \, \p_{y_3} H_3^{\ell_5} -b \, \p_{y_3} u_3^{\ell_5} \big) \Big) \\
 & \, +\sum_{\substack{\ell_1+\cdots+\ell_6=m \\ \ell_5,\ell_6 \ge 1}} \chi^{[\ell_1]} \, \psi^{\ell_2} \, \p_\theta \psi^{\ell_3} \, 
\p_{y_3} \Big( \dot{\chi}^{[\ell_4]} \, \p_{y_3} \big( H_3^{\ell_5} \, \xi_j \, u_j^{\ell_6} -u_3^{\ell_5} \, \xi_j \, H_j^{\ell_6} \big) \Big) \, ,
\end{align*}
and
\begin{align*}
\cE_2^m \, = \, & \, -{\color{ForestGreen} \sum_{\ell_1+\cdots+\ell_4=m} \chi^{[\ell_1]} \, \p_t \psi^{\ell_2} \, \p_{y_j} \psi^{\ell_3} \, 
\p_{y_3} \p_{Y_3} H_j^{\ell_4}} \\
 & \, -{\color{VioletRed} \sum_{\ell_1+\cdots+\ell_4=m+1} \chi^{[\ell_1]} \, \p_{y_j} \psi^{\ell_2} \, \p_\theta \psi^{\ell_3} \, 
\big( H_j^0 \, \xi_{j'} \, \p_{y_3} \p_{Y_3} u_{j'}^{\ell_4} -u_j^0 \, \xi_{j'} \, \p_{y_3} \p_{Y_3} H_{j'}^{\ell_4} 
+c \, \p_{y_3} \p_{Y_3} H_j^{\ell_4} -b \, \p_{y_3} \p_{Y_3} u_j^{\ell_4} \big)} \\
 & \, -{\color{ForestGreen} \sum_{\ell_1+\cdots+\ell_5=m-1} \chi^{[\ell_1]} \, \p_t \psi^{\ell_2} \, \p_{y_j} \psi^{\ell_3} \, 
\p_{y_3} \big( \chi^{[\ell_4]} \, \p_{y_3} H_j^{\ell_5} \big)} \\
 & \, -{\color{VioletRed} \sum_{\ell_1+\cdots+\ell_5=m} \chi^{[\ell_1]} \, \p_{y_j} \psi^{\ell_2} \, \p_\theta \psi^{\ell_3} \, 
\p_{y_3} \Big( \chi^{[\ell_4]} \, \big( H_j^0 \, \xi_{j'} \, \p_{y_3} u_{j'}^{\ell_5} -u_j^0 \, \xi_{j'} \, \p_{y_3} H_{j'}^{\ell_5} 
+c \, \p_{y_3} H_j^{\ell_5} -b \, \p_{y_3} u_j^{\ell_5} \big) \Big)} \\
 & \, -{\color{VioletRed} \sum_{\substack{\ell_1+\cdots+\ell_5=m+1 \\ \ell_4,\ell_5 \ge 1}} \chi^{[\ell_1]} \, \p_{y_j} \psi^{\ell_2} \, 
\p_\theta \psi^{\ell_3} \, \p_{y_3} \p_{Y_3} \big( H_j^{\ell_4} \, \xi_{j'} \, u_{j'}^{\ell_5} -u_j^{\ell_4} \, \xi_{j'} \, H_{j'}^{\ell_5} \big)} \\
 & \, -{\color{VioletRed} \sum_{\substack{\ell_1+\cdots+\ell_6=m \\ \ell_5,\ell_6 \ge 1}} \chi^{[\ell_1]} \, \p_{y_j} \psi^{\ell_2} \, 
\p_\theta \psi^{\ell_3} \, \p_{y_3} \Big( \chi^{[\ell_4]} \, 
\p_{y_3} \big( H_j^{\ell_5} \, \xi_{j'} \, u_{j'}^{\ell_6} -u_j^{\ell_5} \, \xi_{j'} \, H_{j'}^{\ell_6} \big) \Big)} \\
 & \, +\sum_{\ell_1+\cdots+\ell_5=m-1} \chi^{[\ell_1]} \, \psi^{\ell_2} \, \p_{y_j} \psi^{\ell_3} \, 
\p_{y_3} \Big( \dot{\chi}^{[\ell_4]} \, \big( u_j^0 \, \p_{y_3} H_3^{\ell_5} -H_j^0 \, \p_{y_3} u_3^{\ell_5} \big) \Big) \\
 & \, +\sum_{\substack{\ell_1+\cdots+\ell_6=m-1 \\ \ell_5,\ell_6 \ge 1}} \chi^{[\ell_1]} \, \psi^{\ell_2} \, \p_{y_j} \psi^{\ell_3} \, 
\p_{y_3} \Big( \dot{\chi}^{[\ell_4]} \, \p_{y_3} \big( H_3^{\ell_5} \, u_j^{\ell_6} -u_3^{\ell_5} \, H_j^{\ell_6} \big) \Big) \, .
\end{align*}
Let us observe that in the above expressions of $\cE_1^m$ and $\cE_2^m$, the pink terms in the decomposition of $\cE_1^m$ cancel 
exactly with the pink terms in the decomposition of $\cE_2^m$.

We also use the expression of $F_8^\ell$ in $\cE_3^m$ and get:
\begin{align*}
\cE_3^m \, = \, & \, {\color{ForestGreen} \sum_{\ell_1+\cdots+\ell_4=m+1} 
\chi^{[\ell_1]} \, \p_t \psi^{\ell_2} \, \p_\theta \psi^{\ell_3} \, \xi_j \, \p_{y_3} \p_{Y_3} H_j^{\ell_4}} 
+{\color{ForestGreen} \sum_{\ell_1+\cdots+\ell_5=m} \chi^{[\ell_1]} \, \p_t \psi^{\ell_2} \, \p_\theta \psi^{\ell_3} \, 
\p_{y_3} \Big( \chi^{[\ell_4]} \, \p_{y_3} \big( \xi_j \, H_j^{\ell_5} \big) \Big)} \\
 & \, +{\color{ForestGreen} \sum_{\ell_1+\cdots+\ell_4=m} \chi^{[\ell_1]} \, \p_t \psi^{\ell_2} \, \p_{y_j} \psi^{\ell_3} \, 
\p_{y_3} \p_{Y_3} H_j^{\ell_4}} 
+{\color{ForestGreen} \sum_{\ell_1+\cdots+\ell_5=m-1} \chi^{[\ell_1]} \, \p_t \psi^{\ell_2} \, \p_{y_j} \psi^{\ell_3} \, 
\p_{y_3} \big( \chi^{[\ell_4]} \, \p_{y_3} H_j^{\ell_5} \big)} \\
 & \, +\sum_{\ell_1+\cdots+\ell_5=m-1} \chi^{[\ell_1]} \, \psi^{\ell_2} \, \p_t \psi^{\ell_3} \, 
\p_{y_3} \big( \dot{\chi}^{[\ell_4]} \, \p_{y_3} H_3^{\ell_5} \big) \, .
\end{align*}
The green terms in the decomposition of $\cE_3^m$ cancel exactly with the green terms in the decompositions of $\cE_1^m$ 
and $\cE_2^m$. Summing the contributions $\cE_1^m$, $\cE_2^m$, $\cE_3^m$, and canceling the terms with the same color, 
we get:
\begin{align}
\cD^m \, = \, \cE_4^m \, & \, +\sum_{\ell_1+\cdots+\ell_5=m} \chi^{[\ell_1]} \, \psi^{\ell_2} \, \p_\theta \psi^{\ell_3} \, 
\p_{y_3} \Big( \dot{\chi}^{[\ell_4]} \, \big( c \, \p_{y_3} H_3^{\ell_5} -b \, \p_{y_3} u_3^{\ell_5} \big) \Big) \notag \\
& \, +\sum_{\substack{\ell_1+\cdots+\ell_6=m \\ \ell_5,\ell_6 \ge 1}} \chi^{[\ell_1]} \, \psi^{\ell_2} \, \p_\theta \psi^{\ell_3} \, 
\p_{y_3} \Big( \dot{\chi}^{[\ell_4]} \, \p_{y_3} \big( H_3^{\ell_5} \, \xi_j \, u_j^{\ell_6} -u_3^{\ell_5} \, \xi_j \, H_j^{\ell_6} \big) \Big) \notag \\
& \, +\sum_{\ell_1+\cdots+\ell_5=m-1} \chi^{[\ell_1]} \, \psi^{\ell_2} \, \p_{y_j} \psi^{\ell_3} \, 
\p_{y_3} \Big( \dot{\chi}^{[\ell_4]} \, \big( u_j^0 \, \p_{y_3} H_3^{\ell_5} -H_j^0 \, \p_{y_3} u_3^{\ell_5} \big) \Big) 
\label{decompositionDmE4m} \\
& \, +\sum_{\substack{\ell_1+\cdots+\ell_6=m-1 \\ \ell_5,\ell_6 \ge 1}} \chi^{[\ell_1]} \, \psi^{\ell_2} \, \p_{y_j} \psi^{\ell_3} \, 
\p_{y_3} \Big( \dot{\chi}^{[\ell_4]} \, \p_{y_3} \big( H_3^{\ell_5} \, u_j^{\ell_6} -u_3^{\ell_5} \, H_j^{\ell_6} \big) \Big) \notag \\
& \, +\sum_{\ell_1+\cdots+\ell_5=m-1} \chi^{[\ell_1]} \, \psi^{\ell_2} \, \p_t \psi^{\ell_3} \, 
\p_{y_3} \big( \dot{\chi}^{[\ell_4]} \, \p_{y_3} H_3^{\ell_5} \big) \, .\notag
\end{align}
\bigskip

$\bullet$ \underline{Step 6}. The first symmetry formula.

In order to go further in the simplification of the quantity $\cD^m$, we need to exhibit some relations between the functions 
$\chi^{[\ell]}$ and the profiles $\psi^\mu$. Until now, we have never used any relation between these functions, which is the 
reason why there remain so many terms in the decompositions \eqref{decompositionE4m}, \eqref{decompositionDmE4m}. 
Let us recall that the functions $\chi^{[\ell]}$ are defined as the coefficients in the asymptotic expansion \eqref{defchichipointl}, 
where $t,y',x_3,y_3,\theta$ are related by \eqref{defchi} (which needs to be inverted in order to get $x_3$ as a function of all 
other variables). In particular, it should be kept in mind that $\chi^{[\ell]}$ can be expressed as a sum of products of the functions 
$\psi^2,\dots,\psi^\ell$. The first symmetry formula is as follows.

\begin{lemma}[The first symmetry formula]
\label{lem_symmetry_1}
The functions $\chi^{[\ell]}$ and $\psi^\mu$ satisfy the relation:
\begin{equation}
\label{symmetry_1}
\forall \, \ell \ge 0 \, ,\quad \sum_{\ell_1+\ell_2=\ell} \p_\theta \chi^{[\ell_1]} \, \p_t \psi^{\ell_2} 
\, = \, \sum_{\ell_1+\ell_2=\ell} \p_t \chi^{[\ell_1]} \, \p_\theta \psi^{\ell_2}  \, ,
\end{equation}
and similar relations with any couple of tangential partial derivatives chosen among $\{ \p_t,\p_{y_1},\p_{y_2},\p_\theta \}$.

For convenience, we have used in \eqref{symmetry_1} the convention $\psi^0 =\psi^1 =0$, see \eqref{s3-def_U^0,pm}. Since 
$\chi^{[0]}=\chi(y_3)$, both sums actually run over the couples of integers $(\ell_1,\ell_2)$ with $\ell_1 \ge 1,\ell_2 \ge 2$ and 
$\ell_1+\ell_2=\ell$.
\end{lemma}

We postpone the proof of Lemma \ref{lem_symmetry_1} to Appendix \ref{appendixB} and rather examine right now its 
implications on the decomposition \eqref{decompositionE4m}. Applying Lemma \ref{lem_symmetry_1} and possibly using 
a change of index $j \leftrightarrow j'$, we find that all blue terms in the decomposition \eqref{decompositionE4m} vanish. 
In view of \eqref{decompositionDmE4m}, this means that $\cD^m$ can be written as the sum of the black terms on the 
right hand side of \eqref{decompositionE4m} and of the right hand side of \eqref{decompositionDmE4m}. Expanding 
some of the partial derivatives on the right hand side of \eqref{decompositionE4m} and regrouping, we derive the expression:
\begin{align*}
\cD^m \, = \, & \, \sum_{\ell_1+\ell_2+\ell_3=m} \big( \p_{y_3} \chi^{[\ell_1]} \, \p_\theta \psi^{\ell_2} 
-\p_\theta (\dot{\chi}^{[\ell_1]} \, \psi^{\ell_2}) \big) \, \big( c \, \p_{y_3} H_3^{\ell_3} -b \, \p_{y_3} u_3^{\ell_3} \big) \\
 & \, +\sum_{\ell_1+\ell_2+\ell_3=m} \big( \p_{y_3} \chi^{[\ell_1]} \, \p_{y_j} \psi^{\ell_2} 
-\p_{y_j} (\dot{\chi}^{[\ell_1]} \, \psi^{\ell_2}) \big) \, \big( u_j^0 \, \p_{y_3} H_3^{\ell_3} -H_j^0 \, \p_{y_3} u_3^{\ell_3} \big) \\
 & \, +\sum_{\ell_1+\ell_2+\ell_3=m} \big( \p_{y_3} \chi^{[\ell_1]} \, \p_t \psi^{\ell_2} 
-\p_t (\dot{\chi}^{[\ell_1]} \, \psi^{\ell_2}) \big) \, \p_{y_3} H_3^{\ell_3} \\
 & \, +\sum_{\substack{\ell_1+\cdots+\ell_4=m \\ \ell_3,\ell_4 \ge 1}} \big( \p_{y_3} \chi^{[\ell_1]} \, \p_\theta \psi^{\ell_2} 
-\p_\theta (\dot{\chi}^{[\ell_1]} \, \psi^{\ell_2}) \big) \, 
\p_{y_3} \big( \xi_j \, u_j^{\ell_3} \, H_3^{\ell_4} -\xi_j \, H_j^{\ell_3} \, u_3^{\ell_4} \big) \\
 & \, +\sum_{\substack{\ell_1+\cdots+\ell_4=m-1 \\ \ell_3,\ell_4 \ge 1}} \big( \p_{y_3} \chi^{[\ell_1]} \, \p_{y_j} \psi^{\ell_2} 
-\p_{y_j} (\dot{\chi}^{[\ell_1]} \, \psi^{\ell_2}) \big) \, \p_{y_3} \big( u_j^{\ell_3} \, H_3^{\ell_4} -H_j^{\ell_3} \, u_3^{\ell_4} \big) \\
 & \, +{\color{orange} \sum_{\ell_1+\ell_2+\ell_3=m} \dot{\chi}^{[\ell_1]} \, \psi^{\ell_2} \, \p_{y_3} 
\big( b \, \p_\theta u_3^{\ell_3} -c \, \p_\theta H_3^{\ell_3} \big)} \\
 & \, +{\color{orange} \sum_{\ell_1+\ell_2+\ell_3=m-1} \dot{\chi}^{[\ell_1]} \, \psi^{\ell_2} \, \p_{y_3} 
\big( H_j^0 \, \p_{y_j} u_3^{\ell_3} -(\p_t +u_j^0 \, \p_{y_j}) H_3^{\ell_3} \big)} \\
 & \, +{\color{orange} \sum_{\substack{\ell_1+\cdots+\ell_4=m \\ \ell_3,\ell_4 \ge 1}} \dot{\chi}^{[\ell_1]} \, \psi^{\ell_2} \, 
\p_{y_3} \p_\theta \big( \xi_j \, H_j^{\ell_3} \, u_3^{\ell_4} -\xi_j \, u_j^{\ell_3} \, H_3^{\ell_4} \big)} \\
 & \, +{\color{orange} \sum_{\substack{\ell_1+\cdots+\ell_4=m-1 \\ \ell_3,\ell_4 \ge 1}} \dot{\chi}^{[\ell_1]} \, \psi^{\ell_2} \, 
\p_{y_3} \p_{y_j} \big( H_j^{\ell_3} \, u_3^{\ell_4} -u_j^{\ell_3} \, H_3^{\ell_4} \big)} \\
 & \, +{\color{orange} \sum_{\ell_1+\cdots+\ell_4=m+1} \dot{\chi}^{[\ell_1]} \, \psi^{\ell_2} \, \p_{y_3} 
\Big\{ \p_\theta \psi^{\ell_3} \, \big( c \, \p_{Y_3} H_3^{\ell_4} -b \, \p_{Y_3} u_3^{\ell_4} \big) \Big\}} \\
 & \, +{\color{orange} \sum_{\ell_1+\cdots+\ell_4=m} \dot{\chi}^{[\ell_1]} \, \psi^{\ell_2} \, \p_{y_3} 
\Big\{ \big( \p_t +u_j^0 \, \p_{y_j} \big) \psi^{\ell_3} \, \p_{Y_3} H_3^{\ell_4} -H_j^0 \, \p_{y_j} \psi^{\ell_3} \, \p_{Y_3} u_3^{\ell_4} \Big\}} \\
 & \, +{\color{orange} \sum_{\substack{\ell_1+\cdots+\ell_5=m+1 \\ \ell_3,\ell_4 \ge 1}} \dot{\chi}^{[\ell_1]} \, \psi^{\ell_2} \, 
\p_{y_3} \Big\{ \p_\theta \psi^{\ell_3} \, \p_{Y_3} \big( \xi_j \, u_j^{\ell_4} \, H_3^{\ell_5} -\xi_j \, H_j^{\ell_4} \, u_3^{\ell_5} \big) \Big\}} \\
 & \, +{\color{orange} \sum_{\substack{\ell_1+\cdots+\ell_5=m \\ \ell_3,\ell_4 \ge 1}} \dot{\chi}^{[\ell_1]} \, \psi^{\ell_2} \, 
\p_{y_3} \Big\{ \p_{y_j} \psi^{\ell_3} \, \p_{Y_3} \big( u_j^{\ell_4} \, H_3^{\ell_5} -H_j^{\ell_4} \, u_3^{\ell_5} \big) \Big\}} \\
 & \, +\sum_{\ell_1+\cdots+\ell_5=m} \chi^{[\ell_1]} \, \psi^{\ell_2} \, \p_\theta \psi^{\ell_3} \, 
\p_{y_3} \Big( \dot{\chi}^{[\ell_4]} \, \big( c \, \p_{y_3} H_3^{\ell_5} -b \, \p_{y_3} u_3^{\ell_5} \big) \Big) \\
 & \, +\sum_{\substack{\ell_1+\cdots+\ell_6=m \\ \ell_5,\ell_6 \ge 1}} \chi^{[\ell_1]} \, \psi^{\ell_2} \, \p_\theta \psi^{\ell_3} \, 
\p_{y_3} \Big( \dot{\chi}^{[\ell_4]} \, \p_{y_3} \big( H_3^{\ell_5} \, \xi_j \, u_j^{\ell_6} -u_3^{\ell_5} \, \xi_j \, H_j^{\ell_6} \big) \Big) \\
 & \, +\sum_{\ell_1+\cdots+\ell_5=m-1} \chi^{[\ell_1]} \, \psi^{\ell_2} \, \p_{y_j} \psi^{\ell_3} \, 
\p_{y_3} \Big( \dot{\chi}^{[\ell_4]} \, \big( u_j^0 \, \p_{y_3} H_3^{\ell_5} -H_j^0 \, \p_{y_3} u_3^{\ell_5} \big) \Big) \\
 & \, +\sum_{\substack{\ell_1+\cdots+\ell_6=m-1 \\ \ell_5,\ell_6 \ge 1}} \chi^{[\ell_1]} \, \psi^{\ell_2} \, \p_{y_j} \psi^{\ell_3} \, 
\p_{y_3} \Big( \dot{\chi}^{[\ell_4]} \, \p_{y_3} \big( H_3^{\ell_5} \, u_j^{\ell_6} -u_3^{\ell_5} \, H_j^{\ell_6} \big) \Big) \\
 & \, +\sum_{\ell_1+\cdots+\ell_5=m-1} \chi^{[\ell_1]} \, \psi^{\ell_2} \, \p_t \psi^{\ell_3} \, 
\p_{y_3} \big( \dot{\chi}^{[\ell_4]} \, \p_{y_3} H_3^{\ell_5} \big) \, .
\end{align*}

In the above first orange term, we use the relation $b \, \p_\theta u_3^{\ell_3} -c \, \p_\theta H_3^{\ell_3} =-F_6^{\ell_3-1}$, 
which cancels all other orange terms and introduces five new remainders. Once this \emph{final} (!) substitution has been 
performed, the expression of $\cD^m$ reads:
\begin{align*}
\cD^m \, = \, & \, {\color{red} \sum_{\ell_1+\ell_2+\ell_3=m} \big( \p_{y_3} \chi^{[\ell_1]} \, \p_\theta \psi^{\ell_2} 
-\p_\theta (\dot{\chi}^{[\ell_1]} \, \psi^{\ell_2}) \big) \, \big( c \, \p_{y_3} H_3^{\ell_3} -b \, \p_{y_3} u_3^{\ell_3} \big)} \\
 & \, +{\color{ForestGreen} \sum_{\ell_1+\ell_2+\ell_3=m-1} \big( \p_{y_3} \chi^{[\ell_1]} \, \p_{y_j} \psi^{\ell_2} 
-\p_{y_j} (\dot{\chi}^{[\ell_1]} \, \psi^{\ell_2}) \big) \, \big( u_j^0 \, \p_{y_3} H_3^{\ell_3} -H_j^0 \, \p_{y_3} u_3^{\ell_3} \big)} \\
 & \, +\sum_{\ell_1+\ell_2+\ell_3=m-1} \big( \p_{y_3} \chi^{[\ell_1]} \, \p_t \psi^{\ell_2} 
-\p_t (\dot{\chi}^{[\ell_1]} \, \psi^{\ell_2}) \big) \, \p_{y_3} H_3^{\ell_3} \\
 & \, +\sum_{\substack{\ell_1+\cdots+\ell_4=m \\ \ell_3,\ell_4 \ge 1}} \big( \p_{y_3} \chi^{[\ell_1]} \, \p_\theta \psi^{\ell_2} 
-\p_\theta (\dot{\chi}^{[\ell_1]} \, \psi^{\ell_2}) \big) \, 
\p_{y_3} \big( \xi_j \, u_j^{\ell_3} \, H_3^{\ell_4} -\xi_j \, H_j^{\ell_3} \, u_3^{\ell_4} \big) \\
 & \, +\sum_{\substack{\ell_1+\cdots+\ell_4=m-1 \\ \ell_3,\ell_4 \ge 1}} \big( \p_{y_3} \chi^{[\ell_1]} \, \p_{y_j} \psi^{\ell_2} 
-\p_{y_j} (\dot{\chi}^{[\ell_1]} \, \psi^{\ell_2}) \big) \, \p_{y_3} \big( u_j^{\ell_3} \, H_3^{\ell_4} -H_j^{\ell_3} \, u_3^{\ell_4} \big) \\
 & \, +{\color{red} \sum_{\ell_1+\cdots+\ell_5=m} \chi^{[\ell_1]} \, \psi^{\ell_2} \, \p_\theta \psi^{\ell_3} \, 
\p_{y_3} \Big( \dot{\chi}^{[\ell_4]} \, \big( c \, \p_{y_3} H_3^{\ell_5} -b \, \p_{y_3} u_3^{\ell_5} \big) \Big)} \\
 & \, +{\color{ForestGreen} \sum_{\ell_1+\cdots+\ell_5=m-1} \chi^{[\ell_1]} \, \psi^{\ell_2} \, \p_{y_j} \psi^{\ell_3} \, 
\p_{y_3} \Big( \dot{\chi}^{[\ell_4]} \, \big( u_j^0 \, \p_{y_3} H_3^{\ell_5} -H_j^0 \, \p_{y_3} u_3^{\ell_5} \big) \Big)} \\
 & \, +\sum_{\ell_1+\cdots+\ell_5=m-1} \chi^{[\ell_1]} \, \psi^{\ell_2} \, \p_t \psi^{\ell_3} \, 
\p_{y_3} \big( \dot{\chi}^{[\ell_4]} \, \p_{y_3} H_3^{\ell_5} \big) \\
 & \, +\sum_{\substack{\ell_1+\cdots+\ell_6=m \\ \ell_5,\ell_6 \ge 1}} \chi^{[\ell_1]} \, \psi^{\ell_2} \, \p_\theta \psi^{\ell_3} \, 
\p_{y_3} \Big( \dot{\chi}^{[\ell_4]} \, \p_{y_3} \big( H_3^{\ell_5} \, \xi_j \, u_j^{\ell_6} -u_3^{\ell_5} \, \xi_j \, H_j^{\ell_6} \big) \Big) \\
 & \, +\sum_{\substack{\ell_1+\cdots+\ell_6=m-1 \\ \ell_5,\ell_6 \ge 1}} \chi^{[\ell_1]} \, \psi^{\ell_2} \, \p_{y_j} \psi^{\ell_3} \, 
\p_{y_3} \Big( \dot{\chi}^{[\ell_4]} \, \p_{y_3} \big( H_3^{\ell_5} \, u_j^{\ell_6} -u_3^{\ell_5} \, H_j^{\ell_6} \big) \Big) \\
 & \, -{\color{red} \sum_{\ell_1+\cdots+\ell_5=m} \dot{\chi}^{[\ell_1]} \, \psi^{\ell_2} \, \p_\theta \psi^{\ell_3} \, 
\p_{y_3} \Big( \chi^{[\ell_4]} \, \big( c \, \p_{y_3} H_3^{\ell_5} -b \, \p_{y_3} u_3^{\ell_5} \big) \Big)} \\
 & \, -{\color{ForestGreen} \sum_{\ell_1+\cdots+\ell_5=m-1} \dot{\chi}^{[\ell_1]} \, \psi^{\ell_2} \, \p_{y_j} \psi^{\ell_3} \, 
\p_{y_3} \Big( \chi^{[\ell_4]} \, \big( u_j^0 \, \p_{y_3} H_3^{\ell_5} -H_j^0 \, \p_{y_3} u_3^{\ell_5} \big) \Big)} \\
 & \, -\sum_{\ell_1+\cdots+\ell_5=m-1} \dot{\chi}^{[\ell_1]} \, \psi^{\ell_2} \, \p_t \psi^{\ell_3} \, 
\p_{y_3} \big( \chi^{[\ell_4]} \, \p_{y_3} H_3^{\ell_5} \big) \\
 & \, -\sum_{\substack{\ell_1+\cdots+\ell_6=m \\ \ell_5,\ell_6 \ge 1}} \dot{\chi}^{[\ell_1]} \, \psi^{\ell_2} \, \p_\theta \psi^{\ell_3} \, 
\p_{y_3} \Big( \chi^{[\ell_4]} \, \p_{y_3} \big( H_3^{\ell_5} \, \xi_j \, u_j^{\ell_6} -u_3^{\ell_5} \, \xi_j \, H_j^{\ell_6} \big) \Big) \\
 & \, -\sum_{\substack{\ell_1+\cdots+\ell_6=m-1 \\ \ell_5,\ell_6 \ge 1}} \dot{\chi}^{[\ell_1]} \, \psi^{\ell_2} \, \p_{y_j} \psi^{\ell_3} \, 
\p_{y_3} \Big( \chi^{[\ell_4]} \, \p_{y_3} \big( H_3^{\ell_5} \, u_j^{\ell_6} -u_3^{\ell_5} \, H_j^{\ell_6} \big) \Big) \, .
\end{align*}
There are overall fifteen terms, which we can match into five groups, each of three terms. For instance, the two 
first groups have been highlighted in red and green. Once a last cancellation in the $y_3$-derivative has been 
made, we are led to the final expression of $\cD^m$:
\begin{align*}
\cD^m \, = \, & \, \sum_{\ell_1+\ell_2=m} \bX_\theta^{\ell_1} \, \big( c \, \p_{y_3} H_3^{\ell_2} -b \, \p_{y_3} u_3^{\ell_2} \big) 
+\sum_{\ell_1+\ell_2=m-1} \bX_j^{\ell_1} \, \big( u_j^0 \, \p_{y_3} H_3^{\ell_2} -H_j^0 \, \p_{y_3} u_3^{\ell_2} \big) \\
 & \, +\sum_{\ell_1+\ell_2=m-1} \bX_t^{\ell_1} \, \p_{y_3} H_3^{\ell_3} 
+\sum_{\substack{\ell_1+\ell_2+\ell_3=m \\ \ell_2,\ell_3 \ge 1}} \bX_\theta^{\ell_1} \, 
\p_{y_3} \big( \xi_j \, u_j^{\ell_2} \, H_3^{\ell_3} -\xi_j \, H_j^{\ell_2} \, u_3^{\ell_3} \big) \\
 & \, +\sum_{\substack{\ell_1+\ell_2+\ell_3=m-1 \\ \ell_2,\ell_3 \ge 1}} \bX_j^{\ell_1} \, 
 \p_{y_3} \big( u_j^{\ell_2} \, H_3^{\ell_3} -H_j^{\ell_2} \, u_3^{\ell_3} \big) \, ,
\end{align*}
with
\begin{multline}
\label{defbXell}
\forall \, \ell \, = \, 0,\dots,m-1\, ,\\
\bX_\theta^\ell \, := \, \sum_{\ell_1+\ell_2=\ell} \p_{y_3} \chi^{[\ell_1]} \, \p_\theta \psi^{\ell_2} -\p_\theta \big( \dot{\chi}^{[\ell_1]} \, \psi^{\ell_2} \big) 
+\sum_{\ell_1 +\cdots +\ell_4=\ell} \big( \chi^{[\ell_1]} \, \p_{y_3} \dot{\chi}^{[\ell_2]} 
-\dot{\chi}^{[\ell_1]} \, \p_{y_3} \chi^{[\ell_2]} \big) \, \psi^{\ell_3} \, \p_\theta \psi^{\ell_4} \, ,
\end{multline}
and similar definitions for $\bX_1^\ell$, $\bX_2^\ell$, $\bX_t^\ell$ by just considering another tangential derivative acting on the front 
profiles $\psi^{\ell_2}$ or $\psi^{\ell_4}$ in \eqref{defbXell}.
\bigskip

$\bullet$ \underline{Step 7}. The second symmetry formula. Conclusion.

The final argument comes from a second symmetry formula which relates the $y_3$-derivatives of the functions 
$\chi^{[\ell]}$ and $\dot{\chi}^{[\ell]}$ together with tangential derivatives of the front profiles $\psi^\mu$.

\begin{lemma}[The second symmetry formula]
\label{lem_symmetry_2}
The functions $\bX_\theta^\ell$ defined in \eqref{defbXell} vanish, and similarly, the functions $\bX_1^\ell$, $\bX_2^\ell$, 
$\bX_t^\ell$ also vanish.
\end{lemma}

Since both Lemma \ref{lem_symmetry_1} and Lemma \ref{lem_symmetry_2} rely on combinatorial arguments that 
are independent of the analysis of this Chapter, we also postpone the proof of Lemma \ref{lem_symmetry_2} to 
Appendix \ref{appendixB}. (As the reader will see in Appendix \ref{appendixB}, the result of Lemma \ref{lem_symmetry_1} 
and Lemma \ref{lem_symmetry_2} even holds if the first profile $\psi^1$ does not vanish.)

Because of our final decomposition for the quantity $\cD^m$, the result of Lemma \ref{lem_symmetry_2} implies $\cD^m =0$, 
which was exactly the relation we were aiming at. This completes the proof of Lemma \ref{lem_compatibilite_div} and shows that 
the source terms of the fast problem \eqref{fast_problem_m+1} satisfy the solvability condition \eqref{compatibilite_pb_rapide_d}.
\end{proof}

We now know that the source terms $(F^{\, m,\pm},F_8^{\, m,\pm},G^m)$ in \eqref{fast_problem_m+1} satisfy all five solvability 
conditions \eqref{compatibilite_pb_rapide} of Theorem \ref{theorem_fast_problem}. Since the solvability condition 
\eqref{compatibilite_pb_rapide} does not involve the mean of the boundary source term $\widehat{G}^m(0)$, the 
existence of a solution to \eqref{fast_problem_m+1} does not depend on the choice of $\widehat{\psi}^{\, m+1}(0)$. 
However, the solution itself will depend on the choice of $\widehat{\psi}^{\, m+1}(0)$, which is the reason why below 
we deal separately with the zero Fourier mode of $U^{\, m+1,\pm}$.

Rather than solving the fast problem \eqref{fast_problem_m+1} for all Fourier modes in $\theta$, it is convenient to first solve 
\eqref{fast_problem_m+1} for all \emph{nonzero} Fourier modes in $\theta$. Namely, we set:
$$
F_\sharp^{\, m,\pm} \, := \, \sum_{k \neq 0} \widehat{F}^{\, m,\pm}(k) \, {\rm e}^{i \, k \, \theta} \, ,\quad 
F_{8,\sharp}^{\, m,\pm} \, := \, \sum_{k \neq 0} \widehat{F}_8^{\, m,\pm}(k) \, {\rm e}^{i \, k \, \theta} \, ,\quad 
G_\sharp^m \, := \, \sum_{k \neq 0} \widehat{G}^m(k) \, {\rm e}^{i \, k \, \theta} \, .
$$
Observe in particular that $G_\sharp^m$ is fully determined at this stage since it is independent of $\widehat{\psi}^{\, m+1}(0)$, 
which is still unknown. Using the induction assumption $(H(m))$, Lemma \ref{lem_compatibilite_bord} and Lemma 
\ref{lem_compatibilite_div}, we can verify that the source terms $(F_\sharp^{\, m,\pm},F_{8,\sharp}^{\, m,\pm},G_\sharp^m)$ satisfy 
the solvability conditions \eqref{cor_compatibilite_pb_rapide} of Corollary \ref{cor_fast_problem}. By applying Corollary 
\ref{cor_fast_problem}, we can therefore construct a solution $\bU^{\, m+1,\pm} \in S^\pm$ to the following fast problem (observe 
the slight difference with \eqref{fast_problem_m+1}):
\begin{equation}
\label{fast_problem_m+1'}
\begin{cases}
\cL_f^\pm(\partial) \, \bU^{\, m+1,\pm} \, = \, F_\sharp^{\, m,\pm} \, ,& y \in \Omega_0^\pm \, ,\, \pm Y_3 >0 \, ,\\
\p_{Y_3} \cH_3^{\, m+1,\pm} +\xi_j \, \p_\theta \cH_j^{\, m+1,\pm} \, = \, F_{8,\sharp}^{\, m,\pm} \, ,& y \in \Omega_0^\pm \, ,\, \pm Y_3 >0 \, ,\\
B^+ \, \bU^{\, m+1,+}|_{y_3=Y_3=0} +B^- \, \bU^{\, m+1,-}|_{y_3=Y_3=0} \, = \, G_\sharp^m \, ,
\end{cases}
\end{equation}
and this solution $\bU^{\, m+1,\pm}$ is purely oscillating in $\theta$.

Let us already state the following result, which is crucial in view of showing that the boundary condition \eqref{inductionHm3} will 
be satisfied for $\mu=m+1$.

\begin{lemma}[The top and bottom boundary conditions for nonzero Fourier modes]
\label{lem-induction1}
The solution $\bU^{\, m+1,\pm} \in S^\pm$ to \eqref{fast_problem_m+1'} satisfies:
$$
\underline{\cU}_3^{\, m+1,\pm}|_{y_3=\pm 1} \, = \, \underline{\cH}_3^{\, m+1,\pm}|_{y_3=\pm 1} \, = \, 0 \, .
$$
\end{lemma}

\begin{proof}[Proof of Lemma \ref{lem-induction1}]
The source term $F_\sharp^{\, m,\pm}$ in \eqref{fast_problem_m+1'} collects the nonzero Fourier modes in $\theta$ of $F^{\, m,\pm}$ 
whose expression is given in \eqref{s3-def_terme_source_F^m,pm}. Computing the limit $Y_3 \to \pm \infty$ and restricting to 
$|y_3|>2/3$ (recall that all functions $\chi^{[\ell]}$, $\dot{\chi}^{[\ell]}$ vanish for $|y_3|>2/3$ since the original cut-off function 
$\chi$ in \eqref{defchi} has compact support in $[-2/3,2/3]$, see \eqref{formule_chim} and \eqref{formule_chipointm} in Appendix 
\ref{appendixB}), we get:
\begin{align*}
\uF^{\, m,\pm} \, = \, -(L_s^\pm(\p) \uU^{\, m,\pm}) 
-\sum_{\substack{\ell_1+\ell_2=m+1 \\ \ell_1,\ell_2 \ge 1}} \xi_j \, \bA_j(\uU^{\ell_1,\pm},\p_\theta \uU^{\ell_2,\pm}) 
-\sum_{\substack{\ell_1+\ell_2=m \\ \ell_1,\ell_2 \ge 1}} \bA_\alpha(\uU^{\ell_1,\pm},\p_{y_\alpha} \uU^{\ell_2,\pm}) \, .
\end{align*}
We know furthermore from \eqref{inductionHm2} that the functions $\uU^{\, 1,\pm},\dots,\uU^{\, m,\pm}$ satisfy for $|y_3|>2/3$:
$$
\forall \, \mu \, = \, 1,\dots,m\, ,\quad \cA^\pm \, \p_\theta \uU^{\mu,\pm} \, = \, \uF^{\mu-1,\pm} \, .
$$
The argument is now the following: since $\uF^{0,\pm}=0$ and since $\cA^\pm$ are invertible (see Appendix \ref{appendixA}), 
we have $\p_\theta \uU^{\, 1,\pm}=0$ for $|y_3|>2/3$. Specifying the above expression of $\uF^{\, m,\pm}$ to $m=1$ shows that 
$\p_\theta \uF^{\, 1,\pm}=0$ for $|y_3|>2/3$, hence $\p_\theta^2 \uU^{\, 2,\pm}=0$ for $|y_3|>2/3$ and therefore $\p_\theta 
\uU^{\, 2,\pm}=0$ (which implies in particular $\uF^{\, 1,\pm}=0$ for $|y_3|>2/3$). We then show inductively that $\p_\theta \uU^{\, 1,\pm} 
=\cdots=\p_\theta U^{\, m,\pm}=0$ for $|y_3|>2/3$, and this implies ultimately $\p_\theta \uF^{\, m,\pm}=0$ for $|y_3|>2/3$. By 
computing the limit $Y_3 \to \pm \infty$ in \eqref{fast_problem_m+1'}, we have:
$$
\cA^\pm \, \p_\theta \underline{\bU}^{\, m+1,\pm} \, = \, \uF^{\, m,\pm} \, ,
$$
and therefore $\p_\theta \underline{\bU}^{\, m+1,\pm}=0$ for $|y_3| >2/3$. Since $\underline{\bU}^{\, m+1,\pm}$ is purely oscillating 
in $\theta$, we obtain $\underline{\bU}^{\, m+1,\pm}=0$ for $|y_3| >2/3$. This proves the result of Lemma \ref{lem-induction1} 
(and actually even more but never mind).
\end{proof}

The next corrector $(U^{\, m+1,\pm},\psi^{\, m+2})$ in the induction process will be constructed in the following way:
\begin{equation}
\label{decompositionUm+1}
U^{\, m+1,\pm} \, := \, \bU^{\, m+1,\pm} +\widehat{U}^{\, m+1,\pm}(0) 
\pm \sum_{k \neq 0} |k| \, \widehat{\psi}^{\, m+2} (t,y',k) \, \chi (y_3) \, {\rm e}^{\mp \, |k| \, Y_3 +i \, k \, \theta} \, \cR^\pm(k) \, ,
\end{equation}
where the vectors $\cR^\pm(k)$ are defined in \eqref{appA-defRLkpm}, and we still need to explain how we construct 
$\widehat{\psi}^{\, m+1}(0) \in H^\infty$, $\widehat{U}^{\, m+1,\pm}(0) \in S^\pm$ and $\psi^{\, m+2}_\sharp \in H^\infty_\sharp$. 
To make sure that \eqref{fast_problem_m+1} is satisfied, we need $\widehat{\psi}^{\, m+1}(0)$ and $\widehat{U}^{\, m+1,\pm}(0)$ 
to verify:
\begin{equation}
\label{fast_problem_m+1''}
\begin{cases}
A_3^\pm \, \p_{Y_3} \widehat{U}^{\, m+1,\pm}(0) \, = \, \widehat{F}^{\, m,\pm}(0) \, ,& y \in \Omega_0^\pm \, ,\, \pm Y_3 >0 \, ,\\
\p_{Y_3} \widehat{H}_3^{\, m+1,\pm}(0) \, = \, \widehat{F}_8^{\, m,\pm}(0) \, ,& y \in \Omega_0^\pm \, ,\, \pm Y_3 >0 \, ,\\
B^+ \, \widehat{U}^{\, m+1,+}(0)|_{y_3=Y_3=0} +B^- \, \widehat{U}^{\, m+1,-}(0)|_{y_3=Y_3=0} \, = \, \widehat{G}^m(0) \, .
\end{cases}
\end{equation}

Because of the compatibility conditions \eqref{inductionHm6}, solving the fast differential equations in \eqref{fast_problem_m+1''} 
amounts to solving:
\begin{align}
&\p_{Y_3} \widehat{q}^{\, m+1,\pm}(0) \, = \, \widehat{F}_3^{\, m,\pm}(0) \, ,\quad &\p_{Y_3} \widehat{u}_3^{\, m+1,\pm}(0) \, 
= \, \widehat{F}_7^{\, m,\pm}(0) \, ,\notag\\
&\p_{Y_3} \widehat{H}_3^{\, m+1,\pm}(0) \, = \, \widehat{F}_8^{\, m,\pm}(0) \, ,\quad & \label{defqu3H3m+1star}
\end{align}
which is possible because we know from \eqref{inductionHm5} that $\widehat{F}_3^{\, m,\pm}(0)$, $\widehat{F}_7^{\, m,\pm}(0)$ and 
$\widehat{F}_8^{\, m,\pm}(0)$ actually belong to $S_\star^\pm$. We thus define the functions $\widehat{q}_\star^{\, m+1,\pm}(0)$, 
$\widehat{u}_{3,\star}^{\, m+1,\pm}(0)$, $\widehat{H}_{3,\star}^{\, m+1,\pm}(0)$ as the unique solutions in $S_\star^\pm$ to the ordinary 
differential equations:
$$
\p_{Y_3} \widehat{q}_\star^{\, m+1,\pm}(0) \, = \, \widehat{F}_3^{\, m,\pm}(0) \, ,\quad 
\p_{Y_3} \widehat{u}_{3,\star}^{\, m+1,\pm}(0) \, = \, \widehat{F}_7^{\, m,\pm}(0) \, ,\quad 
\p_{Y_3} \widehat{H}_{3,\star}^{\, m+1,\pm}(0) \, = \, \widehat{F}_8^{\, m,\pm}(0) \, .
$$
Whatever choice we make for the remaining functions $\widehat{u}_{j,\star}^{\, m+1,\pm}(0)$ and $\widehat{H}_{j,\star}^{\, m+1,\pm}(0)$ ($j=1,2$), 
and for the slow mean $\widehat{\uU}^{\, m+1,\pm}(0)$, we shall verify the fast differential equations in \eqref{fast_problem_m+1''}. Observe 
that the functions $\widehat{q}_\star^{\, m+1,\pm}(0),\widehat{u}_{3,\star}^{\, m+1,\pm}(0),\widehat{H}_{3,\star}^{\, m+1,\pm}(0)$ enter the 
boundary conditions on $\Gamma_0$ of \eqref{fast_problem_m+1''} but since they have already been determined, we shall now treat them 
as forcing terms in the boundary conditions of \eqref{fast_problem_m+1''}.

In the following Section, we explain how we construct the slow mean $\widehat{\uU}^{\, m+1,\pm}(0)$ together with the slow mean 
of the front profile $\widehat{\psi}^{\, m+1}(0)$. The remaining fast means, namely $\widehat{u}_{j,\star}^{\, m+1,\pm}(0)$ and 
$\widehat{H}_{j,\star}^{\, m+1,\pm}(0)$ ($j=1,2$), will be dealt with afterwards.

\section{The slow mean}

In this Section, we explain how we construct both $\widehat{\uU}^{\, m+1,\pm}(0)$ and $\widehat{\psi}^{\, m+1}(0)$.

\subsection{Collecting the equations}

In order to verify the fast system \eqref{inductionHm2} for $\mu=m+1$, we have seen in the previous Section that it remains to verify 
the boundary conditions for the zero Fourier mode:
$$
B^+ \, \widehat{U}^{\, m+1,+}(0)|_{y_3=Y_3=0} +B^- \, \widehat{U}^{\, m+1,-}(0)|_{y_3=Y_3=0} \, = \, \widehat{G}^m(0) \, ,
$$
which equivalently reads:
\begin{equation}
\label{slow_mean_m+1_saut-bis}
\begin{cases}
\widehat{\uu}_3^{\, m+1,\pm}(0)|_{y_3=0} \, = \, \widehat{G}_1^{\, m,\pm}(0) -\widehat{u}_{3,\star}^{\, m+1,\pm}(0)|_{y_3=Y_3=0} \, ,& \\
\widehat{\uH}_3^{\, m+1,\pm}(0)|_{y_3=0} \, = \, \widehat{G}_2^{\, m,\pm}(0) -\widehat{H}_{3,\star}^{\, m+1,\pm}(0)|_{y_3=Y_3=0} \, ,& \\
\widehat{\uq}^{\, m+1,+}(0)|_{y_3=0} -\widehat{\uq}^{\, m+1,-}(0)|_{y_3=0} \, = \, 
-\widehat{q}_\star^{\, m+1,+}(0)|_{y_3=Y_3=0} +\widehat{q}_\star^{\, m+1,-}(0)|_{y_3=Y_3=0} \, .
\end{cases}
\end{equation}
For future use, we introduce short notation to denote the source terms in \eqref{slow_mean_m+1_saut-bis}, and rewrite 
\eqref{slow_mean_m+1_saut-bis} as:
\begin{equation}
\label{slow_mean_m+1_saut}
\begin{cases}
\widehat{\uu}_3^{\, m+1,\pm}(0)|_{y_3=0} \, = \, \bG_1^{\, m,\pm} \, ,& \\
\widehat{\uH}_3^{\, m+1,\pm}(0)|_{y_3=0} \, = \, \bG_2^{\, m,\pm} \, ,& \\
\widehat{\uq}^{\, m+1,+}(0)|_{y_3=0} -\widehat{\uq}^{\, m+1,-}(0)|_{y_3=0} \, = \, \bG_5^m \, .
\end{cases}
\end{equation}
At this stage, only $\bG_5^m$ is known since $\bG_1^{\, m,\pm}$ and $\bG_2^{\, m,\pm}$ incorporate in their definition the slow mean 
$\widehat{\psi}^{\, m+1}(0)$ which has not been fixed yet.

The boundary conditions on the top and bottom boundaries $\Gamma^\pm$ for the slow mean correspond to enforcing the condition 
\eqref{inductionHm3} with $\mu=m+1$ for the zero Fourier mode only, that is:
\begin{equation}
\label{slow_mean_m+1_hautbas}
\widehat{\uu}_3^{\, m+1,\pm}(0)|_{\Gamma^\pm} \, = \, \widehat{\uH}_3^{\, m+1,\pm}(0)|_{\Gamma^\pm} \, = \, 0 \, .
\end{equation}

The evolution equations inside the domains $\Omega_0^\pm$ correspond to enforcing the condition \eqref{inductionHm5} for $\mu=m+1$. 
Let us make this point clear. Recalling the general expression \eqref{terme_source_F^m,pmbarre}, we have:
\begin{align*}
\uF^{\, m+1,\pm} \, = \, & \, -{\color{blue} L_s^\pm(\p) \uU^{\, m+1,\pm}} \, 
+\sum_{\ell_1+\ell_2+\ell_3=m+2} \chi^{[\ell_1]} \, \p_\theta \psi^{\ell_2} \, \cA^\pm \, \p_{y_3} \uU^{\ell_3,\pm} \\
& \, +\sum_{\ell_1+\ell_2+\ell_3=m+1} \chi^{[\ell_1]} \, (\p_t \psi^{\ell_2} \, A_0 +\p_{y_j} \psi^{\ell_2} \, A_j^\pm) \, \p_{y_3} \uU^{\ell_3,\pm} 
+\sum_{\ell_1+\ell_2+\ell_3=m+1} \dot{\chi}^{[\ell_1]} \, \psi^{\ell_2} \, A_3^\pm \, \p_{y_3} \uU^{\ell_3,\pm} \\
& \, -{\color{red} \sum_{\substack{\ell_1+\ell_2=m+2 \\ \ell_1,\ell_2 \ge 1}} 
\xi_j \, \bA_j (\uU^{\ell_1,\pm},\p_\theta \uU^{\ell_2,\pm})} 
-\sum_{\substack{\ell_1+\ell_2=m+1 \\ \ell_1,\ell_2 \ge 1}} \bA_\alpha (\uU^{\ell_1,\pm},\p_{y_\alpha} \uU^{\ell_2,\pm}) \\
& \, +\sum_{\substack{\ell_1+\cdots+\ell_4=m+2 \\ \ell_3,\ell_4 \ge 1}} \chi^{[\ell_1]} \, \p_\theta \psi^{\ell_2} \, 
\xi_j \, \bA_j (\uU^{\ell_3,\pm},\p_{y_3} \uU^{\ell_4,\pm}) \\
& \, +\sum_{\substack{\ell_1+\cdots+\ell_4=m+1 \\ \ell_3,\ell_4 \ge 1}} \chi^{[\ell_1]} \, \p_{y_j} \psi^{\ell_2} \, 
\bA_j (\uU^{\ell_3,\pm},\p_{y_3} \uU^{\ell_4,\pm}) 
+\sum_{\substack{\ell_1+\cdots+\ell_4=m+1 \\ \ell_3,\ell_4 \ge 1}} \dot{\chi}^{[\ell_1]} \, \psi^{\ell_2} \, 
\bA_3 (\uU^{\ell_3,\pm},\p_{y_3} \uU^{\ell_4,\pm}) \, ,
\end{align*}
where we have highlighted in blue and red the only quantities that are not completely known at this stage. Let us first remark that 
the above red term has zero mean with respect to $\theta$ on $\bT$ whatever choice we make below for $\widehat{\uU}^{\, m+1,\pm}(0)$ 
(use the symmetry of the $\bA_j$'s). In other words, enforcing the condition \eqref{inductionHm5} for $\mu=m+1$ corresponds to 
verifying the linearized (inhomogeneous) MHD equations:
\begin{equation}
\label{slow_mean_m+1_edp}
\begin{cases}
L_s^\pm(\partial) \, \widehat{\uU}^{\, m+1,\pm}(0) \, = \, \bF^{\, m,\pm} \, ,& y \in \Omega_0^\pm \, ,\\
\nabla \cdot \widehat{\uH}^{\, m+1,\pm}(0) \, = \, \bF_8^{\, m,\pm} \, ,& y \in \Omega_0^\pm \, ,
\end{cases}
\end{equation}
with:
\begin{align}
\bF^{\, m,\pm} := {\bf c}_0 \, \Big\{ & 
\sum_{\ell_1+\ell_2+\ell_3=m+2} \chi^{[\ell_1]} \, \p_\theta \psi^{\ell_2} \, \cA^\pm \, \p_{y_3} \uU^{\ell_3,\pm} \notag \\
& +\sum_{\ell_1+\ell_2+\ell_3=m+1} \chi^{[\ell_1]} \, (\p_t \psi^{\ell_2} \, A_0 +\p_{y_j} \psi^{\ell_2} \, A_j^\pm) \, \p_{y_3} \uU^{\ell_3,\pm} 
+\sum_{\ell_1+\ell_2+\ell_3=m+1} \dot{\chi}^{[\ell_1]} \, \psi^{\ell_2} \, A_3^\pm \, \p_{y_3} \uU^{\ell_3,\pm} \notag \\
& -\sum_{\substack{\ell_1+\ell_2=m+1 \\ \ell_1,\ell_2 \ge 1}} \bA_\alpha ( \uU^{\ell_1,\pm},\p_{y_\alpha} \uU^{\ell_2,\pm}) 
+\sum_{\substack{\ell_1+\cdots+\ell_4=m+2 \\ \ell_3,\ell_4 \ge 1}} \chi^{[\ell_1]} \, \p_\theta \psi^{\ell_2} \, 
\xi_j \, \bA_j (\uU^{\ell_3,\pm},\p_{y_3} \uU^{\ell_4,\pm}) \label{defsource_slow_mean_m+1} \\
& +\sum_{\substack{\ell_1+\cdots+\ell_4=m+1 \\ \ell_3,\ell_4 \ge 1}} \chi^{[\ell_1]} \, \p_{y_j} \psi^{\ell_2} \, 
\bA_j (\uU^{\ell_3,\pm},\p_{y_3} \uU^{\ell_4,\pm}) 
+\sum_{\substack{\ell_1+\cdots+\ell_4=m+1 \\ \ell_3,\ell_4 \ge 1}} \dot{\chi}^{[\ell_1]} \, \psi^{\ell_2} \, 
\bA_3 (\uU^{\ell_3,\pm},\p_{y_3} \uU^{\ell_4,\pm}) \Big\} ,\notag
\end{align}
and
\begin{align}
\bF_8^{\, m,\pm} \, := \, {\bf c}_0 \, \Big\{ \, 
\sum_{\ell_1+\ell_2+\ell_3=m+2} \chi^{[\ell_1]} \, \p_\theta \psi^{\ell_2} \, \p_{y_3} \xi_j \, \uH_j^{\ell_3,\pm} 
& \, +\sum_{\ell_1+\ell_2+\ell_3=m+1} \chi^{[\ell_1]} \, \p_{y_j} \psi^{\ell_2} \, \p_{y_3} \uH_j^{\ell_3,\pm} \label{defsource8_slow_mean_m+1} \\
& \, +\sum_{\ell_1+\ell_2+\ell_3=m+1} \dot{\chi}^{[\ell_1]} \, \psi^{\ell_2} \, \p_{y_3} \uH_3^{\ell_3,\pm}
\, \Big\} \, .\notag
\end{align}
Let us recall that the operators $L_s^\pm(\partial)$ in \eqref{slow_mean_m+1_edp} are defined by:
$$
L_s^\pm(\partial) \, = \, A_0 \, \p_t +A_\alpha^\pm \, \p_{y_\alpha} \, ,
$$
and the expressions of the matrices $A_\alpha^\pm$ are given in Appendix \ref{appendixA}.

The evolution equations \eqref{slow_mean_m+1_edp} and boundary conditions \eqref{slow_mean_m+1_saut}, 
\eqref{slow_mean_m+1_hautbas} are supplemented with the initial conditions:
\begin{equation}
\label{slow_mean_m+1_di}
\widehat{\uu}^{\, m+1,\pm}(0)|_{t=0} \, = \, \mathfrak{u}^{\, m+1,\pm}_0 \, ,\quad 
\widehat{\uH}^{\, m+1,\pm}(0)|_{t=0} \, = \, \mathfrak{H}^{\, m+1,\pm}_0 \, ,
\end{equation}
for which we still need to explain how they are chosen.

\subsection{Solvability of the linearized current vortex sheet system}

In this Paragraph, we consider the linearized current vortex sheet system in its general form, that is, with arbitrary source 
terms. In the following Paragraph, we shall specify to the source terms that are explicitly given in \eqref{slow_mean_m+1_edp}, 
\eqref{slow_mean_m+1_saut-bis}, \eqref{slow_mean_m+1_hautbas} and we shall explain why Proposition \ref{prop-systemelent} 
below allows us to construct both $\widehat{\psi}^{\, m+1}(0)$ and $\widehat{\uU}^{\, m+1,\pm}(0)$.

In what follows, the vector $U$ still denotes the collection of unknowns $(u,H,q) \in \bR^7$. We consider the following linearized 
MHD equations in the fixed domains $\Omega_0^\pm$:
\begin{equation}
\label{systeme_lent}
\begin{cases}
L_s^\pm(\partial) \, U^\pm \, = \, \bF^\pm \, ,& y \in \Omega_0^\pm \, ,\\
\nabla \cdot H^\pm \, = \, \bF_8^\pm \, ,& y \in \Omega_0^\pm \, ,
\end{cases}
\end{equation}
which reads explicitly:
\begin{equation}
\label{systeme_lent_exp}
\begin{cases}
\big( \p_t  +u_j^{0,\pm} \, \p_{y_j} \big) u_\alpha^\pm +u_\alpha^{0,\pm} \, \nabla \cdot u^\pm 
-H_j^{0,\pm} \, \p_{y_j} H_\alpha^\pm -H_\alpha^{0,\pm} \, \nabla \cdot H^\pm +\p_{y_\alpha} q^\pm \, = \, \bF_\alpha^\pm \, ,& \alpha=1,2,3\, ,\\
\big( \p_t  +u_j^{0,\pm} \, \p_{y_j} \big) H_\alpha^\pm -u_\alpha^{0,\pm} \, \nabla \cdot H^\pm 
-H_j^{0,\pm} \, \p_{y_j} u_\alpha^\pm +H_\alpha^{0,\pm} \, \nabla \cdot u^\pm \, = \, \bF_{3+\alpha}^\pm \, ,& \alpha=1,2,3\, ,\\
\nabla \cdot u^\pm \, = \, \bF_7^\pm \, ,& \\
\nabla \cdot H^\pm \, = \, \bF_8^\pm \, .&
\end{cases}
\end{equation}
The top, bottom and intermediate boundary conditions are:
\begin{equation}
\label{systeme_lent_hautbas}
u_3^\pm|_{\Gamma^\pm} \, = \, H_3^\pm|_{\Gamma^\pm} \, = \, 0 \, ,
\end{equation}
\begin{equation}
\label{systeme_lent_saut}
\begin{cases}
u_3^\pm|_{y_3=0} \, = \, \bG_1^\pm \, ,& \\
H_3^\pm|_{y_3=0} \, = \, \bG_2^\pm \, ,& \\
q^+|_{y_3=0} -q^-|_{y_3=0} \, = \, \bG_5 \, ,
\end{cases}
\end{equation}
where the source terms $\bF^\pm,\bF_8^\pm$ considered in \eqref{systeme_lent} belong to $H^\infty([0,T] \times \Omega_0^\pm)$, 
and the source terms $\bG_1^\pm,\bG_2^\pm,\bG_5$ in \eqref{systeme_lent_saut} belong to $H^\infty([0,T] \times \bT^2)$. Our goal is 
to understand the compatibility conditions on these source terms and on the initial data for $(u^\pm,H^\pm)$ that ensure the existence 
and uniqueness of a solution in $H^\infty([0,T] \times \Omega_0^\pm)$ to \eqref{systeme_lent}, \eqref{systeme_lent_hautbas}, 
\eqref{systeme_lent_saut}. Our result is the following.

\begin{proposition}[Solvability of the linearized current vortex sheet system]
\label{prop-systemelent}
Let $u_0^\pm,H_0^\pm \in H^\infty(\Omega_0^\pm)$. The problem \eqref{systeme_lent}, \eqref{systeme_lent_hautbas}, \eqref{systeme_lent_saut} 
supplemented with the initial conditions:
$$
(u^\pm,H^\pm)|_{t=0} \, = \, (u_0^\pm,H_0^\pm) \, ,
$$
has a unique solution in $H^\infty([0,T] \times \Omega_0^\pm)$ if and only if the following conditions hold:
\begin{itemize}
 \item $\bF_6^\pm|_{\Gamma^\pm} = 0$ (compatibility at the top and bottom boundaries),
 \item $\bF_6^\pm|_{\Gamma_0} = (\p_t  +u_j^{0,\pm} \, \p_{y_j}) \bG_2^\pm -H_j^{0,\pm} \, \p_{y_j} \bG_1^\pm$ (compatibility on $\Gamma_0$),
 \item $\p_t \bF_8^\pm = \p_{y_\alpha} \bF_{3+\alpha}^\pm$ (compatibility for the divergence of the magnetic field),
 \item $\nabla \cdot u_0^\pm \, = \, \bF_7^\pm|_{t=0}$ and $\nabla \cdot H_0^\pm \, = \, \bF_8^\pm|_{t=0}$ (compatibility for the 
 divergence of the initial conditions),
 \item $u_{0,3}^\pm|_{\Gamma^\pm} = H_{0,3}^\pm|_{\Gamma^\pm} = 0$, $u_{0,3}^\pm|_{\Gamma_0} = \bG_1^\pm|_{t=0}$, 
 $H_{0,3}^\pm|_{\Gamma_0} = \bG_2^\pm|_{t=0}$ (compatibility of the initial conditions at the boundaries),
 \item for all time $t\in [0,T]$, the Laplace problem:
\begin{equation}
\label{pb_laplace_couple}
\begin{cases}
-\Delta \, q^\pm \, = \, -\p_{y_\alpha} \bF_\alpha^\pm +(\p_t  +2\, u_j^{0,\pm} \, \p_{y_j}) \bF_7^\pm -2\, H_j^{0,\pm} \, \p_{y_j} \bF_8^\pm \, ,& 
\text{\rm in $\Omega_0^\pm$,} \\
\p_{y_3} q^\pm|_{\Gamma^\pm} \, = \, \bF_3^\pm|_{\Gamma^\pm} \, ,& \\
\p_{y_3} q^\pm|_{\Gamma_0} \, = \, \bF_3^\pm|_{\Gamma_0} -(\p_t  +u_j^{0,\pm} \, \p_{y_j}) \bG_1^\pm +H_j^{0,\pm} \, \p_{y_j} \bG_2^\pm\, ,& \\
q^+|_{\Gamma_0} -q^-|_{\Gamma_0} \, = \, \bG_5 \, ,&
\end{cases}
\end{equation}
 has a solution $(q^+,q^-) \in H^\infty(\Omega_0^+) \times H^\infty(\Omega_0^-)$.
\end{itemize}
\end{proposition}

\noindent Let us observe that the elliptic problem \eqref{pb_laplace_couple} looks (and actually is) over-determined, as will become 
perhaps even more clear in the following Paragraph. Indeed, on each domain $\Omega_0^\pm$, we impose Neumann boundary 
conditions for $q^\pm$ but we also impose a jump condition across $\Gamma_0$. As in \cite{SWZ}, this is precisely this coupling 
equation on $\Gamma_0$ which will govern the evolution of the slow mean of the front profile.

\begin{proof}[Proof of Proposition \ref{prop-systemelent}]
The proof is quite simple so we sketch it quickly. That the conditions stated in Proposition \ref{prop-systemelent} are necessary 
follows from classical manipulations on \eqref{systeme_lent_exp}. For instance, taking the divergence of the first three evolution 
equations yields the Laplace problem for the total pressure, and taking the divergence for the evolution equations for the magnetic 
field yields the compatibility condition $\p_t \bF_8^\pm = \p_{y_\alpha} \bF_{3+\alpha}^\pm$. The other conditions are rather 
straightforward.

Let us therefore now assume that all conditions stated in Proposition \ref{prop-systemelent} are satisfied. We wish to show 
that the problem \eqref{systeme_lent}, \eqref{systeme_lent_hautbas}, \eqref{systeme_lent_saut} has a unique solution in 
$H^\infty([0,T] \times \Omega_0^\pm)$. Let the total pressure $(q^+,q^-)$ be defined as the solution to \eqref{pb_laplace_couple}. 
(This is precisely part of our assumptions that the latter problem has a solution.) We then determine the velocity and magnetic 
field by solving the hyperbolic system:
$$
\begin{cases}
\big( \p_t  +u_j^{0,\pm} \, \p_{y_j} \big) u_\alpha^\pm -H_j^{0,\pm} \, \p_{y_j} H_\alpha^\pm \, = \, \bF_\alpha^\pm 
-u_\alpha^{0,\pm} \, \bF_7^\pm +H_\alpha^{0,\pm} \, \bF_8^\pm -\p_{y_\alpha} q^\pm \, ,& \alpha=1,2,3\, ,\\
\big( \p_t  +u_j^{0,\pm} \, \p_{y_j} \big) H_\alpha^\pm -H_j^{0,\pm} \, \p_{y_j} u_\alpha^\pm \, = \, \bF_{3+\alpha}^\pm 
+u_\alpha^{0,\pm} \, \bF_8^\pm -H_\alpha^{0,\pm} \, \bF_7^\pm \, ,& \alpha=1,2,3\, ,
\end{cases}
$$
with initial data:
$$
(u^\pm,H^\pm)|_{t=0} \, = \, (u_0^\pm,H_0^\pm) \, .
$$
There is no boundary condition required on $\partial \Omega_0^\pm$ because the vector fields $u_j^{0,\pm} \, \p_{y_j}$ and $H_j^{0,\pm} 
\, \p_{y_j}$ are both tangent to the boundaries $\Gamma^\pm$ and $\Gamma_0$ (which means that the normal variable $y_3$ is a 
parameter). At this stage, we have determined $(u^\pm,H^\pm,q^\pm)$ and we need to verify that all equations in \eqref{systeme_lent}, 
\eqref{systeme_lent_hautbas}, \eqref{systeme_lent_saut} are satisfied. Applying the divergence operator and making use of the elliptic 
system satisfied by the total pressure, we get:
$$
\begin{cases}
\big( \p_t  +u_j^{0,\pm} \, \p_{y_j} \big) \nabla \cdot u^\pm -H_j^{0,\pm} \, \p_{y_j} \nabla \cdot H^\pm \, = \, 
\big( \p_t  +u_j^{0,\pm} \, \p_{y_j} \big) \bF_7^\pm -H_j^{0,\pm} \, \p_{y_j} \bF_8^\pm \, ,& \\
\big( \p_t  +u_j^{0,\pm} \, \p_{y_j} \big) \nabla \cdot H^\pm -H_j^{0,\pm} \, \p_{y_j} \nabla \cdot u^\pm \, = \, 
\big( \p_t  +u_j^{0,\pm} \, \p_{y_j} \big) \bF_8^\pm -H_j^{0,\pm} \, \p_{y_j} \bF_7^\pm \, .& 
\end{cases}
$$
Since the initial conditions $u_0^\pm,H_0^\pm$ have compatible divergence with $\bF_7^\pm|_{t=0}$, $\bF_8^\pm|_{t=0}$, we get by 
the energy method:
$$
\forall \, t \in [0,T] \, ,\quad \nabla \cdot u^\pm \, = \, \bF_7^\pm \, ,\quad \nabla \cdot u^\pm \, = \, \bF_8^\pm \, ,
$$
which already proves that we have indeed solved \eqref{systeme_lent}.

Restricting to the top and bottom boundaries $\Gamma^\pm$ and using the compatibility condition on the source terms $\bF_6^\pm$ 
as well as the Neumann condition satisfied by the total pressure, we get:
$$
\begin{cases}
\big( \p_t  +u_j^{0,\pm} \, \p_{y_j} \big) u_3^\pm|_{\Gamma^\pm} -H_j^{0,\pm} \, \p_{y_j} H_3^\pm|_{\Gamma^\pm} \, = \, 0 \, ,& \\
\big( \p_t  +u_j^{0,\pm} \, \p_{y_j} \big) H_3^\pm|_{\Gamma^\pm} -H_j^{0,\pm} \, \p_{y_j} u_3^\pm|_{\Gamma^\pm} \, = \, 0 \, .& 
\end{cases}
$$
By the compatibility conditions satisfied by the initial data, we get \eqref{systeme_lent_hautbas}. The verification of the first four 
conditions in \eqref{systeme_lent_saut} is in the same spirit. (The fifth condition in \eqref{systeme_lent_saut} is already included 
in \eqref{pb_laplace_couple} and is therefore satisfied.) To complete the proof of Proposition \ref{prop-systemelent}, it remains 
to examine the smoothness of the solution. From all previous arguments, we see that it is sufficient to prove that the solution 
to the coupled Laplace problem \eqref{pb_laplace_couple} can be chosen smoothly with respect to the time variable $t \in [0,T]$. 
This is true indeed because the solution to \eqref{pb_laplace_couple}, if it exists (which depends on compatibility conditions for 
the source terms in \eqref{pb_laplace_couple}), can be explicitly determined thanks to the Fourier series decomposition with 
respect to the variable $y'$. We do not pursue this issue here since we shall perform such a Fourier series decomposition in the 
following Paragraph in order to derive the evolution equation that governs the slow mean of the front profile.
\end{proof}

\subsection{Determining the slow mean of the front profile}

In this Paragraph, we explain why part of the result of Proposition \ref{prop-systemelent} determines the evolution of the mean 
(with respect to $\theta$) of the front profile $\psi^{\, m+1}$. Indeed, we focus on the problem \eqref{slow_mean_m+1_edp}, 
\eqref{slow_mean_m+1_saut}, \eqref{slow_mean_m+1_hautbas} (the source terms in \eqref{slow_mean_m+1_saut} correspond 
to the right hand side of \eqref{slow_mean_m+1_saut-bis}). Applying Proposition \ref{prop-systemelent}, a necessary condition 
for the solvability of \eqref{slow_mean_m+1_edp}, \eqref{slow_mean_m+1_saut}, \eqref{slow_mean_m+1_hautbas} is that the 
coupled Laplace problem:
\begin{equation}
\label{pb_laplace_m+1}
\begin{cases}
-\Delta \, \widehat{\uq}^{\, m+1,\pm}(0) \, = \, -\p_{y_\alpha} \bF_\alpha^{\, m,\pm} +(\p_t  +2\, u_j^{0,\pm} \, \p_{y_j}) \bF_7^{\, m,\pm} 
-2 \, H_j^{0,\pm} \, \p_{y_j} \bF_8^{\, m,\pm} \, , & \text{\rm in $\Omega_0^\pm$,} \\
\p_{y_3} \widehat{\uq}^{\, m+1,\pm}(0)|_{\Gamma^\pm} \, = \, \bF_3^{\, m,\pm}|_{\Gamma^\pm} \, , & \\
\p_{y_3} \widehat{\uq}^{\, m+1,\pm}(0)|_{\Gamma_0} \, = \, 
\bF_3^{\, m,\pm}|_{\Gamma_0} -(\p_t  +u_j^{0,\pm} \, \p_{y_j}) \bG_1^{\, m,\pm} +H_j^{0,\pm} \, \p_{y_j} \bG_2^{\, m,\pm} \, , & \\
\widehat{\uq}^{\, m+1,+}(0)|_{\Gamma_0} -\widehat{\uq}^{\, m+1,-}(0)|_{\Gamma_0} \, = \, \bG_5^m \, , &
\end{cases}
\end{equation}
has a solution in $H^\infty(\Omega_0^+) \times H^\infty(\Omega_0^-)$. We therefore examine the solvability of \eqref{pb_laplace_m+1} 
by splitting the analysis between the nonzero tangential Fourier modes and the zero tangential Fourier mode (that is, the mean with 
respect to $y'$).
\bigskip

$\bullet$ \underline{Solving the coupled Laplace problem for the nonzero Fourier modes}.

Due to the very simple geometry of the domains $\Omega_0^\pm =\bT^2 \times I^\pm$, there is no difficulty in solving \eqref{pb_laplace_m+1}. 
Indeed, we decompose the functions $\widehat{\uq}^{\, m+1,\pm}(0)$ into Fourier series in $y'$:
$$
\widehat{\uq}^{\, m+1,\pm}(0) \, = \, \sum_{j' \in \bZ^2} {\bf c}_{j'} \big( \widehat{\uq}^{\, m+1,\pm}(0) \big) (y_3) \, {\rm e}^{i \, j' \cdot y'} \, ,
$$
where ${\bf c}_{j'}$ refers here to the $j'$ Fourier coefficient with respect to $y' \in \bT^2$. We warn the reader that $j'$ belongs to $\bZ^2$ 
while we used before the notation ${\bf c}_0$ for the mean with respect to $\theta$ on $\bT$. However, we have thought it more convenient 
to keep the same notation for the Fourier coefficient with respect to either $y'$ or $\theta$ and hope that it will not create any confusion.

For the nonzero Fourier modes $j' \neq 0$, the Laplace operator $\Delta$ is uniformly elliptic, hence the Neumann boundary conditions 
on $\Gamma^\pm$ and $\Gamma_0$ uniquely determine ${\bf c}_{j'} (\widehat{\uq}^{\, m+1,\pm}(0))$. We obtain the expressions:
\begin{multline*}
{\bf c}_{j'} \big( \widehat{\uq}^{\, m+1,+}(0) \big) \\
= \, \nu_{j'}^+ \, {\rm e}^{|j'| \, y_3} +\omega_{j'}^+ \, {\rm e}^{-|j'| \, y_3} 
+\dfrac{1}{2\, |j'|} \, \left( \int_0^{y_3} {\rm e}^{-|j'| \, (y_3-z)} \, {\bf c}_{j'} (\cF^{\, m,+}) \, {\rm d}z  
+\int_{y_3}^1 {\rm e}^{|j'| \, (y_3-z)} \, {\bf c}_{j'} (\cF^{\, m,+}) \, {\rm d}z \right) \, ,
\end{multline*}
and
\begin{multline*}
{\bf c}_{j'} \big( \widehat{\uq}^{\, m+1,-}(0) \big) \\
= \, \nu_{j'}^- \, {\rm e}^{|j'| \, y_3} +\omega_{j'}^- \, {\rm e}^{-|j'| \, y_3} 
+\dfrac{1}{2\, |j'|} \, \left( \int_{-1}^{y_3} {\rm e}^{-|j'| \, (y_3-z)} \, {\bf c}_{j'} (\cF^{\, m,-}) \, {\rm d}z  
+\int_{y_3}^0 {\rm e}^{|j'| \, (y_3-z)} \, {\bf c}_{j'} (\cF^{\, m,-}) \, {\rm d}z \right) \, ,
\end{multline*}
where $\cF^{\, m,\pm}$ stand for the source terms in the Laplace equations of \eqref{pb_laplace_m+1}, that is:
\begin{equation*}
\cF^{\, m,\pm} \, := \, 
-\p_{y_\alpha} \bF_\alpha^{\, m,\pm} +(\p_t  +2\, u_j^{0,\pm} \, \p_{y_j}) \bF_7^{\, m,\pm} -2 \, H_j^{0,\pm} \, \p_{y_j} \bF_8^{\, m,\pm} \, ,
\end{equation*}
and the coefficients $\nu_{j'}^\pm,\omega_{j'}^\pm$ are obtained by solving the (invertible) linear system:
\begin{align}
|j'| \, {\rm e}^{|j'|} \, \nu_{j'}^+ -|j'| \, {\rm e}^{-|j'|} \, \omega_{j'}^+ \, = \, & \, 
\dfrac{1}{2} \, \int_0^1 {\rm e}^{-|j'| \, (1-z)} \, {\bf c}_{j'} (\cF^{\, m,+}) \, {\rm d}z +{\bf c}_{j'} (\bF_3^{\, m,+})(1) \, ,\notag \\
|j'| \, \nu_{j'}^+ -|j'| \, \omega_{j'}^+ \, = \, & \, 
-\dfrac{1}{2} \, \int_0^1 {\rm e}^{-|j'| \, z} \, {\bf c}_{j'} (\cF^{\, m,+}) \, {\rm d}z +{\bf c}_{j'} (\bF_3^{\, m,+})(0) \notag \\
& \, -{\bf c}_{j'} \big( (\p_t  +u_j^{0,+} \, \p_{y_j}) \bG_1^{\, m,+} -H_j^{0,+} \, \p_{y_j} \bG_2^{\, m,+} \big) \, ,\notag \\
|j'| \, \nu_{j'}^- -|j'| \, \omega_{j'}^- \, = \, & \, 
\dfrac{1}{2} \, \int_{-1}^0 {\rm e}^{|j'| \, z} \, {\bf c}_{j'} (\cF^{\, m,-}) \, {\rm d}z +{\bf c}_{j'} (\bF_3^{\, m,-})(0) \label{coeffs_pression_lente} \\
& \, -{\bf c}_{j'} \big( (\p_t  +u_j^{0,-} \, \p_{y_j}) \bG_1^{\, m,-} -H_j^{0,-} \, \p_{y_j} \bG_2^{\, m,-} \big) \, ,\notag \\
|j'| \, {\rm e}^{-|j'|} \, \nu_{j'}^+ -|j'| \, {\rm e}^{|j'|} \, \omega_{j'}^+ \, = \, & \, 
-\dfrac{1}{2} \, \int_{-1}^0 {\rm e}^{-|j'| \, (1+z)} \, {\bf c}_{j'} (\cF^{\, m,-}) \, {\rm d}z +{\bf c}_{j'} (\bF_3^{\, m,-})(-1) \, .\notag
\end{align}
Since we have the expression of ${\bf c}_{j'} (\widehat{\uq}^{\, m+1,\pm}(0))$ on the whole interval $I^\pm$, solving 
\eqref{pb_laplace_m+1} for the nonzero Fourier modes in $y'$ simply amounts to imposing the solvability condition:
$$
\forall \, j' \neq 0 \, ,\quad {\bf c}_{j'} \big( \widehat{\uq}^{\, m+1,+}(0) \big)(0) -{\bf c}_{j'} \big( \widehat{\uq}^{\, m+1,-}(0) \big)(0) 
\, = \, {\bf c}_{j'} \big( \bG_5^m \big) \, .
$$
Solving the above linear system \eqref{coeffs_pression_lente} and restricting to $y_3=0$, we thus find that solving the coupled 
Laplace problem \eqref{pb_laplace_m+1} for the nonzero Fourier modes in $y'$ is possible if and only if the source terms satisfy 
the relation:
\begin{multline}
\label{eq_onde_front_m+1}
{\bf c}_{j'} \Big( \big( \p_t  +u_j^{0,+} \, \p_{y_j} \big) \bG_1^{\, m,+} +\big( \p_t  +u_j^{0,-} \, \p_{y_j} \big) \bG_1^{\, m,-} 
-H_j^{0,+} \, \p_{y_j} \bG_2^{\, m,+} -H_j^{0,-} \, \p_{y_j} \bG_2^{\, m,-} \Big) \\
= \, |j'| \,\tanh (|j'|) \, {\bf c}_{j'} \big( \bG_5^m \big) +{\bf c}_{j'} (\bF_3^{\, m,+})(0) +{\bf c}_{j'} (\bF_3^{\, m,-})(0) 
-\dfrac{1}{\cosh (|j'|)} \, \big( {\bf c}_{j'} (\bF_3^{\, m,+})(1) +{\bf c}_{j'} (\bF_3^{\, m,-})(-1) \big) \\
-\int_0^1 \dfrac{\cosh (|j'|(1-z))}{\cosh (|j'|)} \, {\bf c}_{j'} (\cF^{\, m,+}) \, {\rm d}z 
+\int_{-1}^0 \dfrac{\cosh (|j'|(1+z))}{\cosh (|j'|)} \, {\bf c}_{j'} (\cF^{\, m,-}) \, {\rm d}z \, .
\end{multline}

It remains to make more explicit the equation \eqref{eq_onde_front_m+1} which we are going to show to be a wave type 
equation on $\Gamma_0$ for (the mean free part of) the front profile $\widehat{\psi}^{\, m+1}(0)$. Indeed, we recall that 
the source terms $\bG_1^{\, m,\pm}$, $\bG_2^{\, m,\pm}$ are defined, see \eqref{slow_mean_m+1_saut-bis}, by:
$$
\bG_1^{\, m,\pm} \, := \, \widehat{G}_1^{\, m,\pm}(0) -\widehat{u}_{3,\star}^{\, m+1,\pm}(0)|_{y_3=Y_3=0} \, ,\quad 
\bG_2^{\, m,\pm} \, := \, \widehat{G}_2^{\, m,\pm}(0) -\widehat{H}_{3,\star}^{\, m+1,\pm}(0)|_{y_3=Y_3=0} \, ,
$$
with $G_1^{\, m,\pm}$, $G_2^{\, m,\pm}$ defined in \eqref{s3-def_G_1^m,pm}, \eqref{s3-def_G_2^m,pm}. Let us recall that at 
this stage, the only unknown quantity is $\widehat{\psi}^{\, m+1}(0)$ which enters the expression of $\widehat{G}_1^{\, m,\pm}(0)$ 
and $\widehat{G}_2^{\, m,\pm}(0)$. We can thus rewrite \eqref{eq_onde_front_m+1} as:
\begin{multline}
\label{eq_onde_front_m+1'}
\big( \p_t  +u_j^{0,+} \, \p_{y_j} \big) \, \big( \p_t  +u_{j'}^{0,+} \,\p_{y_{j'}} \big) \Psi^{\, m+1} 
+\big( \p_t  +u_j^{0,-} \, \p_{y_j} \big) \, \big( \p_t  +u_{j'}^{0,-} \,\p_{y_{j'}} \big) \Psi^{\, m+1} \\
-H_j^{0,+} \, H_{j'}^{0,+} \, \p_{y_j}\p_{y_{j'}} \Psi^{\, m+1} -H_j^{0,-} \, H_{j'}^{0,-} \, \p_{y_j}\p_{y_{j'}} \Psi^{\, m+1} \, = \, \mathfrak{G}^m \, ,
\end{multline}
where $\Psi^{\, m+1}$ denotes the mean free part (in $y'$) of the slow mean (in $\theta$) $\widehat{\psi}^{\, m+1}(0)$:
$$
\Psi^{\, m+1} (t,y') \, := \, \sum_{j' \in \bZ^2 \setminus \{ 0\}} {\bf c}_{j'} \big( \widehat{\psi}^{\, m+1}(0) \big) \, {\rm e}^{i \, j' \cdot y'} \, ,
$$
and the source term $\mathfrak{G}^m$ on the right hand side of \eqref{eq_onde_front_m+1'} is already determined in terms 
of all previous profiles. The exact expression of $\mathfrak{G}^m$ is useless. The only important property to keep in mind is 
that $\mathfrak{G}^m$ belongs to $H^\infty([0,T] \times \bT^2)$ and has zero mean with respect to $y'$.

The wave like operator on the left hand side of \eqref{eq_onde_front_m+1'} has already been highlighted in \cite{SWZ}. 
Unsurprisingly, its symbol corresponds to the Lopatinskii determinant \eqref{s3-def_det_lop}. Observe also that several 
terms in \eqref{eq_onde_front_m+1} correspond to the action of the Dirichlet-Neumann operator (this is recognized by 
the symbol $|j'| \,\tanh (|j'|)$), which is also reminiscent of several terms ocurring in the analysis of \cite{SWZ}. The 
main point is that, under the assumption \eqref{s3-hyp_stab_nappe_plane}, the wave operator on the left hand side of 
\eqref{eq_onde_front_m+1'} is \emph{hyperbolic} in the time direction. We therefore easily get the following result:

\begin{lemma}
\label{lem-moyenne-front}
There exists a unique function $\Psi^{\, m+1} \in H^\infty([0,T] \times \bT^2)$ with zero mean with respect to $y'$ that satisfies 
\eqref{eq_onde_front_m+1'} together with:
$$
\Psi^{\, m+1}|_{t=0} \, = \, 0 \, ,\quad \p_t \Psi^{\, m+1}|_{t=0} \, = \, 0 \, .
$$
\end{lemma}

\noindent Let us recall that solving \eqref{eq_onde_front_m+1'} amounts to solving the original coupled Laplace problem 
\eqref{pb_laplace_m+1} for the nonzero Fourier modes. At this stage, we have thus fixed the mean free part in $y'$ of 
$\widehat{\psi}^{\, m+1}(0)$ and ensured part of the solvability of \eqref{pb_laplace_m+1}.

Let us observe that the prescription of initial conditions for \eqref{eq_onde_front_m+1'} is arbitrary. Once again, we have 
chosen in the statement of Theorem \ref{thm_principal} to stick to the easiest possible case, but we could equally well 
consider more general initial data for $\widehat{\psi}^{\, m+1}(0)$ than the ones in \eqref{s3-cond_init_profils_psi_0^m}. 
We could also impose nonzero initial data for the oscillating modes in $y'$ of $\p_t \widehat{\psi}^{\, m+1}(0)$.
\bigskip

$\bullet$ \underline{Solving the coupled Laplace problem for the mean}.

It remains to determine the mean on $\bT^2$ of $\widehat{\psi}^{\, m+1}(0)$, which is a function of time only. Integrating 
\eqref{pb_laplace_m+1} on $\bT^2$ and using some obvious cancellations, we need to find a solution to the problem:
\begin{equation}
\label{pb_laplace_m+1_mean}
\begin{cases}
-\Delta \, p^\pm \, = \, -\p_{y_3} \, {\bf c}_0 \, (\bF_3^{\, m,\pm}) +\p_t \, {\bf c}_0 \, (\bF_7^{\, m,\pm}) \, , & \text{\rm in $\Omega_0^\pm$,} \\
\p_{y_3} p^\pm (\pm 1) \, = \, {\bf c}_0 \, (\bF_3^{\, m,\pm}) (\pm 1) \, , & \\
\p_{y_3} p^\pm (0) \, = \, {\bf c}_0 \, (\bF_3^{\, m,\pm})(0) -\p_t \, {\bf c}_0 \, (\bG_1^{\, m,\pm}) \, , & \\
p^+(0) -p^-(0) \, = \, {\bf c}_0 \, (\bG_5^m) \, , &
\end{cases}
\end{equation}
where $p^\pm$ is a short notation for ${\bf c}_0 \, (\widehat{\uq}^{\, m+1,\pm}(0))$, and the reader should be careful that in 
\eqref{pb_laplace_m+1_mean}, ${\bf c}_0$ refers to the zero Fourier coefficient with respect to $y'$ (the source terms to which 
we apply ${\bf c}_0$ in \eqref{pb_laplace_m+1_mean} are themselves means with respect to the fast variable $\theta$ of some 
quantities !).

The solvability of \eqref{pb_laplace_m+1_mean} is submitted to the fulfillment of the Fredholm condition:
\begin{equation}
\label{fredholm_pm}
\dfrac{{\rm d}}{{\rm d}t} \, \left( \int_{\Omega_0^\pm} \bF_7^{\, m,\pm} \, {\rm d}y \pm \int_{\Gamma_0} \bG_1^{\, m,\pm} \, {\rm d}y' \right) 
\, = \, 0 \, .
\end{equation}
Let us make it clear that the condition \eqref{fredholm_pm} should be satisfied on either side of the current vortex sheet, that is 
both for the $+$ and for the $-$ sides. In Lemma \ref{lem-fredholm_pb_laplace} below, we examine the quantities involved in 
\eqref{fredholm_pm} and derive a necessary and sufficient condition on the front profile $\psi^{\, m+1}$ for the verification of 
\eqref{fredholm_pm}.

\begin{lemma}
\label{lem-fredholm_pb_laplace}
The source terms in \eqref{slow_mean_m+1_saut}, \eqref{slow_mean_m+1_edp} satisfy:
\begin{align*}
\forall \, t \in [0,T] \, ,\quad 
\int_{\Omega_0^\pm} \bF_7^{\, m,\pm} \, {\rm d}y \pm \int_{\Gamma_0} \bG_1^{\, m,\pm} \, {\rm d}y' \, = \, & \, 
\pm \int_{\Gamma_0} \p_t \widehat{\psi}^{\, m+1}(0) \, {\rm d}y' \, ,\\
\int_{\Omega_0^\pm} \bF_8^{\, m,\pm} \, {\rm d}y \pm \int_{\Gamma_0} \bG_2^{\, m,\pm} \, {\rm d}y' \, = \, & \, 0 \, .\\
\end{align*}
Consequently the Fredholm condition \eqref{fredholm_pm} holds on either side of the current vortex sheet if and only if:
$$
\dfrac{{\rm d}^2}{{\rm d}t^2} \,  \int_{\Gamma_0 \times \bT} \psi^{\, m+1} (t,y',\theta) \, {\rm d}y' \, {\rm d}\theta \, = \, 0 \, . 
$$
\end{lemma}

\begin{proof}[Proof of Lemma \ref{lem-fredholm_pb_laplace}]
We give the proof for the relation between $\bF_7^{\, m,\pm}$ and $\bG_1^{\, m,\pm}$, and leave the other (easier) case for the 
magnetic field to the interested reader. Let us recall the expressions\footnote{We warn the reader that in all the proof of Lemma 
\ref{lem-fredholm_pb_laplace}, the notation ${\bf c}_0$ refers to the zero Fourier coefficient with respect to the fast variable $\theta$. 
When taking the mean on $\bT^2$ with respect to $y'$, we simply write down the integral.}:
\begin{align*}
\bF_7^{\, m,\pm} \, = \, {\bf c}_0 \, \Big\{ \, 
\sum_{\ell_1+\ell_2+\ell_3=m+2} \chi^{[\ell_1]} \, \p_\theta \psi^{\ell_2} \, \p_{y_3} \xi_j \, \uu_j^{\ell_3,\pm} 
& \, +\sum_{\ell_1+\ell_2+\ell_3=m+1} \chi^{[\ell_1]} \, \p_{y_j} \psi^{\ell_2} \, \p_{y_3} \uu_j^{\ell_3,\pm} \\
& \, +\sum_{\ell_1+\ell_2+\ell_3=m+1} \dot{\chi}^{[\ell_1]} \, \psi^{\ell_2} \, \p_{y_3} \uu_3^{\ell_3,\pm} \, \Big\} \, ,
\end{align*}
see \eqref{defsource8_slow_mean_m+1} for the similiar expression associated with the divergence of the magnetic field, and:
\begin{align*}
\bG_1^{\, m,\pm} \, = \, & \, \big( \p_t +u_j^{0,\pm} \, \p_{y_j} \big) \widehat{\psi}^{\, m+1}(0) \\
& \, +{\bf c}_0 \, \Big\{ \, \sum_{\substack{\ell_1+\ell_2=m+2 \\ \ell_2 \ge 1}} \p_\theta \psi^{\ell_1} \, \xi_j \, u_j^{\ell_2,\pm} 
+\sum_{\substack{\ell_1+\ell_2=m+1 \\ \ell_2 \ge 1}} \p_{y_j} \psi^{\ell_1} \, u_j^{\ell_2,\pm} \Big\}|_{y_3=Y_3=0} 
-\widehat{u}_{3,\star}^{\, m+1,\pm}(0)|_{y_3=Y_3=0} \, .
\end{align*}
Splitting inside the expression of $\bG_1^{\, m,\pm}$ between the residual and surface wave components, we obtain the decomposition:
$$
\int_{\Omega_0^\pm} \bF_7^{\, m,\pm} \, {\rm d}y \pm \int_{\Gamma_0} \bG_1^{\, m,\pm} \, {\rm d}y' \, = \, 
\pm \int_{\Gamma_0} \p_t \widehat{\psi}^{\, m+1}(0) \, {\rm d}y' +\mathfrak{I}^{\, m,\pm}_1 \pm \mathfrak{J}_\star^{\, m,\pm} \, ,
$$
where we have used the following notation:
\begin{align}
\label{def-frakImpm}
\mathfrak{I}_1^{\, m,\pm} \, := \, \int_{\Omega_0^\pm \times \bT} \, & \, 
\sum_{\ell_1+\ell_2+\ell_3=m+2} \chi^{[\ell_1]} \, \p_\theta \psi^{\ell_2} \, \xi_j \, \p_{y_3} \uu_j^{\ell_3,\pm} 
+\sum_{\ell_1+\ell_2+\ell_3=m+1} \chi^{[\ell_1]} \, \p_{y_j} \psi^{\ell_2} \, \p_{y_3} \uu_j^{\ell_3,\pm} \\
& \, -\p_{y_3} \Big( 
\sum_{\substack{\ell_1+\ell_2+\ell_3=m+2 \\ \ell_3 \ge 1}} \chi^{[\ell_1]} \, \p_\theta \psi^{\ell_2} \, \xi_j \, \uu_j^{\ell_3,\pm} 
+\sum_{\substack{\ell_1+\ell_2+\ell_3=m+1 \\ \ell_3 \ge 1}} \chi^{[\ell_1]} \, \p_{y_j} \psi^{\ell_2} \, \uu_j^{\ell_3,\pm} \Big) \notag \\
& \, +\sum_{\ell_1+\ell_2+\ell_3=m+1} \dot{\chi}^{[\ell_1]} \, \psi^{\ell_2} \, \p_{y_3} \uu_3^{\ell_3,\pm} \, {\rm d}y \, {\rm d}\theta \, ,\notag
\end{align}
and
\begin{equation}
\label{def-frakJmpm}
\mathfrak{J}_\star^{\, m,\pm} \, := \, \int_{\Gamma_0} \Big( 
{\bf c}_0 \, \Big\{ \, \sum_{\substack{\ell_1+\ell_2=m+2 \\ \ell_2 \ge 1}} \p_\theta \psi^{\ell_1} \, \xi_j \, u_{j,\star}^{\ell_2,\pm} 
+\sum_{\substack{\ell_1+\ell_2=m+1 \\ \ell_2 \ge 1}} \p_{y_j} \psi^{\ell_1} \, u_{j,\star}^{\ell_2,\pm} \Big\} -\widehat{u}_{3,\star}^{\, m+1,\pm}(0) 
\Big) \Big|_{y_3=Y_3=0} \, {\rm d}y' \, .
\end{equation}
To complete the proof of Lemma \ref{lem-fredholm_pb_laplace}, we see that it is sufficient to prove that the quantities 
$\mathfrak{I}^{\, m,\pm}_1$ and $\mathfrak{J}_\star^{\, m,\pm}$ defined in \eqref{def-frakImpm} and \eqref{def-frakJmpm} 
vanish, which is done below separately for each of these two quantities.
\bigskip

$\bullet$ \underline{The surface wave integral}. Let $y_3 \in I^\pm$ satisfy $|y_3|<1/3$, and let $Y_3 \in \bR^\pm$. We then define:
$$
\mathfrak{J}^{\, m,\pm}(y_3,Y_3) \, := \, \int_{\bT^2} 
{\bf c}_0 \, \Big\{ \, \sum_{\substack{\ell_1+\ell_2=m+2 \\ \ell_2 \ge 1}} \p_\theta \psi^{\ell_1} \, \xi_j \, u_{j,\star}^{\ell_2,\pm} 
+\sum_{\substack{\ell_1+\ell_2=m+1 \\ \ell_2 \ge 1}} \p_{y_j} \psi^{\ell_1} \, u_{j,\star}^{\ell_2,\pm} \Big\} -\widehat{u}_{3,\star}^{\, m+1,\pm}(0) 
\, {\rm d}y' \, ,
$$
in such a way that the definition of $\mathfrak{J}_\star^{\, m,\pm}$ in \eqref{def-frakJmpm} coincides with $\mathfrak{J}^{\, m,\pm}(0,0)$. 
Thanks to the exponential decay at infinity of profiles in $S_\star^\pm$, we easily see that for any given $y_3$, the function 
$\mathfrak{J}^{\, m,\pm}(y_3,\cdot)$ decays exponentially at infinity. We now compute the partial derivative $\p_{Y_3} \mathfrak{J}^{\, m,\pm} 
(y_3,Y_3)$. Using the equation:
$$
\p_{Y_3} \widehat{u}_{3,\star}^{\, m+1,\pm}(0) \, = \, {\bf c}_0 \, \big( F_{7,\star}^{\, m,\pm} \big) \, ,
$$
and the expression of $F_{7,\star}^{\, m,\pm}$, we obtain:
\begin{align*}
\p_{Y_3} \mathfrak{J}^{\, m,\pm}(y_3,Y_3) \, = \, & \, 
\int_{\bT^2} {\bf c}_0 \, \Big\{ \, \sum_{\ell_1+\ell_2=m+2} \p_\theta \psi^{\ell_1} \, \xi_j \, \p_{Y_3} u_{j,\star}^{\ell_2,\pm} 
+\sum_{\ell_1+\ell_2=m+1} \p_{y_j} \psi^{\ell_1} \, \p_{Y_3} u_{j,\star}^{\ell_2,\pm} -F_{7,\star}^{\, m,\pm} \Big\} \, {\rm d}y' \\
= \, & \, \int_{\bT^2 \times \bT} \p_{y_3} u_{3,\star}^{\, m,\pm}
-\sum_{\ell_1+\ell_2=m+1} \p_\theta \psi^{\ell_1} \, \xi_j \, \p_{y_3} u_{j,\star}^{\ell_2,\pm} 
-\sum_{\ell_1+\ell_2=m} \p_{y_j} \psi^{\ell_1} \, \p_{y_3} u_{j,\star}^{\ell_2,\pm} \, {\rm d}y' \, {\rm d}\theta \\
= \, & \, -\p_{y_3} \int_{\bT^2 \times \bT} 
\sum_{\substack{\ell_1+\ell_2=m+1 \\ \ell_2 \ge 1}} \p_\theta \psi^{\ell_1} \, \xi_j \, u_{j,\star}^{\ell_2,\pm} 
+\sum_{\substack{\ell_1+\ell_2=m \\ \ell_2 \ge 1}} \p_{y_j} \psi^{\ell_1} \, u_{j,\star}^{\ell_2,\pm} 
-u_{3,\star}^{\, m,\pm} \, {\rm d}y' \, {\rm d}\theta \\
= \, & \, -\p_{y_3} \mathfrak{J}^{\, m-1,\pm}(y_3,Y_3) \, .
\end{align*}
Inductively, we obtain:
$$
\p_{Y_3}^m \mathfrak{J}^{\, m,\pm}(y_3,Y_3) \, = \, (-1)^m \, \p_{y_3}^m \mathfrak{J}^{0,\pm} (y_3,Y_3) \, = \, 0 \, .
$$
Thanks to the exponential decay of $\mathfrak{J}^{\, m,\pm}(y_3,\cdot)$, we obtain $\mathfrak{J}^{\, m,\pm}(y_3,Y_3)=0$ for all $y_3,Y_3$ 
and in particular at $y_3=Y_3=0$. This means that the integral $\mathfrak{J}_\star^{\, m,\pm}$ in \eqref{def-frakJmpm} vanishes.
\bigskip

$\bullet$ \underline{The residual integral}. From the definition \eqref{def-frakImpm}, we have
\begin{align*}
\mathfrak{I}_1^{\, m,\pm} \, = \, \int_{\Omega_0^\pm \times \bT} 
\sum_{\ell_1+\ell_2+\ell_3=m+1} \dot{\chi}^{[\ell_1]} \, \psi^{\ell_2} \, \p_{y_3} \uu_3^{\ell_3,\pm} 
\, & \, -\sum_{\substack{\ell_1+\ell_2+\ell_3=m+2 \\ \ell_3 \ge 1}} 
\p_{y_3} \chi^{[\ell_1]} \, \p_\theta \psi^{\ell_2} \, \xi_j \, \uu_j^{\ell_3,\pm} \\
& \, -\sum_{\substack{\ell_1+\ell_2+\ell_3=m+1 \\ \ell_3 \ge 1}} 
\p_{y_3} \chi^{[\ell_1]} \, \p_\theta \psi^{\ell_2} \, \xi_j \, \uu_j^{\ell_3,\pm} \, {\rm d}y \, {\rm d}\theta \, ,
\end{align*}
and more generally, we define for any integer $n \in \N$:
\begin{align*}
\mathfrak{I}_n^{\, m,\pm} \, := \, \int_{\Omega_0^\pm \times \bT} \, & \, 
\sum_{\ell_1+\cdots+\ell_{2n+1}=m+1} 
\dot{\chi}^{[\ell_1]} \, \cdots \, \dot{\chi}^{[\ell_n]} \, \psi^{\ell_{n+1}} \, \cdots \, \psi^{\ell_{2n}} \, \p_{y_3} \uu_3^{\ell_{2n+1},\pm} \\
\, & \, -n \, \sum_{\substack{\ell_1+\cdots+\ell_{2n+1}=m+2 \\ \ell_{2n+1} \ge 1}} 
\dot{\chi}^{[\ell_1]} \, \cdots \, \dot{\chi}^{[\ell_{n-1}]} \, {\color{ForestGreen} \p_{y_3} \chi^{[\ell_n]}} \, 
\psi^{\ell_{n+1}} \, \cdots \, \psi^{\ell_{2n-1}} \, \p_\theta \psi^{\ell_{2n}} \, \xi_j \, \uu_j^{\ell_{2n+1},\pm} \\
\, & \, -n \, \sum_{\substack{\ell_1+\cdots+\ell_{2n+1}=m+1 \\ \ell_{2n+1} \ge 1}} 
\dot{\chi}^{[\ell_1]} \, \cdots \, \dot{\chi}^{[\ell_{n-1}]} \, {\color{ForestGreen} \p_{y_3} \chi^{[\ell_n]}} \, 
\psi^{\ell_{n+1}} \, \cdots \, \psi^{\ell_{2n-1}} \, \p_{y_j} \psi^{\ell_{2n}} \, \uu_j^{\ell_{2n+1},\pm} \, {\rm d}y \, {\rm d}\theta \, .
\end{align*}
By performing integration by parts and a substitution, we are going to prove the relation $\mathfrak{I}_n^{\, m,\pm} 
=\mathfrak{I}_{n+1}^{\, m,\pm}$ for all $n \ge 1$, which, by choosing $n$ large enough, will imply $\mathfrak{I}_n^{\, m,\pm}=0$. 
We first use the second symmetry formula (Corollary \ref{corB2} in Appendix \ref{appendixB}) for the green terms in the above 
definition of $\mathfrak{I}_n^{\, m,\pm}$ to get\footnote{Note that we have symmetrized some terms that involve tangential derivatives 
of the front profiles.}:
\begin{align*}
\mathfrak{I}_n^{\, m,\pm} \, = \, \int_{\Omega_0^\pm \times \bT} \, & \, 
\sum_{\ell_1+\cdots+\ell_{2n+1}=m+1} 
\dot{\chi}^{[\ell_1]} \, \cdots \, \dot{\chi}^{[\ell_n]} \, \psi^{\ell_{n+1}} \, \cdots \, \psi^{\ell_{2n}} \, \p_{y_3} \uu_3^{\ell_{2n+1},\pm} \\
& \, -\sum_{\substack{\ell_1+\cdots+\ell_{2n+1}=m+2 \\ \ell_{2n+1} \ge 1}} \dot{\chi}^{[\ell_1]} \, \cdots \, \dot{\chi}^{[\ell_n]} \, 
\p_\theta \big( \psi^{\ell_{n+1}} \, \cdots \, \psi^{\ell_{2n}} \big) \, \xi_j \, \uu_j^{\ell_{2n+1},\pm} \\
& \, -\sum_{\substack{\ell_1+\cdots+\ell_{2n+1}=m+1 \\ \ell_{2n+1} \ge 1}} \dot{\chi}^{[\ell_1]} \, \cdots \, \dot{\chi}^{[\ell_n]} \, 
\p_{y_j} \big( \psi^{\ell_{n+1}} \, \cdots \, \psi^{\ell_{2n}} \big) \, \uu_j^{\ell_{2n+1},\pm} \\
& \, -n \, \sum_{\substack{\ell_1+\cdots+\ell_{2n+3}=m+2 \\ \ell_{2n+3} \ge 1}} 
\dot{\chi}^{[\ell_1]} \, \cdots \, \dot{\chi}^{[\ell_n]} \, \p_{y_3} \chi^{[\ell_{n+1}]} \, 
\psi^{\ell_{n+2}} \, \cdots \, \psi^{\ell_{2n+1}} \, \p_\theta \psi^{\ell_{2n+2}} \, \xi_j \, \uu_j^{\ell_{2n+3},\pm} \\
& \, -n \, \sum_{\substack{\ell_1+\cdots+\ell_{2n+3}=m+1 \\ \ell_{2n+3} \ge 1}} 
\dot{\chi}^{[\ell_1]} \, \cdots \, \dot{\chi}^{[\ell_n]} \, \p_{y_3} \chi^{[\ell_{n+1}]} \, 
\psi^{\ell_{n+2}} \, \cdots \, \psi^{\ell_{2n+1}} \, \p_{y_j} \psi^{\ell_{2n+2}} \, \uu_j^{\ell_{2n+3},\pm} \, {\rm d}y \, {\rm d}\theta \, .
\end{align*}
We now integrate by parts in the second and third sum with respect to $\theta$ and $y_j$ to get:
\begin{align*}
\mathfrak{I}_n^{\, m,\pm} \, = \, \int_{\Omega_0^\pm \times \bT} \, & \, 
\sum_{\ell_1+\cdots+\ell_{2n+1}=m+1} 
\dot{\chi}^{[\ell_1]} \, \cdots \, \dot{\chi}^{[\ell_n]} \, \psi^{\ell_{n+1}} \, \cdots \, \psi^{\ell_{2n}} \, 
\big( \uF_7^{\ell_{2n+1},\pm} +\nabla \cdot \uu^{\ell_{2n+1},\pm} \big) \\
& \, +n \, \sum_{\substack{\ell_1+\cdots+\ell_{2n+1}=m+2 \\ \ell_{2n+1} \ge 1}} 
\dot{\chi}^{[\ell_1]} \, \cdots \, \dot{\chi}^{[\ell_{n-1}]} \, {\color{blue} \p_\theta \dot{\chi}^{[\ell_n]}} \, 
\psi^{\ell_{n+1}} \, \cdots \, \psi^{\ell_{2n}} \, \xi_j \, \uu_j^{\ell_{2n+1},\pm} \\
& \, +n \, \sum_{\substack{\ell_1+\cdots+\ell_{2n+1}=m+1 \\ \ell_{2n+1} \ge 1}} 
\dot{\chi}^{[\ell_1]} \, \cdots \, \dot{\chi}^{[\ell_{n-1}]} \, {\color{blue} \p_{y_j} \dot{\chi}^{[\ell_n]}} \, 
\psi^{\ell_{n+1}} \, \cdots \, \psi^{\ell_{2n}} \, \uu_j^{\ell_{2n+1},\pm} \\
& \, -n \, \sum_{\substack{\ell_1+\cdots+\ell_{2n+3}=m+2 \\ \ell_{2n+3} \ge 1}} 
\dot{\chi}^{[\ell_1]} \, \cdots \, \dot{\chi}^{[\ell_n]} \, \p_{y_3} \chi^{[\ell_{n+1}]} \, 
\psi^{\ell_{n+2}} \, \cdots \, \psi^{\ell_{2n+1}} \, \p_\theta \psi^{\ell_{2n+2}} \, \xi_j \, \uu_j^{\ell_{2n+3},\pm} \\
& \, -n \, \sum_{\substack{\ell_1+\cdots+\ell_{2n+3}=m+1 \\ \ell_{2n+3} \ge 1}} 
\dot{\chi}^{[\ell_1]} \, \cdots \, \dot{\chi}^{[\ell_n]} \, \p_{y_3} \chi^{[\ell_{n+1}]} \, 
\psi^{\ell_{n+2}} \, \cdots \, \psi^{\ell_{2n+1}} \, \p_{y_j} \psi^{\ell_{2n+2}} \, \uu_j^{\ell_{2n+3},\pm} \, {\rm d}y \, {\rm d}\theta \, .
\end{align*}
We then substitute the expression of $\uF_7^{\ell_{2n+1},\pm}$ and use the second symmetry formula (Proposition \ref{propB1} 
in Appendix \ref{appendixB}) for the blue terms $\p_\theta \dot{\chi}^{[\ell_n]}$ and $\p_{y_j} \dot{\chi}^{[\ell_n]}$, which yields:
\begin{align*}
\mathfrak{I}_n^{\, m,\pm} \, = \, \mathfrak{I}_{n+1}^{\, m,\pm} + \, & \, \int_{\Omega_0^\pm \times \bT} 
\sum_{\ell_1+\cdots+\ell_{2n+3}=m+2} \dot{\chi}^{[\ell_1]} \, \cdots \, \dot{\chi}^{[\ell_n]} \, \chi^{[\ell_{n+1}]} \, 
\psi^{\ell_{n+2}} \, \cdots \, \psi^{\ell_{2n+1}} \, \p_\theta \psi^{\ell_{2n+2}} \, \xi_j \, \p_{y_3} \uu_j^{\ell_{2n+3},\pm} \\
& \, +n \, \sum_{\substack{\ell_1+\cdots+\ell_{2n+3}=m+2 \\ \ell_{2n+3} \ge 1}} 
\dot{\chi}^{[\ell_1]} \, \cdots \, \dot{\chi}^{[\ell_{n-1}]} \, \p_{y_3} \dot{\chi}^{[\ell_n]} \, \chi^{[\ell_{n+1}]} \, 
\psi^{\ell_{n+2}} \, \cdots \, \psi^{\ell_{2n+1}} \, \p_\theta \psi^{\ell_{2n+2}} \, \xi_j \, \uu_j^{\ell_{2n+3},\pm} \\
& \, +\sum_{\substack{\ell_1+\cdots+\ell_{2n+3}=m+2 \\ \ell_{2n+3} \ge 1}} 
\dot{\chi}^{[\ell_1]} \, \cdots \, \dot{\chi}^{[\ell_n]} \, \p_{y_3} \chi^{[\ell_{n+1}]} \, 
\psi^{\ell_{n+2}} \, \cdots \, \psi^{\ell_{2n+1}} \, \p_\theta \psi^{\ell_{2n+2}} \, \xi_j \, \uu_j^{\ell_{2n+3},\pm} \\
\, & \, +\sum_{\ell_1+\cdots+\ell_{2n+3}=m+1} \dot{\chi}^{[\ell_1]} \, \cdots \, \dot{\chi}^{[\ell_n]} \, \chi^{[\ell_{n+1}]} \, 
\psi^{\ell_{n+2}} \, \cdots \, \psi^{\ell_{2n+1}} \, \p_{y_j} \psi^{\ell_{2n+2}} \, \p_{y_3} \uu_j^{\ell_{2n+3},\pm} \\
& \, +n \, \sum_{\substack{\ell_1+\cdots+\ell_{2n+3}=m+1 \\ \ell_{2n+3} \ge 1}} 
\dot{\chi}^{[\ell_1]} \, \cdots \, \dot{\chi}^{[\ell_{n-1}]} \, \p_{y_3} \dot{\chi}^{[\ell_n]} \, \chi^{[\ell_{n+1}]} \, 
\psi^{\ell_{n+2}} \, \cdots \, \psi^{\ell_{2n+1}} \, \p_{y_j} \psi^{\ell_{2n+2}} \, \uu_j^{\ell_{2n+3},\pm} \\
& \, +\sum_{\substack{\ell_1+\cdots+\ell_{2n+3}=m+1 \\ \ell_{2n+3} \ge 1}} 
\dot{\chi}^{[\ell_1]} \, \cdots \, \dot{\chi}^{[\ell_n]} \, \p_{y_3} \chi^{[\ell_{n+1}]} \, 
\psi^{\ell_{n+2}} \, \cdots \, \psi^{\ell_{2n+1}} \, \p_{y_j} \psi^{\ell_{2n+2}} \, \uu_j^{\ell_{2n+3},\pm} \, {\rm d}y \, {\rm d}\theta \\
= \, \mathfrak{I}_{n+1}^{\, m,\pm} + \, & \, \int_{\Omega_0^\pm \times \bT} \p_{y_3} \Big( 
\sum_{\substack{\ell_1+\cdots+\ell_{2n+3}=m+2 \\ \ell_{2n+3} \ge 1}} \dot{\chi}^{[\ell_1]} \, \cdots \, \dot{\chi}^{[\ell_n]} \, \chi^{[\ell_{n+1}]} \, 
\psi^{\ell_{n+2}} \, \cdots \, \psi^{\ell_{2n+1}} \, \p_\theta \psi^{\ell_{2n+2}} \, \xi_j \, \uu_j^{\ell_{2n+3},\pm} \Big) \\
& \, +\p_{y_3} \Big( \sum_{\substack{\ell_1+\cdots+\ell_{2n+3}=m+2 \\ \ell_{2n+3} \ge 1}} 
\dot{\chi}^{[\ell_1]} \, \cdots \, \dot{\chi}^{[\ell_n]} \, \chi^{[\ell_{n+1}]} \, 
\psi^{\ell_{n+2}} \, \cdots \, \psi^{\ell_{2n+1}} \, \p_{y_j} \psi^{\ell_{2n+2}} \, \uu_j^{\ell_{2n+3},\pm} \Big) \, {\rm d}y \, {\rm d}\theta \\
= \, \mathfrak{I}_{n+1}^{\, m,\pm} + \, & \, ,
\end{align*}
where the conclusion follows from the fact that all functions $\dot{\chi}^{[\ell]}$ vanish on $\Gamma^\pm$ and $\Gamma_0$ 
(and $n \ge 1$ so there is always at least one function $\dot{\chi}^{[\ell]}$ in each product). We thus have $\mathfrak{I}_1^{\, m,\pm}
=\mathfrak{I}_n^{\, m,\pm}$ for all $n \ge 1$ and, choosing $n$ large enough (recall that $\psi^0$ and $\psi^1$ vanish), we thus 
get  $\mathfrak{I}_1^{\, m,\pm}=0$. This completes the proof of Lemma \ref{lem-fredholm_pb_laplace}.
\end{proof}

At this stage, the front profile $\psi^{\, m+1}$ is fixed as follows. The oscillating modes in $\theta$ that compose $\psi_\sharp^{\, m+1}$ 
are given by the induction assumption $H(m)$, and the slow mean $\widehat{\psi}^{\, m+1}(0)$ is given by:
$$
\widehat{\psi}^{\, m+1}(0) \, = \, \Psi^{\, m+1} \, ,
$$
where the function $\Psi^{\, m+1}$ is given by Lemma \ref{lem-moyenne-front}, and we have chosen the mean (with respect to $y'$) of 
$\widehat{\psi}^{\, m+1}(0)$ to be zero in order to match the initial condition \eqref{s3-def_cond_init_oscil_psi} and the second order 
differential equation imposed by Lemma \ref{lem-fredholm_pb_laplace} (see the following Paragraph for a discussion on the choice 
of the time derivative at $t=0$ of the slow mean of the front profile). With the above choice for $\widehat{\psi}^{\, m+1}(0)$, we ensure 
solvability in $H^\infty(\Omega_0^+) \times H^\infty(\Omega_0^-)$ of the Laplace problem \eqref{pb_laplace_m+1}. Moreover, up to 
adding a function of time only to the solution to \eqref{pb_laplace_m+1}, we can fix the mean value:
$$
\int_{\Omega_0^+} \widehat{\uq}^{\, m+1,+}(0) \, {\rm d}y \, + \, \int_{\Omega_0^-} \widehat{\uq}^{\, m+1,-}(0) \, {\rm d}y
$$
as we want. In other words, we can always choose the solution to \eqref{pb_laplace_m+1} such that the constraint ${\mathcal I}^{\, m+1}(t) 
=0$ for the slow mean of the total pressure is satisfied (we recall that the quantity ${\mathcal I}^\mu$ for any integer $\mu$ is given by 
\eqref{expression_Im(t)}, \eqref{expression_Im(t)-bis}). In other words, we have already managed to enforce $(H(m+1)-4)$. In the 
following Paragraph, we are going to examine the remaining steps in the determination of the slow mean $\widehat{\uU}^{\, m+1,\pm}(0)$.

\subsection{Determining the slow mean of the corrector}

In order to construct the slow mean $\widehat{\uU}^{\, m+1,\pm}(0)$ of the corrector $U^{\, m+1,\pm}$, we need to verify that the source 
terms in \eqref{slow_mean_m+1_edp}, \eqref{slow_mean_m+1_saut}, \eqref{slow_mean_m+1_hautbas} satisfy the solvability 
conditions of Proposition \ref{prop-systemelent}. The corresponding items of Proposition \ref{prop-systemelent} are examined 
one by one below.

\paragraph{Compatibility at the top and bottom boundaries.} Recalling the definition \eqref{defsource_slow_mean_m+1}, 
we have:
$$
\bF^{\, m,\pm}|_{\Gamma^\pm} \, = \, -\, {\bf c}_0 \, \Big\{ \, 
\sum_{\substack{\ell_1+\ell_2=m+1 \\ \ell_1,\ell_2 \ge 1}} \bA_\alpha (\uU^{\ell_1,\pm},\p_{y_\alpha} \uU^{\ell_2,\pm})|_{\Gamma^\pm} 
\, \Big\} \, ,
$$
which gives (the Hessian mappings $\bA_\alpha$ are given in Appendix \ref{appendixA}):
$$
\bF_6^{\, m,\pm}|_{\Gamma^\pm} \, = \, -\, {\bf c}_0 \, \Big\{ \, \sum_{\substack{\ell_1+\ell_2=m+1 \\ \ell_1,\ell_2 \ge 1}} 
\p_{y_j} \big( \uu_j^{\ell_1,\pm} \, \uH_3^{\ell_2,\pm} -\uH_j^{\ell_1,\pm} \, \uu_3^{\ell_2,\pm} \big)|_{\Gamma^\pm} \, \Big\} \, = \, 0 \, ,
$$
where we have used the boundary conditions \eqref{inductionHm3} and the fact that only tangential derivatives with respect to 
$\Gamma^\pm$ are involved.

\paragraph{Compatibility on $\Gamma_0$.} The verification of the conditions:
$$
\bF_6^{\, m,\pm}|_{\Gamma_0} \, = \, \big( \p_t +u_j^{0,\pm} \, \p_{y_j} \big) \bG_2^{\, m,\pm} -H_j^{0,\pm} \, \p_{y_j} \bG_1^{\, m,\pm} \, ,
$$
is performed in Appendix \ref{appendixB}, see Lemma \ref{lemB2}. The (long) proof is in the same spirit as the proof of Lemma 
\ref{lem_compatibilite_bord} above so we have thought it more convenient to refer the interested reader to Appendix \ref{appendixB} 
and proceed with those ingredients that are new in the analysis. We just emphasize the fact that the verification of the latter relation 
is actually independent of our previous determination of the slow mean of the front profile $\widehat{\psi}^{\, m+1}(0)$ since the terms 
in $\bG_1^{\, m,\pm}$ and $\bG_2^{\, m,\pm}$ where $\widehat{\psi}^{\, m+1}(0)$ appears cancel when we compute
$$
\big( \p_t +u_j^{0,\pm} \, \p_{y_j} \big) \bG_2^{\, m,\pm} -H_j^{0,\pm} \, \p_{y_j} \bG_1^{\, m,\pm} \, .
$$
In particular, verifying the compatibility condition on $\Gamma_0$ is independent of our choice of initial conditions for the front profiles 
$\psi^\mu$.

\paragraph{Compatibility for the divergence of the magnetic field.} Once again, the proof of the relation:
$$
\p_t \bF_8^{\, m,\pm} \, = \, \p_{y_\alpha} \bF_{3+\alpha}^{\, m,\pm} \, ,
$$
is postponed in Appendix \ref{appendixB} (we refer the interested reader to Lemma \ref{lemB4}).

\paragraph{Existence of compatible initial conditions.} This is the main new point in the analysis. We explain why it is possible to find 
compatible initial data \eqref{slow_mean_m+1_di} for the slow mean system \eqref{slow_mean_m+1_edp}, \eqref{slow_mean_m+1_saut}, 
\eqref{slow_mean_m+1_hautbas}. We recall that at time $t=0$, we wish to have $\widehat{\psi}^{\, m+1}(0)=0$ because of 
\eqref{s3-cond_init_profils_psi_0^m}. In particular, the initial condition $\mathfrak{H}^{\, m+1,\pm}_0$ in \eqref{slow_mean_m+1_di} 
for the magnetic field should satisfy the following relations:
\begin{equation}
\label{divergence_H_t=0}
\begin{cases}
\nabla \cdot \mathfrak{H}^{\, m+1,\pm}_0 \, = \, \bF_8^{\, m,\pm}|_{t=0} \, , & \text{\rm in $\Omega_0^\pm$,} \\
\mathfrak{H}^{\, m+1,\pm}_{0,3}|_{\Gamma^\pm} \, = \, 0 \, , & \\
\mathfrak{H}^{\, m+1,\pm}_{0,3}|_{\Gamma_0} \, = \, \bG_2^{\, m,\pm}|_{t=0} \, , & 
\end{cases}
\end{equation}
where:
\begin{align*}
\bG_2^{\, m,\pm}|_{t=0} \, = \, & \, {\bf c}_0 \, \Big\{ 
\sum_{\substack{\ell_1+\ell_2=m+2 \\ \ell_2 \ge 1}} \p_\theta \psi^{\ell_1} \, \xi_j \, H_j^{\ell_2,\pm} 
+\sum_{\substack{\ell_1+\ell_2=m+1 \\ \ell_2 \ge 1}} \p_{y_j} \psi^{\ell_1} \, H_j^{\ell_2,\pm} \Big\}|_{t=y_3=Y_3=0} 
-\widehat{H}_{3,\star}^{\, m+1,\pm}(0)|_{t=y_3=Y_3=0} \\
= \, & \, {\bf c}_0 \, \Big\{ \p_\theta \psi^2_0 \, \xi_j \, H_j^{\, m,\pm} +\p_{y_j} \psi^2_0 \, H_j^{\, m-1,\pm} \Big\}|_{t=y_3=Y_3=0} 
-\widehat{H}_{3,\star}^{\, m+1,\pm}(0)|_{t=y_3=Y_3=0} \, .
\end{align*}
The existence of a solution to the divergence problem \eqref{divergence_H_t=0} is equivalent to the fulfillment of the condition:
$$
\int_{\Omega_0^\pm} \bF_8^{\, m,\pm}|_{t=0} \, {\rm d}y \pm \int_{\bT^2} \bG_2^{\, m,\pm}|_{t=0} \, {\rm d}y' \, = \, 0 \, .
$$
Applying Lemma \ref{lem-fredholm_pb_laplace}, we know that the latter condition holds not only at time $t=0$ but also at any time 
$t \in [0,T]$. Since the domain $\Omega_0^\pm$ has an infinitely smooth boundary and the source terms in \eqref{divergence_H_t=0} 
have $H^\infty$ regularity, we can indeed find a solution $\mathfrak{H}^{\, m+1,\pm}_0 \in H^\infty(\Omega_0^\pm)$ to \eqref{divergence_H_t=0}. 
There are of course infinitely many possible choices, and each choice will give rise to one solution to \eqref{slow_mean_m+1_edp}, 
\eqref{slow_mean_m+1_saut}, \eqref{slow_mean_m+1_hautbas}. (This is the reason why we have not claimed any uniqueness property 
in Theorem \ref{thm_principal}.) However, when prescribing zero initial data for the slow mean of the magnetic field is possible, we shall 
then do so, see for instance Chapter \ref{chapter6} for the case of the first corrector.

Let us now examine the existence of a compatible initial condition for the velocity field. The system to solve for $\mathfrak{u}^{\, m+1,\pm}_0$ 
reads:
\begin{equation}
\label{divergence_u_t=0}
\begin{cases}
\nabla \cdot \mathfrak{u}^{\, m+1,\pm}_0 \, = \, \bF_7^{\, m,\pm}|_{t=0} \, , & \text{\rm in $\Omega_0^\pm$,} \\
\mathfrak{u}^{\, m+1,\pm}_{0,3}|_{\Gamma^\pm} \, = \, 0 \, , & \\
\mathfrak{u}^{\, m+1,\pm}_{0,3}|_{\Gamma_0} \, = \, \bG_1^{\, m,\pm}|_{t=0} \, , & 
\end{cases}
\end{equation}
where:
\begin{align*}
\bG_1^{\, m,\pm}|_{t=0} \, = \, & \, \p_t \widehat{\psi}^{\, m+1}(0)|_{t=0} \\
& \, +{\bf c}_0 \, \Big\{ \sum_{\substack{\ell_1+\ell_2=m+2 \\ \ell_2 \ge 1}} \p_\theta \psi^{\ell_1} \, \xi_j \, u_j^{\ell_2,\pm} 
+\sum_{\substack{\ell_1+\ell_2=m+1 \\ \ell_2 \ge 1}} \p_{y_j} \psi^{\ell_1} \, u_j^{\ell_2,\pm} \Big\}|_{t=y_3=Y_3=0} 
-\widehat{u}_{3,\star}^{\, m+1,\pm}(0)|_{t=y_3=Y_3=0} \\
= \, & \, \p_t \widehat{\psi}^{\, m+1}(0)|_{t=0} 
+{\bf c}_0 \, \Big\{ \p_\theta \psi^2_0 \, \xi_j \, u_j^{\, m,\pm} +\p_{y_j} \psi^2_0 \, u_j^{\, m-1,\pm} \Big\}|_{t=y_3=Y_3=0} 
-\widehat{u}_{3,\star}^{\, m+1,\pm}(0)|_{t=y_3=Y_3=0} \, .
\end{align*}
The existence of a solution to the divergence problem \eqref{divergence_u_t=0} is equivalent to the fulfillment of the condition:
$$
\int_{\Omega_0^\pm} \bF_7^{\, m,\pm}|_{t=0} \, {\rm d}y \pm \int_{\bT^2} \bG_1^{\, m,\pm}|_{t=0} \, {\rm d}y' \, = \, 0 \, .
$$
Applying Lemma \ref{lem-fredholm_pb_laplace}, we see that a necessary and sufficient condition for solving \eqref{divergence_u_t=0} is 
to impose:
$$
\dfrac{{\rm d}}{{\rm d}t} \,  \int_{\Gamma_0 \times \bT} \psi^{\, m+1} \, {\rm d}y' \, {\rm d}\theta \, \Big|_{t=0} \, = \, 0 \, . 
$$
Let us observe that when we have determined the mean free part (in $y'$) of $\widehat{\psi}^{\, m+1}(0)$, the choice that was made in 
Lemma \ref{lem-moyenne-front}, namely $\p_t \Psi^{\, m+1}|_{t=0} =0$, was purely a matter of convenience. However, the choice we make 
here for the initial velocity of the mean is not a matter of convenience; it is imposed by the solvability condition for the divergence problem 
\eqref{divergence_u_t=0}. A similar constraint on the mean of the front arises in \cite{SWZ}. The initial condition \eqref{s3-cond_init_profils_psi_0^m} 
for the front profiles together with the solvability condition for the divergence problem \eqref{divergence_u_t=0} and the Laplace problem 
\eqref{pb_laplace_m+1} impose that we take:
$$
\forall \, t \in [0,T] \, ,\quad \int_{\Gamma_0 \times \bT} \psi^{\, m+1} \, {\rm d}y' \, {\rm d}\theta \, := \, 0 \, ,
$$
which explains the decomposition of the front profile $\psi^{\, m+1}$:
$$
\psi^{\, m+1}(t,y',\theta) \, = \, \psi^{\, m+1}_\sharp (t,y',\theta) +\Psi^{\, m+1}(t,y') \, .
$$
Let us recall again that the mean in $\theta$ of $\psi^{\, m+1}$, which we denote $\Psi^{\, m+1}$, is obtained by solving the wave type equation 
\eqref{eq_onde_front_m+1'} on $\Gamma_0$ and that $\Psi^{\, m+1}$ has zero mean in $y'$ on $\bT^2$. Our choice for $\psi^{\, m+1}$ implies 
that we can construct a solution $\mathfrak{u}^{\, m+1,\pm}_0 \in H^\infty(\Omega_0^\pm)$ to the divergence problem \eqref{divergence_u_t=0}. 
Applying then Proposition \ref{prop-systemelent}, we have all the ingredients to solve the slow mean problem \eqref{slow_mean_m+1_edp}, 
\eqref{slow_mean_m+1_saut}, \eqref{slow_mean_m+1_hautbas}.

\paragraph{Summary.} Let us now summarize what we have done so far and which relations of $H(m+1)$ we have already satisfied. We have 
first solved the oscillating modes in $\theta$ of the fast problem \eqref{fast_problem_m+1}, which has given rise to the functions $\bU^{\, m+1,\pm}$ 
that are part of the decomposition \eqref{decompositionUm+1} of the corrector $U^{\, m+1,\pm}$. This first step automatically gave the boundary 
conditions $(H(m+1)-3)$ on $\Gamma^\pm$ for the nonzero Fourier modes in $\theta$. We have then defined the fast means of the normal velocity, 
normal magnetic field and total pressure so that the differential equations in \eqref{fast_problem_m+1''} are satisfied. From there on, no matter what 
we do for the remaining degrees of freedom in \eqref{decompositionUm+1}, the corrector $U^{\, m+1,\pm}$ will satisfy:
\begin{equation*}
\begin{cases}
\cL_f^\pm(\partial) \, U^{\, m+1,\pm} \, = \, F^{\, m,\pm} \, ,& y \in \Omega_0^\pm \, ,\, \pm Y_3 >0 \, ,\\
\p_{Y_3} H_3^{\, m+1,\pm} +\xi_j \, \p_\theta H_j^{\, m+1,\pm} \, = \, F_8^{\, m,\pm} \, ,& y \in \Omega_0^\pm \, ,\, \pm Y_3 >0 \, .
\end{cases}
\end{equation*}
We have then studied the problem that should be satisfied by the slow mean of $U^{\, m+1,\pm}$ which led us to also determine the 
mean $\widehat{\psi}^{\, m+1}(0)$ in order to be able to solve an overdetermined coupled Laplace problem for the slow mean of the 
total pressure $(\widehat{\uq}^{\, m+1,+}(0),\widehat{\uq}^{\, m+1,-}(0))$. When determining the slow mean of $U^{\, m+1,\pm}$, we 
have enforced the jump conditions in \eqref{fast_problem_m+1''}, the boundary conditions $(H(m+1)-3)$ for the zero Fourier mode in 
$\theta$, the normalization condition $(H(m+1)-4)$ and the slow mean conditions $(H(m+1)-5)$.

Independently of our future determination of what remains in \eqref{decompositionUm+1}, we have thus already obtained $(H(m+1)-2)$ 
(including the boundary conditions on $\Gamma_0$), $(H(m+1)-3)$, $(H(m+1)-4)$ and $(H(m+1)-5)$. Moreover, the remaining degrees 
of freedom in \eqref{decompositionUm+1} are the fast means of the tangential components of the velocity and magnetic fields, as well as 
the oscillating modes of the front profile $\psi^{\, m+2}$. In the following Section, we explain how we choose the fast means of the tangential 
components of the velocity and magnetic fields. After that, there will only remain to determine $\psi_\sharp^{\, m+2}$.

\section{The tangential components of the fast mean}

We are going to determine the fast means $\widehat{u}_{j,\star}^{\, m+1,\pm}(0)$, $\widehat{H}_{j,\star}^{\, m+1,\pm}(0)$ by imposing 
the condition $(H(m+1)-6)$ together with the initial conditions \eqref{s3-cond_moyenne_initiale_U^m,pm}. Recalling the general 
definition \eqref{s3-def_terme_source_F^m,pm} of the source term $F^{\mu,\pm}$ and specifying to the case $\mu=m+1$, we get:
\begin{align*}
F^{\, m+1,\pm} \, = \, & \, -L_s^\pm(\p) U^{\, m+1,\pm} 
+\p_\theta \psi^2 \, \cA^\pm \, \p_{Y_3} U^{\, m+1,\pm} +\p_\theta \psi^{\, m+2} \, \cA^\pm \, \p_{Y_3} U^{\, 1,\pm} \\
& \, -\xi_j \, \bA_j (U^{\, 1,\pm},\p_\theta U^{\, m+1,\pm}) -\xi_j \, \bA_j (U^{\, m+1,\pm},\p_\theta U^{\, 1,\pm}) \\
& \, -\bA_3(U^{\, 1,\pm} ,\p_{Y_3} U^{\, m+1,\pm}) -\bA_3(U^{\, m+1,\pm} ,\p_{Y_3} U^{\, 1,\pm}) +\widetilde{F}^{\, m,\pm} \, ,
\end{align*}
where we have first kept all those terms in $F^{\, m+1,\pm}$ that depend on $U^{\, m+1,\pm}$ (and are therefore not fully determined 
at this point), and where $\widetilde{F}^{\, m,\pm}$ is entirely given in terms of all previously determined profiles. Computing the zero 
Fourier coefficient with respect to $\theta$, we thus get:
\begin{align*}
\widehat{F}_\star^{\, m+1,\pm}(0) \, = \, & \, -L_s^\pm(\p)\widehat{U}_\star^{\, m+1,\pm}(0) 
+\big( \widehat{\p_\theta \psi^2 \, \cA^\pm \, \p_{Y_3} U^{\, m+1,\pm}} \big) (0)
+\big( \widehat{\p_\theta \psi^{\, m+2} \, \cA^\pm \, \p_{Y_3} U^{\, 1,\pm}} \big) (0) \\
& \, -\widehat{\bA_3(U^{\, 1,\pm} ,\p_{Y_3} U^{\, m+1,\pm})}(0) -\widehat{\bA_3(U^{\, m+1,\pm} ,\p_{Y_3} U^{\, 1,\pm})}(0) 
+\widetilde{F}^{\, m,\pm} \, ,
\end{align*}
where, now, $\widetilde{F}^{\, m,\pm}$ denotes a function in $S_\star^\pm$ that is entirely computable in terms of all previously determined 
profiles (but whose precise expression is useless for our purpose, which explains why we feel free to use the same notation from one 
line to the other even though the functions are not the same). Similarly, we have:
$$
\widehat{F}_{8,\star}^{\, m+1,\pm}(0) \, = \, -\nabla \cdot \widehat{H}_\star^{\, m+1,\pm}(0) 
+\big( \widehat{\p_\theta \psi^2 \, \xi_j \, \p_{Y_3} H_{j,\star}^{\, m+1,\pm}} \big) (0)
+\big( \widehat{\p_\theta \psi^{\, m+2} \, \xi_j \, \p_{Y_3} H_{j,\star}^{\, 1,\pm}} \big) (0) 
+\widetilde{F}_8^{\, m,\pm} \, ,
$$

At this stage, the general decomposition \eqref{decompositionUm+1} can be rewritten as:
\begin{align}
U^{\, m+1,\pm} \, = \, \widetilde{\bU}^{\, m+1,\pm} 
\, & \, +\big( 0,\widehat{u}_{1,\star}^{\, m+1,\pm}(0),\widehat{u}_{2,\star}^{\, m+1,\pm}(0),
0,\widehat{H}_{1,\star}^{\, m+1,\pm}(0),\widehat{H}_{2,\star}^{\, m+1,\pm}(0),0 \big)^T \label{decompositionUm+1'} \\
& \, \pm \sum_{k \neq 0} |k| \, \widehat{\psi}^{\, m+2} (t,y',k) \, \chi (y_3) \, {\rm e}^{\mp \, |k| \, Y_3 +i \, k \, \theta} \, \cR^\pm(k) \, ,\notag
\end{align}
where $\widetilde{\bU}^{\, m+1,\pm} \in S^\pm$ gathers all previously determined quantities, and the remaining unknown quantities are 
the fast means of the tangential components and the front profile $\psi^{\, m+2}_\sharp$ (or equivalently $\p_\theta \psi^{\, m+2}$). We now 
wish to make explicit the constraint $(H(m+1)-6)$ and perform more or less the same calculations as in the analogous Section of Chapter 
\ref{chapter4}. Namely, we first define:
\begin{equation*}
\mathfrak{F}^{\, m+1,\pm} \, := \, \p_{Y_3} \widehat{\Big( \bA_3(U_\star^{\, 1,\pm},U_\star^{\, m+1,\pm}) \Big)} (0) \, ,\quad 
\mathfrak{F}_8^{\, m+1,\pm} \, := \, 0 \, .
\end{equation*}
The explicit expression of the Hessian mapping $\bA_3$ given in Appendix \ref{appendixA} gives\footnote{This property would hold even 
if we had chosen nonzero initial conditions in \eqref{s3-cond_moyenne_initiale_U^m,pm}.}:
$$
\bA_3 \big( \widehat{U}_\star^{\, 1,\pm} (0),\Pi \, \widehat{U}_\star^{\, m+1,\pm} (0) \big) \, = \, 0 \, ,
$$
independently of the determination of the tangential components $\Pi \, \widehat{U}_\star^{\, m+1,\pm} (0)$. Using the above decomposition 
\eqref{decompositionUm+1'} of $U^{\, m+1,\pm}$ and the decomposition \eqref{s3-U^1,pm_k_final} of the leading profile $U^{\, 1,\pm}$, we 
thus compute:
\begin{equation*}
\mathfrak{F}^{\, m+1,\pm} \, = \, \mp 2 \, \chi(y_3)^2 \, \sum_{k \neq 0} {\rm e}^{\mp 2\, |k| \, Y_3} \, |k|^3 \, \widehat{\psi}^{\, 2} (k) \, 
\widehat{\psi}^{\, m+2} (-k) \, \bA_3 \Big( {\bf R}^\pm,\overline{{\bf R}^\pm} \Big) +\widetilde{\mathfrak{F}}^{\, m,\pm} \, ,
\end{equation*}
where the last term $\widetilde{\mathfrak{F}}^{\, m,\pm}$ is a known function (meaning that it can be expressed in terms of previously 
determined profiles). Recalling the expression:
$$
\bA_3 \Big( {\bf R}^\pm,\overline{{\bf R}^\pm} \Big) \, = \, 2 \, \Big( 0,0,(c^\pm)^2-(b^\pm)^2,0,0,0,0 \Big)^T \, ,
$$
we obtain that the source terms $\mathfrak{F}^{\, m+1,\pm},\mathfrak{F}_8^{\, m+1,\pm}$ defined above satisfy the linear system:
\begin{equation*}
\begin{cases}
u_j^{0,\pm} \, \mathfrak{F}_7^{\, m+1,\pm} -H_j^{0,\pm} \, \mathfrak{F}_8^{\, m+1,\pm} \, = \, \mathfrak{F}_j^{\, m+1,\pm} 
+\widetilde{\mathfrak{F}}_j^{\, m,\pm} \, , & j=1,2 \, ,\\
H_j^{0,\pm} \, \mathfrak{F}_7^{\, m+1,\pm} -u_j^{0,\pm} \, \mathfrak{F}_8^{\, m+1,\pm} \, = \, \mathfrak{F}_{3+j}^{\, m+1,\pm} 
+\widetilde{\mathfrak{F}}_{3+j}^{\, m,\pm} \, , & j=1,2 \, ,
\end{cases}
\end{equation*}
independently of the choice we can make for the tangential components $\Pi \, \widehat{U}_\star^{\, m+1,\pm} (0)$ and for the 
front profile $\psi_\sharp^{\, m+2}$ in the decomposition \eqref{decompositionUm+1'}. Again, here, we use the notation 
$\widetilde{\mathfrak{F}}^{\, m,\pm}$ for known quantities (hence the index $m$ rather than $m+1$).

Going on with the remaining terms in the above decomposition of $\widehat{F}_\star^{\, m+1,\pm}(0)$ and $\widehat{F}_{8,\star}^{\, m+1,\pm}(0)$ 
(the calculations are similar to those in Chapter \ref{chapter4} so we feel free to shorten the details), we eventually obtain that the fulfillment of 
the condition $(H(m+1)-6)$ is equivalent to the verification of a system of four partial differential equations that reads:
\begin{equation*}
\begin{cases}
u_j^{0,\pm} \, \big( -\nabla \cdot \widehat{u}_\star^{\, m+1,\pm}(0) \big) -H_j^{0,\pm} \, \big( -\nabla \cdot \widehat{H}_\star^{\, m+1,\pm}(0) \big) 
\, = \, \big( -L_s^\pm (\p) \, \widehat{U}_\star^{\, m+1,\pm} (0) \big)_j +\cF_{j,\star}^{\, m,\pm} \, , & \\
H_j^{0,\pm} \, \big( -\nabla \cdot \widehat{u}_\star^{\, m+1,\pm}(0) \big) -u_j^{0,\pm} \, \big( -\nabla \cdot \widehat{H}_\star^{\, m+1,\pm}(0) \big) 
\, = \, \big( -L_s^\pm (\p) \, \widehat{U}_\star^{\, m+1,\pm} (0) \big)_{3+j} +\cF_{3+j,\star}^{\, m,\pm} \, , & 
\end{cases}
\end{equation*}
where $\cF_{j,\star}^{\, m,\pm},\cF_{3+j,\star}^{\, m,\pm} \in S_\star^\pm$ are known source terms. The latter system is independent of the 
choice we can make for the front profile $\psi^{\, m+2}_\sharp$, and it can be equivalently rewritten as the symmetric hyperbolic system:
\begin{equation}
\label{evolution_fast_mean_m+1}
\p_t \begin{bmatrix}
\widehat{u}_{1,\star}^{\, m+1,\pm}(0) \\
\widehat{u}_{2,\star}^{\, m+1,\pm}(0) \\
\widehat{H}_{1,\star}^{\, m+1,\pm}(0) \\
\widehat{H}_{2,\star}^{\, m+1,\pm}(0) \end{bmatrix} 
+\begin{bmatrix}
u_j^{0,\pm} & 0 & -H_j^{0,\pm} & 0 \\
0 & u_j^{0,\pm} & 0 & -H_j^{0,\pm} \\
-H_j^{0,\pm} & 0 & u_j^{0,\pm} & 0 \\
0 & -H_j^{0,\pm} & 0 & u_j^{0,\pm} \end{bmatrix} \, \p_{y_j} \begin{bmatrix}
\widehat{u}_{1,\star}^{\, m+1,\pm}(0) \\
\widehat{u}_{2,\star}^{\, m+1,\pm}(0) \\
\widehat{H}_{1,\star}^{\, m+1,\pm}(0) \\
\widehat{H}_{2,\star}^{\, m+1,\pm}(0) \end{bmatrix} \, = \, \begin{bmatrix}
\widetilde{\cF}_{1,\star}^{\, m,\pm} \\
\widetilde{\cF}_{2,\star}^{\, m,\pm} \\
\widetilde{\cF}_{4,\star}^{\, m,\pm} \\
\widetilde{\cF}_{5,\star}^{\, m,\pm} \end{bmatrix} \, ,
\end{equation}
for appropriately computed source terms $\widetilde{\cF}_{j,\star}^{\, m,\pm}$, $\widetilde{\cF}_{3+j,\star}^{\, m,\pm}$ that incorporate, 
for instance, that part of $L_s^\pm (\p) \, \widehat{U}_\star^{\, m+1,\pm} (0)$ that depends on the fast mean of the normal velocity, 
normal magnetic field and total pressure.

The symmetric hyperbolic system \eqref{evolution_fast_mean_m+1} is solved with the initial conditions:
$$
\Pi \, \widehat{U}_\star^{\, m+1,\pm}(0) \big|_{t=0} \, = \, 0 \, ,
$$
in order to be consistent with \eqref{s3-cond_moyenne_initiale_U^m,pm}. Since $y_3 \in I^\pm$ and $Y_3 \in \bR^\pm$ are parameters 
in \eqref{evolution_fast_mean_m+1}, there is no real difficulty to show that the solution $\Pi \, \widehat{U}_\star^{\, m+1,\pm}(0)$ to 
\eqref{evolution_fast_mean_m+1} belongs to $S_\star^\pm$. We have thus constructed $\Pi \, \widehat{U}_\star^{\, m+1,\pm}(0) \in 
S_\star^\pm$, satisfied $(H(m+1)-6)$ and it only remains to determine the front profile $\psi_\sharp^{\, m+2}$ to close the decomposition 
\eqref{decompositionUm+1'}.

\section{The linearized nonlocal Hamilton-Jacobi equation for the front}

Unsurprisingly, we are going to determine the oscillating modes in $\theta$ of the front profile $\psi^{\, m+2}$ by imposing the condition 
$(H(m+1)-7)$. Let us recall indeed that a necessary condition for the existence of a solution to the fast problem:
\begin{equation*}
\begin{cases}
\cL_f^\pm(\partial) \, U^{\, m+2,\pm} \, = \, F^{\, m+1,\pm} \, ,& \\
\p_{Y_3} H_3^{\, m+2,\pm} +\xi_j \, \p_\theta H_j^{\, m+2,\pm} \, = \, F_8^{\, m+1,\pm} \, ,& \\
B^+ \, U^{\, m+2,+}|_{y_3=Y_3=0} +B^- \, U^{\, m+2,-}|_{y_3=Y_3=0} +\partial_\theta \psi^{\, m+3} \, \ub \, = \, G^{\, m+1} \, ,
\end{cases}
\end{equation*}
is that the source terms $F^{\, m+1,\pm}$, $G^{\, m+1}$ should satisfy the orthogonality condition \eqref{compatibilite_pb_rapide_e} of 
Theorem \ref{theorem_fast_problem}. Let us also recall that our corrector $U^{\, m+1,\pm}$ has the form \eqref{decompositionUm+1'}, 
where the only unknown quantity at this stage are the Fourier modes $\widehat{\psi}^{\, m+2}(k)$ for $k \neq 0$. Moreover, the source 
term $F^{\, m+1,\pm}$ has the expression:
\begin{align*}
F^{\, m+1,\pm} \, = \, & \, -L_s^\pm(\p) \dot{U}^{\, m+1,\pm} 
+\p_\theta \psi^2 \, \cA^\pm \, \p_{Y_3} \dot{U}^{\, m+1,\pm} +\p_\theta \psi^{\, m+2} \, \cA^\pm \, \p_{Y_3} U^{\, 1,\pm} \\
& \, -\xi_j \, \bA_j (U^{\, 1,\pm},\p_\theta \dot{U}^{\, m+1,\pm}) -\xi_j \, \bA_j (\dot{U}^{\, m+1,\pm},\p_\theta U^{\, 1,\pm}) \\
& \, -\bA_3(U^{\, 1,\pm} ,\p_{Y_3} \dot{U}^{\, m+1,\pm}) -\bA_3(\dot{U}^{\, m+1,\pm} ,\p_{Y_3} U^{\, 1,\pm}) +\widetilde{F}^{\, m,\pm} \, ,
\end{align*}
where $\widetilde{F}^{\, m,\pm}$ is entirely given in terms of all previously determined profiles, and $\dot{U}^{\, m+1,\pm}$ denotes the 
only still unknown part in \eqref{decompositionUm+1'}, that is:
$$
\dot{U}^{\, m+1,\pm} \, := \, \pm 
\sum_{k \neq 0} |k| \, \widehat{\psi}^{\, m+2} (t,y',k) \, \chi (y_3) \, {\rm e}^{\mp \, |k| \, Y_3 +i \, k \, \theta} \, \cR^\pm(k) \, .
$$
Restricting to $y_3=0$, all terms but one in $F^{\, m+1,\pm}|_{y_3=0}$ can be expressed in terms of the front profile $\psi^{\, m+2}$ or 
its partial derivatives. The only undetermined term, or at least the only term that depends on our choice for the lifting from $\{ y_3=0 \}$ 
to $y_3 \in I^\pm$, is $A_3^\pm \, \p_{y_3} \dot{U}^{\, m+1,\pm}$ but as in Chapter \ref{chapter4}, this term will cancel in the orthogonality 
condition $(H(m+1)-7)$ due to the relation:
$$
\forall \, k \neq 0 \, ,\quad \cL^\pm (k) \sbt A_3^\pm \, \cR^\pm (k) \, = \, 0 \, .
$$
Let us also examine the boundary term $G^{\, m+1}$, whose components read (see \eqref{terme_source_bord_BKW}):
\begin{align*}
G_1^{\, m+1,\pm} \, & \, = \, \p_t \psi^{\, m+2} +u_j^{0,\pm} \, \p_{y_j} \psi^{\, m+2} 
+\p_\theta \psi^2 \, \xi_j \, \dot{u}_j^{\, m+1,\pm}|_{y_3=Y_3=0} +\p_\theta \psi^{\, m+2} \, \xi_j \, u_j^{\, 1,\pm}|_{y_3=Y_3=0} 
+\widetilde{G}_1^{\, m,\pm} \, , \\
G_2^{\, m+1,\pm} \, & \, = \, H_j^{0,\pm} \, \p_{y_j} \psi^{\, m+2} 
+\p_\theta \psi^2 \, \xi_j \, \dot{H}_j^{\, m+1,\pm}|_{y_3=Y_3=0} +\p_\theta \psi^{\, m+2} \, \xi_j \, H_j^{\, 1,\pm}|_{y_3=Y_3=0} 
+\widetilde{G}_2^{\, m,\pm} \, ,
\end{align*}
where $\widetilde{G}_1^{\, m,\pm}$ and $\widetilde{G}_2^{\, m,\pm}$ are given in terms of all previously determined profiles.

It remains to plug the above expressions in the orthogonality condition $(H(m+1)-7)$ and to follow the calculations that have 
been given in full details in Chapter \ref{chapter4} and that we shall therefore not repeat here. It is eventually found that 
$(H(m+1)-7)$ equivalently reads (here $k$ is a nonzero integer):
\begin{multline}
\label{s3-edp_hatpsi^m+2}
\big( c^+ + c^- \big) \, \p_t \widehat{\psi}^{\, m+2} (k) 
+\big( c^+ \, u_j^{0,+} \, + \, c^- \, u_j^{0,-} \, - \, b^+ \, H_j^{0,+} \, - \, b^- \, H_j^{0,-} \big) \, \p_{y_j} \widehat{\psi}^{\, m+2} (k) \\
+2 \, i \, \Big( (c^+)^2 - (c^-)^2 - (b^+)^2 + (b^-)^2 \Big) \, \text{\rm sgn} (k) \, 
\sum_{k_1 +k_2 =k} \dfrac{|k_1| \, |k_2| \, |k_1 +k_2|}{|k_1| + |k_2| + |k_1 +k_2|} \, \widehat{\psi}^{\, 2} (k_1) \, \widehat{\psi}^{\, m+2} (k_2) 
\, = \, 0 \, ,
\end{multline}
which corresponds to the linearization of \eqref{s3-edp_hatpsi^2} around the leading front profile $\psi^2$ (hence the factor $2$ at 
the beginning of the second line in \eqref{s3-edp_hatpsi^m+2}).

The solvability of \eqref{s3-edp_hatpsi^m+2} in the space $\cC^\infty([0,T];H^\infty_\sharp)$ follows from the linear analogue 
of Theorem \ref{s3-thm_hunter}, which we shall omit to state precisely but the reader will easily fill the gaps. (Namely, the reader 
may follow the arguments in \cite{Hunter2006} and verify that the linear analogue of Theorem 4.2 in that reference can be proved 
with similar arguments as those given in \cite{Hunter2006}, the proof being actually simpler since lifespan is not an issue for linear 
equations.) We can thus construct a solution $\psi^{\, m+2}_\sharp$ to \eqref{s3-edp_hatpsi^m+2} in such a way that the profile 
$U^{\, m+1,\pm} \in S^\pm$ in \eqref{decompositionUm+1'} now satisfies $(H(m+1)-7)$. This completes the proof of our induction, 
meaning that with the positive time $T$ being fixed by Theorem \ref{s3-thm_hunter}, the induction assumption \eqref{inductionHm1}, 
\dots, \eqref{inductionHm7} holds for all $m \ge 1$.
\bigskip

Let us now compare with the statement of Theorem \ref{thm_principal}. The initial data for the front profiles have been chosen in 
agreement with \eqref{s3-cond_init_profils_psi_0^m}, and we have seen above that the initial data for the oscillating modes in $y'$ 
of each $\p_t \widehat{\psi}^{\, m}(0)$ could be chosen arbitrarily (at the opposite, the initial condition for the mean with respect to 
$y'$ is not arbitrary and has to vanish). For simplicity, we have chosen to impose $\p_t \widehat{\psi}^{\, m}(0)|_{t=0}=0$, as stated 
in Theorem \ref{thm_principal}, but other choices could be made. Let us turn to the second point in Theorem \ref{thm_principal}. In 
Chapter \ref{chapter4}, we have seen that the leading profile could be constructed in such a way that its slow and fast means vanish. 
Since the residual component of the leading profile $U^{\, 1,\pm}$ does not oscillate with respect to $\theta$, we have obtained in 
Chapter \ref{chapter4} $\uU^{\, 1,\pm}=0$ and $\widehat{U}_\star^{\, 1,\pm}(0)=0$. In Chapter \ref{chapter6} below, we shall examine 
why it is possible to also have $\uU^{\, 2,\pm}=0$. This requires some specific `orthogonality' properties for the leading profile 
$U^{\, 1,\pm}$. When solving the above system \eqref{evolution_fast_mean_m+1}, we can prescribe the initial conditions as we wish, 
and for the sake of simplicity, we have chosen $\Pi \widehat{U}_\star^{\, m,\pm}(0)|_{t=0}=0$ as stated in Theorem \ref{thm_principal}.

Up to the property $\uU^{\, 2,\pm}=0$, which we shall examine separately in Chapter \ref{chapter6}, we have thus already proved 
so far the validity of the first three points in Theorem \ref{thm_principal}. Moreover, the sequence of profiles 
$(U^{\, m,\pm},\psi^{\, m+1})_{m \ge 1}$ satisfies the induction assumption $(H(m))$ for any $m \ge 1$ with a uniform time $T>0$.
We are now going to examine why the various properties satisfied by the profiles $(U^{\, m,\pm},\psi^{\, m+1})_{m \ge 1}$ enable 
us to construct high order approximate solutions for the original free boundary value problem \eqref{s3-equations_nappes_MHD}, 
which will almost complete the proof of Theorem \ref{thm_principal}.

\section{High order approximate solutions}

We follow the notation of Theorem \ref{thm_principal} and define the approximate solutions:
\begin{align*}
\psi_\eps^{{\rm app},M} (t,x') \, &:= \, \psi^0 +\eps \, \psi^1 +\sum_{m=2}^{\, m+1} \eps^m \, \psi^m \left( t,x',\dfrac{\tau \, t +\xi'\cdot x'}{\eps} \right) \, ,\\
U_\eps^{{\rm app},M,\pm}(t,x) \, &:= \, U^{0,\pm} +\sum_{m=1}^M \eps^m \, U^{\, m,\pm} \left( t,x',x_3-\chi(x_3) \, \psi_\eps^{{\rm app},M}, 
\dfrac{x_3 -\psi_\eps^{{\rm app},M}}{\eps},\dfrac{\tau \, t +\xi'\cdot x'}{\eps} \right) \, ,
\end{align*}
where the profiles $(U^{\mu+1,\pm},\psi^{\mu+1})_{\mu=0,\dots,M}$ satisfy the condition $(H(M))$ stated at the beginning of Chapter 
\ref{chapter5}, and in the definition of $U_\eps^{{\rm app},M,\pm}$, the approximate front $\psi_\eps^{{\rm app},M}$ is evaluated at $(t,x')$ 
(as should be clear for the reader who has gone this far, at least do we hope so).

Let us start with the easiest estimate, which is the one on $\Gamma^\pm$. Restricting to $x_3=1$ in the definition of 
$u_{3,\eps}^{{\rm app},M,\pm}$, we get:
\begin{align*}
u_{3,\eps}^{{\rm app},M,\pm} \, = \, & \, 
\sum_{m=1}^M \eps^m \, u_3^{\, m,\pm} \left( t,x',1,\dfrac{1-\psi_\eps^{{\rm app},M}}{\eps},\dfrac{\tau \, t +\xi'\cdot x'}{\eps} \right) \\
\, = \, & \, 
\sum_{m=1}^M \eps^m \, u_{3,\star}^{\, m,\pm} \left( t,x',1,\dfrac{1-\psi_\eps^{{\rm app},M}}{\eps},\dfrac{\tau \, t +\xi'\cdot x'}{\eps} \right) \, ,
\end{align*}
where the second equality follows from $(H(M)-3)$. Using the exponential decay in $Y_3$ of functions in $S_\star^\pm$ and the crude 
bound
$$
\| \psi_\eps^{{\rm app},M} \|_{L^\infty} \, \le \, C \, \eps^2 \, ,
$$
the conclusion of Theorem \ref{thm_principal} on the error terms $R_{b,\eps}^{\, 3,\pm}$ and $R_{b,\eps}^{4,\pm}$ follows immediately. 
Let us also observe that the latter uniform bound yields $\chi(\psi_\eps^{{\rm app},M})=1$ for any sufficiently small $\eps$, which we 
assume to hold from now on.

Let us proceed with the error terms for the jump conditions on the (approximate) current vortex sheet. We compute:
\begin{align*}
H_\eps^{{\rm app},M,\pm}|_{\Gamma_\eps^{{\rm app},M}(t)} \cdot N_\eps^{{\rm app},M} \, & \, = \, 
H_{3,\eps}^{{\rm app},M,\pm}|_{\Gamma_\eps^{{\rm app},M}(t)} 
-\p_{x_j} \psi_\eps^{{\rm app},M} \, H_{j,\eps}^{{\rm app},M,\pm}|_{\Gamma_\eps^{{\rm app},M}(t)} \\
& \, = \, \sum_{m=1}^M \eps^m \, H_3^{\, m,\pm} \left( t,x',0,0,\dfrac{\tau \, t +\xi'\cdot x'}{\eps} \right) \\
& \quad \quad -\p_{x_j} \psi_\eps^{{\rm app},M} \, \left( 
H_j^{0,\pm} +\sum_{m=1}^M \eps^m \, H_j^{\, m,\pm} \left( t,x',0,0,\dfrac{\tau \, t +\xi'\cdot x'}{\eps} \right) \right) \, .
\end{align*}
Expanding the partial derivative $\p_{x_j} \psi_\eps^{{\rm app},M}$ with respect to $\eps$:
$$
\p_{x_j} \psi_\eps^{{\rm app},M} \, = \, \sum_{m=1}^M \eps^m \, \xi_j \, \p_\theta \psi^{\, m+1} \left( t,x',\dfrac{\tau \, t +\xi'\cdot x'}{\eps} \right) 
+\sum_{m=2}^{\, m+1} \eps^m \, \p_{y_j} \psi^m \left( t,x',\dfrac{\tau \, t +\xi'\cdot x'}{\eps} \right) \, ,
$$
and collecting terms, we end up with:
\begin{align*}
H_\eps^{{\rm app},M,\pm}|_{\Gamma_\eps^{{\rm app},M}(t)} \cdot N_\eps^{{\rm app},M} 
\, = \, & \, \sum_{m=1}^M \eps^m \, \left( H_3^{\, m,\pm}|_{y_3=Y_3=0} -b^\pm \, \p_\theta \psi^{\, m+1} -G_2^{\, m-1,\pm} \right) 
\left( t,x',\dfrac{\tau \, t +\xi'\cdot x'}{\eps} \right)  \\
& \, +O(\eps^{\, m+1}) \\
\, = \, & \, O(\eps^{\, m+1}) \, ,
\end{align*}
where the final conclusion comes from the fact that the profiles $(U^{\mu+1,\pm},\psi^{\mu+1})_{\mu=0,\dots,M}$ satisfy the fast 
problems $(H(M)-2)$ (in particular the boundary conditions on $\Gamma_0$ in these fast problems. With the notation of Theorem 
\ref{thm_principal}, we have obtained:
$$
\sup_{t \in [0,T] \, , \, x \in \Gamma_\eps^{{\rm app},M}(t)} \, |R_{b,\eps}^{\, 2,\pm}| \, = \, O(\eps^{\, m+1}) \, ,
$$
and the estimate for $R_{b,\eps}^{\, 1,\pm}$ follows similarly.

It remains to estimate the error terms in the partial differential equations that should be satisfied on either side of the current vortex 
sheet. We begin with the divergence constraints and compute (here $\nabla \cdot$ refers to the divergence with respect to $x$ in 
either of the domains $\Omega_\eps^{{\rm app},M,\pm}(t)$, and we write it down explicitly when a divergence with respect to $y$ 
is taken):
\begin{align*}
\nabla \cdot u_\eps^{{\rm app},M,\pm} \, = \, & \, \sum_{m=0}^{\, m-1} \eps^m \, 
\left( \p_{Y_3} u_3^{\, m+1,\pm} +\xi_j \, \p_\theta u_j^{\, m+1,\pm} +\nabla_y \cdot u^{\, m,\pm} \right. \\
& \, \qquad \left. -\sum_{\ell_1+\ell_2=m+2} \p_\theta \psi^{\ell_1} \, \xi_j \, \p_{Y_3} u_j^{\ell_2,\pm} 
-\sum_{\ell_1+\ell_2=m+1} \p_{y_j} \psi^{\ell_1} \, \p_{Y_3} u_j^{\ell_2,\pm} \right) (...) \\
& \, -\sum_{m=1}^{\, m-1} \eps^m \, 
\chi(x_3) \, \p_{y_3} u_j^{\, m,\pm}(...) \, \p_{x_j} \psi_\eps^{{\rm app},M} \left( t,x',\dfrac{\tau \, t +\xi'\cdot x'}{\eps} \right) \\
& \, -\sum_{m=1}^{\, m-1} \eps^m \, 
\chi'(x_3) \, \p_{y_3} u_3^{\, m,\pm}(...) \, \psi_\eps^{{\rm app},M} \left( t,x',\dfrac{\tau \, t +\xi'\cdot x'}{\eps} \right) 
+O(\eps^M) \, ,
\end{align*}
where the short notation $(...)$ is a substitute for the evaluation at:
$$
\left( t,x',x_3-\chi(x_3) \, \psi_\eps^{{\rm app},M},\dfrac{x_3 -\psi_\eps^{{\rm app},M}}{\eps},\dfrac{\tau \, t +\xi'\cdot x'}{\eps} \right) \, .
$$
It remains to expand $\chi(x_3)$ and $\chi'(x_3)$ with respect to $\eps$ by using the inverse map of:
$$
x_3 \, \longmapsto \, y_3 \, := \, x_3 -\chi(x_3) \, \psi_\eps^{{\rm app},M} (t,x') \, .
$$
The expansion of $\chi(x_3)$ and $\chi'(x_3)$ with respect to $\eps$ is given by the so-called Lagrange inversion formula, which is 
recalled in Lemma \ref{lemLagrange} of Appendix \ref{appendixB}, and which ultimately gives:
\begin{equation*}
\nabla \cdot u_\eps^{{\rm app},M,\pm} \, = \, 
\sum_{m=0}^{\, m-1} \eps^m \, \left( \p_{Y_3} u_3^{\, m+1,\pm} +\xi_j \, \p_\theta u_j^{\, m+1,\pm} -F_7^{\, m,\pm} \right) (...) +O(\eps^M) 
\, = \, O(\eps^M) \, ,
\end{equation*}
where the final conclusion comes again from the fact that the profiles $(U^{\mu+1,\pm},\psi^{\mu+1})_{\mu=0,\dots,M}$ satisfy the fast 
problems $(H(M)-2)$. We have thus obtained the estimate:
$$
\sup_{t \in [0,T] \, ,x \in \Omega_\eps^{{\rm app},M,\pm}(t)} \, |R_\eps^{\, 3,\pm}| \, = \, O(\eps^M) \, ,
$$
and the estimate for the error term $R_\eps^{4,\pm}$ follows similarly.

Since we already have $L^\infty$ estimates for the divergence of the vector fields $u_\eps^{{\rm app},M,\pm}$ and $H_\eps^{{\rm app},M,\pm}$, 
the estimate for the error terms $R_\eps^{\, 1,\pm}$ and $R_\eps^{\, 2,\pm}$ of Theorem \ref{thm_principal} will follow from an estimate of the 
type:
$$
\sup_{t \in [0,T] \, , \, x \in \Omega_\eps^{{\rm app},M,\pm}(t)} \, 
\big| A_0 \, \p_t U_\eps^{{\rm app},M,\pm} +\p_{x_\alpha} f_\alpha (U_\eps^{{\rm app},M,\pm}) \big| \, = \, O(\eps^M) \, ,
$$
where we recall that the fluxes $f_\alpha$ correspond to the conservative form of the incompressible MHD equations \eqref{s3-MHD_conservative}. 
Since the fluxes $f_\alpha$ are quadratic, we have:
$$
f_\alpha (U_\eps^{{\rm app},M,\pm}) \, = \, f_\alpha(U^{0,\pm}) +A_\alpha^\pm \cdot (U_\eps^{{\rm app},M,\pm}-U^{0,\pm}) 
+\dfrac{1}{2} \, \bA_\alpha (U_\eps^{{\rm app},M,\pm}-U^{0,\pm},U_\eps^{{\rm app},M,\pm}-U^{0,\pm}) \, .
$$
It then remains to expand the quantity:
$$
A_0 \, \p_t U_\eps^{{\rm app},M,\pm} +\p_{x_\alpha} f_\alpha (U_\eps^{{\rm app},M,\pm}) \, ,
$$
with respect to $\eps$ (expanding also $\chi(x_3)$ and $\chi'(x_3)$ with respect to $\eps$), and use once again the fact that the profiles 
satisfy the fast problems $(H(M)-2)$. We feel free to skip the details and leave them to the interested reader. We have thus completed the 
proof of Theorem \ref{thm_principal}.

\chapter{The rectification phenomenon}
\label{chapter6}

The aim of this Chapter is to clarify the last property of Theorem \ref{thm_principal} which we have not proved so far, namely 
$\uU^{\, 2,\pm}=0$. More generally, we would like to examine whether one can construct a solution to the WKB cascade that 
is purely localized near the boundary at any amplitude scale. To be more precise, we are going to show that without any 
requirement on the initial front profile $\psi^2_0$, the residual component of the first corrector $U^{\, 2,\pm}$ vanishes (with, 
of course, suitably chosen initial data for several components). We can thus achieve $\uU^{\, 1,\pm}=\uU^{\, 2,\pm}=0$ as 
claimed in Theorem \ref{thm_principal} though, as we explain below, it seems unlikely to have  simultaneously:
$$
\uU^{\, 1,\pm} \, = \, \uU^{\, 2,\pm} \, = \, \uU^{\, 3,\pm} \, = \, 0 \, .
$$
However, we have not succeeded to get a complete proof of the fact $\uU^{\, 3,\pm} \neq 0$. We note still that the situation is 
quite different from the one in elastodynamics \cite{MarcouCRAS} where the first corrector generically has a nonzero residual 
component. This will not happen here due to the specific form of the leading profile \eqref{s3-U^1,pm_k_final} and orthogonality 
properties between some of its components. We therefore review how the first corrector $U^{\, 2,\pm}$ is constructed and explain 
why the residual component $\uU^{\, 2,\pm}$ vanishes. We then examine the construction of the subsequent correctors.

\section{The first corrector}

In the proof of Theorem \ref{thm_principal}, we have first in Chapter \ref{chapter4} identified the leading profile $U^{\, 1,\pm}$. By imposing, 
which was shown to be compatible with all other constraints, the initial conditions:
$$
(\widehat{\uu}^{\, 1,\pm}(0),\widehat{\uH}^{\, 1,\pm}(0)) \big|_{t=0} \, = \, 0 \, ,\quad \Pi \, \widehat{U}_\star^{\, 1,\pm}(0) \big|_{t=0} \, = \, 0 \, ,
$$
we have been led to the decomposition \eqref{s3-U^1,pm_k_final} for $U^{\, 1,\pm}$. In particular, the leading profile satisfies:
$$
U^{\, 1,\pm} \, = \, U_\star^{\, 1,\pm} \, ,\quad \quad \widehat{U}_\star^{\, 1,\pm}(0) \, = \, 0 \, .
$$
The oscillating modes of the leading front $\psi^2$ are governed by a nonlocal Hamilton-Jacobi equation and the initial condition $\psi^2_0$ 
for $\psi^2$ is a nonzero function in $H^\infty_\sharp$ (the space of $H^\infty$ functions with zero mean in $\theta$). As observed in Chapter 
\ref{chapter4}, the choice of how we have lifted $U^{\, 1,\pm}$ from $y_3=0$ to $y_3 \in I^\pm$ contained some arbitrariness, but it presents 
the nice property of satisfying $\p_{y_3} U^{\, 1,\pm}|_{y_3=0}=0$. Using the expression \eqref{s3-U^1,pm_k_final} as well as 
\eqref{appA-defRLkpm}, we get for all Fourier mode $k \neq 0$:
\begin{multline}
\label{expression_profil_principal}
\widehat{U}^{\, 1,\pm} (t,y',0,Y_3,k) \, = \, \widehat{U}_\star^{\, 1,\pm} (t,y',0,Y_3,k) \\
= \, \pm \, \widehat{\psi}^{\, 2} (t,y',k) \, 
\Big( |k| \, \xi_1 \, c^\pm, |k| \, \xi_2 \, c^\pm, \pm i \, k \, c^\pm, |k| \, \xi_1 \, b^\pm, |k| \, \xi_2 \, b^\pm, \pm i \, k \, b^\pm, 
|k| \, ((b^\pm)^2-(c^\pm)^2) \Big)^T \, {\rm e}^{\mp \, |k| \, Y_3} \, .
\end{multline}

Let us now examine how we have constructed the first corrector $U^{\, 2,\pm}$ and explain why our choice of initial conditions in Theorem 
\ref{thm_principal} yields $\uU^{\, 2,\pm} =0$. This means that rectification, if it occurs, does not arise at the level of the first corrector. This 
is one major difference with elastodynamics \cite{MarcouCRAS}. The construction of $\uU^{\, 2,\pm}$ splits in several steps.
\bigskip

$\bullet$ \underline{Step 1}. The oscillating modes of the residual component $\uU^{\, 2,\pm}$. Recalling the expression 
\eqref{s3-def_terme_source_F^1,pm} of $F^{\, 1,\pm}$ and the fact that $\uU^{\, 1,\pm}$ vanishes, we obtain $\uF^{\, 1,\pm}=0$. 
Since the residual component $\uU^{\, 2,\pm}$ must satisfy:
$$
\cA^\pm \, \p_\theta \uU^{\, 2,\pm} \, = \, \uF^{\, 1,\pm} \, = \, 0 \, ,
$$
and since $\cA^\pm$ is invertible, we get $\p_\theta \uU^{\, 2,\pm}=0$. This means that residual component $\uU^{\, 2,\pm}$ of the first corrector 
$U^{\, 2,\pm}$ reduces to the slow mean $\widehat{\uU}^{\, 2,\pm}(0)$.
\bigskip

$\bullet$ \underline{Step 2}. Collecting the equations for the slow mean. We are now going to examine the equations that must be satisfied 
by the slow mean $\widehat{\uU}^{\, 2,\pm}(0)$. Using the expression \eqref{terme_source_F^m,pmbarre} for $m=2$, a necessary condition 
for solving \eqref{dev_BKW_int} with $m=2$ is $\widehat{\uF}^{\, 2,\pm}(0)=0$ together with the divergence constraint $\widehat{\uF}_8^{\, 2,\pm} 
(0)=0$ for the magnetic field. Since we already know that the residual component $\uU^{\, 1,\pm}$ vanishes, we obtain the linear homogeneous 
system:
\begin{equation}
\label{slow_mean_2_edp}
\begin{cases}
L_s^\pm(\p) \widehat{\uU}^{\, 2,\pm}(0) \, = \, 0 \, ,& \text{\rm in } \Omega_0^\pm \, ,\\
\nabla \cdot \widehat{\uH}^{\, 2,\pm}(0) \, = \, 0 \, ,& \text{\rm in } \Omega_0^\pm \, ,
\end{cases}
\end{equation}
which corresponds, with the notation of \eqref{slow_mean_m+1_edp}, to $\bF^{\, 1,\pm}=0$ and $\bF_8^{\, 1,\pm}=0$. The boundary conditions 
on $\Gamma^\pm$ for $\widehat{\uU}^{\, 2,\pm}(0)$ correspond to imposing \eqref{s3-cond_bords_fixes_uH_3^m,pm} with $m=1$ and for the 
zero Fourier mode in $\theta$ only, namely:
\begin{equation}
\label{slow_mean_2_hautbas}
\widehat{\uu}_3^{\, 2,\pm}(0)|_{\Gamma^\pm} \, = \, \widehat{\uH}_3^{\, 2,\pm}(0)|_{\Gamma^\pm} \, = \, 0 \, .
\end{equation}
The boundary conditions \eqref{slow_mean_m+1_saut-bis} on $\Gamma_0$, for $m=1$, read explicitly:
\begin{equation}
\label{slow_mean_2_saut}
\begin{cases}
\widehat{\uu}_3^{\, 2,\pm}(0)|_{\Gamma_0} \, = \, (\p_t +u_j^{0,\pm} \, \p_{y_j}) \widehat{\psi}^{\, 2}(0) 
+{\bf c}_0 \, \big\{ \p_\theta \psi^2 \, \xi_j \, u_j^{\, 1,\pm}|_{y_3=Y_3=0} \big\} -\widehat{u}_{3,\star}^{\, 2,\pm}(0)|_{y_3=Y_3=0} \, ,& \\[1ex]
\widehat{\uH}_3^{\, 2,\pm}(0)|_{\Gamma_0} \, = \, H_j^{0,\pm} \, \p_{y_j} \widehat{\psi}^{\, 2}(0) 
+{\bf c}_0 \, \big\{ \p_\theta \psi^2 \, \xi_j \, H_j^{\, 1,\pm}|_{y_3=Y_3=0} \big\} -\widehat{H}_{3,\star}^{\, 2,\pm}(0)|_{y_3=Y_3=0} \, ,& \\[1ex]
\widehat{\uq}^{\, 2,+}(0)|_{\Gamma_0} -\widehat{\uq}^{\, 2,-}(0)|_{\Gamma_0} \, = \, 
-\widehat{q}_\star^{\, 2,+}(0)|_{y_3=Y_3=0} +\widehat{q}_\star^{\, 2,-}(0)|_{y_3=Y_3=0} \, .
\end{cases}
\end{equation}
Let us eventually recall that the normalization condition \eqref{expression_I2(t)} for the total pressure reduces to:
$$
\forall \, t \in [0,T] \, ,\quad 
\int_{\Omega_0^+} \widehat{\uq}^{\, 2,+}(t,y,0) \, {\rm d}y \, + \, \int_{\Omega_0^-} \widehat{\uq}^{\, 2,-}(t,y,0) \, {\rm d}y \, = \, 0 \, ,
$$
because the fast mean of the leading profile vanishes.

In order to go further in the determination of the slow mean $\widehat{\uU}^{\, 2,\pm}(0)$, we need to compute the source terms on 
the right hand side in \eqref{slow_mean_2_saut}, and therefore determine the noncharacteristic components of the fast mean of 
$U^{\, 2,\pm}$.
\bigskip

$\bullet$ \underline{Step 3}. The noncharacteristic components of the fast mean. Recalling the expression \eqref{s3-def_terme_source_F_8^1,pm} 
of the source term $F_8^{\, 1,\pm}$, we have:
\begin{align*}
\p_{Y_3} \widehat{H}_{3,\star}^{\, 2,\pm}(0) \big|_{y_3=0} \, 
& \, = \, - \big( \nabla \cdot \widehat{H}_\star^{\, 1,\pm}(0) \big) \big|_{y_3=0} \, + \, 
{\bf c}_0 \, \Big\{ \p_\theta \psi^2 \, \xi_j \, \p_{Y_3} H_{j,\star}^{\, 1,\pm} \Big\} \big|_{y_3=0} \\
& \, = \, \p_{Y_3} \, {\bf c}_0 \, \Big\{ \p_\theta \psi^2 \, \xi_j \, H_{j,\star}^{\, 1,\pm} \Big\} \big|_{y_3=0} \, ,
\end{align*}
where we have used the fact that the fast mean of the leading profile vanishes. Integrating with respect to $Y_3$ with the zero limit at 
infinity, we get:
$$
\widehat{H}_{3,\star}^{\, 2,\pm}(0) \big|_{y_3=0} \, = \, {\bf c}_0 \, \Big\{ \p_\theta \psi^2 \, \xi_j \, H_{j,\star}^{\, 1,\pm} \Big\} \big|_{y_3=0} \, ,
$$
and using the expression \eqref{expression_profil_principal}, this yields:
$$
\widehat{H}_{3,\star}^{\, 2,\pm}(0) \big|_{y_3=0} \, = \, \mp b^\pm \, \sum_{k \neq 0} 
\big( i \, k \, \widehat{\psi}^{\, 2}(-k) \big) \, \big( |k| \, \widehat{\psi}^{\, 2}(k) \big) \, {\rm e}^{\mp \, |k| \, Y_3} \, = \, 0 \, ,
$$
where the conclusion follows from the change of index $k \rightarrow -k$. In the same way, we can obtain 
$\widehat{u}_{3,\star}^{\, 2,\pm}(0)|_{y_3=0} =0$. This orthogonality property is specific to the current vortex sheet problem 
we are considering here and to the explicit form of the leading profile. This is where the analysis differs from elastodynamics. 
Note that the two relations $\widehat{u}_{3,\star}^{\, 2,\pm}(0)|_{y_3=0} =\widehat{H}_{3,\star}^{\, 2,\pm}(0)|_{y_3=0}=0$ hold not 
only for $Y_3=0$ but for any $Y_3 \in \bR^\pm$.

Let us now study the fast mean of the total pressure. Recalling the expression \eqref{s3-def_terme_source_F^1,pm}, of which we compute 
the third component of the fast mean, we get:
\begin{align*}
\p_{Y_3} \widehat{q}_\star^{\, 2,\pm}(0) \big|_{y_3=0} \, = \, \widehat{F}_{3,\star}^{\, 1,\pm}(0) \big|_{y_3=0} \, = \, & \, 
{\bf c}_0 \, \Big\{ \p_\theta \psi^2 \, \cA^\pm \, \p_{Y_3} U^{\, 1,\pm} -\dfrac{1}{2} \, \p_{Y_3} \bA_3 (U^{\, 1,\pm},U^{\, 1,\pm}) \Big\}_3 \big|_{y_3=0} \\
= \, & \, \p_{Y_3} \, {\bf c}_0 \, \Big\{ \p_\theta \psi^2 \, \cA^\pm \, U^{\, 1,\pm} -\dfrac{1}{2} \, \bA_3 (U^{\, 1,\pm},U^{\, 1,\pm}) \Big\}_3 \big|_{y_3=0} 
\end{align*}
and we thus obtain the expression:
\begin{align}
\widehat{q}_\star^{\, 2,\pm}(0) \big|_{y_3=0} \, 
& \, = \, {\bf c}_0 \, \Big\{ \p_\theta \psi^2 \, \big( c^\pm \, u_3^{\, 1,\pm} -b^\pm \, H_3^{\, 1,\pm} \big) +(H_3^{\, 1,\pm})^2 -(u_3^{\, 1,\pm})^2 
\Big\} \big|_{y_3=0} \notag \\
& \, = \, \big( (c^\pm)^2 -(b^\pm)^2 \big) \, \sum_{k \neq 0} k^2 \, \widehat{\psi}^{\, 2}(-k) \, \widehat{\psi}^{\, 2}(k) \, \big( {\rm e}^{\mp \, |k| \, Y_3} 
-{\rm e}^{\mp \, 2 \, |k| \, Y_3} \big) \, .\label{expression_correcteur_q2}
\end{align}
In particular, there holds:
$$
\widehat{q}_\star^{\, 2,\pm}(0) \big|_{y_3=Y_3=0} \, = \, 0 \, .
$$
Computing similarly two other terms on the right hand side of \eqref{slow_mean_2_saut}, we can rewrite \eqref{slow_mean_2_saut} more 
simply as:
\begin{equation}
\label{slow_mean_2_saut'}
\begin{cases}
\widehat{\uu}_3^{\, 2,\pm}(0)|_{\Gamma_0} \, = \, \big( \p_t +u_j^{0,\pm} \, \p_{y_j} \big) \widehat{\psi}^{\, 2}(0) \, ,& \\[0.5ex]
\widehat{\uH}_3^{\, 2,\pm}(0)|_{\Gamma_0} \, = \, H_j^{0,\pm} \, \p_{y_j} \widehat{\psi}^{\, 2}(0) \, ,& \\[0.5ex]
\widehat{\uq}^{\, 2,+}(0)|_{\Gamma_0} -\widehat{\uq}^{\, 2,-}(0)|_{\Gamma_0} \, = \, 0 \, .
\end{cases}
\end{equation}
The fact that all source terms on the right hand side of \eqref{slow_mean_2_saut'} have been dropped out for orthogonality reasons 
explains why the residual component of $U^{\, 2,\pm}$ will ultimately vanish.
\bigskip

$\bullet$ \underline{Step 4}. The Laplace problem for the total pressure. Making the equations in \eqref{slow_mean_2_edp} explicit, the 
slow mean $\widehat{\uU}^{\, 2,\pm}(0)$ must satisfy the linear system:
\begin{equation*}
\begin{cases}
\big( \p_t  +u_j^{0,\pm} \, \p_{y_j} \big) \widehat{\uu}_\alpha^{\, 2,\pm}(0) +u_\alpha^{0,\pm} \, \nabla \cdot \widehat{\uu}^{\, 2,\pm}(0) 
-H_j^{0,\pm} \, \p_{y_j} \widehat{\uH}_\alpha^{\, 2,\pm}(0) -H_\alpha^{0,\pm} \, \nabla \cdot \widehat{\uH}_\alpha^{\, 2,\pm}(0) 
+\p_{y_\alpha} \widehat{\uq}^{\, 2,\pm}(0) \, = \, 0 \, , & \\
\big( \p_t  +u_j^{0,\pm} \, \p_{y_j} \big) \widehat{\uH}_\alpha^{\, 2,\pm}(0) -u_\alpha^{0,\pm} \, \nabla \cdot \widehat{\uH}^{\, 2,\pm}(0) 
-H_j^{0,\pm} \, \p_{y_j} \widehat{\uu}_\alpha^{\, 2,\pm}(0) +H_\alpha^{0,\pm} \, \nabla \cdot \widehat{\uu}_\alpha^{\, 2,\pm}(0) \, = \, 0 \, , & \\
\nabla \cdot \widehat{\uu}^{\, 2,\pm}(0) \, = \, \nabla \cdot \widehat{\uH}^{\, 2,\pm}(0) \, = \, 0 \, , &
\end{cases}
\end{equation*}
with the boundary conditions \eqref{slow_mean_2_hautbas} and \eqref{slow_mean_2_saut'}. We thus find that the total pressure corrector 
must satisfy the coupled Laplace problem:
\begin{equation*}
\begin{cases}
-\Delta \, \widehat{\uq}^{\, 2,\pm}(0) \, = \, 0 \, , \qquad \qquad \text{\rm in $\Omega_0^\pm$,} & \\
\p_{y_3} \widehat{\uq}^{\, 2,\pm}(0)|_{\Gamma^\pm} \, = \, 0 \, , & \\
\p_{y_3} \widehat{\uq}^{\, 2,\pm}(0)|_{\Gamma_0} \, = \, -(\p_t  +u_j^{0,\pm} \, \p_{y_j}) \, (\p_t  +u_{j'}^{0,\pm} \, \p_{y_{j'}}) \widehat{\psi}^{\, 2}(0) 
+H_j^{0,\pm} \, H_{j'}^{0,\pm} \, \p_{y_j} \p_{y_{j'}} \widehat{\psi}^{\, 2}(0) \, , & \\
\widehat{\uq}^{\, 2,+}(0)|_{\Gamma_0} -\widehat{\uq}^{\, 2,-}(0)|_{\Gamma_0} \, = \, 0 \, . &
\end{cases}
\end{equation*}
Reproducing the same calculations as in the corresponding Section of Chapter \ref{chapter5}, we find that the nonzero Fourier modes 
in $y'$ of $\widehat{\psi}^{\, 2}(0)$ must satisfy the homogeneous linear wave equation:
\begin{multline*}
\big( \p_t  +u_j^{0,+} \, \p_{y_j} \big) \, \big( \p_t  +u_{j'}^{0,+} \, \p_{y_{j'}} \big) \Psi^2 
+\big( \p_t  +u_j^{0,-} \, \p_{y_j} \big) \, \big( \p_t  +u_{j'}^{0,-} \, \p_{y_{j'}} \big) \Psi^2 \\
-H_j^{0,+} \, H_{j'}^{0,+} \, \p_{y_j} \p_{y_{j'}} \Psi^2 -H_j^{0,-} \, H_{j'}^{0,-} \, \p_{y_j} \p_{y_{j'}} \Psi^2 \, = \, 0 \, ,
\end{multline*}
on the boundary $\Gamma_0$, and the mean of $\widehat{\psi}^{\, 2}(0)$ on $\Gamma_0$ vanishes. By choosing, as in Theorem 
\ref{thm_principal}, zero initial conditions for the oscillating modes $\Psi^2$, we end up with $\widehat{\psi}^{\, 2}(0) =0$ and consequently 
$\widehat{\uq}^{\, 2}(0)=0$ thanks to the normalization condition \eqref{expression_I2(t)}.
\bigskip

$\bullet$ \underline{Step 5}. Conclusion. Since the slow mean of the residual total pressure $\uq^{\, 2,\pm}$ vanishes, it is not difficult to 
show that \eqref{slow_mean_2_edp} implies $\widehat{\uU}^{\, 2,\pm}(0)=0$. This is the same argument as when we have determined 
the slow mean of the leading profile in Chapter \ref{chapter4}. In other words, we have just shown that for any initial condition $\psi^2_0 
\in H^\infty_\sharp$ in \eqref{s3-def_cond_init_oscil_psi}, we can construct a solution to the WKB cascade that satisfies (so far):
\begin{itemize}
 \item $\uU^{\, 1,\pm} \, = \, \uU^{\, 2,\pm} \, = \, 0$ and $\widehat{U}_\star^{\, 1,\pm}(0) \, = \, 0$,
 
 \item $\widehat{u}_{3,\star}^{\, 2,\pm}(0)|_{y_3=0} \, = \, \widehat{H}_{3,\star}^{\, 2,\pm}(0)|_{y_3=0} \, = \, \widehat{q}_\star^{\, 2,\pm}(0)|_{y_3=Y_3=0} \, = \, 0$.
\end{itemize}
For arbitrary $Y_3$, the fast mean $\widehat{q}_\star^{\, 2,\pm}(0)|_{y_3=0}$ is given by \eqref{expression_correcteur_q2}.

We are now going to examine whether it is possible to construct a second corrector $U^{\, 3,\pm}$ in the WKB expansion \eqref{s3-def_dvlpt_BKW} 
that satisfies $\uU^{\, 3,\pm}=0$.

\section{The second corrector}

We shall not give in this Section a rigorous proof of $\uU^{\, 3,\pm} \neq 0$ but we shall rather explain why this fact is quite likely to happen. 
First of all, since both $\uU^{\, 1,\pm}$ and $\uU^{\, 2,\pm}$ vanish, we compute $\uF^{\, 2,\pm}=0$ and therefore the residual component 
of the second corrector $\uU^{\, 3,\pm}$ reduces to the slow mean only:
$$
\uU^{\, 3,\pm} \, = \, \widehat{\uU}^{\, 3,\pm}(0) \, .
$$
Computing the expression of $\uF^{\, 3,\pm}$ and $\uF_8^{\, 3,\pm}$, we then find that the slow mean $\widehat{\uU}^{\, 3,\pm}(0)$ satisfies the 
linear homogeneous system:
\begin{equation*}
\begin{cases}
L_s^\pm(\p) \widehat{\uU}^{\, 3,\pm}(0) \, = \, 0 \, ,& \text{\rm in } \Omega_0^\pm \, ,\\
\nabla \cdot \widehat{\uH}^{\, 3,\pm}(0) \, = \, 0 \, ,& \text{\rm in } \Omega_0^\pm \, .
\end{cases}
\end{equation*}
The boundary conditions for $\widehat{\uU}^{\, 3,\pm}$ read:
\begin{equation*}
\widehat{\uu}_3^{\, 3,\pm}(0)|_{\Gamma^\pm} \, = \, \widehat{\uH}_3^{\, 3,\pm}(0)|_{\Gamma^\pm} \, = \, 0 \, ,
\end{equation*}
\begin{equation}
\label{slow_mean_3_saut}
\begin{cases}
\widehat{\uu}_3^{\, 3,\pm}(0)|_{\Gamma_0} \, = \, \widehat{G}_1^{\, 2,\pm}(0) -\widehat{u}_{3,\star}^{\, 3,\pm}(0)|_{y_3=Y_3=0} \, ,& \\[1ex]
\widehat{\uH}_3^{\, 2,\pm}(0)|_{\Gamma_0} \, = \, \widehat{G}_2^{\, 2,\pm}(0) -\widehat{H}_{3,\star}^{\, 3,\pm}(0)|_{y_3=Y_3=0} \, ,& \\[1ex]
\widehat{\uq}^{\, 3,+}(0)|_{\Gamma_0} -\widehat{\uq}^{\, 3,-}(0)|_{\Gamma_0} \, = \, 
-\widehat{q}_\star^{\, 3,+}(0)|_{y_3=Y_3=0} +\widehat{q}_\star^{\, 3,-}(0)|_{y_3=Y_3=0} \, .
\end{cases}
\end{equation}
If we follow the arguments for the construction of $\uU^{\, 1,\pm}$ and $\uU^{\, 2,\pm}$, the only hope for proving $\widehat{\uU}^{\, 3,\pm}(0) 
\neq 0$ is to show that the right hand side of \eqref{slow_mean_3_saut} does not reduce to the linear terms in $\widehat{\psi}^{\, 3}(0)$. This 
is where the algebra becomes quite involved. Some terms can be made explicit though. For instance, the fast mean of the total pressure 
$q^{\, 3,\pm}$ is obtained by solving:
$$
\p_{Y_3} \widehat{q}_\star^{\, 3,\pm}(0) \, = \, F_{3,\star}^{\, 2,\pm}(0) \, ,
$$
but unfortunately, this yields after integration with respect to $Y_3$ (and quite a few more calculations):
\begin{align*}
\widehat{q}_\star^{\, 3,\pm}(0)|_{y_3=0} \, = \, {\bf c}_0 \, \Big\{ \, & \, 
\p_\theta \psi^3 \, \big( c^\pm \, u_3^{\, 1,\pm} -b^\pm \, H_3^{\, 1,\pm} \big) +H_3^{\, 1,\pm} \, H_3^{\, 2,\pm} -u_3^{\, 1,\pm} \, u_3^{\, 2,\pm} \\
& \, +\big( \p_t +u_j^{0,\pm} \, \p_{y_j} \big) \psi^2 \, u_3^{\, 1,\pm} -H_j^{0,\pm} \, \p_{y_j} \psi^2 \, H_3^{\, 1,\pm} \\
& \, +\p_\theta \psi^2 \, \xi_j \, u_j^{\, 1,\pm} u_3^{\, 1,\pm} -\p_\theta \psi^2 \, \xi_j \, H_j^{\, 1,\pm} H_3^{\, 1,\pm} \Big\} \big|_{y_3=0} \, ,
\end{align*}
and from the latter expression, we get:
$$
\widehat{q}_\star^{\, 3,\pm}(0)|_{y_3=Y_3=0} \, = \, 0 \, .
$$
This means that the last boundary condition in \eqref{slow_mean_3_saut} reads:
$$
\widehat{\uq}^{\, 3,+}(0)|_{y_3=0} -\widehat{\uq}^{\, 3,-}(0)|_{y_3=0} \, = \, 0 \, .
$$
Even the normalization condition does not help since we can use \eqref{expression_correcteur_q2} to compute :
$$
\int_{\bT^2 \times \R^+} \widehat{q}_\star^{\, 2,+} (t,y',0,Y_3,0) \, {\rm d}y' \, {\rm d}Y_3 \, + \, 
\int_{\bT^2 \times \R^-} \widehat{q}_\star^{\, 2,-} (t,y',0,Y_3,0) \, {\rm d}y' \, {\rm d}Y_3 \, = \, 0 \, ,
$$
which means that the slow mean of the total pressure should satisfy:
$$
\int_{\Omega_0^+} \widehat{\uq}^{\, 3,+}(t,y,0) \, {\rm d}y \, + \, 
\int_{\Omega_0^-} \widehat{\uq}^{\, 3,-}(t,y,0) \, {\rm d}y \, = \, 0 \, ,
$$
in order to verify the normalization condition $\cI^3(t)=0$.

Making the source terms $\widehat{G}_1^{\, 2,\pm}(0)$, $\widehat{G}_2^{\, 2,\pm}(0)$ as well as $\widehat{u}_{3,\star}^{\, 3,\pm}(0)|_{y_3=Y_3=0}$ 
and $\widehat{H}_{3,\star}^{\, 3,\pm}(0)|_{y_3=Y_3=0}$ explicit is much more difficult. This gives ultimately a sum of terms that are either quadratic 
or cubic in $\psi^2$ and a sum of products between $\psi^2$ and $\psi^3$. It is likely that the final result will not be zero, though we have not been 
able to complete the calculations.

If one could prove indeed that the slow mean of $U^{\, 3,\pm}$ does not vanish, then the expression \eqref{terme_source_F^m,pmbarre} shows 
that the source terms $\uF^{4,\pm}$ should have nonzero Fourier modes in $\theta$ that do not vanish, due to the product of $\p_\theta \psi^2$ 
with $\p_{y_3} \uU^{\, 3,\pm}$. It is then likely that $\uU^{5,\pm}$ will have nonzero Fourier modes in $\theta$, which explains why in Definition 
\ref{s3-def_espaces_fonctionnels} we have considered residual components in $\uS^\pm$ that also depend on the fast variable $\theta$ (as 
opposed for instance to \cite{Marcou,WW} where residual components only depend on the slow variables).

\appendix
\chapter{Linear and bilinear algebra}
\label{appendixA}

In this Appendix, we give explicit expressions for various matrices and right or left eigenvectors that are 
involved in the analysis of the WKB cascade. It should be remembered, see \eqref{s3-def_U^0,pm}, that 
the two constant states $U^{0,\pm}$ defining the reference steady current vortex sheet are given by:
$$
U^{0,\pm} \, = \, (u_1^{0,\pm},u_2^{0,\pm},0,H_1^{0,\pm},H_2^{0,\pm},0,0)^T \, .
$$
The fluxes $f_\alpha$ are defined by \eqref{s3-def_f_j}. The Jacobian matrices $A_\alpha^\pm ={\rm d}f_\alpha (U^{0,\pm})$, 
$\alpha=1,2,3$, are thus given by:
$$
A_1^\pm \, = \, \begin{bmatrix}
2\, u_1^{0,\pm} & 0 & 0 & -2 \, H_1^{0,\pm} & 0 & 0 & 1 \\
u_2^{0,\pm} & u_1^{0,\pm} & 0 & -H_2^{0,\pm} & -H_1^{0,\pm} & 0 & 0 \\
0 & 0 & u_1^{0,\pm} & 0 & 0 & -H_1^{0,\pm} & 0 \\
0 & 0 & 0 & 0 & 0 & 0 & 0 \\
H_2^{0,\pm} & -H_1^{0,\pm} & 0 & -u_2^{0,\pm} & u_1^{0,\pm} & 0 & 0 \\
0 & 0 & -H_1^{0,\pm} & 0 & 0 & u_1^{0,\pm} & 0 \\
1 & 0 & 0 & 0 & 0 & 0 & 0 \end{bmatrix} \, ,
$$
$$
A_2^\pm \, = \, \begin{bmatrix}
u_2^{0,\pm} & u_1^{0,\pm} & 0 & -H_2^{0,\pm} & -H_1^{0,\pm} & 0 & 0 \\
0 & 2\, u_2^{0,\pm} & 0 & 0 & -2 \, H_2^{0,\pm} & 0 & 1 \\
0 & 0 & u_2^{0,\pm} & 0 & 0 & -H_2^{0,\pm} & 0 \\
-H_2^{0,\pm} & H_1^{0,\pm} & 0 & u_2^{0,\pm} & -u_1^{0,\pm} & 0 & 0 \\
0 & 0 & 0 & 0 & 0 & 0 & 0 \\
0 & 0 & -H_2^{0,\pm} & 0 & 0 & u_2^{0,\pm} & 0 \\
0 & 1 & 0 & 0 & 0 & 0 & 0 \end{bmatrix} \, ,
$$
$$
A_3^\pm \, = \, \begin{bmatrix}
0 & 0 & u_1^{0,\pm} & 0 & 0 & -H_1^{0,\pm} & 0 \\
0 & 0 & u_2^{0,\pm} & 0 & 0 & -H_2^{0,\pm} & 0 \\
0 & 0 & 0 & 0 & 0 & 0 & 1 \\
0 & 0 & H_1^{0,\pm} & 0 & 0 & -u_1^{0,\pm} & 0 \\
0 & 0 & H_2^{0,\pm} & 0 & 0 & -u_2^{0,\pm} & 0 \\
0 & 0 & 0 & 0 & 0 & 0 & 0 \\
0 & 0 & 1 & 0 & 0 & 0 & 0 \end{bmatrix} \, .
$$
In particular, we get from these expressions the matrices ${\mathcal A}^\pm=\tau \, A_0+\xi_j \, A_j^\pm$ that enter the definition 
\eqref{s3-def_L(d)^pm} of the fast operators $\cL_f^\pm(\partial)$:
$$
{\mathcal A}^\pm \, = \, \begin{bmatrix}
c^\pm +\xi_1 \, u_1^{0,\pm} & \xi_2 \, u_1^{0,\pm} & 0 & -(b^\pm +\xi_1 \, H_1^{0,\pm}) & -\xi_2 \, H_1^{0,\pm} & 0 & \xi_1 \\
\xi_1 \, u_2^{0,\pm} & c^\pm +\xi_2 \, u_2^{0,\pm} & 0 & -\xi_1 \, H_2^{0,\pm} & -(b^\pm +\xi_2 \, H_2^{0,\pm}) & 0 & \xi_2 \\
0 & 0 & c^\pm & 0 & 0 & -b^\pm & 0 \\
-\xi_2 \, H_2^{0,\pm} & \xi_2 \, H_1^{0,\pm} & 0 & c^\pm -\xi_1 \, u_1^{0,\pm} & -\xi_2 \, u_1^{0,\pm} & 0 & 0 \\
\xi_1 \, H_2^{0,\pm} & -\xi_1 \, H_1^{0,\pm} & 0 & -\xi_1 \, u_2^{0,\pm} & c^\pm -\xi_2 \, u_2^{0,\pm} & 0 & 0 \\
0 & 0 & -b^\pm & 0 & 0 & c^\pm & 0 \\
\xi_1 & \xi_2 & 0 & 0 & 0 & 0 & 0 \end{bmatrix} \, ,
$$
where we recall the notations:
$$
c^\pm \, := \, \tau +\xi_j \, u_j^{0,\pm} \, ,\quad b^\pm \, := \, \xi_j \, H_j^{0,\pm} \, .
$$
We first collect several observations and useful formulas related to the above matrices.
\bigskip

$\bullet$ In many of the calculations below, we shall need to use the property $(c^\pm)^2 -(b^\pm)^2 \neq 0$. We now 
show why this property follows from Assumptions \eqref{s3-hyp_delta_neq_0} and \eqref{s3-hyp_tau_racine_det_lop}. 
Indeed let us assume for instance $(c^+)^2 -(b^+)^2 = 0$. Then because of \eqref{s3-hyp_tau_racine_det_lop}, we also 
have simultaneously $(c^-)^2 -(b^-)^2 = 0$. Using $c^\pm =\tau +a^\pm$, we get $\tau +a^+ =\vartheta^+ \, b^+$ and 
$\tau +a^- =\vartheta^- \, b^-$, with $\vartheta^+,\vartheta^- \in \{ -1,1\}$. Subtracting, we get $a^+-a^- =\vartheta^+ \, (b^+ -
\vartheta^+ \, \vartheta^- \, b^-)$, which implies either $|a^+-a^-|=|b^+-b^-|$ or $|a^+-a^-|=|b^+ +b^-|$. In any case, the 
latter relation is incompatible with \eqref{s3-hyp_delta_neq_0}.

The property $(c^\pm)^2 -(b^\pm)^2 \neq 0$ is used below to parametrize some eigenspaces for some given matrices. 
Were it not satisfied, these eigenspaces would not be one-dimensional any longer and this degeneracy would seem to 
rule to out the validity of our weakly nonlinear expansion.
\bigskip

$\bullet$ Multiplying on the left the matrix ${\mathcal A}^\pm$ by the upper triangular matrix:
\begin{equation*}
P^\pm \, := \, \begin{bmatrix}
1 & 0 & 0 & 0 & 0 & 0 & -u_1^{0,\pm} \\
0 & 1 & 0 & 0 & 0 & 0 & -u_2^{0,\pm} \\
0 & 0 & 1 & 0 & 0 & 0 & 0 \\
0 & 0 & 0 & 1 & 0 & 0 & -H_1^{0,\pm} \\
0 & 0 & 0 & 0 & 1 & 0 & -H_2^{0,\pm} \\
0 & 0 & 0 & 0 & 0 & 1 & 0 \\
0 & 0 & 0 & 0 & 0 & 0 & 1 \end{bmatrix} \, ,
\end{equation*}
we can easily show that both matrices ${\mathcal A}^\pm$ are invertible. Indeed, any vector $(v^\pm,B^\pm,p^\pm)^T$ 
in the kernel of ${\mathcal A}^\pm$ must satisfy:
\begin{equation*}
\begin{cases}
c^\pm \, v_1^\pm -b^\pm \, B_1^\pm -H_1^{0,\pm} \, \xi' \cdot (B')^\pm +\xi_1 \, p^\pm \, = \, 0 \, ,& \\
c^\pm \, v_2^\pm -b^\pm \, B_2^\pm -H_2^{0,\pm} \, \xi' \cdot (B')^\pm +\xi_2 \, p^\pm \, = \, 0 \, ,& \\
c^\pm \, v_3^\pm -b^\pm \, B_3^\pm \, = \, 0 \, ,& \\
-b^\pm \, v_1^\pm +c^\pm \, B_1^\pm -u_1^{0,\pm} \, \xi' \cdot (B')^\pm \, = \, 0 \, ,& \\
-b^\pm \, v_2^\pm +c^\pm \, B_2^\pm -u_2^{0,\pm} \, \xi' \cdot (B')^\pm \, = \, 0 \, ,& \\
-b^\pm \, v_3^\pm +c^\pm \, B_3^\pm \, = \, 0 \, ,& \\
\xi' \cdot (v')^\pm \, = \, 0 \, .& 
\end{cases}
\end{equation*}
Adding $\xi_1$ times the fourth line, $\xi_2$ times the fifth line and $-b^\pm$ times the sixth line, we get 
$\xi' \cdot (B')^\pm =0$ by using $\tau \neq 0$ (this is the first occurence in this Appendix of this condition 
which will be used in several other places). Taking the scalar product of the two first lines with $\xi'$ (recall 
that $\xi'$ has norm $1$), we thus get $p^\pm=0$. Using then $(c^\pm)^2 -(b^\pm)^2 \neq 0$, which we 
recall follows from \eqref{s3-hyp_delta_neq_0}, \eqref{s3-hyp_tau_racine_det_lop}, we end up with 
$v^\pm=B^\pm=0$.
\bigskip

$\bullet$ We can also compute a right eigenvector for the eigenmode $ \mp 1$ of the differential system:
\begin{equation}
\label{appA-equadiffpm}
A_3^\pm \, \dfrac{{\rm d}U^\pm}{{\rm d}Y_3} +i\, {\mathcal A}^\pm \, U^\pm \, = \, 0 \, ,\quad Y_3 \in \R^\pm \, .
\end{equation}
Namely, by using again the conditions $\tau \neq 0$ and $(c^\pm)^2 -(b^\pm)^2 \neq 0$, we can show that 
the vector spaces $\text{\rm Ker } (\mp \, A_3^\pm +i\, {\mathcal A}^\pm)$ have dimension $1$ and are 
spanned by the vectors:
\begin{equation}
\label{appA-defeigenvectorsRpm}
{\bf R}^\pm \, := \, \Big( \xi_1 \, c^\pm,\xi_2 \, c^\pm,\pm i \, c^\pm,\xi_1 \, b^\pm,\xi_2 \, b^\pm,\pm i \, b^\pm, 
(b^\pm)^2 -(c^\pm)^2 \Big)^T \in \C^7 \, .
\end{equation}
These expressions are to be compared with similar ones in \cite{AliHunter} for the two-dimensional problem. 
It should be noted that the vectors ${\bf R}^\pm$ satisfy the additional condition:
$$
\begin{pmatrix}
0 & 0 & 0 & i\, \xi_1& i\,\xi_2 & \mp 1 & 0 \end{pmatrix} \, {\bf R}^\pm \, = \, 0 \, ,
$$
which is reminiscent of the divergence free constraint on the magnetic field.

When analyzing the WKB cascade, we have also made use of the dual problem associated with \eqref{appA-equadiffpm}, namely:
\begin{equation*}
(A_3^\pm)^T \, \dfrac{{\rm d}U^\pm}{{\rm d}Y_3} +i\, ({\mathcal A}^\pm)^T \, U^\pm \, = \, 0 \, ,\quad Y_3 \in \R^\pm \, ,
\end{equation*}
which admits the same eigenmode $\mp 1$. By using the conditions $\tau \neq 0$ and $(c^\pm)^2 -(b^\pm)^2 \neq 0$, 
we know that the vector spaces $\text{\rm Ker } (\mp \, A_3^\pm +i\, {\mathcal A}^\pm)^T$ have dimension $1$ and we 
can show that they are spanned by the vectors:
\begin{equation}
\label{appA-deflefteigenvectorsLpm}
{\bf L}^\pm \, := \, \Big( \xi_1 \, \tau,\xi_2 \, \tau,\pm i \, \tau,2\, \xi_1 \, b^\pm,2\, \xi_2 \, b^\pm,\pm 2\, i \, b^\pm, 
-\tau \, (a^\pm +c^\pm) \Big)^T \in \C^7 \, .
\end{equation}

The vectors ${\mathcal R}^\pm (k), {\mathcal L}^\pm (k)$ entering the definition of the profiles and the solvability 
condition for the fast problem \eqref{fast_problem}, are then defined by:
\begin{equation}
\label{appA-defRLkpm}
\forall \, k \in \Z \setminus \{ 0 \} \, ,\quad {\mathcal R}^\pm (k) \, := \, \begin{cases}
{\bf R}^\pm \, ,& \text{\rm if $k>0$,} \\
\overline{{\bf R}^\pm} \, ,& \text{\rm if $k<0$,}
\end{cases} \qquad {\mathcal L}^\pm (k) \, := \, \begin{cases}
{\bf L}^\pm \, ,& \text{\rm if $k>0$,} \\
\overline{{\bf L}^\pm} \, ,& \text{\rm if $k<0$.}
\end{cases}
\end{equation}
From \eqref{appA-defeigenvectorsRpm} and \eqref{appA-deflefteigenvectorsLpm}, we get:
\begin{align*}
{\mathcal R}^\pm (k) \, & \, = \, \Big( \xi_1 \, c^\pm,\xi_2 \, c^\pm,\pm i \, \text{\rm sgn} (k) \, c^\pm, 
\xi_1 \, b^\pm,\xi_2 \, b^\pm,\pm i \, \text{\rm sgn} (k) \, b^\pm,(b^\pm)^2 -(c^\pm)^2 \Big)^T \, ,\\
{\mathcal L}^\pm (k) \, & \, = \, \Big( \xi_1 \, \tau,\xi_2 \, \tau,\pm i \, \text{\rm sgn} (k) \, \tau, 
2\, \xi_1 \, b^\pm,2\, \xi_2 \, b^\pm,\pm 2\, i \, \text{\rm sgn} (k) \, b^\pm,-\tau \, (a^\pm +c^\pm) \Big)^T \, ,
\end{align*}
where `sgn' denotes the sign function, which we need only to define on $\Z \setminus \{ 0\}$.

Using \eqref{appA-defRLkpm} and the above expressions for the matrices $A_\alpha^\pm$, we can compute the Hermitian products, 
where the equalities below hold for any pair of nonzero Fourier modes $k_1,k_2$:
\begin{equation}
\label{appA-produits_hermitiens-1}
\begin{aligned}
{\mathcal L}^\pm (k_1) \sbt A_0 \, {\mathcal R}^\pm (k_2) \, = \, & \, \big( \tau \, c^\pm +2\, (b^\pm)^2 \big) 
\, \big( 1+\text{\rm sgn} (k_1) \, \text{\rm sgn} (k_2) \big) \, ,\\
{\mathcal L}^\pm (k_1) \sbt A_1^\pm \, {\mathcal R}^\pm (k_2) \, = \, & \, 
\tau \, \big( u_1^{0,\pm} \, c^\pm -H_1^{0,\pm} \, b^\pm \big) \, \big( 1+\text{\rm sgn} (k_1) \, \text{\rm sgn} (k_2) \big) 
+\xi_1 \, \tau \, \big( (b^\pm)^2 -(c^\pm)^2 \big) \\
& \, +2\, b^\pm \, \big( u_1^{0,\pm} \, b^\pm -H_1^{0,\pm} \, c^\pm \big) \, 
\big( 1+\text{\rm sgn} (k_1) \, \text{\rm sgn} (k_2) \big) \, ,\\
{\mathcal L}^\pm (k_1) \sbt A_2^\pm \, {\mathcal R}^\pm (k_2) \, = \, & \, 
\tau \, \big( u_2^{0,\pm} \, c^\pm -H_2^{0,\pm} \, b^\pm \big) \, \big( 1+\text{\rm sgn} (k_1) \, \text{\rm sgn} (k_2) \big) 
+\xi_2 \, \tau \, \big( (b^\pm)^2 -(c^\pm)^2 \big) \\
& \, +2\, b^\pm \, \big( u_2^{0,\pm} \, b^\pm -H_2^{0,\pm} \, c^\pm \big) \, 
\big( 1+\text{\rm sgn} (k_1) \, \text{\rm sgn} (k_2) \big) \, ,\\
{\mathcal L}^\pm (k_1) \sbt \cA^\pm \, {\mathcal R}^\pm (k_2) \, = \, & \, 
\tau \, \big( (b^\pm)^2 -(c^\pm)^2 \big) \, \big( 1-\text{\rm sgn} (k_1) \, \text{\rm sgn} (k_2) \big) \, .
\end{aligned}
\end{equation}
We recall that $\sbt$ refers to the Hermitian product in $\C^7$:
$$
\forall \, U,V \in \C^7 \, ,\quad U \sbt V \, := \, \sum_{i=1}^7 \, \overline{U_i} \, V_i \, .
$$
The expressions \eqref{appA-produits_hermitiens-1} are used in Chapter \ref{chapter4} to derive the leading amplitude equation 
\eqref{s3-edp_hatpsi^2-final}.
\bigskip

$\bullet$ We now give the expression of the Hessian mappings defined in \eqref{s3-def_bA_j^pm}. For any pair 
of vectors in $\R^7$:
$$
U \, = \, (u_1,u_2,u_3,H_1,H_2,H_3,q)^T \, ,\quad \dot{U} \, = \, (\dot{u}_1,\dot{u}_2,\dot{u}_3,\dot{H}_1,\dot{H}_2,\dot{H}_3,\dot{q})^T \, ,
$$
there holds:
\begin{align*}
\bA_1 (U,\dot{U}) \, & \, = \, \begin{pmatrix}
2\, u_1 \, \dot{u}_1 -2\, H_1 \, \dot{H}_1 \\
u_1 \, \dot{u}_2 +u_2 \, \dot{u}_1 -H_1 \, \dot{H}_2 -H_2 \, \dot{H}_1 \\
u_1 \, \dot{u}_3 +u_3 \, \dot{u}_1 -H_1 \, \dot{H}_3 -H_3 \, \dot{H}_1 \\
0 \\
u_1 \, \dot{H}_2 +H_2 \, \dot{u}_1 -u_2 \, \dot{H}_1 -H_1 \, \dot{u}_2 \\
u_1 \, \dot{H}_3 +H_3 \, \dot{u}_1 -u_3 \, \dot{H}_1 -H_1 \, \dot{u}_3 \\
0 \end{pmatrix} \, ,\\
\bA_2 (U,V) \, & \, = \, \begin{pmatrix}
u_2 \, \dot{u}_1 +u_1 \, \dot{u}_2 -H_2 \, \dot{H}_1 -H_1 \, \dot{H}_2 \\
2\, u_2 \, \dot{u}_2 -2\, H_2 \, \dot{H}_2 \\
u_2 \, \dot{u}_3 +u_3 \, \dot{u}_2 -H_2 \, \dot{H}_3 -H_3 \, \dot{H}_2 \\
u_2 \, \dot{H}_1 +H_1 \, \dot{u}_2 -H_2 \, \dot{u}_1 -u_1 \, \dot{H}_2 \\
0 \\
u_2 \, \dot{H}_3 +H_3 \, \dot{u}_2 -H_2 \, \dot{u}_3 -u_3 \, \dot{H}_2 \\
0 \end{pmatrix} \, ,\\
\bA_3 (U,V) \, & \, = \, \begin{pmatrix}
u_3 \, \dot{u}_1 +u_1 \, \dot{u}_3 -H_3 \, \dot{H}_1 -H_1 \, \dot{H}_3 \\
u_3 \, \dot{u}_2 +u_2 \, \dot{u}_3 -H_3 \, \dot{H}_2 -H_2 \, \dot{H}_3 \\
2\, u_3 \, \dot{u}_3 -2\, H_3 \, \dot{H}_3 \\
u_3 \, \dot{H}_1 +H_1 \, \dot{u}_3 -H_3 \, \dot{u}_1 -u_1 \, \dot{H}_3 \\
u_3 \, \dot{H}_2 +H_2 \, \dot{u}_3 -H_3 \, \dot{u}_2 -u_2 \, \dot{H}_3 \\
0 \\
0 \end{pmatrix} \, .
\end{align*}
Using again \eqref{appA-defRLkpm} and the latter expressions for the mappings $\bA_\alpha$, we can compute for any triple 
of nonzero Fourier modes $k_1,k_2,k_3$:
\begin{equation}
\label{appA-produits_hermitiens-2}
\begin{aligned}
{\mathcal L}^\pm (k_1) \sbt \bA_1\, ({\mathcal R}^\pm (k_2),{\mathcal R}^\pm (k_3)) \, = \, & \, 
\tau \, \xi_1 \, \big( (c^\pm)^2 -(b^\pm)^2 \big) \, \big( 2+\text{\rm sgn} (k_1) \, (\text{\rm sgn} (k_2) +\text{\rm sgn} (k_3)) \big) \, ,\\
{\mathcal L}^\pm (k_1) \sbt \bA_2\, ({\mathcal R}^\pm (k_2),{\mathcal R}^\pm (k_3)) \, = \, & \, 
\tau \, \xi_2 \, \big( (c^\pm)^2 -(b^\pm)^2 \big) \, \big( 2+\text{\rm sgn} (k_1) \, (\text{\rm sgn} (k_2) +\text{\rm sgn} (k_3)) \big) \, ,\\
{\mathcal L}^\pm (k_1) \sbt \bA_3\, ({\mathcal R}^\pm (k_2),{\mathcal R}^\pm (k_3)) \, = \, & \, 
\pm i \, \tau \, \big( (c^\pm)^2 -(b^\pm)^2 \big) \\
& \qquad \big( \text{\rm sgn} (k_2) +\text{\rm sgn} (k_3) +2 \, \text{\rm sgn} (k_1) \, \text{\rm sgn} (k_2) \, \text{\rm sgn} (k_2) \big) \, .
\end{aligned}
\end{equation}
The expressions \eqref{appA-produits_hermitiens-2} are also used in Chapter \ref{chapter4} to derive the leading amplitude equation 
\eqref{s3-edp_hatpsi^2-final}.

\chapter{Compatibility conditions for the construction of correctors}
\label{appendixB}

This Appendix is devoted to the proof of the symmetry relations between the functions $\chi^{[\ell]}$ and the profiles 
$\psi^m$ that have been used in the proof of Lemma \ref{lem_compatibilite_div}. We also collect in this Appendix the 
proof of several intermediate results that have been used in Chapter \ref{chapter5} for the inductive construction of 
correctors in the WKB expansion \eqref{s3-def_dvlpt}.

\section{General considerations}

In this first Section, we recall the definition of the functions $\chi^{[\ell]}$, $\dot{\chi}^{[\ell]}$, $\ell \ge 0$, which can be achieved 
in a slightly more general framework than what has been considered in Chapter \ref{chapter2}. We make the general form 
of $\chi^{[\ell]}$ and $\dot{\chi}^{[\ell]}$ explicit, which will reduce the proof of Lemma \ref{lem_symmetry_1} and Lemma 
\ref{lem_symmetry_2} -the so-called first and second symmetry formulas- to verifying certain `invariance' and `recursive' 
properties in the decompositions \eqref{appB-equation6}, \eqref{appB-equation7} of $\chi^{[\ell]}$ and $\dot{\chi}^{[\ell]}$ 
below. These properties will be proved by means of some basic or more advanced combinatorial arguments.

In all this Appendix, we consider sequences $\bmu =(\mu_1,\mu_2,\dots)$ of integers ($\mu_\ell \in \N$ for all $\ell \ge 1$) 
of finite `length', that is:
$$
|\bmu| \, := \, \sum_{\ell \ge 1} \, \mu_\ell \, < \, +\infty \, .
$$
For such sequences, we define the `weight' of $\bmu$ by setting:
$$
\avbmu \, := \, \sum_{\ell \ge 1} \, \ell \, \mu_\ell \, < \, +\infty \, ,
$$
and we shall also use the notation:
$$
\bmu \, ! \, := \, \prod_{\ell \ge 1} \, \mu_\ell \, ! \, ,
$$
where the product is taken over finitely many indices, hence is well-defined, for the considered sequences. Any sequence 
of finite length is a sum of finitely many `elementary' sequences ${\bf e}_m$, $m \ge 1$, which are defined by:
$$
\forall \, m,\ell \ge 1 \, ,\quad ({\bf e}_m)_\ell \, := \, \delta_{m \, \ell} \, ,
$$
with $\delta$ the Kronecker symbol. In particular, ${\bf e}_m$ is the only sequence of length $1$ and weight $m$ for any 
$m \ge 1$. Eventually, we write $\bnu \le \bmu$ for two sequences $\bnu,\bmu$ if there holds $\nu_\ell \le \mu_\ell$ for 
all $\ell \ge 1$.

In what follows, we consider a sequence of $\cC^\infty$ functions $(\psi^1,\psi^2,\dots)$ defined on $[0,T]_t \times 
\bT^2_{(y_1,y_2)} \times \bT_\theta$, the final time $T>0$ being fixed. The sequence of functions is given here. In 
particular, we do not assume that the first element, namely $\psi^1$, is identically zero. However, when the results 
in this Appendix are applied to the context of incompressible current vortex sheets, $\psi^1$ is zero and the other 
functions $\psi^2,\psi^3,\dots$ correspond to the profiles in the asymptotic expansion \eqref{s3-def_dvlpt_BKW_psi}. 
We shall use the notation:
$$
\bpsi^{\bmu} \, := \, \prod_{\ell \ge 1} \, \big( \psi^\ell \big)^{\mu_\ell} \, ,
$$
which will be useful in several expressions below. Again, the product is taken over finitely many indices for the considered 
sequences. We use the convention $\bpsi^{\bmu}=1$ if $\bmu$ is the `zero' sequence, that is $\mu_\ell=0$ for all $\ell \ge 1$. 
We prove our main results, Corollary \ref{corB1}, Corollary \ref{corB2} and Proposition \ref{propB1} below in the more general 
context where $\psi^1$ may be nonzero since this could be useful for other geometric optics problems with (possibly curved) 
free boundaries, see, \emph{e.g.}, \cite{Williams99}.
\bigskip

In Chapter \ref{chapter2}, we have been considering the change of variable \eqref{defchi}. In our more general framework here, 
we consider a function $\psi (\eps,t,y',\theta)$ that is smooth with respect to $\eps \in \R$ near the origin, and that verifies 
$\psi (0,\cdot) \equiv 0$ and:
$$
\forall \, m \ge 1 \, ,\quad \dfrac{\partial^m \psi}{\partial \eps^m} (0,\cdot) \, = \, m\, ! \, \psi^m \, .
$$
Thanks to the Borel summation procedure, such a function exists. Then for sufficiently small $\eps$, the relation:
\begin{equation}
\label{defchi'}
y_3 \, = \, x_3 - \chi(x_3) \, \psi(\eps,t,y',\theta) \, ,
\end{equation}
defines implicitly $x_3$ in terms of $(\eps,t,y,\theta)$. Recall the notation $y=(y',y_3)$. We write the asymptotic expansion of $x_3$ 
with respect to $\eps$ as follows:
\begin{equation}
\label{defXell}
x_3 \, \sim \, \sum_{\ell \ge 0} \, \eps^\ell \, \bX^{[\ell]} (t,y,\theta) \, .
\end{equation}
Even without any further knowledge on the $\bX^{[\ell]}$'s, we can substitute the asymptotic expansion of $x_3$ into 
$\chi(x_3)$, and write:
$$
\chi (x_3) \, \sim \, \sum_{\ell \ge 0} \, \eps^\ell \, \chi^{[\ell]} (t,y,\theta) \, ,
$$
in agreement with the notation \eqref{defchichipointl}. The functions $\dot{\chi}^{[\ell]}$ are defined similarly by substituting $x_3$ 
in the first derivative $\chi'$ rather than in $\chi$. Identifying the asymptotic expansions with respect to $\eps$ in the relation 
\eqref{defchi'}, we immediately get the relations:
\begin{align}
& \bX^{[0]} (t,y,\theta) \, = \, y_3 \, ,\quad \chi^{[0]} (t,y,\theta) \, = \, \chi (y_3) \, ,\quad 
\dot{\chi}^{[0]} (t,y,\theta) \, = \, \chi' (y_3) \, ,\label{appB-equation1} \\
\forall \, \ell \ge 1 \, ,\quad & \bX^{[\ell]} \, = \, \sum_{\ell_1+\ell_2 =\ell} \chi^{[\ell_1]} \, \psi^{\ell_2} \, ,\label{appB-equation2}
\end{align}
where in \eqref{appB-equation2} we use the convention $\psi^0 \equiv 0$. In particular, the function $\bX^{[\ell+1]}$ is obtained 
from \eqref{appB-equation2} as long as we have already determined $\chi^{[0]},\dots,\chi^{[\ell]}$. In order to close the loop and 
to determine inductively all functions $\bX^{[\ell]}$ and $\chi^{[\ell]}$, the last ingredient is the so-called Fa\`a di Bruno formula 
\cite{Comtet}. Indeed, plugging the asymptotic expansion \eqref{defXell} in the function $\chi$ and expanding with respect to 
$\eps$, the Fa\`a di Bruno formula yields:
\begin{align}
\forall \, m \ge 1 \, ,\quad \chi^{[m]} (t,y,\theta) \, & \, = \, \sum_{\avbmu = m} \, \dfrac{1}{\bmu \, !} \, \chi^{(|\bmu|)}(y_3) \, 
\prod_{k \ge 1} \Big( \bX^{[k]} (t,y,\theta) \Big)^{\mu_k} \, ,\label{appB-equation3} \\
\dot{\chi}^{[m]} (t,y,\theta) \, & \, = \, \sum_{\avbmu = m} \, \dfrac{1}{\bmu \, !} \, \chi^{(1+|\bmu|)}(y_3) \, 
\prod_{k \ge 1} \Big( \bX^{[k]} (t,y,\theta) \Big)^{\mu_k} \, .\label{appB-equation4}
\end{align}
Since $\mu_k=0$ for $k \ge m+1$ in \eqref{appB-equation3}, because $\avbmu = m$, we see that the function $\chi^{[m]}$ can 
be fully determined as long as we already know the functions $\bX^{[0]},\dots,\bX^{[m]}$. Hence the following global induction 
procedure, which is initialized by \eqref{appB-equation1}: assuming that the functions $\bX^{[0]},\dots,\bX^{[\ell]}$, 
$\chi^{[0]},\dots,\chi^{[\ell]}$ have already been determined, one first computes $\bX^{[\ell+1]}$ from \eqref{appB-equation2} 
and then uses this expression, as well as that of $\bX^{[0]},\dots,\bX^{[\ell]}$, in \eqref{appB-equation3} in order to determine 
$\chi^{[\ell+1]}$. Once all functions $\bX^{[\ell]}$ have been determined, thanks to the `auxiliary' sequence of the $\chi^{[\ell]}$'s, 
the functions $\dot{\chi}^{[m]}$ are given for all $m \ge 1$ by \eqref{appB-equation4}.

Let us observe that the decompositions \eqref{appB-equation3}, \eqref{appB-equation4} also hold for $m=0$ with the 
convention $(\bX^{[k]})^0 =1$. One can compute for instance the first expressions of $\bX^{[\ell]},\chi^{[\ell]},\dot{\chi}^{[\ell]}$ 
and get:
\begin{align*}
& \bX^{[0]} \, = \, y_3 \, ,\quad & & \chi^{[0]} \, = \, \chi \, ,\\
& & & \dot{\chi}^{[0]} \, = \, \chi' \, ,\\
& \bX^{[1]} \, = \, {\color{ForestGreen} \chi} \, \psi^1 \, ,\quad & & \chi^{[1]} \, = \, {\color{blue} \chi \, \chi'} \, \psi^1 \, ,\\
& & & \dot{\chi}^{[1]} \, = \, {\color{red} \chi \, \chi''} \, \psi^1 \, ,\\
& \bX^{[2]} \, = \, \chi \, \chi' \, \big( \psi^1 \big)^2 +{\color{ForestGreen} \chi} \, \psi^2 \, ,\quad 
& & \chi^{[2]} \, = \, \left( \dfrac{\chi^2 \, \chi''}{2} +\chi \, \big( \chi' \big)^2 \right) \, \big( \psi^1 \big)^2 
+{\color{blue} \chi \, \chi'} \, \psi^2 \, ,\\
& & & \dot{\chi}^{[2]} \, = \, \left( \dfrac{\chi^2 \, \chi'''}{2} +\chi \, \chi' \, \chi'' \right) \, \big( \psi^1 \big)^2 
+{\color{red} \chi \, \chi''} \, \psi^2 \, ,
\end{align*}
where it is understood that all functions $\chi,\chi',\chi'' \dots$ are evaluated at $y_3$, and all functions $\psi^1,\psi^2,\dots$ 
are evaluated at $(t,y',\theta)$. The latter expressions are consistent with what we have found in Chapter \ref{chapter2} 
(recall that in Chapter \ref{chapter2}, $\psi^1$ is zero hence several simplifications).

From a straightforward induction argument, based on the relations \eqref{appB-equation2}, \eqref{appB-equation3}, 
\eqref{appB-equation4}, we can decompose the functions $\bX^{[m]}$, $\chi^{[m]}$, $\dot{\chi}^{[m]}$ as 
follows\footnote{Normalizing with the constant factor $\bmu \, !$ will be useful later on.}:
\begin{align}
\forall \, m \ge 0 \, ,\quad \bX^{[m]} (t,y,\theta) \, & \, = \, 
\sum_{\avbmu = m} \, \dfrac{1}{\bmu \, !} \, X_{\bmu} (y_3) \, \, \bpsi^{\bmu} (t,y',\theta) \, ,\label{appB-equation5} \\
\chi^{[m]} (t,y,\theta) \, & \, = \, 
\sum_{\avbmu = m} \, \dfrac{1}{\bmu \, !} \, F_{\bmu} (y_3) \, \, \bpsi^{\bmu} (t,y',\theta) \, ,\label{appB-equation6} \\
\dot{\chi}^{[m]} (t,y,\theta) \, & \, = \, 
\sum_{\avbmu = m} \, \dfrac{1}{\bmu \, !} \, \dot{F}_{\bmu} (y_3) \, \, \bpsi^{\bmu} (t,y',\theta) \, .\label{appB-equation7}
\end{align}
The case $m=0$ in \eqref{appB-equation5}, \eqref{appB-equation6}, \eqref{appB-equation7} is only a matter of 
rewriting \eqref{appB-equation1}, with the notation:
$$
X_{(0,0,\dots)}(y_3) \, := \, y_3 \, ,\quad F_{(0,0,\dots)}(y_3) \, := \, \chi (y_3) \, ,\quad \dot{F}_{(0,0,\dots)}(y_3) \, := \, \chi' (y_3) \, . 
$$
The proof of \eqref{appB-equation5}, \eqref{appB-equation6}, \eqref{appB-equation7} then follows by induction on $m$. 
For instance, the relation \eqref{appB-equation2} connects the functions $X_{\bmu}$ in \eqref{appB-equation5} to the 
functions $F_{\bmu}$ in \eqref{appB-equation6} as follows:
\begin{equation}
\label{appB-equation8}
\forall \, m \ge 1 \, ,\quad X_{{\bf e}_m} \, = \, \chi \, ,\qquad \text{\rm (the case of length $1$ sequences)}
\end{equation}
and if $|\bmu| \ge 2$, we get:
\begin{equation}
\label{appB-equation9}
X_{\bmu} \, = \, \sum_{\substack{\bnu \le \bmu \\ |\bnu| = |\bmu| -1}} \dfrac{\bmu \, !}{\bnu !} \, F_{\bnu} \, .
\end{equation}
The relation that gives $F_{\bmu}$ in terms of the $X_{\bnu}$'s is less easy to analyze and follows from the Fa\`a 
di Bruno formula \eqref{appB-equation3} (or \eqref{appB-equation4} for the relation that gives $\dot{F}_{\bmu}$ in 
terms of the $X_{\bnu}$'s). In the following Section, we prove a fundamental property of the functions $X_{\bmu}$, 
$F_{\bmu}$, $\dot{F}_{\bmu}$ by precisely analyzing the Fa\`a di Bruno formula (Lemma \ref{lemB1} below). This 
fundamental property will easily imply the first symmetry formula \eqref{symmetry_1} which was used in the proof 
of Lemma \ref{lem_compatibilite_div}.

\section{The first symmetry formula}

The proof of Lemma \ref{lem_symmetry_1} and (half) the proof of Lemma \ref{lem_symmetry_2} relies on 
an `invariance' property of the functions $X_{\bmu}$, $F_{\bmu}$, $\dot{F}_{\bmu}$ in the decompositions 
\eqref{appB-equation5}, \eqref{appB-equation6}, \eqref{appB-equation7} above, which we prove now.

\begin{lemma}[Invariance]
\label{lemB1}
The functions $X_{\bmu}$, $F_{\bmu}$, $\dot{F}_{\bmu}$ in \eqref{appB-equation5}, \eqref{appB-equation6}, 
\eqref{appB-equation7} only depend on $|\bmu|$. In other words, there holds $X_{\bmu}=X_{\bnu}$ if 
$|\bmu|=|\bnu|$ (and similarly for the $F_{\bmu}$'s and $\dot{F}_{\bmu}$'s).
\end{lemma}

\noindent An easy case of Lemma \ref{lemB1} can be observed in the expressions we have given above for $\bX^{[1]}$, 
$\bX^{[2]}$, $\chi^{[1]}$, $\chi^{[2]}$ and so on. In these expressions, the `coefficients' of $\psi^1$ and $\psi^2$ are equal 
(they are highlighted in the same color). These coefficients correspond to the elementary sequences ${\bf e}_1$, ${\bf e}_2$ 
which are both of length $1$. If one computes the expression of $\chi^{[3]}$, it will be found that the coefficient of 
$\psi^1 \, \psi^2$ equals
$$
\chi^2 \, \chi'' +2 \, \chi \, \big( \chi' \big)^2 \, ,
$$
which is consistent with the result of Lemma \ref{lemB1} since this coefficient $F_{(1,1,0,\dots)}$ should be equal to $F_{(2,0,0,\dots)}$ 
which is twice the coefficient of $(\psi^1)^2$ in the decomposition of $\chi^{[2]}$ (keep in mind the $\bmu \, !$ normalizing factor).

In all this Appendix, we shall use the notation:
\begin{equation}
\label{notation-XFrond}
X_{\bmu} \, = \, \cX_{|\bmu|} \, ,\quad F_{\bmu} \, = \, \cF_{|\bmu|} \, ,\quad \dot{F}_{\bmu} \, = \, \dot{\cF}_{|\bmu|} \, ,
\end{equation}
when we need to use the invariance of the functions $X_{\bmu}$, $F_{\bmu}$, $\dot{F}_{\bmu}$. The functions $\cX_m$, $\cF_m$, 
$\dot{\cF}_m$ in \eqref{notation-XFrond} are indexed by integers while the functions $X_{\bmu}$, $F_{\bmu}$, $\dot{F}_{\bmu}$ are 
indexed by sequences.

\begin{proof}[Proof of Lemma \ref{lemB1}]
The proof of Lemma \ref{lemB1} proceeds by induction on the length of $\bmu$. The case of sequences of length zero is trivial 
since there is only one such sequence, the zero sequence. We thus `initialize' the sequences of functions $(\cX_m)_{m \in \N}$, 
$(\cF_m)_{m \in \N}$, $(\dot{\cF}_m)_{m \in \N}$ by setting:
$$
\cX_0 (y_3) \, := \, y_3 \, ,\quad \cF_0 (y_3) \, := \, \chi (y_3) \, ,\quad \dot{\cF}_0 (y_3) \, := \, \chi' (y_3) \, .
$$
Let us now examine the case $|\bmu|=1$, that is when $\bmu$ is an elementary sequence ${\bf e}_m$ for some $m \ge 1$. The 
relation \eqref{appB-equation8} already shows that $X_{{\bf e}_m}$ is independent of $m$ and equals $\chi$. Let us now show 
that $F_{{\bf e}_m}$ does not depend on $m$. Since ${\bf e}_m \, ! =1$, the function $F_{{\bf e}_m}$ is nothing but the coefficient 
of $\psi^m$ in the decomposition \eqref{appB-equation6} of $\chi^{[m]}$. The complete expression of $\chi^{[m]}$ is given by 
\eqref{appB-equation3}, which we can rewrite as follows:
\begin{equation}
\label{faadibruno}
\chi^{[m]} \, = \, \sum_{\avbmu = m} \, \dfrac{1}{\bmu \, !} \, \chi^{(|\bmu|)} \, 
\prod_{k \ge 1} \left( \sum_{\langle \bnu \rangle = k} \, \dfrac{1}{\bnu \, !} \, X_{\bnu} \, \bpsi^{\bnu} \right)^{\mu_k} \, .
\end{equation}
The term $F_{{\bf e}_m} \, \psi^m$ in the decomposition \eqref{appB-equation6} is split into (possibly) several contributions 
on the right hand side of \eqref{faadibruno}. Namely, each time that we can decompose the sequence ${\bf e}_m$ as:
$$
{\bf e}_m \, = \, \sum_{k \ge 1} \, \sum_{\kappa=1}^{\mu_k} \, \bnu (k,\kappa) \, ,\quad 
\langle \bnu (k,\kappa) \rangle =k \, ,\quad \avbmu = m \, ,
$$
this contributes to $F_{{\bf e}_m}$ for:
$$
\dfrac{1}{\bmu \, !} \, \chi^{(|\bmu|)} \, \prod_{k \ge 1} \, \prod_{\kappa=1}^{\mu_k} \, \dfrac{X_{{\bnu} (k,\kappa)}}{\bnu (k,\kappa) \, !} \, .
$$
Since the only possible such decomposition of ${\bf e}_m$ is obtained with $\bmu ={\bf e}_m$ and $\bnu (m,1) ={\bf e}_m$, 
we get (here we use $X_{{\bf e}_m}=\chi$):
\begin{equation}
\label{appB-equation10}
\forall \, m \ge 1 \, ,\quad F_{{\bf e}_m} \, = \, \chi \, \chi' \, .
\end{equation}
We could have obtained in the same way:
\begin{equation}
\label{appB-equation11}
\forall \, m \ge 1 \, ,\quad \dot{F}_{{\bf e}_m} \, = \, \chi \, \chi'' \, ,
\end{equation}
and it obviously follows from \eqref{appB-equation10}, \eqref{appB-equation11} that $F_{{\bf e}_m}$ and $\dot{F}_{{\bf e}_m}$ 
are independent of $m$. Hence the result of Lemma \ref{lemB1} is valid for sequences of length $1$. Moreover, we can recast 
\eqref{appB-equation8}, \eqref{appB-equation10}, \eqref{appB-equation11} as:
$$
\cX_1 \, := \, \chi \, ,\quad \cF_1 \, := \, \chi \, \chi' \, ,\quad \dot{\cF}_1 \, := \, \chi \, \chi'' \, .
$$

Let us now deal with the general step of the induction. We assume that up some integer $L \ge 1$, there holds:
$$
|\bmu| \, = \, \ell \le L \quad \Longrightarrow \quad X_{\bmu} \, = \, \cX_\ell \, ,\quad F_{\bmu} \, = \, \cF_\ell \, ,\quad 
\dot{F}_{\bmu} \, = \, \dot{\cF}_\ell \, ,
$$
for some appropriate functions $\cX_1,\dots,\cX_L$ and so on. We now consider a sequence $\bmu$ of length $L+1$. Let us 
start by looking at the relation \eqref{appB-equation9}. Since in that decomposition of $X_{\bmu}$, all sequences $\bnu$ on the 
right hand side have length $L$, we can use the induction assumption and get:
$$
X_{\bmu} \, = \, \sum_{\substack{\bnu \le \bmu \\ |\bnu| = |\bmu| -1}} \dfrac{\bmu \, !}{\bnu !} \, \cF_L 
\, = \, |\bmu| \, \cF_L \, = \, (L+1) \, \cF_L \, .
$$
In particular, $X_{\bmu}$ only depends on $|\bmu|$ and, keeping consistent notations with those above, we have obtained the 
relation $\cX_{L+1} =(L+1) \, \cF_L$. (This relation can be verified in the case $L=1$ on the above expression for $\bX^{[2]}$.)

The goal now is to show that $F_{\bmu}$ only depends on $|\bmu|$. We proceed in two steps. The first step is to 
show that $F_{\bmu}$ reads:
\begin{equation}
\label{decompositionFmu}
F_{\bmu} \, = \, \sum_{N=1}^{L+1} \sum_{(\bnu_1,\cdots,\bnu_N) \in \cP (\bmu,N)} 
\dfrac{\bmu \, !}{\bnu_1 \, ! \, \cdots \, \bnu_N \, !} \, \chi^{(N)} \, \prod_{n=1}^N \, \cX_{|\bnu_n|} \, ,
\end{equation}
where we denote by $\cP (\bmu,N)$ the set of all possible partitions of $\bmu$ into $N$ (nontrivial) subsequences, 
that is:
$$
\bmu \, = \, \bnu_1 +\cdots +\bnu_N \, ,\quad \min_n \, |\bnu_n| \ge 1 \, ,
$$
and we do not take the order of the $\bnu_n$'s into account (for instance $(\bnu_1,\bnu_2)$ accounts for the same partition of 
$\bnu_1+\bnu_2$ as $(\bnu_2,\bnu_1)$ if $N=2$ and $\bnu_1 \neq \bnu_2$). The second step is to infer from \eqref{decompositionFmu} 
that $F_{\bmu} = F_{{\bf e}_1 +\cdots +{\bf e}_{L+1}}$, from which we can deduce that $F_{\bmu}$ only depends on $|\bmu|$. 
We shall also deduce a useful `recursive' formula which gives $\cF_{L+1}$ in terms of $\cX_1,\dots,\cX_{L+1}$ (hence in terms 
of $\cF_0,\dots,\cF_L$).

Let us therefore prove the validity of \eqref{decompositionFmu}. For a partition $(\bnu_1,\cdots,\bnu_N)$ of $\bmu$ 
into $N$ subsequences, we define a sequence of integers $\btheta =(\theta_1,\theta_2,\dots)$ by:
$$
\forall \, j \ge 1 \, ,\quad \theta_j \, := \, \# \big\{ 1 \le n \le N \, / \, \langle \bnu_n \rangle =j \big\} \, ,
$$
so that we have $|\btheta|=N$ and $\langle \btheta \rangle =\avbmu$. Then there are $\btheta \, !$ possible ways 
to construct a decomposition (here the order plays a role !):
$$
\bmu \, = \, \sum_{j \ge 1} \, \sum_{k=1}^{\theta_j} \, \bnu (j,k) \, ,\quad \langle \bnu (j,k) \rangle \, = \, j \, ,
$$
where the set of all $\bnu (j,k)$'s equals the set of all $\bnu_n$'s (with same multiplicity). From the Fa\`a di Bruno 
formula \eqref{faadibruno}, each such decomposition of $\bmu$ contributes to $F_{\bmu}$ for:
$$
\dfrac{\bmu \, !}{\btheta \, !} \, \chi^{(|\btheta|)} \, \prod_{j,k \ge 1} \dfrac{X_{\bnu (j,k)}}{\bnu (j,k) \, !} \, = \, 
\dfrac{\bmu \, !}{\btheta \, !} \, \chi^{(N)} \, \prod_{j,k \ge 1} \dfrac{\cX_{|\bnu (j,k)|}}{\bnu (j,k) \, !} \, = \, 
\dfrac{\bmu \, !}{\btheta \, ! \, \bnu_1 \, ! \, \cdots \, \bnu_N \, !} \, \chi^{(N)} \, \prod_{n=1}^N \cX_{|\bnu_n|} \, ,
$$
where we have used the induction assumption. Hence any partition $(\bnu_1,\cdots,\bnu_N)$ of $\bmu$ into $N$ 
subsequences contributes to $F_{\bmu}$ for:
$$
\dfrac{\bmu \, !}{\bnu_1 \, ! \, \cdots \, \bnu_N \, !} \, \chi^{(N)} \, \prod_{n=1}^N \, \cX_{|\bnu_n|} \, ,
$$
whence the relation \eqref{decompositionFmu}. In particular, another way to write \eqref{decompositionFmu} in 
the particular case $\bmu ={\bf e}_1+\cdots +{\bf e}_{L+1}$ is:
\begin{equation}
\label{decompositionFmu'}
F_{{\bf e}_1+\cdots +{\bf e}_{L+1}} \, = \, 
\sum_{N=1}^{L+1} \sum_{(\cL_1,\cdots,\cL_N) \in P(L+1,N)} \chi^{(N)} \, \prod_{n=1}^N \, \cX_{\# \cL_n} \, ,
\end{equation}
where $P(L+1,N)$ denotes the set of partitions of $\{ 1,\dots,L+1 \}$ into $N$ (nontrivial) pieces \cite{Comtet} 
(once again, the order of the $\cL_n$'s is not taken into account here). Indeed, partitions of ${\bf e}_1+\cdots 
+{\bf e}_{L+1}$ correspond in a unique way to partitions of $\{ 1,\dots,L+1 \}$ by setting:
$$
\bnu_n \, = \, \sum_{\ell \in \cL_n} {\bf e}_\ell \, ,\quad \cL_n \, = \, \{ \ell \ge 1 \, / \, (\bnu_n)_\ell =1 \} \, .
$$

It remains at this stage to verify that $F_{\bmu}$ equals $F_{{\bf e}_1 +\cdots +{\bf e}_{L+1}}$ and the proof of Lemma 
\ref{lemB1} will be complete. We fix a mapping:
$$
\sigma \, : \, \{ 1,\dots,L+1 \} \, \longrightarrow \N^* \, ,
$$
such that for all $\ell \ge 1$, $\mu_\ell =\# \{ j \, / \, \sigma (j) =\ell \}$. This is possible since $|\bmu|=L+1$. 
Then given a partition $(\cL_1,\cdots,\cL_N) \in P(L+1,N)$, we can define:
$$
\forall \, n \, = \, 1,\dots,N \, ,\quad \bnu_n \, := \, \sum_{\ell \in \cL_n} {\bf e}_{\sigma (\ell)} \, ,
$$
so that the $\bnu_n$'s form a partition of $\bmu$ into $N$ subsequences with $|\bnu_n| =\# \cL_n$. Conversely, 
given a partition $(\bnu_1,\dots,\bnu_N)$ of $\bmu$ into $N$ subsequences, the problem is to determine how 
many partitions $(\cL_1,\cdots,\cL_N) \in P(L+1,N)$ yield $(\bnu_1,\dots,\bnu_N)$ in the above process. First 
we need to choose among the $\mu_1$ integers that are mapped onto $1$ by $\sigma$, a decomposition into 
$(\bnu_1)_1,\dots,(\bnu_N)_1$ packets. This gives:
$$
\dfrac{\mu_1 \, !}{(\bnu_1)_1 \, ! \, \cdots (\bnu_N)_1 \, !}
$$
possibilities. The first $(\bnu_1)_1$ integers contribute to $\cL_1$ and so on. Repeating the argument for each 
$\ell \ge 1$, we need to decompose the $\mu_\ell$ integers that are mapped onto $\ell$ by $\sigma$ into 
$(\bnu_1)_\ell,\dots,(\bnu_N)_\ell$ packets. Overall, this gives:
$$
\dfrac{\bmu \, !}{\bnu_1 \, ! \, \cdots \bnu_N \, !}
$$
possible partitions $(\cL_1,\cdots,\cL_N) \in P(L+1,N)$ that yield the same partition $(\bnu_1,\dots,\bnu_N)$ of $\bmu$. 
Collecting in the expression \eqref{decompositionFmu'} of $F_{{\bf e}_1 +\cdots +{\bf e}_{L+1}}$ all partitions that 
correspond to the same $(\bnu_1,\dots,\bnu_N)$, and comparing with \eqref{decompositionFmu}, we end up with 
$F_{{\bf e}_1 +\cdots +{\bf e}_{L+1}}=F_{\bmu}$, which completes the induction argument. (We leave the case of 
the functions $\dot{F}_{\bmu}$ to the interested reader.)

For future use, let us state here that we have proved the relations:
\begin{align}
\forall \, L \ge 1 \, ,\quad \cF_L \, & \, = \, \sum_{N=1}^L \, 
\chi^{(N)} \, \sum_{(\cL_1,\cdots,\cL_N) \in P(L,N)} \prod_{n=1}^N \, \cX_{\# \cL_n} \, ,\label{decompcFL} \\
\dot{\cF}_L \, & \, = \, \sum_{N=1}^L \, \chi^{(1+N)} \, \sum_{(\cL_1,\cdots,\cL_N) \in P(L,N)} \prod_{n=1}^N \, \cX_{\# \cL_n} 
\, ,\label{decompcFpointL}
\end{align}
and $\cX_L =L \, \cF_{L-1}$ for any $L \ge 1$ (with the convention $\cF_0 =\chi$).
\end{proof}

\begin{corollary}[The first symmetry formula]
\label{corB1}
The functions $\chi^{[\ell]}$ and $\psi^m$ satisfy the relation:
\begin{equation*}
\forall \, \ell \ge 0 \, ,\quad \sum_{\ell_1+\ell_2=\ell} \p_\theta \chi^{[\ell_1]} \, \p_t \psi^{\ell_2} 
\, = \, \sum_{\ell_1+\ell_2=\ell} \p_t \chi^{[\ell_1]} \, \p_\theta \psi^{\ell_2} \, ,
\end{equation*}
and similar relations with any couple of tangential partial derivatives chosen among $\{ \p_t,\p_{y_1},\p_{y_2},\p_\theta \}$. 
(We keep using the convention $\psi^0 =0$.)
\end{corollary}

\noindent Corollary \ref{corB1} is a more general statement than Lemma \ref{lem_symmetry_1} since we do not 
assume $\psi^1 =0$ here. Hence we shall obtain the claim of Lemma \ref{lem_symmetry_1} as long as we prove 
Corollary \ref{corB1}. (We expect that generalizing to the case $\psi^1 \neq 0$ might be useful in other contexts.)

\begin{proof}[Proof of Corollary \ref{corB1}]
Let us first observe that the result of Corollary \ref{corB1} is immediate if $\ell=0$ or $\ell=1$ for we have $\psi^0 = 0$ and $\chi^{[0]} 
= \chi(y_3)$ hence $\p_\theta \chi^{[0]} = \p_t \chi^{[0]} = 0$. We thus assume $\ell \ge 2$ from now on, and use $\p_\theta \chi^{[0]} 
= 0$ to simplify:
$$
\sum_{\ell_1+\ell_2=\ell} \p_\theta \chi^{[\ell_1]} \, \p_t \psi^{\ell_2} \, = \, 
\sum_{\substack{\ell_1+\ell_2=\ell \\ \ell_1,\ell_2 \ge 1}} \p_\theta \chi^{[\ell_1]} \, \p_t \psi^{\ell_2} \, .
$$
We now use the decomposition \eqref{appB-equation6} of the functions $\chi^{[\ell]}$ to get:
\begin{align*}
\sum_{\ell_1+\ell_2=\ell} \p_\theta \chi^{[\ell_1]} \, \p_t \psi^{\ell_2} \, & \, = \, 
\sum_{\substack{\ell_1+\ell_2=\ell \\ \ell_1,\ell_2 \ge 1}} \p_t \psi^{\ell_2} \, \sum_{\avbmu = \ell_1} 
\dfrac{1}{\bmu \, !} \, F_{\bmu} \, \p_\theta \Big( \bpsi^{\bmu} \Big) \\
& \, = \, \sum_{\substack{\ell_1+\ell_2=\ell \\ \ell_1,\ell_2 \ge 1}} \, \sum_{k \ge 1} \p_t \psi^{\ell_2} \, \p_\theta \psi^k \, 
\sum_{\substack{\avbmu = \ell_1 \\ \bmu \ge {\bf e}_k}} \dfrac{1}{(\bmu -{\bf e}_k) \, !} \, F_{\bmu} \, \bpsi^{\bmu -{\bf e}_k} \\
& \, = \, \sum_{m_1,m_2 \ge 1} \p_t \psi^{\, m_1} \, \p_\theta \psi^{\, m_2} \, \sum_{\langle \bnu \rangle = \ell-m_1-m_2}  
\dfrac{1}{\bnu \, !} \, F_{\bnu +{\bf e}_{m_2}} \, \bpsi^{\bnu} \, ,
\end{align*}
where we have used the fact that only tangential derivatives act here, so $\p_\theta F_{\bmu} =0$, and it is understood 
that if $m_1+m_2>\ell$, the very last sum is zero since it contains no term. We now apply Lemma \ref{lemB1}, which gives 
$F_{\bnu +{\bf e}_{m_2}} =F_{\bnu +{\bf e}_{m_1}}$ for any sequence $\bnu$ and any couple of integers $m_1,m_2$ 
(since $|\bnu +{\bf e}_{m_2}| =|\bnu +{\bf e}_{m_1}| =|\bnu|+1$). We can then `rewind' the previous calculations and 
get by an obvious change of indices:
\begin{equation*}
\sum_{\ell_1+\ell_2=\ell} \p_\theta \chi^{[\ell_1]} \, \p_t \psi^{\ell_2} \, = \, 
\sum_{m_1,m_2 \ge 1} \p_\theta \psi^{\, m_1} \, \p_t \psi^{\, m_2} \, \sum_{\langle \bnu \rangle = \ell-m_1-m_2}  
\dfrac{1}{\bnu \, !} \, F_{\bnu +{\bf e}_{m_2}} \, \bpsi^{\bnu} 
\, = \, \sum_{\ell_1+\ell_2=\ell} \p_t \chi^{[\ell_1]} \, \p_\theta \psi^{\ell_2} \, ,
\end{equation*}
which completes the proof of Corollary \ref{corB1}.
\end{proof}

\section{The second symmetry formula}

Lemma \ref{lemB1} has another consequence in the case where the normal derivative $\p_{y_3}$ acts on $\chi^{[\ell]}$. 
The result is the following.

\begin{corollary}[The second symmetry formula. I]
\label{corB2}
The functions $\chi^{[\ell]}$, $\dot{\chi}^{[\ell]}$ and $\psi^m$ satisfy the relation:
\begin{equation}
\label{formulesym21}
\forall \, \ell \ge 0 \, ,\quad \p_{y_3} \chi^{[\ell]} -\dot{\chi}^{[\ell]} \, = \, 
\sum_{\ell_1+\ell_2+\ell_3=\ell} \p_{y_3} \chi^{[\ell_1]} \, \dot{\chi}^{[\ell_2]} \, \psi^{\ell_3} \, .
\end{equation}
\end{corollary}

\begin{proof}[Proof of Corollary \ref{corB2}]
Let us first look at the case $\ell=0$ in \eqref{formulesym21}. We have:
$$
\p_{y_3} \chi^{[0]} -\dot{\chi}^{[0]} \, = \, \chi' -\chi' \, = \, 0 \, ,
$$
and the right hand side in \eqref{formulesym21} vanishes since $\ell_3$ must be zero. We therefore assume $\ell \ge 1$ from now on, and use the 
decompositions \eqref{appB-equation6}, \eqref{appB-equation7} to compute:
\begin{equation*}
\p_{y_3} \chi^{[\ell]} -\dot{\chi}^{[\ell]} \, = \, \sum_{\avbmu = \ell} \, \dfrac{1}{\bmu \, !} \, \big( F_{\bmu}' -\dot{F}_{\bmu} \big) \, \bpsi^{\bmu} 
\, = \, \big( \cF_1' -\dot{\cF}_1 \big) \, \psi^\ell 
+\sum_{\substack{\avbmu = \ell \\ |\bmu| \ge 2}} \, \dfrac{1}{\bmu \, !} \, \big( \cF_{|\bmu|}' -\dot{\cF}_{|\bmu|} \big) \, \bpsi^{\bmu} \, ,
\end{equation*}
where we have kept the notation introduced in the proof of Lemma \ref{lemB1}, namely $F_{\bmu} =\cF_{|\bmu|}$, $\dot{F}_{\bmu} =\dot{\cF}_{|\bmu|}$. 
Observe that in the very last sum, $|\bmu| \ge 2$, hence $\bpsi^{\bmu}$ is at least a quadratic term in the $\psi^m$'s. In other words, we have isolated 
the linear term which will immediately cancel with part of the right hand side of \eqref{formulesym21}. Since $\cF_1 =\chi \, \chi'$ and $\dot{\cF}_1 
=\chi \, \chi''$, there holds $\cF_1' -\dot{\cF}_1=(\chi')^2$, and we have thus obtained:
\begin{equation}
\label{corB2-eq1}
\p_{y_3} \chi^{[\ell]} -\dot{\chi}^{[\ell]} \, = \, \big( \chi' \big)^2 \, \psi^\ell 
+\sum_{\substack{\avbmu = \ell \\ |\bmu| \ge 2}} \, \dfrac{1}{\bmu \, !} \, \big( \cF_{|\bmu|}' -\dot{\cF}_{|\bmu|} \big) \, \bpsi^{\bmu} \, .
\end{equation}
We now decompose the right hand side of \eqref{formulesym21} as:
\begin{align}
\sum_{\ell_1+\ell_2+\ell_3=\ell} \p_{y_3} \chi^{[\ell_1]} \, \dot{\chi}^{[\ell_2]} \, \psi^{\ell_3} \, 
& \, = \, \p_{y_3} \chi^{[0]} \, \dot{\chi}^{[0]} \, \psi^\ell 
+\sum_{\substack{\ell_1+\ell_2+\ell_3=\ell \\ \min (\ell_1,\ell_2) \ge 1}} \p_{y_3} \chi^{[\ell_1]} \, \dot{\chi}^{[\ell_2]} \, \psi^{\ell_3} \notag \\
& \, = \, \big( \chi' \big)^2 \, \psi^\ell 
+\sum_{\substack{\ell_1+\ell_2+\ell_3=\ell \\ \min (\ell_1,\ell_2) \ge 1}} \p_{y_3} \chi^{[\ell_1]} \, \dot{\chi}^{[\ell_2]} \, \psi^{\ell_3} \, .\label{corB2-eq2}
\end{align}
Comparing the right hand side of \eqref{corB2-eq1} with that of \eqref{corB2-eq2}, we see that the linear terms in the $\psi^m$'s cancel, 
which is a first step in the proof of \eqref{formulesym21}. It remains to see why all other quadratic, cubic and so on terms cancel. To prove 
this, we first need to express the sum on the right hand side of \eqref{corB2-eq2} as a function of the $\cF_\ell$'s and the $\dot{\cF}_\ell$'s. 
Using the decompositions \eqref{appB-equation6}, \eqref{appB-equation7} with he result of Lemma \ref{lemB1}, we have:
\begin{align}
\sum_{\substack{\ell_1+\ell_2+\ell_3=\ell \\ \min (\ell_1,\ell_2) \ge 1}} \p_{y_3} \chi^{[\ell_1]} \, \dot{\chi}^{[\ell_2]} \, \psi^{\ell_3} \, 
& \, = \, \sum_{\substack{\ell_1+\ell_2+\ell_3=\ell \\ \min (\ell_1,\ell_2) \ge 1}} \, 
\sum_{\substack{\avbmu = \ell_1 \\ \langle \bnu \rangle = \ell_2}} \, 
\dfrac{1}{\bmu \, !} \, \dfrac{1}{\bnu \, !} \, \cF_{|\bmu|}' \, \dot{\cF}_{|\bnu|} \, \bpsi^{\bmu +\bnu +{\bf e}_{\ell_3}} \notag \\
& \, = \, \sum_{\substack{\avbmu = \ell \\ |\bmu| \ge 2}} \, \left\{ \sum_{\substack{\bnu \le \bmu \\ |\bnu| = |\bmu|-1}} \, \dfrac{1}{\bnu \, !} \, 
\sum_{\btheta \le \bnu} \dfrac{\bnu \, !}{\btheta \, ! \, (\bnu -\btheta) \, !} \, \cF_{|\btheta|}' \, \dot{\cF}_{|\bnu|-|\btheta|} \right\} \, \bpsi^{\bmu} 
\, .\label{corB2-eq3}
\end{align}
In the interior sum on the right hand side of \eqref{corB2-eq3}, we collect the sequences $\btheta$ by increasing length, and use the so-called 
van der Monde convolution equalities \cite{Comtet} to derive:
\begin{equation}
\label{appB-calculaux}
\sum_{\btheta \le \bnu} \dfrac{\bnu \, !}{\btheta \, ! \, (\bnu -\btheta) \, !} \, \cF_{|\btheta|}' \, \dot{\cF}_{|\bnu|-|\btheta|} 
\, = \, \sum_{\nu=0}^{|\bnu|} \, \left\{ \sum_{\substack{\btheta \le \bnu \\ |\btheta| =\nu}} \dfrac{\bnu \, !}{\btheta \, ! \, (\bnu -\btheta) \, !} 
\right\} \, \cF_{\nu}' \, \dot{\cF}_{|\bnu|-\nu} 
\, = \, \sum_{\nu=0}^{|\bnu|} \, \binom{|\bnu|}{\nu} \, \cF_{\nu}' \, \dot{\cF}_{|\bnu|-\nu} \, .
\end{equation}
Plugging in the right hand side of \eqref{corB2-eq3}, we get:
\begin{align*}
\sum_{\substack{\ell_1+\ell_2+\ell_3=\ell \\ \min (\ell_1,\ell_2) \ge 1}} \p_{y_3} \chi^{[\ell_1]} \, \dot{\chi}^{[\ell_2]} \, \psi^{\ell_3} \, 
& \, = \, \sum_{\substack{\avbmu = \ell \\ |\bmu| \ge 2}} \, \left\{ \sum_{\substack{\bnu \le \bmu \\ |\bnu| = |\bmu|-1}} \, \dfrac{1}{\bnu \, !} \right\} 
\, \left\{ \sum_{\nu=0}^{|\bmu|-1} \, \binom{|\bmu|-1}{\nu} \, \cF_{\nu}' \, \dot{\cF}_{|\bmu|-1-\nu} \right\} \, \bpsi^{\bmu} \\
& \, = \, \sum_{\substack{\avbmu = \ell \\ |\bmu| \ge 2}} \, \dfrac{1}{\bmu \, !} \, \left\{ |\bmu| \, 
\sum_{\nu=0}^{|\bmu|-1} \, \binom{|\bmu|-1}{\nu} \, \cF_{\nu}' \, \dot{\cF}_{|\bmu|-1-\nu} \right\} \, \bpsi^{\bmu} \, .
\end{align*}
Substituting the latter expression in \eqref{corB2-eq2} and then subtracting \eqref{corB2-eq1} and \eqref{corB2-eq2}, we have obtained 
so far:
\begin{multline*}
\p_{y_3} \chi^{[\ell]} -\dot{\chi}^{[\ell]} -\sum_{\ell_1+\ell_2+\ell_3=\ell} \p_{y_3} \chi^{[\ell_1]} \, \dot{\chi}^{[\ell_2]} \, \psi^{\ell_3} \\
= \, \sum_{\substack{\avbmu = \ell \\ |\bmu| \ge 2}} \, \dfrac{1}{\bmu \, !} \, \left\{ \cF_{|\bmu|}' -\dot{\cF}_{|\bmu|} -|\bmu| \, 
\sum_{\nu=0}^{|\bmu|-1} \, \binom{|\bmu|-1}{\nu} \, \cF_{\nu}' \, \dot{\cF}_{|\bmu|-1-\nu} \right\} \, \bpsi^{\bmu} \, .
\end{multline*}

In order to prove the validity of \eqref{formulesym21}, we see that it is necessary and sufficient to prove the following relation between 
the functions $\cF_\ell$ and $\dot{\cF}_\ell$:
\begin{equation*}
\forall \, L \ge 2 \, ,\quad \cF_L' -\dot{\cF}_L -L \, \sum_{\nu=0}^{L-1} \, \binom{L-1}{\nu} \, \cF_{\nu}' \, \dot{\cF}_{L-1-\nu} \, = \, 0 \, .
\end{equation*}
which is equivalent to:
\begin{equation}
\label{corB2-eq4}
\forall \, L \ge 2 \, ,\quad \boxed{\dfrac{1}{L \, !} \, \cF_L' -\dfrac{1}{L \, !} \, \dot{\cF}_L -\sum_{\nu=0}^{L-1} \, \dfrac{1}{\nu \, !} \,\cF_{\nu}' \, 
\dfrac{1}{(L-1-\nu) \, !} \, \dot{\cF}_{L-1-\nu} \, = \, 0} \, .
\end{equation}
Let us observe that the relation \eqref{corB2-eq4} obviously holds for $L=0$ and $L=1$ (use $\cF_0=\chi$, $\cF_1=\chi \, \chi'$, 
$\dot{\cF}_0=\chi'$, $\dot{\cF}_1=\chi \, \chi''$), hence it will hold eventually for any $L \in \N$ and not only for $L \ge 2$. To prove 
the validity of \eqref{corB2-eq4}, we need to go back to the recursive relations \eqref{decompcFL}, \eqref{decompcFpointL}. We 
first need to make these expressions a little bit more explicit by counting how many partitions $(\cL_1,\dots,\cL_N) \in P(L,N)$ 
give rise to the same product $\prod_n \cX_{\# \cL_n}$. Let therefore $(\cL_1,\dots,\cL_N) \in P(L,N)$ and arrange the cardinals 
of the $\cL_n$'s as:
$$
\underbrace{c_1}_{N_1 \, \text{times}} < \cdots < \underbrace{c_J}_{N_J \, \text{times}} \, ,
$$
with $N_1+\cdots+N_J=N$ and $N_j \, c_j=L$. The group of permutations $\mathfrak{S}_L$ acts on the set of partitions $P(L,N)$ 
by acting (on the left hand side) on $\{ 1,\dots,L\}$ and therefore on its subsets. The orbit $\{ \sigma \cdot (\cL_1,\dots,\cL_N) \, / \, 
\sigma \in \mathfrak{S}_L \}$ consists of all partitions $(\cL'_1,\dots,\cL'_N)$ that share the same cardinals with $(\cL_1,\dots,\cL_N)$ 
(counting cardinals with multiplicity). Hence the number of partitions with same cardinals as $(\cL_1,\dots,\cL_N)$ equals (use the 
orbit-stabilizer theorem):
$$
\dfrac{L \, !}{N_1 \, ! \, \cdots N_J \, ! \, (c_1 \, !)^{N_1} \cdots (c_J \, !)^{N_J}} \, .
$$
Hence the recursive formulas \eqref{decompcFL} and \eqref{decompcFpointL} reduce to:
\begin{align}
\forall \, L \ge 1 \, ,\quad \dfrac{1}{L\, !} \, \cF_L \, & \, = \, \sum_{1 \, n_1 +\cdots +L \, n_L =L} \, 
\dfrac{1}{n_1 \, ! \cdots n_L \, !} \, \chi^{(n_1+\cdots+n_L)} \, \prod_{\nu=1}^L \, \left( \dfrac{\cF_{\nu-1}}{(\nu-1) \, !} \right)^{n_\nu} 
\, ,\label{decompcFL'} \\
\dfrac{1}{L\, !} \, \dot{\cF}_L \, & \, = \, \sum_{1 \, n_1 +\cdots +L \, n_L =L} \, 
\dfrac{1}{n_1 \, ! \cdots n_L \, !} \, \chi^{(1+n_1+\cdots+n_L)} \, \prod_{\nu=1}^L \, \left( \dfrac{\cF_{\nu-1}}{(\nu-1) \, !} \right)^{n_\nu} 
\, ,\label{decompcFpointL'}
\end{align}
where we have used the relation $\cX_\nu =\nu \, \cF_{\nu-1}$.

It is rather straightforward now to obtain \eqref{corB2-eq4}. This is mostly a matter of rewriting \eqref{decompcFL'} and 
\eqref{decompcFpointL'} in a convenient way. We introduce the notation $\bF_L := \cF_L / (L\, !)$ and $\dot{\bF}_L := 
\dot{\cF}_L / (L\, !)$ for all $L \ge 0$. Then \eqref{decompcFL'} and \eqref{decompcFpointL'} can be rewritten as:
\begin{align}
\forall \, L \ge 1 \, ,\quad \bF_L \, & \, = \, \sum_{\langle {\bf n} \rangle =L} \, \dfrac{1}{{\bf n} \, !} \, \chi^{(|{\bf n}|)} \, 
\Big( \bF_0,\bF_1,\cdots \Big)^{{\bf n}} \, ,\label{decompcFL''} \\
\dot{\bF}_L \, & \, = \, \sum_{\langle {\bf n} \rangle =L} \, \dfrac{1}{{\bf n} \, !} \, \chi^{(1+|{\bf n}|)} \, \Big( \bF_0,\bF_1,\cdots \Big)^{{\bf n}} 
\, ,\label{decompcFpointL''}
\end{align}
where ${\bf n} =(n_1,n_2,\dots)$ denotes a sequence of integers with finite length, and:
$$
\Big( \bF_0,\bF_1,\cdots \Big)^{{\bf n}} \, := \, \prod_{m \ge 1} \big( \bF_{m-1} \big)^{{\bf n}_m} \, ,
$$
the product involving finitely many terms. Let us observe that \eqref{decompcFL''}, \eqref{decompcFpointL''} also hold for $L=0$ 
with the convention $(\bF_0,\bF_1,\cdots)^{{\bf n}} =1$ if ${\bf n}$ is the zero sequence. We now rewrite the left hand side of 
\eqref{corB2-eq4} using these new functions $\bF_L$ and $\dot{\bF}_\ell$ and use the recursive formulas \eqref{decompcFL''}, 
\eqref{decompcFpointL''}:
\begin{align*}
\dfrac{\cF_L'}{L \, !} -\dfrac{\dot{\cF}_L}{L \, !} -\sum_{\nu=0}^{L-1} \, \dfrac{\cF_\nu'}{\nu \, !} \, \dfrac{\dot{\cF}_{L-1-\nu}}{(L-1-\nu) \, !} \, 
& \, = \, \bF_L' -\dot{\bF}_L -\sum_{\nu=0}^{L-1} \, \bF_\nu' \, \dot{\bF}_{L-1-\nu} \\
& \, = \,  \sum_{k \ge 0} \, \bF_k' \, \sum_{\substack{\langle {\bf n} \rangle =L \\ {\bf n} \ge {\bf e}_{k+1}}} \, 
\dfrac{1}{({\bf n}-{\bf e}_{k+1}) \, !} \, \chi^{(|{\bf n}|)} \, \Big( \bF_0,\bF_1,\cdots \Big)^{{\bf n}-{\bf e}_{k+1}} \\
& \qquad -\sum_{\nu=0}^{L-1} \, \bF_\nu' \, \sum_{\langle {\bf m} \rangle =L-1-\nu} \, 
\dfrac{1}{{\bf m} \, !} \, \chi^{(1+|{\bf m}|)} \, \Big( \bF_0,\bF_1,\cdots \Big)^{{\bf m}} \, = \, 0 \, .
\end{align*}
We have thus proved the validity of \eqref{corB2-eq4}, and this completes the proof of Corollary \ref{corB2}.
\end{proof}

\noindent The last result we need to prove the second symmetry formula (Lemma \ref{lem_symmetry_2}) is the following.

\begin{proposition}[The second symmetry formula. II]
\label{propB1}
The functions $\chi^{[\ell]}$, $\dot{\chi}^{[\ell]}$ and $\psi^m$ satisfy the relation:
\begin{equation}
\label{propB1-eq}
\forall \, \ell \ge 0 \, ,\quad \p_\theta \dot{\chi}^{[\ell]} \, = \, 
\sum_{\ell_1+\ell_2+\ell_3=\ell} \chi^{[\ell_1]} \, \p_{y_3} \dot{\chi}^{[\ell_2]} \, \p_\theta \psi^{\ell_3} \, ,
\end{equation}
and similar relations with any other tangential partial derivative chosen among $\{ \p_t,\p_{y_1},\p_{y_2} \}$. 
(We keep using the convention $\psi^0 =0$.)
\end{proposition}

\begin{proof}[Proof of Proposition \ref{propB1}]
Let us first observe that \eqref{propB1-eq} is clearly satisfied for $\ell=0$ since $\p_\theta \dot{\chi}^{[0]}=0$ and $\psi^0 =0$, and 
\eqref{propB1-eq} also holds for $\ell=1$ since we can compute:
\begin{align*}
\p_\theta \dot{\chi}^{[1]} \, & \, = \, \dot{\cF}_1 \, \p_\theta \psi^1 \, = \, \chi \, \chi'' \, \p_\theta \psi^1 \, ,\\
\sum_{\ell_1+\ell_2+\ell_3=1} \chi^{[\ell_1]} \, \p_{y_3} \dot{\chi}^{[\ell_2]} \, \p_\theta \psi^{\ell_3} \, 
& \, = \, \chi^{[0]} \, \p_{y_3} \dot{\chi}^{[0]} \, \p_\theta \psi^1 \, = \, \chi \, \chi'' \, \p_\theta \psi^1 \, .
\end{align*}
The equality between the left and right hand sides in \eqref{propB1-eq} for $\ell=1$ follows here from the relation $\dot{\cF}_1 
=\cF_0 \, \dot{\cF}_0'$, a generalization of which will yield \eqref{propB1-eq} for any $\ell \ge 2$.

Let us now assume $\ell \ge 2$, and simplify the right hand side of \eqref{propB1-eq}:
\begin{align*}
\sum_{\ell_1+\ell_2+\ell_3=\ell} \chi^{[\ell_1]} \, \p_{y_3} \dot{\chi}^{[\ell_2]} \, \p_\theta \psi^{\ell_3} \, 
& \, = \, \chi \, \chi'' \, \p_\theta \psi^\ell +\sum_{\substack{\ell_1+\ell_2+\ell_3=\ell \\ \ell_1+\ell_2 \ge 1}} \, 
\chi^{[\ell_1]} \, \p_{y_3} \dot{\chi}^{[\ell_2]} \, \p_\theta \psi^{\ell_3} \\
& \, = \, \chi \, \chi'' \, \p_\theta \psi^\ell +\sum_{\substack{\ell_1+\ell_2+\ell_3=\ell \\ \ell_1+\ell_2 \ge 1}} \, 
\p_\theta \psi^{\ell_3} \, \sum_{\substack{\avbmu =\ell_1 \\ \langle \bnu \rangle =\ell_2}} \, \dfrac{1}{\bmu \, !} \, \dfrac{1}{\bnu \, !} \, 
F_{\bmu} \, \dot{F}_{\bnu}' \, \bpsi^{\bmu+\bnu} \\
& \, = \, \chi \, \chi'' \, \p_\theta \psi^\ell +\sum_{k \ge 1} \, \p_\theta \psi^k \, 
\sum_{\substack{ \langle \bmu+\bnu \rangle =\ell -k \\ |\bmu+\bnu| \ge 1}} \, \dfrac{1}{\bmu \, !} \, \dfrac{1}{\bnu \, !} \, 
F_{\bmu} \, \dot{F}_{\bnu}' \, \bpsi^{\bmu+\bnu} \\
& \, = \, \chi \, \chi'' \, \p_\theta \psi^\ell +\sum_{k \ge 1} \, \p_\theta \psi^k \, \sum_{\substack{ \langle \bnu \rangle =\ell -k \\ |\bnu| \ge 1}} \, 
\dfrac{1}{\bnu \, !} \, \left( \sum_{\nu=0}^{|\bnu|} \, \binom{|\bnu|}{\nu} \, \cF_{\nu} \, \dot{\cF}_{|\bnu|-\nu}' \right) \, \bpsi^{\bnu} \, ,
\end{align*}
where we have again used the calculation \eqref{appB-calculaux} to obtain the final expression.

We use the decomposition \eqref{appB-equation7} to get a similar expression of the left hand side of \eqref{propB1-eq}:
\begin{align*}
\p_\theta \dot{\chi}^{[\ell]} \, & \, = \, \sum_{\avbmu =\ell} \, \dfrac{1}{\bmu \, !} \, \dot{F}_{\bmu} \, \p_\theta \Big( \bpsi^{\bmu} \Big) \, = \, 
\dot{\cF}_1 \, \p_\theta \psi^\ell +\sum_{k \ge 1} \, \p_\theta \psi^k \, \sum_{\substack{ \avbmu =\ell \\ \bmu \ge {\bf e}_k \, , \, |\bmu| \ge 2}} 
\, \dfrac{1}{(\bmu -{\bf e}_k) \, !} \, \dot{F}_{\bmu} \, \bpsi^{\bmu-{\bf e}_k} \\
& \, = \, \chi \, \chi'' \, \p_\theta \psi^\ell +\sum_{k \ge 1} \, \p_\theta \psi^k \, \sum_{\substack{ \langle \bnu \rangle =\ell -k \\ |\bnu| \ge 1}} \, 
\dfrac{1}{\bnu \, !} \, \dot{\cF}_{|\bnu|+1} \, \bpsi^{\bnu} \, ,
\end{align*}
and we thus get:
$$
\p_\theta \dot{\chi}^{[\ell]} -\sum_{\ell_1+\ell_2+\ell_3=\ell} \chi^{[\ell_1]} \, \p_{y_3} \dot{\chi}^{[\ell_2]} \, \p_\theta \psi^{\ell_3} \, = \, 
\sum_{k \ge 1} \, \p_\theta \psi^k \, \sum_{\substack{ \langle \bnu \rangle =\ell -k \\ |\bnu| \ge 1}} \, \dfrac{1}{\bnu \, !} \, \left( 
\dot{\cF}_{|\bnu|+1} -\sum_{\nu=0}^{|\bnu|} \, \binom{|\bnu|}{\nu} \, \cF_{\nu} \, \dot{\cF}_{|\bnu|-\nu}' \right) \, \bpsi^{\bnu} \, .
$$
The result of Proposition \ref{propB1} will therefore be valid if we can prove the identity:
\begin{equation}
\label{propB1-id}
\forall \, L \ge 1 \, ,\quad \dot{\cF}_{L+1} \, = \, \sum_{\ell=0}^L \, \binom{L}{\ell} \, \dot{\cF}_\ell' \, \cF_{L-\ell} \,  \, .
\end{equation}
The problem at this stage is that the latter formula does not seem to follow (at least, not in an obvious way and despite 
repeated efforts...) from the recursive formulas \eqref{decompcFL''}, \eqref{decompcFpointL''}. We thus need to find another 
representation for the functions $\cF_\ell$ and $\dot{\cF}_\ell$ in order to derive \eqref{propB1-id}. Hopefully the solution is 
provided by the so-called Lagrange inversion formula \cite{Henrici,Comtet}, or rather one of its direct consequences which 
we recall right now.

\begin{lemma}[\cite{Comtet}, Theorem C, p.150]
\label{lemLagrange}
Let $F$ be a formal series and let $y$ denote the formal series in $x$ such that:
$$
y \, = \, y_0 +x \, F(y) \, ,
$$
with $y_0$ a constant term. Then for any formal series $G$, there holds:
$$
G(y) \, = \, G(y_0) +\sum_{n \ge 1} \, \dfrac{x^n}{n \, !} \, \dfrac{{\rm d}^{n-1}}{{\rm d}y_0^{n-1}} \Big\{ G'(y_0) \, F^n (y_0) \Big\} \, .
$$
\end{lemma}

\noindent Let us apply Lemma \ref{lemLagrange} to \eqref{defchi'}. We first invert the equation:
$$
y_3 \, = \, x_3 -\psi \, \chi (x_3) \, ,
$$
and obtain thanks to Lemma \ref{lemLagrange} (in the sense of formal series in $\psi$):
\begin{align*}
\chi (x_3) \, & \, = \, \chi (y_3) +\sum_{n \ge 1} \, \dfrac{\psi^n}{n \, !} \, \dfrac{{\rm d}^{n-1}}{{\rm d}y_3^{n-1}} \Big\{ \chi'(y_3) \, \chi^n (y_3) \Big\} \, ,\\
\chi' (x_3) \, & \, = \, \chi' (y_3) +\sum_{n \ge 1} \, \dfrac{\psi^n}{n \, !} \, \dfrac{{\rm d}^{n-1}}{{\rm d}y_3^{n-1}} \Big\{ \chi''(y_3) \, \chi^n (y_3) \Big\} \, .
\end{align*}
It then remains to substitute
$$
\psi \, = \, \sum_{m \ge 1} \, \eps^m \, \psi^m \, ,
$$
and to recollect the previous expressions as formal series in $\eps$. From the definition of the functions $\chi^{[m]}$, $\dot{\chi}^{[m]}$, we get:
\begin{align*}
\forall \, m \ge 1 \, ,\quad & \, \chi^{[m]} \, = \, 
\sum_{n=1}^m \, \dfrac{1}{n \, !} \, \dfrac{{\rm d}^{n-1}}{{\rm d}y_3^{n-1}} \Big\{ \chi'(y_3) \, \chi^n (y_3) \Big\} 
\, \sum_{\ell_1+\cdots +\ell_n =m} \psi^{\ell_1} \cdots \psi^{\ell_n} \, , \\
& \, \dot{\chi}^{[m]} \, = \, \sum_{n=1}^m \, \dfrac{1}{n \, !} \, \dfrac{{\rm d}^{n-1}}{{\rm d}y_3^{n-1}} \Big\{ \chi''(y_3) \, \chi^n (y_3) \Big\} 
\, \sum_{\ell_1+\cdots +\ell_n =m} \psi^{\ell_1} \cdots \psi^{\ell_n} \, ,
\end{align*}
where the sums run over all $n$-tuples $(\ell_1,\dots,\ell_n) \in \N^n$ such that $\min_i \ell_i \ge 1$ and $ \ell_1+\cdots +\ell_n =m$ (and 
the families are \emph{ordered}, meaning for instance that both couples $(1,2)$ and $(2,1)$ are taken into account for $m=3$, $n=2$). 
Taking multiplicities into account, we end up with the decompositions\footnote{In particular, we can directly check on the expressions 
\eqref{formule_chim}, \eqref{formule_chipointm} that the functions $\chi^{[m]}$ and $\dot{\chi}^{[m]}$ vanish on $\Gamma_0$ and 
$\Gamma^\pm$ for any $m \ge 1$.}:
\begin{align}
\forall \, m \ge 1 \, ,\quad & \, \boxed{\chi^{[m]} \, = \, 
\sum_{n=1}^m \, \dfrac{{\rm d}^{n-1}}{{\rm d}y_3^{n-1}} \Big\{ \chi'(y_3) \, \chi^n (y_3) \Big\} \, 
\sum_{\substack{\avbmu =m \\ |\bmu| =n}} \dfrac{1}{\bmu \, !} \, \bpsi^{\bmu}} \, ,\label{formule_chim} \\
& \, \boxed{\dot{\chi}^{[m]} \, = \, \sum_{n=1}^m \, \dfrac{{\rm d}^{n-1}}{{\rm d}y_3^{n-1}} \Big\{ \chi''(y_3) \, \chi^n (y_3) \Big\} \, 
\sum_{\substack{\avbmu =m \\ |\bmu| =n}} \dfrac{1}{\bmu \, !} \, \bpsi^{\bmu}} \, .\label{formule_chipointm}
\end{align}
Comparing with the decompositions \eqref{appB-equation6}, \eqref{appB-equation7}, we have thus derived the expressions:
\begin{equation*}
\forall \, \bmu \, ,\quad F_{\bmu} (y_3) \, = \, \dfrac{{\rm d}^{|\bmu|-1}}{{\rm d}y_3^{|\bmu|-1}} \Big\{ \chi'(y_3) \, \chi^{|\bmu|} (y_3) \Big\} \, ,\quad 
\dot{F}_{\bmu} (y_3) \, = \, \dfrac{{\rm d}^{|\bmu|-1}}{{\rm d}y_3^{|\bmu|-1}} \Big\{ \chi''(y_3) \, \chi^{|\bmu|} (y_3) \Big\} \, ,
\end{equation*}
from where the invariance result of Lemma \ref{lemB1} looks incredibly simple ! With the notation \eqref{notation-XFrond}, we get:
\begin{align}
\forall \, L \ge 1 \, ,\quad \cF_L \, & \, = \, \dfrac{{\rm d}^{L-1}}{{\rm d}y_3^{L-1}} \Big\{ \chi'(y_3) \, \chi^L (y_3) \Big\} 
\, = \, \dfrac{1}{L+1} \, \dfrac{{\rm d}^L}{{\rm d}y_3^L} \Big\{ \chi^{L+1} (y_3) \Big\} \, ,\label{expressioncompacte} \\
\dot{\cF}_L \, & \, = \, \dfrac{{\rm d}^{L-1}}{{\rm d}y_3^{L-1}} \Big\{ \chi''(y_3) \, \chi^L (y_3) \Big\} \, .\label{expressioncompacte'}
\end{align}
Let us observe that the expression \eqref{expressioncompacte} also holds for $L=0$ since we have $\cF_0 = \chi$. Using 
$\dot{\cF}_0 = \chi'$, we also deduce from \eqref{expressioncompacte'} the relation:
$$
\forall \, L \ge 0 \, ,\quad \dot{\cF}_L' \, = \, \big\{ \chi'' \, \chi^L \big\}^{(L)} \, .
$$
Therefore, proving the validity of \eqref{propB1-id} amounts to showing the formula:
\begin{equation}
\label{propB1-id'}
\forall \, L \ge 1 \, ,\quad \boxed{(L+1) \, \big\{ \chi'' \, \chi^{L+1} \big\}^{(L)} \, = \, \sum_{\ell=0}^L \, \binom{L+1}{\ell} \, 
\big\{ \chi'' \, \chi^\ell \Big\}^{(\ell)} \, \big\{ \chi^{L+1-\ell} \big\}^{(L-\ell)}} \, ,
\end{equation}
which looks like some (weird) kind of Leibniz formula. Let us recall that in \eqref{propB1-id'}, $\chi$ is a $\cC^\infty$ function on $\R$ 
with compact support. We have not found a (short) combinatorial proof of \eqref{propB1-id'} and we therefore propose an alternative 
method which consists in proving \eqref{propB1-id'} on the Fourier side (just like the Leibniz formula can also be obtained by using 
the binomial identity after performing a Fourier transform). For a given (fixed) integer $L \ge 1$, we define the function:
$$
\Theta \, := \, (L+1) \, \big\{ \chi'' \, \chi^{L+1} \big\}^{(L)} \, - \, \sum_{\ell=0}^L \, \binom{L+1}{\ell} \, 
\big\{ \chi'' \, \chi^\ell \Big\}^{(\ell)} \, \big\{ \chi^{L+1-\ell} \big\}^{(L-\ell)} \, ,
$$
and compute its Fourier transform. We get:
\begin{multline}
\label{propB1-fourier1}
\widehat{\Theta} (\xi_0) \, = \, (2\, \pi)^{L+1} \, i^{L+2} \, \int_{\R^{L+1}} (\xi_0 -\xi_1)^2 \, \left( (L+1) \, \xi_0^L \, - \, \sum_{\ell=0}^L \, 
\binom{L+1}{\ell} \, (\xi_0 -\xi_{\ell+1})^\ell \, \xi_{\ell+1}^{L-\ell} \right) \times \\
\times \widehat{\chi} (\xi_0-\xi_1) \, \widehat{\chi} (\xi_1-\xi_2) \, \cdots \, \widehat{\chi} (\xi_L-\xi_{L+1}) \, \widehat{\chi} (\xi_{L+1}) \, 
{\rm d}\xi_1 \cdots {\rm d}\xi_{L+1} \, . 
\end{multline}
For future use, we define the polynomial:
\begin{equation}
\label{propB1defQ}
Q(\eta_0,\eta_1,\dots,\eta_{L+1}) \, := \, (L+1) \, \big( \eta_0 +\cdots +\eta_{L+1} \big)^L \, - \, \sum_{\ell=0}^L \, \binom{L+1}{\ell} \, 
\big( \eta_0 +\cdots +\eta_\ell \big)^\ell \, \big( \eta_{\ell+1} +\cdots +\eta_{L+1} \big)^{L-\ell} \, ,
\end{equation}
so \eqref{propB1-fourier1} reads:
\begin{multline}
\label{propB1-fourier2}
\widehat{\Theta} (\xi_0) \, = \, (2\, \pi)^{L+1} \, i^{L+2} \, 
\int_{\R^{L+1}} (\xi_0 -\xi_1)^2 \, Q(\xi_0-\xi_1, \xi_1-\xi_2, \dots, \xi_L-\xi_{L+1}, \xi_{L+1}) \, \times \\
\times \, \widehat{\chi} (\xi_0-\xi_1) \, \widehat{\chi} (\xi_1-\xi_2) \, \cdots \, \widehat{\chi} (\xi_L-\xi_{L+1}) \, \widehat{\chi} (\xi_{L+1}) \, 
{\rm d}\xi_1 \cdots {\rm d}\xi_{L+1} \, .
\end{multline}
The important observation now is that affine changes of variables with respect to $(\xi_2,\dots,\xi_{L+1})$ allow to symmetrize 
$Q$ with respect to its last $L+1$ arguments, while leaving the frequencies $(\xi_0,\xi_1)$, hence $\xi_0-\xi_1$, unchanged. For 
instance, the change of variables:
$$
(\xi_1,\dots,\xi_{L+1}) \, \longmapsto \, (\xi_1,\xi_1-\xi_2+\xi_3,\xi_3,\dots,\xi_{L+1}) \, ,
$$
leaves the product
$$
(\xi_0 -\xi_1)^2 \, 
\widehat{\chi} (\xi_0-\xi_1) \, \widehat{\chi} (\xi_1-\xi_2) \, \cdots \, \widehat{\chi} (\xi_L-\xi_{L+1}) \, \widehat{\chi} (\xi_{L+1}) \, ,
$$
in \eqref{propB1-fourier2} unchanged but modifies the factor $Q(\xi_0-\xi_1, \dots, \xi_L-\xi_{L+1}, \xi_{L+1})$ into\footnote{Observe 
the permutation highlighted in blue.}
$$
Q (\xi_0-\xi_1, {\color{blue} \xi_2-\xi_3}, {\color{blue} \xi_1-\xi_2}, \xi_3 -\xi_4, \dots, \xi_L -\xi_{L+1}, \xi_{L+1}) \, .
$$
Other changes of variables (all with jacobian one in absolute value) allow for any transposition, and consequently any permutation, with 
respect to the last $L+1$ arguments of $Q$ in \eqref{propB1-fourier2}. In other words, we can rewrite \eqref{propB1-fourier2} as:
\begin{multline}
\label{propB1-fourier3}
\widehat{\Theta} (\xi_0) \, = \, (2\, \pi)^{L+1} \, i^{L+2} \, 
\int_{\R^{L+1}} (\xi_0 -\xi_1)^2 \, Q_\sharp(\xi_0-\xi_1, \xi_1-\xi_2, \dots, \xi_L-\xi_{L+1}, \xi_{L+1}) \, \times \\
\times \, \widehat{\chi} (\xi_0-\xi_1) \, \widehat{\chi} (\xi_1-\xi_2) \, \cdots \, \widehat{\chi} (\xi_L-\xi_{L+1}) \, \widehat{\chi} (\xi_{L+1}) \, 
{\rm d}\xi_1 \cdots {\rm d}\xi_{L+1} \, ,
\end{multline}
where the polynomial $Q_\sharp$ in \eqref{propB1-fourier3} is obtained from the polynomial $Q$ in \eqref{propB1defQ} by averaging 
with respect to the last $L+1$ arguments on the group of permutations $\mathfrak{S}_{L+1}$:
$$
Q_\sharp(\eta_0,\eta_1,\dots,\eta_{L+1}) \, := \, \dfrac{1}{(L+1) \, !} \, \sum_{\sigma \in \mathfrak{S}_{L+1}} \, 
Q(\eta_0,\eta_{\sigma(1)},\dots,\eta_{\sigma(L+1)}) \, .
$$
Starting from \eqref{propB1defQ}, a little bit of combinatorial analysis yields:
$$
Q_\sharp(\eta_0,\eta_1,\dots,\eta_{L+1}) \, = \, (L+1) \, \big( \eta_0 +\cdots +\eta_{L+1} \big)^L \, - \, 
\sum_{\substack{E \subset \{ 1,\dots,L+1 \} \\ 0 \le \sharp E \le L}} \Big( \eta_0 +\eta(E) \Big)^{\sharp E} \, 
\Big( \eta(E^c) \Big)^{\sharp E^c -1} \, ,
$$
where we have used the notation:
$$
\eta (E) \, := \, \sum_{j \in E} \eta_j \, ,
$$
for any subset $E$ of $\{ 1,\dots,L+1 \}$, and $E^c$ stands for the complementary set of $E$ in $\{ 1,\dots,L+1 \}$. We can rewrite 
the expression of $Q_\sharp$ into the even more symmetric form\footnote{To prove that the right hand side of \eqref{propB1exprQ} 
and the previous expression of $Q_\sharp$ coincide, start from \eqref{propB1exprQ} and divide the subsets $F$ of $\{ 0,\dots,L+1 \}$ 
into those that contain $0$ and those that do not. For the latter, parametrize the sum by $F^c$ rather than by $F$.}:
\begin{equation}
\label{propB1exprQ}
Q_\sharp(\eta_0,\eta_1,\dots,\eta_{L+1}) \, = \, (L+1) \, \big( \eta_0 +\cdots +\eta_{L+1} \big)^L \, - \, \dfrac{1}{2} \, 
\sum_{\substack{F \subset \{ 0,\dots,L+1 \} \\ 1 \le \sharp F \le L+1}} \Big( \eta(F) \Big)^{\sharp F-1} \, 
\Big( \eta(F^c) \Big)^{\sharp F^c -1} \, ,
\end{equation}
from where it follows that $Q_\sharp$ is a homogeneous degree $L$, symmetric expression of all its arguments. It remains to deduce 
from the formula \eqref{propB1exprQ} that $Q_\sharp$ vanishes, which will imply in the relation \eqref{propB1-fourier3} that $\Theta$ 
vanishes (and this will eventually prove the validity of \eqref{propB1-id'} and complete the proof of Proposition \ref{propB1} !).

We now show that the polynomial $Q_\sharp$, whose expression is given by \eqref{propB1exprQ}, is zero. This is proved by an 
induction argument that passes from $L$ to $L+2$, so we first examine the cases $L=1$ and $L=2$. For $L=1$, we have:
\begin{align*}
Q_\sharp(\eta_0,\eta_1,\eta_2) \, = \, 2 \, (\eta_0+\eta_1+\eta_2) \, 
&- \, \dfrac{1}{2} \, \Big( (\eta_1+\eta_2) +(\eta_0+\eta_2) +(\eta_0+\eta_1) \Big) \\
&- \, \dfrac{1}{2} \, \Big( (\eta_1+\eta_2) +(\eta_0+\eta_2) +(\eta_0+\eta_1) \Big) \, = \, 0 \, ,
\end{align*}
and for $L=2$, we have\footnote{To prove that the expression equals zero, just compute the $\eta_0^2$ and $\eta_0 \, \eta_1$ terms 
and use the symmetry with respect to all arguments.}:
\begin{align*}
Q_\sharp(\eta_0,\dots,\eta_3) \, = \, & \, 3 \, (\eta_0+\eta_1+\eta_2+\eta_3)^2 \\
& \, -(\eta_1+\eta_2+\eta_3)^2 -(\eta_0+\eta_2+\eta_3)^2 -(\eta_0+\eta_1+\eta_3)^2 -(\eta_0+\eta_1+\eta_2)^2 \\
& \, -(\eta_0+\eta_1) \, (\eta_2+\eta_3) -(\eta_0+\eta_2) \, (\eta_1+\eta_3) -(\eta_0+\eta_3) \, (\eta_1+\eta_2) \, = \, 0 \, .
\end{align*}
Let us therefore assume that up to some integer $L$, there holds:
\begin{equation}
\label{propB1-recurrence}
(L+1) \, \big( \eta_0 +\cdots +\eta_{L+1} \big)^L \, = \, \dfrac{1}{2} \, \sum_{\substack{F \subset \{ 0,\dots,L+1 \} \\ 1 \le \sharp F \le L+1}} 
\Big( \eta(F) \Big)^{\sharp F-1} \, \Big( \eta(F^c) \Big)^{\sharp F^c -1} \, ,
\end{equation}
and we are now going to try that the same property holds with $L+2$ instead of $L$. In other words, we wish to prove the formula:
\begin{equation}
\label{propB1-polynome}
(L+3) \, \big( \eta_0 +\cdots +\eta_{L+3} \big)^{L+2} \, - \, \dfrac{1}{2} \, \sum_{\substack{F \subset \{ 0,\dots,L+3 \} \\ 1 \le \sharp F \le L+3}} 
\Big( \eta(F) \Big)^{\sharp F-1} \, \Big( \eta(F^c) \Big)^{\sharp F^c -1} \, = \, 0 \, .
\end{equation}
The expression on the left hand side of \eqref{propB1-polynome} is a symmetric polynomial function of $(\eta_0,\dots,\eta_{L+3})$, hence 
can be represented as a polynomial function of the elementary symmetric expressions of $(\eta_0,\dots,\eta_{L+3})$, see \cite{Lang}. 
Since furthermore the left hand side of \eqref{propB1-polynome} is homogeneous degree $L+2$ while having $L+4$ arguments, we 
can write\footnote{Observe that only the elementary symmetric functions up to $L+2$ come into play and not the two last ones, namely 
$e_{L+3}$ and $e_{L+4}$.}:
\begin{multline*}
(L+3) \, \big( \eta_0 +\cdots +\eta_{L+3} \big)^{L+2} \, - \, \dfrac{1}{2} \, \sum_{\substack{F \subset \{ 0,\dots,L+3 \} \\ 1 \le \sharp F \le L+3}} 
\Big( \eta(F) \Big)^{\sharp F-1} \, \Big( \eta(F^c) \Big)^{\sharp F^c -1} \\
= \, \cQ \Big( e_1(\eta_0,\dots,\eta_{L+3}),\dots,e_{L+2}(\eta_0,\dots,\eta_{L+3}) \Big) \, ,
\end{multline*}
for some suitable polynomial $\cQ$. Introducing now some complex numbers $\widetilde{\eta}_0,\dots,\widetilde{\eta}_{L+1}$ such that:
$$
e_1(\widetilde{\eta}_0,\dots,\widetilde{\eta}_{L+1}) \, = \, e_1(\eta_0,\dots,\eta_{L+3}) \, ,\quad \dots \quad 
e_{L+2}(\widetilde{\eta}_0,\dots,\widetilde{\eta}_{L+1}) \, = \, e_{L+2}(\eta_0,\dots,\eta_{L+3}) \, ,
$$
we get\footnote{Here we use the fact that $e_j(\widetilde{\eta}_0,\dots,\widetilde{\eta}_{L+1})$ coincides with 
$e_j(\widetilde{\eta}_0,\dots,\widetilde{\eta}_{L+1},0,0)$ for any $j$.}:
\begin{align}
(L+3) \, \big( \eta_0 +\cdots +\eta_{L+3} \big)^{L+2} \, &- \, \dfrac{1}{2} \, \sum_{\substack{F \subset \{ 0,\dots,L+3 \} \\ 1 \le \sharp F \le L+3}} 
\Big( \eta(F) \Big)^{\sharp F-1} \, \Big( \eta(F^c) \Big)^{\sharp F^c -1} \notag \\
= \, & \, \cQ \Big( e_1(\widetilde{\eta}_0,\dots,\widetilde{\eta}_{L+1},0,0),\dots,e_{L+2}(\widetilde{\eta}_0,\dots,\widetilde{\eta}_{L+1},0,0) \Big) \notag \\
= \, & \, (L+3) \, \big( \widetilde{\eta}_0 +\cdots +\widetilde{\eta}_{L+1} \big)^{L+2} \, - \, \dfrac{1}{2} \, 
\sum_{\substack{F \subset \{ 0,\dots,L+3 \} \\ 1 \le \sharp F \le L+3}} 
\Big( \widetilde{\eta}(F) \Big)^{\sharp F-1} \, \Big( \widetilde{\eta}(F^c) \Big)^{\sharp F^c -1} \, ,\label{relationpropB1eta}
\end{align}
where we use in \eqref{relationpropB1eta} the convention $\widetilde{\eta}_{L+2}=\widetilde{\eta}_{L+3}=0$. In other words, we have 
reduced to the case where (at least) two of the $\eta_j$'s, say the last two, are zero. We now divide the subsets $F$ of $\{ 0,\dots,L+3 \}$ 
into those that are contained in $\{ 0,\dots,L+1 \}$, those that contain only one element among $\{ L+2,L+3 \}$, and those that contain both 
$L+2$ and $L+3$. We can decompose the sum on the right hand side of \eqref{relationpropB1eta} as follows:
\begin{align*}
\sum_{\substack{F \subset \{ 0,\dots,L+3 \} \\ 1 \le \sharp F \le L+3}} 
\Big( \widetilde{\eta}(F) \Big)^{\sharp F-1} \, \Big( \widetilde{\eta}(F^c) \Big)^{\sharp F^c -1} \, 
= & \, \sum_{\substack{E \subset \{ 0,\dots,L+1 \} \\ 1 \le \sharp E \le L+2}} 
\Big( \widetilde{\eta}(E) \Big)^{\sharp E-1} \, \Big( \widetilde{\eta}(E^c) \Big)^{\sharp E^c +1} \\
& \, +2 \, \sum_{\substack{E \subset \{ 0,\dots,L+1 \} \\ 0 \le \sharp E \le L+2}} 
\Big( \widetilde{\eta}(E) \Big)^{\sharp E} \, \Big( \widetilde{\eta}(E^c) \Big)^{\sharp E^c} \\
& \, +\sum_{\substack{E \subset \{ 0,\dots,L+1 \} \\ 0 \le \sharp E \le L+1}} 
\Big( \widetilde{\eta}(E) \Big)^{\sharp E+1} \, \Big( \widetilde{\eta}(E^c) \Big)^{\sharp E^c-1} \\
= & \, \sum_{\substack{E \subset \{ 0,\dots,L+1 \} \\ 1 \le \sharp E \le L+{\color{red} 1}}} 
\Big( \widetilde{\eta}(E) \Big)^{\sharp E-1} \, \Big( \widetilde{\eta}(E^c) \Big)^{\sharp E^c +1} \\
& \, +2 \, \sum_{\substack{E \subset \{ 0,\dots,L+1 \} \\ 0 \le \sharp E \le L+2}} 
\Big( \widetilde{\eta}(E) \Big)^{\sharp E} \, \Big( \widetilde{\eta}(E^c) \Big)^{\sharp E^c} \\
& \, +\sum_{\substack{E \subset \{ 0,\dots,L+1 \} \\ {\color{red} 1} \le \sharp E \le L+1}} 
\Big( \widetilde{\eta}(E) \Big)^{\sharp E+1} \, \Big( \widetilde{\eta}(E^c) \Big)^{\sharp E^c-1} \, .
\end{align*}
Observe the modifications (in red) in two of the terms on the right hand side, which correspond to deleting two terms that contribute for 
zero. It remains to collect and rewrite some terms to get:
\begin{align*}
\sum_{\substack{F \subset \{ 0,\dots,L+3 \} \\ 1 \le \sharp F \le L+3}} 
\Big( \widetilde{\eta}(F) \Big)^{\sharp F-1} \, \Big( \widetilde{\eta}(F^c) \Big)^{\sharp F^c -1} \, 
= \, & \, \sum_{\substack{E \subset \{ 0,\dots,L+1 \} \\ 1 \le \sharp E \le L+1}} 
\Big( \widetilde{\eta}(E) \Big)^{\sharp E -1} \, \Big( \widetilde{\eta}(E^c) \Big)^{\sharp E^c -1} \, \Big( \widetilde{\eta}(E)+\widetilde{\eta}(E^c) \Big)^2 \\
& \, + \, 4 \, \big( \widetilde{\eta}_0 +\cdots +\widetilde{\eta}_{L+1} \big)^{L+2} \\
= \, & \, (2\, L+6) \, \big( \widetilde{\eta}_0 +\cdots +\widetilde{\eta}_{L+1} \big)^{L+2} \, ,
\end{align*}
where we have used the induction assumption \eqref{propB1-recurrence}. Going back to \eqref{relationpropB1eta}, we have thus obtained:
$$
(L+3) \, \big( \eta_0 +\cdots +\eta_{L+3} \big)^{L+2} \, - \, \dfrac{1}{2} \, \sum_{\substack{F \subset \{ 0,\dots,L+3 \} \\ 1 \le \sharp F \le L+3}} 
\Big( \eta(F) \Big)^{\sharp F-1} \, \Big( \eta(F^c) \Big)^{\sharp F^c -1} \, = \, 0 \, ,
$$
which completes the induction argument for proving that the polynomial $Q_\sharp$ in \eqref{propB1exprQ} is zero.
\end{proof}

\noindent Corollary \ref{corB2} and Proposition \ref{propB1} immediately imply the result of Lemma \ref{lem_symmetry_2}, namely 
the second symmetry formula which we have used to complete the proof of Lemma \ref{lem_compatibilite_div}.

\begin{corollary}[The second symmetry formula]
\label{corB3}
The functions $\chi^{[\ell]}$, $\dot{\chi}^{[\ell]}$ and $\psi^m$ satisfy the relation:
\begin{equation}
\label{formulesymetrie2}
\forall \, \ell \ge 0 \, ,\quad 
\sum_{\ell_1+\ell_2 =\ell} \p_{y_3} \chi^{[\ell_1]} \, \p_\theta \psi^{\ell_2} -\p_\theta \big( \dot{\chi}^{[\ell]} \, \psi^{\ell_2} \big) \, = \, 
\sum_{\ell_1+\cdots+\ell_4=\ell} \big( \p_{y_3} \chi^{[\ell_1]} \, \dot{\chi}^{[\ell_2]} -\chi^{[\ell_1]} \, \p_{y_3} \dot{\chi}^{[\ell_2]} \big) 
\, \psi^{\ell_3} \, \p_\theta \psi^{\ell_4} \, ,
\end{equation}
and similar relations with any other tangential partial derivative chosen among $\{ \p_t,\p_{y_1},\p_{y_2} \}$. (We keep using the convention 
$\psi^0 =0$.)
\end{corollary}

\begin{proof}[Proof of Corollary \ref{corB3}]
Using Corollary \ref{corB2}, we first get:
$$
\forall \, \ell \ge 0 \, ,\quad 
\sum_{\ell_1+\ell_2=\ell} \p_{y_3} \chi^{[\ell_1]} \, \p_\theta \psi^{\ell_2} \, - \, \dot{\chi}^{[\ell_1]} \, \p_\theta \psi^{\ell_2} 
\, - \, \sum_{\ell_1 +\cdots +\ell_4=\ell} \p_{y_3} \chi^{[\ell_1]} \, \dot{\chi}^{[\ell_2]} \, \psi^{\ell_3} \, \p_\theta \psi^{\ell_4} \, = \, 0 \, ,
$$
and using Proposition \ref{propB1}, we also get:
$$
\forall \, \ell \ge 0 \, ,\quad \sum_{\ell_1+\ell_2=\ell} \p_\theta \dot{\chi}^{[\ell_1]} \, \psi^{\ell_2} 
\, - \, \sum_{\ell_1 +\cdots +\ell_4=\ell} \chi^{[\ell_1]} \, \p_{y_3} \dot{\chi}^{[\ell_2]} \, \psi^{\ell_3} \, \p_\theta \psi^{\ell_4} \, = \, 0 \, .
$$
Subtracting the latter two relations, we end up showing \eqref{formulesymetrie2}, or equivalently that the function $\bX_\theta^\ell$ 
in \eqref{defbXell} is zero. The same argument applies when the tangential derivative $\p_\theta$ is replaced by any other tangential 
derivative, namely $\p_{y_1}$, $\p_{y_2}$ or $\p_t$.
\end{proof}

Before turning to the verification of even more compatibility conditions on the source terms for the WKB cascade, let us observe 
that the explicit expressions \eqref{expressioncompacte}, \eqref{expressioncompacte'} allow us to rewrite the inductive formula 
\eqref{corB2-eq4}, which we have shown to hold not only for $L \ge 2$ but for any $L \ge 0$. Combining the proofs of Corollary 
\ref{corB2} and Proposition \ref{propB1}, we have indeed obtained the formula:
$$
\forall \, L \ge 1 \, ,\quad \dfrac{1}{L+1} \, \big\{ \chi^{L+1} \big\}^{(L+1)} -\chi' \, \big\{ \chi^L \big\}^{(L)} 
\, = \, \sum_{\ell=0}^{L-1} \binom{L}{\ell} \, \big\{ \chi^\ell \big\}^{(\ell)} \, \big\{ \chi'' \, \chi^{L-\ell} \big\}^{(L-1-\ell)} \, ,
$$
or, equivalently:
\begin{equation}
\label{propB1-id''}
\forall \, L \ge 1 \, ,\quad \boxed{\big\{ \chi' \, \chi^L \big\}^{(L)} -\chi' \, \big\{ \chi^L \big\}^{(L)} \, = \, 
\sum_{\ell=0}^{L-1} \binom{L}{\ell} \, \big\{ \chi^\ell \big\}^{(\ell)} \, \big\{ \chi'' \, \chi^{L-\ell} \big\}^{(L-1-\ell)}} \, .
\end{equation}
We do not know whether the relations \eqref{propB1-id'} and \eqref{propB1-id''} -which definitely have a `Leibniz flavor'- have been 
noticed/proved before or whether they could be useful in other contexts.

\section{Compatibility of the source terms for the slow mean problem}

\subsection{Compatibility for the divergence of the magnetic field}

Let us first recall the result we aim at proving here:

\begin{lemma}[Compatibility for the divergence of the magnetic field in the slow mean problem]
\label{lemB4}
The source terms in \eqref{slow_mean_m+1_edp} satisfy:
\begin{equation}
\label{lemB3-eq1}
\p_t \, \bF_8^{\, m,\pm} \, = \, \p_{y_\alpha} \bF_{3+\alpha}^{\, m,\pm} \, .
\end{equation}
\end{lemma}

\begin{proof}[Proof of Lemma \ref{lemB4}]
The verification of the compatibility condition \eqref{lemB3-eq1} follows from the explicit expressions of the source terms 
$\bF_{3+\alpha}^{\, m,\pm}$ and $\bF_8^{\, m,\pm}$ and from the symmetry formulas (Corollary \ref{corB1}, Corollary \ref{corB2} and 
Proposition \ref{propB1} above). Let us first recall the expressions of the source terms. For simplicity, we omit from now on the 
superscripts $\pm$. We keep the notation  ${\bf c}_0$ for the mean with respect to $\theta$ on $\bT$. The source terms are 
given as follows:
\begin{align*}
\bF_8^m \, = \, {\bf c}_0 \, \Big\{ \, 
{\color{blue} \sum_{\ell_1+\ell_2+\ell_3=m+2} \chi^{[\ell_1]} \, \p_\theta \psi^{\ell_2} \, \xi_j \, \p_{y_3} \uH_j^{\ell_3}} 
\, & \, +\sum_{\ell_1+\ell_2+\ell_3=m+1} \chi^{[\ell_1]} \, \p_{y_j} \psi^{\ell_2} \, \p_{y_3} \uH_j^{\ell_3} \\
& \, +\sum_{\ell_1+\ell_2+\ell_3=m+1} \dot{\chi}^{[\ell_1]} \, \psi^{\ell_2} \, \p_{y_3} \uH_3^{\ell_3} \Big\} \, ,
\end{align*}
\begin{align*}
\bF_{3+j}^m \, = \, {\bf c}_0 \, \Big\{ \, & \, {\color{blue} \sum_{\ell_1+\ell_2+\ell_3=m+2} \chi^{[\ell_1]} \, \p_\theta \psi^{\ell_2} \, 
\big( c \, \p_{y_3} \uH_j^{\ell_3} -b \, \p_{y_3} \uu_j^{\ell_3} -u_j^0 \, \xi_{j'} \, \p_{y_3} \uH_{j'}^{\ell_3} +H_j^0 \, \xi_{j'} \, \p_{y_3} \uu_{j'}^{\ell_3} \big)} \\
& \, +\sum_{\ell_1+\ell_2+\ell_3=m+1} \chi^{[\ell_1]} \, \p_{y_{j'}} \psi^{\ell_2} \, \big( u_{j'}^0 \, \p_{y_3} \uH_j^{\ell_3} -H_{j'}^0 \, \p_{y_3} \uu_j^{\ell_3} 
-u_j^0 \, \p_{y_3} \uH_{j'}^{\ell_3} +H_j^0 \, \p_{y_3} \uu_{j'}^{\ell_3} \big) \\
& \, + \sum_{\ell_1+\ell_2+\ell_3=m+1} \chi^{[\ell_1]} \, \p_t \psi^{\ell_2} \, \p_{y_3} \uH_j^{\ell_3} \\
& \, +\sum_{\ell_1+\ell_2+\ell_3=m+1} \dot{\chi}^{[\ell_1]} \, \psi^{\ell_2} \, \big( H_j^0 \, \p_{y_3} \uu_3^{\ell_3} -u_j^0 \, \p_{y_3} \uH_3^{\ell_3} \big) 
+{\color{red} \sum_{\substack{\ell_1+\ell_2=m+1 \\ \ell_1,\ell_2 \ge 1}} 
\nabla \cdot \big( \uu_j^{\ell_1} \, \uH^{\ell_2} -\uH_j^{\ell_1} \, \uu^{\ell_2} \big)} \\
& \, +{\color{blue} \sum_{\substack{\ell_1+\cdots+\ell_4=m+2 \\ \ell_3,\ell_4 \ge 1}} \chi^{[\ell_1]} \, \p_\theta \psi^{\ell_2} 
\, \p_{y_3} \big( \xi_{j'} \, \uu_{j'}^{\ell_3} \, \uH_j^{\ell_4} -\xi_{j'} \, \uH_{j'}^{\ell_3} \, \uu_j^{\ell_4} \big)} \\
& \, +\sum_{\substack{\ell_1+\cdots+\ell_4=m+1 \\ \ell_3,\ell_4 \ge 1}} \chi^{[\ell_1]} \, \p_{y_{j'}} \psi^{\ell_2} 
\, \p_{y_3} \big( \uu_{j'}^{\ell_3} \, \uH_j^{\ell_4} -\uH_{j'}^{\ell_3} \, \uu_j^{\ell_4} \big) 
+\dot{\chi}^{[\ell_1]} \, \psi^{\ell_2} \, \p_{y_3} \big( \uH_j^{\ell_3} \, \uu_3^{\ell_4} -\uu_j^{\ell_3} \, \uH_3^{\ell_4} \big) \Big\} \, ,
\end{align*}
\begin{align*}
\bF_6^m \, = \, {\bf c}_0 \, \Big\{ \, & \, \sum_{\ell_1+\ell_2+\ell_3=m+2}  \chi^{[\ell_1]} \, \p_\theta \psi^{\ell_2} \, 
\big( c \, \p_{y_3} \uH_3^{\ell_3} -b \, \p_{y_3} \uu_3^{\ell_3} \big) \\
& \, +\sum_{\ell_1+\ell_2+\ell_3=m+1} \chi^{[\ell_1]} \, \Big( \big( \p_t +u_j^0 \, \p_{y_j} \big) \psi^{\ell_2} \, \p_{y_3} \uH_3^{\ell_3} 
-H_j^0 \, \p_{y_j} \psi^{\ell_2} \, \p_{y_3} \uu_3^{\ell_3} \Big) \\
& \, +{\color{red} \sum_{\substack{\ell_1+\ell_2=m+1 \\ \ell_1,\ell_2 \ge 1}} \p_{y_j} \big( \uH_j^{\ell_1} \, \uu_3^{\ell_2} 
-\uu_j^{\ell_1} \, \uH_3^{\ell_2} \big)} +\sum_{\substack{\ell_1+\cdots+\ell_4=m+2 \\ \ell_3,\ell_4 \ge 1}} \chi^{[\ell_1]} \, \p_\theta \psi^{\ell_2} 
\, \p_{y_3} \big( \xi_j \, \uu_j^{\ell_3} \, \uH_3^{\ell_4} -\xi_j \, \uH_j^{\ell_3} \, \uu_3^{\ell_4} \big) \\
& \, +\sum_{\substack{\ell_1+\cdots+\ell_4=m+1 \\ \ell_3,\ell_4 \ge 1}} \chi^{[\ell_1]} \, \p_{y_j} \psi^{\ell_2} \, 
\p_{y_3} \big( \uu_j^{\ell_3} \, \uH_3^{\ell_4} -\uH_j^{\ell_3} \, \uu_3^{\ell_4} \big) \Big\} \, .
\end{align*}
Several terms have been highlighted in blue or red in order to explain the calculations below.

We now define the quantity:
$$
\cD^m \, := \,\p_{y_\alpha} \, \bF_{3+\alpha}^m -\p_t \, \bF_8^m \, ,
$$
and the goal of course is to show that $\cD^m$ vanishes under the induction assumption $H(m)$, which will prove \eqref{lemB3-eq1}. 
The calculation splits again in several steps.
\bigskip

$\bullet$ \underline{Step 1}. Using the first symmetry formula and integration by parts. When applying a slow tangential derivative $\p_t$, 
$\p_{y_1}$ or $\p_{y_2}$, we can use the first symmetry formula (Corollary \ref{corB1} above) and obtain (here $\p_{\rm tan}$ corresponds 
to either $\p_t$, $\p_{y_1}$ or $\p_{y_2}$):
$$
\p_{\rm tan} \, {\bf c}_0 \, \Big\{ \sum_{\ell_1+\ell_2+\ell_3=L} \chi^{[\ell_1]} \, \p_\theta \psi^{\ell_2} \, f^{\ell_3} \Big\} \, = \, 
{\bf c}_0 \, \Big\{ \sum_{\ell_1+\ell_2+\ell_3=L} \chi^{[\ell_1]} \, \p_\theta \psi^{\ell_2} \, \p_{\rm tan} f^{\ell_3} 
-\sum_{\ell_1+\ell_2+\ell_3=L} \chi^{[\ell_1]} \, \p_{\rm tan} \psi^{\ell_2} \, \p_\theta f^{\ell_3}\Big\} \, ,
$$
where integration by parts in $\theta$ corresponds to the calculus rule:
$$
{\bf c}_0 \, \big\{ (\p_\theta f) \, g \big\} \, = \, -{\bf c}_0 \, \big\{ f \, (\p_\theta g) \big\} \, .
$$
The above rule for tangential differentiation is applied to deal with the terms highlighted in blue in the above decompositions of 
$\bF_{3+j}^m$ and $\bF_8^m$. This operation makes fast derivatives in $\theta$ appear, and we then use the fast problems 
solved at the previous steps of the induction:
\begin{align*}
\forall \, \ell \, = \, 1,\dots,m \, ,& \quad \xi_j \, \p_\theta \uH_j^\ell \, = \, \uF_8^{\ell-1} \, ,\\
& \quad c \, \p_\theta \uH_j^\ell -b\, \p_\theta \uu_j^\ell -u_j^0 \, \p_\theta \xi_{j'} \, \uH_{j'}^\ell +H_j^0 \, \p_\theta \xi_{j'} \, \uu_{j'}^\ell 
\, = \, \uF_{3+j}^{\ell-1} \, .
\end{align*}
The only subtlety here is that the mean $\widehat{\psi}^{\, m+1}(0)$ has not been defined yet. However, this is of no consequence 
since the term involving $\psi^{\, m+1}$ in the integration by parts argument is:
$$
{\bf c}_0 \, \Big\{ \chi^{[0]} \, \p_\theta \psi^{\, m+1} \, \xi_j \, \p_{y_3} \uH_j^1 \Big\} \, = \, 
\chi(y_3) \, {\bf c}_0 \, \Big\{ \p_\theta \psi^{\, m+1}_\sharp \, \xi_j \, \p_{y_3} \uH_j^1 +\psi^{\, m+1}_\sharp \, \xi_j \, \p_{y_3} \p_\theta \uH_j^1 \Big\} 
\, = \, 0 \, .
$$
Here we have used the very first (homogeneous) fast problem $\xi_j \, \p_\theta \uH_j^1=0$. Hence the mean of $\psi^{\, m+1}$ is 
not relevant here.

Let us also observe that the terms highlighted in red in $\bF_{3+\alpha}^m$ correspond to a divergence free vector field, 
hence they disappear when computing $\cD^m$. After collecting the terms, we get the following decomposition for $\cD^m$:
$$
\cD^m \, = \, \cD_1^m +\cD_2^m +\cD_3^m \, ,
$$
with:
\begin{equation*}
\cD_1^m \, := \, {\bf c}_0 \, \Big\{ \, \sum_{\ell_1+\ell_2+\ell_3=m+1} \chi^{[\ell_1]} \, \p_t \psi^{\ell_2} \, \p_{y_3} \big( \uF_8^{\ell_3} +\nabla \cdot 
\uH^{\ell_3} \big) +\big( \p_{y_3} \chi^{[\ell_1]} \, \p_t \psi^{\ell_2} -\p_t (\dot{\chi}^{[\ell_1]} \, \psi^{\ell_2}) \big) \, \p_{y_3} \uH_3^{\ell_3} \Big\} \, ,
\end{equation*}
\begin{align*}
\cD_2^m \, := \, {\bf c}_0 \, \Big\{ & \, -\sum_{\ell_1+\ell_2+\ell_3=m+1}  \chi^{[\ell_1]} \, \p_{y_j} \psi^{\ell_2} \, \p_{y_3} \uF_{3+j}^{\ell_3} \\
& \, +\sum_{\ell_1+\ell_2+\ell_3=m+1}  \chi^{[\ell_1]} \, \p_{y_j} \psi^{\ell_2} \, \p_{y_3} 
\big( u_j^0 \, \nabla \cdot \uH^{\ell_3} -H_j^0 \, \nabla \cdot \uu^{\ell_3} -(\p_t +u_{j'}^0 \,\p_{y_{j'}}) \uH_j^{\ell_3} +H_{j'}^0 \,\p_{y_{j'}} \uu_j^{\ell_3} \big) \\
& \, +{\color{orange} \sum_{\ell_1+\ell_2+\ell_3=m+1} \big( \p_{y_3} \chi^{[\ell_1]} \, \p_{y_j} \psi^{\ell_2} -\p_{y_j} (\dot{\chi}^{[\ell_1]} \, \psi^{\ell_2}) \big) \, 
\big( u_j^0 \, \p_{y_3} \uH_3^{\ell_3} -H_j^0 \, \p_{y_3} \uu_3^{\ell_3} \big)} \\
& \, +\sum_{\substack{\ell_1+\cdots+\ell_4=m+2 \\ \ell_3,\ell_4 \ge 1}} \chi^{[\ell_1]} \, \p_{y_j} \psi^{\ell_2} 
\, \p_{y_3} \p_\theta \big( \xi_{j'} \, \uH_{j'}^{\ell_3} \, \uu_j^{\ell_4} -\xi_{j'} \, \uu_{j'}^{\ell_3} \, \uH_j^{\ell_4} \big) \\
& \, +\sum_{\substack{\ell_1+\cdots+\ell_4=m+1 \\ \ell_3,\ell_4 \ge 1}} \chi^{[\ell_1]} \, \p_{y_j} \psi^{\ell_2} 
\, \p_{y_3} \Big( \nabla \cdot \big( \uu_j^{\ell_3} \, \uH^{\ell_4} -\uH_j^{\ell_3} \, \uu^{\ell_4} \big) \Big) \\
& \, +{\color{orange} \sum_{\substack{\ell_1+\cdots+\ell_4=m+1 \\ \ell_3,\ell_4 \ge 1}} \big( \p_{y_3} \chi^{[\ell_1]} \, \p_{y_j} \psi^{\ell_2} 
-\p_{y_j} (\dot{\chi}^{[\ell_1]} \, \psi^{\ell_2}) \big) \, \p_{y_3} \big( \uu_j^{\ell_3} \, \uH_3^{\ell_4} -\uH_j^{\ell_3} \, \uu_3^{\ell_4} \big)} \Big\} \, ,
\end{align*}
\begin{align*}
\cD_3^m \, := \, {\bf c}_0 \, \Big\{ \, & \, {\color{magenta} \sum_{\ell_1+\ell_2+\ell_3=m+2} \chi^{[\ell_1]} \, \p_\theta \psi^{\ell_2} \, \p_{y_3} 
\big( c \, \nabla \cdot \uH^{\ell_3} -b \, \nabla \cdot \uu^{\ell_3} -(\p_t +u_j^0 \, \p_{y_j}) \, \xi_{j'} \, \uH_{j'}^{\ell_3} 
+H_j^0 \, \p_{y_j} \, \xi_{j'} \, \uu_{j'}^{\ell_3} \big)} \\
& \, +{\color{magenta} \sum_{\substack{\ell_1+\cdots+\ell_4=m+2 \\ \ell_3,\ell_4 \ge 1}} \chi^{[\ell_1]} \, \p_\theta \psi^{\ell_2} \, \p_{y_3} \nabla \cdot 
\big( \xi_j \, \uu_j^{\ell_3} \, \uH^{\ell_4} -\xi_j \, \uH_j^{\ell_3} \, \uu^{\ell_4} \big)} \\
& \, +\sum_{\ell_1+\ell_2+\ell_3=m+2} \p_{y_3} \chi^{[\ell_1]} \, \p_\theta \psi^{\ell_2} \, \p_{y_3} \big( c \, \uH_3^{\ell_3} -b \, \uu_3^{\ell_3} \big) \\
& \, +\sum_{\ell_1+\ell_2+\ell_3=m+1} \dot{\chi}^{[\ell_1]} \, \psi^{\ell_2} \, \p_{y_3} \big( H_j^0 \, \p_{y_j} \uu_3^{\ell_3} -(\p_t +u_j^0 \, \p_{y_j}) \uH_3^{\ell_3} \big) \\
& \, +\sum_{\substack{\ell_1+\cdots+\ell_4=m+2 \\ \ell_3,\ell_4 \ge 1}} \p_{y_3} \chi^{[\ell_1]} \, \p_\theta \psi^{\ell_2} 
\, \p_{y_3} \big( \xi_j \, \uu_j^{\ell_3} \, \uH_3^{\ell_4} -\xi_j \, \uH_j^{\ell_3} \, \uu_3^{\ell_4} \big) \\
& \, +\sum_{\substack{\ell_1+\cdots+\ell_4=m+1 \\ \ell_3,\ell_4 \ge 1}} \dot{\chi}^{[\ell_1]} \, \psi^{\ell_2} \, \p_{y_3} \p_{y_j} 
\big( \uH_j^{\ell_3} \, \uu_3^{\ell_4} -\uu_j^{\ell_3} \, \uH_3^{\ell_4} \big) \Big\} \, .
\end{align*}
\bigskip

$\bullet$ \underline{Step 2}. Using the second symmetry formula to simplify $\cD_1^m$. We recall the expression of $\uF_8^\ell$:
\begin{align*}
\forall \, \ell \, = \, 1,\dots,m \, ,\quad \uF_8^\ell \, = \, & \, -\nabla \cdot \uH^\ell 
+\sum_{\ell_1+\ell_2+\ell_3=\ell+1} \chi^{[\ell_1]} \, \p_\theta \psi^{\ell_2} \, \xi_j \, \p_{y_3} \uH_j^{\ell_3} \\
& \, +\sum_{\ell_1+\ell_2+\ell_3=\ell} \chi^{[\ell_1]} \, \p_{y_j} \psi^{\ell_2} \, \p_{y_3} \uH_j^{\ell_3} 
+\sum_{\ell_1+\ell_2+\ell_3=\ell} \dot{\chi}^{[\ell_1]} \, \psi^{\ell_2} \, \p_{y_3} \uH_3^{\ell_3} \, ,
\end{align*}
which we use in the above expression for $\cD_1^m$ together with the second symmetry formula:
$$
\sum_{\ell_1+\ell_2=\ell} \p_{y_3} \chi^{[\ell_1]} \, \p_t \psi^{\ell_2} -\p_t \big( \dot{\chi}^{[\ell_1]} \, \psi^{\ell_2} \big) \, = \, 
\sum_{\ell_1+\cdots+\ell_4=\ell} \big( \dot{\chi}^{[\ell_1]} \, \p_{y_3} \chi^{[\ell_2]} -\chi^{[\ell_1]} \, \p_{y_3} \dot{\chi}^{[\ell_2]} \big) 
\, \psi^{\ell_3} \, \p_t \psi^{\ell_4} \, .
$$
This simplifies the expression of $\cD_1^m$ and we get eventually:
\begin{align}
\cD_1^m \, = \, {\bf c}_0 \, \Big\{ \, & \, {\color{magenta} \sum_{\ell_1+\cdots+\ell_5=m+2} \chi^{[\ell_1]} \, \p_\theta \psi^{\ell_2} \, \p_{y_3} 
\Big(  \chi^{[\ell_3]} \, \p_t \psi^{\ell_4} \, \xi_j \, \p_{y_3} \uH_j^{\ell_5} \Big)} \notag \\
& \, +{\color{ForestGreen} \sum_{\ell_1+\cdots+\ell_5=m+1} \chi^{[\ell_1]} \, \p_{y_j} \psi^{\ell_2} \, \p_{y_3} 
\Big(  \chi^{[\ell_3]} \, \p_t \psi^{\ell_4} \, \p_{y_3} \uH_j^{\ell_5} \Big)} \label{lemB3-eq2} \\
& \, +\sum_{\ell_1+\cdots+\ell_5=m+1} \dot{\chi}^{[\ell_1]} \, \psi^{\ell_2} \, \p_{y_3} 
\Big( \chi^{[\ell_3]} \, \p_t \psi^{\ell_4} \, \p_{y_3} \uH_3^{\ell_3} \Big) \Big\} \, .\notag
\end{align}
\bigskip

$\bullet$ \underline{Step 3}. Using the second symmetry formula to simplify $\cD_2^m$. We recall the expression of $\uF_{3+j}^\ell$:
\begin{align*}
\forall \, \ell \, = \, 1,\dots,m \, ,\, -\uF_{3+j}^\ell \, = & \, {\color{blue} \big( \p_t +u_{j'}^0 \,\p_{y_{j'}} \big) \uH_j^\ell -H_{j'}^0 \,\p_{y_{j'}} \uu_j^\ell 
-u_j^0 \, \nabla \cdot \uH^\ell +H_j^0 \, \nabla \cdot \uu^\ell} \\
& \, +\sum_{\ell_1+\ell_2+\ell_3=\ell+1} \chi^{[\ell_1]} \, \p_\theta \psi^{\ell_2} \, \big( b \, \p_{y_3} \uu_j^{\ell_3} -c \, \p_{y_3} \uH_j^{\ell_3} 
-H_j^0 \, \xi_{j'} \, \p_{y_3} \uu_{j'}^{\ell_3} +u_j^0 \, \xi_{j'} \, \p_{y_3} \uH_{j'}^{\ell_3} \big) \\
& \, +{\color{red} \sum_{\ell_1+\ell_2+\ell_3=\ell} \chi^{[\ell_1]} \,\p_{y_{j'}} \psi^{\ell_2} \, 
\big( H_{j'}^0 \, \p_{y_3} \uu_j^{\ell_3} -u_{j'}^0 \, \p_{y_3} \uH_j^{\ell_3} -H_j^0 \, \p_{y_3} \uu_{j'}^{\ell_3} +u_j^0 \, \p_{y_3} \uH_{j'}^{\ell_3} \big)} \\
& \, -\sum_{\ell_1+\ell_2+\ell_3=\ell} \chi^{[\ell_1]} \,\p_t \psi^{\ell_2} \, \p_{y_3} \uH_j^{\ell_3} 
+\sum_{\ell_1+\ell_2+\ell_3=\ell} \dot{\chi}^{[\ell_1]} \, \psi^{\ell_2} \, \big( u_j^0 \, \p_{y_3} \uH_3^{\ell_3}-H_j^0 \, \p_{y_3} \uu_3^{\ell_3} \big) \\
& \, +{\color{blue} \sum_{\substack{\ell_1+\ell_2=\ell+1 \\ \ell_1,\ell_2 \ge 1}} \p_\theta \big( \xi_{j'} \, \uu_{j'}^{\ell_1} \, \uH_j^{\ell_2} 
-\xi_{j'} \, \uH_{j'}^{\ell_1} \, \uu_j^{\ell_2} \big)} +{\color{blue} \sum_{\substack{\ell_1+\ell_2=\ell \\ \ell_1,\ell_2 \ge 1}} 
\nabla \cdot \big( \uH_j^{\ell_1} \, \uu^{\ell_2} -\uu_j^{\ell_1} \, \uH^{\ell_2} \big)} \\
& \, +\sum_{\substack{\ell_1+\cdots+\ell_4=\ell+1 \\ \ell_3,\ell_4 \ge 1}} \chi^{[\ell_1]} \, \p_\theta \psi^{\ell_2} \, \p_{y_3} 
\big( \xi_{j'} \, \uH_{j'}^{\ell_3} \, \uu_j^{\ell_4} -\xi_{j'} \, \uu_{j'}^{\ell_3} \, \uH_j^{\ell_4} \big) \\
& \, +{\color{red} \sum_{\substack{\ell_1+\cdots+\ell_4=\ell \\ \ell_3,\ell_4 \ge 1}} \chi^{[\ell_1]} \,\p_{y_{j'}} \psi^{\ell_2} \, \p_{y_3} 
\big( \uH_{j'}^{\ell_3} \, \uu_j^{\ell_4} -\uu_{j'}^{\ell_3} \, \uH_j^{\ell_4} \big)} \\
& \, +\sum_{\substack{\ell_1+\cdots+\ell_4=\ell \\ \ell_3,\ell_4 \ge 1}} \dot{\chi}^{[\ell_1]} \, \psi^{\ell_2} \, 
\p_{y_3} \big( \uu_j^{\ell_3} \, \uH_3^{\ell_4} -\uH_j^{\ell_3} \, \uu_3^{\ell_4} \big) \, ,
\end{align*}
which we use in the above expression for $\cD_2^m$ (the blue terms in $\uF_{3+j}^\ell$ cancel with the black terms in $\cD_2^m$ and 
the red terms in $\uF_{3+j}^\ell$ cancel when computing $\p_{y_j} \psi^{\ell_2} \, \uF_{3+j}^{\ell_3}$), together with the second symmetry 
formula (for the terms highlighted in orange in that expression). After some computations, we end up with the final expression:
\begin{align}
\cD_2^m \, = \, {\bf c}_0 \, \Big\{ \, & \, {\color{magenta} \sum_{\ell_1+\cdots+\ell_5=m+2} \chi^{[\ell_1]} \, \p_\theta \psi^{\ell_2} \, \p_{y_3} 
\Big( \chi^{[\ell_3]} \, \p_{y_j} \psi^{\ell_4} \, \p_{y_3} \Big( b \, \uu_j^{\ell_5} -c \, \uH_j^{\ell_5} -H_j^0 \, \xi_{j'} \, \uu_{j'}^{\ell_5} 
+u_j^0 \, \xi_{j'} \, \uH_{j'}^{\ell_5} \Big) \Big)} \notag \\
& \, +{\color{magenta} \sum_{\substack{\ell_1+\cdots+\ell_6=m+2 \\ \ell_5,\ell_6 \ge 1}} \chi^{[\ell_1]} \, \p_\theta \psi^{\ell_2} \, \p_{y_3} 
\Big( \chi^{[\ell_3]} \, \p_{y_j} \psi^{\ell_4} \, \p_{y_3} \Big( \xi_{j'} \, \uH_{j'}^{\ell_5} \, \uu_j^{\ell_6} -\xi_{j'} \, \uu_{j'}^{\ell_5} \, \uH_j^{\ell_6} 
\Big) \Big)} \notag \\
& \, -{\color{ForestGreen} \sum_{\ell_1+\cdots+\ell_5=m+1} \chi^{[\ell_1]} \, \p_{y_j} \psi^{\ell_2} \, \p_{y_3} 
\Big(  \chi^{[\ell_3]} \, \p_t \psi^{\ell_4} \, \p_{y_3} \uH_j^{\ell_5} \Big)} \label{lemB3-eq3} \\
& \, +\sum_{\ell_1+\cdots+\ell_5=m+1} \dot{\chi}^{[\ell_1]} \, \psi^{\ell_2} \, \p_{y_3} 
\Big( \chi^{[\ell_3]} \, \p_{y_j} \psi^{\ell_4} \, \p_{y_3} \big( u_j^0 \, \uH_3^{\ell_5} -H_j^0 \, \uu_3^{\ell_5} \big) \Big) \notag \\
& \, +\sum_{\substack{\ell_1+\cdots+\ell_6=m+2 \\ \ell_5,\ell_6 \ge 1}} \dot{\chi}^{[\ell_1]} \, \psi^{\ell_2} \, \p_{y_3} 
\Big( \chi^{[\ell_3]} \, \p_{y_j} \psi^{\ell_4} \, \p_{y_3} \big( \uu_j^{\ell_5} \, \uH_3^{\ell_6} -\uH_j^{\ell_5} \, \uu_3^{\ell_6} \big) \Big) \Big\} \, .\notag
\end{align}
\bigskip

$\bullet$ \underline{Step 4}. Recollecting the terms and integrating by parts again. Let us first observe that the green terms in $\cD_1^m$ 
and $\cD_2^m$ cancel. We now recollect the terms in $\cD_1^m,\cD_2^m,\cD_3^m$ in order to benefit from future cancellations. We 
rewrite $\cD^m$ as:
$$
\cD^m \, = \, \cD_4^m +\cD_5^m \, ,
$$
where $\cD_4^m$ gathers all the black terms in $\cD_1^m,\cD_2^m,\cD_3^m$, and $\cD_5^m$ gathers all magenta terms, namely:
\begin{align*}
\cD_4^m \, := \, {\bf c}_0 \, \Big\{ \, & \, {\color{blue} \sum_{\ell_1+\ell_2+\ell_3=m+2} \p_{y_3} \chi^{[\ell_1]} \, \p_\theta \psi^{\ell_2} \, 
\p_{y_3} \big( c \, \uH_3^{\ell_3} -b \, \uu_3^{\ell_3} \big)} \\
& \, +\sum_{\ell_1+\ell_2+\ell_3=m+1} \dot{\chi}^{[\ell_1]} \, \psi^{\ell_2} \, 
\p_{y_3} \big( H_j^0 \, \p_{y_j} \uu_3^{\ell_3} -(\p_t +u_j^0 \, \p_{y_j}) \uH_3^{\ell_3} \big) \\
& \, +\sum_{\ell_1+\cdots+\ell_5=m+1} \dot{\chi}^{[\ell_1]} \, \psi^{\ell_2} \, \p_{y_3} 
\Big( \chi^{[\ell_3]} \, \p_{y_j} \psi^{\ell_4} \, \p_{y_3} \big( u_j^0 \, \uH_3^{\ell_5} -H_j^0 \, \uu_3^{\ell_5} \big) \Big) \\
& \, +\sum_{\ell_1+\cdots+\ell_5=m+1} \dot{\chi}^{[\ell_1]} \, \psi^{\ell_2} \, \p_{y_3} 
\Big( \chi^{[\ell_3]} \, \p_t \psi^{\ell_4} \, \p_{y_3} \uH_3^{\ell_3} \Big) \\
& \, +\sum_{\substack{\ell_1+\cdots+\ell_4=m+1 \\ \ell_3,\ell_4 \ge 1}} \dot{\chi}^{[\ell_1]} \, \psi^{\ell_2} \, \p_{y_3} \p_{y_j} 
\big( \uH_j^{\ell_3} \, \uu_3^{\ell_4} -\uu_j^{\ell_3} \, \uH_3^{\ell_4} \big) \\
& \, +\sum_{\substack{\ell_1+\cdots+\ell_6=m+2 \\ \ell_5,\ell_6 \ge 1}} \dot{\chi}^{[\ell_1]} \, \psi^{\ell_2} \, \p_{y_3} 
\Big( \chi^{[\ell_3]} \, \p_{y_j} \psi^{\ell_4} \, \p_{y_3} \big( \uu_j^{\ell_5} \, \uH_3^{\ell_6} -\uH_j^{\ell_5} \, \uu_3^{\ell_6} \big) \Big) \\
& \, +{\color{blue} \sum_{\substack{\ell_1+\cdots+\ell_4=m+2 \\ \ell_3,\ell_4 \ge 1}} \p_{y_3} \chi^{[\ell_1]} \, \p_\theta \psi^{\ell_2} 
\, \p_{y_3} \big( \xi_j \, \uu_j^{\ell_3} \, \uH_3^{\ell_4} -\xi_j \, \uH_j^{\ell_3} \, \uu_3^{\ell_4} \big)} \Big\} \, ,
\end{align*}
\begin{align*}
\cD_5^m \, := \, {\bf c}_0 \, \Big\{ \, & \, \sum_{\ell_1+\ell_2+\ell_3=m+2} \chi^{[\ell_1]} \, \p_\theta \psi^{\ell_2} \, \p_{y_3} 
\big( c \, \nabla \cdot \uH^{\ell_3} -b \, \nabla \cdot \uu^{\ell_3} -(\p_t +u_j^0 \, \p_{y_j}) \, \xi_{j'} \, \uH_{j'}^{\ell_3} 
+H_j^0 \, \p_{y_j} \, \xi_{j'} \, \uu_{j'}^{\ell_3} \big) \\
& \, +\sum_{\ell_1+\cdots+\ell_5=m+2} \chi^{[\ell_1]} \, \p_\theta \psi^{\ell_2} \, \p_{y_3} 
\Big( \chi^{[\ell_3]} \, \p_{y_j} \psi^{\ell_4} \, \p_{y_3} \Big( b \, \uu_j^{\ell_5} -c \, \uH_j^{\ell_5} -H_j^0 \, \xi_{j'} \, \uu_{j'}^{\ell_5} 
+u_j^0 \, \xi_{j'} \, \uH_{j'}^{\ell_5} \Big) \Big) \\
& \, +\sum_{\ell_1+\cdots+\ell_5=m+2} \chi^{[\ell_1]} \, \p_\theta \psi^{\ell_2} \, \p_{y_3} 
\Big(  \chi^{[\ell_3]} \, \p_t \psi^{\ell_4} \, \xi_j \, \p_{y_3} \uH_j^{\ell_5} \Big) \\
& \, +\sum_{\substack{\ell_1+\cdots+\ell_4=m+2 \\ \ell_3,\ell_4 \ge 1}} \chi^{[\ell_1]} \, \p_\theta \psi^{\ell_2} \, \p_{y_3} \Big( 
\nabla \cdot \big( \xi_j \, \uu_j^{\ell_3} \, \uH^{\ell_4} -\xi_j \, \uH_j^{\ell_3} \, \uu^{\ell_4} \big) \Big) \\
& \, +\sum_{\substack{\ell_1+\cdots+\ell_6=m+2 \\ \ell_5,\ell_6 \ge 1}} \chi^{[\ell_1]} \, \p_\theta \psi^{\ell_2} \, \p_{y_3} 
\Big( \chi^{[\ell_3]} \, \p_{y_j} \psi^{\ell_4} \, \p_{y_3} \Big( \xi_{j'} \, \uH_{j'}^{\ell_5} \, \uu_j^{\ell_6} -\xi_{j'} \, \uu_{j'}^{\ell_5} \, \uH_j^{\ell_6} 
\Big) \Big) \, .
\end{align*}

We need to further simplify $\cD_4^m$ and for this, we use the second symmetry formula on the blue terms in the above expression 
of $\cD_4^m$:
$$
\sum_{\ell_1+\ell_2=\ell} \p_{y_3} \chi^{[\ell_1]} \, \p_\theta \psi^{\ell_2} \, = \, 
\sum_{\ell_1+\ell_2=\ell} \p_\theta \big( \dot{\chi}^{[\ell_1]} \, \psi^{\ell_2} \big) 
+\sum_{\ell_1+\cdots+\ell_4=\ell} \big( \dot{\chi}^{[\ell_1]} \, \p_{y_3} \chi^{[\ell_2]} -\chi^{[\ell_1]} \, \p_{y_3} \dot{\chi}^{[\ell_2]} \big) 
\, \psi^{\ell_3} \, \p_\theta \psi^{\ell_4} \, .
$$
For the terms arising with the factor $\p_\theta \big( \dot{\chi}^{[\ell_1]} \, \psi^{\ell_2} \big)$, we integrate by parts with respect to 
$\theta$. Using the fast equation $c \, \p_\theta \uH_3^\ell -b \, \p_\theta \uu_3^\ell =-\uF_6^{\ell-1}$, $\ell=1,\dots,m$, we get:
\begin{align*}
\cD_4^m \, = \, {\bf c}_0 \, \Big\{ \, & \, \sum_{\ell_1+\ell_2+\ell_3=m+1} \dot{\chi}^{[\ell_1]} \, \psi^{\ell_2} \, 
\p_{y_3} \big( -\uF_6^{\ell_3} +H_j^0 \, \p_{y_j} \uu_3^{\ell_3} -(\p_t +u_j^0 \, \p_{y_j}) \uH_3^{\ell_3} \big) \\
& \, +\sum_{\ell_1+\cdots+\ell_5=m+1} \dot{\chi}^{[\ell_1]} \, \psi^{\ell_2} \, \p_{y_3} \Big( \chi^{[\ell_3]} \, \Big( 
(\p_t +u_j^0 \, \p_{y_j}) \psi^{\ell_4} \, \p_{y_3} \uH_3^{\ell_5} -H_j^0 \, \p_{y_j} \psi^{\ell_4} \, \p_{y_3} \uu_3^{\ell_5} \Big) \Big) \\
& \, +\sum_{\substack{\ell_1+\cdots+\ell_4=m+2 \\ \ell_3,\ell_4 \ge 1}} \dot{\chi}^{[\ell_1]} \, \psi^{\ell_2} \, \p_{y_3} \p_\theta 
\big( \xi_j \, \uH_j^{\ell_3} \, \uu_3^{\ell_4} -\xi_j \, \uu_j^{\ell_3} \, \uH_3^{\ell_4} \big) \\
& \, +\sum_{\substack{\ell_1+\cdots+\ell_4=m+1 \\ \ell_3,\ell_4 \ge 1}} \dot{\chi}^{[\ell_1]} \, \psi^{\ell_2} \, \p_{y_3} \p_{y_j} 
\big( \uH_j^{\ell_3} \, \uu_3^{\ell_4} -\uu_j^{\ell_3} \, \uH_3^{\ell_4} \big) \\
& \, +\sum_{\substack{\ell_1+\cdots+\ell_6=m+2 \\ \ell_5,\ell_6 \ge 1}} \dot{\chi}^{[\ell_1]} \, \psi^{\ell_2} \, \p_{y_3} 
\Big( \chi^{[\ell_3]} \, \p_{y_j} \psi^{\ell_4} \, \p_{y_3} \big( \uu_j^{\ell_5} \, \uH_3^{\ell_6} -\uH_j^{\ell_5} \, \uu_3^{\ell_6} \big) \Big) \\
& \, +\sum_{\ell_1+\cdots+\ell_5=m+2} \big( \p_{y_3} \chi^{[\ell_1]} \, \dot{\chi}^{[\ell_2]} -\chi^{[\ell_1]} \, \p_{y_3} \dot{\chi}^{[\ell_2]} \big) 
\, \psi^{\ell_3} \, \p_\theta \psi^{\ell_4} \, \p_{y_3} \big( c \, \uH_3^{\ell_5} -b \, \uu_3^{\ell_5} \big) \\
& \, +\sum_{\substack{\ell_1+\cdots+\ell_6=m+2 \\ \ell_5,\ell_6 \ge 1}} \big( \p_{y_3} \chi^{[\ell_1]} \, \dot{\chi}^{[\ell_2]} 
-\chi^{[\ell_1]} \, \p_{y_3} \dot{\chi}^{[\ell_2]} \big) \, \psi^{\ell_3} \, \p_\theta \psi^{\ell_4} \, 
\p_{y_3} \big( \xi_j \, \uu_j^{\ell_5} \, \uH_3^{\ell_6} -\xi_j \, \uH_j^{\ell_5} \, \uu_3^{\ell_6} \big) \Big\} \\
= \, {\bf c}_0 \, \Big\{ \, & \, \sum_{\ell_1+\cdots+\ell_5=m+2} \dot{\chi}^{[\ell_1]} \, \psi^{\ell_2} \, \p_{y_3} \Big( 
\chi^{[\ell_3]} \, \p_\theta \psi^{\ell_4} \, \p_{y_3} \big( b \, \uu_3^{\ell_5} -c \, \uH_3^{\ell_5} \big) \Big) \\
& \, +\sum_{\substack{\ell_1+\cdots+\ell_6=m+2 \\ \ell_5,\ell_6 \ge 1}} \dot{\chi}^{[\ell_1]} \, \psi^{\ell_2} \, \p_{y_3} \Big( 
\chi^{[\ell_3]} \, \p_\theta \psi^{\ell_4} \, \p_{y_3} \big( \xi_j \, \uH_j^{\ell_5} \, \uu_3^{\ell_6} -\xi_j \, \uu_j^{\ell_5} \, \uH_3^{\ell_6} \big) \Big) \\
& \, +\sum_{\ell_1+\cdots+\ell_5=m+2} \big( \p_{y_3} \chi^{[\ell_1]} \, \dot{\chi}^{[\ell_2]} -\chi^{[\ell_1]} \, \p_{y_3} \dot{\chi}^{[\ell_2]} \big) 
\, \psi^{\ell_3} \, \p_\theta \psi^{\ell_4} \, \p_{y_3} \big( c \, \uH_3^{\ell_5} -b \, \uu_3^{\ell_5} \big) \\
& \, +\sum_{\substack{\ell_1+\cdots+\ell_6=m+2 \\ \ell_5,\ell_6 \ge 1}} \big( \p_{y_3} \chi^{[\ell_1]} \, \dot{\chi}^{[\ell_2]} 
-\chi^{[\ell_1]} \, \p_{y_3} \dot{\chi}^{[\ell_2]} \big) \, \psi^{\ell_3} \, \p_\theta \psi^{\ell_4} \, 
\p_{y_3} \big( \xi_j \, \uu_j^{\ell_5} \, \uH_3^{\ell_6} -\xi_j \, \uH_j^{\ell_5} \, \uu_3^{\ell_6} \big) \Big\} \, ,
\end{align*}
where we have used the expression of $\uF_6^\ell$. We thus end up with the expression:
\begin{align*}
\cD_4^m \, = \, {\bf c}_0 \, \Big\{ \, & \, \sum_{\ell_1+\cdots+\ell_5=m+2} \chi^{[\ell_1]} \, \p_\theta \psi^{\ell_2} \, \p_{y_3} \Big( 
\dot{\chi}^{[\ell_3]} \, \psi^{\ell_4} \, \p_{y_3} \big( b \, \uu_3^{\ell_5} -c \, \uH_3^{\ell_5} \big) \Big) \\
& \, +\sum_{\substack{\ell_1+\cdots+\ell_6=m+2 \\ \ell_5,\ell_6 \ge 1}} \chi^{[\ell_1]} \, \p_\theta \psi^{\ell_2} \, \p_{y_3} \Big( 
\dot{\chi}^{[\ell_3]} \, \psi^{\ell_4} \, \p_{y_3} \big( \xi_j \, \uH_j^{\ell_5} \, \uu_3^{\ell_6} -\xi_j \, \uu_j^{\ell_5} \, \uH_3^{\ell_6} \big) 
\Big) \Big\} \, .
\end{align*}
\bigskip

$\bullet$ \underline{Step 5}. Conclusion. Collecting the latter expression of $\cD_4^m$ with the definition of $\cD_5^m$, we get 
(after a few more lines of simplifications...):
$$
\cD^m \, = \, {\bf c}_0 \, \Big\{ \sum_{\ell_1+\ell_2+\ell_3=m+2} \chi^{[\ell_1]} \, \p_\theta \psi^{\ell_2} \, \p_{y_3} \Big( 
\xi_j \, \uF_{3+j}^{\ell_3} -\tau \, \uF_8^{\ell_3} \Big) \Big\} \, .
$$
We are now going to prove that the functions $\xi_j \, \uF_{3+j}^\ell -\tau \, \uF_8^\ell$, $\ell=1,\dots,m$ vanish, which will complete 
the proof of Lemma \ref{lemB4}. For $\ell=m$, we use Lemma \ref{lem_compatibilite_div}, pass to the limit $|Y_3| = +\infty$, and get:
$$
\p_\theta \big( \xi_j \, \uF_{3+j}^m -\tau \, \uF_8^m \big) \, = \, 0 \, .
$$
This means that we have:
$$
\xi_j \, \uF_{3+j}^m -\tau \, \uF_8^m \, = \, \xi_j \, \widehat{\uF}_{3+j}^m (0) -\tau \, \widehat{\uF}_8^m (0) \, = \, 0 \, ,
$$
where the final equality comes from the induction assumption \eqref{inductionHm5}. For $\ell=1,\dots,m-1$, the argument is similar. 
We have already solved the fast problems \eqref{inductionHm2}. Hence, by Theorem \ref{theorem_fast_problem}, there holds:
$$
\p_{Y_3} F_6^\ell -\xi_j \, \p_\theta F_{3+j}^\ell \, = \, \tau \, \p_\theta F_8^\ell \, ,\quad 
\widehat{\uF}_{3+j}^\ell (0) \, = \, \widehat{\uF}_8^\ell (0) \, = \, 0 \, .
$$
By using the same argument as above, we get $\xi_j \, \uF_{3+j}^\ell -\tau \, \uF_8^\ell=0$. We have thus proved $\cD^m=0$ and 
the proof of Lemma \ref{lemB4} is complete.
\end{proof}

\subsection{Compatibility for the magnetic field at the boundary}

Let us first recall the result we aim at proving here:

\begin{lemma}[Compatibility at the boundary for the slow mean problem]
\label{lemB2}
The source terms in \eqref{slow_mean_m+1_edp}, \eqref{slow_mean_m+1_saut} satisfy:
\begin{equation}
\label{lemB2-equation}
\bF_6^{\, m,\pm}|_{\Gamma_0} \, = \, \big( \p_t +u_j^{0,\pm} \, \p_{y_j} \big) \bG_2^{\, m,\pm} -H_j^{0,\pm} \, \p_{y_j} \bG_1^{\, m,\pm} \, ,
\end{equation}
independently of the choice of the slow mean $\widehat{\psi}^{\, m+1}(0)$.
\end{lemma}

\noindent The verification of the compatibility condition \eqref{lemB2-equation} at the boundary $\Gamma_0$ for the slow mean problem 
relies on a preliminary decomposition which we prove right now.

\begin{lemma}[A preliminary decomposition]
\label{lemB3}
Let $\widehat{u}_{3,\star}^{\, m+1}(0),\widehat{H}_{3,\star}^{\, m+1}(0)$ be defined as the unique solutions in $S_\star$ to the differential equations:
$$
\p_{Y_3} \widehat{u}_{3,\star}^{\, m+1,\pm}(0) \, = \, \widehat{F}_{7,\star}^{\, m,\pm}(0) \, ,\quad 
\p_{Y_3} \widehat{H}_{3,\star}^{\, m+1,\pm}(0) \, = \, \widehat{F}_{8,\star}^{\, m,\pm}(0) \, ,\quad Y_3 \in \bR^\pm \, .
$$
Then there holds:
\begin{align}
\Big( \big( \p_t +u_j^{0,\pm} \, \p_{y_j} \big) \widehat{H}_{3,\star}^{\, m+1,\pm}(0) \, & \, -H_j^{0,\pm} \, \p_{y_j} \widehat{u}_{3,\star}^{\, m+1,\pm}(0) 
\Big) \Big|_{y=3=Y_3=0} \notag \\
= \, {\bf c}_0 \, \Big\{ \, & \, \sum_{\substack{\ell_1+\ell_2=m+2 \\ \ell_2 \ge 1}} 
\big( \p_t +u_j^{0,\pm} \, \p_{y_j} \big) \p_\theta \psi^{\ell_1} \, \xi_{j'} \, H_{j'}^{\ell_2,\pm} 
-H_j^{0,\pm} \, \p_{y_j} \p_\theta \psi^{\ell_1} \, \xi_{j'} \, u_{j'}^{\ell_2,\pm} \notag \\
& \, +\sum_{\ell_1+\ell_2=m+2} \p_\theta \psi^{\ell_1} \, \big( \tau \, F_8^{\ell_2,\pm} -\xi_j \, F_{3+j}^{\ell_2,\pm} \big) \notag \\
& \, +\sum_{\ell_1+\ell_2=m+1} 
\big( \p_t +u_j^{0,\pm} \, \p_{y_j} \big) \psi^{\ell_1} \, F_8^{\ell_2,\pm} -H_j^{0,\pm} \, \p_{y_j} \psi^{\ell_1} \, F_7^{\ell_2,\pm} \notag \\
& \, +\sum_{\substack{\ell_1+\ell_2+\ell_3=m+3 \\ \ell_2,\ell_3 \ge 1}} \p_\theta \psi^{\ell_1} \, 
\p_{Y_3} \big( \xi_j \, u_j^{\ell_2,\pm} \, H_3^{\ell_3,\pm} -\xi_j \, H_j^{\ell_2,\pm} \, u_3^{\ell_3,\pm} \big) \notag \\
& \, +\sum_{\substack{\ell_1+\ell_2+\ell_3=m+2 \\ \ell_2,\ell_3 \ge 1}}  \p_{y_j} \psi^{\ell_1} \, 
\p_{Y_3} \big( u_j^{\ell_2,\pm} \, H_3^{\ell_3,\pm} -H_j^{\ell_2,\pm} \, u_3^{\ell_3,\pm} \big) \notag \\
& \, +\sum_{\ell_1+\ell_2=m+1} 
\big( c^\pm \, \p_\theta \psi^{\ell_1+1} +(\p_t +u_j^{0,\pm} \, \p_{y_j}) \psi^{\ell_1} \big) \, \p_{y_3} H_{3,\star}^{\ell_2,\pm} \label{decompositionlemB3} \\
& \, -\sum_{\ell_1+\ell_2=m+1} 
\big( b^\pm \, \p_\theta \psi^{\ell_1+1} +H_j^{0,\pm} \, \p_{y_j} \psi^{\ell_1} \big) \, \p_{y_3} u_{3,\star}^{\ell_2,\pm} \notag \\
& \, +\sum_{\substack{\ell_1+\ell_2=m+1 \\ \ell_1,\ell_2 \ge 1}} 
\p_{y_j} \big( H_j^{\ell_1,\pm} \, u_3^{\ell_2,\pm} -u_j^{\ell_1,\pm} \, H_3^{\ell_2,\pm} \big)_\star \notag \\
& \, +\sum_{\substack{\ell_1+\ell_2+\ell_3=m+2 \\ \ell_2,\ell_3 \ge 1}} \p_\theta \psi^{\ell_1} \, 
\p_{y_3} \big( \xi_j \, u_j^{\ell_2,\pm} \, H_3^{\ell_3,\pm} -\xi_j \, H_j^{\ell_2,\pm} \, u_3^{\ell_3,\pm} \big)_\star \notag \\
& \, +\sum_{\substack{\ell_1+\ell_2+\ell_3=m+1 \\ \ell_2,\ell_3 \ge 1}}  \p_{y_j} \psi^{\ell_1} \, 
\p_{y_3} \big( u_j^{\ell_2,\pm} \, H_3^{\ell_3,\pm} -H_j^{\ell_2,\pm} \, u_3^{\ell_3,\pm} \big)_\star \Big\} \, ,\notag 
\end{align}
where all functions on the right hand side of \eqref{decompositionlemB3} are evaluated at $y_3=Y_3=0$.
\end{lemma}

\begin{proof}[Proof of Lemma \ref{lemB3}]
Actually, we shall prove in what follows a much more general result than \eqref{decompositionlemB3}, but in the proof of Lemma 
\ref{lemB2} which we shall give later on, we shall only be interested in the expression of the left hand side of \eqref{decompositionlemB3} 
on $y_3=Y_3=0$. From now on, we omit the superscripts $\pm$ and restrict the normal variable $y_3$ to $|y_3|<1/3$. This means 
that all functions $\dot{\chi}^{[\ell]}$, $\ell \ge 0$, vanish, all functions $\chi^{[\ell]}$, $\ell \ge 1$, also vanish and $\chi^{[0]}$ equals 
$1$. At the very end of the proof of Lemma \ref{lemB3}, we shall further restrict to $y_3=0$ but since we shall compute partial 
derivatives with respect to $y_3$ in the core of the argument, it is necessary to keep $y_3$ free, close to $0$, for the time being. 
For any $Y_3$, we set:
\begin{align}
\cA \, := \, & \, 
\big( \p_t +u_j^0 \, \p_{y_j} \big) \widehat{H}_{3,\star}^{\, m+1}(0) \, -H_j^0 \, \p_{y_j} \widehat{u}_{3,\star}^{\, m+1}(0) \, ,\label{lemB3defA} \\
\cB \, := \, & \, {\bf c}_0 \, \Big\{ \, \sum_{\substack{\ell_1+\ell_2=m+2 \\ \ell_2 \ge 1}} 
\big( \p_t +u_j^0 \, \p_{y_j} \big) \p_\theta \psi^{\ell_1} \, \xi_{j'} \, H_{j'}^{\ell_2} 
-H_j^0 \, \p_{y_j} \p_\theta \psi^{\ell_1} \, \xi_{j'} \, u_{j'}^{\ell_2} \notag \\
& \, +\sum_{\ell_1+\ell_2=m+2} \p_\theta \psi^{\ell_1} \, \big( \tau \, F_8^{\ell_2} -\xi_j \, F_{3+j}^{\ell_2} \big) \notag \\
& \, +\sum_{\ell_1+\ell_2=m+1} 
\big( \p_t +u_j^0 \, \p_{y_j} \big) \psi^{\ell_1} \, F_8^{\ell_2} -H_j^0 \, \p_{y_j} \psi^{\ell_1} \, F_7^{\ell_2} \notag \\
& \, +\sum_{\substack{\ell_1+\ell_2+\ell_3=m+3 \\ \ell_2,\ell_3 \ge 1}} \p_\theta \psi^{\ell_1} \, 
\p_{Y_3} \big( \xi_j \, u_j^{\ell_2} \, H_3^{\ell_3} -\xi_j \, H_j^{\ell_2} \, u_3^{\ell_3} \big) \notag \\
& \, +\sum_{\substack{\ell_1+\ell_2+\ell_3=m+2 \\ \ell_2,\ell_3 \ge 1}}  \p_{y_j} \psi^{\ell_1} \, 
\p_{Y_3} \big( u_j^{\ell_2} \, H_3^{\ell_3} -H_j^{\ell_2} \, u_3^{\ell_3} \big) \notag \\
& \, +\sum_{\ell_1+\ell_2=m+1} 
\big( c \, \p_\theta \psi^{\ell_1+1} +(\p_t +u_j^0 \, \p_{y_j}) \psi^{\ell_1} \big) \, \p_{y_3} H_{3,\star}^{\ell_2} \label{lemB3defB} \\
& \, -\sum_{\ell_1+\ell_2=m+1} 
\big( b \, \p_\theta \psi^{\ell_1+1} +H_j^0 \, \p_{y_j} \, \psi^{\ell_1} \big) \, \p_{y_3} u_{3,\star}^{\ell_2} \notag \\
& \, +\sum_{\substack{\ell_1+\ell_2=m+1 \\ \ell_1,\ell_2 \ge 1}} 
\p_{y_j} \big( H_j^{\ell_1} \, u_3^{\ell_2} -u_j^{\ell_1} \, H_3^{\ell_2} \big)_\star \notag \\
& \, +\sum_{\substack{\ell_1+\ell_2+\ell_3=m+2 \\ \ell_2,\ell_3 \ge 1}} \p_\theta \psi^{\ell_1} \, 
\p_{y_3} \big( \xi_j \, u_j^{\ell_2} \, H_3^{\ell_3} -\xi_j \, H_j^{\ell_2} \, u_3^{\ell_3} \big)_\star \notag \\
& \, +\sum_{\substack{\ell_1+\ell_2+\ell_3=m+1 \\ \ell_2,\ell_3 \ge 1}}  \p_{y_j} \psi^{\ell_1} \, 
\p_{y_3} \big( u_j^{\ell_2} \, H_3^{\ell_3} -H_j^{\ell_2} \, u_3^{\ell_3} \big)_\star \Big\} \, ,\notag
\end{align}
so that \eqref{decompositionlemB3} reads $(\cA-\cB)|_{y_3=Y_3=0} =0$. Actually, we shall prove below the even more general 
fact $\cA=\cB$ for $|y_3|<1/3$ and for all $Y_3$, which will obviously yield the expected result. Proving the equality $\cA=\cB$ 
is achieved below by proving first $\p_{Y_3} \cA =\p_{Y_3} \cB$ (for all relevant values of $y_3,Y_3$) and by then computing the 
limit at infinity of $\cA$ and $\cB$. We split again the analysis and the calculations in several steps.
\bigskip

$\bullet$ \underline{Step 1}. Using the induction assumption in a suitable way. We first try to derive a suitable expression for the 
partial derivative $\p_{Y_3} \cA$. Starting from the definition \eqref{lemB3defA}, we have:
$$
\p_{Y_3} \cA \, = \, 
\big( \p_t +u_j^0 \, \p_{y_j} \big) \p_{Y_3} \widehat{H}_{3,\star}^{\, m+1}(0) \, -H_j^0 \, \p_{y_j} \p_{Y_3} \widehat{u}_{3,\star}^{\, m+1}(0) 
\, = \, \big( \p_t +u_j^0 \, \p_{y_j} \big) \widehat{F}_{8,\star}^m(0) \, -H_j^0 \, \p_{y_j} \widehat{F}_{7,\star}^m(0) \, ,
$$
and one could then substitute the expressions of $\widehat{F}_{8,\star}^m(0)$ and $\widehat{F}_{7,\star}^m(0)$. But rather than 
doing so, we use the induction assumption \eqref{inductionHm6} and \eqref{inductionHm5}, and get:
$$
\p_{Y_3} \cA \, = \, \p_t \widehat{F}_{8,\star}^m(0) \, -\p_{y_j} \widehat{F}_{3+j,\star}^m(0) \, = \, 
{\bf c}_0 \, \Big\{ \p_t F_8^m \, -\p_{y_j} F_{3+j}^m \Big\} \, .
$$
Recalling the expressions \eqref{expressionF8m} and \eqref{expressionF3+jm}, we obtain:
\begin{align}
\p_{Y_3} \cA \, = \, & \, {\bf c}_0 \, \Big\{ \, {\color{blue} -\p_{y_3} \big( (\p_t +u_j^0 \, \p_{y_j}) H_3^m -H_j^0 \, \p_{y_j} u_3^m \big)} \notag \\
& \, +\sum_{\ell_1+\ell_2=m+2} \p_\theta \psi^{\ell_1} \, \p_{Y_3} \big( (\p_t +u_j^0 \, \p_{y_j}) \xi_{j'} \, H_{j'}^{\ell_2} 
-H_j^0 \, \p_{y_j} \xi_{j'} \, u_{j'}^{\ell_2} {\color{ForestGreen} -c \, \p_{y_j} H_j^{\ell_2} +b \, \p_{y_j} u_j^{\ell_2}} \big) \notag \\
& \, +\sum_{\ell_1+\ell_2=m+2} (\p_t +u_j^0 \, \p_{y_j}) \p_\theta \psi^{\ell_1} \, \p_{Y_3} \xi_{j'} \, H_{j'}^{\ell_2} 
-H_j^0 \, \p_{y_j} \p_\theta \psi^{\ell_1} \, \p_{Y_3} \xi_{j'} \, u_{j'}^{\ell_2} \notag \\
& \, -\sum_{\ell_1+\ell_2=m+2} 
c \, \p_{y_j} \p_\theta \psi^{\ell_1} \, \p_{Y_3} H_j^{\ell_2} -b \, \p_{y_j} \p_\theta \psi^{\ell_1} \, \p_{Y_3} u_j^{\ell_2} \notag \\
& \, +\sum_{\ell_1+\ell_2=m+1}\p_{y_{j'}} \psi^{\ell_1} \, \p_{Y_3} \big( (\p_t +u_j^0 \, \p_{y_j}) H_{j'}^{\ell_2} 
-H_j^0 \, \p_{y_j} u_{j'}^{\ell_2} {\color{ForestGreen} -u_{j'}^0 \, \p_{y_j} H_j^{\ell_2} +H_{j'}^0 \, \p_{y_j} u_j^{\ell_2}} \big) \notag \\
& \, -{\color{ForestGreen} \sum_{\ell_1+\ell_2=m+1} \p_t \psi^{\ell_1} \, \p_{Y_3} \p_{y_j} H_j^{\ell_2}}\notag \\
& \, +\sum_{\ell_1+\ell_2=m+1} \p_\theta \psi^{\ell_1} \, \p_{y_3} \big( (\p_t +u_j^0 \, \p_{y_j}) \xi_{j'} \, H_{j'}^{\ell_2} 
-H_j^0 \, \p_{y_j} \xi_{j'} \, u_{j'}^{\ell_2} {\color{ForestGreen} -c \, \p_{y_j} H_j^{\ell_2} +b \, \p_{y_j} u_j^{\ell_2}} \big) \notag \\
& \, +{\color{magenta} \sum_{\ell_1+\ell_2=m+1} (\p_t +u_j^0 \, \p_{y_j}) \p_\theta \psi^{\ell_1} \, \p_{y_3} \xi_{j'} \, H_{j'}^{\ell_2} 
-H_j^0 \, \p_{y_j} \p_\theta \psi^{\ell_1} \, \p_{y_3} \xi_{j'} \, u_{j'}^{\ell_2}} \notag \\
& \, -{\color{magenta} \sum_{\ell_1+\ell_2=m+1} 
c \, \p_{y_j} \p_\theta \psi^{\ell_1} \, \p_{y_3} H_j^{\ell_2} -b \, \p_{y_j} \p_\theta \psi^{\ell_1} \, \p_{y_3} u_j^{\ell_2}} \notag \\
& \, +\sum_{\ell_1+\ell_2=m}\p_{y_{j'}} \psi^{\ell_1} \, \p_{y_3} \big( (\p_t +u_j^0 \, \p_{y_j}) H_{j'}^{\ell_2} 
-H_j^0 \, \p_{y_j} u_{j'}^{\ell_2} {\color{ForestGreen} -u_{j'}^0 \, \p_{y_j} H_j^{\ell_2} +H_{j'}^0 \, \p_{y_j} u_j^{\ell_2}} \big) \notag \\
& \, -{\color{ForestGreen} \sum_{\ell_1+\ell_2=m} \p_t \psi^{\ell_1} \, \p_{y_3} \p_{y_j} H_j^{\ell_2}} \label{lemB3expressiondA1} \\
& \, +\sum_{\substack{\ell_1+\ell_2=m+1 \\ \ell_1,\ell_2 \ge 1}} 
\p_{y_j} \p_{Y_3} \big( H_j^{\ell_1} \, u_3^{\ell_2} -u_j^{\ell_1} \, H_3^{\ell_2} \big) 
+{\color{red} \sum_{\substack{\ell_1+\ell_2=m \\ \ell_1,\ell_2 \ge 1}} 
\p_{y_j} \p_{y_3} \big( H_j^{\ell_1} \, u_3^{\ell_2} -u_j^{\ell_1} \, H_3^{\ell_2} \big)} \notag \\
& \, +{\color{ForestGreen} \sum_{\substack{\ell_1+\ell_2+\ell_3=m+2 \\ \ell_2,\ell_3 \ge 1}} \p_\theta \psi^{\ell_1} \, 
\p_{y_j} \p_{Y_3} \big( \xi_{j'} \, H_{j'}^{\ell_2} \, u_j^{\ell_3} -\xi_{j'} \, u_{j'}^{\ell_2} \, H_j^{\ell_3} \big)} \notag \\
& \, +{\color{magenta} \sum_{\substack{\ell_1+\ell_2+\ell_3=m+2 \\ \ell_2,\ell_3 \ge 1}} \p_{y_j} \p_\theta \psi^{\ell_1} \, 
\p_{Y_3} \big( \xi_{j'} \, H_{j'}^{\ell_2} \, u_j^{\ell_3} -\xi_{j'} \, u_{j'}^{\ell_2} \, H_j^{\ell_3} \big)} \notag \\
& \, +{\color{ForestGreen} \sum_{\substack{\ell_1+\ell_2+\ell_3=m+1 \\ \ell_2,\ell_3 \ge 1}}\p_{y_{j'}} \psi^{\ell_1} \, 
\p_{y_j} \p_{Y_3} \big( H_{j'}^{\ell_2} \, u_j^{\ell_3} -u_{j'}^{\ell_2} \, H_j^{\ell_3} \big)} \notag \\
& \, +{\color{ForestGreen} \sum_{\substack{\ell_1+\ell_2+\ell_3=m+1 \\ \ell_2,\ell_3 \ge 1}} \p_\theta \psi^{\ell_1} \, 
\p_{y_j} \p_{y_3} \big( \xi_{j'} \, H_{j'}^{\ell_2} \, u_j^{\ell_3} -\xi_{j'} \, u_{j'}^{\ell_2} \, H_j^{\ell_3} \big)} \notag \\
& \, +\sum_{\substack{\ell_1+\ell_2+\ell_3=m+1 \\ \ell_2,\ell_3 \ge 1}} \p_{y_j} \p_\theta \psi^{\ell_1} \, 
\p_{y_3} \big( \xi_{j'} \, H_{j'}^{\ell_2} \, u_j^{\ell_3} -\xi_{j'} \, u_{j'}^{\ell_2} \, H_j^{\ell_3} \big) \notag \\
& \, +{\color{ForestGreen} \sum_{\substack{\ell_1+\ell_2+\ell_3=m \\ \ell_2,\ell_3 \ge 1}}\p_{y_{j'}} \psi^{\ell_1} \, 
\p_{y_j} \p_{y_3} \big( H_{j'}^{\ell_2} \, u_j^{\ell_3} -u_{j'}^{\ell_2} \, H_j^{\ell_3} \big)} \, \Big\} \, .\notag 
\end{align}
\bigskip

$\bullet$ \underline{Step 2}. Substituting in $\p_{Y_3} \cA$, integrating by parts and collecting terms. From the induction 
assumption \eqref{inductionHm5}, we have $\widehat{\uF}_6^m(0)=0$, and using Lemma \ref{lem_compatibilite_div}, we also 
get $\widehat{F}_{6,\star}^m(0)=0$. We thus have $\widehat{F}_6^m(0)=0$, and using the expression \eqref{expressionF6m}, 
this relation explicitly reads\footnote{Recall that we restrict to $|y_3|<1/3$, hence all simplifications in the functions $\chi^{[\ell]}$ 
and $\dot{\chi}^{[\ell]}$.}:
\begin{align*}
\big( \p_t +u_j^0 \, \p_{y_j} \big) \widehat{H}_3^m(0) -H_j^0 \, \p_{y_j} \widehat{u}_3^m(0) \, = \, & \, {\bf c}_0 \, \Big\{ \, 
{\color{ForestGreen} \sum_{\ell_1+\ell_2=m+2} \p_\theta \psi^{\ell_1} \, \big( c \, \p_{Y_3} H_3^{\ell_2} -b \, \p_{Y_3} u_3^{\ell_2} \big)} \\
& \, +{\color{ForestGreen} \sum_{\ell_1+\ell_2=m+1} \big( \p_t +u_j^0 \, \p_{y_j} \big) \psi^{\ell_1} \, \p_{Y_3} H_3^{\ell_2} 
-H_j^0 \, \p_{y_j} \psi^{\ell_1} \, \p_{Y_3} u_3^{\ell_2}} \\
& \, +{\color{ForestGreen} \sum_{\ell_1+\ell_2=m+1} \p_\theta \psi^{\ell_1} \, \big( c \, \p_{y_3} H_3^{\ell_2} -b \, \p_{y_3} u_3^{\ell_2} \big)} \\
& \, +{\color{ForestGreen} \sum_{\ell_1+\ell_2=m} \big( \p_t +u_j^0 \, \p_{y_j} \big) \psi^{\ell_1} \, \p_{y_3} H_3^{\ell_2} 
-H_j^0 \, \p_{y_j} \psi^{\ell_1} \, \p_{y_3} u_3^{\ell_2}} \\
& \, +{\color{red} \sum_{\substack{\ell_1+\ell_2=m \\ \ell_1,\ell_2 \ge 1}} 
\p_{y_j} \big( H_j^{\ell_1} \, u_3^{\ell_2} -u_j^{\ell_1} \, H_3^{\ell_2} \big)} \\
& \, +{\color{ForestGreen} \sum_{\substack{\ell_1+\ell_2+\ell_3=m+2 \\ \ell_2,\ell_3 \ge 1}} \p_\theta \psi^{\ell_1} \, 
\p_{Y_3} \big( \xi_j \, u_j^{\ell_2} \, H_3^{\ell_3} -\xi_j \, H_j^{\ell_2} \, u_3^{\ell_3} \big)} \\
& \, +{\color{ForestGreen} \sum_{\substack{\ell_1+\ell_2+\ell_3=m+1 \\ \ell_2,\ell_3 \ge 1}} \p_{y_j} \psi^{\ell_1} \, 
\p_{Y_3} \big( u_j^{\ell_2} \, H_3^{\ell_3} -H_j^{\ell_2} \, u_3^{\ell_3} \big)} \\
& \, +{\color{ForestGreen} \sum_{\substack{\ell_1+\ell_2+\ell_3=m+1 \\ \ell_2,\ell_3 \ge 1}} \p_\theta \psi^{\ell_1} \, 
\p_{y_3} \big( \xi_j \, u_j^{\ell_2} \, H_3^{\ell_3} -\xi_j \, H_j^{\ell_2} \, u_3^{\ell_3} \big)} \\
& \, +{\color{ForestGreen} \sum_{\substack{\ell_1+\ell_2+\ell_3=m \\ \ell_2,\ell_3 \ge 1}} \p_{y_j} \psi^{\ell_1} \, 
\p_{y_3} \big( u_j^{\ell_2} \, H_3^{\ell_3} -H_j^{\ell_2} \, u_3^{\ell_3} \big)} \, \Big\} \, .
\end{align*}
We differentiate the latter relation with respect to $y_3$ (hence the need to keep $y_3$ free at least up to here), substitute in 
the blue term on the right hand side of \eqref{lemB3expressiondA1}, observe first the cancelation of the red terms and rearrange 
all green terms together. We also integrate by parts (with respect to $\theta$) the pink terms in \eqref{lemB3expressiondA1}. All 
these manipulations eventually yield:
\begin{align}
\p_{Y_3} \cA \, = \, & \, {\bf c}_0 \, \Big\{ \, 
\sum_{\ell_1+\ell_2=m+2} \p_\theta \psi^{\ell_1} \, \p_{Y_3} \big( (\p_t +u_j^0 \, \p_{y_j}) \, \xi_{j'} \, H_{j'}^{\ell_2} 
-H_j^0 \, \p_{y_j} \, \xi_{j'} \, u_{j'}^{\ell_2} -c \, \nabla \cdot H^{\ell_2} +b \, \nabla \cdot u^{\ell_2} \big) \notag \\
& \, +{\color{orange} \sum_{\ell_1+\ell_2=m+2} (\p_t +u_j^0 \, \p_{y_j}) \p_\theta \psi^{\ell_1} \, \p_{Y_3} \xi_{j'} \, H_{j'}^{\ell_2} 
-H_j^0 \, \p_{y_j} \p_\theta \psi^{\ell_1} \, \p_{Y_3} \xi_{j'} \, u_{j'}^{\ell_2}} \notag \\
& \, -{\color{magenta} \sum_{\ell_1+\ell_2=m+2} 
c \, \p_{y_j} \p_\theta \psi^{\ell_1} \, \p_{Y_3} H_j^{\ell_2} -b \, \p_{y_j} \p_\theta \psi^{\ell_1} \, \p_{Y_3} u_j^{\ell_2}} \notag \\
& \, +\sum_{\ell_1+\ell_2=m+1} \p_{y_j} \psi^{\ell_1} \, \p_{Y_3} \big( (\p_t +u_{j'}^0 \,\p_{y_{j'}}) H_j^{\ell_2} 
-H_{j'}^0 \,\p_{y_{j'}} u_j^{\ell_2} \big) \notag \\
& \, -\sum_{\ell_1+\ell_2=m+1} \big( \p_t +u_j^0 \, \p_{y_j} \big) \psi^{\ell_1} \, \p_{Y_3} \nabla \cdot H^{\ell_2} 
-H_j^0 \, \p_{y_j} \psi^{\ell_1} \, \p_{Y_3} \nabla \cdot u^{\ell_2} \notag \\
& \, +\sum_{\ell_1+\ell_2=m+1} \p_\theta \psi^{\ell_1} \, \p_{y_3} \big( (\p_t +u_j^0 \, \p_{y_j}) \, \xi_{j'} \, H_{j'}^{\ell_2} 
-H_j^0 \, \p_{y_j} \, \xi_{j'} \, u_{j'}^{\ell_2} -c \, \nabla \cdot H^{\ell_2} +b \, \nabla \cdot u^{\ell_2} \big) \notag \\
& \, -{\color{ForestGreen} \sum_{\ell_1+\ell_2=m+1} (\p_t +u_j^0 \, \p_{y_j}) \psi^{\ell_1} \, \p_{y_3} \p_\theta \, \xi_{j'} \, H_{j'}^{\ell_2} 
-H_j^0 \, \p_{y_j} \psi^{\ell_1} \, \p_{y_3} \p_\theta \, \xi_{j'} \, u_{j'}^{\ell_2}} \notag \\
& \, +{\color{Mahogany} \sum_{\ell_1+\ell_2=m+1} \p_{y_j} \psi^{\ell_1} \, \p_{y_3} \big( 
c \, \p_\theta H_j^{\ell_2} -b \, \p_\theta u_j^{\ell_2} \big)} \notag \\
& \, +\sum_{\ell_1+\ell_2=m} \p_{y_j} \psi^{\ell_1} \, \p_{y_3} \big( (\p_t +u_{j'}^0 \,\p_{y_{j'}}) H_j^{\ell_2} 
-H_{j'}^0 \,\p_{y_{j'}} u_j^{\ell_2} \big) \notag \\
& \, -\sum_{\ell_1+\ell_2=m} \big( \p_t +u_j^0 \, \p_{y_j} \big) \psi^{\ell_1} \, \p_{y_3} \nabla \cdot H^{\ell_2} 
-H_j^0 \, \p_{y_j} \psi^{\ell_1} \, \p_{y_3} \nabla \cdot u^{\ell_2} \label{lemB3expressiondA2} \\
& \, +{\color{orange} \sum_{\substack{\ell_1+\ell_2=m+1 \\ \ell_1,\ell_2 \ge 1}} 
\p_{y_j} \p_{Y_3} \big( H_j^{\ell_1} \, u_3^{\ell_2} -u_j^{\ell_1} \, H_3^{\ell_2} \big)} \notag \\
& \, +\sum_{\substack{\ell_1+\ell_2+\ell_3=m+2 \\ \ell_2,\ell_3 \ge 1}} \p_\theta \psi^{\ell_1} \, 
\p_{Y_3} \nabla \cdot \big( \xi_j \, H_j^{\ell_2} \, u^{\ell_3} -\xi_j \, u_j^{\ell_2} \, H^{\ell_3} \big) \notag \\
& \, +\sum_{\substack{\ell_1+\ell_2+\ell_3=m+2 \\ \ell_2,\ell_3 \ge 1}} \p_{y_j} \psi^{\ell_1} \, 
\p_{Y_3} \p_\theta \big( \xi_{j'} \, u_{j'}^{\ell_2} \, H_j^{\ell_3} -\xi_{j'} \, H_{j'}^{\ell_2} \, u_j^{\ell_3} \big) \notag \\
& \, +\sum_{\substack{\ell_1+\ell_2+\ell_3=m+1 \\ \ell_2,\ell_3 \ge 1}} \p_{y_j} \psi^{\ell_1} \, 
\p_{Y_3} \nabla \cdot \big( H_j^{\ell_2} \, u^{\ell_3} -u_j^{\ell_2} \, H^{\ell_3} \big) \notag \\
& \, +\sum_{\substack{\ell_1+\ell_2+\ell_3=m+1 \\ \ell_2,\ell_3 \ge 1}} \p_\theta \psi^{\ell_1} \, 
\p_{y_3} \nabla \cdot \big( \xi_j \, H_j^{\ell_2} \, u^{\ell_3} -\xi_j \, u_j^{\ell_2} \, H^{\ell_3} \big) \notag \\
& \, +\sum_{\substack{\ell_1+\ell_2+\ell_3=m+1 \\ \ell_2,\ell_3 \ge 1}} \p_{y_j} \psi^{\ell_1} \, 
\p_{y_3} \p_\theta \big( \xi_{j'} \, u_{j'}^{\ell_2} \, H_j^{\ell_3} -\xi_{j'} \, H_{j'}^{\ell_2} \, u_j^{\ell_3} \big) \notag \\
& \, +\sum_{\substack{\ell_1+\ell_2+\ell_3=m \\ \ell_2,\ell_3 \ge 1}} \p_{y_j} \psi^{\ell_1} \, 
\p_{y_3} \nabla \cdot \big( H_j^{\ell_2} \, u^{\ell_3} -u_j^{\ell_2} \, H^{\ell_3} \big) \, \Big\} \, .\notag 
\end{align}
\bigskip

$\bullet$ \underline{Step 3}. Computing $\p_{Y_3} \cB$ and using the fast problems of the previous steps. Differentiating the 
definition \eqref{lemB3defB} of $\cB$, we get:
\begin{align}
\p_{Y_3} \cB \, = \, & \, {\bf c}_0 \, \Big\{ \, {\color{orange} \sum_{\ell_1+\ell_2=m+2} 
\big( \p_t +u_j^0 \, \p_{y_j} \big) \p_\theta \psi^{\ell_1} \, \p_{Y_3} \xi_{j'} \, H_{j'}^{\ell_2} 
-H_j^0 \, \p_{y_j} \p_\theta \psi^{\ell_1} \, \p_{Y_3} \xi_{j'} \, u_{j'}^{\ell_2}} \notag \\
& \, +\sum_{\ell_1+\ell_2=m+2} \p_\theta \psi^{\ell_1} \, \p_{Y_3} \big( \tau \, F_8^{\ell_2} -\xi_j \, F_{3+j}^{\ell_2} \big) \notag \\
& \, +\sum_{\ell_1+\ell_2=m+1} \big( \p_t +{\color{red} u_j^0 \, \p_{y_j}} \big) \psi^{\ell_1} \, \p_{Y_3} F_8^{\ell_2} 
-{\color{red} H_j^0 \, \p_{y_j} \psi^{\ell_1} \, \p_{Y_3} F_7^{\ell_2}} \notag \\
& \, +\sum_{\substack{\ell_1+\ell_2+\ell_3=m+3 \\ \ell_2,\ell_3 \ge 1}} \p_\theta \psi^{\ell_1} \, 
\p_{Y_3}^2 \big( \xi_j \, u_j^{\ell_2} \, H_3^{\ell_3} -\xi_j \, H_j^{\ell_2} \, u_3^{\ell_3} \big) \notag \\
& \, +\sum_{\substack{\ell_1+\ell_2+\ell_3=m+2 \\ \ell_2,\ell_3 \ge 1}}  \p_{y_j} \psi^{\ell_1} \, 
\p_{Y_3}^2 \big( u_j^{\ell_2} \, H_3^{\ell_3} -H_j^{\ell_2} \, u_3^{\ell_3} \big) \notag \\
& \, +\sum_{\ell_1+\ell_2=m+1} \big( c \, \p_\theta \psi^{\ell_1+1} +{\color{ForestGreen} (\p_t +u_j^0 \, \p_{y_j}) \psi^{\ell_1}} \big) 
\, \p_{Y_3} \p_{y_3} H_3^{\ell_2} \label{lemB3expressiondB} \\
& \, -\sum_{\ell_1+\ell_2=m+1} \big( b \, \p_\theta \psi^{\ell_1+1} +{\color{ForestGreen} H_j^0 \, \p_{y_j} \, \psi^{\ell_1}} \big) 
\, \p_{Y_3} \p_{y_3} u_3^{\ell_2} \notag \\
& \, +{\color{orange} \sum_{\substack{\ell_1+\ell_2=m+1 \\ \ell_1,\ell_2 \ge 1}} 
\p_{y_j} \p_{Y_3} \big( H_j^{\ell_1} \, u_3^{\ell_2} -u_j^{\ell_1} \, H_3^{\ell_2} \big)} \notag \\
& \, +\sum_{\substack{\ell_1+\ell_2+\ell_3=m+2 \\ \ell_2,\ell_3 \ge 1}} \p_\theta \psi^{\ell_1} \, 
\p_{Y_3} \p_{y_3} \big( \xi_j \, u_j^{\ell_2} \, H_3^{\ell_3} -\xi_j \, H_j^{\ell_2} \, u_3^{\ell_3} \big) \notag \\
& \, +\sum_{\substack{\ell_1+\ell_2+\ell_3=m+1 \\ \ell_2,\ell_3 \ge 1}}  \p_{y_j} \psi^{\ell_1} \, 
\p_{Y_3} \p_{y_3} \big( u_j^{\ell_2} \, H_3^{\ell_3} -H_j^{\ell_2} \, u_3^{\ell_3} \big) \Big\} \, .\notag
\end{align}
Let us already observe that the orange terms in \eqref{lemB3expressiondA2} and \eqref{lemB3expressiondB} will cancel when 
we compute $\p_{Y_3} \cA -\p_{Y_3} \cB$. The red terms will also cancel, but only after we integrate by parts the pink terms in 
\eqref{lemB3expressiondA2}. Namely, by using the fast problems solved at the previous steps of the induction, we have:
$$
\forall \, \mu \, = \, 1,\dots,m \, ,\quad c \, \p_\theta H_j^\mu -b \, \p_\theta u_j^\mu \, = \, 
F_{3+j}^{\mu-1} +u_j^0 \, F_8^{\mu-1} -H_j^0 \, F_7^{\mu-1} \, ,
$$
and therefore, after integrating by parts, the pink term in \eqref{lemB3expressiondA2} reads:
\begin{multline*}
{\bf c}_0 \, \Big\{ \, \sum_{\ell_1+\ell_2=m+2} 
c \, \p_{y_j} \p_\theta \psi^{\ell_1} \, \p_{Y_3} H_j^{\ell_2} -b \, \p_{y_j} \p_\theta \psi^{\ell_1} \, \p_{Y_3} u_j^{\ell_2} \, \Big\} \\
= \, - \, {\bf c}_0 \, \Big\{ \, \sum_{\ell_1+\ell_2=m+1} \p_{y_j} \psi^{\ell_1} \, 
\p_{Y_3} \big( F_{3+j}^{\ell_2} +{\color{red} u_j^0 \, F_8^{\ell_2} -H_j^0 \, F_7^{\ell_2}} \big) \, \Big\} \, ,
\end{multline*}
and the red terms in the latter expression will exactly cancel with those in \eqref{lemB3expressiondB} when we compute 
$\p_{Y_3} \cA -\p_{Y_3} \cB$. At last, let us observe that when we compute the difference $\p_{Y_3} \cA -\p_{Y_3} \cB$, 
the green terms in \eqref{lemB3expressiondA2} and \eqref{lemB3expressiondB} will yield the quantity:
$$
{\bf c}_0 \, \Big\{ \sum_{\ell_1+\ell_2=m} (\p_t +u_j^0 \, \p_{y_j}) \psi^{\ell_1} \, \p_{y_3} F_8^{\ell_2} 
-H_j^0 \, \p_{y_j} \psi^{\ell_1} \, \p_{y_3} F_7^{\ell_2} \, \Big\} \, ,
$$
which will partially cancel (for the same reason as just above) with the brown term in \eqref{lemB3expressiondA2}.

Using all simplifications mentioned above, we end up with the following decomposition:
$$
\p_{Y_3} \cA -\p_{Y_3} \cB \, = \cT_1 +\cT_2 +\cT_3 +\cT_4 \, ,
$$
with:
\begin{align*}
\cT_1 \, := \, & \, {\bf c}_0 \, \Big\{ \, 
\sum_{\ell_1+\ell_2=m+2} \p_\theta \psi^{\ell_1} \, \p_{Y_3} \big( \xi_j \, F_{3+j}^{\ell_2} -\tau \, F_8^{\ell_2} \big) \\
& \, +\sum_{\ell_1+\ell_2=m+2} \p_\theta \psi^{\ell_1} \, \p_{Y_3} \big( (\p_t +u_j^0 \, \p_{y_j}) \, \xi_{j'} \, H_{j'}^{\ell_2} 
-H_j^0 \, \p_{y_j} \, \xi_{j'} \, u_{j'}^{\ell_2} -c \, \nabla \cdot H^{\ell_2} +b \, \nabla \cdot u^{\ell_2} \big) \\
& \, +\sum_{\substack{\ell_1+\ell_2+\ell_3=m+3 \\ \ell_2,\ell_3 \ge 1}} \p_\theta \psi^{\ell_1} \, 
\p_{Y_3}^2 \big( \xi_j \, u_j^{\ell_2} \, H_3^{\ell_3} -\xi_j \, H_j^{\ell_2} \, u_3^{\ell_3} \big) \\
& \, +\sum_{\substack{\ell_1+\ell_2+\ell_3=m+2 \\ \ell_2,\ell_3 \ge 1}} \p_\theta \psi^{\ell_1} \, 
\p_{Y_3} \nabla \cdot \big( \xi_j \, H_j^{\ell_2} \, u^{\ell_3} -\xi_j \, u_j^{\ell_2} \, H^{\ell_3} \big) \, \Big\} \, ,
\end{align*}
\begin{align*}
\cT_2 \, := \, & \, {\bf c}_0 \, \Big\{ \, 
\sum_{\ell_1+\ell_2=m+1} \p_\theta \psi^{\ell_1} \, \p_{y_3} \big( \xi_j \, F_{3+j}^{\ell_2} -\tau \, F_8^{\ell_2} \big) \\
& \, +\sum_{\ell_1+\ell_2=m+1} \p_\theta \psi^{\ell_1} \, \p_{y_3} \big( (\p_t +u_j^0 \, \p_{y_j}) \, \xi_{j'} \, H_{j'}^{\ell_2} 
-H_j^0 \, \p_{y_j} \, \xi_{j'} \, u_{j'}^{\ell_2} -c \, \nabla \cdot H^{\ell_2} +b \, \nabla \cdot u^{\ell_2} \big) \\
& \, +\sum_{\substack{\ell_1+\ell_2+\ell_3=m+2 \\ \ell_2,\ell_3 \ge 1}} \p_\theta \psi^{\ell_1} \, 
\p_{Y_3} \p_{y_3} \big( \xi_j \, u_j^{\ell_2} \, H_3^{\ell_3} -\xi_j \, H_j^{\ell_2} \, u_3^{\ell_3} \big) \\
& \, +\sum_{\substack{\ell_1+\ell_2+\ell_3=m+1 \\ \ell_2,\ell_3 \ge 1}} \p_\theta \psi^{\ell_1} \, 
\p_{y_3} \nabla \cdot \big( \xi_j \, H_j^{\ell_2} \, u^{\ell_3} -\xi_j \, u_j^{\ell_2} \, H^{\ell_3} \big) \, \Big\} \, ,
\end{align*}
\begin{align*}
\cT_3 \, := \, & \, {\bf c}_0 \, \Big\{ \, 
\sum_{\ell_1+\ell_2=m+1} \p_{y_j} \psi^{\ell_1} \, \p_{Y_3} F_{3+j}^{\ell_2} -\p_t \psi^{\ell_1} \, \p_{Y_3} F_8^{\ell_2} \\
& \, +\sum_{\ell_1+\ell_2=m+1} \p_{y_j} \psi^{\ell_1} \, \p_{Y_3} \big( (\p_t +u_{j'}^0 \,\p_{y_{j'}}) H_j^{\ell_2} 
-H_{j'}^0 \,\p_{y_{j'}} u_j^{\ell_2} \big) \\
& \, -\sum_{\ell_1+\ell_2=m+1} \big( \p_t +u_j^0 \, \p_{y_j} \big) \psi^{\ell_1} \, \p_{Y_3} \nabla \cdot H^{\ell_2} 
-H_j^0 \, \p_{y_j} \psi^{\ell_1} \, \p_{Y_3} \nabla \cdot u^{\ell_2} \\
& \, +\sum_{\substack{\ell_1+\ell_2+\ell_3=m+2 \\ \ell_2,\ell_3 \ge 1}}  \p_{y_j} \psi^{\ell_1} \, \p_{Y_3} 
\Big( \p_\theta \big( \xi_{j'} \, u_{j'}^{\ell_2} \, H_j^{\ell_3} -\xi_{j'} \, H_{j'}^{\ell_2} \, u_j^{\ell_3} \big) 
+\p_{Y_3} \big( H_j^{\ell_2} \, u_3^{\ell_3} -u_j^{\ell_2} \, H_3^{\ell_3} \big) \Big) \\
& \, +\sum_{\substack{\ell_1+\ell_2+\ell_3=m+1 \\ \ell_2,\ell_3 \ge 1}} \p_{y_j} \psi^{\ell_1} \, 
\p_{Y_3} \nabla \cdot \big( H_j^{\ell_2} \, u^{\ell_3} -u_j^{\ell_2} \, H^{\ell_3} \big) \, \Big\} \, ,
\end{align*}
and
\begin{align*}
\cT_4 \, := \, & \, {\bf c}_0 \, \Big\{ \, 
\sum_{\ell_1+\ell_2=m} \p_{y_j} \psi^{\ell_1} \, \p_{y_3} F_{3+j}^{\ell_2} -\p_t \psi^{\ell_1} \, \p_{y_3} F_8^{\ell_2} \\
& \, +\sum_{\ell_1+\ell_2=m} \p_{y_j} \psi^{\ell_1} \, \p_{y_3} \big( (\p_t +u_{j'}^0 \,\p_{y_{j'}}) H_j^{\ell_2} 
-H_{j'}^0 \,\p_{y_{j'}} u_j^{\ell_2} \big) \\
& \, -\sum_{\ell_1+\ell_2=m} \big( \p_t +u_j^0 \, \p_{y_j} \big) \psi^{\ell_1} \, \p_{y_3} \nabla \cdot H^{\ell_2} 
-H_j^0 \, \p_{y_j} \psi^{\ell_1} \, \p_{y_3} \nabla \cdot u^{\ell_2} \\
& \, +\sum_{\substack{\ell_1+\ell_2+\ell_3=m+1 \\ \ell_2,\ell_3 \ge 1}}  \p_{y_j} \psi^{\ell_1} \, \p_{y_3} 
\Big( \p_\theta \big( \xi_{j'} \, u_{j'}^{\ell_2} \, H_j^{\ell_3} -\xi_{j'} \, H_{j'}^{\ell_2} \, u_j^{\ell_3} \big) 
+\p_{Y_3} \big( H_j^{\ell_2} \, u_3^{\ell_3} -u_j^{\ell_2} \, H_3^{\ell_3} \big) \Big) \\
& \, +\sum_{\substack{\ell_1+\ell_2+\ell_3=m \\ \ell_2,\ell_3 \ge 1}} \p_{y_j} \psi^{\ell_1} \, 
\p_{y_3} \nabla \cdot \big( H_j^{\ell_2} \, u^{\ell_3} -u_j^{\ell_2} \, H^{\ell_3} \big) \, \Big\} \, .
\end{align*}
\bigskip

$\bullet$ \underline{Step 4}. Substituting in $\cT_1$ and $\cT_2$. In both expressions of $\cT_1$ and $\cT_2$, we use 
\eqref{expressionF8m} and \eqref{expressionF3+jm} and substitute accordingly the value of $\xi_j \, F_{3+j}^{\ell_2} -\tau 
\, F_8^{\ell_2}$. This yields:
\begin{align*}
\cT_1 \, = \, & \, {\bf c}_0 \, \Big\{ \, 
\sum_{\ell_1+\ell_2+\ell_3=m+3} \p_\theta \psi^{\ell_1} \, \big( (\p_t +u_j^0 \, \p_{y_j}) \psi^{\ell_2} \, \p_{Y_3}^2 \xi_{j'} \, H_{j'}^{\ell_3} 
-H_j^0 \, \p_{y_j} \psi^{\ell_2} \, \p_{Y_3}^2 \xi_{j'} \, u_{j'}^{\ell_3} \big) \\
& \, +\sum_{\ell_1+\ell_2+\ell_3=m+3} \p_{y_j} \psi^{\ell_1} \, \p_\theta \psi^{\ell_2} \, 
\big( b \, \p_{Y_3}^2 u_j^{\ell_3} -c \, \p_{Y_3}^2 H_j^{\ell_3} \big) \\
& \, +\sum_{\ell_1+\ell_2+\ell_3=m+2} \p_\theta \psi^{\ell_1} \, \big( (\p_t +u_j^0 \, \p_{y_j}) \psi^{\ell_2} \, 
\p_{Y_3} \p_{y_3} \xi_{j'} \, H_{j'}^{\ell_3} -H_j^0 \, \p_{y_j} \psi^{\ell_2} \, \p_{Y_3} \p_{y_3} \xi_{j'} \, u_{j'}^{\ell_3} \big) \\
& \, +\sum_{\ell_1+\ell_2+\ell_3=m+2} \p_{y_j} \psi^{\ell_1} \, \p_\theta \psi^{\ell_2} \, 
\big( b \, \p_{Y_3} \p_{y_3} u_j^{\ell_3} -c \, \p_{Y_3} \p_{y_3} H_j^{\ell_3} \big) \\
& \, +\sum_{\substack{\ell_1+\cdots+\ell_4=m+3 \\ \ell_3,\ell_4 \ge 1}} \p_{y_j} \psi^{\ell_1} \, \p_\theta \psi^{\ell_2} \, 
\p_{Y_3}^2 \big( \xi_{j'} \, H_{j'}^{\ell_3} \, u_j^{\ell_4} -\xi_{j'} \, u_{j'}^{\ell_3} \, H_j^{\ell_4} \big) \\
& \, +\sum_{\substack{\ell_1+\cdots+\ell_4=m+2 \\ \ell_3,\ell_4 \ge 1}} \p_{y_j} \psi^{\ell_1} \, \p_\theta \psi^{\ell_2} \, 
\p_{Y_3} \p_{y_3} \big( \xi_{j'} \, H_{j'}^{\ell_3} \, u_j^{\ell_4} -\xi_{j'} \, u_{j'}^{\ell_3} \, H_j^{\ell_4} \big) \, \Big\} \, ,
\end{align*}
and
\begin{align*}
\cT_2 \, = \, & \, {\bf c}_0 \, \Big\{ \, 
\sum_{\ell_1+\ell_2+\ell_3=m+2} \p_\theta \psi^{\ell_1} \, \big( (\p_t +u_j^0 \, \p_{y_j}) \psi^{\ell_2} \, \p_{Y_3} \p_{y_3} \xi_{j'} \, H_{j'}^{\ell_3} 
-H_j^0 \, \p_{y_j} \psi^{\ell_2} \, \p_{Y_3} \p_{y_3} \xi_{j'} \, u_{j'}^{\ell_3} \big) \\
& \, +\sum_{\ell_1+\ell_2+\ell_3=m+2} \p_{y_j} \psi^{\ell_1} \, \p_\theta \psi^{\ell_2} \, 
\big( b \, \p_{Y_3} \p_{y_3} u_j^{\ell_3} -c \, \p_{Y_3} \p_{y_3} H_j^{\ell_3} \big) \\
& \, +\sum_{\ell_1+\ell_2+\ell_3=m+1} \p_\theta \psi^{\ell_1} \, \big( (\p_t +u_j^0 \, \p_{y_j}) \psi^{\ell_2} \, 
\p_{y_3}^2 \xi_{j'} \, H_{j'}^{\ell_3} -H_j^0 \, \p_{y_j} \psi^{\ell_2} \, \p_{y_3}^2 \xi_{j'} \, u_{j'}^{\ell_3} \big) \\
& \, +\sum_{\ell_1+\ell_2+\ell_3=m+1} \p_{y_j} \psi^{\ell_1} \, \p_\theta \psi^{\ell_2} \, 
\big( b \, \p_{y_3}^2 u_j^{\ell_3} -c \, \p_{y_3}^2 H_j^{\ell_3} \big) \\
& \, +\sum_{\substack{\ell_1+\cdots+\ell_4=m+2 \\ \ell_3,\ell_4 \ge 1}} \p_{y_j} \psi^{\ell_1} \, \p_\theta \psi^{\ell_2} \, 
\p_{Y_3} \p_{y_3} \big( \xi_{j'} \, H_{j'}^{\ell_3} \, u_j^{\ell_4} -\xi_{j'} \, u_{j'}^{\ell_3} \, H_j^{\ell_4} \big) \\
& \, +\sum_{\substack{\ell_1+\cdots+\ell_4=m+1 \\ \ell_3,\ell_4 \ge 1}} \p_{y_j} \psi^{\ell_1} \, \p_\theta \psi^{\ell_2} \, 
\p_{y_3}^2 \big( \xi_{j'} \, H_{j'}^{\ell_3} \, u_j^{\ell_4} -\xi_{j'} \, u_{j'}^{\ell_3} \, H_j^{\ell_4} \big) \, \Big\} \, ,
\end{align*}
\bigskip

$\bullet$ \underline{Step 5}. Conclusion. At this point, it remains to substitute the expressions \eqref{expressionF8m}, 
\eqref{expressionF3+jm} of $F_8^{\ell_2}$ and $F_{3+j}^{\ell_2}$ in the definition of $\cT_3$, and the previous expression 
of $\cT_1$ gives $\cT_1+\cT_3=0$. Similarly, the substitution of $F_8^{\ell_2}$ and $F_{3+j}^{\ell_2}$ in the definition of 
$\cT_4$ yields $\cT_2+\cT_4=0$. We have thus proved the relation $\p_{Y_3} (\cA-\cB)=0$.

We now go back to the definitions \eqref{lemB3defA}, \eqref{lemB3defB}. Since $\cA-\cB$ do not depend on $Y_3$, we have:
\begin{align}
\cA-\cB \, = \, \lim_{Y_3 \to \infty} \cA-\cB \, = \, - {\bf c}_0 \, \Big\{ \, & \, \sum_{\substack{\ell_1+\ell_2=m+2 \\ \ell_2 \ge 1}} 
\big( \p_t +u_j^0 \, \p_{y_j} \big) \p_\theta \psi^{\ell_1} \, \xi_{j'} \, \uH_{j'}^{\ell_2} 
-H_j^0 \, \p_{y_j} \p_\theta \psi^{\ell_1} \, \xi_{j'} \, \uu_{j'}^{\ell_2} \notag \\
& \, +{\color{blue} \sum_{\ell_1+\ell_2=m+2} \p_\theta \psi^{\ell_1} \, \big( \tau \, \uF_8^{\ell_2} -\xi_j \, \uF_{3+j}^{\ell_2} \big)} 
\label{lemB3conclusion} \\
& \, +\sum_{\ell_1+\ell_2=m+1} 
\big( \p_t +u_j^0 \, \p_{y_j} \big) \psi^{\ell_1} \, \uF_8^{\ell_2} -H_j^0 \, \p_{y_j} \psi^{\ell_1} \, \uF_7^{\ell_2} \Big\} \, ,\notag
\end{align}
and we need to show that the right hand side of \eqref{lemB3conclusion} vanishes. Let us first consider the intermediate blue 
term. Since we have already solved the fast problems \eqref{inductionHm2}, we can apply Theorem \ref{theorem_fast_problem} 
and we have:
$$
\forall \, \mu \, = \, 1,\dots,m-1 \, ,\quad \p_{Y_3} F_6^\mu +\xi_j \, \p_\theta F_{3+j}^\mu \, = \, \tau \, \p_\theta F_8^\mu \, ,
$$
so we get (here we use again Theorem \ref{theorem_fast_problem} for concluding that the slow mean vanishes):
$$
\forall \, \mu \, = \, 1,\dots,m-1 \, ,\quad \tau \, \uF_8^\mu -\xi_j \, \uF_{3+j}^\mu 
\, = \, \tau \, \widehat{\uF}_8^\mu(0) -\xi_j \, \widehat{\uF}_{3+j}^\mu(0) \, = \, 0 \, .
$$
Using Lemma \ref{lem_compatibilite_div}, we have:
$$
\p_{Y_3} F_6^m +\xi_j \, \p_\theta F_{3+j}^m \, = \, \tau \, \p_\theta F_8^m \, ,
$$
and combining with \eqref{inductionHm5}, we also get $\tau \, \uF_8^m -\xi_j \, \uF_{3+j}^m=0$. This implies that the blue term in 
\eqref{lemB3conclusion} vanishes. We are then left with:
\begin{align*}
\cA-\cB \, = \, - {\bf c}_0 \, \Big\{ \, & \, \sum_{\substack{\ell_1+\ell_2=m+2 \\ \ell_2 \ge 1}} 
\big( \p_t +u_j^0 \, \p_{y_j} \big) \p_\theta \psi^{\ell_1} \, \xi_{j'} \, \uH_{j'}^{\ell_2} 
-H_j^0 \, \p_{y_j} \p_\theta \psi^{\ell_1} \, \xi_{j'} \, \uu_{j'}^{\ell_2} \\
& \, +\sum_{\ell_1+\ell_2=m+1} 
\big( \p_t +u_j^0 \, \p_{y_j} \big) \psi^{\ell_1} \, \uF_8^{\ell_2} -H_j^0 \, \p_{y_j} \psi^{\ell_1} \, \uF_7^{\ell_2} \Big\} \\
= \, - {\bf c}_0 \, \Big\{ \, & \, \sum_{\substack{\ell_1+\ell_2=m+2 \\ \ell_2 \ge 1}} 
\big( \p_t +u_j^0 \, \p_{y_j} \big) \p_\theta \psi^{\ell_1} \, \xi_{j'} \, \uH_{j'}^{\ell_2} 
-H_j^0 \, \p_{y_j} \p_\theta \psi^{\ell_1} \, \xi_{j'} \, \uu_{j'}^{\ell_2} \\
& \, +\sum_{\ell_1+\ell_2=m+1} \big( \p_t +u_j^0 \, \p_{y_j} \big) \psi^{\ell_1} \, \p_\theta \, \xi_{j'} \, \uH_{j'}^{\ell_2+1} 
-H_j^0 \, \p_{y_j} \psi^{\ell_1} \, \p_\theta \, \xi_{j'} \, \uu_{j'}^{\ell_2+1} \, \Big\} \, = \, 0 \, .
\end{align*}
We then take the double trace of $\cA-\cB$ on $y_3=Y_3=0$ and the result of Lemma \ref{lemB3} follows.
\end{proof}

\noindent We now turn to the proof of Lemma \ref{lemB2}, which is crucial in view of determining the slow mean of the corrector 
$U^{\, m+1,\pm}$.

\begin{proof}[Proof of Lemma \ref{lemB2}]
We keep dropping the superscript $\pm$. Let us first recall the expression of the source terms in the slow mean problem. The trace 
on $\Gamma_0$ of the interior source term $\bF_6^m$ is given by:
\begin{align}
\bF_6^m|_{\Gamma_0} \, = \, {\bf c}_0 \, \Big\{ \, & \, 
\sum_{\ell_1+\ell_2=m+1} \big( c \, \p_\theta \psi^{\ell_1+1} +(\p_t +u_j^0 \, \p_{y_j}) \psi^{\ell_1} \big) \, \p_{y_3} \uH_3^{\ell_2} 
-\big( b \, \p_\theta \psi^{\ell_1+1} +H_j^0 \, \p_{y_j} \psi^{\ell_1} \big) \p_{y_3} \uu_3^{\ell_2} \notag \\
& \, +\sum_{\substack{\ell_1+\ell_2=m+1 \\ \ell_1,\ell_2 \ge 1}} 
\p_{y_j} \big( \uH_j^{\ell_1} \, \uu_3^{\ell_2} -\uu_j^{\ell_1} \, \uH_3^{\ell_2} \big) 
+\sum_{\substack{\ell_1+\ell_2+\ell_3=m+2 \\ \ell_2,\ell_3 \ge 1}} \p_\theta \psi^{\ell_1} \, 
\p_{y_3} \big( \xi_j \, \uu_j^{\ell_2} \, \uH_3^{\ell_3} -\xi_j \, \uH_j^{\ell_2} \, \uu_3^{\ell_3} \big) \label{lemB2-eq01} \\
& \, +\sum_{\substack{\ell_1+\ell_2+\ell_3=m+1 \\ \ell_2,\ell_3 \ge 1}} \p_{y_j} \psi^{\ell_1} \, 
\p_{y_3} \big( \uu_j^{\ell_2} \, \uH_3^{\ell_3} -\uH_j^{\ell_2} \, \uu_3^{\ell_3} \big) \Big\} \, ,\notag
\end{align}
and the boundary source terms are given by:
$$
\bG_1^m \, = \, \widehat{G}_1^m(0) -\widehat{u}_{3,\star}^{\, m+1}(0)|_{y_3=Y_3=0} \, ,\quad 
\bG_2^m \, = \, \widehat{G}_2^m(0) -\widehat{H}_{3,\star}^{\, m+1}(0)|_{y_3=Y_3=0} \, ,
$$
where $G_1^m$, $G_2^m$ are given in \eqref{s3-def_G_1^m,pm}, \eqref{s3-def_G_2^m,pm}, and the fast means 
$\widehat{u}_{3,\star}^{\, m+1}(0),\widehat{H}_{3,\star}^{\, m+1}(0)$ are determined by solving:
$$
\p_{Y_3} \widehat{u}_{3,\star}^{\, m+1}(0) \, = \, \widehat{F}_{7,\star}^m(0) \, ,\quad 
\p_{Y_3} \widehat{H}_{3,\star}^{\, m+1}(0) \, = \, \widehat{F}_{8,\star}^m(0) \, ,
$$
with the condition of exponential decay at infinity. We thus need to show the relation
\begin{multline}
\label{lemB2-eq02}
\bF_6^m|_{\Gamma_0} +\Big( (\p_t +u_j^0 \, \p_{y_j}) \, \widehat{H}_{3,\star}^{\, m+1}(0) 
-H_j^0 \, \p_{y_j} \, \widehat{u}_{3,\star}^{\, m+1}(0) \Big) \big|_{y=3=Y_3=0} \\
= \, (\p_t +u_j^0 \, \p_{y_j}) \, \widehat{G}_2^m(0) -H_j^0 \, \p_{y_j} \, \widehat{G}_1^m(0) \, ,
\end{multline}
and in this relation, the mean $\widehat{\psi}^{\, m+1}(0)$ only enters through $\widehat{G}_1^m(0)$ and $\widehat{G}_2^m(0)$. 
The verification of \eqref{lemB2-eq02} is done in several steps.
\bigskip

$\bullet$ \underline{Step 1}. Using Lemma \ref{lemB3} and collecting terms. Let us recall indeed the result of Lemma 
\ref{lemB3} which decomposes part of the left hand side of \eqref{lemB2-eq02}:
\begin{align}
\Big( (\p_t +u_j^0 \, \p_{y_j}) \, \widehat{H}_{3,\star}^{\, m+1}(0) \, & \, -H_j^0 \, \p_{y_j} \, \widehat{u}_{3,\star}^{\, m+1}(0) 
\Big) \Big|_{y=3=Y_3=0} \notag \\
= \, {\bf c}_0 \, \Big\{ \, & \, \sum_{\substack{\ell_1+\ell_2=m+2 \\ \ell_2 \ge 1}} 
(\p_t +u_j^0 \, \p_{y_j}) \, \p_\theta \psi^{\ell_1} \, \xi_{j'} \, H_{j'}^{\ell_2} 
-H_j^0 \, \p_{y_j} \, \p_\theta \psi^{\ell_1} \, \xi_{j'} \, u_{j'}^{\ell_2} \notag \\
& \, +\sum_{\ell_1+\ell_2=m+2} \p_\theta \psi^{\ell_1} \, \big( \tau \, F_8^{\ell_2} -\xi_j \, F_{3+j}^{\ell_2} \big) \notag \\
& \, +\sum_{\ell_1+\ell_2=m+1} 
(\p_t +u_j^0 \, \p_{y_j}) \, \psi^{\ell_1} \, F_8^{\ell_2} -H_j^0 \, \p_{y_j} \psi^{\ell_1} \, F_7^{\ell_2} \notag \\
& \, +\sum_{\substack{\ell_1+\ell_2+\ell_3=m+3 \\ \ell_2,\ell_3 \ge 1}} \p_\theta \psi^{\ell_1} \, 
\p_{Y_3} \big( \xi_j \, u_j^{\ell_2} \, H_3^{\ell_3} -\xi_j \, H_j^{\ell_2} \, u_3^{\ell_3} \big) \notag \\
& \, +\sum_{\substack{\ell_1+\ell_2+\ell_3=m+2 \\ \ell_2,\ell_3 \ge 1}}  \p_{y_j} \psi^{\ell_1} \, 
\p_{Y_3} \big( u_j^{\ell_2} \, H_3^{\ell_3} -H_j^{\ell_2} \, u_3^{\ell_3} \big) \notag \\
& \, +{\color{blue} \sum_{\ell_1+\ell_2=m+1} 
\big( c \, \p_\theta \psi^{\ell_1+1} +(\p_t +u_j^0 \, \p_{y_j}) \, \psi^{\ell_1} \big) \, \p_{y_3} H_{3,\star}^{\ell_2}} \label{lemB2-eq03} \\
& \, -{\color{blue} \sum_{\ell_1+\ell_2=m+1} 
\big( b \, \p_\theta \psi^{\ell_1+1} +H_j^0 \, \p_{y_j} \, \psi^{\ell_1} \big) \, \p_{y_3} u_{3,\star}^{\ell_2}} \notag \\
& \, +{\color{blue} \sum_{\substack{\ell_1+\ell_2=m+1 \\ \ell_1,\ell_2 \ge 1}} 
\p_{y_j} \big( H_j^{\ell_1} \, u_3^{\ell_2} -u_j^{\ell_1} \, H_3^{\ell_2} \big)_\star} \notag \\
& \, +{\color{blue} \sum_{\substack{\ell_1+\ell_2+\ell_3=m+2 \\ \ell_2,\ell_3 \ge 1}} \p_\theta \psi^{\ell_1} \, 
\p_{y_3} \big( \xi_j \, u_j^{\ell_2} \, H_3^{\ell_3} -\xi_j \, H_j^{\ell_2} \, u_3^{\ell_3} \big)_\star} \notag \\
& \, +{\color{blue} \sum_{\substack{\ell_1+\ell_2+\ell_3=m+1 \\ \ell_2,\ell_3 \ge 1}}  \p_{y_j} \psi^{\ell_1} \, 
\p_{y_3} \big( u_j^{\ell_2} \, H_3^{\ell_3} -H_j^{\ell_2} \, u_3^{\ell_3} \big)_\star} \Big\} \, .\notag 
\end{align}
We now decompose the left hand side of \eqref{lemB2-eq02} as:
\begin{equation}
\label{lemB2-defHm}
\bH^m \, := \, \bF_6^m|_{\Gamma_0} +\Big( (\p_t +u_j^0 \, \p_{y_j}) \, \widehat{H}_{3,\star}^{\, m+1}(0) 
-H_j^0 \, \p_{y_j} \, \widehat{u}_{3,\star}^{\, m+1}(0) \Big) \big|_{y=3=Y_3=0} \, = \, \bH_1^m +\bH_2^m +\bH_3^m +\dot{\bH}_1^m \, ,
\end{equation}
where $\bH_1^m$ incorporates $\bF_6^m|_{\Gamma_0}$, whose expression is given in \eqref{lemB2-eq01}, and the terms 
highlighted in blue in \eqref{lemB2-eq03}, namely:
\begin{align}
\bH_1^m \, := \, {\bf c}_0 \, \Big\{ \, & \, 
\sum_{\ell_1+\ell_2=m+1} {\color{red} \big( c \, \p_\theta \psi^{\ell_1+1} +(\p_t +u_j^0 \, \p_{y_j}) \psi^{\ell_1} \big)} \, \p_{y_3} H_3^{\ell_2} 
-{\color{red} \big( b \, \p_\theta \psi^{\ell_1+1} +H_j^0 \, \p_{y_j} \psi^{\ell_1} \big)} \, \p_{y_3} u_3^{\ell_2} \notag \\
& \, +\sum_{\substack{\ell_1+\ell_2=m+1 \\ \ell_1,\ell_2 \ge 1}} 
\p_{y_j} \big( H_j^{\ell_1} \, u_3^{\ell_2} -u_j^{\ell_1} \, H_3^{\ell_2} \big) 
+\sum_{\substack{\ell_1+\ell_2+\ell_3=m+2 \\ \ell_2,\ell_3 \ge 1}} \p_\theta \psi^{\ell_1} \, 
\p_{y_3} \big( \xi_j \, u_j^{\ell_2} \, H_3^{\ell_3} -\xi_j \, H_j^{\ell_2} \, u_3^{\ell_3} \big) \label{lemB2-eq04} \\
& \, +\sum_{\substack{\ell_1+\ell_2+\ell_3=m+1 \\ \ell_2,\ell_3 \ge 1}} \p_{y_j} \psi^{\ell_1} \, 
\p_{y_3} \big( u_j^{\ell_2} \, H_3^{\ell_3} -H_j^{\ell_2} \, u_3^{\ell_3} \big) \Big\} \, ,\notag
\end{align}
and $\bH_2^m$, $\bH_3^m$, $\dot{\bH}_1^m$ incorporate the remaining (black) terms in \eqref{lemB2-eq03}:
\begin{equation}
\label{lemB2-eq05}
\bH_2^m \, := \, {\bf c}_0 \, \Big\{ \, 
\sum_{\ell_1+\ell_2=m+2} \p_\theta \psi^{\ell_1} \, \big( \tau \, F_8^{\ell_2} -\xi_j \, F_{3+j}^{\ell_2} \big) 
+\sum_{\substack{\ell_1+\ell_2+\ell_3=m+3 \\ \ell_2,\ell_3 \ge 1}} \p_\theta \psi^{\ell_1} \, 
\p_{Y_3} \big( \xi_j \, u_j^{\ell_2} \, H_3^{\ell_3} -\xi_j \, H_j^{\ell_2} \, u_3^{\ell_3} \big) \Big\} \, ,
\end{equation}
\begin{equation}
\label{lemB2-eq06}
\bH_3^m \, := \, {\bf c}_0 \, \Big\{ \, \sum_{\ell_1+\ell_2=m+1} 
\big( \p_t +u_j^0 \, \p_{y_j} \big) \psi^{\ell_1} \, F_8^{\ell_2} -H_j^0 \, \p_{y_j} \psi^{\ell_1} \, F_7^{\ell_2} 
+\sum_{\substack{\ell_1+\ell_2+\ell_3=m+2 \\ \ell_2,\ell_3 \ge 1}}  \p_{y_j} \psi^{\ell_1} \, 
\p_{Y_3} \big( u_j^{\ell_2} \, H_3^{\ell_3} -H_j^{\ell_2} \, u_3^{\ell_3} \big) \Big\} \, ,
\end{equation}
\begin{equation}
\label{lemB2-eq07}
\dot{\bH}_1^m \, := \, {\bf c}_0 \, \Big\{ \, \sum_{\substack{\ell_1+\ell_2=m+2 \\ \ell_2 \ge 1}} 
\big( \p_t +u_j^0 \, \p_{y_j} \big) \, \p_\theta \psi^{\ell_1} \, \xi_{j'} \, H_{j'}^{\ell_2} 
-H_j^0 \, \p_{y_j} \, \p_\theta \psi^{\ell_1} \, \xi_{j'} \, u_{j'}^{\ell_2} \Big\} \, .
\end{equation}
The reason why $\dot{\bH}_1^m$ is denoted differently from the other terms in the decomposition of $\bH^m$ is that it is 
easily recognizable as part of the right hand side of \eqref{lemB2-eq02}. We therefore keep it separate from the other terms.
\bigskip

$\bullet$ \underline{Step 2}. Two substitutions and several simplifications. We use the boundary conditions on $\Gamma_0$ of the 
previous steps in the induction and substitute accordingly in the first line of $\bH_1^m$ in \eqref{lemB2-eq04} (the red terms). After 
simplifying, this yields:
\begin{align*}
\bH_1^m \, = \, {\bf c}_0 \, \Big\{ \, & \, 
{\color{ForestGreen} \sum_{\substack{\ell_1+\ell_2=m+1 \\ \ell_1 \ge 1}} H_j^{\ell_1} \, \p_{y_j} u_3^{\ell_2} -u_j^{\ell_1} \, \p_{y_j} H_3^{\ell_2}} 
+\sum_{\ell_1+\ell_2=m+1} u_3^{\ell_1} \, \nabla \cdot H^{\ell_2} -H_3^{\ell_1} \, \nabla \cdot u^{\ell_2} \\
& \, +\sum_{\substack{\ell_1+\ell_2+\ell_3=m+2 \\ \ell_2,\ell_3 \ge 1}} \p_\theta \psi^{\ell_1} \, 
\big( H_3^{\ell_2} \, \xi_j \, \p_{y_3} u_j^{\ell_3} -u_3^{\ell_2} \, \xi_j \, \p_{y_3} H_j^{\ell_3} \big) \\
& \, +\sum_{\substack{\ell_1+\ell_2+\ell_3=m+1 \\ \ell_2,\ell_3 \ge 1}} \p_{y_j} \psi^{\ell_1} \, 
\big( H_3^{\ell_2} \, \p_{y_3} u_j^{\ell_3} -u_3^{\ell_2} \, \p_{y_3} H_j^{\ell_3} \big) \Big\} \, .
\end{align*}
Since only tangential differentiation is involved in the above green term, we can use again the boundary conditions on $\Gamma_0$ 
of the previous steps in the induction. Defining for future use:
\begin{equation}
\label{lemB2-eq08}
\dot{\bH}_2^m \, := \, {\bf c}_0 \, \Big\{ \, \sum_{\substack{\ell_1+\ell_2=m+1 \\ \ell_2 \ge 1}} 
\big( \p_t +u_j^0 \, \p_{y_j} \big) \,\p_{y_{j'}} \psi^{\ell_1} \, H_{j'}^{\ell_2} -H_j^0 \, \p_{y_j}\p_{y_{j'}} \psi^{\ell_1} \, u_{j'}^{\ell_2} \Big\} \, ,
\end{equation}
we get:
\begin{align}
\bH_1^m \, = \, \dot{\bH}_2^m +{\bf c}_0 \, \Big\{ \, & \, 
\sum_{\ell_1+\ell_2=m+1} u_3^{\ell_1} \, \nabla \cdot H^{\ell_2} -H_3^{\ell_1} \, \nabla \cdot u^{\ell_2} 
+{\color{Mahogany} \sum_{\substack{\ell_1+\ell_2=m+2 \\ \ell_1 \ge 1}} (c \, H_j^{\ell_1} -b \, u_j^{\ell_1}) \, \p_{y_j} \p_\theta \psi^{\ell_2}} \notag \\
& \, +{\color{Mahogany} \sum_{\substack{\ell_1+\ell_2+\ell_3=m+2 \\ \ell_2,\ell_3 \ge 1}} \p_{y_j} \p_\theta \psi^{\ell_1} \, 
\big( \xi_{j'} \, u_{j'}^{\ell_2} \, H_j^{\ell_3} -\xi_{j'} \, H_{j'}^{\ell_2} \, u_j^{\ell_3} \big)} \notag \\
& \, +\sum_{\substack{\ell_1+\ell_2+\ell_3=m+2 \\ \ell_2,\ell_3 \ge 1}} \p_\theta \psi^{\ell_1} \, 
\big( H^{\ell_2} \cdot \nabla (\xi_j \, u_j^{\ell_3}) -u^{\ell_2} \cdot \nabla (\xi_j \, H_j^{\ell_3}) \big) \label{lemB2-eq09} \\
& \, +\sum_{\substack{\ell_1+\ell_2+\ell_3=m+1 \\ \ell_2,\ell_3 \ge 1}} \p_{y_j} \psi^{\ell_1} \, 
\big( H^{\ell_2} \cdot \nabla u_j^{\ell_3} -u^{\ell_2} \cdot \nabla H_j^{\ell_3} \big) \Big\} \, .\notag
\end{align}
\bigskip

$\bullet$ \underline{Step 3}. Integrating by parts. We integrate by parts (with respect to $\theta$) the brown terms in the 
decomposition \eqref{lemB2-eq09}. Let us observe that there holds $c\, \p_\theta H_j^1 -b \, \p_\theta u_j^1 \equiv 0$, 
hence we get:
\begin{align*}
\bH_1^m \, = \, \dot{\bH}_2^m +{\bf c}_0 \, \Big\{ \, & \, 
\sum_{\ell_1+\ell_2=m+1} u_3^{\ell_1} \, \nabla \cdot H^{\ell_2} -H_3^{\ell_1} \, \nabla \cdot u^{\ell_2} 
+\sum_{\substack{\ell_1+\ell_2=m+2 \\ \ell_2 \ge 2}} \p_{y_j} \psi^{\ell_1} \, (b \, \p_\theta u_j^{\ell_2} -c \, \p_\theta H_j^{\ell_2}) \\
& \, +\sum_{\substack{\ell_1+\ell_2+\ell_3=m+2 \\ \ell_2,\ell_3 \ge 1}} \p_{y_j} \psi^{\ell_1} \, \p_\theta 
\big( \xi_{j'} \, H_{j'}^{\ell_2} \, u_j^{\ell_3} -\xi_{j'} \, u_{j'}^{\ell_2} \, H_j^{\ell_3} \big) \\
& \, +\sum_{\substack{\ell_1+\ell_2+\ell_3=m+2 \\ \ell_2,\ell_3 \ge 1}} \p_\theta \psi^{\ell_1} \, 
\big( H^{\ell_2} \cdot \nabla (\xi_j \, u_j^{\ell_3}) -u^{\ell_2} \cdot \nabla (\xi_j \, H_j^{\ell_3}) \big) \\
& \, +\sum_{\substack{\ell_1+\ell_2+\ell_3=m+1 \\ \ell_2,\ell_3 \ge 1}} \p_{y_j} \psi^{\ell_1} \, 
\big( H^{\ell_2} \cdot \nabla u_j^{\ell_3} -u^{\ell_2} \cdot \nabla H_j^{\ell_3} \big) \Big\} \, .
\end{align*}
Since we have already solved the fast problems \eqref{inductionHm2} up to the step $m$, we can use in the latter decomposition 
the relation:
$$
b \, \p_\theta u_j^{\ell_2} -c \, \p_\theta H_j^{\ell_2} \, = \, -F_{3+j}^{\ell_2-1} -u_j^0 \, F_8^{\ell_2-1} +H_j^0 \, F_7^{\ell_2-1} \, ,
$$
which yields:
\begin{align*}
\bH_1^m \, = \, \dot{\bH}_2^m +{\bf c}_0 \, \Big\{ \, & \, 
\sum_{\ell_1+\ell_2=m+1} u_3^{\ell_1} \, \nabla \cdot H^{\ell_2} -H_3^{\ell_1} \, \nabla \cdot u^{\ell_2} 
-\sum_{\ell_1+\ell_2=m+1} \p_{y_j} \psi^{\ell_1} \, F_{3+j}^{\ell_2} \\
& \, +{\color{orange} \sum_{\ell_1+\ell_2=m+1} H_j^0 \, \p_{y_j} \psi^{\ell_1} \, F_7^{\ell_2} -u_j^0 \, \p_{y_j} \psi^{\ell_1} \, F_8^{\ell_2}} \\
& \, +\sum_{\substack{\ell_1+\ell_2+\ell_3=m+2 \\ \ell_2,\ell_3 \ge 1}} \p_{y_j} \psi^{\ell_1} \, \p_\theta 
\big( \xi_{j'} \, H_{j'}^{\ell_2} \, u_j^{\ell_3} -\xi_{j'} \, u_{j'}^{\ell_2} \, H_j^{\ell_3} \big) \\
& \, +\sum_{\substack{\ell_1+\ell_2+\ell_3=m+2 \\ \ell_2,\ell_3 \ge 1}} \p_\theta \psi^{\ell_1} \, 
\big( H^{\ell_2} \cdot \nabla (\xi_j \, u_j^{\ell_3}) -u^{\ell_2} \cdot \nabla (\xi_j \, H_j^{\ell_3}) \big) \\
& \, +\sum_{\substack{\ell_1+\ell_2+\ell_3=m+1 \\ \ell_2,\ell_3 \ge 1}} \p_{y_j} \psi^{\ell_1} \, 
\big( H^{\ell_2} \cdot \nabla u_j^{\ell_3} -u^{\ell_2} \cdot \nabla H_j^{\ell_3} \big) \Big\} \, .
\end{align*}
It is now time to incorporate the quantity $\bH_3^m$, whose expression is given in \eqref{lemB2-eq06}. Adding $\bH_1^m$ 
with $\bH_3^m$ cancels the orange term in the latter decomposition and we get:
\begin{align}
\bH_1^m +\bH_3^m \, = \, \dot{\bH}_2^m +{\bf c}_0 \, \Big\{ \, & \, 
\sum_{\ell_1+\ell_2=m+1} u_3^{\ell_1} \, \nabla \cdot H^{\ell_2} -H_3^{\ell_1} \, \nabla \cdot u^{\ell_2} 
+\sum_{\ell_1+\ell_2=m+1} {\color{Magenta} \p_t \psi^{\ell_1} \, F_8^{\ell_2}} -{\color{blue} \p_{y_j} \psi^{\ell_1} \, F_{3+j}^{\ell_2}} \notag \\
& \, +{\color{blue} \sum_{\substack{\ell_1+\ell_2+\ell_3=m+2 \\ \ell_2,\ell_3 \ge 1}} \p_{y_j} \psi^{\ell_1} \, \Big( \p_\theta 
\big( \xi_{j'} \, H_{j'}^{\ell_2} \, u_j^{\ell_3} -\xi_{j'} \, u_{j'}^{\ell_2} \, H_j^{\ell_3} \big) 
+\p_{Y_3} \big( u_j^{\ell_2} \, H_3^{\ell_3} -H_j^{\ell_2} \, u_3^{\ell_3} \big) \Big)} \notag \\
& \, +{\color{ForestGreen} \sum_{\substack{\ell_1+\ell_2+\ell_3=m+2 \\ \ell_2,\ell_3 \ge 1}} \p_\theta \psi^{\ell_1} \, 
\big( H^{\ell_2} \cdot \nabla (\xi_j \, u_j^{\ell_3}) -u^{\ell_2} \cdot \nabla (\xi_j \, H_j^{\ell_3}) \big)} \label{lemB2-eq10} \\
& \, +{\color{blue} \sum_{\substack{\ell_1+\ell_2+\ell_3=m+1 \\ \ell_2,\ell_3 \ge 1}} \p_{y_j} \psi^{\ell_1} \, 
\big( H^{\ell_2} \cdot \nabla u_j^{\ell_3} -u^{\ell_2} \cdot \nabla H_j^{\ell_3} \big)} \Big\} \, .\notag
\end{align}
\bigskip

$\bullet$ \underline{Step 4}. Substituting in $\bH_2^m$ and recollecting the terms. We now substitute the quantity 
$\tau \, F_8^{\ell_2} -\xi_j \, F_{3+j}^{\ell_2}$ in $\bH_2^m$, whose expression is given in \eqref{lemB2-eq05}. For 
future use, we define:
\begin{equation}
\label{lemB2-eq11}
\dot{\bH}_3^m \, := \, {\bf c}_0 \, \Big\{ \, \sum_{\ell_1+\ell_2=m+2} \p_\theta \psi^{\ell_1} \, \Big( 
(\p_t +u_j^0 \, \p_{y_j}) \, (\xi_{j'} \, H_{j'}^{\ell_2}) -H_j^0 \, \p_{y_j} (\xi_{j'} \, u_{j'}^{\ell_2}) \Big) \Big\} \, .
\end{equation}
We then obtain:
\begin{align}
\bH_2^m \, = \, \dot{\bH}_3^m +{\bf c}_0 \, \Big\{ \, & \, \sum_{\ell_1+\ell_2=m+2} \p_\theta \psi^{\ell_1} \, 
\big( b \, \nabla \cdot u^{\ell_2} -c \, \nabla \cdot H^{\ell_2} \big) 
-{\color{Magenta} \sum_{\ell_1+\ell_2+\ell_3=m+3} \p_t \psi^{\ell_1} \, \p_\theta \psi^{\ell_2} \, \xi_j \, \p_{Y_3} H_j^{\ell_3}} \notag \\
& \, +{\color{blue} \sum_{\ell_1+\ell_2+\ell_3=m+3} \p_{y_j} \psi^{\ell_1} \, \p_\theta \psi^{\ell_2} \, \Big( 
H_j^0 \, \p_{Y_3} (\xi_{j'} \, u_{j'}^{\ell_3}) -u_j^0 \, \p_{Y_3} (\xi_{j'} \, H_{j'}^{\ell_3}) 
+c \, \p_{Y_3} H_j^{\ell_3} -b \, \p_{Y_3} u_j^{\ell_3} \Big)} \notag \\
& \, +{\color{blue} \sum_{\ell_1+\ell_2+\ell_3=m+2} \p_{y_j} \psi^{\ell_1} \, \p_\theta \psi^{\ell_2} \, \Big( 
H_j^0 \, \p_{y_3} (\xi_{j'} \, u_{j'}^{\ell_3}) -u_j^0 \, \p_{y_3} (\xi_{j'} \, H_{j'}^{\ell_3}) 
+c \, \p_{y_3} H_j^{\ell_3} -b \, \p_{y_3} u_j^{\ell_3} \Big)} \notag \\
& \, {\color{Magenta} -\sum_{\ell_1+\ell_2+\ell_3=m+2} \p_t \psi^{\ell_1} \, \p_\theta \psi^{\ell_2} \, \xi_j \, \p_{y_3} H_j^{\ell_3}} 
\label{lemB2-eq12} \\
& \, +{\color{ForestGreen} \sum_{\substack{\ell_1+\ell_2+\ell_3=m+2 \\ \ell_2,\ell_3 \ge 1}} \p_\theta \psi^{\ell_1} \, 
\nabla \cdot \big( \xi_j \, H_j^{\ell_2} \, u^{\ell_3} -\xi_j \, u_j^{\ell_2} \, H^{\ell_3} \big)} \notag \\
& \, +{\color{blue} \sum_{\substack{\ell_1+\cdots+\ell_4=m+3 \\ \ell_3,\ell_4 \ge 1}} \p_{y_j} \psi^{\ell_1} \, \p_\theta \psi^{\ell_2} \, 
\p_{Y_3} \big( \xi_{j'} \, H_{j'}^{\ell_3} \, u_j^{\ell_4} -\xi_{j'} \, H_{j'}^{\ell_3} \, u_j^{\ell_4} \big)} \notag \\
& \, +{\color{blue} \sum_{\substack{\ell_1+\cdots+\ell_4=m+2 \\ \ell_3,\ell_4 \ge 1}} \p_{y_j} \psi^{\ell_1} \, \p_\theta \psi^{\ell_2} \, 
\p_{y_3} \big( \xi_{j'} \, H_{j'}^{\ell_3} \, u_j^{\ell_4} -\xi_{j'} \, H_{j'}^{\ell_3} \, u_j^{\ell_4} \big)} \Big\} \, .\notag
\end{align}
Observe the partial cancelation between the green terms in \eqref{lemB2-eq10} and \eqref{lemB2-eq12}.

Going back to the definition \eqref{lemB2-defHm} and using \eqref{lemB2-eq10}, \eqref{lemB2-eq12}, we can decompose 
$\bH^m$ under the form:
$$
\bH^m \, = \, \dot{\bH}_1^m +\dot{\bH}_2^m +\dot{\bH}_3^m +\bH_4^m +\bH_5^m +\bH_6^m \, ,
$$
where:
\begin{itemize}
 \item $\dot{\bH}_1^m$, $\dot{\bH}_2^m$, $\dot{\bH}_3^m$ are defined respectively in \eqref{lemB2-eq07}, \eqref{lemB2-eq08}, 
 \eqref{lemB2-eq11},
 \item $\bH_4^m$ incorporates the pink terms in \eqref{lemB2-eq10} and \eqref{lemB2-eq12},
 \item $\bH_5^m$ incorporates the blue terms in \eqref{lemB2-eq10} and \eqref{lemB2-eq12},
 \item $\bH_6^m$ incorporates all other terms in \eqref{lemB2-eq10} and \eqref{lemB2-eq12},
\end{itemize}
which corresponds to:
\begin{align*}
\bH_4^m \, := \, {\bf c}_0 \, \Big\{ \, \sum_{\ell_1+\ell_2=m+1} \p_t \psi^{\ell_1} \, F_8^{\ell_2} \, 
& \, -\sum_{\ell_1+\ell_2+\ell_3=m+3} \p_t \psi^{\ell_1} \, \p_\theta \psi^{\ell_2} \, \xi_j \, \p_{Y_3} H_j^{\ell_3} \\
& \, -\sum_{\ell_1+\ell_2+\ell_3=m+2} \p_t \psi^{\ell_1} \, \p_\theta \psi^{\ell_2} \, \xi_j \, \p_{y_3} H_j^{\ell_3} \Big\} \, ,
\end{align*}
\begin{align*}
\bH_5^m \, := \, {\bf c}_0 \, \Big\{ \, & \, -\sum_{\ell_1+\ell_2=m+1} \p_{y_j} \psi^{\ell_1} \, F_{3+j}^{\ell_2} \\
& \, +\sum_{\ell_1+\ell_2+\ell_3=m+3} \p_{y_j} \psi^{\ell_1} \, \p_\theta \psi^{\ell_2} \, \Big( 
H_j^0 \, \p_{Y_3} (\xi_{j'} \, u_{j'}^{\ell_3}) -u_j^0 \, \p_{Y_3} (\xi_{j'} \, H_{j'}^{\ell_3}) 
+c \, \p_{Y_3} H_j^{\ell_3} -b \, \p_{Y_3} u_j^{\ell_3} \Big) \\
& \, +\sum_{\ell_1+\ell_2+\ell_3=m+2} \p_{y_j} \psi^{\ell_1} \, \p_\theta \psi^{\ell_2} \, \Big( 
H_j^0 \, \p_{y_3} (\xi_{j'} \, u_{j'}^{\ell_3}) -u_j^0 \, \p_{y_3} (\xi_{j'} \, H_{j'}^{\ell_3}) 
+c \, \p_{y_3} H_j^{\ell_3} -b \, \p_{y_3} u_j^{\ell_3} \Big) \\
& \, +\sum_{\substack{\ell_1+\ell_2+\ell_3=m+2 \\ \ell_2,\ell_3 \ge 1}} \p_{y_j} \psi^{\ell_1} \, \Big( 
\p_\theta \big( \xi_{j'} \, H_{j'}^{\ell_2} \, u_j^{\ell_3} -\xi_{j'} \, u_{j'}^{\ell_2} \, H_j^{\ell_3} \big) 
+\p_{Y_3} \big( u_j^{\ell_2} \, H_3^{\ell_3} -H_j^{\ell_2} \, u_3^{\ell_3} \big) \Big) \\
& \, +\sum_{\substack{\ell_1+\ell_2+\ell_3=m+1 \\ \ell_2,\ell_3 \ge 1}} \p_{y_j} \psi^{\ell_1} \, 
\big( H^{\ell_2} \cdot \nabla u_j^{\ell_3} -u^{\ell_2} \cdot \nabla H_j^{\ell_3} \big) \\
& \, +\sum_{\substack{\ell_1+\cdots+\ell_4=m+3 \\ \ell_3,\ell_4 \ge 1}} \p_{y_j} \psi^{\ell_1} \, \p_\theta \psi^{\ell_2} \, 
\p_{Y_3} \big( \xi_{j'} \, H_{j'}^{\ell_3} \, u_j^{\ell_4} -\xi_{j'} \, H_{j'}^{\ell_3} \, u_j^{\ell_4} \big) \\
& \, +\sum_{\substack{\ell_1+\cdots+\ell_4=m+2 \\ \ell_3,\ell_4 \ge 1}} \p_{y_j} \psi^{\ell_1} \, \p_\theta \psi^{\ell_2} \, 
\p_{y_3} \big( \xi_{j'} \, H_{j'}^{\ell_3} \, u_j^{\ell_4} -\xi_{j'} \, H_{j'}^{\ell_3} \, u_j^{\ell_4} \big) \Big\} \, ,
\end{align*}
\begin{align*}
\bH_6^m \, := \, {\bf c}_0 \, \Big\{ \, & \, 
\sum_{\ell_1+\ell_2=m+1} \big( u_3^{\ell_1} -c \, \p_\theta \psi^{\ell_1+1} \big) \, \nabla \cdot H^{\ell_2} 
-\big( H_3^{\ell_1} -b \, \p_\theta \psi^{\ell_1+1} \big) \, \nabla \cdot u^{\ell_2} \\
& \, +\sum_{\substack{\ell_1+\ell_2+\ell_3=m+2 \\ \ell_2,\ell_3 \ge 1}} \p_\theta \psi^{\ell_1} \, 
\big( \xi_j \, H_j^{\ell_2} \, \nabla \cdot u^{\ell_3} -\xi_j \, u_j^{\ell_2} \, \nabla \cdot H^{\ell_3} \big) \Big\} \, .
\end{align*}
\bigskip

$\bullet$ \underline{Step 5}. Substituting in $\bH_4^m$ and using boundary conditions. Substituting the value of 
$F_8^{\ell_2}$ in $\bH_4^m$, we obtain:
\begin{align*}
\bH_4^m \, = \, {\bf c}_0 \, \Big\{ \, -\sum_{\ell_1+\ell_2=m+1} \p_t \psi^{\ell_1} \, \nabla \cdot H^{\ell_2} \, 
& \, +\sum_{\ell_1+\ell_2+\ell_3=m+2} \p_t \psi^{\ell_1} \, \p_{y_j} \psi^{\ell_2} \, \p_{Y_3} H_j^{\ell_3} \\
& \, +\sum_{\ell_1+\ell_2+\ell_3=m+1} \p_t \psi^{\ell_1} \, \p_{y_j} \psi^{\ell_2} \, \p_{y_3} H_j^{\ell_3} \Big\} \, ,
\end{align*}
and we therefore get:
\begin{align}
\bH_4^m +\bH_6^m \, = \, {\bf c}_0 \, \Big\{ \, & \, 
\sum_{\ell_1+\ell_2=m+1} \big( u_3^{\ell_1} -c \, \p_\theta \psi^{\ell_1+1} -\p_t \psi^{\ell_1} \big) \, \nabla \cdot H^{\ell_2} 
-\big( H_3^{\ell_1} -b \, \p_\theta \psi^{\ell_1+1} \big) \, \nabla \cdot u^{\ell_2} \notag \\
& \, +\sum_{\substack{\ell_1+\ell_2+\ell_3=m+2 \\ \ell_2,\ell_3 \ge 1}} \p_\theta \psi^{\ell_1} \, 
\big( \xi_j \, H_j^{\ell_2} \, \nabla \cdot u^{\ell_3} -\xi_j \, u_j^{\ell_2} \, \nabla \cdot H^{\ell_3} \big) \Big\} \notag \\
+ \, {\bf c}_0 \, \Big\{ \, & \, \sum_{\ell_1+\ell_2+\ell_3=m+2} \p_t \psi^{\ell_1} \, \p_{y_j} \psi^{\ell_2} \, \p_{Y_3} H_j^{\ell_3} 
+\sum_{\ell_1+\ell_2+\ell_3=m+1} \p_t \psi^{\ell_1} \, \p_{y_j} \psi^{\ell_2} \, \p_{y_3} H_j^{\ell_3} \Big\} \notag \\
= \, {\bf c}_0 \, \Big\{ \, & \, {\color{ForestGreen} \sum_{\ell_1+\ell_2=m+1} \p_{y_j} \psi^{\ell_1} \, 
\big( u_j^0 \, \nabla \cdot H^{\ell_2} -H_j^0 \, \nabla \cdot u^{\ell_2} \big)} \notag \\
& \, {\color{orange} \sum_{\substack{\ell_1+\ell_2+\ell_3=m+1 \\ \ell_2,\ell_3  \ge 1}} \p_{y_j} \psi^{\ell_1} \, 
\big( u_j^{\ell_2} \, \nabla \cdot H^{\ell_3} -H_j^{\ell_2} \, \nabla \cdot u^{\ell_3} \big)} \label{lemB2-eq13} \\
& \, {\color{blue} \sum_{\ell_1+\ell_2+\ell_3=m+2} \p_{y_j} \psi^{\ell_1} \, \p_t \psi^{\ell_2} \, \p_{Y_3} H_j^{\ell_3}} 
+{\color{blue} \sum_{\ell_1+\ell_2+\ell_3=m+1} \p_{y_j} \psi^{\ell_1} \, \p_t \psi^{\ell_2} \, \p_{y_3} H_j^{\ell_3}} \Big\} \, .\notag
\end{align}
where we have used again the boundary conditions on $\Gamma_0$ of the previous steps in the induction.
\bigskip

$\bullet$ \underline{Step 6}. Conclusion. We now substitute $F_{3+j}^{\ell_2}$ in the expression of $\bH_5^m$ and get:
\begin{align*}
\bH_5^m \, = \, {\bf c}_0 \, \Big\{ \, & \, \sum_{\ell_1+\ell_2=m+1} \p_{y_j} \psi^{\ell_1} \, \Big( 
(\p_t +u_{j'}^0 \,\p_{y_{j'}}) H_j^{\ell_2} -H_{j'}^0 \,\p_{y_{j'}} u_j^{\ell_2} \, 
{\color{ForestGreen} -u_j^0 \, \nabla \cdot H^{\ell_2} +H_j^0 \, \nabla \cdot u^{\ell_2}} \Big) \\
& \, -{\color{blue} \sum_{\ell_1+\ell_2+\ell_3=m+2} \p_{y_j} \psi^{\ell_1} \, \p_t \psi^{\ell_2} \, \p_{Y_3} H_j^{\ell_3}} 
-{\color{blue} \sum_{\ell_1+\ell_2+\ell_3=m+1} \p_{y_j} \psi^{\ell_1} \, \p_t \psi^{\ell_2} \, \p_{y_3} H_j^{\ell_3}} \\
& \, +{\color{orange} \sum_{\substack{\ell_1+\ell_2+\ell_3=m+2 \\ \ell_2,\ell_3 \ge 1}} \p_{y_j} \psi^{\ell_1} \, \Big( 
H_j^{\ell_2} \, \nabla \cdot u^{\ell_3} -u_j^{\ell_2} \, \nabla \cdot H^{\ell_3} \big)} \Big\} \, .
\end{align*}
Combining with \eqref{lemB2-eq13} cancels the green, blue and orange terms, and we are eventually left with:
$$
\bH^m \, = \, \dot{\bH}_1^m +\dot{\bH}_2^m +\dot{\bH}_3^m +{\bf c}_0 \, \Big\{ \, \sum_{\ell_1+\ell_2=m+1} \p_{y_j} \psi^{\ell_1} \, 
\Big( (\p_t +u_{j'}^0 \,\p_{y_{j'}}) H_j^{\ell_2} -H_{j'}^0 \,\p_{y_{j'}} u_j^{\ell_2} \, \Big) \Big\} \, .
$$
Using the definitions \eqref{lemB2-eq07}, \eqref{lemB2-eq08}, \eqref{lemB2-eq11} of $\dot{\bH}_1^m$, $\dot{\bH}_2^m$, 
$\dot{\bH}_3^m$ as well as the expressions \eqref{s3-def_G_1^m,pm}, \eqref{s3-def_G_2^m,pm} of $G_1^m$, $G_2^m$, 
we obtain eventually:
$$
\bH^m \, = \, \big( \p_t +u_j^0 \, \p_{y_j} \big) \widehat{G}_2^m(0) -H_j^0 \, \p_{y_j} \widehat{G}_1^m(0) \, ,
$$
which completes the proof of Lemma \ref{lemB2}.
\end{proof}

\bibliographystyle{alpha}
\bibliography{Biblio_MHD}
\end{document}